\documentclass[11pt]{amsart}
\usepackage{amscd}
\usepackage[all]{xy}
\usepackage{graphicx}
\usepackage{amsmath}
\usepackage{amsfonts}
\usepackage{amssymb}
\usepackage{latexsym}
\usepackage{dsfont}
\usepackage{slashed}

\usepackage{soul}
\usepackage{comment}
\usepackage{color}
\usepackage{rotating}
\usepackage{url}

\usepackage{longtable}
\usepackage[table]{xcolor}

\usepackage{tikz}
\usetikzlibrary{arrows,decorations.markings}
\usetikzlibrary{matrix}
\usetikzlibrary{graphs}
\usetikzlibrary{backgrounds}

\newcommand{\gr}{\rowcolor[gray]{.89}}
\newcommand{\f}[1]{{$\scriptstyle{#1}$}}
\newcommand{\y}[1]{{$\scriptstyle{#1}$}}
\newcommand{\ttt}{{\mathfrak t}}
\newcommand{\uuu}{{\mathfrak u}}
\renewcommand{\sp}{\mathfrak{\mathop{sp}}}

\definecolor{grey}{RGB}{128,128,128}

\numberwithin{equation}{subsection}
\theoremstyle{plain}
\newtheorem{lemma}[equation]{Lemma}
\newtheorem{proposition}[equation]{Proposition}
\newtheorem{theorem}[equation]{Theorem}
\newtheorem{corollary}[equation]{Corollary}

\theoremstyle{definition}
\newtheorem{definition}[equation]{Definition}
\newtheorem{remark}[equation]{Remark}
\newtheorem{example}[equation]{Example}
\newtheorem{examples}[equation]{Examples}
\newtheorem{subsec}[equation]{}
\newtheorem{construction}[equation]{Construction}

\newtheorem*{proposition*}{Proposition}
\newtheorem*{lemma*}{Lemma}

\newcommand{\R}{{\mathds R}}
\newcommand{\C}{{\mathds C}}
\newcommand{\F}{{\mathds F}}

\newcommand{\Z}{{\mathds Z}}

\newcommand{\Q}{{\mathds Q}}

\newcommand{\Aa}{{\mathcal A}}

\newcommand{\Cc}{{\mathcal C}}    

\newcommand{\Nn}{{\mathcal N}}
\newcommand{\Oo}{{\mathcal O}}

\newcommand{\Tt}{{\mathcal T}}

\newcommand{\Zz}{{\mathcal Z}}

\newcommand{\ad}{{\rm ad \,}}
\newcommand{\Ad}{{\rm Ad }}

\newcommand{\om}{{\omega}}

\renewcommand{\frak}[1]{{\mathfrak {#1}}}
\renewcommand{\a}{{\mathfrak a}}

\newcommand{\e}{{\mathfrak e}}

\newcommand{\z}{{\mathfrak z}}

\newcommand{\g}{{\mathfrak g}}
\newcommand{\gl}{{\mathfrak {gl}}}
\newcommand{\h}{{\mathfrak h}}

\renewcommand{\c}{{\mathfrak c}}

\newcommand{\s}{{\mathfrak s}}

\renewcommand{\u}{{\mathfrak u}}
\newcommand{\su}{\mathfrak{su}}

\newcommand{\ssl}{{\mathfrak {sl}}}

\renewcommand{\t}{{\mathfrak t }}

\newcommand{\coker}{\mbox{coker}}

\newcommand{\GL}{{\rm GL}}
\newcommand{\SL}{{\rm SL}}

\newcommand{\Sp}{{\rm Sp}}

\newcommand{\hs}{\kern 0.8pt}
\newcommand{\hsss}{\kern 1.2pt}
\newcommand{\hm}{\kern -0.8pt}
\newcommand{\hmm}{\kern -1.2pt}

\newcommand{\hssh}{\kern 1.2pt}
\newcommand{\hshs}{\kern 1.6pt}
\newcommand{\hssss}{\kern 2.0pt}

\newcommand{\into}{\hookrightarrow}
\newcommand{\isoto}{\overset{\sim}{\to}}

\newcommand{\labelto}[1]{\xrightarrow{\makebox[1.5em]{\scriptsize ${#1}$}}}

\newcommand{\upgam}{\hs^\gamma\kern-0.5pt }
\newcommand{\upalpha}{\hs^\alpha\kern-0.5pt}
\newcommand{\emm}{\bfseries}
\newcommand{\PGL}{{\rm PGL}}
\newcommand{\PSL}{\mathrm{PSL}}
\newcommand{\SU}{\mathrm{SU}}

\newcommand{\Ga}{{\Gamma}}

\renewcommand{\AA}{{\bf A}}
\newcommand{\BB}{{\bf B}}
\newcommand{\CC}{{\bf C}}
\newcommand{\GG}{{\bf G}}
\newcommand{\HH}{{\bf H}}
\newcommand{\KK}{{\bf K}}
\newcommand{\MM}{{\bf M}}
\newcommand{\NN}{{\bf N}}
\renewcommand{\SS}{{\bf S}}
\newcommand{\TT}{{\bf T}}
\newcommand{\YY}{{\bf Y}}
\newcommand{\ZZ}{{\bf Z}}

\newcommand{\GGtil}{{\widetilde{\GG}}}

\newcommand{\wt}{\widetilde}
\DeclareMathOperator{\Lie}{Lie}
\DeclareMathOperator{\Inn}{Inn}
\DeclareMathOperator{\Out}{Out}

\newcommand{\sT}{{\mathcal T}}

\newcommand{\Gtil}{{\widetilde{G}}
}
\newcommand{\Hom}{{\rm Hom}}
\newcommand{\Gal}{{\rm Gal}}
\newcommand{\Zm}{{\mathcal Z}}

\newcommand{\SmallMatrix}[1]{\text{{\tiny\arraycolsep=0.4\arraycolsep\ensuremath
    {\begin{pmatrix}#1\end{pmatrix}}}}}

\newcommand{\im}{{\rm im\,}}
\newcommand{\ov}{\overline}

\newcommand{\Ho}{{\mathrm{H}\kern 0.3pt}}
\newcommand{\Zl}{{\mathrm{Z}\kern 0.2pt}}
\newcommand{\Bd}{{\mathrm{B}\kern 0.2pt}}

\newcommand{\id}{{\rm id}}

\newcommand{\Cl}{{\rm Cl}}

\newcommand{\stand}{{\rm st}}
\newcommand{\tw}{{\rm tw}}

\newcommand{\Stab}{{\rm Stab}}
\newcommand{\Aut}{{\rm Aut}}

\newcommand{\Spec}{{\rm Spec\,}}
\newcommand{\diag}{{\rm diag}}

\def\G{{\mathbb G}}

\newcommand{\cc}{\raise 1.7pt \hbox{\Tiny{$\bullet$}}}

\newcommand{\Nm}{{\mathcal N}}

\def\GmC{{\C^\times}}
\newcommand{\Xf}{{\sf X}}

\newcommand{\cC}{{\scriptscriptstyle{\C}}}

\def\upgam{{\hs^\gamma\hm}}
\newcommand{\Res}{{ {\rm R}_{\C/\R}\hs}}

\def\Am{{\mathcal A}}
\def\Fm{{\mathcal F}}
\def\OOm{{\mathcal O}}

\def\Cg{{\mathfrak C}}

\def\pbar{{\ov p}}
\def\gbar{{\ov g}}

\def\wt{\widetilde}
\def\GGtil{{\wt \GG}}

\def\Zgg{{\mathfrak z}}

\newcommand{\vs}{\varsigma}
\newcommand{\triv}{{\rm triv}}
\newcommand{\diagg}{{\rm diag}}
\newcommand{\X}{{\sf X}}

\begin{document}


\def\DynkinNodeSize{1.5mm}
\def\DynkinArrowLength{1.5mm}
\tikzset{
dnode/.style={
circle,
inner sep=0pt,
minimum size=\DynkinNodeSize,
fill=white,
draw},
middlearrow/.style={
decoration={markings,
mark=at position 0.6 with
{\draw (0:0mm) -- +(+135:\DynkinArrowLength); \draw (0:0mm) -- +(-135:\DynkinArrowLength);},
},
postaction={decorate}
},
leftrightarrow/.style={
decoration={markings,
mark=at position 0.999 with
{
\draw (0:0mm) -- +(+135:\DynkinArrowLength); \draw (0:0mm) -- +(-135:\DynkinArrowLength);
},
mark=at position 0.001 with
{
\draw (0:0mm) -- +(+45:\DynkinArrowLength); \draw (0:0mm) -- +(-45:\DynkinArrowLength);
},
},
postaction={decorate}
},
sedge/.style={
},
dedge/.style={
middlearrow,
double distance=0.5mm,
},
tedge/.style={
middlearrow,
double distance=1.0mm+\pgflinewidth,
postaction={draw}, 
},
infedge/.style={
leftrightarrow,
double distance=0.5mm,
},
}

\title[Trivectors in $\R^9$]%
{Real graded  Lie algebras, Galois cohomology,\\
and classification of trivectors in $\R^9$}

\author{Mikhail  Borovoi}
\address{Raymond and Beverly Sackler School of Mathematical Sciences,
Tel Aviv University, 6997801 Tel Aviv, Israel}
\email{borovoi@tauex.tau.ac.il}

\author{Willem de Graaf}
\address{Department of Mathematics, University of Trento, Povo (Trento), Italy}
\email{degraaf@science.unitn.it}

\author{H\^ong V\^an L\^e}
\address{Institute of Mathematics, Czech  Academy of Sciences,
Zitna 25, 11567 Praha 1, Czech Republic}
\email{hvle@math.cas.cz}

\thanks{Borovoi was partially supported
by the Israel Science Foundation (grant 870/16).
Research  of L\^e was supported  by  GA\v CR-project 18-01953J and	 RVO: 67985840.
De Graaf was partially supported by an Australian Research Council
grant, identifier DP190100317.}

\date{\today}

\begin{abstract}
In this paper we classify real trivectors in dimension 9.
The corresponding classification over the field $\C$ of complex numbers was done by Vinberg and Elashvili in 1978.
One of the main tools used for their classification was the construction of the representation
of $\SL(9,\C)$ on the space of complex trivectors of $\C^9$ as a theta-representation
corresponding to a $\Z/3\Z$-grading of the simple complex Lie algebra of type $E_8$.
This divides the trivectors into three groups: nilpotent, semisimple, and mixed trivectors.
Our classification follows the same pattern.
We use Galois cohomology to obtain the classification over $\R$.
For the nilpotent orbits this is in principle rather straightforward
and the main problem is to determine the first Galois cohomology sets
of a long list of centralizers (we compute the centralizers using computer).
For the semisimple and mixed orbits we develop new methods based on Galois cohomology, first and second.
We also consider real theta-representations in general, and derive a number of results
that are useful for the classification of their orbits.
\end{abstract}

\maketitle
\tableofcontents





\section{Introduction}

Let $V$ be an $n$-dimensional vector space over a field
$k$. The group $\GL(V)$ naturally acts on the spaces $\bigwedge^k V$.
The elements of $\bigwedge^2 V$ are called bivectors. The
elements of $\bigwedge^3 V$ are called trivectors. If the ground field is
$\R$ or $\C$, then the orbits of $\GL(V)$
on the space of bivectors can be listed for all $n$ (see \cite{Gurevich1964}, \S 34).
The situation for trivectors is much more complicated and a lot of
effort has gone into finding orbit classifications for particular $n$.
For $n$ up to $5$ it is straightforward to obtain the orbits of
$\GL(V)$ on the space of trivectors (see \cite{Gurevich1964}, \S 35).
The classification for higher $n$ and $k=\C$ goes back to the thesis of Reichel
(\cite{Reichel}) who classified the orbits for $n=6$ and with a few
omissions also for $n=7$. In 1931 Schouten  published
a classification for $n=7$  (\cite{Schouten31}). In 1935
Gurevich obtained a classification for $n=8$  (\cite{Gurevich1935a},
see also \cite{Gurevich1964}, \S 35). In these cases the number of orbits is
finite. This ceases to be the case for $n\geq 9$. Vinberg and Elashvili
classified the orbits of trivectors for $n=9$ under the group
$\SL(V)$ (\cite{VE1978}). In this classification there are several parametrized
families of orbits. The maximum number of parameters of such a family is 4.

More recently classifications have appeared for different fields.
For $n=6$,
Revoy \cite{Revoy} gave a classification for arbitrary field $k$.
For $n=7$,
Westwick \cite{Westwick} classified the trivectors for $k=\R$,
Cohen and Helminck \cite{CH1988} treated the case of a perfect field $k$
of cohomological dimension $\le 1$ (which includes finite fields),
and Noui and Revoy \cite{NR1994} treated the cases of a
perfect field of cohomological dimension $\le 1$ and of a $p$-adic field.
For $n=8$,
Djokovi\'c \cite{Djokovic1983} treated the case $k=\R$,
Noui \cite{Noui1997} treated the case of an algebraically closed field $k$
of arbitrary characteristic,
and Midoune and Noui \cite{MN2013} treated the case of a finite field.
For $n=9$,
Hora and Pudl\'ak \cite{HP2020} treated the case of the finite field $\F_2$ of two elements.
In the present paper we give a classification of trivectors under
the action of $\SL(V)$ for $n=9$ and $k=\R$.

For their classification Vinberg and Elashvili used a particular construction of
the action of $\SL(V)$ on $\bigwedge^3 V$ ($\dim V=9$, $k=\C$). They
considered the complex Lie algebra $\g^\cC$ of type $E_8$ and its
adjoint group $G$ (here $G$ is equal to the automorphism
group of $\g^\cC$). This Lie algebra has a $\Z_3$-grading $\g^\cC = \g_{-1}^\cC
\oplus \g_0^\cC\oplus \g_1^\cC$ such that $\g_0^\cC \cong \ssl(9,\C)$.
Let $G_0$ be the connected algebraic subgroup of $G$ with Lie algebra
$\g_0^\cC$. Since the Lie algebra $\g_0^\cC$ preserves $\g_1^\cC$ when acting on $\g^\cC$,
the same is true for the group $G_0$.
An isomorphism $\psi: \ssl(9,\C)\to \g_0^\cC$ lifts to a surjective morphism of
algebraic groups $\Psi: \SL(9,\C)\to G_0$. This also defines an action of
$\SL(9,\C)$ on $\g_1^\cC$, and it turns out that $\g_1^\cC \cong \bigwedge^3 \C^9$
as $\SL(9,\C)$-modules. This construction pertains to Vinberg's
theory of $\theta$-groups (\cite{Vinberg1975}, \cite{Vinberg1976},
\cite{Vinberg1979}). So this theory can be used to study
the orbits of trivectors when $n=9$. Among the technical tools that this
makes available we mention the following:
\begin{itemize}
\item The elements of $\g_1^\cC$ have a {\em Jordan decomposition}, that is,
  every $x\in \g_1^\cC$ can be uniquely written as $x=s+n$ where $s\in \g_1^\cC$
  is {\em semisimple} (i.e., its $G_0$-orbit is closed), $n\in \g_1^\cC$ is
  {\em nilpotent} (i.e., the closure of its orbit contains 0) and $[s,n]=0$.
  This naturally splits the orbits into three types according to whether the
  elements of the orbit are nilpotent, semisimple, or mixed (that is,
  neither semisimple nor nilpotent). So also the classification problem splits
  into three subproblems. We also remark that an element $y\in\g_1^\cC$ is semisimple
  (respectively nilpotent) if and only if the linear map $\ad y : \g^\cC\to
  \g^\cC$ is semisimple (respectively nilpotent).
\item  Any nilpotent $e\in \g_1^\cC$ lies in a {\em homogeneous $\ssl_2$-triple},
  meaning that there are $h\in \g_0^\cC$, $f\in \g_{-1}^\cC$ with
  $[h,e]=2e$, $[h,f]=-2f$, $[e,f]=h$. The $G_0$-orbits of $e$ and of the
  triple $(h,e,f)$ determine each other.
  Furthermore, $e$ lies in a {\em carrier algebra} and the theory of
  these algebras can be used to classify the nilpotent orbits
  (\cite{Vinberg1979}).
\item  A {\em Cartan subspace} is a maximal subspace $\Cg$  of $\g_1^\cC$
 consisting of commuting semisimple elements.
Any other Cartan subspace is conjugate to $\Cg$ under the action of $G_0$.
 Each semisimple orbit has a
  point in $\Cg$. Furthermore, two elements of $\Cg$ are $G_0$-conjugate
  if and only if they are conjugate under the finite group
  $W=\Nm_{G_0}(\Cg)/\Zm_{G_0}(\Cg)$ called the {\em Weyl
  group of the graded Lie algebra} $\g^\cC$, or also the {\em little Weyl group.}
\item Fix a semisimple element $s\in\g_1^\cC$, let $\Zm_s$ denote the stabilizer
  of $s$ in $\SL(9,\C)$, and let $\z_{\g^\cC}(s)$ denote the centralizer of $s$ in
  $\g^\cC$. Then the grading of $\g^\cC$ induces a grading
  of its Lie subalgebra  $\z_{\g^\cC}(s)$,
  and  the classification of the orbits of mixed elements $s+e$,
  where $e$ is nilpotent, reduces to the classification of the {\em nilpotent}
  $\Zm_s$-orbits in the graded Lie algebra $\z_{\g^\cC}(s)$.
\end{itemize}

Over the real numbers we use a similar construction of the action of $\SL(V)$
on $\bigwedge^3 V$. There is a Chevalley
basis of $\g^\cC$ such that the $\Z_3$-grading is defined over $\R$ (that is,
the spaces $\g_i^\cC$ for $i=-1,0,1$ all have bases whose elements are
$\R$-linear combinations of the given basis of $\g^\cC$). By $\g_i$ for
$i=-1,0,1$ we denote the real spans of these bases. Then $\g_0$ is isomorphic
to $\ssl(9,\R)$. We can define $\psi$ in such a way that it restricts to
an isomorphism $\psi : \ssl(9,\R)\to \g_0$. Then also $\Psi$ is defined over
$\R$ and restricts to a surjective homomorphism $\Psi : \SL(9,\R) \to
G_0(\R)$. In this way $\g_1$ becomes an $\SL(9,\R)$-module isomorphic to
$\bigwedge^3 \R^9$. In order to use this set-up for classifying orbits, we
need to extend a number of constructions for graded Lie algebras from the
complex setting to the real setting. This is done in Section
\ref{sec:gradedLie}.

As in the earlier works on classification of trivectors
\cite{Revoy} and \cite{Djokovic1983},
our main workhorse for classifying the $\SL(9,\R)$-orbits on
$\bigwedge^3 \R^9$ is Galois cohomology. Very briefly this amounts to the
following. Let $\Oo$ be an $\SL(9,\C)$-orbit in $\bigwedge^3 \C^9$. We are
interested in determining the $\SL(9,\R)$-orbits that are contained in
$\Oo\cap\bigwedge^3 \R^9$. It can happen that this intersection is empty
(we say that $\Oo$ has no real points). In that case we just discard
$\Oo$. On the other hand, if $\Oo$ has a real point $x$, then we
consider the stabilizer $\Zm_x$ of $x$ in $\SL(9,\C)$. The
$\SL(9,\R)$-orbits contained in $\Oo$ are in bijection with the first Galois
cohomology set $\Ho^1 \Zm_x:=\Ho^1(\R, \Zm_x)$.
Moreover, from an explicit cocycle representing  a cohomology class in
$\Ho^1 \Zm_x$\hs, we can effectively compute
a representative of the corresponding real orbit.
Section \ref{sec:galcohom} has further details on Galois
cohomology over $\R$  and methods to compute
first Galois cohomology sets in a quite  general  setting.

To classify the nilpotent $\SL(9,\R)$-orbits we execute the above
procedure. It is shown that every nilpotent $\SL(9,\C)$-orbit in
$\bigwedge^3 \C^9$ has a real point (Proposition \ref{prop:nilreal}).
However, instead of a real point $e$ in the orbit, we work with a real
homogeneous
$\ssl_2$-triple $t=(h,e,f)$ containing it. Since nilpotent orbits correspond
bijectively to orbits of homogeneous $\ssl_2$-triples, it suffices to
compute the stabilizer $\Zm_t$ of the triple $t$ and to determine the set
$\Ho^1 \Zm_t$. In Section \ref{sec:compcen} we describe the computational
methods that we have used to determine $\Zm_t$. Section \ref{sec:nilpcen}
has the details of the determination of the Galois cohomology sets.

In \cite{VE1978} an explicit Cartan subspace $\Cg$ of $\g_1^\cC$ is
given. In $\Cg$ seven {\em canonical sets} are defined
with the property that every semisimple orbit has a point in one of the
canonical sets. Moreover, two elements of the same canonical set have
the same stabilizer in $\SL(9,\C)$. Let $\Fm$ be such a canonical set.
Let $\Nm$ be its normalizer in $\SL(9,\C)$ ($\Nm$ is the set of $g\in
\SL(9,\C)$ with $gp\in \Fm$ for all $p\in \Fm$) and let $\Zm$ be its
centralizer in $\SL(9,\C)$ ($\Zm$ is the set of all $g\in \SL(9,\C)$ with
$gp=p$ for all $p\in \Fm$). Write $\Aa=\Nm/\Zm$. In all cases we have that
$\Aa$ is isomorphic to a subgroup of the little Weyl group $W$, with the
isomorphism respecting the actions of $\Aa$ and $W$ on $\Fm$.
We explicitly determine these subgroups. In Section \ref{sec:semsim}
we divide the complex orbits of the elements of $\Fm$ having real points
into several classes that are in bijection with the Galois cohomology set $\Ho^1 \Aa$.
For each class we can explicitly find a real representative
of each orbit in this class.
Let $q$ be such a real representative, and let $\Zm_q$ be its centralizer
in $\SL(9,\C)$. Then the real orbits contained in the $\SL(9,\C)$-orbit of
$q$ are in bijection with $\Ho^1 \Zm_q$. In Section \ref{sec:semsim}
we find that $\Ho^1 \Zm_q$ always happens to be trivial
(contains only one element $1$).

Section \ref{sec:mixed} is devoted to the elements of mixed type.
To classify the orbits consisting of
mixed elements, we fix a real semisimple element $p$ and
consider the problem of listing the $\SL(9,\R)$-orbits of mixed elements
having a representative of the form $p+e$ where $e$ is nilpotent and
$[p,e]=0$. Let $\Zm_p(\R)$ denote the stabilizer of $p$ in $\SL(9,\R)$.
Let $\a=\z_{\g}(p)$ be the centralizer of $p$ in $\g$. Note that the grading
of $\g$ induces a grading on $\a$. Then our problem amounts to classifying the
nilpotent $\Zm_p(\R)$-orbits in $\a_1$. In principle this can be done with
Galois cohomology in the same way as for the nilpotent $\SL(9,\R)$-orbits
in $\g_1$. In this case, however, there is a nasty detail. If the reductive
subalgebra $\a$ is not split over $\R$, then it can happen that certain complex
nilpotent orbits in $\a_1^\cC$ (where $\a^\cC=\z_{\g^\cC}(p)$) do not have
real points. In Section \ref{sec:mixed} we develop a method for checking
whether a complex nilpotent orbit has real points and for finding one in
the affirmative case. Among other things, this method uses calculations
in the {\em second} (abelian) Galois cohomology group.

The main results of this paper are the tables in Section \ref{sec:tables}
containing representatives of the orbits of $\SL(9,\R)$ in $\bigwedge^3 \R^9$.

In order to obtain the results of this paper, essential use of the computer has been made.
In Section \ref{sec:compmeth} we give details of the
computational methods that we used. For the implementation of these
methods we have used the computer algebra system {\sf GAP}4 \cite{gap4}
and its package {\sf SLA}  \cite{sla}. For the computations involving Gr\"obner
bases we have used the computer algebra system {\sf Singular} \cite{DGPS}.

Finally we would like to  add   a few   more  motivations
for the problem of classification   of trivectors of $\R^9$
 and  put our method  in a broader  perspective.
Special  geometries  defined  by   differential  forms   are of   central  significance
in Riemannian geometry and in physics; see
\cite{HL1982},   \cite{Hitchin2000}, \cite{Joyce2007};  see also \cite{LV2020} for a survey.
One  of  important problems  in geometry  defined  by  differential forms  is to  understand  the orbit space
of the standard $\GL (n, \R)$-action on the vector space $\bigwedge^k (\R^n)^*$ of alternating $k$-forms on $\R^n$,
which is  in  a bijection   with the orbit space
of  the standard $\GL (n, \R)$-action on the space of $k$-vectors  $\bigwedge^k \R^{n}$.
For general  $k$ and $n$, this classification problem is intractable.
In the  case of trivectors of $\R^8$ and  $\R^9$,
this classification problem is  tractable  thanks  to  its formulation
as an  equivalent problem  for   graded  semisimple Lie algebras,
discovered  first  by Vinberg  in his  works \cite{Vinberg1975, Vinberg1976}.
Semisimple graded Lie algebras and their  adjoint orbits
 play an important role  in geometry  and supersymmetries; see
 \cite{WG1968}, \cite{Nahm1978}, \cite{Kac1995}, \cite{CS2009}, \cite{Swann1991}.
\bigskip

{\bf Acknowledgements.}
We  are indebted   to Alexander Elashvili for his discussions with us
on his work  with Vinberg \cite{VE1978} and related  problems
in graded   semisimple Lie algebras over years;
it was he who brought  us together  for working on this paper. We  are  grateful  to
Domenico Fiorenza for his helpful comments on an early version  of this paper.

\subsection{Notation  and conventions}

Here we briefly describe some notation and conventions that we use in the
paper.

By $\Z$, $\Q$, $\R$, $\C$ we denote, respectively, the ring of integers
and the fields of rational, real, and complex numbers.
By $i\in \C$ we denote an imaginary unit.

By $\mu_n$ we denote the cyclic group consisting of the $n$-th roots of
unity in $\C$. Occasionally by $\mu_n$ we denote the set of matrices consisting
of scalar matrices with an $n$-th root of unity on the diagonal. From the
context it will be clear what is meant.

If $G$ is a group, then by $\Nm_G(A)$, $\Zm_G(A)$ we denote the normalizer
and centralizer in $G$ of $A$. If $v$ is an element of a $G$-module then
$\Zm_G(v)$ denotes the stabilizer of $v$ in $G$.

If $\g$ is a Lie algebra then $\z_\g(A)$ denotes the centralizer of $A$ in
$\g$.

We write $\Gamma$ for the group $\Gal(\C/\R)=\{1,\gamma\}$, where $\gamma$ is
the complex conjugation. By a $\Gamma$-group we mean a  group $A$ on which
$\Gamma$ acts by automorphisms. Clearly, a $\Gamma$-group is a pair
$(A,\sigma)$,where $A$ is a group and  $\sigma\colon A\to A$ is an
automorphism of $A$ such that $\sigma^2={\rm id}_A$\hs.
If $a\in A$, we write $\upgam a$  for $\sigma(a)$.

For a {\em linear algebraic group over $\R$} (for brevity $\R$-group) we mean a pair $\MM=(M,\sigma)$,
where $M$ is a linear algebraic group over $\C$, and $\sigma$ is an anti-regular involution of $M(\C)$;
see Subsection \ref{s:Galois-coh}.
Then $\Gamma$ naturally acts on $M(\C)$ (namely, the complex conjugation $\gamma$
acts as $\sigma$). In other words, $M(\C)$ naturally is a $\Gamma$-group.

Let $T$ be a diagonalizable algebraic group.
By $\Xf^*(T)=\Hom(T,\GmC)$ we denote the character group of $T$, where $\Hom$
denotes the group of homomorphisms of {\em algebraic $\C$-groups}.
By $\Xf_*(T)=\Hom(\GmC,T)$ we denote the cocharacter group of $T$.

In this paper the spaces $\bigwedge^3 \C^9$ and $\bigwedge^3 \R^9$ play
a major role. We let $e_1,\ldots,e_9$ denote the standard basis of
$\C^9$ (or of $\R^9$). Then by $e_{ijk}$ we denote the basis vector
$e_i\wedge e_j\wedge e_k$ of $\bigwedge^3 \C^9$ (or of $\bigwedge^3 \R^9$).


\section{The tables}\label{sec:tables}

In this section we list the representatives of the orbits of $\SL(9,\R)$ on
$\bigwedge^3 \R^9$. Except for the semisimple orbits we list the representatives
in tables. We also give various centralizers in the real Lie algebra
$\ssl(9,\R)$. These centralizers are always reductive. For the semisimple
subalgebras we use standard notation for the real forms of the complex
simple Lie algebras.
By $\ttt$ we denote a 1-dimensional subalgebra of $\ssl(9,\R)$
spanned by a semisimple  matrix all of whose eigenvalues are real. By $\uuu$ we denote a
1-dimensional subalgebra of $\ssl(9,\R)$ spanned by a semisimple  matrix all of whose
eigenvalues are imaginary.

We recall that by $e_{ijk}$ we denote the basis vector $e_i\wedge e_j\wedge
e_k$, where $e_1,\ldots,e_9$ are the standard basis vectors of $\R^9$.

\subsection{The nilpotent orbits}

Table \ref{tab:orbitreps} contains representatives of the nilpotent orbits of
$\SL(9,\R)$ on
$\bigwedge^3 \R^9$. In the first column we list the number of the complex orbit
as contained in Table 6 of \cite{VE1978}. The second column has the
characteristic of the complex orbit (see Definition \ref{def:charac}  for this concept).
In the third column we give the dimension $d$ of the
complex orbit. The fourth column lists the
rank $d_s$ of a trivector $e$ in the complex orbit, that is,
the minimal dimension of a subspace $U$ of $\C^9$ such that
$e\in\bigwedge^3 U$. Let $t=(e,h,f)$ be a homogeneous $\ssl_2$-triple
containing a representative $e$ of the complex orbit. Let $\Zm_{\SL(9,\C)}(t)$ denote
the stabilizer in $\SL(9,\C)$ of $t$. The fifth and sixth columns have,
respectively, the size of the component group $K$ of $\Zm_{\SL(9,\C)}(t)$ and
a description of the structure of $K$. We omit the latter
when the component group  has order 1 or prime order. The seventh column has representatives
of all $\SL(9,\R)$-orbits contained in the complex orbit. The last column
has a description of the structure of the centralizer in $\ssl(9,\R)$ of a real
homogeneous $\ssl_2$-triple $t=(e,h,f)$ containing the real representative $e$ in the
same line.

\begin{longtable}{|r|l|l|l|l|l|l|l|}
\caption{Real nilpotent orbits}\label{tab:orbitreps}
\endfirsthead
\hline
\endhead
\hline
\endfoot
\endlastfoot

\hline
No. &\quad Char.  & \f{d}  & \f{d_s} &\f{|K|} &\ \ \f{K} & \qquad\quad Representatives & $\mathfrak{z}(\mathfrak{p})$ \\
\hline

\gr
1 & {\tiny 6 6 6 6 6 6 6 12 }& {\tiny 80} & \y9 & \y3   &
  &${\scriptstyle e_{126}+e_{135}-e_{234}+e_{279}+e_{369}+e_{459}+e_{478}+e_{568}}$
& ${\scriptstyle 0}$  \\
2 & {\tiny 6 6 6 0 6 6 6 6 }& {\tiny 79} & \y9 & \y3   &
  &${\scriptstyle e_{126}+e_{145}-e_{234}+e_{279}+e_{369}-e_{378}+e_{478}+e_{568}}$
& ${\scriptstyle 0}$  \\
\gr
3 & {\tiny 6 6 6 0 6 0 6 6 }& {\tiny 78} & \y9 & \y3   &
  &${\scriptstyle e_{126}+e_{145}-e_{235}+e_{279}+e_{369}-e_{378}+e_{478}+e_{568}}$
&  ${\scriptstyle 0}$  \\
4 & {\tiny 6 0 6 0 6 6 0 6 }& {\tiny 78} & \y9 & \y6   & \f{\mu_6}
  &${\scriptstyle e_{127}+e_{145}-e_{234}+e_{279}+e_{369}-e_{378}+e_{478}+e_{568}}$
  &  ${\scriptstyle 0}$ \\
  & & &    & &
  &${\scriptstyle e_{127}-e_{135}-e_{146}+e_{234}-e_{369}+e_{378}+e_{459}+e_{568}}$
& ${\scriptstyle 0}$  \\
\gr
5 & {\tiny 0 6 0 6 0 6 0 6 }& {\tiny 77} & \y9 & \y6   & \f{\mu_6}
  &${\scriptstyle e_{127}+e_{145}-e_{235}+e_{279}+e_{369}-e_{468}+e_{478}+e_{568}}$
&  ${\scriptstyle 0}$  \\
\gr
  & & &    & &
  &${\scriptstyle 2e_{126}+2e_{137}+2e_{235}+2e_{279}-2e_{369}+2e_{459}+2e_{468}+e_{578}}$
& ${\scriptstyle 0}$  \\
6 & {\tiny 6 1 5 1 5 6 1 5 }& {\tiny 77} & \y9 & \y1   &
  &${\scriptstyle e_{136}+e_{145}-e_{234}+e_{279}-e_{378}+e_{469}+e_{568}}$
&  ${\scriptstyle \ttt }$ \\
\gr
7& {\tiny 0 6 0 0 6 0 6 0 }& {\tiny 76} & \y9 & \y6   & \f{\mu_6}
  &${\scriptstyle e_{127}+e_{145}-e_{235}+e_{368}-e_{378}-e_{468}+e_{469}+e_{579}}$
& ${\scriptstyle 0}$  \\
\gr
  & & &    & &
  &${\scriptstyle 2e_{127}+2e_{134}-2e_{245}+2e_{368}+\frac12 e_{379}+e_{479}+e_{569}-e_{578}}$
& ${\scriptstyle 0}$  \\
8 & {\tiny 6 0 6 0 6 0 0 6 }& {\tiny 76} & \y9 & \y6   & \f{\mu_6}
  &${\scriptstyle e_{136}+e_{145}-e_{235}+e_{279}-e_{378}+e_{469}+e_{478}+e_{568}}$
  &   ${\scriptstyle 0}$ \\
  & & &    & &
  &${\scriptstyle -e_{135}+e_{146}-e_{179}+e_{236}-e_{245}+e_{369}+e_{378}+e_{459}+e_{568}}$
& ${\scriptstyle 0}$  \\
\gr
9 & {\tiny 0 0 6 0 0 6 0 6 }& {\tiny 76} & \y9 & \y{18}   & \f{\mu_3\times S_3}
  &${\scriptstyle e_{127}+e_{145}-e_{234}+e_{279}+e_{369}+e_{568}-e_{578}+e_{678}}$
& ${\scriptstyle 0}$    \\
\gr
  & & &    & &
  &${\scriptstyle e_{125}+e_{136}-e_{147}+2e_{234}-2e_{279}-2e_{459}+2e_{568}-2e_{578}}$
&   ${\scriptstyle 0}$  \\
10 & {\tiny 2 2 2 2 2 4 2 2 }& {\tiny 75} & \y9 & \y9   & \f{\mu_3\times \mu_3}
  &${\scriptstyle e_{127}+e_{136}+e_{145}-e_{234}+e_{379}+e_{469}+e_{478}+e_{568}}$
& ${\scriptstyle 0}$  \\
\gr
11 & {\tiny 6 1 5 1 5 0 1 5 }& {\tiny 75} & \y9 & \y1   &
  &${\scriptstyle e_{136}+e_{145}-e_{236}+e_{279}-e_{378}+e_{469}+e_{568}}$
&  ${\scriptstyle \ttt }$ \\
12 & {\tiny 0 1 5 0 1 5 1 5 }& {\tiny 75} & \y9 & \y1   &
  &${\scriptstyle e_{127}+e_{145}-e_{234}+e_{369}+e_{479}+e_{568}-e_{578}}$
&  ${\scriptstyle \ttt }$ \\
\gr
13 & {\tiny 1 5 1 1 4 1 5 1 }& {\tiny 75} & \y9 & \y1   &
  &${\scriptstyle e_{126}+e_{145}-e_{235}+e_{379}+e_{469}+e_{478}+e_{568}}$
&  ${\scriptstyle \ttt }$ \\
14 & {\tiny 2 2 2 2 2 2 2 2 }& {\tiny 74} & \y9 & \y9   & \f{\mu_9}
  &${\scriptstyle e_{127}+e_{136}+e_{145}-e_{235}+e_{379}+e_{469}+e_{478}+e_{568}}$
& ${\scriptstyle 0}$  \\
\gr
15 & {\tiny 6 0 0 0 6 0 0 6 }& {\tiny 74}   & \y9 & \y6   & \f{\mu_6}
  &${\scriptstyle e_{136}+e_{145}-e_{235}+e_{279}+e_{469}+e_{478}-e_{578}}$
  & ${\scriptstyle \ttt }$  \\
\gr
  & & &   & &
  &${\scriptstyle e_{134}-e_{156}-e_{179}+e_{235}-e_{246}-e_{369}+e_{378}+e_{459}}$
  & ${\scriptstyle \uuu }$  \\
\gr
 & & &   & &
 &${\scriptstyle e_{134}+e_{156}-e_{179}+e_{235}-e_{246}+e_{369}+e_{378}+e_{459}}$
& ${\scriptstyle \uuu }$  \\
16  & {\tiny 1 1 4 1 1 5 1 4 }& {\tiny 74} & \y9 & \y3   &
  &${\scriptstyle e_{127}+e_{136}+e_{145}-e_{234}+e_{379}+e_{469}+e_{568}}$
&  ${\scriptstyle \ttt }$ \\
\gr
17 & {\tiny 0 6 0 0 0 6 0 0 }& {\tiny 73} & \y9 & \y{18 } & \f{\mu_3\times S_3}
  &${\scriptstyle e_{137}+e_{146}-e_{245}-e_{268}+e_{278}+e_{368}+e_{479}+e_{569}}$
  &  ${\scriptstyle 0}$ \\
\gr
  & & &    & &
  &${\scriptstyle -e_{126}-e_{147}+e_{279}+2e_{345}-e_{368}-e_{469}+4e_{569}+e_{578}}$
& ${\scriptstyle 0}$  \\
18  & {\tiny 2 2 0 2 2 2 2 2 }& {\tiny 73} & \y9 & \y9   & \f{\mu_3\times\mu_3}
  &${\scriptstyle e_{127}+e_{136}+e_{145}-e_{235}+e_{379}+e_{469}+e_{478}-e_{578}}$
&  ${\scriptstyle 0}$ \\
\gr
19 & {\tiny 6 0 1 0 5 0 1 5 }& {\tiny 73} & \y9 & \y2   &
  &${\scriptstyle e_{136}+e_{145}-e_{236}+e_{279}-e_{378}+e_{478}+e_{569}}$
&  ${\scriptstyle \ttt }$ \\
\gr
  & & &   & &
  &${\scriptstyle 2e_{135}-2e_{146}-e_{236}-e_{245}+e_{279}+2e_{378}+2e_{569}}$
& ${\scriptstyle \ttt }$  \\
20 & {\tiny 3 0 3 3 0 6 0 3 }& {\tiny 73} & \y9 & \y3   &
  &${\scriptstyle e_{127}+e_{145}-e_{234}+e_{379}+e_{469}+e_{478}+e_{568}}$
& ${\scriptstyle \ssl(2,\R)}$  \\
\gr
21 & {\tiny 1 4 1 1 1 3 1 1 }& {\tiny 72} & \y9 & \y1   &
  &${\scriptstyle e_{127}+e_{136}-e_{245}+e_{379}+e_{469}+e_{478}+e_{568}}$
& ${\scriptstyle \ttt }$  \\
22 & {\tiny 1 4 0 1 1 4 1 1 }& {\tiny 72} & \y9 & \y3   &
  &${\scriptstyle e_{127}+e_{136}-e_{235}+e_{379}+e_{469}+e_{478}-e_{578}}$
& ${\scriptstyle \ttt }$  \\
\gr
23 & {\tiny 0 3 0 3 0 3 3 0 }& {\tiny 72} & \y9 & \y6   & \f{\mu_6}
  &${\scriptstyle e_{127}+e_{136}+e_{145}-e_{235}-e_{468}+e_{479}+e_{568}}$
&  ${\scriptstyle \ttt }$ \\
\gr
  & & &    & &
  &${\scriptstyle -2e_{126}+2e_{137}+e_{145}-2e_{234}-e_{468}-e_{479}+2e_{569}-2e_{578}}$
& ${\scriptstyle \uuu }$  \\
24 & {\tiny 0 0 6 0 0 0 0 6 }& {\tiny 72} & \y9 & \y6   & \f{\mu_6}
  &${\scriptstyle e_{127}+e_{136}+e_{145}-e_{235}+e_{379}+e_{469}-e_{578}}$
  &  ${\scriptstyle \ttt }$ \\
  & & &    & &
  &${\scriptstyle -e_{125}-e_{137}+e_{146}-e_{179}+e_{247}+e_{269}+e_{345}+e_{459}+e_{578}}$
  & ${\scriptstyle \uuu }$  \\
  & & &   & &
  &${\scriptstyle -e_{125}+e_{137}+e_{146}+e_{179}+e_{247}+e_{269}-e_{345}-e_{459}+e_{578}}$
  & ${\scriptstyle \uuu }$  \\
\gr
25 & {\tiny 2 0 4 2 0 6 0 4 }& {\tiny 72} & \y9 & \y1   &
  &${\scriptstyle e_{127}+e_{145}-e_{234}+e_{379}+e_{469}+e_{568}}$
&  ${\scriptstyle \ssl(2,\R) + \ttt }$ \\
26 & {\tiny 3 0 3 0 3 0 3 0 }& {\tiny 71} & \y9 & \y6   & \f{\mu_6}
  &${\scriptstyle e_{127}+e_{136}-e_{245}-e_{378}+e_{479}+e_{568}+e_{569}}$
  & ${\scriptstyle \ttt }$  \\
  & & &    & &
  &${\scriptstyle e_{127}-e_{136}-e_{145}+e_{235}-e_{246}+e_{378}-e_{479}-2e_{568}}$
& ${\scriptstyle \uuu }$  \\
\gr
27 & {\tiny 0 1 5 0 0 1 0 5 }& {\tiny 71} & \y9 & \y2   &
  &${\scriptstyle e_{127}+e_{136}-e_{245}+e_{379}+e_{469}+e_{568}-e_{578}}$
& ${\scriptstyle \ttt }$  \\
\gr
  & & &   & &
  &${\scriptstyle e_{126}+e_{136}+e_{147}+e_{279}+e_{345}-e_{379}+e_{469}-e_{578}}$
& ${\scriptstyle \ttt }$  \\
28 & {\tiny 1 1 2 1 1 1 1 4 }& {\tiny 71} & \y9 & \y1   &
  &${\scriptstyle e_{127}+e_{136}+e_{145}-e_{235}+e_{379}+e_{469}+e_{678}}$
& ${\scriptstyle \ttt }$  \\
\gr
29 & {\tiny 0 4 0 2 0 4 2 0 }& {\tiny 71} & \y9 & \y2   &
  &${\scriptstyle e_{127}+e_{136}-e_{235}-e_{468}+e_{479}+e_{568}}$
& ${\scriptstyle 2\ttt }$  \\
\gr
  & & &   & &
  &${\scriptstyle -2e_{126}+2e_{137}+2e_{235}+2e_{468}-2e_{479}-e_{569}-e_{578}}$
& ${\scriptstyle \ttt +\uuu }$  \\
30 & {\tiny 6 1 0 1 4 1 0 5 }& {\tiny 71} & \y9 & \y1   &
  &${\scriptstyle e_{146}-e_{179}-e_{236}-e_{245}-e_{378}+e_{569}}$
& ${\scriptstyle \ssl(2,\R) + \ttt }$  \\
\gr
31 & {\tiny 0 0 0 0 6 0 0 0 }& {\tiny 70} & \y9 &\y{120} & \y{\mu_3\times S_5}
  &${\scriptstyle e_{137}-e_{246}-e_{247}+e_{348}-e_{358}+e_{368}+e_{458}+e_{569}}$
& ${\scriptstyle 0}$  \\
\gr
  & & &    & &
  &${\scriptstyle -e_{134}-2e_{167}+e_{245}-2e_{358}+2e_{368}-2e_{378}+2e_{379}-e_{458}+e_{469}}$
&  ${\scriptstyle 0}$ \\
\gr
  & & &      & &
  &${\scriptstyle -e_{135}-e_{245}-2e_{267}+2e_{368}+e_{378}-2e_{468}-2e_{479}-e_{569}}$
&  ${\scriptstyle 0}$ \\
32 & {\tiny 2 2 2 2 0 2 0 2 }& {\tiny 70} & \y9 & \y3   &
  &${\scriptstyle e_{127}+e_{146}-e_{236}-e_{245}+e_{379}+e_{478}+e_{568}}$
&   ${\scriptstyle \ttt }$ \\
\gr
33 & {\tiny 1 0 5 0 1 0 1 4 }& {\tiny 70} & \y9 & \y2   &
  &${\scriptstyle e_{127}+e_{136}-e_{245}+e_{379}+e_{479}+e_{568}}$
& ${\scriptstyle 2\ttt }$  \\
\gr
  & & &   & &
  &${\scriptstyle -e_{127}+e_{135}-e_{146}-e_{236}-e_{245}+e_{379}+e_{568}}$
& ${\scriptstyle \ttt +\uuu}$  \\
34 & {\tiny 2 0 2 2 0 2 0 4 }& {\tiny 70} & \y9 & \y1   &
  &${\scriptstyle e_{127}+e_{145}-e_{236}+e_{379}+e_{469}-e_{578}}$
& ${\scriptstyle 2\ttt }$  \\
\gr
35 & {\tiny 2 0 2 0 2 0 2 2 }& {\tiny 69} & \y9 & \y{18}   & \f{\mu_6\times\mu_3}
  &${\scriptstyle e_{127}+e_{136}-e_{245}+e_{379}+e_{479}+e_{569}-e_{578}+e_{678}}$
&  ${\scriptstyle 0}$ \\
\gr
  & & &    & &
  &${\scriptstyle -e_{127}+e_{135}+e_{146}-e_{236}+e_{245}-e_{379}+e_{569}-e_{578}}$
&  ${\scriptstyle 0}$ \\
36 & {\tiny 0 0 1 0 5 0 0 1 }& {\tiny 69} & \y9 &  \y{6}   & \f{S_3}
  &${\scriptstyle e_{136}-e_{245}+e_{379}+e_{479}+e_{568}-e_{578}+e_{678}}$
  & ${\scriptstyle \ttt}$  \\
  & & &    & &
  &${\scriptstyle -e_{135}+e_{147}+e_{236}+e_{379}+e_{459}-e_{578}+e_{678}}$
& ${\scriptstyle \ttt}$  \\
\gr
37  & {\tiny 1 1 4 1 0 1 0 4 }& {\tiny 69} & \y9 & \y1   &
  &${\scriptstyle e_{127}+e_{146}-e_{236}-e_{245}+e_{379}+e_{568}}$
&  ${\scriptstyle 2\ttt }$ \\
38 & {\tiny 2 1 1 1 1 1 1 2 }& {\tiny 68} & \y9 & \y3   &
  &${\scriptstyle e_{127}+e_{146}-e_{236}-e_{245}+e_{379}+e_{569}-e_{578}}$
&  ${\scriptstyle \ttt }$ \\
\gr
39 & {\tiny 1 0 1 0 4 0 1 1 }& {\tiny 68} & \y9 & \y6   & \f{\mu_6}
  &${\scriptstyle e_{136}-e_{245}+e_{379}+e_{479}+e_{569}-e_{578}+e_{678}}$
& ${\scriptstyle \ttt }$  \\
\gr
  & & &    & &
  &${\scriptstyle e_{135}-e_{146}+e_{236}+e_{245}+e_{479}-e_{569}+e_{578}}$
&  ${\scriptstyle \ttt }$ \\
40  & {\tiny 0 1 0 1 4 0 1 0 }& {\tiny 68} & \y9 & \y2   &
  &${\scriptstyle e_{137}-e_{236}-e_{245}-e_{468}+e_{478}+e_{569}}$
  &  ${\scriptstyle 2\ttt}$ \\
  & & &   & &
  &${\scriptstyle -e_{136}+e_{237}-2e_{245}-e_{468}+2e_{479}-e_{569}-2e_{578}}$
&   ${\scriptstyle \ttt + \uuu}$\\
\gr
41 & {\tiny 1 0 1 1 2 1 1 1 }& {\tiny 67} & \y9 & \y3   &
  &${\scriptstyle e_{137}+e_{145}-e_{236}+e_{479}+e_{569}-e_{578}+e_{678}}$
& ${\scriptstyle \ttt }$  \\
42 & {\tiny 0 3 0 0 0 3 0 3 }& {\tiny 67} & \y9 & \y3   &
  &${\scriptstyle e_{127}+e_{136}-e_{245}+e_{379}+e_{479}+e_{569}+e_{678}}$
&  ${\scriptstyle \ssl(2,\R)}$ \\
\gr
43 & {\tiny 3 0 0 0 3 0 3 0 }& {\tiny 67} & \y9 & \y6   & \f{\mu_6}
  &${\scriptstyle e_{127}+e_{136}-e_{245}-e_{378}+e_{478}+e_{579}+e_{679}}$
&  ${\scriptstyle \ssl(2,\R)}$ \\
\gr
  & & &    & &
  &   ${\scriptstyle e_{127}-e_{134}-e_{156}-e_{236}-e_{245}+e_{578}-e_{679}}$
& ${\scriptstyle \ssl(2,\R)}$  \\
44 & {\tiny 1 1 0 1 3 1 0 1 }& {\tiny 67} &  \y9 & \y1   &
  &${\scriptstyle e_{137}+e_{146}-e_{236}-e_{245}+e_{478}+e_{569}}$
& ${\scriptstyle 2\ttt}$  \\
\gr
45 & {\tiny 3 0 3 3 0 0 0 3 }& {\tiny 67} & \y9 & \y3   &
  &${\scriptstyle e_{127}+e_{146}-e_{245}+e_{379}+e_{478}+e_{568}}$
& ${\scriptstyle 2\hs\ssl(2,\R) }$  \\
46 & {\tiny 1 1 1 1 1 1 1 1 }& {\tiny 66} & \y9 & \y1   &
  &${\scriptstyle e_{137}+e_{146}-e_{236}-e_{245}+e_{479}+e_{569}-e_{578}}$
&  ${\scriptstyle \ttt}$ \\
\gr
47 & {\tiny 0 2 0 0 2 2 0 2 }& {\tiny 66} & \y9 & \y2   &
  &${\scriptstyle e_{136}+e_{147}-e_{245}+e_{379}+e_{569}+e_{678}}$
  & ${\scriptstyle 2\ttt }$  \\
\gr
  & & &   & &
  &${\scriptstyle -e_{136}-e_{147}+e_{157}-e_{235}-e_{379}-e_{469}-e_{569}-e_{678}}$
& ${\scriptstyle \ttt +\uuu }$  \\
48 & {\tiny 2 0 0 0 4 0 2 0 }& {\tiny 66} & \y9 & \y2   &
  &${\scriptstyle e_{136}-e_{245}-e_{378}+e_{478}+e_{579}+e_{679}}$
  & ${\scriptstyle \ssl(2,\R) + \ttt }$  \\
  & & &   & &
  &${\scriptstyle e_{134}-e_{156}-e_{236}+e_{245}+e_{578}-e_{679}}$
& ${\scriptstyle \ssl(2,\R) + \ttt }$  \\
\gr
49 & {\tiny 3 0 1 0 2 1 2 0 }& {\tiny 66} & \y9 & \y1   &
  &${\scriptstyle e_{127}+e_{156}-e_{236}-e_{245}-e_{378}+e_{479}}$
&  ${\scriptstyle \ssl(2,\R) + \ttt }$ \\
50 & {\tiny 2 0 4 2 0 0 0 4 }& {\tiny 66} & \y9 & \y1   &
  &${\scriptstyle e_{127}+e_{146}-e_{245}+e_{379}+e_{568}}$
& ${\scriptstyle 2\hs\ssl(2,\R) + \ttt }$  \\
\gr
51 & {\tiny 1 1 1 1 0 1 1 2 }& {\tiny 65} & \y9 & \y3   &
  &${\scriptstyle e_{137}+e_{146}-e_{236}-e_{245}+e_{479}+e_{569}+e_{678}}$
&  ${\scriptstyle \ttt}$ \\
52 & {\tiny 1 1 0 1 1 2 1 1 }& {\tiny 65} & \y9 & \y1   &
  &${\scriptstyle e_{137}+e_{146}-e_{235}+e_{479}+e_{579}+e_{678}}$
& ${\scriptstyle 2\ttt }$  \\
\gr
53 & {\tiny 2 0 1 0 3 1 1 0 }& {\tiny 65} & \y9 & \y1   &
  &${\scriptstyle e_{156}-e_{236}-e_{245}-e_{378}+e_{479}}$
& ${\scriptstyle \ssl(2,\R) + 2\ttt}$  \\
54 & {\tiny 0 3 0 0 3 0 0 0 }& {\tiny 64} & \y9 & \y3 &
  &${\scriptstyle e_{147}+e_{156}-e_{237}-e_{246}-e_{345}+e_{368}+e_{579}}$
&  ${\scriptstyle \ssl(2,\R)}$ \\
\gr
55 & {\tiny 2 0 2 0 0 2 0 2 }& {\tiny 64} & \y9 & \y1 &
  &${\scriptstyle e_{137}+e_{156}-e_{236}-e_{245}+e_{479}-e_{578}}$
&  ${\scriptstyle 2\ttt}$ \\
56 & {\tiny 0 0 0 3 0 3 0 0 }& {\tiny 64} & \y9 & \y3  &
  &${\scriptstyle e_{146}+e_{157}-e_{237}+e_{458}+e_{478}+e_{569}}$
  &  ${\scriptstyle 2\hs\ssl(2,\R)}$ \\
  & & &  & &
  &${\scriptstyle e_{145}+2e_{167}-e_{235}-2e_{469}+2e_{478}+e_{568}+e_{579}}$
  & ${\scriptstyle \ssl(2,\R)+\su(2)}$  \\
  \gr
57 & {\tiny 0 0 0 0 0 0 0 6 }& {\tiny 64} & \y8 &\y6 & \f{\mu_6}
  &${\scriptstyle e_{127}+e_{136}-e_{245}+e_{379}+e_{479}+e_{569}}$
  & ${\scriptstyle \ssl(3,\R)}$  \\
  \gr
  & & &    & &
  &${\scriptstyle e_{123}-e_{146}+e_{179}+2e_{247}+2e_{259}+2e_{357}+2e_{369}}$
  & ${\scriptstyle \su(1,2)}$  \\
  \gr
  & & &    & &
  &${\scriptstyle e_{123}-e_{156}+e_{179}+e_{247}+e_{269}+e_{349}+e_{359}+e_{367}}$
  &  ${\scriptstyle \su(3)}$ \\
58 & {\tiny 1 1 1 1 1 0 1 1 }& {\tiny 63} & \y9 &\y1 &
  &${\scriptstyle e_{137}-e_{246}-e_{345}+e_{479}+e_{569}-e_{578}}$
& ${\scriptstyle 2\ttt}$  \\
\gr
59 & {\tiny 0 4 0 0 2 0 0 0 }& {\tiny 63} & \y9 & \y1 &
  &${\scriptstyle e_{147}+e_{156}-e_{237}-e_{246}+e_{368}+e_{579}}$
&  ${\scriptstyle \ssl(2,\R) + \ttt}$ \\
60 & {\tiny 0 0 3 0 0 0 3 0 }& {\tiny 63} & \y9 & \y{6} &  \f{\mu_6}
  &${\scriptstyle e_{137}+e_{146}-e_{236}-e_{245}+e_{568}+e_{679}}$
  & ${\scriptstyle \ssl(2,\R) + \ttt}$  \\
  & & &     & &
  &${\scriptstyle -e_{136}-e_{145}-e_{147}+e_{235}-e_{237}+e_{246}-e_{568}-e_{579}}$
  &  ${\scriptstyle \ssl(2,\R) + \uuu}$ \\
  & & &    & &
  &${\scriptstyle e_{127}+e_{136}-e_{145}+e_{235}+e_{246}-e_{347}-e_{568}+e_{579}}$
  &  ${\scriptstyle \su(2) + \uuu}$ \\
\gr
61 & {\tiny 0 0 0 0 1 0 0 5 }& {\tiny 63} &\y8 & \y1 &
  &${\scriptstyle e_{137}+e_{146}-e_{236}-e_{245}+e_{479}+e_{569}}$
&  ${\scriptstyle \ssl(2,\R)+\ttt}$ \\
\gr
  & & &    & &
  &${\scriptstyle -2e_{134}-e_{145}+e_{167}-e_{246}+\frac{1}{2}e_{257}-2e_{369}+e_{479}-e_{569}}$
&  ${\scriptstyle \su(2)+\ttt}$ \\
62 & {\tiny 0 1 0 2 0 3 0 1 }& {\tiny 63} & \y9 &  \y2   &
  &${\scriptstyle e_{146}-e_{235}+e_{479}+e_{579}+e_{678}}$
  & ${\scriptstyle \ssl(2,\R) + 2\ttt}$  \\
  & & &   & &
  &${\scriptstyle -e_{146}+2e_{157}+e_{234}-e_{479}-2e_{569}+2e_{678}}$
& ${\scriptstyle \ssl(2,\R) + \ttt + \uuu}$  \\
\gr
63 & {\tiny 1 1 1 0 1 0 2 1 }& {\tiny 62} & \y9 & \y1   &
  &${\scriptstyle e_{127}+e_{146}-e_{236}-e_{345}+e_{579}+e_{678}}$
& ${\scriptstyle \ssl(2,\R) + \ttt}$  \\
64 & {\tiny 0 1 2 0 1 0 2 0 }& {\tiny 62} & \y9 &\y2  &
  &${\scriptstyle e_{137}-e_{246}-e_{345}+e_{568}+e_{579}}$
  &  ${\scriptstyle 3\ttt}$ \\
  & & &    & &
  &${\scriptstyle -e_{136}-e_{147}+e_{237}-e_{246}+e_{345}+e_{568}+e_{579}}$
  & ${\scriptstyle \ttt+2\uuu}$  \\
\gr
65 & {\tiny 0 2 0 0 2 0 0 2 }& {\tiny 61} & \y9 &\y{18} & \f{\mu_3\times S_3}
  &${\scriptstyle e_{137}-e_{246}-e_{247}-e_{345}+e_{569}+e_{678}}$
& ${\scriptstyle 2\ttt}$  \\
\gr
  & & &  & &
  &${\scriptstyle -e_{137}+e_{247}-e_{256}-e_{345}+e_{469}+e_{579}+e_{678}}$
&  ${\scriptstyle \ttt+\uuu}$ \\
66 & {\tiny 1 2 1 1 0 0 1 1 }& {\tiny 61} & \y9 & \y1  &
  &${\scriptstyle e_{137}-e_{246}+e_{479}+e_{569}-e_{578}}$
&  ${\scriptstyle \ssl(2,\R) + 2\ttt}$ \\
\gr
67 & {\tiny 0 0 1 0 0 0 1 4 }& {\tiny 61} &\y8 & \y2 &
  &${\scriptstyle e_{137}+e_{146}-e_{236}-e_{245}+e_{579}}$
  &  ${\scriptstyle \ssl(2,\R) + 2\ttt}$ \\
  \gr
  & & &    & &
  &${\scriptstyle -e_{125}-e_{126}+e_{147}-e_{237}-e_{345}+e_{346}+e_{679}}$
  & ${\scriptstyle \ssl(2,\R) +  \ttt+\uuu}$  \\
  \gr
  & & &    & &
  &${\scriptstyle -e_{127}-e_{136}+e_{145}+e_{235}+e_{246}-e_{347}-e_{679}}$
  &   ${\scriptstyle \su(2) + \ttt+\uuu}$\\
68 & {\tiny 1 0 1 1 0 3 1 0 }& {\tiny 61} & \y9 & \y1   &
  &${\scriptstyle e_{147}+e_{156}-e_{234}-e_{578}+e_{679}}$
& ${\scriptstyle 2\hs\ssl(2,\R) + \ttt}$  \\
\gr
69  & {\tiny 1 1 0 1 1 0 1 1 }& {\tiny 60} & \y9 & \y1 &
  &${\scriptstyle e_{156}-e_{237}-e_{246}-e_{345}+e_{479}+e_{678}}$
& ${\scriptstyle 2\ttt}$  \\
70 & {\tiny 0 1 0 0 1 0 0 4 }& {\tiny 60} &\y8 & \y{6} & \y{S_3}
  &${\scriptstyle e_{137}-e_{246}-e_{247}-e_{345}+e_{569}}$
  &  ${\scriptstyle 3\ttt}$ \\
  & & &    & &
  &${\scriptstyle -e_{136}-e_{147}-e_{257}+e_{345}+e_{379}-e_{469}}$
& ${\scriptstyle 2\ttt +\uuu}$  \\
\gr
71 & {\tiny 0 2 2 0 0 0 2 0 }& {\tiny 59} & \y9 & \y2 &
  &${\scriptstyle e_{137}-e_{246}+e_{568}+e_{579}}$
& ${\scriptstyle 2\hs\ssl(2,\R) + 2 \ttt}$  \\
\gr
  & & &    & &
  &${\scriptstyle -e_{126}-e_{147}-e_{237}-e_{346}+e_{568}+e_{579}}$
& ${\scriptstyle \ssl(2,\C) +  \ttt+\uuu}$  \\
72 & {\tiny 1 0 1 1 0 1 1 0 }& {\tiny 58} & \y9 & \y9 & \f{\mu_3\times \mu_3}
  &${\scriptstyle e_{147}+e_{156}-e_{237}-e_{246}-e_{345}-e_{578}+e_{679}}$
&  ${\scriptstyle \ssl(2,\R)}$ \\
\gr
73 & {\tiny 2 0 1 0 1 1 0 1 }& {\tiny 58} & \y9 & \y1 &
  &${\scriptstyle e_{137}-e_{256}-e_{346}+e_{479}-e_{578}}$
&  ${\scriptstyle \ssl(2,\R) + 2 \ttt}$ \\
74 & {\tiny 1 0 0 1 0 0 1 3 }& {\tiny 58} &\y8 & \y1   &
  &${\scriptstyle e_{156}-e_{237}-e_{246}-e_{345}+e_{479}}$
&  ${\scriptstyle \ssl(2,\R) + 2\ttt}$ \\
\gr
75 & {\tiny 0 3 0 0 0 0 0 3 }& {\tiny 57} & \y9 & \y{18} &\f{S_3\times\mu_3}
  &${\scriptstyle e_{137}-e_{246}-e_{247}+e_{569}+e_{678}}$
  &  ${\scriptstyle 3\hs\ssl(2,\R) }$ \\
\gr
  & & &  & &
  &${\scriptstyle -e_{126}-e_{346}-e_{379}+e_{457}+e_{569}+e_{678}}$
&  ${\scriptstyle \ssl(2,\R)+\ssl(2,\C)}$ \\
76 & {\tiny 0 1 0 0 1 0 1 2 }& {\tiny 56} & \y8 & \y3   &
  &${\scriptstyle e_{147}+e_{156}-e_{237}-e_{246}-e_{345}+e_{679}}$
&  ${\scriptstyle \ssl(2,\R) + \ttt}$ \\
\gr
77 & {\tiny 1 2 0 0 0 1 0 2 }& {\tiny 56} & \y9 & \y1 &
  &${\scriptstyle e_{137}-e_{246}-e_{356}+e_{579}+e_{678}}$
& ${\scriptstyle 2\hs\ssl(2,\R)+ \ttt}$  \\
78 &  {\tiny 0 2 0 0 0 0 0 4 }& {\tiny 56} & \y8 & \y{6}   & \f{S_3}
  &${\scriptstyle e_{137}-e_{246}-e_{247}+e_{569}}$
  &  ${\scriptstyle 3\hs\ssl(2,\R) + \ttt}$ \\
  & & &  & &
  &${\scriptstyle -e_{126}+e_{179}-e_{257}+e_{346}+e_{569}}$
& ${\scriptstyle \ssl(2,\R)+\ssl(2,\C) + \ttt}$  \\
\gr
79  & {\tiny 0 0 2 0 0 4 0 0 }& {\tiny 56} & \y9 & \y1 &
  &${\scriptstyle e_{157}-e_{234}+e_{568}+e_{679}}$
&  ${\scriptstyle 2\hs\ssl(3,\R)}$ \\
80 & {\tiny 1 1 0 1 0 1 0 1 }& {\tiny 55} & \y9 & \y3 &
  &${\scriptstyle e_{147}-e_{237}-e_{256}-e_{346}+e_{579}+e_{678}}$
&  ${\scriptstyle \ssl(2,\R) + \ttt}$ \\
\gr
81 & {\tiny 0 0 2 0 0 2 0 0 }& {\tiny 55} & \y9 & \y3   &
  &${\scriptstyle e_{157}-e_{237}-e_{246}-e_{345}+e_{568}+e_{679}}$
& ${\scriptstyle \ssl(3,\R)}$  \\
82 & {\tiny 1 1 0 0 0 1 0 3 }& {\tiny 55} & \y8 & \y1    &
  &${\scriptstyle e_{137}-e_{246}-e_{356}+e_{579}}$
&  ${\scriptstyle 2\hs\ssl(2,\R) + 2\ttt}$ \\
\gr
83 & {\tiny 0 0 0 3 0 0 0 0 }& {\tiny 54} & \y9 & \y3 &
  &${\scriptstyle e_{157}-e_{247}-e_{256}-e_{346}+e_{458}+e_{679}}$
& ${\scriptstyle \sp(4,\R)}$  \\
84 & {\tiny 1 0 0 1 0 1 0 2 }& {\tiny 54} & \y8 & \y1    &
  &${\scriptstyle e_{147}-e_{237}-e_{256}-e_{346}+e_{579}}$
& ${\scriptstyle \ssl(2,\R) + 2\ttt}$  \\
\gr
85 & {\tiny 0 1 0 2 0 0 0 1 }& {\tiny 53} & \y9 & \y1 &
  &${\scriptstyle e_{147}-e_{256}-e_{346}+e_{579}+e_{678}}$
& ${\scriptstyle 2\hs\ssl(2,\R) + \ttt}$  \\
86 & {\tiny 0 0 0 2 0 0 0 2 }& {\tiny 52} & \y8 & \y2    &
  &${\scriptstyle e_{147}-e_{256}-e_{346}+e_{579}}$
  &  ${\scriptstyle 2\hs\ssl(2,\R)+ 2\ttt}$ \\
  & & &    & &
 &${\scriptstyle -e_{145}-e_{167}+e_{257}-e_{346}-e_{479}-e_{569}}$
& ${\scriptstyle \ssl(2,\C) +\ttt+\uuu}$  \\
\gr
87 & {\tiny 2 1 0 1 0 0 0 2 }& {\tiny 52} & \y9 & \y1 &
  &${\scriptstyle e_{127}+e_{379}-e_{456}+e_{678}}$
&  ${\scriptstyle \ssl(2,\R)+\sp(4,\R) + \ttt}$ \\
88 & {\tiny 0 1 0 1 0 0 1 1 }& {\tiny 51} & \y8 & \y1    &
  &${\scriptstyle e_{157}-e_{247}-e_{256}-e_{346}+e_{679}}$
&  ${\scriptstyle \ssl(2,\R) + 2\ttt}$  \\
\gr
89 & {\tiny 2 0 0 1 0 1 0 0 }& {\tiny 49} & \y9 & \y3 &
  &${\scriptstyle e_{157}-e_{237}-e_{456}+e_{478}+e_{679}}$
&  ${\scriptstyle \ssl(2,\R)+\ssl(3,\R)}$ \\
90 & {\tiny 0 0 0 0 0 0 3 0 }& {\tiny 49} & \y7 & \y3    &
  &${\scriptstyle e_{147}+e_{156}-e_{237}-e_{246}-e_{345}}$
  & ${\scriptstyle \ssl(2,\R)+{\sf G}_2^{\rm spl}}$   \\
  & & &   & &
 &${\scriptstyle -e_{123}-e_{145}-e_{167}+e_{246}-e_{257}-e_{347}-e_{356}}$
& ${\scriptstyle \ssl(2,\R)+{\sf G}_2^{\mathrm{c}}}$  \\
\gr
91 & {\tiny 2 0 0 1 0 0 0 3 }& {\tiny 49} & \y8 &  \y1    &
  &${\scriptstyle e_{127}+e_{379}-e_{456}}$
&  ${\scriptstyle \ssl(3,\R)+\sp(4,\R) + \ttt}$ \\
92 & {\tiny 1 0 1 0 0 1 0 1 }& {\tiny 48} & \y8 &  \y1   &
  &${\scriptstyle e_{157}-e_{247}-e_{356}+e_{679}}$
&  ${\scriptstyle 2\hs\ssl(2,\R)+ 2\ttt}$ \\
\gr
93 & {\tiny 0 0 0 1 0 0 2 0 }& {\tiny 48} & \y7 &  \y1   &
  &${\scriptstyle e_{157}-e_{247}-e_{256}-e_{346}}$
  & ${\scriptstyle 3\hs\ssl(2,\R) + \ttt}$  \\
\gr
  & & &    & &
  &${\scriptstyle -e_{145}+e_{167}-4e_{246}-e_{257}+e_{347}-e_{356}}$
& ${\scriptstyle \ssl(2,\R)+2\hs\su(2) + \ttt}$  \\
94 & {\tiny 0 1 0 0 0 1 1 0 }& {\tiny 45} & \y7 & \y2    &
  &${\scriptstyle e_{167}-e_{247}-e_{356}}$
  &  ${\scriptstyle 3\hs \ssl(2,\R)+ 2\ttt}$ \\
  & & &   & &
  &${\scriptstyle e_{167}+e_{236}+e_{257}-e_{347}-e_{456}}$
&  ${\scriptstyle \ssl(2,\R)+\ssl(2,\C) + \ttt + \uuu }$ \\
\gr
95 & {\tiny 1 0 0 1 0 0 1 0 }& {\tiny 42} & \y7 & \y1    &
  &${\scriptstyle e_{167}-e_{257}-e_{347}-e_{456}}$
  &  ${\scriptstyle \ssl(2,\R)+\ssl(3,\R) + \ttt}$ \\
96 & {\tiny 0 0 0 0 0 2 0 0 }& {\tiny 38} & \y6 & \y2   &
  &${\scriptstyle -e_{247}-e_{356}}$
  & ${\scriptstyle  3\hs\ssl(3,\R)}$  \\
  & & &    & &
  &${\scriptstyle e_{234}+e_{267}+e_{357}+e_{456}}$
& ${\scriptstyle \ssl(3,\R)+\ssl(3,\C)}$  \\
\gr
97 & {\tiny 0 0 1 0 0 1 0 0 }& {\tiny 37} & \y6 & \y1 &
  &${\scriptstyle -e_{267}-e_{357}-e_{456}}$
&  ${\scriptstyle  2\hs\ssl(3,\R)+ \ttt}$ \\
98 & {\tiny 3 0 0 0 0 0 0 0 }& {\tiny 36} & \y9 & \y3     &
  &${\scriptstyle e_{167}-e_{257}-e_{347}-e_{789}}$
&  ${\scriptstyle \sp(8,\R)}$ \\
\gr
99 &  {\tiny 2 0 0 0 0 0 1 0 }& {\tiny 35} & \y7 & \y1    &
  &${\scriptstyle e_{167}-e_{257}-e_{347}}$
&  ${\scriptstyle \ssl(2,\R)+\sp(6,\R) + \ttt}$ \\
100 & {\tiny 1 0 0 0 1 0 0 0 }& {\tiny 30} & \y5 &\y1    &
  &${\scriptstyle -e_{367}-e_{457}}$
& ${\scriptstyle \ssl(4,\R)+\sp(4,\R) + \ttt}$ \\
\gr
101 & {\tiny 0 0 1 0 0 0 0 0 }& {\tiny 19} & \y3 & \y1     &
  &${\scriptstyle -e_{567}}$
  & ${\scriptstyle \ssl(3,\R)+\ssl(6,\R)}$ \\
\hline
\end{longtable}

\subsection{The semisimple orbits}

In \cite{VE1978} it is shown that

\begin{align*}
  p_1 &= e_{123}+e_{456}+e_{789}\\
  p_2 &= e_{147}+e_{258}+e_{369}\\
  p_3 &= e_{159}+e_{267}+e_{348}\\
  p_4 &= e_{168}+e_{249}+e_{357}
\end{align*}
span a Cartan subspace in $\g_1^\cC$ denoted $\Cg$. Each complex semisimple
orbit has a point in this space. This can still be refined a bit: in
\cite{VE1978} seven {\em canonical sets} $\Fm_1,\ldots,\Fm_7$ in $\Cg$ are
described with the property that each complex semisimple orbit has a point
in one of the $\Fm_i$. The real semisimple orbits are divided into two
groups, the canonical and the noncanonical semisimple elements. We say that a
real semisimple element is canonical if it lies in one of the $\Fm_i$ (or more
generally, if it is $\SL(9,\R)$-conjugate to an element of one of the
$\Fm_i$). The noncanonical elements are those that are $\SL(9,\C)$-conjugate
(but not $\SL(9,\R)$-conjugate) to elements of an $\Fm_i$.

Below we give the representatives of the orbits of canonical and noncanonical
semisimple elements. The representatives of the noncanonical orbits do
not lie in $\Cg$. However, because over $\R$ we have that $\g_1$ has
a unique Cartan subspace (Theorem \ref{thm:cartanr}), they are
$\SL(9,\R)$-conjugate to
elements of $\Cg$. In each case we also give these elements. For computing
the mixed elements we have preferred to work with the representatives
that lie outside of $\Cg$ because with the representatives in $\Cg$ the
nilpotent parts tend to become a bit bulky (see Section \ref{sec:mixedtables}).

Finally, for each parametrized class we give a finite matrix group that
determines when different elements of the same class are conjugate.
More precisely, consider a class of semisimple elements $p(\lambda_1,\ldots,
\lambda_m)$, where $\lambda_i\in \R$ must satisfy certain polynomial conditions.
Write $\lambda$ for the column vector consisting of the $\lambda_i$. Let $\mu$
be a second column vector with coordinates $\mu_1,\ldots,\mu_m$.
Then we describe a finite matrix group $\mathcal{G}$ with the property that
$p(\lambda_1,\ldots,\lambda_m)$ and $p(\mu_1,\ldots,\mu_m)$ are
$\SL(9,\R)$-conjugate if and only if there is a $g\in \mathcal{G}$ with
$g\cdot \lambda = \mu$. We express this by saying that the conjugacy of
the elements $p(\lambda_1,\ldots,\lambda_m)$ is determined by $\mathcal{G}$.

{\em Canonical set $\Fm_1$:} The real elements in $\Fm_1$ consist of
$$p^{1,1}_{\lambda_1,\lambda_2,\lambda_3,\lambda_4} = \lambda_1 p_1+\lambda_2p_2
+\lambda_3p_3+\lambda_4p_4$$
where the $\lambda_i\in \R$ are such that
\begin{align*}
\lambda_1\lambda_2\lambda_3\lambda_4 &\neq 0\\
(\lambda_2^3+\lambda_3^3+\lambda_4^3)^3 -(3\lambda_2\lambda_3\lambda_4)^3 &\neq 0\\
(\lambda_1^3+\lambda_3^3+\lambda_4^3)^3 +(3\lambda_1\lambda_3\lambda_4)^3 &\neq 0\\
(\lambda_1^3+\lambda_2^3+\lambda_4^3)^3 +(3\lambda_1\lambda_2\lambda_4)^3 &\neq 0\\
(\lambda_1^3+\lambda_2^3+\lambda_3^3)^3 +(3\lambda_1\lambda_2\lambda_3)^3 &\neq 0.\\
\end{align*}

There are no noncanonical real semisimple elements that are
$\SL(9,\C)$-conjugate to elements of $\Fm_1$.

The conjugacy of the elements $p^{1,1}_{\lambda_1,\lambda_2,\lambda_3,\lambda_4}$ is
determined by a group of order 48, isomorphic to $\GL(2,\mathbb{F}_3)$ and
generated by
$$\SmallMatrix{0&0&0&-1\\0&-1&0&0\\0&0&1&0\\-1&0&0&0},
\SmallMatrix{0&0&-1&0\\1&0&0&0\\0&-1&0&0\\0&0&0&1}.$$

{\em Canonical set $\Fm_2$:} The real elements in $\Fm_2$ consist of
$$p^{2,1}_{\lambda_1,\lambda_2,\lambda_3} = \lambda_1 p_1+\lambda_2p_2
-\lambda_3p_3$$
where the $\lambda_i\in \R$ are such that
$$\lambda_1\lambda_2\lambda_3(\lambda_1^3-\lambda_2^3)(\lambda_2^3-\lambda_3^3)
(\lambda_3^3-\lambda_1^3)[(\lambda_1^3+\lambda_2^3+\lambda_3^3)^3-
  (3\lambda_1\lambda_2\lambda_3)^3]\neq 0.$$
Here we have the noncanonical semisimple elements
\begin{align*}
p^{2,2}_{\lambda_1,\lambda_2,\lambda_3}= &\lambda_1(-\tfrac{1}{2}e_{126}-\tfrac{1}{2}e_{349}+2e_{358}-e_{457}+e_{789})+\lambda_2(
    -2e_{137}-\tfrac{1}{4}e_{249}-e_{258}-\tfrac{1}{2}e_{456}-\tfrac{1}{2}e_{689})\\
    &-\lambda_3(-e_{159}-e_{238}-\tfrac{1}{2}e_{247}-\tfrac{1}{2}e_{346}-e_{678}),\quad \lambda_1,\lambda_2,\lambda_3\in \R.
\end{align*}
We have that $p^{2,2}_{\lambda_1,\lambda_2,\lambda_3}$ is $\SL(9,\C)$-conjugate to the
{\em purely imaginary} trivector $p^{2,1}_{i\lambda_1,i\lambda_2,i\lambda_3}$.
So here the $\lambda_i\in \R$ are required to satisfy
$$\lambda_1\lambda_2\lambda_3(\lambda_1^3-\lambda_2^3)(\lambda_2^3-\lambda_3^3)
(\lambda_3^3-\lambda_1^3)[(\lambda_1^3+\lambda_2^3+\lambda_3^3)^3+
  (3\lambda_1\lambda_2\lambda_3)^3]\neq 0.$$

We have that $p^{2,2}_{\lambda_1,\lambda_2,\lambda_3}$ is $\SL(9,\R)$-conjugate to
$$\frac{1}{3}\sqrt{3}((\lambda_2-\lambda_3)p_1+(\lambda_1-\lambda_3)p_2
+(\lambda_1+\lambda_2+\lambda_3)p_3+(\lambda_1-\lambda_2)p_4)$$
which lies in $\Cg$.

The conjugacy of the elements $p^{2,1}_{\lambda_1,\lambda_2,\lambda_3}$ is
determined by a group isomorphic to the dihedral group of order 12 and
generated by
$$\SmallMatrix{0&1&0\\0&0&1\\1&0&0},
\SmallMatrix{-1&0&0\\0&0&-1\\0&-1&0}.$$
The conjugacy of the elements $p^{2,1}_{\lambda_1,\lambda_2,\lambda_3}$ is
determined by exactly the same group.

{\em Canonical set $\Fm_3$:} The real elements in $\Fm_3$ consist of
$$p^{3,1}_{\lambda_1,\lambda_2} = \lambda_1 p_1+\lambda_2p_2$$
where the $\lambda_i\in \R$ are such that
$\lambda_1\lambda_2(\lambda_1^6-\lambda_2^6)\neq 0$.
Here we have the noncanonical semisimple elements
\begin{align*}
p^{3,2}_{\lambda_1,\lambda_2}&=\lambda_1(-\tfrac{1}{2}e_{126}-\tfrac{1}{2}e_{349}+
  2e_{358}-e_{457}+e_{789})+\lambda_2(-2e_{137}-\tfrac{1}{4}e_{249}-e_{258}-
  \tfrac{1}{2}e_{456}-\tfrac{1}{2}e_{689})\\
p^{3,3}_{\lambda_1,\lambda_2}&= \lambda_1(e_{129}-e_{138}+2e_{237}-\tfrac{1}{4}e_{456}-
   2e_{789})+\lambda_2(-\tfrac{1}{2}e_{147}-e_{258}+2e_{369})\\
p^{3,4}_{\lambda,\mu}&=\lambda(e_{147}-2e_{169}-e_{245}+e_{289}-e_{356}-
   \tfrac{1}{2}e_{378})+
   \mu(-e_{124}-e_{136}-\tfrac{1}{2}e_{238}+e_{457}-2e_{569}+e_{789}),
\end{align*}
where $\lambda_1,\lambda_2,\lambda,\mu\in \R$.
These are $\SL(9,\C)$-conjugate to, respectively, the nonreal elements
$$i\lambda_1p_1+i\lambda_2p_2,\qquad i\lambda_1 p_1+\lambda_2 p_2,\qquad
(\lambda+i\mu) p_1+(\lambda-i\mu) p_2.$$
So for $p^{3,2}_{\lambda_1,\lambda_2}$ the $\lambda_i$ are required to satisfy
$\lambda_1\lambda_2(\lambda_1^6-\lambda_2^6)\neq 0$. For
$p^{3,3}_{\lambda_1,\lambda_2}$ the $\lambda_i$ are required to satisfy
$\lambda_1\lambda_2(\lambda_1^6+\lambda_2^6)\neq 0$. For $p^{3,4}_{\lambda,\mu}$
the $\lambda,\mu$ are required to satisfy
$\lambda\mu(3\lambda^4-10\lambda^2\mu^2+3\mu^4)\neq 0$.

We have that $p^{3,2}_{\lambda_1,\lambda_2}$, $p^{3,3}_{\lambda_1,\lambda_2}$,
$p^{3,4}_{\lambda,\mu}$ are $\SL(9,\R)$-conjugate to respectively
\begin{align*}
& \frac{1}{3}\sqrt{3}(\lambda_2p_1+\lambda_1p_2+(\lambda_1+\lambda_2)p_3
+(\lambda_1-\lambda_2)p_4,\\
& -\lambda_2p_1 +\frac{1}{3}\sqrt{3} \lambda_1 (p_2+p_3+p_4),\\
& (\lambda+\frac{1}{3}\sqrt{3}\mu)p_2+(-\lambda+\frac{1}{3}\sqrt{3}\mu)p_3-
\frac{2}{3}\sqrt{3} \mu p_4,
\end{align*}

which lie in the Cartan subspace.

The conjugacy of the elements $p^{3,1}_{\lambda_1,\lambda_2}$ is
determined by a group isomorphic to the dihedral group of order 8 generated
by
$$\SmallMatrix{-1&0\\0&1}, \SmallMatrix{0&1\\1&0}.$$
The conjugacy of the elements $p^{3,2}_{\lambda_1,\lambda_2}$ is determined by
exactly the same group. The conjugacy of the elements $p^{3,3}_{\lambda_1,\lambda_2}$
is determined by a group of order 4 generated
by
$$\SmallMatrix{-1&0\\0&1}, \SmallMatrix{1&0\\0&-1}.$$
The conjugacy of the elements $p^{3,4}_{\lambda,\mu}$ is determined by
exactly the same group of order 4.

{\em Canonical set $\Fm_4$:} The real elements in $\Fm_4$ consist of
$$p^{4,1}_{\lambda,\mu} = \lambda p_1+\mu (p_3-p_4)$$
where the $\lambda,\mu\in \R$ are such that $\lambda\mu(\lambda^3-\mu^3)
(\lambda^3+8\mu^3)\neq 0$. Here there are no noncanonical semisimple
elements.

The conjugacy of the elements $p^{4,1}_{\lambda,\mu}$ is determined by the
group consisting of $\SmallMatrix{1&0\\0&1}$, $\SmallMatrix{-1&0\\0&-1}$.

{\em Canonical set $\Fm_5$:} The real elements in $\Fm_5$ consist of
$$p^{5,1}_{\lambda} = \lambda  (p_3-p_4)$$
where the $\lambda\in \R$ is nonzero.
Here we have the noncanonical semisimple elements
$$p^{5,2}_\lambda = \lambda (e_{148}-e_{159}-e_{238}+\tfrac{1}{2}e_{239}-\tfrac{1}{2}e_{247}+e_{257}-\tfrac{1}{2}e_{346}-e_{356}-e_{678}-\tfrac{1}{2}e_{679})$$
where $\lambda\in \R$ is nonzero. We have that $p^{5,2}_\lambda$ is
$\SL(9,\C)$-conjugate to the {\em purely imaginary} trivector
$p^{5,1}_{i\lambda}$.

Furthermore, $p^{5,2}_\lambda$ is $\SL(9,\R)$-conjugate to
$$\frac{1}{3}\sqrt{3}\lambda (2p_2-p_3-p_4)$$
which lies in the Cartan subspace.

The conjugacy of the elements $p^{5,1}_\lambda$ is determined by the group
consisting of $1,-1$. The conjugacy of the elements $p^{5,2}_\lambda$ is
determined by the same group.

{\em Canonical set $\Fm_6$:} The real elements in $\Fm_6$ consist of
$$p^{6,1}_{\lambda} = \lambda  p_1$$
where the $\lambda\in \R$ is nonzero.
Here we have the noncanonical semisimple elements
$$p^{6,2}_\lambda =\lambda (-\tfrac{1}{2}e_{126}-\tfrac{1}{2}e_{349}+2e_{358}-
e_{457}+e_{789})$$
where $\lambda\in \R$ is nonzero. We have that $p^{6,2}_\lambda$ is
$\SL(9,\C)$-conjugate to the {\em purely imaginary} trivector
$p^{6,1}_{i\lambda}$.

We have that $p^{6,2}_\lambda$ is $\SL(9,\R)$-conjugate to
$$\frac{1}{3}\sqrt{3}\lambda (p_2+p_3+p_4)$$
which lies in the Cartan subspace.

The conjugacy of the elements $p^{6,1}_\lambda$ is determined by the group
consisting of $1,-1$. The conjugacy of the elements $p^{6,2}_\lambda$ is
determined by the same group.

{\em Canonical set $\Fm_7$:} this set consists just of 0.

\subsection{The mixed orbits}\label{sec:mixedtables}

Here we give tables of representatives of the orbits of mixed type.
The representatives of those orbits have a semisimple part which is equal to
one of the semisimple elements listed above. So the semisimple parts are
of the form $p^{2,1}_{\lambda_1,\lambda_2,\lambda_3},p^{2,2}_{\lambda_1,\lambda_2,\lambda_3},
\ldots $. For each such semisimple part we have a table listing the possible
nilpotent parts. (Note that the centralizer of
$p^{1,1}_{\lambda_1,\lambda_2,\lambda_3,\lambda_4}$ in $\g^\cC$ is trivial, so there are
no mixed elements with this particular element as semisimple part.)

In the previous subsection we also computed elements of the Cartan subspace
that are $\SL(9,\R)$-conjugate to the elements
$p^{2,2}_{\lambda_1,\lambda_2,\lambda_3}, p^{3,2}_{\lambda_1,\lambda_2},\ldots$. However
we do not work with those, because the semisimple parts tend to become
rather complicated. For example, the mixed element
$p^{6,2}_\lambda-2e_{137}$ is $\SL(9,\R)$-conjugate to
\begin{align*}
\frac{1}{3}\sqrt{3}\lambda (p_2+p_3+p_4)&+\frac{1}{9}\sqrt{3}(
e_{123}+e_{126}+e_{129}-e_{135}-e_{138}+e_{156}+e_{159}-e_{168}+e_{189}+e_{234}\\
& +e_{237}-e_{246}-e_{249}+e_{267}-e_{279}+e_{345}+e_{348}-e_{357}+e_{378}+e_{456}\\
& +e_{459}-e_{468}+
e_{489}+e_{567}-e_{579}+e_{678}+e_{789}).
\end{align*}

All tables have three columns. The first column has the number of the complex
nilpotent orbit, which corresponds to the numbering in \cite{VE1978}. The
second column has the representatives of the real nilpotent orbits. In the
third column we display the isomorphism type of the centralizer in
$\z_{\g_0}(p)$ of a homogeneous $\ssl_2$-triple in $\z_\g(p)$ containing the
nilpotent element on the same line. Among the nilpotent elements we also
include 0; in that case the centralizer is equal to $\z_{\g_0}(p)$.

\begin{longtable}{|r|l|l|}
  \caption{Nilpotent parts of mixed elements with semisimple part
  $p^{2,1}_{\lambda_1,\lambda_2,\lambda_3}$}\label{tab:fam2_1}
\endfirsthead
\hline
\endhead
\hline
\endfoot
\endlastfoot

\hline
No. &\qquad\quad Reps. of nilpotent parts & $\qquad\mathfrak{z}_0(p,h,e,f)$\\
\hline
1 & ${\scriptstyle e_{168}+e_{249}}$ & ${\scriptstyle 0}$ \\
\gr
2 & ${\scriptstyle e_{168}}$ & ${\scriptstyle \ttt}$ \\
3 & ${\scriptstyle  0}$ & ${\scriptstyle 2\ttt}$ \\
\hline
\end{longtable}

\begin{longtable}{|r|l|l|}
    \caption{Nilpotent parts of mixed elements with semisimple part
  $p^{2,2}_{\lambda_1,\lambda_2,\lambda_3}$}\label{tab:fam2_2}
\endfirsthead
\hline
\endhead
\hline
\endfoot
\endlastfoot

\hline
No. &\qquad\quad Reps. of nilpotent parts & $\qquad\mathfrak{z}_0(p,h,e,f)$ \\
\hline
1 & ${\scriptstyle e_{235}+\tfrac{1}{2}e_{279}-\tfrac{1}{2}e_{369}-e_{567}}$ &
${\scriptstyle 0}$\\
\gr
2 & ${\scriptstyle -e_{148}}$ & ${\scriptstyle \uuu}$\\
\gr
  & ${\scriptstyle e_{148}}$ & ${\scriptstyle \uuu}$\\
3 & ${\scriptstyle 0}$ & ${\scriptstyle \ttt+\uuu}$\\
\hline
\end{longtable}

\begin{longtable}{|r|l|l|}
      \caption{Nilpotent parts of mixed elements with semisimple part
  $p^{3,1}_{\lambda_1,\lambda_2}$}\label{tab:fam3_1}
\endfirsthead
\hline
\endhead
\hline
\endfoot
\endlastfoot

\hline
No. &\qquad\quad Reps. of nilpotent parts & $\qquad\mathfrak{z}_0(p,h,e,f)$ \\
\hline
1 & ${\scriptstyle e_{159}+e_{168}+e_{249}+e_{267}}$ & ${\scriptstyle 0}$ \\
\gr
2 & ${\scriptstyle e_{159}+e_{168}+e_{249}}$ & ${\scriptstyle \ttt}$\\
3 & ${\scriptstyle e_{159}+e_{168}+e_{267}}$ & ${\scriptstyle \ttt}$\\
\gr
4 & ${\scriptstyle e_{159}+e_{168}}$ & ${\scriptstyle 2\ttt}$\\\
5 & ${\scriptstyle e_{159}+e_{267}}$ & ${\scriptstyle 2\ttt}$\\
\gr
6 & ${\scriptstyle e_{168}+e_{249}}$ & ${\scriptstyle 2\ttt}$\\
7 & ${\scriptstyle e_{159}}$ & ${\scriptstyle 3\ttt}$\\
\gr
8 & ${\scriptstyle e_{168}}$ & ${\scriptstyle 3\ttt}$\\
9 & ${\scriptstyle 0}$ & ${\scriptstyle 4\ttt}$\\
\hline
\end{longtable}
\bigskip

\begin{longtable}{|r|l|l|}
        \caption{Nilpotent parts of mixed elements with semisimple part
  $p^{3,2}_{\lambda_1,\lambda_2}$}\label{tab:fam3_2}
\endfirsthead
\hline
\endhead
\hline
\endfoot
\endlastfoot

\hline
No. &\qquad\quad Reps. of nilpotent parts & $\qquad\mathfrak{z}_0(p,h,e,f)$ \\
\hline
1 & ${\scriptstyle  e_{235}+e_{238}+\tfrac{1}{2}e_{247}+\tfrac{1}{2}e_{279}+
  \tfrac{1}{2}e_{346}-\tfrac{1}{2}e_{369}-e_{567}+e_{678}}$ & ${\scriptstyle 0}$\\
\gr
2 & ${\scriptstyle e_{159}+e_{235}+\tfrac{1}{2}e_{279}-\tfrac{1}{2}e_{369}-e_{567}}$
& ${\scriptstyle \uuu}$\\
\gr
& ${\scriptstyle -e_{159}-e_{235}-\tfrac{1}{2}e_{279}+\tfrac{1}{2}e_{369}+e_{567}}$
& ${\scriptstyle \uuu}$ \\
3 &  ${\scriptstyle -e_{148}+\tfrac{1}{2}e_{234}+e_{278}+e_{368}+
  \tfrac{1}{2}e_{467}}$ &${\scriptstyle \uuu}$\\
& ${\scriptstyle e_{148}-\tfrac{1}{2}e_{234}-e_{278}-e_{368}-\tfrac{1}{2}e_{467}}$ &
${\scriptstyle \uuu}$\\
\gr
4 & ${\scriptstyle -e_{148}+e_{159}}$ & ${\scriptstyle 2\uuu}$\\
\gr
& ${\scriptstyle -e_{148}-e_{159}}$ & ${\scriptstyle 2\uuu}$\\
\gr
& ${\scriptstyle e_{148}+e_{159}}$ & ${\scriptstyle 2\uuu}$\\
\gr
& ${\scriptstyle e_{148}-e_{159}}$ & ${\scriptstyle 2\uuu}$\\
5 & ${\scriptstyle \tfrac{1}{2}e_{234}+e_{278}+e_{368}+\tfrac{1}{2}e_{467}}$ &
${\scriptstyle \ttt+\uuu}$\\
\gr
6 & ${\scriptstyle e_{235}+\tfrac{1}{2}e_{279}-\tfrac{1}{2}e_{369}-e_{567}}$ &
${\scriptstyle \ttt+\uuu}$\\
7 & ${\scriptstyle -e_{159}}$ &${\scriptstyle \ttt+2\uuu}$\\
& ${\scriptstyle e_{159}}$ & ${\scriptstyle \ttt+2\uuu}$\\
\gr
8 & ${\scriptstyle -e_{148}}$ & ${\scriptstyle \ttt+2\uuu}$\\
\gr
& ${\scriptstyle e_{148}}$ & ${\scriptstyle \ttt+2\uuu}$ \\
9 & ${\scriptstyle 0}$ & ${\scriptstyle 2\ttt+2\uuu}$\\
\hline
\end{longtable}

\begin{longtable}{|r|l|l|}
        \caption{Nilpotent parts of mixed elements with semisimple part
  $p^{3,3}_{\lambda_1,\lambda_2}$}\label{tab:fam3_3}
\endfirsthead
\hline
\endhead
\hline
\endfoot
\endlastfoot

\hline
No. &\qquad\quad Reps. of nilpotent parts & $\qquad\mathfrak{z}_0(p,h,e,f)$ \\
\hline
1 & ${\scriptstyle e_{159}-2e_{249}-2e_{348}+2e_{357}}$ & ${\scriptstyle 0}$\\
\gr
4 & ${\scriptstyle e_{159}+2e_{357}}$ & ${\scriptstyle \ttt+\uuu}$\\
9 & ${\scriptstyle 0}$ & ${\scriptstyle 2\ttt+2\uuu}$\\
\hline
\end{longtable}
\bigskip

\begin{longtable}{|r|l|l|}
  \caption{Nilpotent parts of mixed elements with semisimple part
  $p^{3,4}_{\lambda,\mu}$}\label{tab:fam3_4}
\endfirsthead
\hline
\endhead
\hline
\endfoot
\endlastfoot

\hline
No. &\qquad\quad Reps. of nilpotent parts & $\qquad\mathfrak{z}_0(p,h,e,f)$ \\
\hline
1 & ${\scriptstyle e_{123}-2e_{179}-2e_{259}-2e_{267}-\tfrac{1}{2}e_{349}+e_{357}}$
& ${\scriptstyle 0}$ \\
\gr
2 & ${\scriptstyle e_{123}-2e_{179}-2e_{259}-\tfrac{1}{2}e_{349}+e_{357}}$ &
${\scriptstyle \ttt}$\\
3 & ${\scriptstyle -2e_{267}-\tfrac{1}{2}e_{349}+\tfrac{1}{4}e_{468}}$ &
${\scriptstyle \uuu}$ \\
& ${\scriptstyle 2e_{267}-\tfrac{1}{2}e_{349}-\tfrac{1}{4}e_{468}}$ &
${\scriptstyle \uuu}$\\
\gr
4 & ${\scriptstyle -2e_{349}+e_{468}}$ & ${\scriptstyle \ttt+\uuu}$\\
\gr
& ${\scriptstyle -2e_{349}-e_{468}}$ & ${\scriptstyle \ttt+\uuu}$\\
5 & ${\scriptstyle -2e_{267}-\tfrac{1}{2}e_{349}}$ & ${\scriptstyle \ttt+\uuu}$\\
\gr
6 & ${\scriptstyle e_{123}-2e_{179}-2e_{259}+e_{357}}$ & ${\scriptstyle 2\ttt}$\\
7 & ${\scriptstyle e_{349}}$ & ${\scriptstyle 2\ttt+\uuu}$\\
\gr
8 & ${\scriptstyle e_{468}}$ &  ${\scriptstyle 2\ttt+\uuu}$\\
\gr
& ${\scriptstyle -e_{468}}$ &  ${\scriptstyle 2\ttt+\uuu}$\\
9 & ${\scriptstyle 0}$ & ${\scriptstyle 3\ttt+\uuu}$\\
\hline
\end{longtable}
\bigskip

\begin{longtable}{|r|l|l|}
          \caption{Nilpotent parts of mixed elements with semisimple part
  $p^{4,1}_{\lambda,\mu}$}\label{tab:fam4_1}
\endfirsthead
\hline
\endhead
\hline
\endfoot
\endlastfoot

\hline
No. &\qquad\quad Reps. of nilpotent parts & $\qquad\mathfrak{z}_0(p,h,e,f)$ \\
\hline
1 & ${\scriptstyle e_{149}+e_{167}+e_{258}+e_{347} }$ & ${\scriptstyle 0}$\\
\gr
2 & ${\scriptstyle e_{149}+e_{158}+e_{167}+e_{248}+e_{257}+e_{347} }$ &
${\scriptstyle 0}$\\
3 & ${\scriptstyle e_{147}+e_{258} }$ & ${\scriptstyle 0}$\\
& ${\scriptstyle e_{147}-e_{158}-e_{248}-e_{257} }$ & ${\scriptstyle 0}$\\
\gr
4 & ${\scriptstyle e_{148}+e_{157}+e_{247} }$ & ${\scriptstyle \ttt}$  \\
5 & ${\scriptstyle e_{147} }$ & ${\scriptstyle \ssl(2,\R)}$ \\
\gr
6 & ${\scriptstyle 0}$ & ${\scriptstyle \ssl(3,\R)}$\\
\hline
\end{longtable}

\begin{longtable}{|r|l|l|}
            \caption{Nilpotent parts of mixed elements with semisimple part
  $p^{5,1}_{\lambda}$}\label{tab:fam5_1}
\endfirsthead
\hline
\endhead
\hline
\endfoot
\endlastfoot

\hline
No. &\qquad\quad Reps. of nilpotent parts & $\qquad\mathfrak{z}_0(p,h,e,f)$ \\
\hline
1 & ${\scriptstyle e_{123}+e_{149}+e_{167}+e_{258}+e_{347}+e_{456}}$ &
${\scriptstyle 0}$ \\
\gr
2 & ${\scriptstyle e_{123}+e_{149}+e_{158}+e_{167}+e_{248}+e_{257}+e_{347}+e_{456}}$
& ${\scriptstyle 0}$\\
3 & ${\scriptstyle e_{123}+e_{149}+e_{167}+e_{258}+e_{347}}$ &
${\scriptstyle \ttt}$ \\
\gr
4 & ${\scriptstyle e_{123}+e_{147}+e_{258}+e_{456}}$ & ${\scriptstyle 0}$ \\
\gr
& ${\scriptstyle e_{123}-2e_{148}-2e_{157}-2e_{247}+2e_{258}+e_{456}}$ &
${\scriptstyle 0}$ \\
5 & ${\scriptstyle e_{123}+e_{149}+e_{158}+e_{167}+e_{248}+e_{257}+e_{347}}$ &
${\scriptstyle \ttt}$\\
\gr
6 & ${\scriptstyle e_{149}+e_{167}+e_{258}+e_{347}}$ & ${\scriptstyle 2\ttt}$ \\
7 & ${\scriptstyle e_{123}+e_{148}+e_{157}+e_{247}+e_{456}}$ &
${\scriptstyle \ttt}$\\
\gr
8 & ${\scriptstyle e_{123}+e_{147}+e_{258}}$ & ${\scriptstyle \ttt}$\\
\gr
& ${\scriptstyle e_{123}-2e_{148}-2e_{157}-2e_{247}+2e_{258}}$ &
${\scriptstyle \ttt}$\\
9 & ${\scriptstyle e_{149}+e_{158}+e_{167}+e_{248}+e_{257}+e_{347}}$ &
${\scriptstyle 2\ttt}$\\
\gr
10 & ${\scriptstyle e_{123}+e_{148}+e_{157}+e_{247}}$ & ${\scriptstyle 2\ttt}$\\
11 & ${\scriptstyle e_{147}+e_{258}}$ & ${\scriptstyle \ttt}$ \\
& ${\scriptstyle e_{147}-e_{158}-e_{248}-e_{257}}$ & ${\scriptstyle \ttt}$ \\
\gr
12 & ${\scriptstyle e_{123}+e_{147}+e_{456}}$ & ${\scriptstyle \ssl(2,\R)}$\\
13 & ${\scriptstyle e_{148}+e_{157}+e_{247}}$ & ${\scriptstyle 3\ttt}$\\
\gr
14 & ${\scriptstyle e_{123}+e_{147}}$ & ${\scriptstyle \ssl(2,\R)+\ttt}$\\
15 & ${\scriptstyle e_{147}}$ & ${\scriptstyle \ssl(2,\R)+2\ttt}$\\
\gr
16 & ${\scriptstyle e_{123}+e_{456}}$ & ${\scriptstyle \ssl(3,\R)}$\\
17 & ${\scriptstyle e_{123}}$ & ${\scriptstyle \ssl(3,\R)+\ttt}$\\
\gr
18 & ${\scriptstyle 0}$ & ${\scriptstyle \ssl(3,\R)+2\ttt}$\\
\hline
\end{longtable}

\begin{longtable}{|r|l|l|}
              \caption{Nilpotent parts of mixed elements with semisimple part
  $p^{5,2}_{\lambda}$}\label{tab:fam5_2}
\endfirsthead
\hline
\endhead
\hline
\endfoot
\endlastfoot

\hline
No. &\qquad\quad Reps. of nilpotent parts & $\qquad\mathfrak{z}_0(p,h,e,f)$ \\
\hline
1 & ${\scriptstyle 2e_{138}+e_{139}+e_{147}+2e_{157}+2e_{237}+e_{345}-e_{389}-e_{468}+
  \tfrac{1}{2}e_{469}+\tfrac{1}{2}e_{479}+2e_{568}-e_{569}-2e_{578}}$ &
${\scriptstyle 0}$\\
\gr
2 & ${\scriptstyle
  -\tfrac{1}{2}e_{134}+e_{135}-e_{178}+\tfrac{1}{2}e_{179}-\tfrac{1}{4}e_{248}
+\tfrac{1}{8}e_{249}+\tfrac{1}{2}e_{258}-\tfrac{1}{4}e_{259}+e_{345}+e_{367}-e_{389}
+\tfrac{1}{2}e_{456}+\tfrac{1}{2}e_{479}-2e_{578}+\tfrac{1}{2}e_{689}}$ &
${\scriptstyle 0}$\\
3 & ${\scriptstyle -\tfrac{1}{2}e_{126}+e_{134}-2e_{135}+2e_{178}-e_{179}-e_{248}-
  \tfrac{1}{2}e_{249}-2e_{258}-e_{259}-2e_{367}}$ & ${\scriptstyle \uuu}$ \\
& ${\scriptstyle \tfrac{1}{2}e_{126}+e_{134}-2e_{135}+2e_{178}-e_{179}-e_{248}-
  \tfrac{1}{2}e_{249}-2e_{258}-e_{259}-2e_{367}}$ & ${\scriptstyle \uuu}$ \\
\gr
4 & ${\scriptstyle \tfrac{1}{2}e_{245}+\tfrac{1}{2}e_{289}+e_{345}-e_{389}+
  \tfrac{1}{4}e_{469}+\tfrac{1}{2}e_{479}+e_{568}-2e_{578}}$ &
${\scriptstyle 0}$\\
\gr
& ${\scriptstyle  \tfrac{1}{8}e_{248}+\tfrac{1}{16}e_{249}+\tfrac{1}{4}e_{258}+
  \tfrac{1}{8}e_{259}+e_{345}-e_{389}-e_{468}+\tfrac{1}{2}e_{469}+
  \tfrac{1}{2}e_{479}+2e_{568}-e_{569}-2e_{578}}$ & ${\scriptstyle 0}$\\
5 & ${\scriptstyle -\tfrac{1}{2}e_{126}+2e_{138}+e_{139}+e_{147}+\tfrac{1}{4}e_{148}-
  \tfrac{1}{8}e_{149}+2e_{157}-\tfrac{1}{2}e_{158}+\tfrac{1}{4}e_{159}+2e_{237}+
  \tfrac{1}{4}e_{346}-\tfrac{1}{2}e_{356}+\tfrac{1}{2}e_{678}-\tfrac{1}{4}e_{679}}$
& ${\scriptstyle \uuu}$ \\
  & ${\scriptstyle \tfrac{1}{2}e_{126}+2e_{138}+e_{139}+e_{147}+\tfrac{1}{4}e_{148}-
  \tfrac{1}{8}e_{149}+2e_{157}-\tfrac{1}{2}e_{158}+\tfrac{1}{4}e_{159}+2e_{237}+
  \tfrac{1}{4}e_{346}-\tfrac{1}{2}e_{356}+\tfrac{1}{2}e_{678}-\tfrac{1}{4}e_{679}}$
& ${\scriptstyle \uuu}$\\
\gr
6 & ${\scriptstyle 2e_{138}+e_{139}+e_{147}+2e_{157}+2e_{237}-e_{468}+
  \tfrac{1}{2}e_{469}+2e_{568}-e_{569}}$ & ${\scriptstyle \ttt+\uuu}$\\
7 & ${\scriptstyle e_{134}-2e_{135}+2e_{178}-e_{179}+e_{345}-2e_{367}-e_{389}+
  \tfrac{1}{2}e_{479}-2e_{578}}$ & ${\scriptstyle \ttt}$\\
\gr
8 & ${\scriptstyle  -\tfrac{1}{2}e_{126}-\tfrac{1}{4}e_{249}-e_{258}-
  \tfrac{1}{2}e_{456}-\tfrac{1}{2}e_{689}}$ & ${\scriptstyle \uuu}$ \\
\gr
& ${\scriptstyle -\tfrac{1}{2}e_{126}+\tfrac{1}{4}e_{137}+e_{468}-
  \tfrac{1}{2}e_{469}-2e_{568}+e_{569}}$ & ${\scriptstyle \uuu}$ \\
\gr
& ${\scriptstyle \tfrac{1}{2}e_{126}-\tfrac{1}{4}e_{249}-e_{258}-
  \tfrac{1}{2}e_{456}-\tfrac{1}{2}e_{689}}$ & ${\scriptstyle \uuu}$\\
\gr
& ${\scriptstyle \tfrac{1}{2}e_{126}+\tfrac{1}{4}e_{137}+e_{468}-
  \tfrac{1}{2}e_{469}-2e_{568}+e_{569}}$ & ${\scriptstyle \uuu}$\\
9 & ${\scriptstyle  2e_{138}+e_{139}+e_{147}+\tfrac{1}{4}e_{148}-\tfrac{1}{8}e_{149}
  +2e_{157}-\tfrac{1}{2}e_{158}+\tfrac{1}{4}e_{159}+2e_{237}+\tfrac{1}{4}e_{346}-
  \tfrac{1}{2}e_{356}+\tfrac{1}{2}e_{678}-\tfrac{1}{4}e_{679}}$ &
${\scriptstyle \ttt+\uuu}$\\
\gr
10 & ${\scriptstyle -\tfrac{1}{2}e_{126}+2e_{138}+e_{139}+e_{147}+2e_{157}+2e_{237}}$
& ${\scriptstyle \ttt+\uuu}$\\
11 & ${\scriptstyle \tfrac{1}{2}e_{245}+\tfrac{1}{2}e_{289}+\tfrac{1}{4}e_{469}+
  e_{568}}$ & ${\scriptstyle \ttt+\uuu}$\\
 & ${\scriptstyle e_{248}+\tfrac{1}{2}e_{249}+2e_{258}+e_{259}+\tfrac{1}{8}e_{468}-
  \tfrac{1}{16}e_{469}-\tfrac{1}{4}e_{568}+\tfrac{1}{8}e_{569}}$ &
${\scriptstyle \ttt+\uuu}$\\
\gr
12 & ${\scriptstyle e_{345}-e_{389}-e_{468}+\tfrac{1}{2}e_{469}+\tfrac{1}{2}e_{479}+
  2e_{568}-e_{569}-2e_{578}}$ & ${\scriptstyle \ssl(2,\R)}$ \\
13 & ${\scriptstyle e_{134}-2e_{135}+2e_{178}-e_{179}-2e_{367}}$ &
${\scriptstyle 2\ttt+\uuu}$ \\
\gr
14 & ${\scriptstyle -\tfrac{1}{2}e_{126}-2e_{137}}$ &
${\scriptstyle \ssl(2,\R)+\uuu}$ \\
\gr
& ${\scriptstyle \tfrac{1}{2}e_{126}-2e_{137}}$ &
${\scriptstyle \ssl(2,\R)+\uuu}$\\
15 & ${\scriptstyle -e_{468}+\tfrac{1}{2}e_{469}+2e_{568}-e_{569}}$ &
${\scriptstyle \ssl(2,\R)+\ttt+\uuu}$\\
\gr
16 & ${\scriptstyle e_{345}-e_{389}+\tfrac{1}{2}e_{479}-2e_{578}}$ &
${\scriptstyle \ssl(3,\R)}$\\
17 & ${\scriptstyle -e_{126}}$ & ${\scriptstyle \ssl(3,\R)+\uuu}$\\
& ${\scriptstyle e_{126}}$ & ${\scriptstyle \ssl(3,\R)+\uuu}$\\
\gr
18 & ${\scriptstyle 0}$ & ${\scriptstyle \ssl(3,\R)+\ttt+\uuu}$\\
\hline
\end{longtable}

\begin{longtable}{|r|l|l|}
                \caption{Nilpotent parts of mixed elements with semisimple part
  $p^{6,1}_{\lambda}$}\label{tab:fam6_1}
\endfirsthead
\hline
\endhead
\hline
\endfoot
\endlastfoot

\hline
No. &\qquad\quad Reps. of nilpotent parts & $\qquad\mathfrak{z}_0(p,h,e,f)$ \\
\hline
1 & ${\scriptstyle e_{159}+e_{168}+e_{249}+e_{258}+e_{267}+e_{347}}$ & 0\\
\gr
2 & ${\scriptstyle e_{159}+e_{168}+e_{249}+e_{257}+e_{258}+e_{347}}$ & 0 \\
3 & ${\scriptstyle e_{149}+e_{158}+e_{167}+e_{248}+e_{259}+e_{347}}$ & 0\\
& ${\scriptstyle -e_{148}-e_{159}-e_{249}+e_{258}-e_{267}-e_{357}}$ & 0\\
\gr
4 & ${\scriptstyle e_{149}+e_{158}+e_{248}+e_{257}+e_{367}}$ &
${\scriptstyle \ttt}$ \\
5 & ${\scriptstyle e_{149}+e_{167}+e_{168}+e_{257}+e_{348}}$ &
${\scriptstyle \ttt}$ \\
\gr
6 & ${\scriptstyle e_{149}+e_{158}+e_{248}+e_{267}+e_{357}}$ &
${\scriptstyle \ttt}$\\
7 & ${\scriptstyle e_{149}+e_{158}+e_{167}+e_{248}+e_{357}}$ &
${\scriptstyle \ttt}$\\
\gr
8 & ${\scriptstyle e_{149}+e_{167}+e_{258}+e_{347}}$ &
${\scriptstyle 2\ttt}$\\
9 & ${\scriptstyle e_{147}+e_{158}+e_{258}+e_{269}}$ &
${\scriptstyle 2\ttt}$ \\
& ${\scriptstyle 2e_{147}-e_{159}+e_{168}-e_{249}-2e_{257}}$ &
${\scriptstyle \ttt + \uuu}$ \\
\gr
10 & ${\scriptstyle e_{149}+e_{158}+e_{167}+e_{248}+e_{257}+e_{347}}$ & 0 \\
11 & ${\scriptstyle e_{149}+e_{167}+e_{248}+e_{357}}$ &
${\scriptstyle 2\ttt}$\\
\gr
12 & ${\scriptstyle e_{149}+e_{167}+e_{247}+e_{258}}$ & ${\scriptstyle 2\ttt}$\\
13 & ${\scriptstyle e_{149}+e_{158}+e_{167}+e_{248}+e_{257}}$ &
${\scriptstyle \ttt}$\\
\gr
14 & ${\scriptstyle e_{149}+e_{157}+e_{168}+e_{247}+e_{348}}$ &
${\scriptstyle \ssl(2,\R)}$ \\
15 & ${\scriptstyle e_{158}+e_{169}+e_{247}}$ &
${\scriptstyle \ssl(2,\R)+2\ttt}$\\
\gr
16 & ${\scriptstyle e_{149}+e_{158}+e_{167}+e_{247}}$ &
${\scriptstyle 2\ttt}$\\
17 & ${\scriptstyle e_{148}+e_{157}+e_{249}+e_{267}}$ &
${\scriptstyle \ssl(2,\R)+\ttt}$\\
\gr
18 & ${\scriptstyle e_{147}+e_{158}+e_{248}+e_{259}}$ &
${\scriptstyle \ssl(2,\R)+\ttt}$\\
19 & ${\scriptstyle e_{149}+e_{157}+e_{248}}$ & ${\scriptstyle 3\ttt}$\\
\gr
20 & ${\scriptstyle e_{147}+e_{258}}$ & ${\scriptstyle 4\ttt}$\\
\gr
& ${\scriptstyle 2e_{147}-2e_{158}-2e_{248}-2e_{257}}$ &
${\scriptstyle 2\ttt + 2\uuu}$\\
21 & ${\scriptstyle e_{148}+e_{157}+e_{247}}$ & ${\scriptstyle 3\ttt}$\\

\gr
22 & ${\scriptstyle  e_{147}+e_{158}+e_{169}}$ &
${\scriptstyle \ssl(2,\R)+\ssl(3,\R)}$\\
23 & ${\scriptstyle e_{147}+e_{158}}$ &
${\scriptstyle 2\ssl(2,\R)+2\ttt}$\\
\gr
24 & ${\scriptstyle e_{147}}$ & ${\scriptstyle 3\ssl(2,\R)+2\ttt}$\\
25 & ${\scriptstyle 0}$ & ${\scriptstyle 3\ssl(3,\R)}$\\
\hline
\end{longtable}

\begin{longtable}{|r|l|l|}
\caption{Nilpotent parts of mixed elements with semisimple part
  $p^{6,2}_{\lambda}$}\label{tab:fam6_2}
\endfirsthead
\hline
\endhead
\hline
\endfoot
\endlastfoot

\hline
No. &\qquad\quad Reps. of nilpotent parts & $\qquad\mathfrak{z}_0(p,h,e,f)$ \\
\hline
1 & ${\scriptstyle e_{139}-e_{148}+2e_{157}+e_{245}+e_{289}+2e_{367}}$ & 0\\
\gr
2 & ${\scriptstyle -2e_{135}-e_{179}+\tfrac{1}{2}e_{234}+e_{278}+e_{368}+
  \tfrac{1}{2}e_{467}-e_{569}}$ & 0\\
3 & ${\scriptstyle -2e_{135}-e_{145}-e_{179}-e_{189}+\tfrac{1}{2}e_{234}+e_{278}+
  e_{368}+\tfrac{1}{2}e_{467}}$ & 0\\
& ${\scriptstyle e_{145}+e_{189}-2e_{237}-e_{248}+e_{346}+2e_{678}}$ & 0 \\
\gr
4 & ${\scriptstyle e_{159}+2e_{237}-e_{248}-\tfrac{1}{4}e_{249}-e_{258}-
  \tfrac{1}{2}e_{456}-\tfrac{1}{2}e_{689}}$ &  ${\scriptstyle \uuu}$ \\
\gr
& ${\scriptstyle -4e_{159}+\tfrac{1}{32}e_{237}-\tfrac{1}{2}e_{238}-
  \tfrac{1}{8}e_{239}-\tfrac{1}{4}e_{247}+20e_{248}+5e_{249}-\tfrac{1}{4}e_{257}+
  20e_{258}+4e_{259}-\tfrac{1}{4}e_{356}+2e_{456}-\tfrac{1}{8}e_{679}+2e_{689}}$ &
${\scriptstyle \uuu}$\\
5 &  ${\scriptstyle e_{159}-\tfrac{1}{2}e_{239}-e_{248}-e_{257}+e_{356}+
  \tfrac{1}{2}e_{679}}$ & ${\scriptstyle \uuu}$\\
& ${\scriptstyle -e_{159}-\tfrac{1}{2}e_{239}-e_{248}-e_{257}+e_{356}+
  \tfrac{1}{2}e_{679}}$ & ${\scriptstyle \uuu}$ \\
\gr
6 & ${\scriptstyle -e_{134}-2e_{178}+2e_{235}+e_{279}-e_{569}}$ &
${\scriptstyle \ttt}$\\
7 & ${\scriptstyle e_{139}-e_{148}+2e_{157}+e_{238}+\tfrac{1}{2}e_{247}+
  \tfrac{1}{2}e_{346}+e_{678}}$ & ${\scriptstyle \uuu}$ \\
& ${\scriptstyle -4e_{135}-\tfrac{1}{4}e_{148}-2e_{179}-\tfrac{1}{2}e_{238}-
  \tfrac{1}{4}e_{247}+\tfrac{1}{4}e_{346}+\tfrac{1}{2}e_{678}}$ &
${\scriptstyle \uuu}$ \\
\gr
8 & ${\scriptstyle  -e_{145}-e_{189}+2e_{237}+e_{468}}$ &
${\scriptstyle \ttt+\uuu}$\\
9 & ${\scriptstyle e_{235}+\tfrac{1}{2}e_{279}-\tfrac{1}{2}e_{369}+e_{468}-e_{567}}$
& ${\scriptstyle \ttt+\uuu}$ \\
& ${\scriptstyle \tfrac{1}{2}e_{237}+e_{248}+e_{259}+\tfrac{1}{2}e_{469}+2e_{568}}$
& ${\scriptstyle 2\uuu}$ \\
& ${\scriptstyle \tfrac{1}{2}e_{237}-e_{248}-e_{259}+\tfrac{1}{2}e_{469}+2e_{568}}$
& ${\scriptstyle 2\uuu}$  \\
\gr
10 & ${\scriptstyle  e_{159}+2e_{237}-\tfrac{1}{2}e_{239}-\tfrac{1}{2}e_{249}-
  e_{257}-2e_{258}-e_{356}-\tfrac{1}{2}e_{679}}$ & 0 \\
11 & ${\scriptstyle -2e_{135}-e_{179}+\tfrac{1}{2}e_{245}+\tfrac{1}{2}e_{289}-
  \tfrac{1}{4}e_{469}-e_{568}}$ & ${\scriptstyle \ttt+\uuu}$\\
\gr
12 & ${\scriptstyle  -2e_{135}-e_{179}+e_{468}-e_{569}}$ &
${\scriptstyle \ttt+\uuu}$\\
\gr
& ${\scriptstyle -2e_{139}-4e_{157}+\tfrac{1}{4}e_{468}+4e_{569}}$ &
${\scriptstyle \ttt+\uuu}$\\
13 & ${\scriptstyle -2e_{138}-e_{147}+e_{159}-e_{239}-2e_{257}}$ &
${\scriptstyle \ttt}$\\
\gr
14 & ${\scriptstyle  e_{159}-\tfrac{1}{2}e_{239}-\tfrac{1}{4}e_{249}-e_{257}-e_{258}
  +e_{356}-\tfrac{1}{2}e_{456}+\tfrac{1}{2}e_{679}-\tfrac{1}{2}e_{689}}$
& ${\scriptstyle \ssl(2,\R)}$ \\
15 & ${\scriptstyle  -2e_{135}-e_{179}+e_{468}}$ &
${\scriptstyle \ssl(2,\R)+\ttt+\uuu}$\\
& ${\scriptstyle \tfrac{1}{4}e_{137}+8e_{159}-e_{468}}$ &
${\scriptstyle \su(2)+\ttt+\uuu}$\\
\gr
16 & ${\scriptstyle  -2e_{137}+\tfrac{1}{2}e_{149}+2e_{158}-e_{259}}$ &
${\scriptstyle \ttt+\uuu}$\\
17 & ${\scriptstyle e_{235}+\tfrac{1}{2}e_{245}+\tfrac{1}{2}e_{279}+
  \tfrac{1}{2}e_{289}-\tfrac{1}{2}e_{369}+\tfrac{1}{4}e_{469}-e_{567}+e_{568}}$
& ${\scriptstyle \ssl(2,\R)+\ttt}$\\
\gr
18 & ${\scriptstyle  e_{159}+2e_{237}-\tfrac{1}{2}e_{239}-e_{257}+e_{356}+
  \tfrac{1}{2}e_{679}}$ & ${\scriptstyle \ssl(2,\R)+\uuu}$\\
\gr
& ${\scriptstyle \tfrac{44}{5}e_{137}-3e_{139}-6e_{157}+2e_{159}-
  \tfrac{28}{5}e_{237}+e_{239}+2e_{257}+10e_{356}+5e_{679}}$ &
${\scriptstyle \su(2)+\uuu}$\\
\gr
& ${\scriptstyle \tfrac{4}{5}e_{137}+e_{139}+2e_{157}+2e_{159}+\tfrac{12}{5}e_{237}
  +e_{239}+2e_{257}-10e_{356}-5e_{679}}$ & ${\scriptstyle \su(2)+\uuu}$\\
19 & ${\scriptstyle e_{159}-\tfrac{1}{2}e_{239}-e_{257}-e_{259}+e_{356}+
  \tfrac{1}{2}e_{679}}$ & ${\scriptstyle \ttt+2\uuu}$\\
& ${\scriptstyle -4e_{159}+e_{239}+2e_{257}+4e_{259}-2e_{356}-e_{679}}$ &
${\scriptstyle \ttt+2\uuu}$\\
\gr
20 & ${\scriptstyle e_{235}+\tfrac{1}{2}e_{279}-\tfrac{1}{2}e_{369}-e_{567}}$ &
${\scriptstyle 2\ttt+2\uuu}$\\
\gr
& ${\scriptstyle \tfrac{1}{4}e_{148}+4e_{159}-\tfrac{1}{2}e_{469}-2e_{568}}$ &
${\scriptstyle 2\ttt+2\uuu}$\\
21 & ${\scriptstyle -2e_{135}-e_{179}-e_{569}}$ & ${\scriptstyle 2\ttt+\uuu}$\\
\gr
22 & ${\scriptstyle e_{139}-e_{148}+2e_{157}}$ &
${\scriptstyle \ssl(2,\R)+\su(1,2)}$\\
\gr
& ${\scriptstyle 18e_{137}+22e_{138}+4e_{139}+11e_{147}+14e_{148}+\tfrac{5}{2}
  e_{149}+8e_{157}+10e_{158}+2e_{159}}$ & ${\scriptstyle \ssl(2,\R)+\su(3)}$\\
23 & ${\scriptstyle -2e_{135}-e_{179}}$ & ${\scriptstyle 2\ssl(2,\R)+\ttt+\uuu}$\\
& ${\scriptstyle e_{137}+2e_{159}}$ & ${\scriptstyle \ssl(2,\R)+\su(2)+\ttt+\uuu}$\\
\gr
24 & ${\scriptstyle -2e_{137}}$ & ${\scriptstyle \ssl(2,\R)+\ssl(2,\C)+\ttt+\uuu}$\\
25 & ${\scriptstyle 0}$ & ${\scriptstyle \ssl(3,\R)+\ssl(3,\C)}$\\
\hline
\end{longtable}


\section{Real Galois cohomology}\label{sec:galcohom}

\subsection{Group cohomology for a group of order 2}

\label{s:group-coh}
In this subsection we give an alternative  definition of    $\Ho^1(\Gamma,A)$ and $\Ho^2(\Gamma,A)$ and establish   their  properties.

\begin{definition}\label{d:H1-nonab}
Let $A$ be a $\Gamma$-group, {\em not necessarily abelian}, where $\Gamma=\{1,\gamma\}$ is a group of order 2.
We define
\begin{equation}\label{d:Z1}
\Zl^1\hm A=\{c\in A\mid c\cdot{}^\gamma\kern-0.8pt c=1\}.
\end{equation}
We say that an element $c$ as in \eqref{d:Z1} is a {\em $1$-cocycle of\/ $\Gamma$ in $A$}.
The group $A$ acts on $\Zl^1\hm A$ on the right by
\[ c*a=a^{-1}\cdot c\cdot\upgam a\quad \text{for } c\in \Zl^1\hm A,\ a\in A.\]
We say that the cocycles $c$ and $c*a$ are {\em cohomologous} or {\em equivalent}.
We denote by  $\Ho^1(\Gamma,A)$, or for brevity $\Ho^1\hm A$, the set of equivalence classes, that is, the set of orbits of $A$ in $\Zl^1\hm A$.
If $c\in \Zl^1\hm A$, we denote by $[c]\in \Ho^1\hm A$ its cohomology class.
\end{definition}

In general $\Ho^1\hm A$ has no natural group structure, but it has a {\em neutral element} denoted by  $[1]$,
 the class of the unit element $1\in \Zl^1\hm A\subseteq A$.

\begin{definition}\label{d:H1}
Let $A$ be an {\em abelian} $\Gamma$-group, written additively, where $\Gamma$ is a group of order 2.
We define
\[\Zl^1\hm A=\{c\in A\mid c+\upgam c=0\},\quad \Bd^1\hm A=\{a-\upgam a\mid a\in A\}.\]
We set
\[  \Ho^1\hm A=\Zl^1\hm A\hs/\hs\Bd^1\hm A,\]
which is an abelian group.
\end{definition}

\begin{definition}\label{d:H2}
Let $A$ be an {\em abelian} $\Gamma$-group written additively, where $\Gamma$ is a group of order 2.
We define
\[\Zl^2\hm A=A^\Gamma:=\{c\in A\mid \upgam c=c\},\quad \Bd^{2}\hm A=\{a+\upgam a\mid a\in A\}.\]
We set
\[  \Ho^2\hm A=\Zl^2\hm A\hs/\hs\Bd^2\hm A,\]
which is an abelian group.
\end{definition}

\begin{remark}\label{r:def}
Definitions \ref{d:H1} and \ref{d:H2} (given for  a group $\Gamma$ of order 2 only!)
are equivalent to the standard ones.
Namely, for $c\in \Zl^1\hm A$ we construct a function of one variable
\[f_c\colon\Gamma\to A,\quad f_c(1)=1,\ f_c(\gamma)=c,\]
which is a 1-cocycle in the sense of  \cite[Section I.5.1]{Serre1997}.
Similarly, for $c\in \Zl^2\hm A$ we construct a function of two variables
\[ \phi_c\colon \Gamma\times\Gamma\to A,\quad  \phi_c(1,1)=\phi_c(\gamma,1)
     =\phi_c(1,\gamma)=0,\ \,\phi_c(\gamma,\gamma)=c,\]
which is a 2-cocycle in the sense of \cite[Section I.2.2]{Serre1997}.
In this way we obtain canonical isomorphisms of pointed sets
(for Definition \ref{d:H1-nonab}), and of abelian groups
(for Definitions \ref{d:H1} and \ref{d:H2})
between the cohomology sets and groups defined above
and the corresponding cohomology sets and groups defined in \cite{Serre1997}.
\end{remark}

We shall consider the following $\Gamma$-groups:
\begin{enumerate}
\item $\Z_\triv$ is $\Z$ with the trivial $\Gamma$-action $\upgam x=x$;
\item $\Z_\tw$ is $\Z$ with the twisted $\Gamma$-action $\upgam x= -x$;
\item $\Z^2_{\tw+}$ is $\Z^2$
with the twisted $\Gamma$-action $\upgam (x,y)=( y,x)$;
\item $Z^2_{\tw-}$ is $\Z^2$
with the twisted $\Gamma$-action $\upgam (x,y)=( -y,-x)$.
\end{enumerate}
We have an isomorphism
\[ \Z^2_{\tw+}\isoto \Z^2_{\tw-}\hs,\quad (x,y)\mapsto (x,-y).\]

\begin{examples}\label{x:cohom-Z}
\begin{enumerate}
\item $\Ho^1\hsss \Z_\triv=0$;
\item $\Ho^2\hssh\Z_\triv=\Z/2\Z$;
\item $\Ho^1\hssh \Z_\tw=\Z/2\Z$;
\item $\Ho^2\hssh\Z_\tw=0$;
\item $\Ho^1\hssh \Z^2_{\tw+}=0$ and $\Ho^2\hssh \Z^2_{\tw+}=0$;
\item  $\Ho^1\hssh \Z^2_{\tw-}=0$ and $\Ho^2\hssh \Z^2_{\tw-}=0$.
\end{enumerate}
\end{examples}

We shall consider the following $\Gamma$-groups:
\begin{enumerate}
\item[{\rm (st)}] $\C^\times_\stand$ is $\C^\times$ with the standard $\Gamma$-action $\upgam x=\bar x$;
\item[{\rm (tw)}] $\C^\times_\tw$ is $\C^\times$ with the twisted $\Gamma$-action $\upgam x=\bar x^{-1}$;
\item[{\rm (tw$_{2+}$)}] $(\C^\times)^2_{\tw+}$ is $(\C^\times)^2$
with the twisted $\Gamma$-action $\upgam (x,y)=(\bar y,\bar x)$;
\item[{\rm (tw$_{2-}$)}]  $(\C^\times)^2_{\tw-}$ is $(\C^\times)^2$
with the twisted $\Gamma$-action $\upgam (x,y)=(\bar y^{-1},\bar x^{-1})$.
\end{enumerate}
We have an isomorphism
\[ (\C^\times)^2_{\tw+}\isoto (\C^\times)^2_{\tw-}\hs,\quad (x,y)\mapsto (x,y^{-1}).\]

\begin{examples}\label{x:cohom}
\begin{enumerate}
\item $\Ho^1\hssh \C^\times_\stand=\{x\in\C^\times\mid x\bar x=1\}/({\rm squares})=1$;
\item $\Ho^2\hssh\C^\times_\stand=\R^\times/\{z\bar z\mid z\in \C^\times\}=\{[1],[-1]\}$;
\item $\Ho^1\hssh \C^\times_\tw=\R^\times/\{z\bar z\mid z\in \C^\times\}=\{[1],[-1]\}$;
\item $\Ho^2\hssh\C^\times_\tw=\{x\in\C^\times\mid x\bar x=1\}/({\rm squares})=1$;
\item $\Ho^1\hssh (\C^\times)^2_{\tw+}=1$ and $\Ho^2\hssh (\C^\times)^2_{\tw+}=1$;
\item $\Ho^1\hssh (\C^\times)^2_{\tw-}=1$ and $\Ho^2\hssh(\C^\times)^2_{\tw-}=1$.
\end{enumerate}
\end{examples}

\begin{remark}
The assertions of Examples \ref{x:cohom},
which follow easily from the definitions,
are special cases of more general results.
In particular, Example \ref{x:cohom}(1) is Hilbert's Theorem 90 for $\R$.
Example \ref{x:cohom}(2) is the well-known fact that the Brauer group
${\rm Br}(\R):=\Ho^2(\Ga,\C^\times_\stand)$ is of order 2.
The assertions of Examples  \ref{x:cohom}(5) and (6)
follow from a lemma of Faddeev and Shapiro;
see \cite[I.2.5, Proposition 10]{Serre1997}.
\end{remark}

\begin{subsec}{\bf Twisting.}
Let $B$ be a $\Gamma$-group acting $\Gamma$-equivariantly on a $\Gamma$-group $A$
\[(b,a)\mapsto b(a).\]
Let $b\in \Zl^1 B$ be a 1-cocycle, that is,
\begin{equation}\label{e:1cocycle}
b\cdot\!\upgam b=1_B\hs.
\end{equation}
We define the {\em $b$-twisted} $\Gamma$-group $_bA$.
It has the same underlying group $A$ and a new $\Gamma$-action
\begin{equation}\label{e:twisted}
\hs^{\gamma *}a=b(\hs^\gamma\kern-0.8pt  a).
\end{equation}
It follows from the cocycle condition \eqref{e:1cocycle}
that
\[\hs^{\gamma *}(\hs^{\gamma *}\kern-0.8pt a)=a\quad \text{for all }a\in A,\]
that is, \eqref{e:twisted} is indeed a $\Gamma$-action on $A$.

Note that the $\Gamma$-groups $A$ and $_bA$ have the same underlying groups.
In particular, if $A$ is a finite group of odd order,
then $_bA$  is a finite group of odd order too.
\end{subsec}

\begin{subsec}\label{subsec:twistedbij}
Let $a\in \Zl^1\hm A$.
The group $A$ acts on itself by inner automorphisms,
and so we obtain the twisted group $_a A$.

We define a map
\[t_a\colon \Zl^1{}_a A\to \Zl^1\hm A,\quad  x\mapsto xa.\]
We check that $t_a(x)$ is indeed a 1-cocycle.
We have $a\in \Zl^1\hm A$, hence $a\cdot\upgam a=1$.
Furthermore, $x\in \Zl^1{}_a A$, hence, $x\cdot a\hs\upgam x\hs a^{-1}=1$.
We obtain:
\[ t_a(x)\cdot\upgam(t_a(x))=x\hs a\cdot\upgam x\upgam a
     =(x\cdot a \upgam x\hs a^{-1})\cdot (a\cdot \upgam a)=1\cdot 1 =1,\]
as required.

The map $t_a$ is a bijection taking 1 to $a$.
We show that it sends cohomologous cocycles to cohomologous cocycles.
Indeed, let $x,y\in \Zl^1{}_a A$, and assume that $x\sim y$.
Then there exists $b\in A$ such that
\[y=b^{-1}\cdot x\cdot {}^{\gamma*}b=b^{-1}\cdot x\cdot a\upgam b\hs\hs a^{-1}.\]
Then
\[t_a(y)=b^{-1}\cdot x\cdot a\upgam b\hs\hs a^{-1}\cdot a =b^{-1}\cdot x a\cdot\upgam b
                           =b^{-1}\cdot t_a(x)\cdot\upgam b\,\sim\, t_a(x),\]
as required.
Thus the map $t_a$ induces a map
\begin{equation}\label{e:tau-con}
\tau_a\colon  \Ho^1{}_a A\to \Ho^1\hm A,\quad  [x]\mapsto [xa].
\end{equation}
One can easily check that $a^{-1}\in \Zl^1{}_a A$ and that the map
\[t_{a^{-1}}\colon \Zl^1\hm A\to \Zl^1{}_a A,\quad  y\mapsto y\hs a^{-1}.\]
is inverse to $t_a$.
By passing to cohomology, we obtain that the map
\[\tau_{a^{-1}}\colon \Ho^1{}_a A\to \Ho^1\hm A,\quad  [y]\mapsto [y\hs a^{-1}].\]
is inverse to $\tau_a$.
Thus $\tau_a$ is a bijection.
\end{subsec}

\begin{subsec}\label{ss:hom-twisted}
Let $\varphi\colon A\to B$ be a morphism of $\Gamma$-groups.
Let $a\in \Zl^1\hm A$, and write $b=\varphi(a)\in \Zl^1 B$.
Then by twisting we obtain a homomorphism
\[ _a \varphi\colon\hs_a A\to \hs_b B,\]
and it is easy to see that the following diagram is commutative:
\begin{equation*}
\xymatrix{
\Ho^1{}_a A\ar[r]^{_a\varphi_*}\ar[d]_{\tau_a} &\Ho^1{}_b B\ar[d]^{\tau_b}\\
\Ho^1\hm A\ar[r]^{\varphi_*}                     &\Ho^1 B
}
\end{equation*}
\end{subsec}

We shall need the following result of Borel and Serre:

\begin{proposition}[{\cite[Section I.5.4, Corollary 1 of Proposition 36]{Serre1997}}]
\label{p:serre}
Let $B$ be a $\Gamma$-group, $A\subseteq B$ be a $\Gamma$-subgroup, and $Y=B/A$,
which has a natural structure of a $\Gamma$-set. Then there is a canonical exact sequence of pointed sets
\[1\to A^\Gamma\to \Bd^\Gamma\to Y^\Gamma\labelto{\delta} \Ho^1\hm A\to \Ho^1 B.\]
Moreover, the group $\Bd^\Gamma$ naturally acts on $Y^\Gamma$, and the connecting map $\delta$
induces a canonical bijection between the set of orbits $Y^\Gamma/\Bd^\Gamma$ and the kernel
\[\ker\left[\hs\Ho^1\hm A\to \Ho^1 B\hs\right].\]
\end{proposition}

\begin{proposition}[{\cite[Section I.5.5, Proposition 38]{Serre1997}}]
\label{p:serre-prop38}
Let
\begin{equation}\label{e:ABC}
1\to A\labelto{i} B\labelto{j} C\to 1
\end{equation}
be a short exact sequence of $\Gamma$-groups.
Then the following sequence of pointed sets is exact:
\begin{equation}\label{e:ABC-long}
1\to A^\Gamma\to \Bd^\Gamma\to C^\Gamma\labelto{\delta} \Ho^1\hm A\labelto{i_*} \Ho^1 B\labelto{j_*} \Ho^1 C.
\end{equation}
\end{proposition}

\begin{proof}
It suffices to show  that if $\beta\in\ker\hs j_*\subseteq \Ho^1 B$, then
there exists $\alpha\in \Ho^1\hm A$ such that $j_*(\alpha)=\beta$.

Let
\[\beta\in \ker[\Ho^1 B\labelto{j_*} \Ho^1 C].\]
Write $\beta=[b]$, where $b\in \Zl^1 B\subseteq B$.
Then $j(b)$ is a coboundary in $C$, that is, there exists $c'\in C$ such that
\[(c')^{-1}\cdot j(b)\cdot\upgam c'=1\hs\in C.\]
Let $b'\in B$ be a preimage of $c'$. Then
\[a:=(b')^{-1}\cdot b\cdot\upgam b'\in A,\]
where we identify $A$ with $i(A)\subseteq B$.
Clearly, $a$ is a cocycle, that is, $a\in \Zl^1\hm A$.
Write $\alpha=[a]\in \Ho^1(\gamma, A)$.
Then
\[i_*(\alpha)=[a]=[b]=\beta,\]
as required.
\end{proof}

\begin{corollary}
\label{c:prop38}
In Proposition \ref{p:serre-prop38}, assume that $\Ho^1\hm A=\{1\}$ and $\Ho^1 C=\{1\}$.
Then $\Ho^1 B=\{1\}$.
\end{corollary}

\begin{proof}
Let $\beta\in \Ho^1(\Gamma, B)$.
Then $j_*(\beta)=1$, because $\Ho^1 C=\{1\}$.
By Proposition \ref{p:serre-prop38} we have $\beta=j_*(\alpha)$
for some $\alpha\in \Ho^1\hm A$.
Since $\Ho^1\hm A=\{1\}$, we see that $\alpha=1$, and hence, $\beta=1$, as required.
\end{proof}

We wish to describe the fibers of the map $i_*\colon \Ho^1\hm A\to \Ho^1 B$.

\begin{construction}\label{con:rightact}
For an exact sequence \eqref{e:ABC}, we construct
a right action of $C^\Gamma$ on $\Ho^1\hm A$.
Let $c\in C^\Ga$.
We lift $c$ to $b\in B$.
Then $\upgam b=b\cdot x$ for some $x\in A$.
To each cocycle $a\in \Zl^1\hm A$, we associate the cocycle
\[a'=b^{-1}abx=b^{-1}\cdot a\cdot\upgam b\in \Zl^1\hm A.\]
We define
\[ [a]\cdot c= [a'].\]
\end{construction}

\begin{proposition}[{Serre, \cite[I.5.5, Proposition 39]{Serre1997}}]
\label{p:action-C-Gamma}
\begin{enumerate}
\item[\rm (i)] If $c\in C^\Ga$, then $\delta(c)=1\cdot c$,
where $1\in \Ho^1\hm A$ is the neutral element.
\item[\rm (ii)] Two elements of $\Ho^1\hm A$ have the same image in $\Ho^1 B$
if and only if they are in the same $C^\Ga$-orbit.
\item[\rm (iii)] Let $a\in \Zl^1\hm A$,
write $\alpha=[a]\in \Ho^1\hm A$, and let $c\in C^\Ga$.
For $\alpha\cdot c=\alpha$, it is necessary and sufficient
that $c\in {\rm im}[\hs({}_aB)^\Ga\to C^\Ga\hs]$.
\end{enumerate}
\end{proposition}

[We denote by $_a B$ the group obtained by twisting $B$ with the cocycle $a$,
with $A$ acting on $B$ by inner automorphisms.]

\begin{proof}
The equality $\delta(c)=1\cdot c$ follows from the definition of $\delta$.
On the other hand, if two cocycles $a,a'\in \Zl^1\hm A$ are cohomologous in $B$,
then there exists $b\in B$ such that $a=b^{-1}\cdot a \cdot\upgam b$.
If $c$ is the image of $b$ in $C$, then $\upgam c=c$, whence $c\in C^\Ga$,
and it is clear that  $[a]\cdot c=[a']$.
Conversely, if $[a']=[a]\cdot c$, then $a'=b^{-1}\cdot a\cdot\upgam b$
for a preimage $b$ of $c$, whence $a\sim a'$ in $B$, which proves (ii).
Finally, if $b\in B$ is a lift of $c$, and $\alpha\cdot c=\alpha$,
then there exists $x\in A$ such that $a=x^{-1} b^{-1}a\hs \upgam b\hs\upgam x$.
This can be also written as $bx=a\hs\upgam(bx)\hs a^{-1}$,
that is, $bx\in (\hs_a B)^\Ga$.
Clearly, the image of $bx$ in $C$ is $c$; hence (iii).
\end{proof}

\begin{corollary}\label{c:39-cor1}
The kernel of $\Ho^1 B\to \Ho^1 C$ can be identified
with the quotient of $\Ho^1\hm A$ by the action of the group $C^\Ga$.
\end{corollary}

\begin{proof}
This follows from Proposition \ref{p:action-C-Gamma}(ii).
\end{proof}

\begin{corollary}\label{c:39-cor2}
Let $\beta\in \Ho^1 B$, $\beta=[b]$, where $b\in \Zl^1 B$.
Then the map  on cocycles
 \[a\mapsto a\hs b\colon \Zl^1{}_b A\to \Zl^1 B\]
induces a surjective map on cohomology
\[\Ho^1{}_b A\,\to\, j_*^{-1}(j_*(\beta))\subset \Ho^1 B,\]
which in turn induces a bijection
\[\Ho^1{}_b A/(\hs_c C)^\Ga\,\isoto\, j_*^{-1}(j_*(\beta)),\]
where $c=j(b)\in C$.
Here for $a\in \Zl^1\hs_b A$, the stabilizer in $(\hs_c C)^\Ga$ of $[a]\in \Ho^1\hs_b A$
is the image of the homomorphism $(\hs_{ab} B)^\Ga\to (\hs_c C)^\Ga$.
\end{corollary}

[The group $B$ acts on itself by inner automorphisms, and leaves $A$ stable.
This allows us to twist the exact sequence \eqref{e:ABC}
by the cocycle $b$ or by the cocycle $ab$,
as in Subsection \ref{ss:hom-twisted}.]

\begin{proof}
This follows from Corollary \ref{c:39-cor1}
and Proposition \ref{p:action-C-Gamma}(iii)
by twisting as in Subsection \ref{ss:hom-twisted}.
\end{proof}

\begin{corollary}\label{c:split}
Assume that the exact sequence of \,$\Gamma$-groups \eqref{e:ABC}
admits a {\emm splitting} $\vs$, that is, a homomorphism of $\Gamma$-groups
\[\vs\colon C\to B\]
such that $j\circ\vs=\id_C$.
For $c\in \Zl^1 C$ consider the twisted exact sequence
\begin{equation}\label{e:ABC-tw}
1\to \hs_c A\to \hs_c B\to \hs_c C\to 1
\end{equation}
and the splitting
\[_c \vs\colon\hs_c C\to \hs_c B\]
where by abuse of notation we write $_c A$ for $_{\vs(c)} A$ and $_c B$ for $_{\vs(c)} B$.
Then:
\begin{enumerate}
\item[(i)] The homomorphism $j\colon B^\Gamma\to C^\Gamma$ is surjective.
\item[(ii)] The map $j_*\colon \Ho^1 B\to\Ho^1 C$ is surjective.
\item[(iii)] The split exact sequence \eqref{e:ABC-tw} induces a bijection
\begin{equation}\label{e:cA-cC}
(\Ho^1{}_c A)/(\hs_c C)^\Gamma \labelto{\sim} j_*^{-1}[c],
\end{equation}
where $(\hs_c C)^\Gamma$ acts on the right on $\Ho^1\hs_c A$ by
\begin{equation}\label{e:c'-action}
[a]\cdot c'=[\vs(c')^{-1} a\, \vs(c')]\quad
    \text{for } a\in\Zl^1\hs_c A,\ c'\in (\hs_c C)^\Gamma.
\end{equation}
\item[(iv)] The bijections  of (iii) for all $[c]\in \Ho^1 C$
induce a canonical bijection
\[\bigsqcup_{[c]\in \Ho^1 C} (\Ho^1{}_c A)/(\hs_c C)^\Gamma\,\labelto{\sim}\, \Ho^1 B,\]
where $\sqcup$ denotes the disjoint union.
\end{enumerate}
\end{corollary}

\begin{proof}
Assertions (i) and (ii) follow  by functoriality from the formula $j\circ\vs=\id_C$.
Assertion (iv) follows from (ii) and (iii).
We obtain the bijection \eqref{e:cA-cC} in (iii)  from
Corollary \ref{c:39-cor2}.

It remains to prove formula \eqref{e:c'-action}.
Let $c'\in (\hs_c C)^\Gamma$.
Set $b'=\vs(c')\in (\hs_c B)^\Gamma$.
By Construction \ref{con:rightact} we have
\[ [a]\cdot c'= [(b')^{-1} a\hs\hs^{\gamma*} b']=
      [(b')^{-1} a\hs b']=[\vs(c')^{-1} a\hs\hs \vs(c')],\]
where $^{\gamma*}b'$ denote the image of $b'$ under the $c$-twisted action of $\gamma$,
and $^{\gamma*}b'=b'$ because $b'\in (\hs_c B)^\Gamma$.
This completes the proof of the corollary.
\end{proof}

\begin{corollary}\label{c:split-action}
In Corollary \ref{c:split}(iii):
\begin{enumerate}
\item[(i)] the group $(\hs_c C)^\Gamma$ acts on $[1]\in \Ho^1{}_c A$ trivially;
\item[(ii)] If $\#\Ho^1{}_c A>1$, then $\#j_*^{-1}[c]>1$.
\end{enumerate}
\end{corollary}

\begin{proof}
Assertion (i) follows from formula \eqref{e:c'-action}, and (ii) follows from (i)
and \eqref{e:cA-cC}.
\end{proof}

\begin{lemma}\label{lem:1c}
Let
\[ 1\to A\to B\to C\to 1\]
be an exact sequence of $\Gamma$-groups.
Let $c\in C^\Ga$ be such that $c^2=1$
(then $c\in \Zl^1 C$\hs).
Assume that $c$ is the image of some  $b\in \Zl^1 B$ (then $\upgam b=b^{-1}$).
Consider the twisted exact sequence
\[ 1\to {}_b A\to {}_b B\to{}_c C\to 1.\]
Then for $[1]\in \Ho^1 A$ we have
\[ [1]\cdot c=[a],\quad\text{where }a=b^{-2}\in A.\]
In particular, if $b^2=1$, then $[1]\cdot c=[1]$.
\end{lemma}

\begin{proof}
By definition, $[1]\cdot c$ is the class of the cocycle
\[b^{-1} 1 \hs^{\gamma*}\hmm b=b^{-1}  b\hs\upgam\hm b\hs\hs b^{-1}
      =b^{-1}  b^{-1}=b^{-2}.\qedhere\]
\end{proof}

\begin{construction}
Let
\begin{equation}\label{e:ABC-c-c}
1\to A\labelto{i} B\labelto{j} C\to 1
\end{equation}
be a short exact sequence of $\Gamma$-groups,
where the subgroup $i(A)$ of $B$ is {central}
that is, contained in the center of $B$).
Let $c\in \Zl^1 C$.
We lift $c$ to an element $b$ of $B$ (which need not to be a cocycle)
and set $a=b\cdot\upgam b\in A$.
One check immediately that $a\in \Zl^2\hm A$ and
that in this way we obtain a well-defined map
\[\Delta\colon \Ho^1 C\to \Ho^2\hm A.\]
\end{construction}

\begin{proposition}[{Serre, \cite[I.5.7, Proposition 43]{Serre1997}}]
\label{p:Serre-43}
If in the exact sequence \eqref{e:ABC-c-c}, $i(A)$ is central in $B$,
then the following sequence of pointed sets is exact:
\begin{align*}
1\to A^\Gamma\to \Bd^\Gamma\to C^\Gamma \labelto{\delta} \Ho^1\hm A\labelto{i_*} \Ho^1 B
\labelto{j_*} \Ho^1 C\labelto{\Delta} \Ho^2\hm A.
\end{align*}
\end{proposition}

\begin{proof} Exercise. \end{proof}

\subsection{Group cohomology with coefficients in a group of order $2p^n$}

We  give an alternative  proof of the   following well-known result
\cite[Section 6, Corollary 1 of Proposition 8]{AW}.
\begin{lemma}
\label{l:2=0}
For $\Gamma=\{1,\gamma\}$, an {\emm abelian} $\Gamma$-group $A$,
and any $\xi\in \Ho^n\hm A$ $(n=1,2)$,
we have $2\xi=0\in \Ho^n\hm A$.
\end{lemma}

\begin{proof}
We prove for $n=1$.
Let $\xi\in \Ho^1\hm A$, $\xi=[c]$, $c\in \Zl^1\hm A$,
 that is, $c\in A$ and $\upgam c=-c$.
Then $2c=c+c=c-\upgam c\in \Bd^1\hm A$. Thus $2\xi=[2c]=0$, as required.

We prove for $n=2$. Let $\xi\in \Ho^2\hm A$, $\xi=[c]$, $c\in \Zl^2\hm A$,
that is, $c\in A$ and $\upgam c=c$.
Then $2c=c+c=c+\upgam c\in \Bd^2\hm A$. Thus $2\xi=[2c]=0$, as required.
\end{proof}

\begin{corollary}\label{c:2m+1}
Let $A$ be a finite  abelian $\Gamma$-group {\emm of odd order} $2m+1$,
where $\Gamma=\{1,\gamma\}$ is a group of order 2.
Then $\Ho^1\hm A=0$ and $\Ho^2\hm A=0$.
\end{corollary}

\begin{proof}
Let $n=1$ or $n=2$, and let $\xi\in \Ho^n\hm A$, $\xi=[c]$,
where $c\in \Zl^n\hm A\subseteq A$.
Then $(2m+1)c=0$, and hence, $(2m+1)\xi=0$.
On the other hand, by Lemma \ref{l:2=0} we have $2\xi=0$. We see that
\[\xi=(2m+1)\xi-m\cdot (2\xi)=0-m\cdot 0=0.\qedhere\]
\end{proof}

\begin{lemma}\label{l:H1-bijective}
Let $A$ be a $\Gamma$-group,and let $M\subset A$ be a {\emm normal abelian}
finite $\Gamma$-subgroup {\emm of odd order}.
Then the canonical map
\[ \Ho^1\hm A\to \Ho^1 (A/M)\]
is bijective.
\end{lemma}

\begin{proof}
First we prove the lemma in the case when $M$ is central, that is, $M\subseteq Z(A)$.
Write $B=A/M$.
Since $M$ is central, the short exact sequence
\[ 1\to M\labelto{i} A\labelto{j} B\to 1\]
induces a cohomology exact sequence
\[\dots \to \Ho^1 M\labelto{i_*} \Ho^1\hm A\labelto{j_*} \Ho^1 B\to \Ho^2 M.\]
By Corollary \ref{c:2m+1} we have    $\Ho^2 M=\{1\}$, and hence, $j_*$ is surjective.
Since $M$ is central,  by Corollary \ref{c:39-cor2} each fiber of $j_*$ comes from $\Ho^1 M$.
Since by Corollary \ref{c:2m+1} we have $\Ho^1 M=\{1\}$,
we see that each fiber of $j_*$ has only one element, that is, $j_*$ is injective.
Thus $j_*$ is bijective, as required.

The  case of noncentral $M$ is a bit more delicate.
Let $[b]\in \Ho^1 B$, $b\in \Zl^1 B$.
Then the obstruction $\Delta(b)$ for lifting $[b]$ to $\Ho^1\hm A$ lives in $\Ho^2\hssh_b M$;
see \cite[I.5.6, Proposition 41]{Serre1997}.
Here $_b M$ denotes the group $M$ twisted by the cocycle
$b\in \Zl^1 B$, where $B$ naturally acts on $M$.
Note that $_b M$ has the same underlying group as $M$, only the $\Gamma$-action changes.
Therefore, $_b M$ is a group of odd order and hence
we have $\Ho^2\hssh_b M=\{1\}$ by Corollary \ref{c:2m+1}.
We see that $\Delta(b)=1$ and $[b]$ lifts to $\Ho^1\hm A$.
Thus $j_*$ is surjective; cf. \cite[I.5.6, Corollary to Proposition 41]{Serre1997}.

Furthermore, for any $a\in \Zl^1\hm A$, the fiber
$j_*^{-1}(j_*([a]))\subset \Ho^1\hm A$  comes from $\Ho^1{}_a M$;
see \cite[I.5.5, Corollary 2 to Proposition 39]{Serre1997}.
We have $\Ho^1{}_a M=\{1\}$ by Corollary  \ref{c:2m+1}
because $_a M$ is of odd order.
Thus each fiber of $j_*$ has only one element, that is, $j_*$ is injective.
Thus $j_*$ is bijective, as required.
\end{proof}

\begin{lemma}\label{l:p}
Let $p$ be an odd prime.
For any finite  $\Gamma$-group $A$ of order $2\hs p^m$, we have  $\#\Ho^1\hm A=2$.
\end{lemma}

\begin{proof}
Let $n_p$ denote the number of Sylow $p$-subgroups in $A$.
By Sylow's Theorem  we have $n_p|2$ and $n_p\equiv 1\pmod p$;
see e.g. \cite[Theorem 4.12]{Rotman1995}.
It follows that $n_p=1$.
Let $P$ be the unique Sylow $p$-subgroup of $A$;
it is normal and $\Gamma$-invariant.
We show by induction on $m$ that the canonical map
\[\Ho^1\hm A\to \Ho^1\hs A/P\]
is bijective.
Since $A/P$  is a group of order 2,
we have $\#\hs\Ho^1 A/P=2$, which will prove the lemma.

We proceed by d\'evissage. If $m=0$, there is nothing to prove.
Assume that $m\ge 1$.

Since $P$ is a nontrivial $p$-group, it has nontrivial center;
denote it by $Z_P$ and its order by $p^k$, where $k\ge 1$.
Then $Z_P$ is a characteristic subgroup of $P$, and therefore,
it is $\Gamma$-invariant and normal in $A$.
Consider the homomorphism of $\Gamma$-groups $A\to A/Z_P$.
Since $Z_P$ is a normal abelian $\Gamma$-subgroup of $A$ of odd order $p^k$,
by Lemma \ref{l:H1-bijective} the canonical map
\[\Ho^1\hm A\to \Ho^1\hs A/Z_P\]
is bijective. We have $\#\hssh A/Z_P\hs=p^l$ with $l=m-k<m$.
By the induction hypothesis, the canonical map
\[ \Ho^1\hs A/Z_P\to \Ho^1\hs A/P\]
is bijective.
Thus the canonical map
\[\Ho^1\hm A\to \Ho^1\hs A/P\]
is bijective, as required.
\end{proof}

\begin{lemma}\label{l:fixed-order2}
Let $p$ be an odd prime.
For any finite  $\Gamma$-group $A$ of order $2\hs p^m$,
there exists a $\Gamma$-fixed element $c$ of order 2.
\end{lemma}

\begin{proof}
Let $S_2$ denote the set of subgroups $H\subset A$ of order 2.
Let $n_2=\# S_2$.
By Sylow's Theorem $n_2$ is odd.

The group $\Gamma$ acts on $S_2$. Since $\# S_2$ is odd,
$\Gamma$ has a fixed point in $S_2$, say $H=\{1,c\}$.
Then $c$ is a $\Gamma$-fixed element of order $2$.
\end{proof}

\begin{lemma}\label{l:explicit}
With the assumptions and  notation of Lemmas \ref{l:p} and \ref{l:fixed-order2},
let $c\in A$ be a $\Gamma$-fixed element of order 2.
Then $c\in \Zl^1\hm A$ and $\Ho^1\hm A=\{[1],[c]\}$.
\end{lemma}

\begin{proof}
We have $c\cdot{}^\gamma c=c^2=1$, and hence $c\in \Zl^1\hm A$.
Consider the canonical map
\[\pi\colon A\to A/P.\]
Clearly, $c\notin P$, hence $\pi(c)\neq 1$,
and therefore, \,$\pi_*[c]\neq [1]\in \Ho^1\hs A/P$ \,
(we use the fact that $A/P$ is a group of order 2).
Thus $\Ho^1\hs A/P=\{[1],\pi_*[c]\}$.
By the proof of Lemma \ref{l:p}, the canonical map
\[\pi_*\colon \Ho^1\hm A\to \Ho^1\hs A/P\]
is bijective.
Thus $\Ho^1\hm A=\{[1],[c]\}$, as required.
\end{proof}

\subsection{Galois cohomology of real algebraic groups}
\label{s:Galois-coh}

\begin{subsec}\label{ss:C-group}
Let $G$ be a linear algebraic group over $\C$.
This means that we have a tuple $(G,\mu,\iota,e)$,
where $G$ is an affine variety over $\C$ (not necessarily irreducible),
where $\mu$ and $\iota$ are morphisms and $e$ is a $\C$-point:
\begin{equation*}
\mu\colon G\times_\C G\to G, \quad \iota\colon G\to G,\quad e\in G(\C),
\end{equation*}
corresponding to multiplication, inversion, and to the neutral element in $G$, respectively.
Certain diagrams containing these morphisms must commute;
see, for instance, \cite[1.a, Definition 1.1]{Milne2017}.
See \cite{Borel1966} for  a few elementary definitions of a linear algebraic group.
We write $G(\C)$ for the group of $\C$-points of $G$; in the classical approach
(see the books by Borel, Humphreys, and Springer titled ``Linear Algebraic Groups'')
the $\C$-group $G$ is identified with $G(\C)$.
The group $G(\C)$ is naturally a complex Lie group.

Let $\sigma\colon G(\C)\to G(\C)$ be an anti-holomorphic involution of $G(\C)$,
that is, an anti-holomorphic group automorphism such that $\sigma^2={\rm id}$.
The automorphism $\sigma$ acts on holomorphic functions on $G$ as follows:
\[(\hs^{\sigma}\!f)(x)=\upgam(f(\sigma^{-1}(x)))\quad\text{for }x\in G(\C),\]
where $\gamma\in\Gamma$ denotes the complex conjugation (and of course $\sigma^{-1}=\sigma$.)
Note that $\sigma$ acts on the constants by the complex conjugation $\gamma$.

Let $\C[G]$ denote the ring of regular functions on $G$.
We say that an anti-holomorphic involution $\sigma$ of $G(\C)$ is {\em anti-regular},
if for any regular function $f\in\C[G])$, the function $^\sigma\! f$ is again regular.
\end{subsec}

\begin{remark}\label{rem:antiholomorphic}
In general, not all anti-holomorphic involutions of $G(\C)$ are anti-regular.
However, if the identity component of $G$ is reductive,
then any anti-holomorphic involution is anti-regular;
see \cite[Lemma 3.1]{AT2018}.
\end{remark}

\begin{subsec}
Let $\GG$ be a linear algebraic group over $\R$, that is,
we have a tuple $(\GG,\boldsymbol{\mu},\boldsymbol{\iota},\boldsymbol{e})$,
where $\GG$ is an affine variety over $\R$ (not necessarily irreducible),
$\boldsymbol{\mu}$ and $\boldsymbol{\iota}$
are $\R$-morphisms and $\boldsymbol{e}$ is an $\R$-point:
\begin{equation}\label{e:R-group}
\boldsymbol{\mu}\colon  \GG\times_\R \GG\to \GG, \quad
    \boldsymbol{\iota}\colon \GG\to \GG,\quad \boldsymbol{e}\in\GG(\R),
\end{equation}
that fit into the  commutative diagrams of \cite[1.a, Definition 1.1]{Milne2017}.
We again refer to \cite{Borel1966}
for elementary definitions of a linear algebraic group over a nonclosed field.
We write $\GG(\R)$ for the real Lie group of the $\R$-points of $\GG$,
and $\GG(\C)$ for the complex Lie group of the $\C$-points of $\GG$.

Set
$$G=\GG_\C:=\GG\times_\R\C.$$
Then the complex conjugation $\gamma$ naturally acts on $G(\C)=\GG(\C)$
and thus defines an anti-regular involution
$\sigma=\sigma_\GG\colon G(\C)\to G(\C)$.
We obtain a functor
\begin{equation}\label{e:functor}
 \GG\rightsquigarrow (G,\sigma).
\end{equation}
\end{subsec}

\begin{proposition}[well-known; see {\cite[Theorem 2.2(b)]{Jahnel}}]
\label{p:equiv}
The functor \eqref{e:functor} from the category
of linear algebraic $\R$-groups to the category of pairs $(G,\sigma)$,
 where $G$ is a linear algebraic $\C$-group and $\sigma\colon G(\C)\to G(\C)$
is an anti-regular involution, is an equivalence of categories.
\end{proposition}

\begin{proof}[Idea of proof]
It suffices to prove that the functor  \eqref{e:functor} is essentially surjective.
We start from a pair $(G,\sigma)$ and construct $\GG$.
Now $G$ is an affine algebraic group over $\C$, that is, we have
a tuple $(G,\mu,\iota,e)$ as in Subsection \ref{ss:C-group}.
We assume that $G$ is a closed subvariety
in the $n$-dimensional affine space $\mathbb{A}^n_\C$
defined by an ideal $I$ in the ring of polynomials $\C[X_1,X_2,\dots, X_n]$.
Then by definition, the $\C$-algebra of regular functions $\C[G]$
is $\C[X_1,X_2,\dots, X_n]/I$.
It has a finite set of generators $x_1,x_2,\dots, x_n$ over $\C$,
the images of $X_1,\dots,X_n$ in $\C[G]$.

The Galois group $\Gamma=\{1,\gamma\}$ naturally acts on $\C[G]$
(namely, $\gamma$ acts by $\sigma$).
We consider $A:=\C[G]^\Gamma$,  the ring of fixed points of $\sigma$ in $\C[G]$.
Consider the set
\[\Xi=\{y_k=x_k+\sigma(x_k),\ z_k=\sqrt{-1}(x_k-\sigma(x_k))\mid k=1,\dots,n\}.\]
Then $y_k,\hs z_k\in A$, and $\Xi$ is a finite set of generators of $A$ over $\R$.
We set $\GG=\Spec A$, the spectrum of the ring $A$,
that is, an affine variety $\GG$ over $\R$ with algebra of regular functions $\R[\GG]=A$.
It is constructed as follows. We consider the algebra of polynomials
$B=\R[Y_1,Z_1,Y_2,Z_2,\dots,Y_n,Z_n]$, and  the surjective  homomorphism
\[\psi\colon B\to A \quad\text{defined by } Y_k\mapsto y_k,\ Z_k\mapsto z_k\hs.\]
Consider the ideal $\boldsymbol{I}=\ker\,\psi\subset B$.
We take for $\GG$   the closed subvariety in the $2n$-dimensional  affine space
$\mathbb A^{2n}_\R$ defined by the ideal $\boldsymbol{I}$.
Then $\R[\GG]=B/\boldsymbol{I}\cong A$, as required.

We have  $\R[\GG]\otimes_\R\C=\C[G]$, and hence, $\GG\times_\R\C=G$.
Since $\sigma$ is a {\em group automorphism} of  $G(\C)$,
the morphisms $\mu$ and $\iota$ commute with the $\Gamma$-action on $\GG(\C)$,
and therefore, they are ``defined over $\R$'', that is,
they come from (unique) morphisms
$\boldsymbol{\mu},\boldsymbol{\iota}$ as in \eqref{e:R-group}.
Similarly, the $\C$-point $e$ comes from an $\R$-point $\boldsymbol{e}$.
Then $(\GG,\boldsymbol{\mu},\boldsymbol{\iota},\boldsymbol{e})$
is a desired affine algebraic group over $\R$, as required.
\end{proof}

\begin{subsec}
From now on, by a {\em linear algebraic group over $\R$}
(for brevity: an $\R$-group) we mean a pair $\MM=(M,\sigma)$,
where $M$ is a linear algebraic group over $\C$,
and $\sigma$ is an anti-regular involution of $M(\C)$.
Then $\Gamma$ naturally acts on $M(\C)$
(namely, the complex conjugation $\gamma$ acts as $\sigma$).
We set
\begin{align*}
\MM(\C):=M(\C),\quad &\MM(\R)=M(\C)^\sigma,\quad
                      \Ho^1(\R,\MM):=\Ho^1(\Gamma, \MM(\C)).
\end{align*}
For brevity we write $\Ho^1\hs\MM$ for $\Ho^1(\R,\MM)$.

A {\em homomorphism} $\MM\to \NN$ of $\R$-groups,
where $\MM=(M,\sigma_M)$ and $\NN=(N,\sigma_N)$,
is a homomorphism $\phi\colon M\to N$
such that  the induced homomorphism $M(\C)\to N(\C)$
is $\Gamma$-equivariant, where $\gamma$ act as
$\sigma_M$ on $M(\C)$, and as $\sigma_N$ on $N(\C)$.
Such a homomorphism $\phi$ induces a homomorphism
\[\phi_\R\colon \MM(\R)\to \NN(\R)\]
and a map
\[\phi_*\colon \Ho^1 \hs\MM\to \Ho^1\hs\NN.\]

Note that if $\MM=(M,\sigma_M)$ and $\NN=(N,\sigma_N)$
are isomorphic, then $M$ and $N$ are isomorphic.
\end{subsec}

\begin{example} Let $\mu_{3,\R}$ denote the group of 3rd roots of unity in $\C^\times$
 with the standard (nontrivial) action of the complex conjugation.
 Then $\mu_{3,\R}(\R)=\{1\}$, but $\mu_{3,\R}\neq 1$, because $\#\mu_{3,\R}(\C)=3$.
\end{example}

\begin{example}
Consider the group $\SL_{9,\R}:=(\SL_{9,\C}, \sigma)$,
where $\sigma$ acts by  complex conjugation of the matrix entries.
We have a canonical embedding
\[\mu_{3,\R}\into \SL_{9,\R},\quad x\mapsto x\hs I,\quad \text{where }I=\diag(1,1,\dots,1).\]
We set
\[G_0=\SL_{9,\R}/\mu_{3,\R}.\]
This means that $G_{0}=\SL_{9,\C}/\mu_{3,\C}$
with the induced anti-regular automorphism of $G_0(\C)$.
Then $G_0$ is not isomorphic to $\SL_{9,\R}$,
because $G_{0}$ is not isomorphic to $\SL_{9,\C}$
(since  $G_{0}$ has center of order 3, while $\SL_{9,\C}$  has center of order 9).
However, the canonical homomorphism
\[\SL_{9,\R}\to G_0\]
induces an isomorphism of real Lie groups
\[ \SL(9,\R)\to \GG_0(\R);\]
see Corollary \ref{c:9-mu3} below.
\end{example}

\begin{lemma}\label{l:H/M}
Let $\HH$ be an $\R$-group, and let $\MM\subseteq \HH$
be a finite abelian normal $\R$-subgroup {\emm of odd order},
that is,  the order of $M$ is odd.
(We do not assume  that $M$ is contained in the center of $H$).
 Then the canonical homomorphism on $\R$-points
\[ \HH(\R)\to (\HH/\MM)(\R)\]
is surjective.
In particular, if the topological group $\HH(\R)$ is connected,
then the topological space $(\HH/\MM)(\R)$ is connected as well.
\end{lemma}

\begin{proof}
The short exact sequence of $\R$-groups
\[1\to \MM\labelto{i} \HH\labelto{j} \HH/\MM\to 1\]
gives rise to a cohomology exact sequence
\[1\to \MM(\R)\labelto{i_{(\R)}} \HH(\R)\labelto{j_{(\R)}}
      (\HH/\MM)(\R)\labelto{\delta} \Ho^1\hs\MM.\]
By assumption the group $M$ is of odd order.
By Corollary \ref{c:2m+1} we have $\Ho^1\hs\MM=\{1\}$, and hence,
the map $j_{(\R)}\colon \HH(\R)\to  (\HH/\MM)(\R)$ is surjective, as required.
If $\HH(\R)$ is connected, then its image $(\HH/\MM)(\R)=j_{(\R)}(\HH(\R))$
under the continuous map $j_{(\R)}$ is connected, as required.
\end{proof}

\begin{corollary}\label{c:connected}
Let $\HH$ be a simply connected semisimple $\R$-group,
and $\GG=\HH/\MM$, where $\MM\subseteq \Zm(\HH)$
is a finite  central $\R$-subgroup {\emm of odd order}
(that is, $M$ is contained in the center of $H$ and is of odd order).
Then the group $\GG(\R)$ is connected.
\end{corollary}

\begin{proof}
Since $\HH$ is simply connected, the group of $\R$-points $\HH(\R)$ is connected;
see \cite[Corollary 4.7]{Borel-Tits},
or Gorbatsevich, Onishchik, and Vinberg  \cite[4.2.2, Theorem 2.2]{GOV},
or Platonov and Rapinchuk \cite[Proposition 7.6 in Section 7.2 on page 407]{PR}.
Since $M$ is of odd order, by Lemma \ref{l:H/M}
the group of $\R$-points $\GG(\R)$ is connected, as required.
\end{proof}

\begin{remark}
The assertion of Corollary \ref{c:connected} becomes false
if we do not assume that $M$ is of odd order.
For example, if we take $\GG=\PGL_{2,\R}=\GL_{2,\R}/\G_{m,\R}=\SL_{2,\R}/\mu_2$,
then $\GG(\R)=\GL(2,\R)/\R^\times$ is not connected.
\end{remark}

\begin{lemma}\label{l:mu-n}
Let $\HH$ be an $\R$-group,
and let $\MM$ be an abelian normal algebraic $\R$-subgroup of $\HH$
(that is, $M$ is an abelian normal $\Gamma$-invariant algebraic subgroup of  $H$).
Assume that $\MM$ is isomorphic over $\R$ to $\mu_{n,\R}$,
where $n$ is an {\emm odd} natural number.
Set $\GG_0=\HH/\MM$.
Then the canonical epimorphism
\[j\colon \HH\to \GG_0\]
induces an isomorphism
\[j_{(\R)}\colon \HH(\R)\isoto \GG_0(\R)\]
and a bijection
\[ j_*\colon \Ho^1\hs\HH\isoto \Ho^1\GG_0\hs.\]
\end{lemma}

\begin{proof}
We have  $\ker j_{(\R)}=\MM(\R)\simeq \mu_n(\R)=\{1\}$ (because $n$ is odd),
hence the homomorphism $j_{(\R)}$ is injective.
Since $M$ is of odd order,  by Lemma \ref{l:H/M}
the homomorphism $j_{(\R)}$ is surjective.
Thus $j_{(\R)}\colon  \HH(\R)\to \GG_0(\R)$ is bijective.
Since $M$ is of odd order, by Lemma \ref{l:H1-bijective}
the map  $j_*\colon \Ho^1\hs\HH\to \Ho^1\GG_0$ is bijective,
which completes the proof.
\end{proof}

\begin{corollary}\label{c:9-mu3}
Set $\GG_0=\SL_{9,\R}/\mu_3$ and  consider the canonical homomorphism
\[j\colon \SL_{9,\R}\to \GG_0\hs.\]
Then induced homomorphism on $\R$-points
\[j_{(\R)}\colon\SL(9,\R)\to \GG_0(\R)\]
is an isomorphism, and $\Ho^1\GG_0=1$.
\end{corollary}

\begin{lemma} \label{lem:g1}
Let
\[\GG=\GL_{n_1, \R}\times \cdots \times \GL_{n_k, \R}\hs
\quad\text{with }n_1+\dots+n_k=9.\]
We naturally  embed $\GG$ into $\GL_{9,\R}$ and set
 \[\SS=\ker[\GG\labelto{\det} \G_{m,\R}],\quad \HH=\SS/\mu_3\hs.\]
Then  $\Ho^1\hs\HH = 1$.
\end{lemma}

\begin{proof}
Since $\mu_3(\C)$ is of odd order 3,
by Lemma \ref{l:H1-bijective} we have
\[ \Ho^1\hs\HH=\Ho^1(\Gamma,S(\C)/\mu_3(\C))\cong \Ho^1(\Gamma,S(\C))=\Ho^1\hs\SS.\]
We have $\Ho^1\GG=1$; see \cite[X.1, Proposition 3]{Serre1979}.
From the short exact sequence
\[ 1\to \SS\to \GG\to \G_{m,\R}\to 1\]
we obtain a cohomology exact sequence
\[\GG(\R)\labelto{\det} \R^\times\labelto{\delta} \Ho^1\SS\to \Ho^1\GG=1,\]
whence
\begin{align*}
 \Ho^1\SS    \cong \coker[\GG(\R)\labelto{\det} \R^\times]=1.
\end{align*}
Thus $\Ho^1\hs\HH=1$, as required.
\end{proof}

\begin{proposition}\label{p:C-3}
Let $\BB$ be an $\R$-group and let $\BB^\circ$
denote the identity component of $\BB$. Assume that
\begin{enumerate}
\item $\BB/\BB^\circ$ is a finite group of order $p^n$
for some {\emm odd} prime $p$ and some integer $n\ge 0$;
\item $\#\Ho^1\hs\BB^\circ<p$.
\end{enumerate}
Then the canonical map $\Ho^1\hs\BB^\circ\hm\to\hm \Ho^1\hs\BB$ is bijective.
\end{proposition}

\begin{proof}
Write $\AA=\BB^\circ$, $\CC=\BB/\BB^\circ$.
Let $\pi\colon \BB\to \CC$ denote the canonical epimorphism.
We have a cohomology exact sequence
\[\CC(\R)\to \Ho^1\hm\AA\to \Ho^1\BB\labelto{\pi_*} \Ho^1\CC,\]
where  $\Ho^1\CC=1$ because $\CC$ is a group of order $p^n$.
We conclude that the map $\pi_*\colon \Ho^1\BB^\circ\to \Ho^1\BB$ is  surjective.
It remains to show that $\pi_*$  is injective.

By Corollary \ref{c:39-cor1} the map $\pi_*$ identifies
$\Ho^1\BB$ with $\Ho^1\BB^\circ/\hs\CC(\R)$ with respect
to a certain action of $\CC(\R)$ on $\Ho^1\BB^\circ$.
Since $\#C=p^n$, we see that $\#\CC(\R)= p^k$ for some natural number $k$.
Since $\#\Ho^1\BB^\circ<p$, we conclude that  $\CC(\R)$
acts  on $\Ho^1\BB^\circ$ trivially.
Thus the map $\pi_*$ is injective, and hence bijective, which completes the proof.
\end{proof}

\begin{proposition}\label{p:C-3-2}
Let $\BB$ be an $\R$-group.
Write $\AA=\BB^\circ$, $\CC=\BB/\BB^\circ$.
Assume that
\begin{enumerate}
\item $\Ho^1\hm\AA$ is trivial;
\item $C$ is a finite group of order $2 p^n$
for some odd prime $p$ and some integer $n\ge 0$,
\item An element c of order 2 in  $\CC(\R)$
(which exists because $C$ is of order $2\cdot p^n$; see Lemma \ref{l:fixed-order2})
lifts to some {\emm cocycle} $b\in \Zl^1\BB$;
\item $\Ho^1\hs_b \AA=1$.
\end{enumerate}
Then $\#\Ho^1\BB=2$ with cocycles $1,b$.
\end{proposition}

\begin{proof}
By Lemma \ref{l:explicit} we have $\#\Ho^1\CC=2$ with cocycles $1,c$.
We have a short exact sequence
\[1\to \AA\labelto{i}\BB\labelto{j} \CC\to 1.\]
By Corollary \ref{c:39-cor1},
the kernel of $j_*\colon \Ho^1\BB\to \Ho^1\CC$ comes from $\Ho^1\hm\AA$.
By assumption, $\Ho^1\hm\AA$ is trivial, and hence $\ker j_*=\{\hs[1]\hs\}$.
By assumption, $[c]$ comes from $[b]$ for some $b\in \Zl^1\BB$.
Thus the fiber $j_*^{-1}[c]$ is not empty.
By Corollary \ref{c:39-cor2}, this fiber is in a bijection
with the quotient of $\Ho^1{}_b \AA$ by an action of the group $(\hs_b \CC)(\R)$.
By assumption $\Ho^1{}_b \AA=1$, and hence $j_*^{-1}[c]=\{\hs[b]\hs\}$.
Thus $\Ho^1\BB=\{\hs[1],[b]\hs\}$, as required.
\end{proof}

\subsection{Weil restriction}

Let $G'$ be linear algebraic $\C$-group.
We assume that $G'\subseteq \GL(n,\C)$.
Consider the group
\[ \gamma\hs G':=\{\upgam g\mid g\in G'\}.\]
Then $\gamma\hs G'$ is again a linear algebraic $\C$-group in $\GL(n,\C)$.
Set
\[G=G'\times \gamma\hs G'\]
and define a map
\[ \sigma\colon G\to G,\quad (g_1\hs,g_2)\mapsto (\upgam g_2\hs,\upgam g_1).\]
It is easy to see that $\sigma$ is an anti-regular involution of $G$.
In  this way we obtain a linear algebraic $\R$-group
\[ \GG=(G'\times \gamma\hs G',\sigma).\]

We have a map
\begin{equation} \label{e:rest}
G'(\C) \to\GG(\R),\quad g\mapsto (g,\upgam g).
\end{equation}
It is easy to see that \eqref{e:rest} is an isomorphism of {\em real}
Lie groups.
In other words, $G(\R)$ is canonically isomorphic to the restriction of scalars
(from $\C$ to $\R$) of the complex Lie group $G'(\C)$.

\begin{definition}\label{d:Res}
$\Res G'=(G'\times \gamma\hs G',\sigma)$.
It is called the {\em Weil restriction} (of scalars from $\C$ to $\R$) of $G'$.
\end{definition}

\begin{remark}
Our construction of $\Res G'$ does not depend on the embedding of $G'$ into
$\GL(n,\C)$.
Indeed, one can obtain $\gamma\hs G'$ from $G'$ using the base change
$\gamma\colon \C\to \C$;  see \cite[Sections 1.1 and 1.15]{Borovoi2020}
(where $\gamma\hs G'$ is denoted $\gamma_* G'$).
\end{remark}

\begin{proposition}[well known]
\label{p:Weil}
Let $G'$ be a linear algebraic $\C$-group, and write $\GG=\Res G'$. Then:
\begin{enumerate}
\item[\rm (i)] $\Ho^1\hs\GG=\{1\}$;
\item[\rm (ii)] if, moreover, $G'$ is abelian, then $\Ho^2\hs\GG=\{1\}$.
\end{enumerate}
\end{proposition}

\begin{proof}
If $z\in\Zl^1\hs\GG$, then $z=(g,\upgam g^{-1})$ for some $g\in G'$.
Set $h=(g,1)\in G$; then
\[ z=h\cdot\sigma(h)^{-1}\sim 1,\]
which proves (i).

If $G'$ is abelian, then $\GG$ is abelian as well.
Let $z\in\Zl^2\hs\GG$; then $z=(g,\upgam g)$ for some $g\in G'$.
Set $h=(g,1)\in G$; then
$$ z=h\cdot \sigma(h)\sim 1,$$
which proves (ii).
\end{proof}

\begin{remark}
Proposition \ref{p:Weil} is a trivial special case of a lemma
of D.\,K. Faddeev and Arnold Shapiro;
see \cite[I.5.8(b) and I.2.5, Corollary of Proposition 10]{Serre1997}.
\end{remark}

\subsection{Galois cohomology of real tori}

As above, a {\em real algebraic torus} (for brevity: an $\R$-torus)
is a pair $\TT=(T,\sigma_T)$, where $T$ is a complex algebraic torus
and $\sigma_T\colon T\to T$ is an anti-regular involution.

\begin{proposition}[well known,  see, e.g., {\cite[Section 10.1]{Vos}}\hs]
\label{p:indecomposable}
Any $\R$-torus is a product of tori isomorphic to one
of these three indecomposable $\R$-tori:
$\R^\times$ , $U(1)$, $\C^*$.
\end{proposition}

Note that the real torus $\R^\times$ is also known as $\mathbb{G}_{m,\R}$,
and the real torus $\C^*$ is also known as $R_{\C/\R}\mathbb{G}_{m,\C}$.

\begin{definition}\label{def:tauT}
Let $\TT=(T,\sigma_T)$ be a real torus.
We write
$$X_T={\sf X}_*(T)=\Hom(\mathbb{G}_{m,\C}, T)$$
(the cocharacter group). We define an involution $\tau_T$ of $X_T$ by
\[\sigma_T(x(z))=(\tau_T(x))(\ov z)\quad \text{for }x\in X_T,\ z\in \C^\times.\]
\end{definition}
This corresponds to the formula
\[\ov{x(z)}=\ov x(\ov z).\]

\begin{construction}\label{cons:e*}
We define a homomorphism
\[e\colon X_T\to \TT(\C),\quad x\mapsto x(-1).\]
We have
\[\sigma_T(e(x))=\sigma_T(x(-1))=(\tau_T(x))(\ov{-1})=(\tau_T(x))({-1})=e(\tau_T(x)).\]
Thus the homomorphism $e$ is $\Gamma$-equivariant,
where the complex conjugation $\gamma\in \Gamma=\Gal(\C/\R)$
acts on $T(\C)$ via $\sigma_T$ and on $X_T$ via $\tau_T$.
For $n=1,2$ we obtain an induced homomorphism
\[e_*^{(n)}\colon \Ho^n X_T\to \Ho^n\hssh \TT.\]
\end{construction}

\begin{proposition}\label{p:e*}
The homomorphism $e_*^{(n)}\colon \Ho^n X_T\to \Ho^n\hssh \TT$ is an isomorphism.
\end{proposition}

\begin{proof}
It sufficed to prove the proposition for the three indecomposable tori
of Proposition \ref{p:indecomposable}, which is easy.
\end{proof}

\begin{definition}
By a {\em $\Gamma$-lattice} we mean a finitely generated
abelian group $X$ with a $\Gamma$-action.
We say that $X$ is a {\em permutation lattice}
if it has a $\Gamma$-stable basis.
In other words, a permutation $\Gamma$-lattice
is a direct sum of $\Gamma$-lattices
isomorphic to $\Z_\triv$ and $\Z^2_{\tw+}$ of Examples \ref{x:cohom-Z}.
\end{definition}

\begin{proposition}\label{p:permutation}
Let $X$ be a permutation $\Gamma$-lattice. Then $\Ho^1\hs X=1$.
\end{proposition}

\begin{proof}
The proposition follow immediately from Examples
 \ref{x:cohom-Z}(1) and  \ref{x:cohom-Z}(5).
\end{proof}

\begin{definition}
An $\R$-torus $\TT$ is called {\em quasi-trivial}
if its cocharacter group $X_T$ is a permutation $\Gamma$-module.
In other word, $\TT$ is quasi-trivial if and only
if it is isomorphic to a direct product of tori of types $\C^\times_\stand$
and $(\C^\times)^2_{\tw_+}$ of Examples \ref{x:cohom}.
\end{definition}

\begin{proposition}[{Sansuc}]
\label{p:quasi}
Let $\TT$ be a quasi-trivial $\R$-torus.
Then $H^1\hs\TT=1$ if and only if $\TT$ is quasi-trivial.
\end{proposition}

\begin{proof}
If $\TT$ is quasi-trivial, then by Proposition \ref{p:permutation}
we have $\Ho^1\hs\X_*(\TT)=1$,
and by Proposition \ref{p:e*} we have $\Ho^1\hs\TT=1$.

Conversely, suppose that  $H^1\hs\TT=1$.
We can write $\TT$ as a direct product of indecomposable $\R$-tori.
Since $H^1\hs\TT=1$, this product does not contain direct factors isomorphic to $U(1)$.
It follows that  $\TT$ is quasi-trivial.
\end{proof}

\begin{remark}
The ``if'' assertion of Proposition \ref{p:quasi}
is a special case of  \cite[formula (1.9.1)]{Sansuc1981}.
\end{remark}

\subsection{Using Galois cohomology for finding real orbits in homogeneous spaces}

\begin{subsec}
Let $\GG$ be a real algebraic group acting (over $\R$)
 on a real algebraic variety $\YY$.
We write $\GG=(G,\sigma)$, $\YY=(Y,\mu)$.
Here $\mu$ is an anti-holomorphic involution of $Y$ given by $\mu(y)=\upgam y$.
The assertion that {\em $\GG$ acts on $\YY$ over $\R$}
means that $G$ acts on $Y$ by $(g,y)\mapsto g\cdot y$,
and this action is $\Gamma$-equivariant:
for $g\in G$, $y\in Y$ we have
\[ \upgam(g\cdot y)=\upgam g\cdot\hm\upgam y,\quad\text{that is,}
    \quad \mu(g\cdot y)=\sigma(g)\cdot \mu(y).\]

We assume that $\GG$ acts on $\YY$ {\em transitively},
that is, $G$ acts on $Y$ transitively.
\end{subsec}

\begin{subsec}
We have $\GG(\R)=G^\sigma$, $\YY(\R)=Y^\mu$.
The group $\GG(\R)$ naturally acts on $\YY(\R)$,
and this action might not be transitive.
Assuming that $\YY(\R)$ is {\em nonempty},
we describe the set of orbits $\YY(\R)/\GG(\R)$ in terms of Galois cohomology.

Let $e\in \YY(\R)=Y^\mu$ be an $\R$-point.
Let $C=\Stab_G(e)$.
Since $\sigma(e)=e$, we have $\sigma(C)=C$.
We consider the real algebraic subgroup
\[ \CC=\Stab_\GG(e)=(C,\sigma_C),\quad\text{where }\sigma_C=\sigma|_C\hs,\]
and the homogeneous space $\GG/\CC$, on which $\Gamma$ acts by $\upgam(gC)=\upgam g\cdot C$.
We have a canonical bijection
\[G/C\isoto Y, \quad gC\mapsto g\cdot e,\]
and an easy calculation shows that this bijection is $\Gamma$-equivariant.
Taking into account Proposition \ref{p:serre}, we obtain bijections
\begin{equation}\label{e:ker-G/C-Y/C}
\ker[\hs\Ho^1\hs\CC\to\Ho^1\hs\GG\hs]\isoto (G/C)^\Gamma/\GG(\R)\isoto\YY(\R)/\GG(\R).
\end{equation}
\end{subsec}

\begin{construction}\label{con:e-c}
We describe explicitly the composite bijection \eqref{e:ker-G/C-Y/C}.
Let $c\in\Zl^1\CC$ be such that $[c]\in \ker[\Ho^1\hs\CC\to\Ho^1\hs\GG]$.
Then there exists $g\in G$ such that $c=g^{-1}\cdot\upgam g$.
We set $e_c=g\cdot e\in Y$.
We compute
\[\mu(e_c)=\mu(g\cdot e)=\sigma(g)\cdot\mu(e)=gc\cdot e=g\cdot(c\cdot e)=g\cdot e=e_c\hs.\]
Thus $e_c\in\YY(\R)$.
To $c$ we associate the $\GG(\R)$-orbit $\GG(\R)\cdot e_c\subseteq\YY(\R)$.
\end{construction}

\begin{proposition}\label{p:coh-orbits}
The correspondence $c\mapsto e_c$ of Construction \ref{con:e-c} induces a bijection
 \[\ker[\hs\Ho^1\hs\CC\to\Ho^1\hs\GG\hs]\isoto\YY(\R)/\GG(\R).\]
\end{proposition}

\begin{proof}
The proposition follows from Proposition \ref{p:serre}.
\end{proof}

\begin{subsec}
Let $\GG$ be a real algebraic group and let  $h\in\g:=\Lie \,\GG$.
Then $\GG$ acts on $\g$ by the adjoint representation.
Let $\YY_h$ be the orbit of $h$, that is, the real algebraic variety defined by
\[Y_h=\{h'\in \g^\cC \mid h'=\Ad(g)\cdot h \ \text{ for some } g\in G\},\]
where $\g^\cC=\g\otimes_\R\C$.
Let $\ZZ_h$ denote the centralizer of $h$ in $\GG$.
The group $\GG$ acts transitively on the left on $\YY_h$
(that is, $G$ acts transitively on $Y_h$) by
\[(g,h')\mapsto \Ad(g)\cdot h',\quad g\in G,\,h'\in Y_h\]
with stabilizer $\ZZ_h$ of $h$.
\end{subsec}

\begin{corollary}[from Proposition \ref{p:coh-orbits}]
The set of the $\GG(\R)$-conjugacy classes in $\YY_h(\R)$
is in a canonical bijection with
\[\ker[\Ho^1\hs \ZZ_h\to\Ho^1\hs\GG].\]
\end{corollary}

\begin{corollary}\label{c:Rconj}
If $\Ho^1\hs \ZZ_h=1$, then any element $h'\in\g$ that is conjugate to $h$
over $\C$, is conjugate to $h$ over $\R$.
\end{corollary}

\subsection{Using real orbits for calculating Galois cohomology}

We consider  $\GG=\SU(m,n)$, the special unitary group of the Hermitian form
$F_{m,n}$ on $V$ with matrix
$$M_{m,n}=\diagg(a_1,\dots,a_{m+n}),\quad\text{where }\,
                 a_1=\dots=a_m=1,\ \ a_{m+1}=\dots, a_{m+n}=-1.$$
Here by abuse of notation we denote by $\SU(m,n)$
both the real algebraic group $\GG$
and the group of real points $\GG(\R)$.

\begin{proposition}[well-known]
\label{p:U(m,n)}
$ \Ho^1\hs \SU(m,n)$ is in a canonical bijections with the set
$$S_{m,n}=\{k\in\Z\mid  0\le k\le m+n,\ k\equiv m\!\!\!\pmod{2}\}$$
with explicit cocycles $c^{(k)}=\diagg(c^{(k)}_1,\dots,c^{(k)}_{m+n})$, where:\\
(i) if $k\le m$, then $c^{(k)}_i=-1$ for $k<i\le m$, and 1 for the rest of $i$;\\
(ii) if $k>m$, then  $c^{(k)}_i=-1$ for $m<i\le k$, and 1 for the rest of $i$.
\end{proposition}

\begin{proof}
Write $r=m+n$. Consider the set of  Hermitian
$r\times r$-matrices $M$ of determinant $(-1)^n$.
It is the set of $\R$-points of a certain real algebraic variety $\YY$.
The real algebraic group $\GG=R_{\C/\R}\SL(r,\C)$
acts on $\YY$ by
\begin{equation}\label{e:U}
g\colon M\mapsto g\cdot M\cdot  \ov g^{\hs t},\quad
      \text{where }g\in \SL(r,\C)=\GG(\R),\ M\in \YY(\R),
\end{equation}
and this action extends to a polynomial action of $\GG(\C)$ on $\YY(\C)$.
One checks that the latter action is transitive.
The stabilizer in $\GG(\R)$ of the matrix $M_{m,n}\in\YY(\R)$  is $\SU(m,n)$.
Now by Proposition \ref{p:coh-orbits} we have a bijection
\[ \Ho^1\hs\SU(m,n)=\ker[\hs \Ho^1\hs\SU(m,n)\to \Ho^1\hs \GG]\,\labelto\sim \YY(\R)/\GG(\R),\]
where $\Ho^1 \GG=\{1\}$ by Proposition \ref{p:Weil}(i).
Using the Hermitian version of Sylvester's law of inertia,
we see that the Hermitian matrices $M_{k,\hs m+n-k}$ for $k\in S_{m,n}$
represent the orbits of $\GG(\R)=\SL(r,\C)$ in $\YY(\R)$.
Thus we know that $\# \Ho^1\hs\SU(m,n)=\# S_{m,n}$.
It remains to compute explicit cocycles.

It suffices to consider the case $r=1$, $c=-1$.
Then our matrix $M=M_{m,n}$ is either $M_{1,0}=1$ or $M_{0,1}=-1$.
We embed $\GG':=R_{\C/\R}\C^\times$ into $\GL(2,\R)$ by
$$a+bi\mapsto\SmallMatrix{a&b\\-b&a}=aI+bJ,\quad\text{where }\,
     I=\SmallMatrix{1&0\\0&1},\ J=\SmallMatrix{0&1\\-1&0}.$$
Then $c=\diagg(-1,-1)$.
According to Construction \ref{con:e-c}, we look for $g\in \GG'(\C)$
such that $g^{-1}\cdot\ov g=c=\diagg(-1,-1)$.
We may take $g=\diagg(i,i)\in \GG'(\C)\subset \GL(2,\C)$.
Then we have $\ov g^{\hs t}=g$ (sic!) in the sense of \eqref{e:U}.
 Indeed, $g=aI+bJ$, where $a=i$, $b=0$.
Therefore, $\ov g=aI-bJ=aI+bJ=g$, because $b=0$.
Thus $g\cdot M:= g\hs M\hs  \ov g^{\hs t}=g^2 M=-M=cM$.

Returning to the case of arbitrary $m$ and $n$, we obtain that to $c^{(k)}$
we associate the orbit of $c^{(k)}\hs M_{m,n}$\hs.
The formulas for $c^{(k)}$ are written in such a way that
\[c^{(k)} \hs M_{m,n}=M_{k,m+n-k}\quad \text{for all }\, k\in S_{m,n}\hs,\]
which completes the proof.
\end{proof}

\begin{example}
Let $\GG=\SU(1,2)$. Then $S_{1,2}=\{1,3\}$, whence $\#\Ho^1\hs\GG=2$.
Our explicit cocycles  are $c^{(1)}=\diagg(1,1,1)$ and $c^{(3)}=\diagg(1,-1,-1)$.
Our representatives of orbits are
\[c^{(1)}\cdot M_{1,2}=\diagg(1,-1,-1)=M_{1,2}\quad \text{and}
      \quad  c^{(3)}\cdot M_{1,2}=\diagg(1,1,1)=M_{3,0}\hs.\]
\end{example}

\subsection{Using $\Ho^2$ for finding a real point in a complex orbit}

\begin{subsec}\label{sec:findreal}
In order to use Construction \ref{con:e-c}, we need a real point.
We describe a general method of finding a real point using second Galois cohomology,
following an idea of Springer \cite[Section 1.20]{Springer1966};
see also \cite[Section 7.7]{Borovoi1993} and \cite[Section 2]{DLA2019}.
We choose a {\em $\C$-point }$e\in Y$ and set $C=\Stab_G(e)\subset G$.
In this article we consider only the special case when
the complex algebraic group $C$ is {\em abelian};
this suffices for our applications.

Consider  $\mu(e)\in Y$. We have
\[e=g\cdot \mu(e)\quad\text{for some } g\in G,\]
because $Y$ is homogeneous.
Since $\mu^2=\id$, we have
\begin{align*}
e=\mu(\mu(e))=\mu(g^{-1}\cdot e)=&\sigma(g^{-1})\cdot\mu(e)
=\sigma(g^{-1})\cdot g^{-1}\cdot e=(g\cdot\sigma(g))^{-1}\cdot e.
\end{align*}
Thus $g\cdot\sigma(g)\in C$.
We set
\[d=g\cdot\sigma(g)\in C.\]

We define an anti-regular involution $\nu$ of $C$.
Let $c\in C$.  We calculate:
\begin{align*}
c\cdot e=e\quad &\Rightarrow \quad \sigma(c)\cdot\mu(e)=\mu(e)\quad \Rightarrow \quad
   g\cdot\sigma(c)\cdot\mu(e)=g\cdot\mu(e) \quad \Rightarrow \\
&\Rightarrow\quad  g\hs\sigma(c)g^{-1}\cdot g\cdot\mu(e)=g\cdot\mu(e) \quad \Rightarrow \quad
   g\hs\sigma(c)g^{-1}\cdot e=e.
\end{align*}
We see that $ g\hs\sigma(c)g^{-1}\in C$. We set
\begin{equation}\label{e:nu}
\nu(c)=g\hs\sigma(c)g^{-1}\in C.
\end{equation}
We compute
\begin{align*}
\nu^2(c)=\nu(\nu(c))=g\cdot \sigma(\nu(c))\cdot g^{-1}
   &=g\cdot\sigma(g\hs\hs\sigma(c)\hs g^{-1})\cdot g^{-1}\\
&=(g\hs\hs\sigma(g))\cdot c\cdot (g\hs\hs\sigma(g))^{-1}=c,
\end{align*}
because $g\hs\hs \sigma(g)\in C$ and $C$ is abelian.
Thus $\nu^2=1$, that is, $\nu$ is an involution,
and it is clear from \eqref{e:nu} that it is anti-regular.
We have
\[\nu(d)=g\hs\sigma(d) g^{-1}=g\hs\sigma(g)\hs g\hs g^{-1}=g\hs\sigma(g)=d.\]
This $d\in C^\nu$.

We have constructed an anti-regular involution
\[\nu\colon C\to C,\quad c\mapsto g\hs\hs\sigma(c)g^{-1}\]
and an element
\[d=g\hs\hs\sigma(g)\in C^\nu.\]
\end{subsec}

\begin{subsec}
We consider  the real algebraic group $\CC=(C,\nu)$.
Recall that
\[\Ho^2\hs\CC=\Zl^2\hs\CC/\Bd^2\hs\CC,\]
where
\begin{align*}
\Zl^2\hs\CC=C^\nu:=\{c\in C\mid\nu(c)=c\},\\
\Bd^2\hs\CC=\{c'\cdot\nu(c')\mid c'\in C\}.
\end{align*}

We define the {\em class of $Y$}
$$\Cl(Y)=[d]\in \Ho^2\hs\CC.$$
\end{subsec}

\begin{lemma}
The real algebraic group $\CC=(C,\nu)$
and the cohomology class $\Cl(Y)\in \Ho^2\hs\CC$
do not depend on the choices of $e\in Y$ and $g\in G$
(up to a canonical automorphism).
\end{lemma}

\begin{proof}
We check that $\CC=(C,\nu)$ and $\Cl(Y)\in \Ho^2\hs\CC$
do not depend on the choice of $g\in G$.
Recall that we have chosen $g$ such that
\[e=g\cdot \mu(e).\]
If $g'\in G$ is another element such that
\[e=g'\cdot \mu(e),\]
then $g'=cg$ for some $c\in C$.
For any $c_1\in C$, we obtain
\[\nu'(c_1):=g'\hs\hs\sigma(c_1)\hs (g')^{-1}=c\hs g\hs\hs\sigma(c_1)\hs
      g^{-1}c^{-1}=c\hs\hs\nu(c_1)\hs c^{-1}=\nu(c_1),\]
because the group $C$ is abelian.
We have
\begin{equation}\label{e:d'-d}
\begin{aligned}
d':=g'\hs\hs\sigma(g')=c\hs g\hs\hs\sigma(c\hs g)=&c\hs g\hs\hs\sigma(c)\hs\sigma(g)\\
                      =&c\cdot g\hs\hs\sigma(c)\hs g^{-1}\cdot g\hs\hs\sigma(g)
                      =c\cdot\nu(c)\cdot d.
\end{aligned}
\end{equation}
It follows that
\[ [d']=[d]\in \Ho^2\hs\CC.\]

Now we check that $\CC=(C,\nu)$ and $\Cl(Y)\in \Ho^2\hs\CC$
do not depend on the choice of $e\in Y$.
Let $e'\in Y$ be another point.
Choose $h\in G$ such that $e'=h\cdot e$.
Write $C'=\Zm_G(e')$. We have an isomorphism
\[\alpha\colon C\to C',\quad c\mapsto hch^{-1},\]
Since the group $C$ is abelian, the isomorphism $\alpha$
does not depend on the choice of $h$.

We have
\[\mu(e')=\sigma(h)\cdot\mu(e)=\sigma(h)\cdot g^{-1}\cdot e
      =\sigma(h)\hs g^{-1} h^{-1}\cdot e'.\]
We set $g'=h\hs g\hs\sigma(h)^{-1}$; then
\[\mu(e')=(g')^{-1}\cdot e'.\]
For $c'\in C'$ we set
\[\nu'(c')=g'\hs \sigma(c')\hs\hs(g')^{-1}
    =h g\sigma(h)^{-1}\cdot\sigma(c')\cdot \sigma(h) g^{-1} h^{-1}.\]
We obtain
\begin{align*}
\nu'(\alpha(c))=&hg\hs\sigma(h^{-1})\cdot\sigma(hch^{-1})\cdot \sigma(h) g^{-1}h^{-1}\\
=&hg\sigma(c)g^{-1}h^{-1}=h\hs\nu(c) h^{-1}=\alpha(\nu(c)).
\end{align*}
Thus $\alpha$ is an isomorphism of real algebraic groups
\[\alpha\colon (C,\nu)\isoto(C',\nu')\]
and induces an isomorphism on $\Ho^2$
\[\alpha_*\colon \Ho^2\hs\CC\isoto \Ho^2\hs\CC'.\]
Set
\[d'=g'\sigma(g')\in C'.\]
Then
\[
d':=g'\sigma(g')=hg\hs\sigma(h^{-1})\cdot\sigma(hg\hs\sigma(h^{-1}))
      =hg\hs\sigma(g)h^{-1}=\alpha(g\hs\sigma(g))=\alpha(d).
\]
It follows that
\[[d']=\alpha_*[d],\]
as required. We conclude that the real algebraic group $\CC=(C,\nu)$
and the cohomology class $\Cl(Y)\in \Ho^2\hs\CC$
indeed  do not depend on the choices made
(up to a canonical isomorphism).
\end{proof}

\begin{proposition}
If $Y$ has a $\mu$-fixed point, then $\Cl(Y)=1$.
\end{proposition}

\begin{proof}
Assume that $Y$ has a $\mu$-fixed point $y$, that is, $y=\mu(y)$.
Since $\Cl(Y)$ does not depend on the choices of $e$ and $g$, we may take $e=y$ and $g=1$.
We obtain $d=g\cdot\sigma(g)=1$ and $\Cl(Y)=[1]=1$, as required.
\end{proof}

\begin{proposition}\label{p:mu-fixed}
If  $\Cl(Y)=1$ and $\Ho^1\hs\GG=1$, then $Y$ has a $\mu$-fixed point.
\end{proposition}

\begin{proof}
Choose $e\in Y$ and $g\in G$ such that $e=g\cdot\mu(e)$.
Write
\[d=g\cdot\sigma(g)\in \Zl^2(\R,C)=C^\nu.\]
By assumption $[d]=1$, that is,
\[c\cdot\nu(c)\cdot d =1\quad\text{for some } c\in C.\]
Set $g'=cg$.
By \eqref{e:d'-d} we have
\[g'\cdot\sigma(g')=c\cdot\nu(c)\cdot d=1.\]
Thus $g'$ is a 1-cocycle, $g'\in \Zl^1\GG$.

By assumption $\Ho^1\GG=1$.
It follows that there exists
$u\in G$ such that
\[u^{-1}\hs g'\sigma(u)=1.\]
Set $y=u^{-1}\cdot e\in Y$.
Then
\begin{align*}
\mu(y)&=\sigma(u^{-1})\cdot\mu(e)=\sigma(u)^{-1}g^{-1}\cdot e=\sigma(u)^{-1}g^{-1}\cdot c^{-1}\cdot e\\
          &=\sigma(u)^{-1}(g')^{-1}\cdot e=\sigma(u)^{-1 }(g')^{-1}u\cdot y=\big(u^{-1}g'\hs\sigma(u)\big)^{-1}\cdot y=y.
\end{align*}
Thus $y$ is a $\mu$-fixed point in $Y$, as required.
\end{proof}

\begin{remark}
The proof of Proposition \ref{p:mu-fixed} gives an algorithm
of finding a real (that is, $\mu$-fixed) point in $Y$
or proving that there is no such point, assuming that $C$ is abelian and $\Ho^1\GG=1$.
\end{remark}

\begin{remark}
If we do not assume that $C$ is abelian, then a homogeneous space $\YY$ of $\GG$
defines an {\em $\R$-kernel $\kappa$ for $C$} and a {\em cohomology class} $\Cl(Y)\in \Ho^2(C,\kappa)$.
The second nonabelian cohomology  set $\Ho^2(C,\kappa)$ has a {\em distinguished subset
of neutral elements} ${\rm N}^2(C,\kappa)$, which may have more than one element.
See \cite[Section 1.20]{Springer1966}, or \cite[Section 7.7]{Borovoi1993},
or \cite[Section 2]{DLA2019} for details.
As above, one can show that if $Y$ has a real point, then $\Cl(Y)$ is neutral.
Conversely, if $\Cl(Y)$ is neutral and $\Ho^1\GG=1$, then $Y$ has a real point.
The proofs of these assertions give an algorithm of finding a real point of $Y$
or proving that $Y$ has no real points (in the case when  $\Ho^1\GG=1$).
\end{remark}


\def\Zgg{{\mathfrak z}}
\def\wt{\widetilde}
\def\GGtil{{\wt \GG}}

\section{Semisimple  and nilpotent elements  in  real semisimple \\ $\Z_m$-graded  Lie algebras}\label{sec:gradedLie}

In this section   we  first   compile   and  prove some    important   properties   of   $\Z_m$-graded semisimple Lie algebras  $\g$ over    a  field  $k$ of characteristic  0  (Proposition \ref{prop:gradedjordan}, Theorem \ref{thm:3}).
 We focus mainly on the case $k =\C$  or $k = \R$  and   consider the $\theta$-group  of    $\Z_m$-graded
	semisimple Lie algebras $\g$  over $k =\C$ or $k =\R$ (Remark \ref{rem:grading}). The action of the $\theta$-group defines  an equivalence  relation on  $\g_1$  and  on the set of  Cartan subspaces in $\g_1$. For $k = \R$, we  describe the  equivalence  classes  of  homogeneous nilpotent  elements
and the equivalence classes  of  Cartan subspaces in  $\g_1\subset \g$ in terms of  Galois cohomology  and the structure  of the associated   objects  in  the complexified  $\Z_m$-graded Lie algebra $\g ^\cC : = \g \otimes _\R \C$ (Theorems \ref{thm:galois1}, \ref{thm:cartan}).
 We   give a  Levi decomposition  of the   centralizer  $\Zz_{G_0} (e)$ of a  homogeneous    nilpotent  element $e$  in a complex $\Z_m$-graded  semisimple  Lie algebra $\g^\cC$ (Theorem  \ref{thm:misha}).  As a consequence, we   give a  Levi decomposition  of  the   stabilizer  algebra  of a    homogeneous mixed element in   a complex semisimple Lie algebra $\g^\cC$ (Theorem \ref{thm:centrmix}).
We then study the conjugacy of semisimple elements in a fixed Cartan subspace
(Section \ref{subs:weyl}) using the Weyl group of the graded Lie algebra.
We derive a number of results that will be used in Section \ref{sec:semsim}
for the classification of the real semisimple elements.
Finally we describe  the equivalence classes of  homogeneous
elements  in  terms    of the  equivalence  classes of   its semisimple  and nilpotent  parts in the Jordan   decomposition (Proposition  \ref{prop:mix2}).

\subsection{$\Z_m$-grading on  real  semisimple  Lie algebras  and their  complexification}\label{subs:zmgrading}
Let $\Z_m$ denote $\Z/m\Z$ when $m\geq 2$,   and $\Z$ when $m=\infty$.
Let
$$\g = \bigoplus_{i\in \Z_m} \g_i$$
be a real semisimple $\Z_m$-graded Lie algebra with complexification
$$\g^\cC = \bigoplus_{i\in \Z_m} \g^\cC_i\hs.$$
When $m<\infty$, the $\Z_m$-grading  on  $\g^\cC$ defines an automorphism  $\theta \in \Aut\,\g ^\cC$ such that \cite{Vinberg1975,Vinberg1976}
\begin{equation}\label{e:theta}
 \theta (x) = \om^k   \cdot  x \text{ for }  x \in \g_k^\cC\hs,
 \end{equation}
where  $\om = \exp  {\frac{2\pi\sqrt{-1}} {m}}$.
Conversely, any   automorphism $\theta \in \Aut(\g^\cC)$
	such that  $\theta^m = \id$ defines  a $\Z_m$-grading  on $\g^\cC$ where  $\g_k^\cC$ is
	the eigen-subspace corresponding  to the  eigenvalue $\om^k$ of  the   action of $\theta$.
	When $m = \infty$,  we define $\theta \in \Aut\,\g ^\cC$
by the formula \eqref{e:theta} with $\om=2$.  Vice versa, if $\theta$ is a semisimple automorphism of $\mathfrak{g}^{\mathbb C}$ with eigenvalues in the multiplicative group generated by $2\in \mathbb{C}$, we can define a $\mathbb{Z}$-grading from $\theta$.

Let  $\GGtil=(\Gtil,\tilde\sigma)$  denote the connected simply connected
real algebraic group with Lie algebra $\g^\cC$.
Then the   automorphism $\theta$ of $\g$  that   defines  the $\Z_m$-grading  on $\g$
can be   lifted   uniquely to an automorphism  $\Theta$  of  $\Gtil$.
Let $G_0$ denote the  algebraic subgroup  of fixed points of $\Theta$ in $\Gtil$.
Then clearly  $\Lie\,G_0=\g_0^\cC$.
The automorphism $\theta$ of $\g$ is semisimple, and by
\cite[Theorem 8.1]{Steinberg1968} the complex reductive group $G_0$ is connected. (We can also  apply a  result by Rashevski   on the  connectedness  of the   subgroup of fixed points of  an automorphism $\chi$  of a simply connected Lie group $G$, providing that  $\chi$ is of compact type, i.e.  for  any $g \in G$  the set $\{ \chi ^n g|\,  n \in \Z\}$ is relatively compact in $G$  \cite[Theorem A]{Rashevski1974}).
Since the Lie subalgebra $\g_0^\cC$ is defined over $\R$, and $G_0$ is connected,
we see that $G_0$ is defined over $\R$, that is, $\tilde\sigma$-stable.
We write $\GG_0=(G_0,\sigma_0)$ for the corresponding $\R$-group,
where $\sigma_0$ is the restriction of $\tilde\sigma$.

Note that the restriction  of the  adjoint action of  $\GGtil$ on $\g$ to $\GG_0$   preserves the  $\Z_m$-gradation on $\g$ (because the similar assertion
is true for $\g_0$\hs,  and $G_0$ is connected).
Let $\rho_k$ denote the corresponding action of $\GG_0$ on $\g_k$.
We wish to classify the $\GG_0(\R)$-orbits in $\g_k$.
If $k\neq 1$, then following \cite{Vinberg1976}, we can consider
a new $\Z_{\bar m}$-graded Lie algebra $\bar \g$, $ \bar m = \frac{m}{(m,k)}$
and $\bar \g_p: = \g_{pk}$ for $p \in \Z_{\bar m}$.
Now  we can regard the adjoint action
of  $\GG_0$  on $\g_k$ as the action of  $\GG_0$  on $\bar \g_1$.
Thus in what follows we consider the $\GG_0(\R)$-orbits in $\g_1$ under the action $\rho_1$.

\begin{remark}\label{rem:grading}  (1)   Since  $\rho_1 (G_0) \subset \Aut (\g_1 ^\cC)$ is  the $\g_1$-component  of the restriction  of the  adjoint   group  of $\g^\cC$ to the connected  subgroup whose Lie algebra is $\g_0^\cC$,
	we can replace $\Gtil$  (resp. $G_0$) by any connected  Lie     group  $G$ having the Lie algebra  $\g^\cC $  (resp. the connected  Lie subgroup  in $G$  whose  Lie algebra  is $\g_0 ^\cC $)   and we   still  have the  same  action  $\rho_1  (G_0)$ on $\g_1^\cC$, which  is called   the  $\theta$-group  action of the  underlying  $\Z_m$-graded  semisimple Lie  algebra. The orbits of the  $\theta$-group action  define   an equivalence relation  in $\g_1^\cC $ (resp.  of $\g_1$).     If we   want   to use
	 the corresponding automorphism  $\Theta$  of $G$ whose differential is  $\theta$,   we can  consider both the simply connected  group $\Gtil$  and   the adjoint  group  $\Ad_\g = \Gtil/Z(\Gtil)$, since  $\Theta (Z(\Gtil)) = Z(\Gtil)$.
	
	(2) Let  us consider  a       $\Z_m$-graded  semisimple Lie algebra    over an algebraically closed field  $k$ of characteristic 0.  As in the  case $k = \C$,  which  was  considered first by Vinberg   \cite{Vinberg1975, Vinberg1976},  if $m < \infty$, the  $\Z_m$-grading is  defined  uniquely by    the automorphism  $\theta \in \Aut (\g)$   defined in (\ref{e:theta})
		If $m = \infty$ then the formula
	$$\theta_t (x) : = t^k x \text{ for } x \in \g_k$$
	defines  a 1-parameter   group   of automorphisms  of $\g$.
	
Thus   we   shall  use  the notation $(\g, \theta)$   for  a $\Z_m$-graded  semisimple Lie  algebra $\g$ over  a field $k$ of  characteristic  0, where $\overline k$ is the algebraic closure of $k$,  and $\theta$  is the automorphism  of the Lie algebra
$\overline \g : = \g \otimes _k \overline{k}$   that defines  the  induced $\Z_m$-grading  on  $\overline \g$ as follows
\begin{equation*}
	\g \otimes _k   \overline k =  \bigoplus_{i \in \Z_m}  \g_i \otimes _k  \overline{k}.
\end{equation*}
If $\theta$ is the identity, then  we may  omit  $\theta$  and  say that $\g$  is ungraded.

(3) As it has been   noted in \cite{Vinberg1976},  results  for   $\Z_m$-graded  semisimple   complex Lie  algebras  are  easily extended  to  $\Z_m$-graded   algebraic reductive  complex  Lie  algebras.
Similarly, results  for   $\Z_m$-graded  semisimple   real Lie  algebras  are  easily extended  to  $\Z_m$-graded   algebraic reductive  real  Lie  algebras.  For the sake of simplicity  of the exposition,  we shall consider
mainly  $\Z_m$-graded  semisimple  Lie algebras in our paper, and  we shall   explicitly  extend   results  for    $\Z_m$-graded semisimple Lie algebras  to   $\Z_m$-graded   algebraic reductive Lie algebras only when it is necessary.  Finally we   refer  the reader to \cite{Kac1995}  for    relations    between $\Z_m$-graded    complex semisimple Lie algebras  and the Kac-Moody Lie algebras.
\end{remark}

\subsection{Jordan  decomposition of  homogeneous   elements}\label{subs:jordan}
In this subsection we prove  the graded  version of the   Jordan decomposition  of
a  homogeneous element  in    a $\Z_m$-graded  semisimple Lie algebra  $\g$ over    a  field $k$ of characteristic zero, extending Vinberg's result  for the case $k = \C$ \cite[\S 1.4]{Vinberg1976} and  L\^e's   result  for the case $k = \R$ \cite[\S 2]{Le2011}.

\begin{proposition}\label{prop:gradedjordan}  Let    $ x\in \g_1$  be a  homogeneous  element  in a  $\Z_m$-graded
	semisimple Lie  algebra  $\g$ over  a field  $k$  of  characteristic   0. Then   $x =  x_s + x_n$, where  $x_s\in \g_1$ is a   semisimple element,   $x_n \in  \g_1$ is a nilpotent  element  and $[x_s, x_n] =0$.
	\end{proposition}
\begin{proof}  Since  $\g$  is   semisimple, we  have the  Jordan  decomposition  $ x =  x_s + x_n$, see e.g. \cite[p. 17, 18]{Humphreys1980}.    Recall that $\overline k$ denotes the algebraic  closure   of $k$.  Then   the $\Z_m$-grading on $\g$ extends  naturally   to a $\Z_m$-grading on  $\g \otimes _k   \overline k$,  which  is  defined  by  an automorphism $\theta \in \Aut (\g \otimes_k \overline k)$,
see Remark \ref{rem:grading} (2).
Then  we have
	$$\theta (x) = \om \cdot x   = \om  \cdot (x_s  + x_n).$$
	Since the  Jordan  decomposition is unique, it follows  that  $x_s \in \g_1 $ and  $x_n \in \g_1$.
	This completes  the proof of Proposition \ref{prop:gradedjordan}.
\end{proof}

\subsection{Homogeneous $\ssl_2$-triples  and nilpotent  elements}
 Let $\g$ be a  Lie algebra  over a field $k$ of characteristic 0.
 Recall that  a triple $(h,e, f)$ of elements  in $\g$ is called {\it  an $\ssl_2$-triple},
 if   $e$ and $f$ are nilpotent, $h$ is  semisimple, and
 $$ [e,f]=h,\:  [h,e]=2e,\: [h,f]=-2f.$$
 Now assume  further that $\g$  is $\Z_m$-graded.
 We say that an $\ssl_2$-triple $(h,e,f)$ is {\em homogeneous} if
 $h\in \g_0$, $e\in \g_1$, $f\in \g_{-1}$.
From now on by an $\ssl_2$-triple we mean a homogeneous $\ssl_2$-triple.

\begin{lemma}\label{lem:fisf}
	Let $(h,e,f)$ and  $(h,e,f')$ be  $\ssl_2$-triples
in a Lie algebra over a field of characteristic 0. Then $f=f'$.
\end{lemma}

\begin{proof}
This is \cite[ Lemma 32.2.3]{tauvelyu}.
\end{proof}

\begin{lemma}\label{lem:jm1}	
Let $\g = \oplus_{i \in \Z_m} \g_i$ be a $\Z_m$-graded semisimple Lie algebra
over a field of characteristic 0. Let $e\in \g_1$ be nilpotent.
Then there exist $f\in \g_{-1}$ and  $h\in \g_0$ such that $(h,e,f)$ is an
$\ssl_2$-triple.
\end{lemma}

\begin{proof}
 This is proved in the same way as \cite[ Proposition 4]{KR71}  (in the cited
  paper $\Z_2$-gradings are considered, but the same ideas work more generally,
see also \cite{Graaf2017}, Lemma 8.3.5).
\end{proof}

\begin{lemma}\label{lem:2}
	Let $\g=\oplus_{i\in \Z_m} \g_i$ be a $\Z_m$-graded semisimple Lie algebra
	over a field of characteristic 0.
	Let $(h,e,f)$, $(h',e,f')$ be two homogeneous $\ssl_2$-triples.
	Then there is a nilpotent element $z\in \g_0$ such that, setting
	$\sigma = \exp(\ad z)$,
	we have $\sigma(h)=h'$, $\sigma(e)=e$ and $\sigma(f)=f'$.
\end{lemma}

\begin{proof}
  The proof is similar to the proof of \cite[Lemma 32.2.6]{tauvelyu}; in this
  book the assumption is that the field is algebraically closed. However, for
  the proof of this lemma that assumption plays no role.
  See also  \cite{Graaf2017}, the proof of Theorem 8.3.6.
\end{proof}

\begin{definition}\label{def:charac}  Given  a  $\ssl_2$-triple $(h, e,f)$  we   call  $h$ {\it the characteristic} of  the nilpotent element  $e$.
	\end{definition}

Lemma \ref{lem:2}  implies that   if   $e$ and $e'$  are  $\GG_0$-conjugate, then   their  characteristics  are
$\GG_0$-conjugate.  If the    ground field $k $ is $\C$  then  the converse  is also true \cite[Theorem 1(4)]{Vinberg1979}.

\begin{theorem}\label{thm:3}
	Let $\g=\oplus_{i\in \Z_m} \g_i$ be a $\Z_m$-graded semisimple Lie algebra
	over a field of characteristic 0. Let $H$ be a group of automorphisms of
	$\g$ such that each element of $H$
	 stabilizes each space $\g_i$. Assume that
	$H$ contains $\exp(\ad z)$ for all nilpotent $z\in \g_0$. Let $e,e'\in\g_1$ be
	nilpotent elements lying in homogeneous $\ssl_2$-triples $(h,e,f)$, $(h',e',f')$.
	Then $e,e'$ are $H$-conjugate if and only if there exists  $\sigma\in H$
	with $\sigma(h)=h'$, $\sigma(e)=e'$, $\sigma(f)=f'$ (in other words, the
	two $\ssl_2$-triples are $H$-conjugate).
\end{theorem}

\begin{proof}
	Of course only one direction needs a proof. Let $\tau_1\in H$ be such that
	$\tau_1(e)=e'$. By Lemma \ref{lem:2} there exists $\tau_2\in H$ with
	$\tau_2(\tau_1(h)) = h'$, $\tau_2(e')=e'$, $\tau_2(\tau_1(f))=f'$ so we
	can set $\sigma = \tau_2\tau_1$.
\end{proof}

Now  let $\g$ be a  $\Z_m$-graded semisimple Lie algebra over $\R$.
Let $\sT^\cC$, $\sT$ be
the sets of homogeneous $\ssl_2$-triples in $\g^\cC$ and $\g$, respectively.
If we let $\sigma$ denote the complex conjugation of $\g^\cC$ with respect to a
basis of $\g$, then $\sT = \{ t\in \sT^\cC \mid \sigma(t)=t\}$.

\begin{lemma}\label{lem:hcen}
Let $h\in M_n(\R)$ be a semisimple matrix whose all eigenvalues $\lambda_1,\dots,\lambda_k$ are real,
where $\lambda_1,\dots,\lambda_k$ are pairwise different.
 Then  the centralizer $\Zz_{\SL (n,\R)}(h)$  of $h$ in $\SL (n, \R)$ is conjugate over $\R$
 to $S(\GL(n_1,\R)\times \cdots \times\GL(n_k,\R))$,
 which is the subgroup consisting of all matrices of determinant 1
 that are block diagonal with blocks of respective sizes $n_1,\ldots,n_k$,
where $n_i$ are the multiplicities of the eigenvalue of $\lambda_i$\hs,, and hence  $n_1+\cdots +n_k=n$.
\end{lemma}
\begin{proof}    Lemma \ref{lem:hcen}     is known to experts,
	see  e.g. \cite[Table A.3, Appendix A]{Le1998}  for a
	similar  assertion and  \cite[p. 20]{Djokovic1983}  for the case $n =8$.
	For the  sake  of  reader's convenience   we   give here a   short proof  of this Lemma.
For $i=1,\dots,k$, let $V_i\subseteq V:=\R^n$ denote the eigenspace corresponding to the eigenvalue $\lambda_i$.
By the definition of a semisimple matrix, $V=\bigoplus V_i$
(we take into account that the real numbers  $\lambda_1,\dots,\lambda_k$ are {\em all} complex eigenvalues of $h$).
Hence the centralizer of $h$ in $\GL(n,\R)=\GL(V)$ is $\prod \GL(V_i)$,
the centralizer of $h$ in $\SL(V)$ is $S\left(\prod \GL(V_i)\right)$,
and the lemma follows.
\end{proof}

\begin{lemma}\label{lem:sameh}
Let $e \in \g_1$  be a nilpotent  element contained in a homogeneous
$\ssl_2$-triple $(h,e,f)$. Let $e'\in G_0(e) \cap \g_1$; then
$e'$ is $\GG_0(\R)$-conjugate to some $e''$ lying in a
homogeneous $\ssl_2$-triple $(h,e'',f'')$.
\end{lemma}

\begin{proof}
We  present  two  proofs    of this Lemma.  In the first   proof  we  use   \cite[Lemma 4.1 (ii)]{Le2011}, which asserts that  if  $(h, e,f) \in \g_1$ is    a homogeneous  $\ssl_2$-triple   then
$\Ad _{G_0} (h) \cap  \g_0 = \Ad_{\GG_0 (\R)} (h)$.    Now  let  $e'\in G_0 (e)  \cap \g_1$  and
$(h', e', f')$ is a  homogeneous    $\ssl_2$-triple.  Then  $h' \in G_0 (h) \cap \g_1$  and by
\cite[Lemma 4.1.(ii)]{Le2011}  we have  $h'  = \Ad_{g_0'}(h)$  for some $g_0' \in \GG_0 (\R)$.  Let $e'' = \Ad_{g_0'}  ^{-1}  (e') $. Then  $(h, e'',  f'')$ is   the required   homogeneous  $\ssl_2$-triple.

	 In the  second  proof  of   Lemma \ref{lem:sameh}  we  use  a  result  from Galois cohomology  theory. For the $\Z/3\Z$-graded Lie algebra  $\e_{8(8)}$  that we consider in this
paper, we  argue as follows. Let $(h',e',f')$ be a homogeneous
$\ssl_2$-triple containing $e'$. Then $h$ and $h'$ are $G_0$-conjugate by Theorem
\ref{thm:3}.

Since $h$ fits into an $\ssl_2$-triple, all its eigenvalues   of the adjoint action on $\g$ and hence on $\g_0=\ssl(9,\R)$  (see  \cite{VE1978}  and  Section \ref{sec:compmeth})  \hs are integral.
Let $\lambda=a+bi$ be an eigenvalue of $h$ on $\C^9$.
Since $h$ is real, the complex conjugate number $\bar\lambda=a-bi$  is an eigenvalue as well.
Then the eigenvalue $\lambda-\bar\lambda=2bi$ on $\g_0$ is integral, whence $b=0$.
We conclude that all eigenvalues of $h$ on $\C^9$ are real.

 By Lemma \ref{lem:hcen} the centralizer $\Zz_{G_0}(h)$  is isomorphic to
$S(\GL(n_1,\R)\times \cdots \times \GL(n_k,\R))$
and hence it has trivial Galois cohomology by Lemma \ref{lem:g1}.
So by Corollary \ref{c:Rconj}
there is a $g\in \GG_0(\R)$ with $g\cdot h' = h$. Hence if we set
$e'' = g\cdot e'$, $f'' = g\cdot f'$, then we see that $(h,e,f)$ is
$\GG_0(\R)$-conjugate to $(h,e'',f'')$.
\end{proof}

\begin{theorem}\label{thm:galois1}  Let $e \in \g_1$ be a nilpotent    element  and  $t=(h, e, f)$ be a  homogeneous $\ssl_2$-triple.
\begin{enumerate}
\item
 The   set  of  $\GG_0(\R)$-orbits in  $G_0 (e) \cap \g_1$  is  in  a canonical  bijection   with
	\begin{equation}\label{eq:galois2}
	\ker[\hs \Ho^1 \Zz_{\GG_0}(t) \to \Ho^1 \GG_0\hs].
	\end{equation}
\item The   set  of $\GG_0(\R)$-orbits  in  $G_0(e) \cap \g_1$ is  in
a canonical bijection   with
\begin{equation}\label{eq:galois3}
\ker [\hs\Ho^1 \Zz_{\GG_0}(t) \to  \Ho^1 \Zz_{\GG_0}(h)\hs],
\end{equation}	
where $h$  is    the  semisimple  element in the  $\ssl_2$-triple $t=(h, e, f)$.
\end{enumerate}
\end{theorem}

\begin{proof}
By Theorem \ref{thm:3},
the $\GG_0(\R)$-orbits of nilpotent elements in $\g_1$ are in
a canonical bijection with the $\GG_0(\R)$-orbits in $\sT$.
By  Proposition \ref{p:coh-orbits}, the set of $\GG_0(\R)$-orbits in $\sT$
is in a canonical bijection with the kernel \eqref{eq:galois2}, and
the   first  assertion  of  Theorem \ref{thm:galois1} follows.

For the second part we note that by Lemma \ref{lem:sameh}
each $\GG_0(\R)$-orbit in $G_0(e)\cap \g_1$ has a
representative $e'$ lying in a homogeneous $\ssl_2$-triple $(h,e',f')$
(with $e\in \g_1$, $f\in \g_{-1}$). Consider the set
$$\sT_h^\cC = \{ \text{ homogeneous $\ssl_2$-triples } (h,e'',f'') \text{ with }
e''\in \g_1^{\cC}, f'' \in \g_{-1}^\cC\}.$$
Then by Theorem \ref{thm:3} the $\GG_0(\R)$-orbits in $G_0(e)$ are in
a canonical bijection with the $\GG_0(\R)$-conjugacy classes in the set
$\sT_h$ consisting of homogeneous $\ssl_2$-triples $(h,e'',f'')$ with
this $h$ and with $e''\in \g_1$, $f''\in \g_{-1}$. It is obvious that
two elements of $\sT_h^{\cC}$ are $G_0$-conjugate if and only if they are
$\Zz_{G_0}(h)$-conjugate. Furthermore, two complex homogeneous $\ssl_2$-triples
$(h_1,e_1,f_1)$, $(h_2,e_2,f_2)$ are $G_0$-conjugate if and only if
$h_1,h_2$ are $G_0$-conjugate by \cite[Theorem 1.4]{Vinberg1979}, see also   \cite[ Theorem 8.3.6]{Graaf2017}.
Hence $\sT_h^\cC$ is a single
$\Zz_{G_0}(h)$-orbit. Now the assertion of (2) follows from Proposition
\ref{p:coh-orbits}.
\end{proof}

\begin{remark}\label{rem:jm}
(1) Lemma \ref{lem:jm1}	extends  Jacobson-Morozov's result   concerning  $\ssl_2$-triples in  ungraded Lie algebras over  a field $k$ of characteristic 0  \cite[Lemma 7  and Theorem 17, Chapter III]{jac}, see also  \cite{CM1993}  for  the case $k = \C$ or $k = \R$,  Vinberg's result   for     homogeneous    $\ssl_2$-triples  in graded Lie algebras over $\C$ \cite[Theorem 1]{Vinberg1979}  and L\^e's  result \cite[Theorem 2.1]{Le2011} for homogeneous    $\ssl_2$-triples in graded Lie algebras  over $\R$.
	
(2) Lemma \ref{lem:2} and Theorem  \ref{thm:3}  extend Kostant's  result  for ungraded  complex simple Lie algebras \cite{Kostant1959}, Vinberg's result   for    graded Lie algebras over $\C$ \cite[Theorem 1]{Vinberg1979} and L\^e's  result for graded Lie algebras over $\R$ \cite[Theorem 1.2]{Le2011}.

(3)   Theorem  \ref{thm:galois1} (2) extends   Djokovi\'c's    result   that   Dojokovi\'c used  for his classification  of 3-vectors on $\R^8$ \cite{Djokovic1983}.       There is  another  proof of   Theorem \ref{thm:galois1} (2)    that employs    \cite[Theorem 4.3]{Le2011}, which also   served   for L\^e's   method   for a  classification of
  conjugacy  classes  of   homogeneous  nilpotent  elements  in a  real $\Z_m$-graded  semisimple    Lie algebra using     algorithms   in  real algebraic  geometry. Note  that      the latter method     is   very  hard to implement.

(4) It is known  that  an  element  $x\in \g_1$ (resp. $x\in \g_1^\cC$)
is nilpotent if and only if  the closure of its orbit $G_0 (\R)\cdot x$ (resp. the closure  of its orbit $G_0\cdot x$) contains  zero \cite[Lemma 2.5]{Le2011} (resp. \cite[Proposition 1]{Vinberg1976}), and  $x$ is semisimple  if and only if  the orbit $G_0 (\R)\cdot x$ (resp. the orbit $G_0 \cdot x$) is closed \cite[Lemma 2.5]{Le2011}(resp. \cite[Proposition 3]{Vinberg1976}).
\end{remark}

In Theorem \ref{thm:misha}  below  we shall    describe  a  Levi-decomposition  of  the    stabiliser   $\Zz_{G_0} (e)$,   generalizing  a  result  due  to  Barbasch-Vogan and Kostant, see \cite[Lemma 3.7.3, p. 50]{CM1993},  and combining  it with   Galois cohomology   technique  (Proposition \ref{prop:e,h,f}) we shall give a new proof of Theorem \ref{thm:galois1} (1).

 Let $t = (h,e,f)$ be a homogeneous $\ssl_2$-triple  in a   $\Z_m$-graded semisimple   complex  Lie algebra $\g$.

\begin{lemma}\label{lem:red-part}    $\Zz_{\g_0} (t)$  is  a  reductive   Levi  subalgebra   of the  algebraic  Lie  algebra $\Zz_{\g_0}(e)$ and  we have  the  following  decomposition
	\begin{equation}\label{eq:bvk0}
\Zz_{\g_0} (e) = (\Zz_{\g_0} (e) \cap [\g, e])	 \oplus \Zz_{\g_0} (t).
	\end{equation}
\end{lemma}
\begin{proof}   Let $\theta$ be   the  automorphism  of $\g$ that  defines  the $\Z_m$-grading. Since  $e$ is an eigenvector  of $\theta$, it follows  that  $\Zz_{\g} (e)$ is  invariant  under  the   action of $\theta$. Hence we have the  following decomposition
\begin{equation}\label{eq:decomtheta}
\Zz_{\g} (e) = \bigoplus_{i \in \Z_m} (\Zz_\g (e) \cap \g_i).
\end{equation}
By Barbasch-Vogan's and Kostant's  result  \cite[Lemma 3.7.3, p. 50]{CM1993}  we have
\begin{equation}\label{eq:bvk}
\Zz_\g (e) = (\Zz_\g (e) \cap [\g, e]) \oplus \Zz_\g  (t).
\end{equation}
Since     each summand in RHS   of (\ref{eq:bvk})  is invariant  under the action of $\theta$, we obtain   (\ref{eq:bvk0}) immediately.

Next  we  observe  that  $\Zz_{\g_0} (t)$  is  a  reductive    subalgebra  of $\Zz_{\g_0} (e)$  since   $\Zz_\g (t)$    and $\g_0$ are  reductive  Lie subalgebra.
Since $\Zz_{\g}(e)   \cap [\g, e] $ is  a nilpotent  ideal  of $\Zz_\g(e)$  by \cite[Lemma 3.7.3, p. 50]{CM1993},
$\Zz_{\g_0}(e)   \cap [\g, e] $ is a nilpotent  ideal of  $\Zz_{\g_0} (e)$. Hence
$\Zz_{\g_0} (t)$   is the   Levi  reductive subalgebra  of  $\Zz_{\g_0} (e)$. This completes  the  proof of  Lemma \ref{lem:red-part}.
\end{proof}

\begin{theorem}\label{thm:misha} Let $R_u\hs \Zz_{G_0}(e)$ denote the unipotent radical of the identity component of $\Zz_{G_0}(e)$.
	Then $R_u\hs \Zz_{G_0}(e)\cdot \Zz_{G_0}(t)=\Zz_{G_0}(e)$.
	\end{theorem}
\begin{proof}  Denote  by $U^e$  the   connected  Lie  subgroup in the adjoint group
	$\Ad_\g$  of the  Lie algebra  $\g$  corresponding to the Lie  subalgebra  $\Zz_\g (e)  \cap  [\g, e]$.  By  \cite[Lemma 3.7.3, p. 50]{CM1993},
	 we have the  Levi decomposition
	\begin{equation}\label{eq:bvkg}
	\Zz_{\Ad_\g} (e)= U^e  \cdot \Zz_{\Ad_\g}  (t).
	\end{equation}
	Recall  that  $G_0$  is the  connected  Lie subgroup of $\Ad_\g$  corresponding  to the  Lie  algebra $\g_0$.
Denote   by $\Theta$ the automorphism  of the adjoint  group $\Ad_\g$ whose differential is $\theta$, namely we set
	\begin{equation}\label{eq:ad}
	\Theta (\Ad _{\exp \xi}):  = \Ad _{\exp (\theta(\xi))}
	\end{equation}
	for  any $\xi \in \g$, see   Remark  \ref{rem:grading}(1).
	Since  $U^e$  and  $\Zz_\g (e)$ are  invariant  under  $\Theta$, taking into  account that the map
	$\exp: \Zz_\g (e)  \cap  [\g, e] \to U^e $ is a  diffeomorphism,  we obtain  from  (\ref{eq:bvkg}) and (\ref{eq:ad}):
	$$\Zz_{G_0} (e) =  (U^e \cap \Zz_{G_0} ^0(e)) \cdot \Zz_{G_0} (t).$$
	Noting that $U^e \cap \Zz_{G_0} ^0(e) = R_u\hs \Zz_{G_0}(e)$,  this completes  the   proof  of Theorem \ref{thm:misha}.
\end{proof}

\begin{proposition}\label{prop:e,h,f}
	The natural map
	\[\Ho^1(k, \Zz_{G_0}(t)\hs)\labelto{i_*} \Ho^1(k, \Zz_{G_0}(e)\hs)\]
	induced by the inclusion map $i\colon  \Zz_{G_0}(t)\into  \Zz_{G_0}(e)$
	is bijective.
\end{proposition}

For  the   proof   of Proposition \ref{prop:e,h,f} we   need Sansuc's lemma.

\begin{proposition}[\cite{Sansuc1981}, Lemma 1.13]
	\label{p:Sansuc}
	Let $G$ be a linear algebraic group over a  field $k$ of characteristic 0,
	and let $U\subseteq G$ be a unipotent $k$-subgroup.
	Assume that $U$ is normal in $G$.
	Then the canonical map $H^1(k,G)\to H^1(k,G/U)$ is bijective.
\end{proposition}

Sansuc assumes that $G$ is connected, but his concise proof does not use this assumption.
See \cite[Proposition 3.1]{BDR} for a detailed proof.

\begin{proof}[Proof  of Proposition \ref{prop:e,h,f}] Since $R_u\hs \Zz_{G_0}(e)$ is a normal unipotent subgroup of $\Zz_{G_0}(e)$, we see that
	$R_u\hs \Zz_{G_0}(e)\cap \Zz_{G_0}(t)$ is a normal unipotent subgroup
	of the reductive group $\Zz_{G_0}(e,h,f)$, and hence equals $\{1\}$.
	Since  $R_u\hs \Zz_{G_0}(e)\cdot \Zz_{G_0}(t)=\Zz_{G_0}(e)$ by Theorem \ref{thm:misha},
	we see that the composite map
	\[\Zz_{G_0}(t)\labelto{i} C(e)\labelto{j} \Zz_{G_0}(e)/R_u\hs \Zz_{G_0}(e)\]
	is bijective.
	Thus the composite map
	\[\Ho^1(k,\Zz_{G_0}(t)\hs)\labelto{i_*} \Ho^1(k,\Zz_{G_0}(e)\hs)\labelto{j_*} \Ho^1(k,\hs \Zz_{G_0}(e)/R_u\hs \Zz_{G_0}(e)\hs)\]
	is bijective.
	
	Since by Sansuc's lemma (Proposition \ref{p:Sansuc}) the map $j_*$ is bijective, so is the map $i_*$\hs, as required.
\end{proof}
\begin{proof}[Alternative proof   of  Theorem \ref{thm:galois1}(1)]  Combining   Proposition \ref{p:coh-orbits}
with   Proposition \ref{prop:e,h,f}, we obtain immediately  Theorem \ref{thm:galois1}(1).
\end{proof}

Finally  we   show the  existence  of a real   representative  of  nilpotent  orbits under  a certain  assumption.

\begin{proposition}\label{prop:nilreal} (cf. \cite[Lemmas 6.3, 6.5]{Djokovic1982})
Assume  that   a  $\Z_m$-graded   real semisimple  Lie algebra  $\g$ is split, i.e.  $\g$  is   the split form  of $\g^\cC$.
Then  every   complex  orbit $G_0 (e)$ of a  homogeneous nilpotent  element  $e \in \g_1 ^\cC$
	has a  real  representative in $\g_1$.
\end{proposition}

\begin{proof}  Let  $e \in \g_1^\cC$  be a   nilpotent  element lying in the
homogeneous $\ssl_2$-triple $(h,e,f)$. The adjoint map $\ad h : \g^\cC\to \g^\cC$
has integral eigenvalues. This implies that $h\in \g$. Now for $i,k\in \Z$
define $\g^\cC_i(k)=\{ u\in \g_i^\cC \mid [h,u] = ku \}$ and
$$U = \{ u\in \g^\cC_1(2) \mid [\g^\cC_0(0),u] = \g_1^\cC(2)\}$$
then $U$ is open in $\g^\cC_1(2)$ (\cite[Proposition 8.4.1]{Graaf2017}) and
nonempty as $e\in U$. Since $h\in \g$ all spaces involved have bases
consisting of elements of $\g$. So $U\cap \g_1$ is nonempty; let $e'$ be an
element of it. Then $e'$ lies in a homogeneous $\ssl_2$-triple $(h,e',f')$
(again by \cite[Proposition 8.4.1]{Graaf2017}).
As the first element of this triple is the same as the first element of
the triple containing $e$ we have that $e$ and $e'$ are $G_0$-conjugate
(\cite[Theorem 8.3.6]{Graaf2017}).
\end{proof}

\subsection{Cartan subspaces  in  graded  semisimple  Lie algebras over $\R$}\label{subs:ss}

In this  subsection  we consider a    graded  semisimple real Lie  algebra
$\g = \oplus _{i \in \Z_m} \g_i$.
Recall  that  {\it a Cartan subspace in $\g_1$} 
is a maximal  subspace in $\g_1$  
consisting of commuting semisimple elements   (cf.   \cite{Vinberg1976} and \cite[\S 2]{Le2011}).
Two  Cartan  subspaces $\h $ and $\h'$  in $\g_1$ 
are  called {\it conjugate} 
if  there exists an element $g \in \GG_0 (\R)$ 
such that $g \cdot \h = \h'$.

For any  subspace $V $ in $\g$ we   use the shorthand notation $V^\cC$
for $V \otimes _\R \C  \subset \g^\cC$.

Recall  that we denote by $\Zgg_\g (V)$  the centralizer  of  $V $  in $\g$.  Set
\begin{equation}
\Zgg_\g(V) _i : = \Zgg_\g (V) \cap \g_i. \label{eq:gr1}
\end{equation}

\begin{lemma}\label{lem:centr1}  Let $\h \subset \g_1$  be a subspace consisting  of
	commuting  semisimple  elements. Then $\Zgg_\g (\h)$ is an algebraic  reductive Lie subalgebra of maximal rank. Moreover, we have
	\begin{equation}
	\Zgg_\g (\h) = \bigoplus _i \Zgg_\g (\h) _i.\label{eq:dec1}
	\end{equation}
	Denote by $Z(\Zgg_\g(\h))$  the center of $\Zgg_\g (\h)$ and by 
$D(\Zgg_\g(\h))$ the derived algebra (the commutator  subalgebra) of $\Zgg_\g(\h)$.
	Then
	\begin{equation}
	\Zgg_\g (\h) _i = (Z(\Zgg_\g (\h)) \cap \g _i)  \oplus (D(\Zgg_\g (\h)) \cap \g _i). \label{eq:dec2}
	\end{equation}
\end{lemma}

\begin{proof} The first assertion of Lemma \ref{lem:centr1} is  well-known.
It is an easy consequence of  the  fact that
the centralizer of any set of commuting semisimple elements
in a  complex  algebraic reductive Lie algebra is
an algebraic reductive Lie algebra of maximal rank, see e.g. \cite[\S 2.2]{Vinberg1976}.
		
	Now let us prove the second assertion of Lemma \ref{lem:centr1}. For any $ g \in \Zgg_\g (\h)$ write $ g = \sum _i g_i$, where $g_i \in \g_i$.
	Since  $[g, t ] = 0$ for any $t \in \h$, it follows that $[g_i, t ] \in  \g_{i +1} $ must vanish. This proves
	(\ref{eq:dec1}).
	
	Since $D(\Zgg_\g(\h))  = [\Zgg_\g(\h), \Zgg_\g(\h) ]$, using (\ref{eq:dec1}) we obtain that
	\begin{equation}
	D(\Zgg_\g(\h)) =  \oplus _i D(\Zgg_\g(\h))\cap \g _i.\label{eq:dec3}
	\end{equation}
		 Since  $\Zgg_\g(\h)\otimes _\R \C = \Zgg_{\g ^\cC} (\h)$ is invariant  under the  action of  $\theta$, we have
		\begin{equation}\label{eq:centr}
		Z(\Zgg_\g (\h)) = \oplus _i Z (\Zgg_\g (\h))\cap \g_i.
		\end{equation} Combining (\ref{eq:centr})  with (\ref{eq:dec1}) and  (\ref{eq:dec3}), taking into account that $\Zgg _\g  (\h)$
is  a reductive  Lie subalgebra,    we  obtain (\ref{eq:dec2}) immediately.
		 This proves Lemma \ref{lem:centr1}.
\end{proof}

We set $Z(\Zgg_\g (\h)) _i : = Z(\Zgg_\g(\h)) \cap \g _i$,
and $D(\Zgg_\g(\h)) _i : = D(\Zgg_\g(\h)) \cap \g_i$.
The following lemma is a version of \cite[Prop. 6]{Vinberg1976}
for  real $\Z_m$-graded  semisimple Lie algebras.

\begin{lemma}\label{lem:carc}   A subspace $\h$  consisting of commuting semisimple elements in $\g_1$ is a Cartan subspace
if and only if  $Z(\Zgg_\g(\h))_1 = \h$  and $D(\Zgg_\g(\h))_1$ consists of nilpotent elements.
\end{lemma}

\begin{proof} Assume that $\h$ is a Cartan subspace. Then $Z(\Zgg_\g(\h))_1 = \h$, since $\h \subset  Z(\Zgg_\g(\h))_1$.
Now we wish to  show that $D(\Zgg_\g(\h))_1$ contains only nilpotent elements.
Let $v \in D(\Zgg_\g(\h))_1$. We  will show that $v$ is nilpotent.
The grading in $\g$ induces a grading in $D(\Zgg_\g(\h))$; see \eqref{eq:dec3}.
We write the Jordan decomposition $v=s+n$ in the real semisimple  Lie algebra $D(\Zgg_\g(\h))$,
where $s\in D(\Zgg_\g(\h))_1$ is semisimple and $n\in D(\Zgg_\g(\h))_1 $ is nilpotent;
see Subsection  \ref{subs:jordan}.
Then $s\in D(\Zgg_\g(\h))\subset \Zgg_\g(\h)$.
	It follows that $s$ commutes with $\h$.
Since $s$ is semisimple and $s\in\g_1$, we conclude that
$s\in\h=Z(\Zgg_\g(\h))_1$.
	On the other hand,  $s$ is contained in $D(\Zgg_\g(\h))_1$.
Since
\[ Z(\Zgg_\g(\h))_1\cap D(\Zgg_\g(\h))_1\subseteq Z(\Zgg_\g(\h))\cap D(\Zgg_\g(\h))=\{0\},\]
we conclude that $s=0$ and $v=n$ is nilpotent, which proves  the ``only" part of Lemma \ref{lem:carc}.
Now we assume that $Z(\Zgg_\g(\h))_1 = \h$  and $D(\Zgg_\g(\h))_1$ consists of nilpotent elements.
Then any subspace $\t\supseteq \h$  consisting of commuting semisimple elements in $\g_1$  lies in $Z(\Zgg_\g (\h))_1$, so $\t = \h$.
This proves Lemma \ref{lem:carc}.
\end{proof}

\begin{corollary}\label{cor:complc}  The complexification of a real Cartan subspace
	$\h  \subset \g_1$ is  a complex Cartan subspace in $\g^\cC$.
	Furthermore, if   the complexification  $\h^\cC$ of a subspace $\h \subset \g_1$  is a  Cartan subspace  in $\g_1 ^\cC$  then  $\h$  is a Cartan  subspace.
\end{corollary}

\begin{proof} Assume  that  $\h \subset \g_1$ is a Cartan  subspace.     By Lemma \ref{lem:carc},
	$D(\Zgg_{\g_1}(\h))\subset\gl(\g)$ consists of nilpotent elements.
By   \cite[Corollary 3.3]{Humphreys1980} (a version of Engel's theorem),
     since the Lie subalgebra $D(\Zgg_{\g_1}(\h))\subset\gl(\g)$ consists of nilpotent elements,
		its complexification
    $D(\Zgg_{\g^\cC_1}(\h^\cC))\subset\gl(\g^\cC)$ consists of nilpotent elements.  Next we note  that
    $Z(\Zgg_{\g^\cC}(\h^\cC))_1 = \h^\cC$, since 	$Z(\Zgg_\g(\h))_1 = \h$  by Lemma \ref{lem:carc}.
	By \cite[Proposition 6]{Vinberg1976}, it follows that  $\h ^\cC$ is a Cartan subspace in $\g_1 ^\cC$.
    This proves  the first  assertion  of Corollary \ref{cor:complc}.

	It follows  from the first assertion  that  all  Cartan subspaces in $\g_1$
    have  the same  dimension, which is the  dimension  of   any     Cartan subspace  in $\g_1^\cC$.
	Now assume that $\h ^\cC$  is a Cartan  subspace.
    Then  $\h$ consists  of    commuting semisimple elements
	and  $\dim _\R \h = \dim  _\C \h^\cC$.  Hence $\h$ is a real  Cartan  subspace  in $\g_1$.
    This completes the  proof  of Corollary \ref{cor:complc}.
\end{proof}

In Theorem \ref{thm:cartan} below we    establish a canonical bijection
between  the  conjugacy classes  of Cartan  subspaces  in $\g_1$
and  a  subgroup  of   Galois cohomology, using   Proposition \ref{p:serre}.
Let  $\h$ be  a  Cartan  subspace  in $\g_1$.   By Corollary
\ref{cor:complc}, $\h^\cC$ is a  Cartan  subspace  in $\g^\cC_1$.
We denote  $\Nn_0=\Nn_{\GG_0}(\h^\cC)$.

\begin{theorem}\label{thm:cartan}  There is a canonical  bijection
between  the  conjugacy classes  of  Cartan subspaces  in $\g_1$  and
the kernel  $\ker[\hs \Ho^1 \Nn_0 \to \Ho^1 \GG_0\hs]$.
\end{theorem}

\begin{proof}
Let $d$ denote the dimension
of a Cartan subspace  in $\g_1^\cC$.
Then the set $Y$ of Cartan subspaces in $\g_1^\cC$ is a closed algebraic  subvariety
in the Grassmann variety ${\rm Gr}(\g_1^\cC,d)$
of $d$-dimensional linear subspaces in $\g_1^\cC$.
This subvariety is defined over $\R$ (with respect to the $\R$-structure in $\g_1^\cC$),
because the composition law in $\g^\cC$ is defined over $\R$.
Let $\YY$ denote the corresponding real variety.
By   Vinberg's theorem \cite[Theorem 1]{Vinberg1976},
the variety $\YY$ is homogeneous, that is,
$G_0$ acts on $Y$ transitively.
By Corollary \ref{cor:complc}, a real subspace $\h\subset \g_1$ is a real Cartan subspace
if and only if its complexification $\h^\cC$ is a complex Cartan subspace in $\g_1^\cC$.
In other words, the set of real Cartan subspaces of $\g_1$
is the set of real points $\YY(\R)$  of $\YY$.
Now  Theorem \ref{thm:cartan} follows from  Proposition  \ref{p:coh-orbits}.
\end{proof}

It is known  that   any $\Z_m$-graded     complex   semisimple   Lie algebra    has a  compact real  form $\g$  \cite[Lemma 5.2, p. 491]{Helgason1978}.

\begin{proposition}\label{prop:transit}  All Cartan subspaces  in
	$\g_1$  are $G_0 (\R)$-conjugate,  if   the     simply-connected  Lie group $\Gtil (\R)$ is compact.
\end{proposition}

\begin{proof}  We   shall     apply  the  argument  in  the   proof  of  the conjugacy of maximal tori in a compact Lie group in \cite[p. 250]{OV}
to  the   algebraic closure  $\overline {\h^\cC}$  of   a Cartan subspace  $\h^\cC \subset \g_1^\cC$ in a  $\Z_m$-graded  semisimple complex  Lie algebra  $\g^\cC$.
Recall  that  $\overline{\h^\cC}$  is   the  smallest  algebraic  subalgebra  of $\g^\cC$ that contains $\h^\cC$.  In \cite[\S 3]{Vinberg1976}  Vinberg showed that
\begin{equation}\label{eq:closure}
\overline{\h^\cC} = \bigoplus  _{(k,m)=1}   (\overline{\h^\cC})_k, \text{  where  }  (\overline{\h^\cC})_k  = \overline{\h^\cC} \cap  \g_k ^\cC.
\end{equation}
Let $\theta$ be the  automorphism of  the Lie algebra $\g^\cC$ that defines   the $\Z_m$-grading on $\g^\cC$. Clearly  $\overline {\h^\cC}$ is  $\theta$-invariant.	

  Now let us assume the  condition  of Proposition \ref{prop:transit}.  Let $\h \subset \g_1$  be a  Cartan subspace  and  $\overline{\h ^\cC}$  the algebraic   closure  of its complexification.
  Let $\overline \h : = \overline{\h^\cC}\cap \g$.   Note that
	$\overline {\h ^\cC}$ is the  Lie algebra  of a maximal $\Theta$-invariant   torus  $T$ in  $\Gtil$  and
	$\overline  \h$ is the  Lie algebra  of its  maximal  compact   torus $T_c$.   Now  let  $\h' \subset \g_1$  be another  Cartan  subspace,  and $T'$ -  the  $\Theta$-invariant  torus  in $\Gtil$  whose Lie algebra is  $\overline{(\h ')^\cC}$. By   Vinberg's theorem \cite[Theorem 1]{Vinberg1976},  there exists an element  $ g \in  G_0$ such that   $ \Ad g  (T') = T$.  Then  $\Ad_g$  moves  the maximal   compact torus $T'_c$ in $T'$  to  the  maximal compact torus  $T_c$ in $T$, i.e.,
	\begin{equation}\label{eq:conj2}
	\Ad _g (T_c') = T_c
	\end{equation}
	Now  we shall consider  the polar decomposition
	\begin{equation}\label{eq:polar}
	 \Gtil =  \Gtil (\R) \cdot  \exp  (\sqrt{-1}\g)
	 \end{equation}
	of $\Gtil$  generated by  the decomposition $\g^\cC = \g \oplus  \sqrt{-1}  \g$ as a  real vector space.  This  polar  decomposition  induces  the polar decomposition
	$$G_0 =  G_0 (\R) \cdot  \exp (\sqrt{-1} \g_0).$$
	Let  $ g = k \cdot p$  be the    corresponding  polar  decomposition of  $g\in G_0$. Let  $a \in  T_c'\subset  \Gtil(\R)$.  We shall look  at
	$$\Ad _g (a)  =  \Ad _{k} \cdot \Ad _p  (a) \in T_c \subset  \Gtil (\R).$$
	Since $k \in \Gtil(\R)$, we have  $l : = \Ad_p (a) \in \Gtil(\R)$.   Then  we have
	\begin{equation}\label{eq:rewrite}
	a^{-1} p a =  a^{-1} l p.
	\end{equation}
	Since  $\Ad_{\Gtil(\R)}  (\exp \sqrt{-1} \g)  = \exp \sqrt{-1} \g$, it follows  from (\ref{eq:rewrite}) and the  uniqueness   of the polar decomposition  (\ref{eq:polar})
	\begin{equation}\label{eq:inv1}
	a^{-1} p a = p.
	\end{equation}
	Hence $p a p ^{-1} = a$ for any  $a \in T_c'$.  Using   (\ref{eq:conj2})  we obtain
	\begin{equation}\label{eq:conj3}
	Ad_k (T_c') = T_c.
	\end{equation}
	It follows  that  $\Ad_k (\h') = \h$. This completes  the  proof of Proposition \ref{prop:transit}.
\end{proof}

\begin{remark}\label{rem:Cartanred}  It is easy to check that  the arguments  in the proofs of  Lemmas \ref{lem:centr1}, \ref{lem:carc}, Corollary \ref{cor:complc},   Theorem \ref{thm:cartan}  and Proposition \ref{prop:transit} are  also  valid for   $\Z_m$-graded  algebraic  reductive Lie algebras over $\R$.  Hence  all the results  in Subsection \ref{subs:ss}   hold  for $\Z_m$-graded  algebraic  reductive Lie algebras over $\R$.
\end{remark}

\subsection{Centralizers of homogeneous semisimple elements  and the Weyl group}\label{subs:weyl}

In this section we fix a Cartan subspace $\h^\cC$ in $\g^\cC_1$.
The group
$$W(\g^\cC, \theta): = \frac{\Nn_{G_0} (\h^\cC)}{\Zz_{G_0}(\h^\cC)} $$
is  finite  and is called  {\it the Weyl group} of the $\Z_m$-graded semisimple Lie algebra $\g^\cC$ \cite{Vinberg1975, Vinberg1976}. We regard $W(\g^\cC, \theta)$ as a  group of complex linear   transformations of $\h^\cC$. The Weyl  group $W(\g^\cC, \theta)$ depends  on $\h^\cC$ but its  $G_0$-conjugacy   class does not depend  on  the choice of $\h^\cC$  since  all  the Cartan  subspaces  in $\g_1^\cC$ are $G_0$-conjugate \cite[Theorem 1]{Vinberg1976}. Vinberg has shown that  two elements    of $\h^\cC$ are  in the  same  $G_0$-orbit in $\g_1^\cC$, if and only if they  are in the same   $W(\g^\cC, \theta)$-orbit \cite[Theorem 2]{Vinberg1976}.  In what follows, if no  misunderstanding can occur,  we shall    write  $W$  instead of $W(\g^\cC, \theta)$.

For $p \in \h^\cC$ we set
\begin{align}
&W_p : = \Zz_W (p), \ \, \Gamma_p : = \Nn_W  (W_p) /W_p,\label{eq:wp}\\
& \h^\cC _p : =\{  q \in \h^\cC \mid  g (q) = q \, \text{ for all }  g \in W_p \}
= \{  q \in \h^\cC \mid W_p \subset W_q\},\label{eq:hp}\\
&\h^{\cC, \circ} _p : = \{  q\in \h^\cC_p \mid   W_q = W_p\}.\label{eq:h0p}
\end{align}

Since $W_p$  is a finite  group,  it  is generated  by   elements
$w_1, \cdots, w_k \in W \subset \GL(\h^\cC)$.
Vinberg proved that $W_p$ is  generated by  complex  reflections
\cite[Proposition 14]{Vinberg1976} so    we can  take  $w_1, \cdots, w_k$ to
be  complex reflections. Clearly  $q \in \h^\cC$ belongs  to $ \h ^\cC_p$  if
and only  if   it   satisfies  the  following
system  of	homogeneous   linear  equations
\begin{equation}\label{eq:wfp}
(w_i -1) q = 0  \text{ for } i \in [1,k].
\end{equation}
Thus  $\h^\cC_p$  is a  linear  subspace  in  $\h^\cC$  which  is  the
intersection  of a finite number of hyperplanes. Furthermore, a $q\in \h_p^\cC$
lies in $\h_p^{\cC,\circ}$ if and only if $(w-1)q\neq 0$ for all complex
reflections $w\in W$ with $w\not \in W_p$. Hence $\h_p^{\cC,\circ}$ is Zariski-open
in $\h_p^{\cC,\circ}$. It is straightforward to see that for $p,p'\in \h^\cC$ and
$w\in W$ we have
\begin{equation}\label{eq:CpCq}
w\h_p^\cC = \h_{p'}^\cC \text{ if and only if } wW_pw^{-1} =W_{p'}.
\end{equation}

\begin{lemma}\label{lem:redweyl}
Let $p\in \h^\cC$ and $q \in \h^{\cC, \circ}_p$. If  $q  = w \cdot p$  for a
$w \in W$   then  $w \in \Nn_W(W_p)$.
\end{lemma}	

\begin{proof}
From $w \cdot p  = q$   we obtain $W_q =  wW_pw^{-1}$.  Taking
into account $W_p = W_q$   we  conclude  that  $w \in \Nn_W (W_p)$.
\end{proof}

Since $W$ is a finite  group,  there  exist elements  $p_1, \cdots, p_N \in
\h^\cC$ such  that for any $q \in \h^\cC$ the  stabilizer   group $W_q$  is
$W$-conjugate   to one  and only one of the  stabilizer   groups
$W_{p_i}$ for $1\leq i\leq N$.  So by \eqref{eq:CpCq} for  any   point
$q \in \h^\cC$   there is a unique $i \in [1,N]$ such that $q$ is
$W$-conjugate  to a   point  in $\h_{p_i}^{\cC,\circ}$.

For any $p \in \h^\cC$ the linear   subspace   $\h^{\cC}_p$  is a stratified space, which  can be  written as  a disjoint union:
\begin{equation}\label{eq:strat}
\h^{\cC}_p = \h^{\cC, \circ}_p \cup \bigcup_{j = 1}^{n} \h^{\cC,\circ}_{q_j},
\end{equation}
where  $q_j \in \h^{\cC}_p$   and  $n = n(p)$.

If $p$ is a generic  point  of  the  $W$-action, i.e., $W_p = 1$, then  $\h^\cC_p = \h^\cC$, and
(\ref{eq:strat})  gives a stratification of    the Cartan subspace  $\h^\cC$.

Denote  by $\Sigma (\h^\cC_p)$  the  weight system of  the adjoint representation   of the   subalgebra $\h^\cC_p$ on $\g^\cC$.   Note  that  $\Sigma  (\h^\cC_p)$ is the restriction  of  $\Sigma (\h^\cC)$ to $\h^\cC_p$.  We define
$$\h^{\cC, \mathrm{reg}}_p : = \{ q \in \h^\cC_p|\,  \sigma (q) \not = 0  \text{ for all } \sigma \in \Sigma(\h^\cC_p)\setminus\{0\}\}.$$
Elements in $\h^{\cC, \mathrm{reg}}_p$  will be  called {\it $\Sigma (\h^\cC_p)$-regular.}
\begin{lemma}\label{lem:generic}
Let $p\in \h^\cC$ and $q\in \h_p^{\cC,\circ}$.
Then $q$ is $\Sigma (\h^\cC_p)$-regular if and only if
$\Zm_{\Gtil} (q) = \Zm_{\Gtil } (\h^\cC_p)$.
\end{lemma}
\begin{proof}
Suppose that $q$ is $\Sigma (\h^\cC_p)$-regular.
Then  $\z_{\g ^\cC} (q ) = \z_{\g^\cC} (\h^\cC_p)$.
Since $\Zz_{G} (q)$ and $ \Zz_{G } (\h^\cC_p)$  are connected by \cite[Theorem 3.14]{Steinberg1975},   we conclude  that  $\Zm_{\Gtil} (q) = \Zm_{\Gtil }(\h^\cC_p)$.

Conversely, if $\Zm_{G} (q) = \Zm_{G } (\h^\cC_p)$ then $\z_{\g ^\cC} (q ) =
\z_{\g^\cC} (\h^\cC_p)$ implying that $q$ is $\Sigma (\h^\cC_p)$-regular.
\end{proof}

Following \cite{AE1982},    we  say that  two elements $p, q \in \h^\cC$  are {\it in the same family} or  {\it  in  the same $W$-family}, if  their centralizer  subalgebras
$\z_{\g^\cC} (p)$  and   $\z_{\g^\cC} (q)$ are  conjugate   by elements in $G_0$  or   by  elements  in $W = W(\g^\cC, \theta)$, respectively.

 Given  an action of a group $K$ on a  set $S$,   if two subsets $S_1, S_2\subset S$ are $K$-conjugate,  i.e. on the same  orbit of  the action of $K$,  we shall  write
$S_1 \sim_{K} S_2$.

For $p \in \h^\cC$    let $( \Zgg_{\g ^\cC}(p), \theta)$ be  the    $\Z_m$-graded  Lie algebra  $ \Zgg_{\g ^\cC}(p)$ with  the induced  $\Z_m$ grading,
which  is   induced  from the  restriction  of the action of $\theta$  to $\Zgg_{\g ^\cC}(p)$.
Note   that  $\h^\cC$  is  a Cartan  subspace   in $( \Zgg_{\g ^\cC}(p), \theta)$.
We denote by $W ( \Zgg_{\g ^\cC}(p), \theta)$  its Weyl group, which acts on $\h^\cC$.

\begin{proposition}\label{prop:family}  (1) Let $p,q\in \h^\cC$. If they are in
the same family then they are in the same $W$-family. If $\h_r^{\cC,\circ} =
\h_r^{\cC,\mathrm{reg}}$ for all $r\in \h^\cC$ then the converse holds as well.
	
(2)	If  $p, q \in \h^\cC$   are in the  same family, then
	$\Zz_{G_0}(p)\sim_{G_0}\Zz_{G_0} (q)$.	
	
	(3)    Assume  that  $\Zgg_{\g ^\cC} (\h ^\cC)$ is a  Cartan   subalgebra  of $\g^\cC$. Then  $W_p  = W_q$  implies   $W ( \Zgg_{\g ^\cC}(p), \theta) = W ( \Zgg_{\g ^\cC}(q), \theta)$.
\end{proposition}

\begin{proof} (1)
Assume that   $p, q$ are in the same family, that is,  $\Zgg_{\g^\cC}(q) =  g  \cdot \Zgg_{\g^\cC}  (p)$
	for some   $g  \in G_0$.
	Note that  both $\h^\cC$  and $g \cdot   \h ^\cC$  are   Cartan  subspaces   of  $\Zgg_{\g^\cC} (q)$.
Furthermore,  $\Zgg_{\g^\cC_0}(p)$  is the  subalgebra of   $\Zgg_{\g^\cC}(p)$ consisting of  elements of degree $0$.
By \cite[Theorem 1]{Vinberg1976}  there  exists  an element  $g' \in  \Zz_{G_0}(q)$ such
	that   $g'\hm g\cdot  \h^\cC = \h^\cC$.
	Then
	$g'\hm g\cdot  \Zgg_{\g ^\cC}(p) = \Zgg_{\g^\cC}(q)$.
	It follows   that
	$$ (g'\hm g)\Zz_{\Gtil}(p)(g'\hm g)^{-1} = \Zz_{\Gtil}(q)$$
	since $\Zz_{G}(p)$ is  connected by  \cite[Theorem 3.14]{Steinberg1975}.
	Denote the  restriction  of $g'\hm g$ to $\h ^\cC$   by
	$w$. Then  $wW_pw^{-1} = W_q$. Hence $p,q$ lie in the same $W$-family.

Now suppose that $\h_r^{\cC,\circ} =\h_r^{\cC,\mathrm{reg}}$ for all $r\in \h^\cC$
and that $p,q$ lie in the same $W$-family. The latter means that there is a
$w\in W$ with $wW_pw^{-1}=W_q$. Then from \eqref{eq:CpCq} it follows that
$w\cdot \h_p^{\cC,\circ} = \h_q^{\cC,\circ}$. By our hypothesis this entails
$w\cdot \h_p^{\cC,\mathrm{reg}} = \h_q^{\cC,\mathrm{reg}}$. Let $g\in G_0$ be a preimage
of $w$. Then
$$g\cdot \z_{\g^\cC}(\h_p^{\cC,\mathrm{reg}}) = \z_{\g^\cC}(\h_q^{\cC,\mathrm{reg}}).$$
But $\z_{\g^\cC}(\h_s^{\cC,\mathrm{reg}})=\z_{\g^\cC}(s)$ for $s=p,q$. Hence $p,q$ lie
in the same family.

	(2) Assume that  $p, q \in \h^\cC $  are in the same  family.  By \cite[Theorem 3.14]{Steinberg1975}  $\Zz_G (p)$
	and $\Zz_G (q)$  are connected.    Since   their  Lie algebras  are  conjugate   via   an element $g_0$ in $G_0$,
it  follows    that   the groups  $\Zz_G (p)$
	and $\Zz_G (q)$  are   conjugate   via $g_0$.   This completes the proof of   Proposition \ref{prop:family}(2).

(3) Assume  that  $\frak C  =\Zgg_{\g ^\cC} (\h ^\cC)$ is a  Cartan   subalgebra  of $\g^\cC$.  Then
		$\frak C$  is  the algebraic  closure   $\overline {\h^\cC}$ of   $\h^\cC$, see  (\ref{eq:closure}), since
		otherwise, the   subalgebra $\Zgg_{\g ^\cC} (\h ^\cC)$      is   strictly larger  than $\frak C$.
		Hence $\g^\cC$    satisfies the  condition  of Proposition  13  in \cite{Vinberg1976},  and therefore   $W_p  = W (\Zgg_{\g^\cC} (p), \theta)$   \cite[\S 6.3]{Vinberg1976}.
	 This  identity  is a consequence  of \cite[Proposition 19]{Vinberg1976}, asserting that if the algebraic closure  $\overline{\h^\cC}$ is a  Cartan subalgebra    $\frak  C$ in $\g^\cC$  then  the Weyl   group $W(\g^\cC, \theta)$ of the     $\Z_m$-graded Lie algebra $\g^\cC$ can be identified   with the  subgroup  in the  Weyl group $W(\g^\cC, \theta = Id)$  of the   ungraded
		Lie algebra $\g^\cC$  that   leaves  $\theta$ invariant, taking into account the fact that    for the  ungraded  Lie algebra $(\g^\cC, \theta = Id) $,  we have  $W_p(\g^\cC, \theta = Id)  = W(\Zgg_{\g^\cC}(p), \theta = Id)$,
		since in this case $W_p (\g^\cC, \theta = Id)$  is  generated  by  the reflections  $ r_\alpha\in W (\Zgg_{\g^\cC}(p), \theta = Id)$, where   $\Sigma  (\frak C)$ denotes the root system  of $\g^\cC$ w.r.t.  $\frak C$,  and
		$$ \alpha  \in \Sigma  (\frak C, p) : =\{ \alpha \in \Sigma  (\frak C)|\, \alpha (p) = 0\},$$
		$$\Zgg_{\g^\cC}(p) = \frak C \bigoplus_{\alpha \in \Sigma(\frak C, p)} \g^\cC_\alpha.$$
		Here	$\g^\cC_\alpha$  denotes the corresponding   root  subspace.
		
		Hence  $W_p = W_q$ implies $W ( \Zgg_{\g ^\cC}(p), \theta) = W ( \Zgg_{\g ^\cC}(q), \theta)$.
		This completes  the proof  Proposition \ref{prop:family} (3).
		It  is known  that   $W(\g^\cC, \h^\cC)$  is generated  by   complex reflections  $w_\alpha$ for   $\alpha  \in \Sigma (\h^\cC)$,
        whose   fixed point  $q$ is defined   by the equation     $\alpha  (q) = 0$,  see  e.g. \cite[Theorem 10 p. 163]{OV}.
	\end{proof}

\begin{lemma}\label{lem:Zp}
Let $p\in \h^\cC$ and assume that $\h_p^{\cC,\circ} =\h_p^{\cC,\mathrm{reg}}$.
Let $q\in \h_p^{\cC,\circ}$. Then $\Zm_{G_0}(p) = \Zm_{G_0}(q)$.
\end{lemma}

\begin{proof}
By Lemma \ref{lem:generic} we see that $\Zm_{\Gtil}(p) = \Zm_{\Gtil}(q)$.
Their intersections with $G_0$ are $\Zm_{G_0}(p)$ and $\Zm_{G_0}(q)$ respectively.
Hence these are equal as well.
\end{proof}

\begin{lemma}\label{lem:gp1p2}
Let $p\in \h^\cC$ and assume that $\h_p^{\cC,\circ} =\h_p^{\cC,\mathrm{reg}}$.
Let $p_1,p_2\in \h_p^{\cC,\circ}$ and let $g\in G_0$ be such that $gp_1=p_2$.
Then $gq\in \h_p^{\cC,\circ}$ for all $q\in \h_p^{\cC,\circ}$.
\end{lemma}

\begin{proof}
Let $q\in \h_p^{\cC,\circ}$. Let $\Nm_W(W_p)$ be the normalizer of $W_p$ in $W$.
Since $p_1,p_2$ are $G_0$-conjugate there is a $w\in \Nm_W(W_p)$ with
$wp_1=p_2$ (Lemma \ref{lem:redweyl}). Let $\hat w\in G_0$ be a preimage of
$w$. Then $g^{-1}\hat w \in \Zm_{G_0}(p_1)$. Hence by Lemma \ref{lem:Zp} we
see that $g^{-1}\hat w \in \Zm_{G_0}(q)$. It follows that
$gq=\hat w q$ which lies in $\h_p^{\cC,\circ}$.
\end{proof}

Now fix a $p\in \h^{\cC}$ and write $\Fm=\h_p^{\cC,\circ}$. We set
\begin{align*}
  \Zm_{G_0}(\Fm) & = \{ g\in G_0 \mid gq=q \text{ for all }
  q\in \Fm\},\\
  \Nm_{G_0}(\Fm) & = \{ g\in G_0 \mid gq\in \Fm \text{ for all }
  q\in \Fm\}.
\end{align*}

We define a map $\varphi : \Nm_{G_0}(\Fm) \to \Gamma_p$ (where $\Gamma_p =
\Nm_W(W_p)/W_p$, see \eqref{eq:wp}).
Let $g\in
\Nm_{G_0}(\Fm)$. Then $gp=q$ for some $q\in \Fm$. By Lemma \ref{lem:redweyl}
there is a $w\in \Nm_W(W_p)$ such that $wp=q$. Set $\varphi(g) = wW_p$.
Note that this is well defined: if $w'\in \Nm_W(W_p)$ also satisfies $w'p=q$
then $w^{-1}w'\in W_p$ so that $w'W_p = wW_p$.

\begin{lemma}\label{lem:Np}
Suppose that $\h_p^{\cC,\circ} =\h_p^{\cC,\mathrm{reg}}$.
Then $\varphi$ is a surjective group homomorphism with kernel $\Zm_{G_0}(\Fm)$.
\end{lemma}

\begin{proof}
  First we note the following: let $g\in \Nm_{G_0}(\Fm)$ be such that
  $gp=q$ and let $w\in \Nm_W(W_p)$ be such that $wp=q$; then for all $r\in \Fm$
  we have that $gr=wr$. Indeed, Let $\hat w\in G_0$ be a preimage
  of $w$. Then $\hat w^{-1} g\in \Zm_{G_0}(p)$. So by Lemma
  \ref{lem:Zp} it follows that $\hat w^{-1} g\in \Zm_{G_0}(r)$, or
  $gr=\hat wr=wr$.

  Now let $g_1,g_2\in \Nm_{G_0}(\Fm)$ and let $w_1,w_2\in \Nm_W(W_p)$ be
  such that $w_ip = g_ip$, $i=1,2$. Then by the first part of the proof
  we see that $g_1g_2 p = w_1w_2p$ implying that $\varphi$ is a group
  homomorphism.

  Let $\hat w\in G_0$ be a preimage of $w\in \Nm_W(W_p)$. Then
  $\hat w \in \Nm_{G_0}(\Fm)$ and $\varphi(\hat w) = wW_p$ so that
  $\varphi$ is surjective. If $g\in \ker \varphi$ then $gp=p$ and by the
  first part of the proof, $g\in \Zm_{G_0}(\Fm)$.
\end{proof}

Now we let $\h$ be a Cartan subspace in $\g_1$ so that its complexification
$\h^\cC$ is a Cartan subspace in $\g_1^\cC$. For a $p\in \h$ we consider
the sets $\h_p = \h_p^\cC \cap \h$ and $\h_p^\circ = \h_p^{\cC,\circ}$.

\begin{proposition}\label{prop:sconj}
(1) If  $p\in \h$ is   regular, then $G_0 (p) \cap \g_1$  consists  of $L$ $\GG_0(\R)$-orbits, where $L$ is   the number  of conjugacy classes of Cartan subspaces  in $\g_1$.

(2)  Assume that
all  Cartan  subspaces in $\g_1$ are conjugate.    Let $p \in \h^\cC$. Then the orbit
$G_0 (p)$ contains  a real point   in $\g_1$ if and only if the  orbit $W(p)$ contains  a real point  in $\h$. For any  $q \in \h$   the set  of  $\GG_0 (\R)$-orbits in  $G_0 (q) \cap \g_1$ is  in  a  canonical  bijection with the  set
	$$\ker  [ H^1 W_q \to H^1 W].$$
\end{proposition}

\begin{proof} (1) For any   regular semisimple element  $p \in \h$,  there is unique  Cartan subspace $\h(p)\subset \g_1 $ that contains   $p$.   This implies  the  first  assertion  of  Proposition \ref{prop:sconj}.
	
(2) Assume that all  Cartan  subspaces in $\g_1$ are $\GG_0 (\R)$-conjugate   and  the orbit
$G_0 (p)$, $p \in \h^\cC$,  contains  a real point    $p _\R\in \g_1$. Since   all Cartan  subspaces in $\g_1$ are conjugate, $p_\R$ is $\GG_0 (\R)$-conjugate  to  a point $p'_\R$  in $\h$. Hence $p$ is  $G_0$-conjugate to  $p'_\R \in \h$, and therefore  it is $W$-conjugate  to $p'_\R$. This  proves the first assertion of Proposition \ref{prop:sconj} (2).

The last  assertion of  Proposition \ref{prop:sconj}(2) follows  from   \cite[Theorem 2]{Vinberg1976}, the first assertion  of   Proposition \ref{prop:sconj}(2), 	 and  Proposition
\ref{p:serre}, (also  Proposition \ref{p:coh-orbits}).
\end{proof}

\begin{proposition}\label{prop:Requivalence}
Let $p\in \h$.
Assume that  $\h^{\cC,\circ}_p = \h^{\cC, \mathrm{reg}}_p$  and  let
$q, q' \in \h^\circ_p$.  Then   $q$  and $q'$  are  $\GG_0 (\R)$-conjugate, if
and only if there is a $g\in \Nm_{\GG_0(\R)}(\h_p)$ with $g\cdot q=q'$.
\end{proposition}

\begin{proof}
This follows directly from Lemma \ref{lem:gp1p2}.
\end{proof}

The groups $W$ and $\Gamma_p$ have a conjugation induced from $\GG_0$ and we
denote the group of the real points of $\Gamma_p$ by $\Gamma_p(\R)$.
Also we write
	$$ \Gamma_p (\g, \h): = \frac{\Nn_{\GG_0(\R)} (\h_{p})}{\Zz_{\GG_0(\R)} (\h_p)}$$
	
	The  group $ \Gamma_p (\g, \h)$ acts on $\h_p$   by restricting the  action
	of $\Nn_{\GG_0(\R)} (\h_{p})$ to $\h_p$.

\begin{lemma}\label{lem:weylr}  Assume that  $p \in \h$  and  $\h^{\cC,\circ}_p = \h^{\cC, \mathrm{reg}}_p$.
	
(1)	$ \Nn_{G_0}(\h^\cC_p) (p)  =  \Gamma_p(p)$.

(2)  There  is a   1-1  correspondence between the  set of
$\Gamma_p (\g, \h)$-orbits   in   $\Gamma _p   (p) \cap \h_p$  and
$\ker  [\Ho^1 \Zz_{G_0} (\h_{p}^\cC) \to\Ho^1\Nn_{G_0}  (\h_p^\cC)]$.

(3) The set $\Gamma _p (p)\cap \h_p$ consists of a single $\Gamma_p(\R)$-orbit.
 \end{lemma}

\begin{proof}  	The first  assertion follows  from  Lemma \ref{lem:Np}. The  second   and last assertion  of Lemma \ref{lem:weylr}
are direct  consequences of  Proposition \ref{p:serre},  taking into account
	that      $\Gamma_p$ acts  freely  on $\h_p ^{\circ}$,  which implies,  in particular, that the stabilizer of $p$ in $\Gamma_p$ is trivial and therefore
 $\Gamma_p(\R)$ has one orbit in $\Gamma _p (p)\cap \h_p$.
\end{proof}

\begin{corollary}\label{cor:weylr}  Assume    the condition  of  Lemma \ref{lem:weylr}.
	
(1)	There  is a natural   embedding $\Gamma_p (\g, \h)$  into $\Gamma_p (\R)$   and
	there  is  a 1-1   correspondence  between $\Gamma_p (\R)/\Gamma_p (\g, \h)$ and
$\ker  [\Ho^1 \Zz_{G_0} (\h_{p}^\cC)\to \Ho^1\Nn_{G_0}  (\h_p^\cC)]$.

(2) If $\Ho^1 \Zz_{G_0} (\h_{p}^\cC)=1$ then $\Gamma_p(\g,\h)=\Gamma_p(\R)$.
	\end{corollary}

\begin{proof}
	
	(1) If  $\h^{\cC,\circ}_p = \h^{\cC, \mathrm{reg}}_p$  then $\Gamma_p (\g, \h)$  acts
	on $\h_p$   without fixed  point.   Then   the  orbit  $\Gamma_p (\g, \h) (p)$  is a   subset of   the orbit $\Gamma_p (\R) (p)$  by Lemma \ref{lem:weylr} (2).
It is not hard to see  that this natural set inclusion  is     group homomorphism.  The first  assertion
	of Corollary  \ref{cor:weylr} then  follows  from Lemma \ref{lem:weylr}.

The second assertion follows directly from the first one.
\end{proof}

\subsection{Conjugacy  and centralizers of mixed elements}\label{subs:mix}

Let  $x = p_x +e_x $   and $y = p_y + e_y$ be the   Jordan decompositions
of   homogeneous elements  $x ,y \in \g_1$  of a  real  (resp. complex) $\Z_m$-graded  semisimple
Lie algebra  $\g$ (resp. $\g^\cC )$.

The following proposition is straightforward.

\begin{proposition}\label{prop:mix1}  Elements $x, y \in \g_1 ^\cC $  (resp. in $\g_1$) are  in the same $G_0$-orbit  (resp. in the same $\GG_0 (\R)$-orbit)
	if  and  only if there  exists  an element $g_0 \in G_0$  (resp. $g_0 \in G_0 (\R)$) such that
	$\Ad _{g_0} p_x = p_y$ and $\Ad_{g_0} e_x$ and $e_y$  are  $\Zz_{G_0}(p_y)$-conjugate  (resp.  $\Zz_{ \GG_0 (\R)}(p_y)$ -conjugate), that is, they are  in the  same $\Zz_{G_0}(p_y)$-orbit  (resp.  $\Zz_{ \GG_0 (\R)}(p_y)$-orbit)  in the  $\Z_m$-graded   reductive   Lie algebra $\Zgg_{\g^\cC } (p_y)$  (resp. $\Zgg_\g (p_y)$).
\end{proposition}

Vinberg observed that  any  $\Z_m$-graded   reductive complex  Lie algebra $\g$  is a direct  sum of
its commutative  ideal $\g _a$   and     its   derived  algebra  $D\g$   with  induced  grading,
\cite[\S 3]{Vinberg1976}. The analogous  assertion    also holds   for   $\Z_m$-graded   reductive  real Lie algebras   $\g$ by considering the  complexification $\g^\cC$.

 Assume  that  $p_x +e_x$ is the    Jordan decomposition  of a homogeneous element $x\in \g_1 ^\cC$  (resp. in $\g_1$)  in a
 $\Z_m$-graded   complex (resp.  real) semisimple Lie algebra  $\g ^\cC$  (resp. $\g$). Then $e_x $ lies in the  derived algebra $D(\Zgg_{\g^\cC} (p_x))$ of
$\Zgg_{\g^\Cc}(p)$   (resp.   in  $D(\Zgg_{\g} (p_x))$  of
$\Zgg_{\g}(p_x)$).   Note   that  $\Zz_{G} (p_x)$  is  a connected    algebraic    reductive   group  and
its can  be  written as   a  product   of  its   connected component  of its center  $Z (\Zz_{G} (p_x))$  and
 a connected semisimple   Lie   subgroup    $D\Zz_{G} (p_x)$  whose Lie algebra  is the derived  Lie algebra
 $D(\Zgg_{\g} (p_x))$.
 Thus we have    the  projection
 $$\pi: \Zz_{G}  (p_x) \to    D\Zz_{G} (p_x).$$
 The    center $Z(\Zz_{G} (p_x))$ acts   trivially  on    $D(\Zgg_{\g} (p_x))$.
 Denote  by $D\Zz_{G_0} (p_x)$ the image $\pi (\Zz_{G_0} (p_x))\subset D(\Zz_{G} (p_x))$.  Hence  we obtain  from
 Proposition \ref{prop:mix1}     immediately \cite[p. 80]{VE1978}

\begin{proposition}\label{prop:mix2} Let $ p_x +e_x$ and  $p_x + e_y$ be the   Jordan decomposition    of two    homogeneous  elements   $  x, y \in \g_1^\cC  \subset \g^\cC $  (resp. in $\g_1 \subset \g$)  with $p_x = p_y = p$.   Then     $p + e_x $ and $p + e_y$  are in the same  $G_0$-orbit  (resp. $\GG_0 (\R)$-orbit), if and  only if  $e_x$ and $e_y$  are in the same  orbit  of the  action of
	$D\Zz_{G_0}(p)$  (resp.  of $D\Zz_{\GG_0 (\R)} (p)$)  on the         $\Z_m$-graded   derived  algebra     $D\Zgg_{\g^\cC }(p)$ of $\Zgg_{\g^\cC }(p)$ (resp.  of $D\Zgg_\g (p)\subset \Zgg_{\g } (p)$).
\end{proposition}

Now  we   shall relate   the centralizer  $\Zz_{G_0} (p +e)$  of  a  mixed  homogeneous  element  $p+e$  with
  the   centralizer   $\Zz_{G_0}  (p)$.

\begin{theorem}\label{thm:centrmix}  Let  $ p +e$ be  the      Jordan decomposition  of a  homogeneous  element  $p +e \in \g_1$ in a    $\Z_m$-graded  semisimple Lie  algebra $\g$ over  $\C$.  Then   $\Zz_{G_0} (p +e)  $  is  a subgroup    of  the   reductive  algebraic   group $\Zz_{G_0} (p)$   and
	the Lie  algebra    $\Zgg_{\g_0} (p + e)$   of the group $\Zz_{G_0} (p +e)  $     has the following
	 Levi decomposition
	 $$\Zgg_{\g_0} (p + e)  =   (\Zgg_{\g_0}  (p) \cap [\Zgg_{\g_0} (p) , e]) \oplus \Zgg_{\g_0} (p) [  t(e)]$$
	where   $t(e)$ is  the   $\ssl_2$-triple  associated  to the nilpotent  element  $e$ in
	$D\Zgg_{\g} (p)$.
\end{theorem}

\begin{proof}  Since the  Jordan  decomposition is unique  we  have
	$$ \Zz_{G_0} (p +e) = \Zz_{\Zz_{G_0} (p)} (e).$$
	Using this  we  obtain  Theorem   \ref{thm:centrmix} immediately from Lemma \ref{lem:red-part}.
	\end{proof}


\section{Computational methods}
\label{sec:compmeth}

In this section we discuss the computational problems that we dealt with
in order to compute the various objects that play a role in our investigation.

First we describe a special construction of the representation
$\rho : \SL(9,\C)\to \GL( \bigwedge^3 \C^9)$, also used by Vinberg
and Elashvili (\cite{VE1978}, see also \cite[Example 3.3.v]{Le2011}).
Let $\g^\cC$ denote the simple complex Lie algebra
of type $E_8$. We let $\Phi$ denote its root system with basis of simple
roots $\Delta$. We use a fixed Chevalley basis of $\g^\cC$. This basis
consists of vectors $h_\alpha$ for $\alpha\in \Delta$ and $x_\alpha$ for
$\alpha\in \Phi$ (see \cite{Humphreys1980}, Chapter 25). We use the following
enumeration of the Dynkin diagram
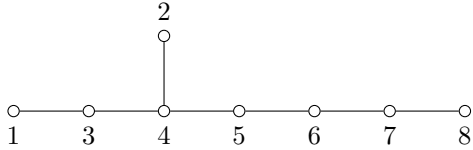
\begin{center}
\begin{tikzpicture}
\draw (-1,0.5) node[anchor=east] {};
\node[dnode,label=below:{\small $1$}] (1) at (0,0) {};
\node[dnode,label=above:{\small $2$}] (2) at (2,1) {};
\node[dnode,label=below:{\small $3$}] (3) at (1,0) {};
\node[dnode,label=below:{\small $4$}] (4) at (2,0) {};
\node[dnode,label=below:{\small $5$}] (5) at (3,0) {};
\node[dnode,label=below:{\small $6$}] (6) at (4,0) {};
\node[dnode,label=below:{\small $7$}] (7) at (5,0) {};
\node[dnode,label=below:{\small $8$}] (8) at (6,0) {};
\path (1) edge[sedge] (3)
(3) edge[sedge] (4)
(4) edge[sedge] (5)
edge[sedge] (2)
(5) edge[sedge] (6)
(6) edge[sedge] (7)
(7) edge[sedge] (8);
\end{tikzpicture}
\end{center}

We write $\Delta=\{\alpha_1,\ldots,\alpha_8\}$. Then the set consisting of
$h_{\alpha_i}$, $x_{\alpha_i}$, $x_{-\alpha_i}$ for $1\leq i\leq 8$ is a
canonical generating set of $\g^\cC$. Fix a primitive third root of unity
$\zeta\in\C$. For $i\neq 2$ we set $h_{\alpha_i}'=h_{\alpha_i}$,
$x'_{\alpha_i} = x_{\alpha_i}$, $x'_{-\alpha_i}=x_{-\alpha_i}$. Furthermore we set
$h'_{\alpha_2}=h_{\alpha_2}$, $x'_{\alpha_2}=\zeta x_{\alpha_2}$,
$x'_{-\alpha_2} = \zeta^2 x_{-\alpha_2}$. Then the set consisting of the
$h'_{\alpha_i}$, $x'_{\alpha_i}$, $x'_{-\alpha_i}$ for $1\leq i\leq 8$ is also a
canonical generating set of $\g^\cC$. Hence mapping $h_{\alpha_i}\mapsto
h'_{\alpha_i}$, $x_{\alpha_i} \mapsto x'_{\alpha_i}$, $x_{-\alpha_i}\mapsto x'_{-\alpha_i}$
extends to a unique automorphism $\theta$ of $\g^\cC$ (\cite{jac}, Chapter IV,
Theorem 3). We let $\g_i^\cC$ be the eigenspace of $\theta$ corresponding to
the eigenvalue $\zeta^i$. Then $\g_0^\cC$ is isomorphic to $\ssl(9,\C)$. The
space $\g_1^\cC$ is a $\g_0^\cC$-module, and hence a $\ssl(9,\C)$-module,
isomorphic to $\bigwedge^3 \C^9$.

From the construction it is clear that
the spaces $\g_i^\cC$ are defined over $\R$, that is, they have bases whose
elements are $\R$-linear combinations of the basis elements of $\g^\cC$.
By $\g_i$ we denote the $\R$-span of these bases.
Furthermore, the $h_{\alpha_i}$ span a Cartan subalgebra of $\g_0^\cC$ that
is split over $\R$. Hence we have an isomorphism $\psi^\cC : \ssl(9,\C)
\to \g_0^\cC$ which restricts to an isomorphism $\psi : \ssl(9,\R) \to
\g_0$. Because $\SL(9,\C)$ is simply connected, $\psi^\cC$ lifts to a
homomorphism of algebraic groups $\Psi^\cC : \SL(9,\C) \to G_0$, which in turn
restricts to a homomorphism $\Psi: \SL(9,\R)\to G_0(\R)$.

The map $\Psi^\cC $ is not explicitly given, but we do have its differential
$\psi^\cC  : \ssl(9,\C) \to \g_0^\cC $ explicitly (that is, we have an algorithm
that computes $\psi^\cC (x)$ for any given $x\in \ssl(9,\C)$).
Via the map $\psi^\cC $ we let $\ssl(9,\C)$ act on $\g_1^\cC $, and we fix
a $\ssl(9,\C)$-module isomorphism $\varphi^\cC  : \g_1^\cC \to \bigwedge^3(\C^9)$.
This map is not unique, but if we fix a choice of positive roots in the root
system of $\ssl(9,\C)$, then the highest-weight vectors in $\g_1^\cC $ and
$\bigwedge^3(\C^9)$ are uniquely determined (up to multiplication by nonzero
scalars), and they in turn determine
$\varphi^\cC $ up to a scalar multiple. Furthermore, $\varphi^\cC $ is also a
$\SL(9,\C)$-module isomorphism.

We use this maps to realize the action of $\SL(9,\C)$ on $\g_0^\cC$ and
on $\g_1^\cC$. For $g\in \SL(9,\C)$ and $x\in \g_0^\cC $, $u\in \g_1^\cC $ we
have $g\cdot x = \psi^\cC ( g (\psi^\cC )^{-1}(x) g^{-1})$ and
$g\cdot u = (\varphi^\cC )^{-1}( g\cdot \varphi^C(u) )$.

\subsection{Computing centralizers}\label{sec:compcen}

Let $e\in \g_1^\cC $ be a nilpotent element lying in the homogeneous
$\ssl_2$-triple $t=(h,e,f)$. We want to compute the centralizer
$$\Zm_0(t) = \{ g\in \SL(9,\C) \mid g\cdot h=h, g\cdot e=e, g\cdot f=f\}.$$
Moreover, we require a description of $\Zm_0(t)$ that is useful for the
determination of the cohomology set $\Ho^1\Zm_0(t)$ (which we use to find
the real forms of the complex orbit of $e$, see Theorem \ref{thm:galois1}).
 We cannot say
precisely what that means, but as a general rule it means that we need a
description from which we can easily read off the structure of the group
(for example, we need to know the isomorphism type of the identity
component, and the structure of the component group).

As a first step we determine a set of polynomial equations that
define $\Zm_0(t)$ as a subgroup of $\SL(9,\C)$. Observe that
by Lemma \ref{lem:fisf} we have that $\Zm_0(t)$ consists of $g\in \SL(9,\C)$
with $g\cdot h=h$ and $g\cdot e=e$.
We note that the set of $g\in \SL(9,\C)$ with $g\cdot h=h$ is the same
as the set of $g\in \SL(9,\C)$ with $g (\psi^\cC )^{-1}(h) =(\psi^\cC )^{-1}(h)g$.
In order to compute the corresponding polynomial equations we
let $g$ be a matrix whose entries are indeterminates, $x_{11}, x_{12},\ldots,
x_{99}$. Then the equality $g (\psi^\cC )^{-1}(h) =(\psi^\cC )^{-1}(h)g$ simply
translates to a set of linear equations in the $x_{ij}$.
For the equations corresponding to $g\cdot e = e$ we set $u=\varphi^\cC (e)$.
Again we let $g$ be a $9\times 9$-matrix whose entries are indeterminates.
Then $g\cdot u$ is a linear combination of basis elements of $\wedge^3 (\C^9)$
with coefficients that are polynomials in the indeterminates. This way we
get a set of polynomials describing the set of $g\in \SL(9,\C)$ with
$g\cdot e = e$.
Finally we add the polynomial $\det(g)-1$ to the set of polynomials.
In order to arrive at a description of the group $\Zm_0(t)$ we first compute a
Gr\"obner basis of the ideal generated by the polynomials that we obtained.
For a general introduction into the theory of Gr\"obner bases we refer to
\cite{CLO15}.
Here we use a lexicographical monomial order to get a Gr\"obner basis with a
triangular structure. In many cases the Gr\"obner basis computation succeeds
and allows us to read off the group structure.

In quite a few cases however, the Gr\"obner basis computation turned out to
be too difficult. Here we list the methods that we have applied in those cases.

First of all, we note that the centralizer $\z_0(t) = \{ x\in \ssl(9,\C) \mid
[\psi^\cC (x),h] = [\psi^\cC (x),e]=[\psi^\cC (x),f]=0\}$ can readily be obtained by
solving a set of linear equations. We have that $\z_0(t) = \Lie(\Zm_0(t))$,
so from $\z_0(t)$ the identity component of $\Zm_0(t)$ can in principle be
determined. This is particularly straightforward if the identity component
is a torus. In that case $\Zm_0(t)^\circ$ is determined by a lattice in the
character group of a maximal torus containing it; and this lattice
is easily computed from $\z_0(t)$ (see \cite{Graaf2017}, Example 4.2.5).
If $\Zm_0(t)$ is semisimple then from $\z_0(t)$ we can deduce the root datum
of $\Zm_0(t)^\circ$, and hence its isomorphism type. In general $\Zm_0(t)^\circ$
is the product of its semisimple part and a central torus. However, it can
happen that these two groups have a finite but nontrivial intersection.
This intersection needs to be determined as well.

Now we give some examples of situations where additional tricks helped to
determine $\Zm_0(t)$.

\begin{example}\label{exa:cencomps}
\begin{enumerate}
\item Suppose that $\Zm_0(t)^\circ$ is a 1-dimensional torus consisting of
  elements $T(a)$ for $a\in \C^\times$.
  This torus has two algebraic automorphisms,
  the identity and the map $T(a)\mapsto T(a^{-1})$. As $\Zm_0(t)^\circ$ is normal
  in $\Zm_0(t)$ we see that the latter group is the disjoint union of
  $$A=\{ g\in \Zm_0(t) \mid gT(a)g^{-1} = T(a)\}, B=\{ g\in \Zm_0(t) \mid
  gT(a)g^{-1} = T(a^{-1})\}.$$
  Let $u_0$ be a fixed nonzero element of $\z_0(t)$. Then
  \begin{align*}
A&=\{ g\in \Zm_0(t) \mid gu_0g^{-1} = u_0\} = \{g\in \Zm_0(t) \mid gu_0 = u_0g\},\\
B&=\{ g\in \Zm_0(t) \mid gu_0g^{-1} = -u_0\} = \{g\in \Zm_0(t) \mid gu_0 = -u_0g\}.
  \end{align*}
  So we compute two Gr\"obner bases: one of the ideal generated by the
  original polynomials along with linear polynomials equivalent to
  $gu_0 = u_0g$. For the second we add the linear polynomials that are
  equivalent to $gu_0 = -u_0g$. These two Gr\"obner basis computations
  are significantly easier than the original one.
\item Suppose that $\z_0(t)$ has a semisimple part $\s$. Then $\Zm_0(t)$ acts
  by conjugation on $\s$. Let $S$ be the connected subgroup of $\Zm_0(t)$
  with Lie algebra $\s$. Then the inner automorphisms of $\s$ are realized
  as conjugation by elements of $S$.

  Let $h_i,e_i,f_i$ for $1\leq i\leq \ell$ be
  a canonical generating set of $\s$. This means that $e_i$ are root vectors
  corresponding to the positive simple roots, the $f_i$ are root vectors
  corresponding to the negative simple roots and the $h_i=[e_i,f_i]$ satisfy
  $[h_i,e_i]=2e_i$, $[h_i,f_i]=-2f_i$ (see \cite{jac}, Chapter IV).
  Canonical generating sets are by no means unique. But
  it is well known that if $h_i',e_i',f_i'$ is a second canonical generating
  set then there is an inner automorphism $\sigma$ of $\s$ such
  that $\sigma(h_i') = h_{\pi(i)}$, $\sigma(e_i') = e_{\pi(i)}$,
  $\sigma(f_i') = f_{\pi(i)}$, where $\pi$ is a permutation of
  $\{1,\ldots,\ell\}$. Here $\pi$ corresponds to a diagram automorphisms of
  the Dynkin diagram of $\s$.

  This implies the following. Let $h_i,e_i,f_i$ for $1\leq i\leq \ell$ be
  a fixed canonical generating set of $\s$. Then each component of
  $\Zm_0(t)$ has an element whose conjugation action on $\s$ maps
  $h_i \mapsto h_{\pi(i)}$, $e_i\mapsto e_{\pi(i)}$, $f_i\mapsto f_{\pi(i)}$,
  where $\pi$ is a diagram automorphism of the Dynkin diagram of $\s$.
  Now for example $gh_ig^{-1} = h_{\pi(i)}$ is the same (for invertible $g$)
  as $gh_i = h_{\pi(i)}g$. This means that the condition that $g$ conjugates
  $h_i$ to $h_{\pi(i)}$ can be expressed by linear equations in the entries of
  $g$. Now if we want to find an element of each component of $\Zm_0(t)$ then
  we add the resulting linear equations to the original polynomial equations.
  This usually leads to a much easier Gr\"obner basis computation.
  The resulting polynomial equations tend to be easily solvable. So we can
  try to determine the component group this way. Of course, it may happen
  that the same $\pi$ is realized by elements in more then one component.

  We illustrate this with the determination of $\Zm_0(t)$ in case {\bf 54} of
  the list of nilpotent orbits. Here $\z_0(t)$ is of type $A_1$, and the
  9-dimensional module splits as a direct sum of modules of dimensions 2, 3, 4.
  Let $v_1,\ldots,v_9$ denote the standard basis of the 9-dimensional module.
  Then the submodules are $\langle v_6,v_7\rangle$, $\langle v_3,v_4,v_5\rangle$,
  $\langle v_1,v_2,v_8,v_9\rangle$. So a matrix in $\Zm_0(t)^\circ$ consists of
  three blocks of respective sizes $2\times 2$, $3\times 3$ and $4\times 4$.
  The rows and columns 6, 7 have a block of size $2\times 2$. So on that
  position we have a copy of $\SL(2,\C)$, and the other blocks contain
  homomorphic images of that group. The Lie algebra $\z_0(t)$ has only inner
  automorphisms. So
  each coset of the component group contains an element that is the
  identity on the Lie algebra. Adding the corresponding linear equations we
  find that the subgroup of $\Zm_0(t)$ that acts as the identity on $\z_0(t)$
  consists of $diag(t,t,t^4,t^4,t^4,t,t,t,t)$, where $t^6=1$. For $t=1$ and
  $t=-1$ these elements lie in the identity component of $\Zm_0(t)$.
  Hence $\Zm_0(t) \cong \SL(2,\C)\times \mu_3$.
\item If the identity component is a torus whose weight spaces on the
  natural 9-dimensional module are all 1-dimensional, then we can use the
  corresponding characters to determine the component group. We illustrate
  this by the computation of $\Zm_0(p_1+2p_2)$, where $p=p_1+2p_2$ is an
  element of the third canonical set $\Fm_3$
  considered in Section \ref{subsec:semcent}.

  From computing the Lie algebra we see that $T_4=\Zm_0(p)^\circ$ is a
  4-dimensional torus consisting of
$$T_4(t_1,t_2,t_3,t_4)=
\diag(t_1,t_2,(t_1t_2)^{-1},t_3,t_4,(t_3t_4)^{-1},(t_1t_3)^{-1},(t_2t_4)^{-1},
t_1t_2t_3t_4),$$
for $t_1,t_2,t_3,t_4\in \C^\times$.
Let ${\sf X}(T_4)$ denote its character lattice,
whose group operation we write additively. Then ${\sf X}(T_4) =
\Z \mu_1+\cdots +\Z \mu_4$, where $\mu_i(T_4(t_1,t_2,t_3,t_4)) =
t_i$. Let ${\sf X}_9$ denote the set of characters of $T_4$ that correspond
to the natural $9$-dimensional module. Then
$${\sf X}_9 = \{ \mu_1,\mu_2,\mu_3,\mu_4,-\mu_1-\mu_2,\ldots,\mu_1+\mu_2+\mu_3+\mu_4\}.
$$
If $g\in \GL(9,\C)$ normalizes $T_4$ then for $\mu \in {\sf X}(T_4)$ we have
that
$g\cdot \mu$, defined by $(g\cdot \mu)(t) = \mu(g^{-1}tg)$, also lies in
${\sf X}(T_4)$. The map $\mu \mapsto g\cdot \mu$ is an automorphism of ${\sf X}(T_4)$
and thus can be represented by an element of $\GL(4,\Z)$.
Moreover, $\mu \in {\sf X}_9$ if and only if $g\cdot \mu \in {\sf X}_9$. Now for an element
of $\GL(4,\Z)$ that stabilizes ${\sf X}_9$ we only have 9 choices for the image of
each $\mu_i$. So we can write down all possibilities and check by computer
which correspond to elements of $\GL(4,\Z)$ stabilizing ${\sf X}_9$. The result
is a group $\mathcal{G}_{72}$
of order 72 generated by $\pi_1$, $\pi_2$, $\pi_3$, where
\begin{enumerate}
\item $\pi_1$ interchanges $\mu_1$, $\mu_4$ and fixes $\mu_2$, $\mu_3$,
\item $\pi_2$ maps $\mu_1\mapsto \mu_2\mapsto \mu_4\mapsto \mu_3\mapsto \mu_1$,
\item $\pi_3$ fixes $\mu_1$, $\mu_2$ and maps $\mu_3\mapsto -\mu_1-\mu_3$,
  $\mu_4 \mapsto -\mu_2-\mu_4$.
\end{enumerate}

Write ${\sf X}_9 = \{\nu_1,\ldots,\nu_9\}$. The weight spaces corresponding to these
characters are spanned by the standard basis vectors $v_1,\ldots,v_9$.
Let $g\in \GL(9,\C)$ normalize $T_4$ and let $\pi\in S_9$ be such that
$g\cdot \nu_i = \nu_{\pi(i)}$. A small calculation shows that $gv_i$ lies in
the weight space corresponding to $\nu_{\pi(i)}$ so that $gv_i$ is a nonzero
scalar multiple of $v_{\pi(i)}$. Hence $g=M_\pi U$, where $M_\pi \in \GL(9,\C)$
is such that $M_\pi(v_i) = v_{\pi(i)}$ and $U\in T_9$, where $T_9$ is the
standard 9-dimensional torus of $\GL(9,\C)$.

The possibilities for the permutations $\pi$ follow directly from the group
$\mathcal{G}_{72}$. For each of the possible $\pi$ we determine the intersection
of $\Zm_0(p)$ and $M_\pi T_9$ (this can be done by Gr\"obner bases as we know
the defining equations for the group $\Zm_0(p)$). It turns out that this
intersection is either empty or equal to $M_\pi T_4$. Hence the component group
consists of the cosets of the $M_\pi$ where $\pi$ is such that the intersection
is not empty.
\item The next example concerns just one case. We want to compute the
  centralizer $\Zm_0(t)$ of the homogeneous $\ssl_2$-triple $t=(h,e,f)$ with
  $$e=e_{137}-e_{246}-e_{247}+e_{348}-e_{358}+e_{368}+e_{458}+e_{569}.$$
  (This is case {\bf 31} in Section \ref{sec:nilpcen}.) From the Lie algebra
  of the centralizer we know that $\Zm_0(t)$ is a finite group.

  We let $e_1,\ldots,e_9$ denote the standard basis of $\C^9$. Let
  $V_4 = \langle e_1,e_2,e_8,e_9\rangle$, $V_5=\langle e_3,\ldots,e_7\rangle$.
  Then the elements of $\Zm_0(h)$ consist of block-matrices, with a
  $4\times 4$ and a $5\times 5$-block, the first one stabilizing $V_4$ and
  the second one stabilizing $V_5$ So we identify $\Zm_0(h)$ with
  $$\{ (g_1,g_2) \in \GL(V_5)\times \GL(V_4) \mid \det g_1 \det g_2 = 1\}.$$
  Write $v_1=e_3,\ldots,v_5=e_7$, $w_1=e_1$, $w_2=e_2$, $w_3=e_8$, $w_4=e_9$.
  Then with $u_1=v_1\wedge v_5$, $u_2= -v_2\wedge v_4 -v_2\wedge v_5$,
  $u_3=v_1\wedge v_2-v_1\wedge v_3+v_1\wedge v_4+v_2\wedge v_3$,
  $u_4=v_3\wedge v_4$ (so $u_i\in \bigwedge^2 V_5$)  we have
  $$e=u_1\wedge w_1+u_2\wedge w_1+u_3\wedge w_3+u_4\wedge w_4.$$
  Now let $g\in \Zm_0(t)$ then in particular $g\in \Zm_0(h)$ so that
  $g=(g_1,g_2)$, $g_1\in \GL(V_5)$, $g_2\in \GL(V_4)$. Write
  $g_2\cdot w_j = \sum_{k=1}^4 a_{kj} w_k$, then
  $$g\cdot e = \sum_{j=1}^4 g_1\cdot u_j\wedge g_2\cdot w_j =
  \sum_{k=1}^4 (\sum_{j=1}^4 a_{kj} g_1\cdot u_j) \wedge w_k.$$
  But this is equal to $e$ so $\sum_{j=1}^4 a_{kj} g_1\cdot u_j = u_k$
  for $1\leq k\leq 4$. Let $(b_{ij})$ be the inverse of the matrix
  $(a_{kj})$, that is $\sum_{k=1}^4 b_{ik} a_{kj} = \delta_{ij}$ for
  $1\leq i,j\leq 4$. Applying $\sum_{k=1}^4 b_{ik}$ to the equality just obtained
  we get $g_1\cdot u_i = \sum_{k=1}^4 b_{ik} u_k$. In particular, if we let
  $U$ be the subspace of $\bigwedge^2 V_5$ spanned by $u_1,\ldots,u_4$, then
  $g_1\cdot U=U$.

  Now we look for elements of order dividing 2 in $\Zm_0(t)$. For this we
  first consider the group $H_5 = \{g_1\in \GL(V_5)\mid g_1\cdot U = U\}$,
  which is defined by a set of polynomial equations of degree 2. To these
  we add the equations expressing that $g_1^2=1$. This yields a 0-dimensional
  ideal whose Gr\"obner basis we can easily compute. By inspecting it we
  see that the set of polynomial equations has 52 solutions, which we denote
  by $g_1^1,\ldots,g_1^{52}$. Next, for each $r$ we find $b_{ik}^r$ with
  $g_1^r\cdot u_i = \sum_{k=1}^4 b_{ik}^r u_k$. We let the matrix $(a_{ik}^r)$ be
  the inverse of $(b_{ik}^r)$ and set $g_2^r = (a_{ik}^r)$. Finally we check
  whether $\det(g_1^r)\det(g_2^r)=1$.
  This then yields 26 elements of $\Zm_0(t)$ of
  order dividing 2, and hence 25 elements $h_1,\ldots,h_{25}$ of order
  exactly 2.

  Finally, for $1\leq i<j \leq 25$ we consider the subgroups
  $M_{ij} = \{ g\in \Zm_0(t) \mid gh_i = h_jg\}$ (note that a $g\in \Zm_0(t)$
  conjugates an element of order 2 to an element of order 2, so each such $g$
  lies in some of these subgroups). To compute them we add the linear
  polynomials expressing that $gh_i = h_jg$ to the defining polynomials
  of $\Zm_0(t)$. It turned out that the corresponding ideals have easily
  computable Gr\"obner bases. By solving the corresponding polynomial equations
  we then arrive at the group outlined in case {\bf 31} of Section
  \ref{sec:nilpcen}.
\end{enumerate}
\end{example}

\subsection{Computing representatives}\label{sec:comprep}

Let $e\in \g_1^\cC $ be nilpotent lying in the homogeneous $\ssl_2$-triple
$t=(h,e,f)$.  The elements of $\Ho^1\Zm_0(t)$ correspond to the
$\SL(9,\R)$-orbits that are contained in the $\SL(9,\C)$-orbit of $t$.
Here we consider the problem of computing representatives of these orbits.
That is, for a given $[g] \in \Ho^1 \Zm_0(t)$ we want to find a
$h\in \SL(9,\C)$ such that $h\cdot t$ is a representative of the
$\SL(9,\R)$-orbit corresponding to $[g]$.

The map $\delta$ from Proposition \ref{p:serre} maps the $\SL(9,\R)$-orbit
of $h\cdot t$ to $[c]\in \Ho^1\Zm_0(t)$ where $c=h^{-1} \upgam h$. In our
situation the cocycle $c=g$ is given, the problem is to find $h$.
So we need to find $h\in \SL(9,\C)$ such that $\upgam h = hg$.

For this we work in the space $M_9(\C)$ of $9\times 9$-complex matrices,
which we view as a vector space over $\R$ of dimension 162. We note that
the equation $\upgam h = hg$ is equivalent to a set of 81 $\R$-linear equations
on the entries of $h$. We solve this set of linear equations over $\R$,
to get a basis $b_1,\ldots,b_m$
of an $\R$-subspace of $M_9(\C)$. Any complex matrix $h$
that satisfies $\upgam h = hg$ is an $\R$-linear combination of the elements
of this basis.

Now we can take indeterminates $x_1,\ldots,x_m$, form the matrix
$b= \sum_i x_i b_i$, compute its determinant $d\in \C[x_1,\ldots,x_m]$,
and determine $\alpha_1,\ldots,\alpha_m\in \R$ such that
$d(\alpha_1,\ldots,\alpha_m)=1$. Then $h=\sum_i \alpha_i b_i$ does the job.

Note that it is straightforward to find a subset $b_{i_1},\ldots,b_{i_r}$
of the basis whose span contains invertible matrices. Then we can do
the same as before, that is, set $b= \sum_k x_{i_k} b_{i_k}$ and find
values for the $x_{i_j}$ such that the determinant is 1. This way we have
to deal with a lot fewer indeterminates. Alternatively, in practice it
turned out also to be possible to try some random values of the
$x_{i_k}$ to find an element of determinant 1.

\subsection{Computing the Galois cohomology of a torus}\label{sec:H1T}

Let $\TT=(T,\sigma_T)$ be a complex torus. Let $X_T$ be its cocharacter
group which we write additively, so that $X_T\cong \Z^r$ for some $r>0$.
Let $\tau_T$ be as in Definition \ref{def:tauT} and let
\[e\colon X_T\to \TT(\C),\quad x\mapsto x(-1).\]
be as in Construction \ref{cons:e*}. As seen in Proposition \ref{p:e*}
the map $e_*^{(1)}\colon \Ho^1 X_T\to \Ho^1\hssh \TT$ is an isomorphism.
We have $e_*^{(1)} ([x]) = [ x(-1) ]$.

Now $X_T$ is abelian so Definition \ref{d:H1} translates to
$$ \Ho^1 X_T = \ker(1+\tau_T)/\im(1-\tau_T).$$
So computing $\Ho^1 \hssh \TT$ reduces to computing $\ker(1+\tau_T)$,
$\im(1-\tau_T)$ and the quotient of the two. We write the matrix of $\tau_T$
with respect to the standard basis of $\Z^r$. Then a basis of $\im(1-\tau_T)$
can be computed using algorithms for computing Hermite normal forms
(see \cite{Sims}). Computing $\ker(1+\tau_T)$ amounts to solving a set of
homogeneous linear equations over $\Z$. There are well-known algorithms for
this, but they are not very well documented; one reference is
\cite{Graaf2017}, \S 6.2. Finally computing the quotient can be done using
a Smith normal form computation (see \cite{Sims}).

\begin{example}
  Let $T$ be a 3-dimensional torus consisting of the elements $T(s,t,u)$ where
  $s,t,u\in\C^\times$ and $T(s,t,u)T(s',t',u') = T(ss',tt',uu')$.
  Let $\sigma : T\to T$ be the conjugation given by
  $$\sigma( T(s,t,u) ) = T(\ov s^{-2} \ov t^{-2}\ov u^{-1},
  \ov s^{2} \ov t^{-3}\ov u^{2}, \ov s^{-1} \ov t^{-2}\ov u^{-2} ).$$
  The cocharacter group has basis $x_s,x_t,x_u$, where $x_s(z)=T(z,1,1)$,
  $x_t(z)=T(1,z,1)$, $x_u(z)=T(1,1,z)$. We have
  $$\tau_T(x_s)(z) = \sigma(x_s(\ov z )) = T(z^{-2},z^2,z^{-1})
  =x_s(z)^{-2}x_t(z)^2x_u(z)^{-1}.$$
  Similarly we compute $\tau_T(x_t)$, $\tau_T(x_u)$. Then the matrix of $\tau_T$
  with respect to the standard basis $e_1,e_2,e_3$ of $\Z^3$ is
  $$M=\SmallMatrix{-2 &-2 & -1\\ 2 & 3 & 2\\ -1 & -2 & -2}.$$
  The kernel of $1+M$ is spanned by $e_1-e_3$, $e_2-2e_3$. The image of
  $1-M$ is spanned by $e_1-e_3$, $2e_2-4e_3$. We see that the quotient has
  order 2 and consists of $[0]$ and $[e_2-2e_3]$. The second element
  corresponds to $x=x_tx_u^{-2}$ and $x(-1) = T(1,-1,1)$. Hence $\Ho^1
  \TT = \{ [1], [T(1,-1,1)]\}$.
\end{example}

\subsection{Deciding conjugacy of nilpotent elements}\label{sec:decconj}

Here we let $\a$ be a {\em complex} semisimple Lie algebra with a
$\Z_m$-grading $\a = \oplus_{i\in \Zm} \a_i$. We note that $\a_0$ is
reductive in $\a$ (see \cite[Lemma 8.1(c)]{Kac1995}). This means that
$\a_0=\s\oplus \t$, where $\s$ is semisimple and $\t$ is a central subalgebra
such that $\ad_\a x$ is diagonalizable for all $x\in\t$.
We let $A_0$ be an algebraic group
of automorphisms of $\a$ stabilizing each component $\a_i$ and
such that the identity component $A_0^\circ$
is the connected algebraic subgroup of the adjoint group of $\a$
with Lie algebra $\a_0$ (or, more precisely, $\ad_{\a} \a_0$).
We also assume that the component group is given explicitly (that is, we have
explicit automorphisms of $\a$ whose cosets form the component group).
The problem that we consider is the following: given nilpotent $e_1,e_2
\in \a_1$, decide whether they are conjugate under $A_0$.

First we describe how to associate a sequence of integers, called the
{\em characteristic} to a nilpotent element.
Let $e\in \a_1$ be nilpotent lying in the homogeneous $\ssl_2$-triple
$(h,e,f)$. Let $\h_0\subset \a_0$ be a Cartan subalgebra containing $h$.
Let $\Phi_0$ denote the root system of $\a_0$ with respect to $\h_0$, and
let $\Delta_0$ be a set of simple roots. The Weyl group $W(\Phi_0)$ of
$\Phi_0$ acts on $\h_0$ in the following way. Let $\alpha\in \Phi$
and let $s_\alpha$ be the corresponding reflection in $W(\Phi_0)$.
Let $h_\alpha$ be the unique element of $[\a_{0,\alpha},\a_{0,-\alpha}]$
(where $\a_{0,\pm \alpha}$ are the root spaces corresponding to $\pm \alpha$)
with $\alpha(h_\alpha)=2$. Then $s_\alpha(u) = u -\alpha(u)h_\alpha$ for
$u\in \h_0$. Let $z=x+iy$, $x,y\in \R$, then we write $z>0$ if
$x>0$ or $x=0$ and $y>0$. Let $C_0 = \{ u\in \h \mid \alpha(u)\geq 0
\text{ for all } \alpha\in \Delta_0\}$. It is routine to show that every
$u\in \h_0$ is $W(\Phi_0)$-conjugate to a unique element of $C_0$.
Let $h'\in C_0$ be conjugate to $h$. Then the characteristic of the
orbit of $e$ is the sequence consisting of $\alpha(h')$ for $\alpha\in
\Delta_0$.

Now we discuss to which extent the $A_0^\circ$-orbit of $e$ determines its
characteristic. For the computation of the characteristic we make many choices.
All Cartan subalgebras of $\a_0$ are conjugate under
$A_0^\circ$ hence the characteristic does not depend on the choice of
Cartan subalgebra containing $h$ and neither does it depend on the choice
of representative $e$ of the orbit, nor on the choice of $\ssl_2$-triple, nor
on the choice of set of simple roots $\Delta_0$. However, it does depend
on the order in which we list the simple roots. This order can be prescribed
to some extend by fixing an enumeration of the nodes of the Dynkin diagram.
But if the Dynkin diagram has nontrivial automorphisms there are more
ways to list the simple roots according to the prescribed ordering.
So the characteristic is determined by the $A_0^\circ$-orbit of $e$ up to
an automorphism of the Dynkin diagram of $\Phi_0$. We say that two sequences
of integers, whose entries correspond to the nodes of the Dynkin diagram
of $\Phi_0$, are essentially different if there is no diagram automorphism
mapping one to the other. Otherwise we say that they are essentially equal.
So using Theorem \ref{thm:3} we conclude that
if the characteristics of two nilpotent elements are essentially different
then they lie in different orbits.

We note also that if $\t$ is nonzero then there are infinitely many
elements of the set $C_0$ yielding the same characteristic. So in that case
quite a few nilpotent orbits may share the same characteristic. The extreme
case happens when $\s=0$, as then all characteristics are empty, so all
nilpotent elements have the same characteristic.

Now we describe an algorithm by which we can always decide
$A_0^\circ$-conjugacy of two nilpotent elements.
It is, however, computationally more difficult than
the algorithm to compute the characteristic.
Let $t_i = (h_i,e_i,f_i)$ be homogeneous $\ssl_2$-triples for $i=1,2$.
By \cite[Theorem 8.3.6]{Graaf2017} the $e_i$ are $A_0^\circ$-conjugate
if and only if $h_1$, $h_2$ are $A_0^\circ$-conjugate.
Therefore, starting with $e_1,e_2$ the first step is to compute homogeneous
$\ssl_2$-triples $t_i$. After that we compute Cartan subalgebras
$\h_0^1,\h_0^2$ of $\a_0$ such that $\h_0^i$ contains $h_i$.
Then we compute the root systems $\Phi_0^i$ of $\a_0$ corresponding to
$\h_0^i$. Corresponding to each root system we compute a canonical
generating set of $\a_0$. Mapping the elements of the first canonical generating
set (corresponding to $\Phi_0^1$) to the corresponding elements of
the second canonical generating set (corresponding to $\Phi_0^2$) extends
to a unique automorphism $\sigma$ of $\a_0$ (see \cite{jac}, Chapter IV,
Theorem 3). By the methods of \cite{CMT04}, see also
\cite[Section 5.7]{Graaf2017}, we can decide whether $\sigma$ is an
inner automorphism, that is, whether $\sigma\in A_0^\circ$. If not we
compose $\sigma$ with a permutation of the canonical generators corresponding
to $\Phi_0^1$ to obtain a $\sigma\in A_0^\circ$ with $\sigma(\h_1)=\h_2$.
Set $h_1'=\sigma(h_1)$. Let $C_0^2 = \{ u\in \h_2\mid \alpha(u)\geq 0\}$
be as above. Compute $w_1,w_2\in W(\Phi_2)$ such that $w_1h_1'$ and $w_2h_2$
lie in $C_0^2$. Then $h_1, h_2$ (and hence $e_1$, $e_2$) are $A_0^\circ$-conjugate
if and only if $wh_1' = wh_2$.

If the component group of $A_0$ is not trivial then for each coset $\varphi
A_0^\circ$ in the component group one has to check whether $e_1$ and
$\varphi(e_1)$ are $A_0^\circ$-conjugate.

\begin{example}
  Let $\g^\cC$
  be the $\Z_3$-graded Lie algebra considered in this paper (see the
  beginning of the section). Then $\g^\cC_0$ is simple of type $A_8$. Hence the
  Dynkin diagram has one nontrivial automorphism. However, by inspecting the
  characteristics of the nilpotent orbits (they are listed in Table
  \ref{tab:orbitreps})
  it is seen that applying the corresponding permutation to a characteristic
  does not yield the characteristic of a different orbit. So in this case
  the characteristic determines the nilpotent orbit.

  Let $p\in \g_1^{\cC}$ be a semisimple element of the set $\Fm_6$ (see Section
  \ref{sec:semsim}). Let $\a = \z_{\g^\cC}(p)$ be the centralizer of $p$ in
  $\g^\cC$. Then $\a$ inherits the grading from $\g^\cC$. Here $\a$ is of
  type $E_6$, and $\a_0$ is semisimple of type $3A_2$. So there are many
  nontrivial diagram automorphisms. However, here $A_0 = \Zm_{\Gtil_0}^\circ$ is
  not connected and the elements of the component group happen to permute
  the nilpotent $A_0^\circ$-orbits in such a way that from each group of
  orbits with essentially equal characteristics only one remains. So also in
  this case the characteristic can be used to determine conjugacy of
  nilpotent elements.

  Let $p\in \g_1^{\cC}$ be a semisimple element of the set $\Fm_5$. Again set
  $\a = \z_{\g^\cC}(p)$. Then $\a_0$ is reductive of type $A_2+T_2$ where
  $T_2$ denotes a 2-dimensional toral center. In this case there are
  different nilpotent orbits having the essentially the same characteristic,
  because of the $T_2$ component. (It turns out that there are six
  characteristics and each of them is the characteristic of three nilpotent
  orbits.) So here we cannot decide conjugacy by only looking at the
  characteristic and we have to resort to the more brute force method
  outlined above.
\end{example}


\section{The centralizers of the nilpotent orbits and their Galois cohomology}\label{sec:nilpcen}

For each nilpotent orbit we describe the centralizer in $\SL(9,\C)$ of
a corresponding $\ssl_2$-triple. In all cases this centralizer is denoted
by $\Zm_0$. Its Lie algebra is denoted by $\z_0$.
Secondly, we determine the Galois cohomology set $\Ho^1\hm \Zm_0$.

Throughout we use the facts that $\Ho^1 \GL(n,\C) = \Ho^1 \SL(n,\C)=1$. In
particular, if $T$ is a torus with the standard $\Gamma$-action then
$\Ho^1 T=1$. For tori with a different $\Gamma$-action the Galois cohomology
is computed using the method of Section \ref{sec:H1T}. When
we say that two groups with $\Gamma$-action are isomorphic we mean
that they are $\Gamma$-equivariantly isomorphic.

The numbers in bold refer the position in the list of Elashvili and
Vinberg (\cite{VE1978}). However, we used different representatives
that were computed with the algorithms described in \cite{Graaf2017}, Chapter 8.

The {\em rank} of a trivector $u$ in $\bigwedge^3 \C^9$ is the dimension of the
smallest subspace $U$ of $\C^9$ such that $u\in \bigwedge^3 U$. The
classification of the nilpotent trivectors in $\bigwedge^3 \C^9$
of rank $<9$ coincides with the classification of the trivectors in
$\bigwedge^3 \C^8$ (we remark that all trivectors in this space are
nilpotent). The latter classification has been carried out by Gurevich,
see \cite{Gurevich1935a}, who obtained 23 orbits. To each orbit Gurevich
assigned a roman numeral, see \cite{Gurevich1964}. Below for each trivector
of rank $<9$ we also give its rank, and its Gurevich number. The orbit
with number I just consists of 0, so we do not list it.

\noindent{\bf 1} Representative: $e_{126}+e_{135}-e_{234}+e_{279}+e_{369}+e_{459}+e_{478}+e_{568}$.\\
Here
\[\Zm_0 = \mu_3:=\{\diag(a,a,a,a,a,a,a,a,a)\mid a^3=1\}.\]
Hence $\Ho^1\hm \Zm_0$ is trivial by Corollary \ref{c:2m+1}.

\noindent{\bf 2} Representative: $e_{126}+e_{145}-e_{234}+e_{279}+e_{369}-e_{378}+e_{478}+e_{568}$.\\
Here $\Zm_0 = \mu_3$. Hence $\Ho^1\hm \Zm_0$ is trivial by
Corollary \ref{c:2m+1}.

\noindent{\bf 3} Representative: $e_{126}+e_{145}-e_{235}+e_{279}+e_{369}-e_{378}+e_{478}+e_{568}$.\\
Here $\Zm_0 = \mu_3$. Hence $\Ho^1\hm \Zm_0$ is trivial by
Corollary \ref{c:2m+1}.

\noindent{\bf 4} Representative: $e_{127}+e_{145}-e_{234}+e_{279}+e_{369}-e_{378}+e_{478}+e_{568}$.\\
Here $\Zm_0$ is of order 6 and consists of the elements of $\mu_3$ along with
$$X(\zeta) = \SmallMatrix{ 0 & 0 & 0 & 0 & 0 & 0 & 0 & 0 & \zeta\\
                0 & -\zeta & 0 & 0 & 0 & 0 & 0 & 0 & 0\\
                0 & 0 & 0 & -\zeta & 0 & 0 & 0 & 0 & 0\\
                0 & 0 & -\zeta & 0 & 0 & 0 & 0 & 0 & 0\\
                0 & 0 & 0 & 0 & 0 & -\zeta & 0 & 0 & 0\\
                0 & 0 & 0 & 0 & -\zeta & 0 & 0 & 0 & 0\\
                0 & 0 & 0 & 0 & 0 & 0 & -\zeta & 0 & 0\\
                0 & 0 & 0 & 0 & 0 & 0 & 0 & -\zeta & 0\\
                \zeta & 0 & 0 & 0 & 0 & 0 & 0 & 0 & 0\\},$$
where $\zeta^3=1$. This is a cyclic group of order $6$, generated by
$X(\zeta_0)$ where $\zeta_0$ is a primitive third root of unity.
By Lemma \ref{l:explicit} we have $\Ho^1\hm \Zm_0=\{[1],[c]\}$, where
$c=X(\zeta_0)^3$.

\noindent{\bf 5} Representative: $e_{127}+e_{145}-e_{235}+e_{279}+e_{369}-e_{468}+e_{478}+e_{568}$.\\
Here $\Zm_0$ is of order 6 and consists of the elements of $\mu_3$ along with
$$X(\zeta)=\SmallMatrix{ -\zeta & 0 & 0 & 0 & 0 & 0 & 0 & 0 & \zeta\\
                0 & 0 & -\zeta & 0 & 0 & 0 & 0 & 0 & 0\\
                0 & -\zeta & 0 & 0 & 0 & 0 & 0 & 0 & 0\\
                0 & 0 & 0 & \zeta & 0 & 0 & 0 & 0 & 0\\
                0 & 0 & -\zeta & -\zeta & 0 & 0 & 0 & 0 & 0\\
                0 & 0 & 0 & 0 & 0 & 0 & -\zeta & 0 & 0\\
                0 & 0 & 0 & 0 & 0 & -\zeta & 0 & 0 & 0\\
                0 & 0 & 0 & 0 & 0 & 0 & 0 & \zeta & 0\\
                0 & 0 & 0 & 0 & 0 & 0 & 0 & 0 & \zeta\\},$$
where $\zeta^3=1$. This is a cyclic group of order $6$, generated by
$X(\zeta_0)$ where $\zeta_0$ is a primitive third root of unity.
By Lemma \ref{l:explicit} we have $\Ho^1\hm \Zm_0=\{[1],[c]\}$, where
$c=X(\zeta_0)^3$.

\noindent{\bf 6}
Representative: $e_{136}+e_{145}-e_{234}+e_{279}-e_{378}+e_{469}+e_{568}$.\\
Here $\Zm_0$ consists of the elements
$$\diag(a,a,a^{-2},a,a^{-2},a,a,a,a^{-2})\text{ for } a\in \C.$$
We see that $\Zm_0$ is a 1-dimensional torus,
therefore $\Ho^1\hm \Zm_0$ is trivial.

\noindent{\bf 7} Representative: $e_{127}+e_{145}-e_{235}+e_{368}-e_{378}-e_{468}+e_{469}+e_{579}$.\\
Here $\Zm_0$ is cyclic of order 6, generated by
$$X(\zeta_0)= \SmallMatrix{ 0 & -\zeta_0 & 0 & 0 & 0 & 0 & 0 & 0 & 0\\
                -\zeta_0 & 0 & 0 & 0 & 0 & 0 & 0 & 0 & 0\\
                0 & 0 & 0 & \zeta_0 & 0 & 0 & 0 & 0 & 0\\
                0 & 0 & \zeta_0 & 0 & 0 & 0 & 0 & 0 & 0\\
                0 & 0 & 0 & 0 & -\zeta_0 & 0 & 0 & 0 & 0\\
                0 & 0 & 0 & 0 & 0 & -\zeta_0 & 0 & 0 & 0\\
                0 & 0 & 0 & 0 & 0 & \zeta_0 & \zeta_0 & 0 & 0\\
                0 & 0 & 0 & 0 & 0 & 0 & 0 & -\zeta_0 & 0\\
                0 & 0 & 0 & 0 & 0 & 0 & 0 & \zeta_0 & \zeta_0\\},$$
where $\zeta_0$ is a primitive sixth root of unity. By Lemma \ref{l:explicit} we
have $\Ho^1\hm \Zm_0=\{[1],[c]\}$, where $c=X(\zeta_0)^3$.

\noindent{\bf 8} Representative: $e_{136}+e_{145}-e_{235}+e_{279}-e_{378}+e_{469}+e_{478}+e_{568}$.\\
Here $\Zm_0$ is cyclic of order 6, generated by
$$X(\zeta_0)=\SmallMatrix{ -\zeta_0 & 0 & 0 & 0 & 0 & 0 & 0 & 0 & 0\\
                0 & 0 & 0 & 0 & 0 & 0 & 0 & 0 & \zeta_0\\
                0 & 0 & 0 & -\zeta_0 & 0 & 0 & 0 & 0 & 0\\
                0 & 0 & -\zeta_0 & 0 & 0 & 0 & 0 & 0 & 0\\
                0 & 0 & 0 & 0 & 0 & \zeta_0 & 0 & 0 & 0\\
                0 & 0 & 0 & 0 & \zeta_0 & 0 & 0 & 0 & 0\\
                0 & 0 & 0 & 0 & 0 & 0 & -\zeta_0 & 0 & 0\\
                0 & 0 & 0 & 0 & 0 & 0 & 0 & -\zeta_0 & 0\\
                0 & \zeta_0 & 0 & 0 & 0 & 0 & 0 & 0 & 0\\},$$
where $\zeta_0$ is a primitive third root of unity. By Lemma \ref{l:explicit} we
have $\Ho^1\hm \Zm_0=\{[1],[c]\}$, where $c=X(\zeta_0)^3$.

\noindent{\bf 9} Representative: $e_{127}+e_{145}-e_{234}+e_{279}+e_{369}+e_{568}-e_{578}+e_{678}$.\\
Here $\Zm_0\simeq A\times \mu_3$,  where the group $A$ is isomorphic to $S_3$ and
generated by the matrices
$$X=\SmallMatrix{ 0 & 0 & 0 & 0 & 0 & 0 & 0 & 0 & 1\\
                0 & -1 & 0 & 0 & 0 & 0 & 0 & 0 & 0\\
                0 & 0 & 0 & -1 & 0 & 0 & 0 & 0 & 0\\
                0 & 0 & -1 & 0 & 0 & 0 & 0 & 0 & 0\\
                0 & 0 & 0 & 0 & 0 & -1 & 0 & 0 & 0\\
                0 & 0 & 0 & 0 & -1 & 0 & 0 & 0 & 0\\
                0 & 0 & 0 & 0 & 0 & 0 & -1 & 0 & 0\\
                0 & 0 & 0 & 0 & 0 & 0 & 0 & -1 & 0\\
                1 & 0 & 0 & 0 & 0 & 0 & 0 & 0 & 0\\},
\quad
Y=\SmallMatrix{ 0 & 0 & 0 & 0 & 0 & 0 & 0 & 0 & -1 \\
                0 & 0 & -1 & 0 & 0 & 0 & 0 & 0 & 0 \\
                0 & 0 & 0 & 1 & 0 & 0 & 0 & 0 & 0 \\
                0 & -1 & 0 & 0 & 0 & 0 & 0 & 0 & 0 \\
                0 & 0 & 0 & 0 & 0 & 0 & 1 & 0 & 0 \\
                0 & 0 & 0 & 0 & 1 & 0 & 0 & 0 & 0 \\
                0 & 0 & 0 & 0 & 0 & 1 & 0 & 0 & 0 \\
                0 & 0 & 0 & 0 & 0 & 0 & 0 & 1 & 0 \\
                1 & 0 & 0 & 0 & 0 & 0 & 0 & 0 & -1}. $$
Here $X$ is of order 2 and $Y$ is of order 3. By Lemma \ref{l:explicit} we
have $\Ho^1\hm \Zm_0=\{[1],[X]\}$.

\noindent{\bf 10} Representative: $e_{127}+e_{136}+e_{145}-e_{234}+e_{379}+e_{469}+e_{478}+e_{568}$.\\
Here $\Zm_0$ consists of the elements
$$\diag(a^2b^2,b,a^2b^2,a,b,a^2b^2,a,a,b)\text{ for } a,b\in \C \text{ with }
a^3=b^3=1.$$
This is a finite abelian group of order $9$. Hence $\Ho^1\hm \Zm_0$ is trivial by
Corollary \ref{c:2m+1}.

\noindent{\bf 11} Representative: $e_{136}+e_{145}-e_{236}+e_{279}-e_{378}+e_{469}+e_{568}$.\\
Here $\Zm_0$ consists of the elements
$$\diag(a,a,a^{-2},a,a^{-2},a,a,a,a^{-2})\text{ for } a\in \C.$$
We see that $\Zm_0$ is a 1-dimensional torus, therefore $\Ho^1\hm \Zm_0=1$.

\noindent{\bf 12} Representative: $e_{127}+e_{145}-e_{234}+e_{369}+e_{479}+e_{568}-e_{578}$.\\
Here $\Zm_0$ consists of the elements
$$\diag(a,a^{-2},a,a,a^{-2},a,a,a,a^{-2})\text{ for } a\in \C.$$
We see that $\Zm_0$ is a 1-dimensional torus, therefore $\Ho^1\hm \Zm_0=1$.

\noindent{\bf 13} Representative: $e_{126}+e_{145}-e_{235}+e_{379}+e_{469}+e_{478}+e_{568}$.\\
Here $\Zm_0$ consists of the elements
$$\diag(a,a^{-2},a^4,a,a^{-2},a,a^{-2},a,a^{-2})\text{ for } a\in \C.$$
We see that $\Zm_0$ is a 1-dimensional torus, therefore $\Ho^1\hm \Zm_0=1$.

\noindent{\bf 14} Representative: $e_{127}+e_{136}+e_{145}-e_{235}+e_{379}+e_{469}+e_{478}+e_{568}$.\\
Here $\Zm_0$ consists of the elements
$$\diag(a^7,a^4,a,a^7,a^4,a,a^7,a^4,a) \text{ with } a^9=1.$$
This is a finite abelian group of order $9$. Hence $\Ho^1\hm \Zm_0$ is trivial by
Corollary \ref{c:2m+1}.

\noindent{\bf 15} Representative: $e_{136}+e_{145}-e_{235}+e_{279}+e_{469}+e_{478}-e_{578}$.\\
Here $\Zm_0$ is the union of two sets. One set consists of the matrices
$$X(a,b,c)=\diag(a,b,c,a,a,b,a,a,c) \text{ with } a^3=1,~ bc=a^2.$$
So it is a diagonal algebraic group of dimension 1 with three components.

The other set consists of the matrices
$$Y(a,b,c)=\SmallMatrix{ a & 0 & 0 & 0 & 0 & 0 & 0 & 0 & 0\\
                0 & 0 & 0 & 0 & 0 & 0 & 0 & 0 & -b\\
                0 & 0 & 0 & 0 & 0 & -c & 0 & 0 & 0\\
                0 & 0 & 0 & 0 & a & 0 & 0 & 0 & 0\\
                0 & 0 & 0 & a & 0 & 0 & 0 & 0 & 0\\
                0 & 0 & b & 0 & 0 & 0 & 0 & 0 & 0\\
                0 & 0 & 0 & 0 & 0 & 0 & a & 0 & 0\\
                0 & 0 & 0 & 0 & 0 & 0 & 0 & a & 0\\
                0 & c & 0 & 0 & 0 & 0 & 0 & 0 & 0\\},$$
with $a^3=-1$ and $bc=-a^2$. This set has three components too.

The multiplication law in $\Zm_0$ is as follows:
\[ X(a,b,c)\cdot X(a',b',c')=X(aa',bb', cc').\]
\[X(a,b,c)\cdot Y(a',b',c')=Y(aa',bb',cc')\]
\[Y(a',b',c')\cdot X(a,b,c)=Y(aa',cb',bc')\]
\[ \quad Y(a,b,c)\cdot Y(a',b',c')=X(aa',-bc',-cb')\]

\begin{lemma*} $\Ho^1\hm \Zm_0=\{1,[b_1],[b_2]\}$, where $b_1=Y(-1,-1,1)$
and $b_2=Y(-1,1,-1)$.
\end{lemma*}

\begin{proof}
Write $C=\mu_6=\{a\in\C^\times\mid a^6=1\}$.
Consider the homomorphism
\[j\colon \Zm_0\to C,\quad X(a,b,c)\mapsto a,\quad Y(a,b,c)\mapsto a.\]
Set
\[A=\ker\,j=\{X(1,b,c)\mid bc=1\}.\]
We have a short exact sequence
\begin{equation}\label{e:short-85}
1\to A\labelto{i} \Zm_0\labelto{j} C\to 1
\end{equation}
where $i$ is the inclusion map.
By Proposition \ref{p:serre-prop38} we obtain a cohomology exact sequence
\begin{equation}\label{e:cohom-85}
1\to A^\Gamma\to \Zm_0^\Gamma\to C^\Gamma\labelto{\delta} \Ho^1\hm A\labelto{i_*} \Ho^1\hm \Zm_0\labelto{j_*} \Ho^1 C.
\end{equation}

We have $\Ho^1 C=\{[1],[-1]\}$.
Thus $\Ho^1\hm  \Zm_0=j_*^{-1}[1]\cup j_*^{-1}[-1]$.

Since $A$ is a torus we have $\Ho^1\hm A=1$. Hence from \eqref{e:cohom-85}
it follows that $j_*^{-1}[1]=\{[1]\}$.

We compute $j_*^{-1}[-1]$.
Write $c_1=-1\in C$, $b_1=Y(-1,-1,1)\in \Zm_0$; then $b_1$ is a lift of $c_1$.
We have   $\upgam b_1=\bar b_1=b_1$ and $b_1^2=1$.
Thus $b_1\in \Zl^1\hm \Zm_0$, and we can twist the exact sequence \eqref{e:short-85} by the cocycle $b_1$.
By Corollary \ref{c:39-cor1}, the set $j_*^{-1}[-1]$ is in a canonical bijection with the set of orbits
of the group $(\hs_{b_1}\hm C)^\Gamma$ in the set $\Ho^1\hm \hs_{b_1}\hm A$.

Since $C$ is an abelian group, we have $_{b_1}\hm C=\hs_{c_1}\hm C=C$.
Thus $(\hs_{b_1}\hm C)^\Ga=C^\Ga=\{1,-1\}$.

For $b\in \C^\times$, we have
\[^{\gamma*} X(1,b,b^{-1})=b_1\cdot\overline{X(1,b,b^{-1})}\cdot b_1^{-1}=
b_1\cdot X(1,\bar b,\bar b^{-1})\cdot b_1^{-1}=X(1,\bar b^{-1},\bar b).\]
We  $\Gamma$-equivariantly  identify $_{b_1}\hm A$ with
\[\C_{\rm tw}\ :=\C^\times\text{ with the twisted   $\Gamma$-action}\  \upgam b=\bar b^{-1}.\]
By Example \ref{x:cohom}(3) we have
\[ \Ho^1\hm \hs_{b_1}\hm A=\Ho^1\hs \C_{\rm tw}=\{[1],[-1]\}.\]

We compute the action of $c_1=-1\in (\hs_{b_1} C)^\Ga=C^\Ga$ on $1\in \Ho^1\hm\hs_{b_1}\hm A$.
Since $b_1^2=1$, by Lemma \ref{lem:1c} the element $c_1\in C^\Ga$ fixes
$[1]\in \Ho^1\hm \hs_{b_1}\hm A$.
Since $\#  \Ho^1\hm \hs_{b_1}\hm A=2$, we see that $c_1$ acts on
$\Ho^1\hm \hs_{b_1}\hm A$ trivially.
Thus $(\hs_{b_1}\hm C)^\Ga=C^\Ga$ acts on $ \Ho^1\hm \hs_{b_1}\hm A$ trivially.
We see that $(\hs_{b_1}\hm C)^\Ga$ has {\em two} orbits in  $\Ho^1\hm \hs_{b_1}\hm A$,
with representatives $1$ and $-1=X(1,-1-1)$ in $\Zl^1\hm  \hs_{b_1}\hm A$.
Their images in $\Ho^1\hm \Zm_0$ are $\tau_{\hs b_1}(1)=[1\cdot b_1]=[b_1]$ and
\[ \tau_{\hs b_1}[X(1,-1,-1)]=[X(1,-1,-1)\cdot Y(-1,-1, 1)]=[Y(-1,1,-1)]=[b_2], \]
where $\tau_{\hs b_1}$ is the map of \eqref{e:tau-con}.
\end{proof}

\begin{proof}[Alternative proof]
Consider the homomorphism
\[ \alpha_2\colon \Zm_0\to \mu_3,\quad  X(a,b,c)\mapsto a^2,\quad Y(a,b,c)\mapsto a^2.\]
We have a short exact sequence
\[1\to B'\to \Zm_0\to \mu_3\to 1,\]
where $B'=\ker\alpha_2$\hs.
Since $\Ho^1\hs\mu^3=\{1\}$ and $\mu_3(\R)=\{1\}$,
we see from the cohomology exact sequence \eqref{e:ABC-long}
that the natural map $\Ho^1\hs B'\to \Ho^1\hs \Zm_0$ is bijective.

We consider the homomorphism
\[ B'\to \mu_2,\quad X(a,b,c)\mapsto a,\quad Y(a,b,c)\mapsto a.\]
We have a short exact sequence
\begin{equation*}
1\to A\to B'\to \mu_2\to 1.
\end{equation*}
It admits a {\em splitting}
\[ -1\mapsto b_1=Y(-1,-1,1)\in B'.\]
Indeed, $\upgam b_1=b_1$ and $b_1^2=1$.
We have $\Ho^1\hs\mu_2=\{[1],[-1]\}$.
By Corollary \ref{c:split}(iv) we have
\[ \Ho^1\hs B'\cong (\Ho^1 A)/\mu_2\, \sqcup\hs (\Ho^1{}_{b_1}\hm A)/\mu_2\hs.\]
We have $\Ho^1 A=1$, whence $(\Ho^1 A)/\mu_2=1$.
As above, we have ${}_{b_1}\hm A\cong\C^\times_\tw$,
whence $\Ho^1{}_{b_1}\hm A\cong \Ho^1\hs\C^\times_\tw=\{[1],[-1]\}$.
By Corollary \ref{c:split-action}(i) the group $\mu_2$ acts trivially on $1\in \Ho^1{}_{b_1}\hm A$.
Hence it acts  on $\Ho^1{}_{b_1}\hm A$ with two orbits with representatives
$1$ and $-1=X(1,-1-1)$ in $\Zl^1\hm  \hs_{b_1}\hm A$.
As above, we obtain that $\Ho^1\Zm_0=\{[1], [b_1], [b_2]\}$.
\end{proof}

\noindent{\bf 16} Representative: $e_{127}+e_{136}+e_{145}-e_{234}+e_{379}+e_{469}+e_{568}$.\\
Here $\Zm_0$ consists of the elements
$$\diag(a^2b^2,b,a^2b^2,a,b,a^2b^2,a,a,b)
   \text{ for } a,b\in \C \text{ with } (ab)^3=1.$$
We see that $\Zm_0^\circ$ is a 1-dimensional torus (so that $\Ho^1 \Zm_0^\circ = 1$)
and the component group is of order 3. By Proposition \ref{p:C-3} it follows
that $\Ho^1 \Zm_0=1$.

\noindent{\bf 17} Representative: $e_{137}+e_{146}-e_{245}-e_{268}+e_{278}+e_{368}+e_{479}+e_{569}$.\\
Let $\zeta$ be a primitive third root of unity.
The centralizer is of order 18 and generated by the three elements
$h_1=\diag(\zeta,\ldots,\zeta)$,
$$h_2 = \SmallMatrix{-1 &  0 &  0 &  0 &  0 &  0 &  0 &  1 &  0\\
 0 & -1 & -1 &  0 &  0 &  0 &  0 &  0 &  0\\
 0 &  0 &  1 &  0 &  0 &  0 &  0 &  0 &  0\\
 0 &  0 &  0 & -1 &  1 &  0 &  0 &  0 &  0\\
 0 &  0 &  0 &  0 &  1 &  0 &  0 &  0 &  0\\
 0 &  0 &  0 &  0 &  0 &  1 &  1 &  0 &  0\\
 0 & 0 &  0 &  0 &  0 &  0 & -1 &  0 &  0\\
 0 & 0 &  0 &  0 &  0 &  0 &  0 &  1 &  0\\
 0 &  0 &  0 &  0 &  0 &  0 &  0 &  0 &  1},
\quad
h_3 = \SmallMatrix{ -\zeta - 1 & 0 & 0 & 0 & 0 & 0 & 0 & 0 & -\zeta - 2\\
0 & 1 & 0 & 0 & 0 & 0 & 0 & 0 & 0\\
0 & 0 & 1 & 0 & 0 & 0 & 0 & 0 & 0\\
0 & 3\zeta + 6 & 2\zeta + 4 & -\zeta - 1 & 0 & 0 & 0 & 0 & 0\\
0 & 6\zeta + 3 & 3\zeta + 3 & -2\zeta - 1 & \zeta & 0 & 0 & 0 & 0\\
0 & 0 & 0 & 0 & 0 & -\zeta - 1 & 0 & 0 & 0\\
0 & 0 & 0 & 0 & 0 & 2\zeta + 1 & \zeta & 0 & 0\\
-2\zeta - 1 & 0 & 0 & 0 & 0 & 0 & 0 & \zeta & -3\\
0 & 0 & 0 & 0 & 0 & 0 & 0 & 0 & 1}.$$

The elements $h_i$ satisfy the relations $h_1^3=h_2^2=h_3^3=1$, $h_2h_3h_2^{-1} = h_3^2$.
Moreover, these elements are polycyclic generators, meaning that every
element of the group can uniquely be written $h_1^ih_2^jh_3^k$ with
$0\leq i\leq 3$, $0\leq j\leq 1$, $0\leq k\leq 2$.

The element $h_2$ is of order 2 and fixed by $\Gamma$. Hence
by Lemma \ref{l:explicit} we conclude that $H^1=\{[1], [h_2]\}$.

\noindent{\bf 18} Representative: $e_{127}+e_{136}+e_{145}-e_{235}+e_{379}+e_{469}+e_{478}-e_{578}$.\\
Here $\Zm_0$ consists of
$$\diag(a^2b^2,a,b,a^2b^2,a^2b^2,a,b,a,b) \text{ with } a^3=b^3=1.$$
This is a finite abelian group of order 9. Hence $\Ho^1\hm \Zm_0$ is trivial by
Corollary \ref{c:2m+1}.

\noindent{\bf 19} Representative: $e_{136}+e_{145}-e_{236}+e_{279}-e_{378}+e_{478}+e_{569}$.\\
Here the identity component $\Zm_0^\circ$ consists of
$$\diag(a,a,a^{-2},a^{-2},a,a,a,a,a^{-2}) \text{ with } a\in \C^\times.$$
Set
$$Y=\SmallMatrix{ 1 & 0 & 0 & 0 & 0 & 0 & 0 & 0 & 0\\
                -1 & -1 & 0 & 0 & 0 & 0 & 0 & 0 & 0\\
                0 & 0 & 0 & -1 & 0 & 0 & 0 & 0 & 0\\
                0 & 0 & -1 & 0 & 0 & 0 & 0 & 0 & 0\\
                0 & 0 & 0 & 0 & 0 & -1 & 0 & 0 & 0\\
                0 & 0 & 0 & 0 & -1 & 0 & 0 & 0 & 0\\
                0 & 0 & 0 & 0 & 0 & 0 & 1 & 0 & 0\\
                0 & 0 & 0 & 0 & 0 & 0 & 0 & 1 & 0\\
                0 & 0 & 0 & 0 & 0 & 0 & 0 & 0 & -1\\}.$$
Then $Y$ is of order $2$. We set  $C=\{1,Y\}$; then we have
$\Zm_0 = C\times \Zm_0^\circ$.  Thus
\[\Ho^1\hm \Zm_0=\Ho^1\hm \Zm_0^\circ\times \Ho^1 C=1\times \Ho^1 C=\{1,[Y]\}.\]

\noindent{\bf 20} Representative: $e_{127}+e_{145}-e_{234}+e_{379}+e_{469}+e_{478}+e_{568}$.\\
Here $\Zm_0$ consists of the elements
$$\SmallMatrix{ d & 0 & 0 & 0 & 0 & 0 & 0 & 0 & -c\\
                0 & a & b & 0 & 0 & 0 & 0 & 0 & 0\\
                0 & c & d & 0 & 0 & 0 & 0 & 0 & 0\\
                0 & 0 & 0 & e & 0 & 0 & 0 & 0 & 0\\
                0 & 0 & 0 & 0 & a & b & 0 & 0 & 0\\
                0 & 0 & 0 & 0 & c & d & 0 & 0 & 0\\
                0 & 0 & 0 & 0 & 0 & 0 & e & 0 & 0\\
                0 & 0 & 0 & 0 & 0 & 0 & 0 & e & 0\\
                -b & 0 & 0 & 0 & 0 & 0 & 0 & 0 & a\\},$$
with $ad-bc=e^2$ and $e^3=1$. So $\Zm_0$ has three connected components, one for
each value of $e$. The identity component is isomorphic to $\SL(2,\C)$.
Since  $\Ho^1\hs \SL(2,\C)=1$, by Proposition \ref{p:C-3} we have $\Ho^1\hm \Zm_0=1$.

\noindent{\bf 21} Representative: $e_{127}+e_{136}-e_{245}+e_{379}+e_{469}+e_{478}+e_{568}$.\\
Here $\Zm_0$ consists of
$$\diag(b^4,b^{-5},b^{-2},b,b^4,b^{-2},b,b^{-2},b) \text{ with } b\in \C^\times.$$
Hence $\Zm_0$ is a 1-dimensional torus, and therefore $\Ho^1\hm \Zm_0$ is trivial.

\noindent{\bf 22} Representative: $e_{127}+e_{136}-e_{235}+e_{379}+e_{469}+e_{478}-e_{578}$.\\
Here $\Zm_0$ consists of
$$\diag(a,b,c,a,a,b,c,b,c) \text{ with } ab=c^2,~ c^3=1.$$
The identity component is a torus and hence has trivial cohomology.
The component group has order 3. By Proposition \ref{p:C-3}
it follows that $\Ho^1\hm \Zm_0=1$.

\noindent{\bf 23} Representative: $e_{127}+e_{136}+e_{145}-e_{235}-e_{468}+e_{479}+e_{568}$.\\
Here $\Zm_0$ is the union of two sets: the set
$$X(a,b)=\diag(a^2b^2,b,a,a^2b^2,a^2b^2,b,a,a,b) \text{ with } a^3b^3=1,$$
which is a diagonal algebraic group with three components, and the set
$$Y(a,b)=\SmallMatrix{ -a^2b^2 & 0 & 0 & 0 & 0 & 0 & 0 & 0 & 0\\
                0 & 0 & -b & 0 & 0 & 0 & 0 & 0 & 0\\
                0 & -a & 0 & 0 & 0 & 0 & 0 & 0 & 0\\
                0 & 0 & 0 & -a^2b^2 & 0 & 0 & 0 & 0 & 0\\
                0 & 0 & 0 & a^2b^2 & a^2b^2 & 0 & 0 & 0 & 0\\
                0 & 0 & 0 & 0 & 0 & 0 & -b & 0 & 0\\
                0 & 0 & 0 & 0 & 0 & -a & 0 & 0 & 0\\
                0 & 0 & 0 & 0 & 0 & 0 & 0 & 0 & a\\
                0 & 0 & 0 & 0 & 0 & 0 & 0 & b & 0\\},$$
with $a^3b^3=-1$. The latter has three components as well.
We have
\begin{align*}
  & X(a,b)X(a',b') = X(aa',bb')\\
  & X(a,b)Y(a',b') = Y(aa',bb')\\
  & Y(a,b)X(a',b') = Y(ab',ba')\\
  & Y(a,b)Y(a',b') = X(ab',ba').
\end{align*}

\begin{lemma*}
$\Ho^1\hm \Zm_0=\{1, [y_1]\}$, where $y_1=Y(i,i)$ with $i^2=-1$.
\end{lemma*}

\begin{proof}
We set
\[A=\{X(a,b)\mid ab=1\},\quad B=\Zm_0,\quad C=\mu_6.\]
We define a homomorphism
\[j\colon B\to C,\quad X(a,b)\mapsto ab,\ Y(a,b)\mapsto ab.\]
Then $\ker\,j=A$.
We have a short exact sequence
\begin{equation*}
1\to A\labelto{i} B\labelto{j} C\to 1
\end{equation*}
and the corresponding cohomology exact sequence
\begin{align*}
 C^\Gamma \labelto{\delta}
 \Ho^1\hm A &\labelto{i_*} \Ho^1 B \labelto{j_*} \Ho^1 C.
\end{align*}
Since $\Ho^1\hm A=1$ (as $A$ is a torus), we see that
$j_*^{-1}[1]=\{[1]\}\subset \Ho^1\hm B$.

Set $c_1=-1\in Z^1(C)\subset C$. We compute $j_*^{-1}[c_1]$.
We lift $c_1=-1$ to $y_1=Y(i,i)\in B$. Then
\[y_1\cdot\upgam y_1=Y(i,i)\cdot Y(-i,-i)=X(1,1)=1.\]
Thus $y_1\in \Zl^1 B$.
We twist the exact sequence by $y_1$.
We have
\[{}^{\gamma*}X(a,a^{-1})=y_1\cdot X(\bar a, \bar a^{-1})y_1^{-1}=X(\bar a^{-1}, \bar a).\]
Thus $_{y_1}\hm A\cong \C^\times_{\tw}$.
By Example \ref{x:cohom}(3),
\[ \Ho^1\hm\hs_{y_1}\hm A =\Ho^1\hs\C^\times_\tw=\{[1],[-1]\}.\]

We compute the action of
$c_1\in (_{c_1}\hm C)^\Ga=C^\Ga$ on $[1]\in H^1 \hs_{y_1}\hm A$.
By Lemma \ref{lem:1c} we have
\[[1]\cdot c_1=[y_1^{-2}]=[X(-1,-1)]\neq[1].\]
We see that $C^\Gamma$ acts nontrivially on $\Ho^1\hm\hs_{y_1}\hm A$,
and thus it has exactly one orbit.
We conclude that $\#j_*^{-1}[c_1]=\{\hs[1]\}$, $j_*^{-1}[c_1]=\{\hs[y_1]\}$.
Thus $\Ho^1\hm \Zm_0=\Ho^1\hm B=\{[1],[y_1]\}$, as required.
\end{proof}

\noindent{\bf 24} Representative: $e_{127}+e_{136}+e_{145}-e_{235}+e_{379}+e_{469}-e_{578}$.\\
Here $\Zm_0$ consists of
$$X(\delta,a) = \diag( \delta^2a^{-1},\delta,a,\delta^2a^2,\delta^2a^{-1},
\delta,a,\delta,a^{-2})$$
with $\delta^3=1$, $a\in \C^\times$ along with
$$Y(\zeta,b) = \SmallMatrix{ 0 & 0 & -\zeta^2b^{-1} & 0 & 0 & 0 & 0 & 0 & 0\\
                0 & \zeta & 0 & 0 & 0 & 0 & 0 & 0 & 0\\
                -b & 0 & 0 & 0 & 0 & 0 & 0 & 0 & 0\\
                0 & 0 & 0 & 0 & 0 & 0 & 0 & 0 & \zeta^2b^2\\
                0 & 0 & 0 & 0 & 0 & 0 & \zeta^2b^{-1} & 0 & 0\\
                0 & 0 & 0 & 0 & 0 & \zeta & 0 & 0 & 0\\
                0 & 0 & 0 & 0 & b & 0 & 0 & 0 & 0\\
                0 & 0 & 0 & 0 & 0 & 0 & 0 & \zeta & 0\\
                0 & 0 & 0 & b^{-2} & 0 & 0 & 0 & 0 & 0\\},$$
where $\zeta^3=-1$ and $b\in \C^\times$. We have
\begin{align*}
X(\delta,a)X(\delta',a') &= X(\delta\delta',aa')\\
X(\delta,a)Y(\zeta,b) &= Y(\delta\zeta,ab)\\
Y(\zeta,b)X(\delta,a) &= Y(\delta\zeta,\delta^2a^{-1}b)\\
Y(\zeta,b)Y(\zeta',b') &= X(\zeta\zeta',(\zeta')^2b\hs(b')^{-1}).
\end{align*}

\begin{lemma*}
$\Ho^1\hm \Zm_0=\{1,[y],[z]\}$ where $y=Y(-1,1)$, $z=Y(-1,-1)$.
\end{lemma*}

\begin{proof}
We set $A=\{X(1,a)\mid a\in\C^\times\}$ and $B=\Zm_0$.
We have a short exact sequence
\[1\to A\to B\labelto{j} C\to 1.\]
where $C=\mu_6$ and the homomorphism $j$ is defined by
\[X(\delta,a)\mapsto\delta,\ Y(\zeta,b)\mapsto \zeta.\]
We write $c=-1\in C^\Ga$. Then $c=j(y)$ and $y\in \Zl^1\hm B$, $y^2=1$.
We have $\Ho^1\hm C=\{[1],[c]\}$.
Thus $\Ho^1\hm B=\ker\,j_*\cup j_*^{-1}[c]$.
We have $C^\Gamma=\{1,c\}$, $(\hs_c C)^\Gamma=C^\Gamma=\{1,c\}$.

The kernel $\ker\,j_*$ is the image of $\Ho^1\hm A$ which is 1.
Thus $\ker\,j_*=\{1\}$.

By Corollary \ref{c:39-cor2} $j_*^{-1}[c]=\tau_y(\Ho^1\hm \hs_y A)$,
where the map $\tau_y$ is induced by the map
\[ t\mapsto ty\colon\, \Zl^1\hm  {}_y A\to \Zl^1\hm  B,\]
and the map $\tau_y$ in turn induces a bijection of the set
of orbits $\Ho^1\hm {}_y A/(\hs_c C)^\Ga$ with $j_*^{-1}[c]$.

We have   $y\overline{X(1,a)}y^{-1} = X(1,\bar a^{-1})$.
Hence by Example \ref{x:cohom}(3) we see that
\[\#\Ho^1\hm  \hs_y A=2,\quad \Ho^1\hm  \hs_y A=\{1,[X(1,-1)]\}.\]
Since  $y^2=1$,
by Lemma \ref{lem:1c} the element $c\in (\hs_ c C)^\Gamma=C^\Gamma$
acts trivially on $[1]\in \Ho^1\hm   \hs_y A$.
It follows that the group $(\hs_c C)^\Gamma$ acts trivially on
$\Ho^1\hm {}_y A$. By Corollary \ref{c:39-cor2}  the map
\[\tau_y\colon \Ho^1\hm  {}_y A\to j_*^{-1}[c]\]
is bijective.
Thus
\[\#j_*^{-1}[c]=2,\quad  j_*^{-1}[c]=\{\hs\tau_y[1], \tau_y[X(1,-1)]\hs\}=\{[y],[z]\},\]
where
\[ z=X(1,-1)\cdot y=X(1,-1)\cdot Y(-1,1)=Y(-1,-1).\]
We conclude  that $\Ho^1\hm B=\{1,[y],[z]\}$, as required.

Alternatively, this case can be treated using Corollary \ref{c:split}(iv), similarly to Case 15.
\end{proof}

\noindent{\bf 25} Representative: $e_{127}+e_{145}-e_{234}+e_{379}+e_{469}+e_{568}$.\\
Here $\Zm_0$ consists of the elements
$$\SmallMatrix{ d & 0 & 0 & 0 & 0 & 0 & 0 & 0 & -c\\
                0 & a & b & 0 & 0 & 0 & 0 & 0 & 0\\
                0 & c & d & 0 & 0 & 0 & 0 & 0 & 0\\
                0 & 0 & 0 & e & 0 & 0 & 0 & 0 & 0\\
                0 & 0 & 0 & 0 & a & b & 0 & 0 & 0\\
                0 & 0 & 0 & 0 & c & d & 0 & 0 & 0\\
                0 & 0 & 0 & 0 & 0 & 0 & e & 0 & 0\\
                0 & 0 & 0 & 0 & 0 & 0 & 0 & e & 0\\
                -b & 0 & 0 & 0 & 0 & 0 & 0 & 0 & a\\},$$
with $e(ad-bc)=1$. Hence $\Zm_0$ is isomorphic
to $\GL(2,\C)$. Therefore $\Ho^1\hm \Zm_0=1$.

\noindent{\bf 26} Representative: $e_{127}+e_{136}-e_{245}-e_{378}+e_{479}+e_{568}+e_{569}$.\\
Then $\Zm_0$ consists of
$$X(a,\zeta) = \diag(a,\zeta^2a^{-1},\zeta,\zeta,
   a,\zeta^2a^{-1},\zeta,\zeta,\zeta),
   \quad\text{where }\zeta^3=1,$$
and
$$Y(b,\delta) = \SmallMatrix{ 0 & -b & 0 & 0 & 0 & 0 & 0 & 0 & 0\\
                \delta^2b^{-1} & 0 & 0 & 0 & 0 & 0 & 0 & 0 & 0\\
                0 & 0 & 0 & -\delta & 0 & 0 & 0 & 0 & 0\\
                0 & 0 & -\delta & 0 & 0 & 0 & 0 & 0 & 0\\
                0 & 0 & 0 & 0 & 0 & b & 0 & 0 & 0\\
                0 & 0 & 0 & 0 & -\delta^2b^{-1} & 0 & 0 & 0 & 0\\
                0 & 0 & 0 & 0 & 0 & 0 & \delta & 0 & 0\\
                0 & 0 & 0 & 0 & 0 & 0 & 0 & 0 & \delta\\
                0 & 0 & 0 & 0 & 0 & 0 & 0 & \delta & 0\\},
   \quad\text{where }\delta^3=1.$$
We have
\begin{align*}
X(a,\zeta)X(b,\delta) &= X(ab,\zeta\delta),\\
X(a,\zeta)Y(b,\delta)&=Y(ab,\zeta\delta),\\
Y(b,\delta)X(a,\zeta) &= Y(\zeta^2a^{-1}b,\zeta\delta),\\
Y(a,\zeta)Y(b,\delta) &= X(-\delta^2ab^{-1},\zeta\delta).
\end{align*}

\begin{lemma*}
$\Ho^1\hm \Zm_0=\{1,[y]\}$ where $y=Y(i,1)$.
\end{lemma*}

\begin{proof}
Set $B= \Zm_0$ and
\[T=B^{0}=\{X(a,1)\mid a\in\C^\times\},\quad C=B/T.\]
We have a short exact sequence
\[1\to T\to B\labelto{j} C\to 1.\]
Then $C$ is a finite abelian group of order 6. By Lemma \ref{l:explicit}
we have $\#\Ho^1\hm  C=\{ [1], [c] \}$ where $c\in C$ is the unique
element of order 2. A calculation shows that $c=Y(1,1)\cdot T\in B/T=C$.

We have
\[\Ho^1\hm B=j_*^{-1}[1]\cup j_*^{-1}[c).\]
Since $\Ho^1 T=1$, we see that $j_*^{-1}[1]=\{\hs[1_B]\}$.

Write $y=Y(i,1)$, where $i^2=-1$.
Then
\[y\cdot\upgam y=Y(i,1)\cdot Y(-i,1)=X(1,1)=1.\]
Thus $y\in \Zl^1\hm B$. Clearly, $j(y)=c$.
By Corollary \ref{c:39-cor2} the set $j_*^{-1}([c])$ is in bijection to
$\Ho^1\hm  \hs_y T/(\hs_c C)^\Gamma=\Ho^1\hm  \hs_y T/C^\Gamma$.
We have $C^\Gamma=\{1,c\}$,
\[^{\gamma *}X(a,1)=y\cdot\ov{X(a,1)}\hs y^{-1}=X(\bar a,1),\]
whence $_y T\cong \C^\times_\tw$ and
$\Ho^1\hm_y T=\{[1], [X(-1,1)]\}$.
We compute the action of $c$ on $1\in \Ho^1\hm  \hs_y T$.
By Lemma \ref{lem:1c},  $[1]\cdot c=[y^2]=[X(-1,1)]\neq [1]$.
Thus $C^\Gamma$ has one orbit on $\Ho^1\hm  \hs_y T$, and therefore
$j_*^{-1}([c])=\{[y]\}.$ We conclude that $\Ho^1\hm  B=\{[1],[Y(i,1)]\}$.
\end{proof}

\noindent{\bf 27} Representative: $e_{127}+e_{136}-e_{245}+e_{379}+e_{469}+e_{568}-e_{578}$.\\
Here $\Zm_0$ consists of
$$X(a) = \diag(a^{-2},a,a,a,a^{-2},a,a,a,a^{-2}) \text{ for } a\in \C^\times$$
and
$$Y(b) = \SmallMatrix{ 0 & 0 & 0 & 0 & 0 & 0 & 0 & 0 & b^{-2}\\
                0 & 0 & 0 & b & 0 & 0 & 0 & 0 & 0\\
                0 & 0 & b & 0 & 0 & 0 & 0 & 0 & 0\\
                0 & b & 0 & 0 & 0 & 0 & 0 & 0 & 0\\
                0 & 0 & 0 & 0 & -b^{-2} & 0 & 0 & 0 & 0\\
                0 & 0 & 0 & 0 & 0 & 0 & b & 0 & 0\\
                0 & 0 & 0 & 0 & 0 & b & 0 & 0 & 0\\
                0 & 0 & 0 & 0 & 0 & 0 & 0 & b & 0\\
                b^{-2} & 0 & 0 & 0 & 0 & 0 & 0 & 0 & 0\\}\quad\text{for $b\in \C^\times$.}$$
We see that $\Zm_0$ has two components. We have
\[X(a)X(b) = X(ab), \quad X(a)Y(b) = Y(b)X(a) = Y(ab), \quad  Y(a)Y(b) = X(ab).\]
Set
\[ T=\{X(a)\mid a\in \C^\times\}, \quad C=\{1, Y(1)\}.\]
Then it is clear that $T$ and $C$ are $\Gamma$-invariant subgroups and that
$\Zm_0=T\times C$ (direct product).
We have $\Ho^1  T=\{1\}$ and
\[\Ho^1\hm  \Zm_0=\Ho^1  T\times \Ho^1 C=\Ho^1 C=\{[1], [y]\},\quad\text{where } y=Y(1).\]

\noindent{\bf 28} Representative: $e_{127}+e_{136}+e_{145}-e_{235}+e_{379}+e_{469}+e_{678}$.\\
Here $\Zm_0$ consists of
$$\diag(a^3b,a,a^4,a^7,b,a^3b,a,a,b) \text{ with } a^5b=1.$$
So $\Zm_0$ is a torus of dimension 1. Hence $\Ho^1\hm  \Zm_0=1$.

\noindent{\bf 29} Representative: $e_{127}+e_{136}-e_{235}-e_{468}+e_{479}+e_{568}$.\\
Here $\Zm_0$ is the union of two sets consisting
respectively of
$$X(a,b,c) = \diag(a,c,b,a,a,c,b,b,c)\quad \text{with } abc=1,$$
$$Y(a,b,c) = \SmallMatrix{ -a & 0 & 0 & 0 & 0 & 0 & 0 & 0 & 0\\
                0 & 0 & -c & 0 & 0 & 0 & 0 & 0 & 0\\
                0 & -b & 0 & 0 & 0 & 0 & 0 & 0 & 0\\
                0 & 0 & 0 & -a & 0 & 0 & 0 & 0 & 0\\
                0 & 0 & 0 & a & a & 0 & 0 & 0 & 0\\
                0 & 0 & 0 & 0 & 0 & 0 & -c & 0 & 0\\
                0 & 0 & 0 & 0 & 0 & -b & 0 & 0 & 0\\
                0 & 0 & 0 & 0 & 0 & 0 & 0 & 0 & b\\
                0 & 0 & 0 & 0 & 0 & 0 & 0 & c & 0\\}
     \quad\text{with $abc=-1$.}$$
 We have
\begin{align*}
  & X(a,b,c)X(a',b',c') = X(aa',bb',cc')\\
  & X(a,b,c)Y(a',b',c') = Y(aa',bb',cc')\\
  & Y(a,b,c)X(a',b',c') = Y(aa',bc',cb')\\
  & Y(a,b,c)Y(a',b',c') = X(aa',bc',cb').
\end{align*}

\begin{lemma*}
We have $\#\Ho^1\hm  \Zm_0=\{1,[y_1]\}$ where $y_1=Y(-1,1,1)$.
\end{lemma*}

\begin{proof}
Write $B=\Zm_0$ and consider the subgroup
\[ A=\{X(a,b,c)\}=\{X((bc)^{-1}, b, c)\}.\]
Set $C=\mu_2=\{\pm 1\}$ and consider the homomorphism
\[j\colon B\to C, \quad X(a,b,c)\mapsto 1,\ Y(a,b,c)\mapsto -1.\]
Then $A=\ker\, j$.
Write $c_1=-1\in \Zl^1 C=C$, and set $y_1=Y(-1,1,1)\in B$.
Then $j(y_1)=c_1$,  $\bar y_1=y_1$,  $y_1^2=1$,
hence $y_1\in \Zl^1\hm B$.
It follows that the map $j_*\colon \Ho^1\hm  B\to \Ho^1\hm  C$ is surjective.
We compute $j_*^{-1}[1]$ and $j_*^{-1}[c_1]$.
Since  $A\cong (\C^\times_{\rm st})^2$ we have $\Ho^1\hm  A=1$, and therefore,
$j_*^{-1}[1]=\{[1]\}$.
We twist the exact sequence
\[1\to A\to B\to C\to 1\]
by the cocycle $y_1$.
We compute $_{y_1}\hm A$.
A calculation shows that
\[^{\gamma*}X(a,b,c)=Y(-1,1,1)\cdot X(\bar a,\bar b,\bar c)\cdot Y(-1,1,1)^{-1}=X(\bar a,\bar c,\bar b).\]
It follows that $_{y_1} A\simeq (\C^\times)^2_{{\rm tw}+}$\hs.
By Example \ref{x:cohom}(5) we have $\Ho^1\hm  \hs_{y_1}A=1$.
Thus $j_*^{-1}[1]=\{[y_1]\}$.
We conclude that  $\Ho^1\hm  B=\{[1],[y_1]\}$, as required
\end{proof}

\noindent{\bf 30} Representative: $e_{146}-e_{179}-e_{236}-e_{245}-e_{378}+e_{569}$.\\
Here $\Zm_0$ consists of
$$\SmallMatrix{ b & 0 & 0 & 0 & 0 & 0 & 0 & 0 & a\\
                0 & e & 0 & 0 & 0 & 0 & 0 & 0 & 0\\
                0 & 0 & (bd-ac)^2 & 0 & 0 & 0 & 0 & 0 & 0\\
                0 & 0 & 0 & d & -c & 0 & 0 & 0 & 0\\
                0 & 0 & 0 & -a & b & 0 & 0 & 0 & 0\\
                0 & 0 & 0 & 0 & 0 & e & 0 & 0 & 0\\
                0 & 0 & 0 & 0 & 0 & 0 & e & 0 & 0\\
                0 & 0 & 0 & 0 & 0 & 0 & 0 & e & 0\\
                c & 0 & 0 & 0 & 0 & 0 & 0 & 0 & d\\},$$
with $e(bd-ac)=1$. Hence $\Zm_0$ is isomorphic to $\GL(2,\C)$ and therefore
$\Ho^1\hm  \Zm_0=1$.

\noindent{\bf 31} Representative: $e_{137}-e_{246}-e_{247}+e_{348}-e_{358}+e_{368}+e_{458}+e_{569}$.\\
Here $\Zm_0$ is finite and $\Zm_0 \cong \mu_3 \times H$, where $H$ is isomorphic
to the group $S_5$ (the symmetric group on five letters). The group $H$ is
generated by
$$
a=\SmallMatrix{
  -1 & 0 & 0 & 0 & 0 & 0 & 0 & 0 & 0 \\
   0 & 0 & 0 & 0 & 0 & 0 & 0 & 0 & -1 \\
   0 & 0 & -1 & 0 & 0 & 0 & 0 & 0 & 0 \\
   0 & 0 & 0 & 0 & -1 & 0 & 0 & 0 & 0 \\
   0 & 0 & 0 & -1 & 0 & 0 & 0 & 0 & 0 \\
   0 & 0 & 0 & 1 & -1 & -1 & 0 & 0 & 0 \\
   0 & 0 & -1 & 0 & -1 & -1 & 1 & 0 & 0 \\
   0 & -1 & 0 & 0 & 0 & 0 & 0 & -1 & 1 \\
   0 & -1 & 0 & 0 & 0 & 0 & 0 & 0 & 0 },\quad
b=\SmallMatrix{
  0 & 1 & 0 & 0 & 0 & 0 & 0 & 0 & 0 \\
  0 & 0 & 0 & 0 & 0 & 0 & 0 & 0 & 1 \\
  0 & 0 & -1 & -1 & 0 & 0 & 1 & 0 & 0 \\
  0 & 0 & 0 & 0 & 0 & 1 & -1 & 0 & 0 \\
  0 & 0 & 1 & 0 & 0 & 0 & 0 & 0 & 0 \\
  0 & 0 & 0 & 0 & 1 & 1 & -1 & 0 & 0 \\
  0 & 0 & 0 & 0 & 1 & 1 & 0 & 0 & 0 \\
  0 & -1 & 0 & 0 & 0 & 0 & 0 & -1 & 0 \\
  -1 & -1 & 0 & 0 & 0 & 0 & 0 & -1 & 0 }.$$
Here $a$ is the image of the transposition $(1,2)$ and $b$ is the image
of the 5-cycle $(1,2,3,4,5)$ with respect to an isomorphism $S_5\isoto H$.

We note that $\Ho^1\hm  \Zm_0 = \Ho^1  H$.
Since $H$ consists of {\em real} matrices, the nontrivial cohomology classes
coincide with the conjugacy classes of elements of order 2. There are 2
such classes, one with representative $(1,2)$, the other with representative
$(1,2)(3,4)$. We conclude that $\Ho^1\hm  \Zm_0=\{[1],[a],[c]\}$,
where $c=b^{-2}ab^2a$.

\noindent{\bf 32} Representative: $e_{127}+e_{146}-e_{236}-e_{245}+e_{379}+e_{478}+e_{568}$.\\
Here $\Zm_0$ consists of
$$X(a,\zeta) = \diag(\zeta^2a^{-1},a,a^{-2},\zeta,\zeta^2a^{-1},a,\zeta,\zeta,
\zeta^{-1} a^2),$$
where $\zeta^3=1$. Hence $\Zm_0^\circ$ is a torus so that $\Ho^1\hm  \Zm_0^\circ=1$.
The component group has order 3, and so has trivial Galois cohomology as well.
By Corollary \ref{c:prop38} we have $\Ho^1\hm  \Zm_0=1$.

\noindent{\bf 33} Representative: $e_{127}+e_{136}-e_{245}+e_{379}+e_{479}+e_{568}$.\\
Here $\Zm_0$ consists of
$$X(a,b) = \diag(a^{-1}b^{-1},a,b,b,a^{-1}b^{-1},a,b,b,b^{-2}) \text{ for }
a,b\in \C^\times$$
and
$$Y(a,b) = \SmallMatrix{ 0 & -a & 0 & 0 & 0 & 0 & 0 & 0 & 0\\
                a^{-1}b^{-1} & 0 & 0 & 0 & 0 & 0 & 0 & 0 & 0\\
                0 & 0 & 0 & -b & 0 & 0 & 0 & 0 & 0\\
                0 & 0 & -b & 0 & 0 & 0 & 0 & 0 & 0\\
                0 & 0 & 0 & 0 & 0 & a & 0 & 0 & 0\\
                0 & 0 & 0 & 0 & -a^{-1}b^{-1} & 0 & 0 & 0 & 0\\
                0 & 0 & 0 & 0 & 0 & 0 & b & 0 & 0\\
                0 & 0 & 0 & 0 & 0 & 0 & 0 & b & 0\\
                0 & 0 & 0 & 0 & 0 & 0 & 0 & 0 & -b^{-2}\\},$$
for $a,b\in \C^\times$. So there are two components.
We have
\begin{align*}
X(a,b)X(a',b') &= X(aa',bb'),\\
X(a,b)Y(a',b') &=Y(a^{-1}b^{-1}a',bb'),\\
Y(a',b')X(a,b) &= Y(aa',bb'),\\
Y(a,b)Y(a',b') &= X(-a^{-1}b^{-1}a',bb').
\end{align*}

\begin{lemma*}
$\Ho^1\hm  \Zm_0=\{[1], [y]\}$ with $y=Y(1,-1)$.
\end{lemma*}

\begin{proof}
We write $B=\Zm_0$ and set $y=Y(1,-1)$.
Then $y\in B^\Gamma$ and $y^2=1$; hence, $y\in \Zl^1 B$.
We have
\[ y\cdot X(a,b)\cdot y^{-1}=Y(1,-1)\cdot X(a,b)\cdot Y(1,-1)=Y(a,-b)\cdot Y(1,-1)=X(a^{-1}b^{-1} ,b).\]
We introduce new parameters $c,d\in\C^\times$
by $a=c^{-1},\ b=cd$.
We write
\[U(c,d)=X(a,b)=X(c^{-1},cd).\]
Then
\[ y\cdot U(c,d)\cdot y^{-1}=y\cdot X(c^{-1},cd)y^{-1}=X(d^{-1},cd)=U(d,c).\]
Let $T$ denote the identity component of $B$. Then $T$ is a 2-dimensional
torus, whence  $\Ho^1  T=1$. We consider $_y T$. We have
\[^{\gamma*} U(c,d)=U(\upgam d, \upgam c).\]
By Example \ref{x:cohom}(5) we have  $\Ho^1\hm  \hs_y T=1$.

Set $C=B/T$.
Then $C=\{1,c\}$, where $c=y T$.
We have $\Ho^1  C=\{1,[c]\}$.
Consider the surjective homomorphism $j\colon B\to C$.
Then $\ker[j_*\colon \Ho^1\hm   B\to \Ho^1  C]$
is the image of $\Ho^1  T=\{1\}$ (Proposition \ref{p:serre-prop38}).
We conclude that $\ker\hs j_*=\{[1]\}$.
Similarly, the fiber $j_*^{-1}[c]$ corresponds bijectively to
$\Ho^1\hm  \hs_y T=\{1\}$ (Corollary \ref{c:39-cor2}),
and hence contains only one element.
We conclude that  $j_*^{-1}[c]=\{[y]\}$.
Thus $\Ho^1\hm B=\{[1],[y]\}$, as required.
\end{proof}

\noindent{\bf 34} Representative: $e_{127}+e_{145}-e_{236}+e_{379}+e_{469}-e_{578}$.\\
Here $\Zm_0$ consists of
$$\diag(a^{-1}b^2,ab^{-3},b,a,b^{-2},a^{-1}b^2,b,b,b^{-2}) \text{ for } a,b\in \C^\times.
$$
So $\Zm_0$ is a 2-dimensional torus. Therefore $\Ho^1\hm  \Zm_0=1$.

\noindent{\bf 35} Representative: $e_{127}+e_{136}-e_{245}+e_{379}+e_{479}+e_{569}-e_{578}+e_{678}$.\\
Let $\zeta$ be a primitive third root of unity.
Here $\Zm_0$ is of order 18 and generated by
$$h_1=\SmallMatrix{ 0 & 1 & 0 & 0 & 0 & 0 & 0 & 0 & 0\\
                1 & 0 & 0 & 0 & 0 & 0 & 0 & 0 & 0\\
                0 & 0 & 0 & 1 & 0 & 0 & 0 & 0 & 0\\
                0 & 0 & 1 & 0 & 0 & 0 & 0 & 0 & 0\\
                0 & 0 & 0 & 0 & 0 & -1 & 0 & 0 & 0\\
                0 & 0 & 0 & 0 & -1 & 0 & 0 & 0 & 0\\
                0 & 0 & 0 & 0 & 0 & 0 & -1 & 0 & 0\\
                0 & 0 & 0 & 0 & 0 & 0 & 0 & -1 & 0\\
                0 & 0 & 0 & 0 & 0 & 0 & 0 & 0 & -1\\}, $$
$h_2=\diag(\zeta,\zeta,1,1,\zeta^2,\zeta^2,\zeta,1,\zeta^2)$,
$h_3=\diag(\zeta,\zeta,\zeta^2,\zeta^2,1,1,\zeta,\zeta^2,1)$.
We have $\Zm_0\simeq B_1\times B_2\times B_3$, where $h_1$ is the generator of the group $B_1$ of order 2,
$h_2$ is a generator of the group $B_2$ of order 3,
and $h_3$ is a generator of the group $B_3$ of order 3.
We obtain
\[ \Ho^1\hm  \Zm_0=\Ho^1\hm  B_1\times \Ho^1\hm  B_2\times \Ho^1\hm  B_3=
\Ho^1\hm  B_1=\{[1], [h_1]\}.\]

\noindent{\bf 36} Representative: $e_{136}-e_{245}+e_{379}+e_{479}+e_{568}-e_{578}+e_{678}$.\\
The identity component $T$ of $\Zm_0$ consists of
$$X(a)=\diag(a,a,a^{-2},a^{-2},a,a, a, a^{-2},a), a\in \C^\times.$$
The component group $C$ is nonabelian of order 6 (hence isomorphic to $S_3$),
generated by
$$h_1 = \SmallMatrix{ 0 & -1 & 0 & 0 & 0 & 0 & 0 & 0 & 0\\
                -1 & 0 & 0 & 0 & 0 & 0 & 0 & 0 & 0\\
                0 & 0 & 0 & 1 & 0 & 0 & 0 & 0 & 0\\
                0 & 0 & 1 & 0 & 0 & 0 & 0 & 0 & 0\\
                0 & 0 & 0 & 0 & 0 & 1 & 0 & 0 & 0\\
                0 & 0 & 0 & 0 & 1 & 0 & 0 & 0 & 0\\
                0 & 0 & 0 & 0 & 0 & 0 & 1 & 0 & 0\\
                0 & 0 & 0 & 0 & 0 & 0 & 0 & -1 & 0\\
                0 & 0 & 0 & 0 & 0 & 0 & 0 & 0 & 1\\},\quad
h_2= \SmallMatrix{ 0 & 0 & 0 & 0 & 0 & 0 & 0 & 0 & -1\\
                -1 & 0 & 0 & 0 & 0 & 0 & 0 & 0 & 0\\
                0 & 0 & 0 & -1 & 0 & 0 & 0 & 0 & 0\\
                0 & 0 & 1 & -1 & 0 & 0 & 0 & 0 & 0\\
                0 & 0 & 0 & 0 & 0 & 1 & 0 & 0 & 0\\
                0 & 0 & 0 & 0 & 0 & 0 & 1 & 0 & 0\\
                0 & 0 & 0 & 0 & 1 & 0 & 0 & 0 & 0\\
                0 & 0 & 0 & 0 & 0 & 0 & 0 & 1 & 0\\
                0 & 1 & 0 & 0 & 0 & 0 & 0 & 0 & 0\\},
$$
which satisfy the relations $h_1^2=h_2^3 =1$, $h_2h_1 = h_1h_2^2$
 modulo the identity component.

Since $\Gamma$ acts trivially on $C$, by definition we have
\[\Ho^1  C = C_2/\text{conjugation},\]
where $C_2$ denote the set of elements of order dividing 2 in $C$.
Thus $\Ho^1 C=\{[1],[c_1]\}$, where $c_1$ is the image of $h_1$ in $C$.

We have a short exact sequence
\[1\to T\labelto{ }\Zm_0\labelto{j} C\to 1,\]
where $j$ is the canonical epimorphism.
We obtain that
\[\Ho^1\hm \Zm_0=j_*^{-1}[1]\,\cup\, j_*^{-1}[c_1].\]
Now $j_*^{-1}[1]$ is the image of $\Ho^1 T=1$ (Proposition
\ref{p:serre-prop38}). Thus $j_*^{-1}[1]=\{[1]\}$.
Similarly, $j_*^{-1}[c_1]$ corresponds bijectively to $(\Ho^1{}_{h_1}\hm  T)/(\hs_{h_1} \hm C)^\Gamma$
(Corollary \ref{c:39-cor2}). We have
\[h_1^2=1\quad\text{and} \quad h_1 X(a) h_1^{-1}=X(a).\]
Thus
\[_{h_1}\hm  T=T\quad\text{and} \quad \Ho^1{}_{h_1}\hm  T=\Ho^1 T=1.\]
We see that $j_*^{-1}[c_1]=\{[h_1]\}$ and conclude that
\[ \Ho^1\hm \Zm_0=\{[1],[h_1]\}.\]

\noindent{\bf 37} Representative: $e_{127}+e_{146}-e_{236}-e_{245}+e_{379}+e_{568}$.\\
Here $\Zm_0$ consists of
$$X(a,b) = \diag(a^{-3}b,a,a^{-2},a^2b^{-1},a^{-3}b,a,a^2b^{-1},a^2b^{-1},b), a,b\in \C^\times.$$
So $\Zm_0$ is a 2-dimensional torus. Therefore $\Ho^1\hm  \Zm_0=1$.

\noindent{\bf 38} Representative: $e_{127}+e_{146}-e_{236}-e_{245}+e_{379}+e_{569}-e_{578}$.\\
Here $\Zm_0$ consists of
$$\diag(\zeta^2 a^{-2}, \zeta^2a, \zeta a^{-2}, \zeta a, a^{-2}, a, \zeta^2 a,
\zeta a, a ),\quad\text{where $a\in \C^\times$ and $\zeta^3=1$.} $$
 So the identity component is a 1-dimensional
torus, and the component group has order 3.  Hence $\Ho^1\hm  \Zm_0^\circ=1$ and
the component group has trivial Galois cohomology as well.
By Corollary \ref{c:prop38} we conclude that $\Ho^1\hm  \Zm_0=1$.

\noindent{\bf 39} Representative: $e_{136}-e_{245}+e_{379}+e_{479}+e_{569}-e_{578}+e_{678}$.\\
Here $\Zm_0$ consists of
$$X(a,\zeta) = \diag(\zeta^2a^{-1},\zeta^2 a^{-1}, a, a, \zeta, \zeta,
\zeta^2 a^{-1}, a, \zeta)\quad\text{where $\zeta^3=1$ and $a\in \C^\times$,}$$
$$Y(b,\delta)= \SmallMatrix{ 0 & -\delta^2b^{-1} & 0 & 0 & 0 & 0 & 0 & 0 & 0\\
                -\delta^2b^{-1} & 0 & 0 & 0 & 0 & 0 & 0 & 0 & 0\\
                0 & 0 & 0 & -b & 0 & 0 & 0 & 0 & 0\\
                0 & 0 & -b & 0 & 0 & 0 & 0 & 0 & 0\\
                0 & 0 & 0 & 0 & 0 & -\delta & 0 & 0 & 0\\
                0 & 0 & 0 & 0 & -\delta & 0 & 0 & 0 & 0\\
                0 & 0 & 0 & 0 & 0 & 0 & \delta^2b^{-1} & 0 & 0\\
                0 & 0 & 0 & 0 & 0 & 0 & 0 & b & 0\\
                0 & 0 & 0 & 0 & 0 & 0 & 0 & 0 & -\delta\\},\quad\text{where $\delta^3=1$ and $b\in \C^\times$.} $$
We have
\begin{align*}
  X(a,\zeta)X(b,\delta) &= X(ab,\zeta\delta)\\
  X(a,\zeta)Y(b,\delta) &= Y(ab,\zeta\delta)\\
  Y(b,\delta)X(a,\zeta) &= Y(ab,\zeta\delta)\\
  Y(a,\zeta)Y(b,\delta) &= X(ab,\zeta\delta).
\end{align*}

We define three $\Gamma$-invariant subgroups of $\Zm_0$:
\begin{align*}
B_1&=\{X(a,1)\mid a\in\C^\times\},\\
B_2&=\{X(1,\zeta)\mid\zeta^3=1\},\\
B_3&=\{1,Y(1,1)\}\cong\{\pm1\}.
\end{align*}
We see that $\Zm_0=B_1\times B_2\times B_3$, and therefore,
\[\Ho^1\hm B=\Ho^1\hm B_1\times \Ho^1\hm B_2\times \Ho^1\hm B_3=\{1\}\times
\{1\}\times \Ho^1\hm B_3.\]
We see that $\Ho^1\hm \Zm_0=\Ho^1\hm B_3=\{[1], [Y(1,1)]\}$.

\noindent{\bf 40} Representative: $e_{137}-e_{236}-e_{245}-e_{468}+e_{478}+e_{569}$.\\
Here $\Zm_0$ consists of
$$\SmallMatrix{ a_{11} & 0 & 0 & 0 & 0 & 0 & 0 & a_{18} & a_{19}\\
                0 & a_{22} & a_{23} & a_{24} & 0 & 0 & 0 & 0 & 0\\
                0 & a_{32} & a_{33} & a_{34} & 0 & 0 & 0 & 0 & 0\\
                0 & a_{42} & a_{43} & a_{44} & 0 & 0 & 0 & 0 & 0\\
                0 & 0 & 0 & 0 & a_{99} & -a_{91} & a_{98} & 0 & 0\\
                0 & 0 & 0 & 0 & -a_{19} & a_{11} & a_{18} & 0 & 0\\
                0 & 0 & 0 & 0 & a_{89} & -a_{81} & a_{88} & 0 & 0\\
                a_{81} & 0 & 0 & 0 & 0 & 0 & 0 & a_{88} & a_{89}\\
                a_{91} & 0 & 0 & 0 & 0 & 0 & 0 & a_{98} & a_{99}\\} $$
where
$$ \det \SmallMatrix{ a_{22} & a_{23} & a_{24}\\
                      a_{32} & a_{33} & a_{34}\\
                      a_{42} & a_{43} & a_{44} } = 1,\quad
\det\SmallMatrix{a_{11} & a_{18} & a_{19}\\
                      a_{81} & a_{88} & a_{89}\\
                      a_{91} & a_{98} & a_{99} } = 1.$$
Hence $\Zm_0$ is isomorphic to $\SL(3,\C)\times \SL(3,\C)$, so that
$\Ho^1\hm  \Zm_0=1$.

\noindent{\bf 41} Representative: $e_{137}+e_{145}-e_{236}+e_{479}+e_{569}-e_{578}+e_{678}$.\\
Here $\Zm_0$ consists of
$$\diag(\zeta^2a^{-1},a^{-2}, \zeta^2 a^2, a, \zeta, \zeta,
\zeta^2 a^{-1}, a, \zeta)$$
(where $\zeta^3=1$ and $a\in \C^\times$). Hence $\Zm_0^\circ$ is a torus
so that $\Ho^1\hm \Zm_0^\circ=1$. The component group is of order 3 and
so has trivial cohomology too.
By Corollary \ref{c:prop38} we conclude that $\Ho^1\hm  \Zm_0=1$.

\noindent{\bf 42} Representative: $e_{127}+e_{136}-e_{245}+e_{379}+e_{479}+e_{569}+e_{678}$.\\
Here $\Zm_0$ consists of
$$X\left(\SmallMatrix{a_{76} & a_{77}\\ a_{91} & a_{99}\\},a_{88}\right)=
\SmallMatrix{ a_{77} & 0 & 0 & 0 & 0 & 0 & 0 & 0 & -a_{76}\\
  0 & a_{88}^2a_{99}^2 &  a_{88}^2a_{91}a_{99} &  -2a_{88}^2a_{91}a_{99} &
  -a_{88}^2a_{91}^2 & 0 & 0 & 0 & 0\\
                0 & 0 & a_{88} & 0 & 0 & 0 & 0 & 0 & 0\\
      0 & a_{76}a_{88}^2a_{99} & -a_{77}a_{88}^2a_{99} + a_{88} &
      2a_{77}a_{88}^2a_{99} - a_{88} &  a_{77}a_{88}^2a_{91} & 0 & 0 & 0 & 0\\
      0 & -a_{76}^2a_{88}^2 &  a_{76}a_{77}a_{88}^2 &
      -2a_{76}a_{77}a_{88}^2 &  a_{77}^2a_{88}^2 & 0 & 0 & 0 & 0\\
                0 & 0 & 0 & 0 & 0 & a_{99} & -a_{91} & 0 & 0\\
                0 & 0 & 0 & 0 & 0 & a_{76} & a_{77} & 0 & 0\\
                0 & 0 & 0 & 0 & 0 & 0 & 0 & a_{88} & 0\\
                a_{91} & 0 & 0 & 0 & 0 & 0 & 0 & 0 & a_{99}\\} $$

with $a_{88}^3=1$, $a_{76}a_{91} + a_{77}a_{99} = a_{88}^2$.

We have $X(A,a)X(B,b) = X(AB,ab)$. Hence $\Zm_0^\circ$ is
isomorphic to $\SL(2,\C)$ so that $\Ho^1\hm \Zm_0^\circ=1$. The component group
is of order 3 and so has trivial cohomology as well.
By Corollary \ref{c:prop38} we conclude that $\Ho^1\hm  \Zm_0=1$.

\noindent{\bf 43} Representative: $e_{127}+e_{136}-e_{245}-e_{378}+e_{478}+e_{579}+e_{679}$.\\
Here $\Zm_0^\circ$ consists of
$$X(a,b,c,d)=\SmallMatrix{ 1 & 0 & 0 & 0 & 0 & 0 & 0 & 0 & 0\\
                0 & 1 & 0 & 0 & 0 & 0 & 0 & 0 & 0\\
                0 & 0 & d & 0 & 0 & c & 0 & 0 & 0\\
                0 & 0 & 0 & d & -c & 0 & 0 & 0 & 0\\
                0 & 0 & 0 &-b & a & 0 & 0 & 0 & 0\\
                0 & 0 & b & 0 & 0 & a & 0 & 0 & 0\\
                0 & 0 & 0 & 0 & 0 & 0 & 1 & 0 & 0\\
                0 & 0 & 0 & 0 & 0 & 0 & 0 & a & b\\
                0 & 0 & 0 & 0 & 0 & 0 & 0 & c & d\\}
                $$
with $ad-bc=1$. The component group $C$ is of order 6 and generated by
the image of
$$h=\SmallMatrix{ 0 & \zeta^2 & 0 & 0 & 0 & 0 & 0 & 0 & 0\\
                \zeta^2 & 0 & 0 & 0 & 0 & 0 & 0 & 0 & 0\\
                0 & 0 & 0 & \zeta & 0 & 0 & 0 & 0 & 0\\
                0 & 0 & \zeta & 0 & 0 & 0 & 0 & 0 & 0\\
                0 & 0 & 0 & 0 & 0 & -1 & 0 & 0 & 0\\
                0 & 0 & 0 & 0 & -1 & 0 & 0 & 0 & 0\\
                0 & 0 & 0 & 0 & 0 & 0 & -\zeta^2 & 0 & 0\\
                0 & 0 & 0 & 0 & 0 & 0 & 0 & 1 & 0\\
                0 & 0 & 0 & 0 & 0 & 0 & 0 & 0 & \zeta\\} $$
where $\zeta$ is a primitive third root of unity.

The element $h^3$ is real and of order 2, in particular it is a cocycle.
Furthermore,
\begin{equation}\label{eq:hX}
  hX(a,b,c,d)h^{-1} = X(a,\zeta^2 b, \zeta c, d ).\tag{$*$}
\end{equation}

\begin{lemma*}
  $\Ho^1\hm \Zm_0 = \{ [1], [h^3] \}$.
\end{lemma*}

\begin{proof}
  Write $A=\Zm_0^\circ$, $B=\Zm_0$, $C=B/A$. Then $C$ is cyclic of order 6
  and $\Ho^1 C = \{ [1], [h^3] \}$. Also we have the short exact sequence
  $$1\to A\labelto{} B\labelto{j} C\to 1.$$
  Then $\Ho^1\hm B= j_*^{-1}[1]\cup j_*^{-1}[h^3]$. Firstly, $j_*^{-1}[1]$
  is the image of $\Ho^1\hm A$, which is trivial because $A\cong \SL(2,\C)$.
  Hence $j_*^{-1}[1]=\{[1]\}$.

  For the second set we twist the exact sequence by $h^3$ and obtain
  $$1\to \hs_{h^3} A\labelto{}\hs _{h^3} B\labelto{j}\hs _{h^3} C\to 1.$$
  There is a canonical bijection from $j_*^{-1}[h^3]$ to the set of orbits
  of $_{h^3}C^\Ga$ on $\Ho^1\hm _{h^3}A$ (Corollary \ref{c:39-cor2}). But from
  \eqref{eq:hX} it follows that $_{h^3} A$ is a split form of $\SL(2,\C)$,
  and hence it is isomorphic to $\SL(2,\C)$ with the standard $\Gamma$-action.
  It follows that $\Ho^1\hm _{h^3}A=1$.
  Therefore $j_*^{-1}[h^3] = \{ [h^3] \}$.
\end{proof}

\noindent{\bf 44} Representative: $e_{137}+e_{146}-e_{236}-e_{245}+e_{478}+e_{569}$.\\
Here $\Zm_0$ consists of
$$\diag(a^{-1}b^{-1},a^{-3}b^{-2},a^4b^3,a^2b^2,a,a^{-1}b^{-1},a^{-3}b^{-2},a,b).$$
So it is a 2-dimensional torus, hence $\Ho^1\hm \Zm_0=1$.

\noindent{\bf 45} Representative: $e_{127}+e_{146}-e_{245}+e_{379}+e_{478}+e_{568}$.\\
Here $\Zm_0$ consists of
$$\SmallMatrix{ a_{55} & a_{56} & 0 & 0 & 0 & 0 & 0 & 0 & 0\\
                a_{65} & a_{66} & 0 & 0 & 0 & 0 & 0 & 0 & 0\\
                0 & 0 & a_{33} & 0 & 0 & 0 & 0 & 0 & a_{39}\\
                0 & 0 & 0 & a_{88} & 0 & 0 & 0 & 0 & 0\\
                0 & 0 & 0 & 0 & a_{55} & a_{56} & 0 & 0 & 0\\
                0 & 0 & 0 & 0 & a_{65} & a_{66} & 0 & 0 & 0\\
                0 & 0 & 0 & 0 & 0 & 0 & a_{88} & 0 & 0\\
                0 & 0 & 0 & 0 & 0 & 0 & 0 & a_{88} & 0\\
                0 & 0 & a_{93} & 0 & 0 & 0 & 0 & 0 & a_{99}\\} $$
with $a_{88}^3=1$, $a_{55}a_{66}-a_{56}a_{65}=a_{88}^2$, $a_{33}a_{99}-a_{39}a_{93}=
a_{88}^2$. So $\Zm_0^\circ$ is the direct product
of two copies of $\SL(2,\C)$, which has trivial cohomology. The component
group has order 3, which has trivial cohomology as well.
So by Corollary \ref{c:prop38} we conclude that $\Ho^1\hm  \Zm_0=1$.

\noindent{\bf 46} Representative: $e_{137}+e_{146}-e_{236}-e_{245}+e_{479}+e_{569}-e_{578}$.\\
Here $\Zm_0$ consists of
$$\diag(a^4,a^7,a^{-8},a^{-5},a^{-2},a,a^4,a^{-2},a).$$
This is a torus, therefore $\Ho^1\hm \Zm_0=1$.

\noindent{\bf 47} Representative: $e_{136}+e_{147}-e_{245}+e_{379}+e_{569}+e_{678}$.\\
The identity component $\Zm_0^\circ$ consists of
$$X(a,b) = \diag(a^{-1}b^{-1},a^{-2},a,a^2b^2,b^{-2},b,a^{-1}b^{-1},a,b),
\text{ for } a,b\in \C^\times.$$
The component group $C$ is of order 2 and generated by the image of
$$h=\SmallMatrix{ 0 & 0 & 0 & 0 & 0 & 0 & 0 & 0 & -1\\
                0 & -1 & 0 & 0 & 0 & 0 & 0 & 0 & 0\\
                0 & 0 & 1 & 0 & 0 & 0 & 0 & 0 & 0\\
                0 & 0 & 0 & 0 & -1 & 0 & 0 & 0 & 0\\
                0 & 0 & 0 & -1 & 0 & 0 & 0 & 0 & 0\\
                0 & 0 & 0 & 0 & 0 & 0 & -1 & 0 & 0\\
                0 & 0 & 0 & 0 & 0 & 1 & 0 & 0 & 0\\
                0 & 0 & 0 & 0 & 0 & 0 & 0 & 1 & 0\\
                1 & 0 & 0 & 0 & 0 & 0 & 0 & 0 & 0\\}. $$
We have
\begin{align*}
h^2=X(1,-1),\quad h\hs X(a,b)\hs h^{-1}=X(a,a^{-1}b^{-1}).
\end{align*}
We set $h_1=h\cdot X(-1,1)$. Then
$$h_1^2=1,\quad h_1\hs X(a,b)\hs h_1^{-1}=X(a,a^{-1}b^{-1}).$$
In particular, $h_1$ is a cocycle with image $c_1$ in $C$.
We have $\Ho^1\hm \Zm_0^\circ=1$, hence $\ker\,j_*=\{[1]\}$, where
$j\colon \Zm_0\to C$ is the canonical epimorphism.
Let $T(s,t)$ denote the torus in Example \ref{x:cohom}(5). Then
$^\gamma T(s,t) = T(\bar t, \bar s)$.  Define
$\tau : \hs_{h_1}\hm  \Zm_0^\circ \to T$ by $\tau(X(a,b)) = T(b,a^{-1}b^{-1})$.
Then $\tau$ is a $\Ga$-equivariant isomorphism.
It follows that  $\Ho^1\hm \hs_{h_1}\hm  \Zm_0^\circ=1$ and $j_*^{-1}[c_1]=\{[h_1]\}$.
By Corollary \ref{c:39-cor2} we conclude that $\Ho^1\hm \Zm_0=\{[1], [h_1]\}$.

\noindent{\bf 48} Representative: $e_{136}-e_{245}-e_{378}+e_{478}+e_{579}+e_{679}$.\\
The identity component $\Zm_0^\circ$ consists of
$$X(a,b,c,d,e)=\SmallMatrix{ e & 0 & 0 & 0 & 0 & 0 & 0 & 0 & 0\\
                0 & e & 0 & 0 & 0 & 0 & 0 & 0 & 0\\
                0 & 0 & d & 0 & 0 & c & 0 & 0 & 0\\
                0 & 0 & 0 & d & -c & 0 & 0 & 0 & 0\\
                0 & 0 & 0 & -b & a & 0 & 0 & 0 & 0\\
                0 & 0 & b & 0 & 0 & a & 0 & 0 & 0\\
                0 & 0 & 0 & 0 & 0 & 0 & e & 0 & 0\\
                0 & 0 & 0 & 0 & 0 & 0 & 0 & a & b\\
                0 & 0 & 0 & 0 & 0 & 0 & 0 & c & d\\} $$
with $e(ad-bc)=1$. So it is isomorphic to $\GL(2,\C)$. The component group is
of order 2 and generated by
$$h=\SmallMatrix{ 0 & 1 & 0 & 0 & 0 & 0 & 0 & 0 & 0\\
                1 & 0 & 0 & 0 & 0 & 0 & 0 & 0 & 0\\
                0 & 0 & 0 & 1 & 0 & 0 & 0 & 0 & 0\\
                0 & 0 & 1 & 0 & 0 & 0 & 0 & 0 & 0\\
                0 & 0 & 0 & 0 & 0 & -1 & 0 & 0 & 0\\
                0 & 0 & 0 & 0 & -1 & 0 & 0 & 0 & 0\\
                0 & 0 & 0 & 0 & 0 & 0 & -1 & 0 & 0\\
                0 & 0 & 0 & 0 & 0 & 0 & 0 & 1 & 0\\
                0 & 0 & 0 & 0 & 0 & 0 & 0 & 0 & 1\\}. $$
Clearly, $h^2=1$. Since  $h$ commutes with $X(a,b,c,d,e)$,
we have $\Zm_0=\Zm_0^\circ\times C$, where $C=\{1,h\}$,
and hence,
\[\Ho^1\hm \Zm_0=\Ho^1\hm \Zm_0^\circ \times \Ho^1 C=\Ho^1 C=\{[1],[h]\}.\]

\noindent{\bf 49} Representative: $e_{127}+e_{156}-e_{236}-e_{245}-e_{378}+e_{479}$.\\Here $\Zm_0$ consists of
$$
X(a,b,c,d,e)=\SmallMatrix{ e^{-2} & 0 & 0 & 0 & 0 & 0 & 0 & 0 & 0\\
                0 & e & 0 & 0 & 0 & 0 & 0 & 0 & 0\\
                0 & 0 & e^{-3}d & e^{-3}c & 0 & 0 & 0 & 0 & 0\\
                0 & 0 & e^{-3}b & e^{-3}a & 0 & 0 & 0 & 0 & 0\\
                0 & 0 & 0 & 0 & d & -c & 0 & 0 & 0\\
                0 & 0 & 0 & 0 & -b & a & 0 & 0 & 0\\
                0 & 0 & 0 & 0 & 0 & 0 & e & 0 & 0\\
                0 & 0 & 0 & 0 & 0 & 0 & 0 & a & b\\
                0 & 0 & 0 & 0 & 0 & 0 & 0 & c & d\\} $$
with $ad-bc=e^2$.

We define a surjective homomorphism
\[j\colon \Zm_0\to\GL(2,\C),\quad X(a,b,c,d,e)\mapsto \SmallMatrix{a&b\\c&d}\]
with kernel $\mu_2$. We obtain a cohomology exact sequence
\[\Zm_0(\R)\labelto{j}\GL(2,\R)\labelto\delta \Ho^1\mu_2\labelto{} \Ho^1\Zm_0\labelto{j_*} \Ho^1\hs \GL(2,\C)=1.\]
We have $j(\Zm_0(\R))=\GL(2,\R)^+$, the subgroup of $\GL(2,\R)$
consisting of matrices with positive determinant.
Since $\GL(2,\R)^+$ is a subgroup of index 2 in $\GL(2,\R)$, and $\# \Ho^1\mu_2=2$, we see that
the map $\delta$ is surjective, whence $\ker j_* =1$, and we obtain that $\Ho^1\Zm_0=1$.

\noindent{\bf 50} Representative: $e_{127}+e_{146}-e_{245}+e_{379}+e_{568}$.\\
Here $\Zm_0$ consists of
$$\SmallMatrix{ a_{55} & a_{56} & 0 & 0 & 0 & 0 & 0 & 0 & 0\\
                a_{65} & a_{66} & 0 & 0 & 0 & 0 & 0 & 0 & 0\\
                0 & 0 & a_{33} & 0 & 0 & 0 & 0 & 0 & a_{39}\\
                0 & 0 & 0 & a_{88} & 0 & 0 & 0 & 0 & 0\\
                0 & 0 & 0 & 0 & a_{55} & a_{56} & 0 & 0 & 0\\
                0 & 0 & 0 & 0 & a_{65} & a_{66} & 0 & 0 & 0\\
                0 & 0 & 0 & 0 & 0 & 0 & a_{88} & 0 & 0\\
                0 & 0 & 0 & 0 & 0 & 0 & 0 & a_{88} & 0\\
                0 & 0 & a_{93} & 0 & 0 & 0 & 0 & 0 & a_{99}\\} $$
with $a_{88}(a_{55}a_{66}-a_{56}a_{65})=1$ and $a_{33}a_{99}-a_{39}a_{93} =
a_{55}a_{66}-a_{56}a_{65}$. This is isomorphic to $B=\{ (g,h) \in \GL(2,\C)
\times \GL(2,\C) \mid \det(g)=\det(h)\}$.

Write $G=\GL(2,\C)\times \GL(2,\C)$ and define
\[d\colon G\to \C^\times,\quad(g_1,g_2)\mapsto \det(g_1)\cdot\det(g_2)^{-1}.\]
We have a short exact sequence
\[1\to B\to G\labelto{d} \C^\times\to 1.\]
By Proposition \ref{p:serre-prop38} we get an exact sequence
\[ 1\to B(\R)\to G(\R)\labelto{d}\R^\times\to \Ho^1\hm B\to \Ho^1 G=1.\]
Since the map $d\colon \GL(2,\R)\times\GL(2,\R)\to\R^\times$
is clearly surjective, we conclude that $\Ho^1\hm B=1$.

\noindent{\bf 51} Representative: $e_{137}+e_{146}-e_{236}-e_{245}+e_{479}+e_{569}+e_{678}$.\\
Here $\Zm_0$ consists of
$$\diag(\zeta a^2, \zeta^2 a^{3}, \zeta^2 a^{-3}, a^{-2}, \zeta a^{-1}, \zeta^2,
a,\zeta a^{-1},a)$$
with $\zeta^3 =1$ and $a\in \C^\times$.
So $\Zm_0^\circ$ is a torus and hence has trivial cohomology. The component
group has order 3, which has trivial cohomology as well.
So by Corollary \ref{c:prop38} we conclude that $\Ho^1\hm  \Zm_0=1$.

\noindent{\bf 52} Representative: $e_{137}+e_{146}-e_{235}+e_{479}+e_{579}+e_{678}$.\\
Here $\Zm_0$ consists of
$$\diag(a^{-1}b^{-1},a^{-3}b^{-2},a^2b^2,a,a,b,a^{-1}b^{-1},a,b)\quad\text{with $a,b\in \C^\times$.}$$
Hence it is a torus, and therefore $\Ho^1\hm \Zm_0=1$.

\noindent{\bf 53} Representative: $e_{156}-e_{236}-e_{245}-e_{378}+e_{479}$.\\
Here $\Zm_0$ consists of
$$\SmallMatrix{ d^{-1} & 0 & 0 & 0 & 0 & 0 & 0 & 0 & 0\\
                0 & a_{77} & 0 & 0 & 0 & 0 & 0 & 0 & 0\\
                0 & 0 & d^{-1}a_{77}^{-1}a_{99} & d^{-1}a_{77}^{-1}a_{98} & 0 & 0 & 0 & 0 & 0\\
                0 & 0 & d^{-1}a_{77}^{-1}a_{89} & d^{-1}a_{77}^{-1}a_{88} & 0 & 0 & 0 & 0 & 0\\
                0 & 0 & 0 & 0 & a_{99} & -a_{98} & 0 & 0 & 0\\
                0 & 0 & 0 & 0 & -a_{89} & a_{88} & 0 & 0 & 0\\
                0 & 0 & 0 & 0 & 0 & 0 & a_{77} & 0 & 0\\
                0 & 0 & 0 & 0 & 0 & 0 & 0 & a_{88} & a_{89}\\
                0 & 0 & 0 & 0 & 0 & 0 & 0 & a_{98} & a_{99}\\}, $$
where $d=a_{88}a_{99}-a_{89}a_{98}$ and $a_{77}\in\C^\times$. Hence $\Zm_0$ is
isomorphic to $\GL(1,\C)\times \GL(2,\C)$. Therefore
$\Ho^1\hm \Zm_0=1$.

\noindent{\bf 54} Representative: $e_{147}+e_{156}-e_{237}-e_{246}-e_{345}+e_{368}+e_{579}$.\\
Here $\Zm_0$ is isomorphic to $\SL(2,\C)\times \mu_3$.
Hence $\Ho^1\hm \Zm_0=1$.

\noindent{\bf 55} Representative: $e_{137}+e_{156}-e_{236}-e_{245}+e_{479}-e_{578}$.\\
Here $\Zm_0$ consists of $\diag(a,a^2b^{-1},a^{-2},a^{-1}b^2,
a^{-1}b^{-1}, b,a,b,b^{-2})$ with $a,b\in \C^\times$. Hence $\Zm_0$ is a 2-dimensional
torus. Therefore $\Ho^1\hm \Zm_0=1$.

\noindent{\bf 56} Representative: $e_{146}+e_{157}-e_{237}+e_{458}+e_{478}+e_{569}$.\\
Here $\Zm_0 \cong \SL(2,\C)\times \PSL(2,\C) \times \mu_3$. Hence
\begin{align*}
\Ho^1\hm \Zm_0 = \Ho^1\hs \SL(2,\C) \times \Ho^1\hs \PSL(2,\C)\times
\Ho^1 \mu_3
= \Ho^1\hs \PSL(2,\C)=\{[1],[w]\},
\end{align*}
where $w$ is the image of the matrix $\SmallMatrix{\phantom{-}0&1\\-1&0}$
in the direct factor of $\Zm_0$ that is isomorphic
to $\PSL(2,\C)$. Let $h,e,f$ be the standard generators of $\ssl(2,\C)$;
then the matrix above is $\exp(e)\exp(-f)\exp(e)$. Let $h_2,e_2,f_2$ be the
standard generators of the direct summand of $\z_0$ that is the Lie algebra
of the direct factor of $\Zm_0$ isomorphic to $\PSL(2,\C)$. Then $w=
\exp(e_2)\exp(-f_2)\exp(e_2)$.

\noindent{\bf 57} Representative: $e_{127}+e_{136}-e_{245}+e_{379}+e_{479}+e_{569}$.
Rank 8, Gurevich number XXIII.\\
Here the Lie algebra $\z_0$ of the centralizer $\Zm_0$ is isomorphic to
$\ssl(3,\C)$. The natural $9$-dimensional $\z_0$-module splits as a direct sum of
two modules $V_1$, $V_8$ of dimensions 1 and 8. Here
$V_1$ is the trivial representation and $V_8$ is the adjoint representation,
that is, we may identify $V_8$ with $\z_0$.
Let $N$ denote the normalizer of $\z_0$ in $\SL(9,\C)$.
Then we have a natural homomorphism
\[\phi\colon N\to \Aut(\z_0).\]
Set
\[A=\{\diag(v^{-8},v,v,v,v,v,v,v,v)\mid v\in\C^\times\}.\]
Then it follows from Schur's lemma that $\ker\phi=A$.
Thus we have a short exact sequence
\[1\to A\labelto{ } N\labelto{\phi} \Aut(\z_0)\to 1,\]
which admits a canonical splitting
\[\Aut(\z_0)\into \GL(\z_0)=\GL(V_8)\into\SL(9,\C).\]
Thus
\[ N=A\times\Aut(\z_0).\]

It is known that $\Aut(\z_0)=\Aut(\ssl_3)$
fits into a short exact sequence
\[1\to \Inn(\ssl_3)\to\Aut(\ssl_3)\to \Out(\ssl_3)\to 1,\]
where $\Inn(\ssl_3)=\PSL(3,\C)$ and $\Out(\ssl_3)$ is a group of order two, whose nontrivial element
lifts to the involution $X\mapsto -X^T$ of $\ssl_3$.

Let $N_1$ denote the preimage of $\Inn(\ssl_3)$ in $N$, and set $B_1=B\cap N_1$.
Set $A_1=A\cap B_1$.
A computer calculation shows that
\[A_1=\{\diag(\zeta,\ldots,\zeta),\quad \zeta^3=1\}.\]
It is easy to show that
\[ B_1=A_1\,\times\, \Inn(\ssl_3)=A_1\,\times\, \PSL(3,\C).\]

Let
\[e_1,\ e_2,\ e_{12}=[e_1, e_2],\ f_1,\ f_2,\ f_{12}=[f_1,f_2],\ h_1,\ h_2\]
be a ``canonical basis'' of $\z_0=\ssl_3$.
Consider the linear map
\[X\colon\
\z_0\to\z_0,\quad
e_1\leftrightarrow e_2,\ e_{12}\mapsto -e_{12},\
f_1\leftrightarrow f_2,\ f_{12}\mapsto -f_{12},\
h_1\leftrightarrow h_2\hs.
\]
Then $X\in \Aut(\z_0)$, $X$ is an {\em outer} automorphism of order 2.
Using methods outlined in Section \ref{sec:compcen} we see by a
computer calculation that  there exists an element $b$ of $\Zm_0$
acting on $\ssl_3$ by the outer automorphism $X$. We have
$$b=\SmallMatrix{ 0 & 0 & 0 & 0 & 0 & 0 & -1 & 0 & 0\\
                0 & -1 & 0 & 0 & 0 & 0 & 0 & 0 & 0\\
                0 & 0 & -1 & 0 & 0 & 0 & 0 & 0 & 0\\
                0 & 0 & -1 & 1 & 0 & 0 & 0 & 0 & 0\\
                0 & 0 & 0 & 0 & -1 & 0 & 0 & 0 & 0\\
                0 & 0 & 0 & 0 & 0 & 0 & 0 & 0 & -1\\
                -1 & 0 & 0 & 0 & 0 & 0 & 0 & 0 & 0\\
                0 & 0 & 0 & 0 & 0 & 0 & 0 & -1 & 0\\
                0 & 0 & 0 & 0 & 0 & -1 & 0 & 0 & 0\\}, $$
which satisfies $b^2=1$.

We conclude that $\Zm_0\cong\mu_3\times\Aut(\ssl_3)$,
and hence $\Ho^1\hm   \Zm_0\cong \Ho^1 \Aut(\ssl_3)$.
Now $\Ho^1 \Aut(\ssl_3)$  classifies real forms of $\ssl_3$.
There are three of them: $\g_0^{(1)}=\ssl(3,\R),\ \g_0^{(2)}=\su(2,1)$, and
$\g_0^{(3)}=\su(3)$. Thus $\#H^1=3$ in this case.
For explicit cocycles we can take {\em real} elements $b_1,\ b_2, b_3\in B$ such that
$b_i^2=1$  for $i=1,2,3$ and
\[\{y\in \z_0\mid \Ad_{b_i}(\bar y)=y\}\simeq \g_0^{(i)}\hs.\]
Clearly, we can take $b_1=1$, $b_2=b$.

Consider the outer automorphism
\[X'\colon\
\z_0\to\z_0,\quad
e_1\leftrightarrow -f_1,\ e_2\leftrightarrow -f_2,\
h_1\mapsto -h_1,\ h_2\mapsto -h_2\hs.
\]
Then
\[\{y\in\ssl_3\mid X'(\bar y)=y\}=\su(3).\]
Again by Gr\"obner basis methods we find
$b'\in B$ realizing $X'$. It is
$$b'=\SmallMatrix{ 0 & 0 & 0 & 0 & 0 & 1 & 0 & 0 & 0 \\
                0 & 0 & 0 & 0 & -1 & 0 & 0 & 0 & 0 \\
                0 & 0 & -1 & 0 & 0 & 0 & 0 & 0 & 0 \\
                0 & 0 & 0 & -1 & 0 & 0 & 0 & 0 & 0 \\
                0 & -1 & 0 & 0 & 0 & 0 & 0 & 0 & 0 \\
                1 & 0 & 0 & 0 & 0 & 0 & 0 & 0 & 0 \\
                0 & 0 & 0 & 0 & 0 & 0 & 0 & 0 & 1 \\
                0 & 0 & 0 & 0 & 0 & 0 & 0 & -1 & 0 \\
                0 & 0 & 0 & 0 & 0 & 0 & 1 & 0 & 0}. $$
Then $b'$ is real and of order 2, so we can take $b_3=b'$.

\noindent{\bf 58} Representative: $e_{137}-e_{246}-e_{345}+e_{479}+e_{569}-e_{578}$.\\
Here $\Zm_0$ consists of
$$\diag(s^{-1}t^{-4},s^{-1}t^2,st^3,t^{-2},s^{-1}t^{-1},s,t,s,t), \text{ where }
s,t\in \C^\times.$$
Hence $\Zm_0$ is a 2-dimensional torus. Therefore $\Ho^1\hm \Zm_0=1$.

\noindent{\bf 59} Representative: $e_{147}+e_{156}-e_{237}-e_{246}+e_{368}+e_{579}$.\\
The Lie algebra of $\Zm_0$ is isomorphic to $A_1+T_1$. Let $\s$ denote the
semisimple part of $\z_0$. The natural 9-dimensional $\s$-module splits
as a direct sum of irreducible modules of dimensions 2, 3, 4 with bases
$\{v_6,v_7\}$, $\{v_3,v_4,v_5\}$, $\{v_1,v_2,v_8,v_9\}$ respectively.
Hence the connected subgroup $S$ of $\Zm_0$ with Lie algebra $\s$ is isomorphic
to $\SL(2,\C)$.

The Gr\"obner basis of the ideal generated by the defining polynomials of
$\Zm_0$ is difficult to compute. However, by adding equations it
is easy to check that the elements of $\Zm_0$ that act as the
identity on $\s$ form the 1-dimensional connected torus consisting of
$$T(a)=(a,a,a^{-2},a^{-2},a^{-2},a,a,a,a), \text{ for } a\in \C^\times.$$
We denote this torus by $T$. Since $\s$ has no outer automorphisms it follows
that $\Zm_0$ is connected and equal to $ST$. We have that $\Zm_0\cong \GL(2,\C)$.
Hence $\Ho^1\hm \Zm_0=1$.

\noindent{\bf 60} Representative: $e_{137}+e_{146}-e_{236}-e_{245}+e_{568}+e_{679}$.\\
Here $\Zm_0^\circ$ consists of
$$X(\SmallMatrix{ a_{22} & a_{24} \\ a_{42} & a_{44} } ) =
\SmallMatrix{  s a_{22} & 0 & - s a_{24} & 0 & 0 & 0 & 0 & 0 & 0\\
                0 & a_{22} & 0 & a_{24} & 0 & 0 & 0 & 0 & 0\\
                - s a_{42} & 0 &  s a_{44} & 0 & 0 & 0 & 0 & 0 & 0\\
                0 & a_{42} & 0 & a_{44} & 0 & 0 & 0 & 0 & 0\\
                0 & 0 & 0 & 0 & s & 0 & 0 & 0 & 0\\
                0 & 0 & 0 & 0 & 0 & 1 & 0 & 0 & 0\\
                0 & 0 & 0 & 0 & 0 & 0 &  s^{-1} & 0 & 0\\
                0 & 0 & 0 & 0 & 0 & 0 & 0 &  s^{-1} & 0\\
                0 & 0 & 0 & 0 & 0 & 0 & 0 & 0 & s\\} $$
where $s=(a_{22}a_{44}-a_{24}a_{42})^{-1}$. Hence $\Zm_0^\circ$ is isomorphic
to $\GL(2,\C)$. The component group is abelian of order 6 and generated by
the image of
$$q=\SmallMatrix{ 0 & -\zeta & 0 & 0 & 0 & 0 & 0 & 0 & 0\\
                1 & 0 & 0 & 0 & 0 & 0 & 0 & 0 & 0\\
                0 & 0 & 0 & -\zeta & 0 & 0 & 0 & 0 & 0\\
                0 & 0 & 1 & 0 & 0 & 0 & 0 & 0 & 0\\
                0 & 0 & 0 & 0 & 0 & 0 & -1 & 0 & 0\\
                0 & 0 & 0 & 0 & 0 & -\zeta-1 & 0 & 0 & 0\\
                0 & 0 & 0 & 0 & -\zeta & 0 & 0 & 0 & 0\\
                0 & 0 & 0 & 0 & 0 & 0 & 0 & 0 & \zeta\\
                0 & 0 & 0 & 0 & 0 & 0 & 0 & 1 & 0\\}, $$
where $\zeta$ is a primitive third root of unity.

Here the computation is based on the method of Example \ref{exa:cencomps}(1).
The torus used is the center of $\Zm_0^\circ$.

We have
\begin{align*}
  & X(A)X(B) = X(AB)\\
  & q^6 = X\SmallMatrix{ -1 & 0\\ 0 & -1}\\
  & \bar q = q^5 X\SmallMatrix{ \zeta^2 & 0\\ 0 & \zeta^2}\\
  & qX\SmallMatrix{ a_{22} & a_{24} \\ a_{42} & a_{44} } q^{-1} =
  X\SmallMatrix{s a_{22} & -sa_{24} \\ -sa_{42} & sa_{44} }  \text{ where }
  s = (a_{22}a_{44}-a_{24}a_{42})^{-1}.\\
\end{align*}
The latter also implies that $q^2X(A)q^{-2} = X(A)$.

Now we set $p=q^3$, then
$$p=\SmallMatrix{ 0 & -\zeta-1 & 0 & 0 & 0 & 0 & 0 & 0 & 0\\
                -\zeta & 0 & 0 & 0 & 0 & 0 & 0 & 0 & 0\\
                0 & 0 & 0 & -\zeta-1 & 0 & 0 & 0 & 0 & 0\\
                0 & 0 & -\zeta & 0 & 0 & 0 & 0 & 0 & 0\\
                0 & 0 & 0 & 0 & 0 & 0 & -\zeta & 0 & 0\\
                0 & 0 & 0 & 0 & 0 & 1 & 0 & 0 & 0\\
                0 & 0 & 0 & 0 & \zeta+1 & 0 & 0 & 0 & 0\\
                0 & 0 & 0 & 0 & 0 & 0 & 0 & 0 & -\zeta-1\\
                0 & 0 & 0 & 0 & 0 & 0 & 0 & \zeta & 0\\}. $$

We fix a primitive  12-th root of unity $\delta$  and set $\zeta=\delta^4$,
$\eta = \delta^{11} = \bar \delta$. Also we set
\[ p'=X\SmallMatrix{\eta&0\\0&\eta}\cdot p,\qquad p''=X\SmallMatrix{0&1\\1&0}\cdot p'.\]
Calculations show that
\[p'\cdot\hs\ov{p'}= p''\cdot\hs\ov{p''}=1,\quad ( p')^2=(p'')^2=X\SmallMatrix{-1&0\\0&-1},\quad
( p'')^{-1}\cdot\hs\ov{p''}=X\SmallMatrix{-1&0\\0&-1}.\]

\begin{lemma*}
$\Ho^1\hm \Zm_0=\{[1], [p'], [p'']\}$.
\end{lemma*}

\begin{proof}
 We write $B=\Zm_0$ and
 \[G=\left\{X\SmallMatrix{x&y\\z&t}\right\}.\]
Then $G\cong\GL(2,\C)$. We set $C=B/G$.
Then $C\simeq \mu_6$ over $\R$,
that is, $\#C=6$, $C$ is generated by $q G$,
and $\ov q  G=q^{-1}\hs G$.
Thus $\Zl^1 C=C$, $B^1(\Ga,C)=C^2$, and
$$\Ho^1 C=C/C^2\cong \Z/2\Z$$
with generator $c=pG=p'G$.

We have a short exact sequence
\[1\to G\to B\to C\to 1.\]
We write
\[S=\left\{X \SmallMatrix{x&y\\z&t}\mid xt-yz=1\right\}.\]
Then $S\cong \SL(2,\C)$.
We set $B_1=B/S$, $T=G/S\cong \C^\times$.
We have a short exact sequence
\[1\to T\labelto{ } B_1\labelto{\pi} C\to 1.\]

We compute $\Ho^1\hm  B_1$.
We have
\[\Ho^1\hm B_1=\pi_*^{-1}[1]\cup\pi_*^{-1}[c].\]
where $c=pG\in \Zl^1 C=C$.
Since $\Ho^1 T=1$, we have $\pi_*^{-1}[1]=\{[1]\}$.

We compute $\pi_*^{-1}[c]$ by twisting the short exact sequence by
$c=p'G \in \Zl^1\hm C=C$.
We can lift the cocycle $c$
to the cocycle $p'S$ in $B_1$. We twist the exact sequence by $p'S$.

A class $tS$ in $T$ has a representative $t=X\SmallMatrix{ a & 0\\0 & a}$.
We have
$$p' t(p')^{-1} = X\SmallMatrix{ \eta & 0\\0 & \eta}\hs p\hs
X\SmallMatrix{ a & 0\\0 & a}\hs p^{-1} X\SmallMatrix{ \eta^{-1} & 0\\0 & \eta^{-1}}
= X\SmallMatrix{ a^{-1} & 0\\0 & a^{-1}}.$$
It follows that the $\Ga$-action in $_{c}T$ is
\[tS\mapsto\bar t^{-1}S.\]
Thus $(\hs_{c}T)(\R)\simeq \C^\times_\tw$ and by Example \ref{x:cohom}(3)
we have $\Ho^1\hm \C^\times_\tw=\{[1],[-1]\}$; hence
$\#\Ho^1\hm \hs_{c} T=2$.

We consider the right action of the group $(\hs_c C)(\R)=C(\R)=\{1,c\}$ on
$\Ho^1\hm \hs_c T$; see Construction \ref{con:rightact}.
Since $c\in(\hs_c C)(\R)$, $c^2=1$,  $c=p'S$,
and $p'S\in \Zl^1\hm  B_1$, by Lemma \ref{lem:1c} we have
\[ [1]\cdot c=[(p')^{-2}S]=[1]\in \Ho^1\hm \hs_c T\]
(because $(p')^{-2}\in S$). We conclude that $C(\R)$ acts trivially on
$\Ho^1\hm \hs_c T$.
We consider
$$X\SmallMatrix{0&1\\1&0}\cdot S\in \Zl^1\hm \hs_c T,$$
which is not cohomologous to 1, and
\[p'' = X\SmallMatrix{0&1\\1&0}\cdot p';\]
then $p''\cdot \ov{p''}=1$, that is, $p''\in \Zl^1\hm B$.
Since $C(\R)$ acts trivially on  $\Ho^1\hm \hs_c T$,
by Corollary \ref{c:39-cor2} the cocycles $p'S$ and $p''S$ are not
cohomologous in $\Zl^1 B_1$ and we have $\#\pi_*^{-1}[c]=2$ with cocycles
$[p'S]$ and $[p''S]$.
Thus $\#\Ho^1\hm B_1=3$ with cocycles $[1],\ [p'S],\ [p''S]$.

We proceed to compute $\Ho^1\hm B$.
We have a short exact sequence
\[1\to S\labelto{ }B\labelto{\phi} B_1\to 1.\]
We see that
\[\Ho^1\hm B=\phi_*^{-1}[1]\,\cup\,\phi_*^{-1}[p'S]\,\cup\,
 \phi_*^{-1}[p''S].\]
 We compute these three sets.

 The set $\phi_*^{-1}[1]$ comes from $\Ho^1\hm S$. Since
 $S\cong\SL(2,\C)$ ($\Gamma$-equivariantly), we have
 $\Ho^1\hm S=1$, and hence, $\phi_*^{-1}[1]=[1]$.

The set $\phi_*^{-1}[p'S]$ comes from $\Ho^1\hm \hs_{p'}S$.
Consider the subgroup
\[\left\{s_\lambda=X\SmallMatrix{\lambda&0\\0&\lambda^{-1}}\mid\lambda\in\R^\times\right\}.\]
One can easily check that
\[p'\cdot\ov{s_\lambda}\cdot (p')^{-1}=s_\lambda.\]
It follows that the group $(\hs_{p'}S)(\R)$ contains a closed subgroup isomorphic to $\R^\times$.
We see that $(\hs_{p'}S)(\R)$ is not compact, and therefore, isomorphic to $\SL(2,\R)$.
We conclude that  $\Ho^1\hm \hs_{p'}S=1$, and hence
$\#\phi_*^{-1}[p'S]=1$ with cocycle $p'$.

The set $\phi_*^{-1}[p''S]$
comes from $\Ho^1\hm \hs_{p''}S$ and is in a canonical bijection
with the set of orbits $\Ho^1\hm \hs_{p''}S/(\hs_{p''}B_1)(\R)$.
We compute $(\hs_{p''}S)(\R)$.
For
$$s=X\SmallMatrix{x&y\\z&t}\quad\text{with \,}xt-yz=1,$$
we have
\[p''\cdot\bar s\cdot (p'')^{-1}=
X\SmallMatrix{\bar t&-\bar z\\-\bar y&\bar x}.\]
Hence, the condition
\[p''\cdot\bar s\cdot (p'')^{-1}=s\]
gives $t=\bar x$ and $z=-\bar y$.
It follows that
\[(\hs_{p''}S)(\R)=\left\{X\SmallMatrix{x&y\\-\bar y&\bar x}\mid x\bar x+y\bar y=1\right\}.\]
We see that  $(\hs_{p''}S)(\R)\simeq \SU(2)$, and hence
$\#\Ho^1\hm \hs_{p''}S=2$ with cocycles 1 and $X\SmallMatrix{-1&0\\0&-1}$;
see \cite[Section III.4.5, Examples]{Serre1997}.

We compute the orbits of the action of $(\hs_{p''}B_1)(\R)$
on $\Ho^1\hm \hs_{p''}S$. Since
\[\ov{p''}\cdot (p'')^{-1}=X\SmallMatrix{-1&0\\0&-1}\qquad\text{and}\qquad (p'')^{2}=X\SmallMatrix{-1&0\\0&-1},\]
we have $p''S\in B_1(\R)$ and $(p''S)^2=1$.
We compute $[1]\cdot p''S\in \Ho^1\hm \hs_{p''}S$. By Lemma \ref{lem:1c}
it is the class of the cocycle
\[(p'')^{-2}=X\SmallMatrix{-1&0\\0&-1}.\]
We see that the group $(\hs_{p''}B_1)(\R)$ acts on $\Ho^1\hm \hs_{p''}S$
transitively.
It follows that  $\#\phi_*^{-1}[p''S]=1$
 and $\phi_*^{-1}[p''S]=\{[p'']\}$.
Thus $\Ho^1\hm B=\{[1],[p'],[p'']\}$, as required.
\end{proof}

\noindent{\bf 61} Representative: $e_{137}+e_{146}-e_{236}-e_{245}+e_{479}+e_{569}$.
Rank 8, Gurevich number XXII.\\
Here $\z_0 = \s+\t_1$, where $\s$ is semisimple of type $A_1$ and $\t_1$ denotes
a 1-dimensional center. We consider the natural 9-dimensional $\s$-module.
It splits into a direct sum of three irreducible modules of dimensions 5, 3, 1
(of highest weights 4,2,0).
The centralizer in $\Zm_0$ of $\s$ is a 1-dimensional torus $T_1$ consisting of
$\diag(a^{-2},a^{-2},a,a,a,a,a,a,a^{-2})$ for $a\in \C^\times$, which is
exactly the connected torus with Lie algebra $\t_1$. Let $S$ be the connected
subgroup of $\Zm_0$ with Lie algebra $\s$, then $S\cap T_1 = \{1\}$
(this follows because one of the summands of the $\s$-module is trivial).
Since $\s$ has no outer
automorphisms this implies that $\Zm_0$ is isomorphic
to $\mathrm{PSL}(2,\C)\times T_1$. (Since the weight lattice in this case is
equal to the root lattice.)

Since $\PSL(2,\C)=\Aut(\ssl(2,\C))$, the set $\Ho^1  \PSL(2,\C)$ is in a
bijection with the set of isomorphism classes of real forms of $\ssl(2,\C)$.
There two nonisomorphic real forms, $\ssl(2,\R)$ and $\su(2)$.
Thus $\# \Ho^1\hm \PSL(2,\C)=2$.

For explicit cocycles we can take 1 and the image $b$ in $\PSL(2,\C)$ of the
matrix
\[ s=\SmallMatrix{ 0 & 1\\ -1 & 0 }\in \SL(2,\R).\]
Note that this image is real and of order 2, hence a cocycle.

\noindent{\bf 62} Representative: $e_{146}-e_{235}+e_{479}+e_{579}+e_{678}$.\\
Here $\Zm_0^\circ$ consists of
$$X(u,v,x,y;b)=\SmallMatrix{ a^{-1}b^{-1} & 0 & 0 & 0 & 0 & 0 & 0 & 0 & 0\\
                0 & u & v & 0 & 0 & 0 & 0 & 0 & 0\\
                0 & x & y & 0 & 0 & 0 & 0 & 0 & 0\\
                0 & 0 & 0 & a & 0 & 0 & 0 & 0 & 0\\
                0 & 0 & 0 & 0 & a & 0 & 0 & 0 & 0\\
                0 & 0 & 0 & 0 & 0 & b & 0 & 0 & 0\\
                0 & 0 & 0 & 0 & 0 & 0 & a^{-1}b^{-1} & 0 & 0\\
                0 & 0 & 0 & 0 & 0 & 0 & 0 & a & 0\\
                0 & 0 & 0 & 0 & 0 & 0 & 0 & 0 & b\\}$$
where $a=(uy-vx)^{-1}$, $b\in \C^\times$;
so it is isomorphic to $\GL(2,\C)\times \GL(1,\C)$.
The component group $C$ is of order 2 and generated by the image $c$ of
$$h=\SmallMatrix{ 0 & 0 & 0 & 0 & 0 & 0 & 0 & 0 & -1\\
                0 & 1 & 0 & 0 & 0 & 0 & 0 & 0 & 0\\
                0 & 0 & 1 & 0 & 0 & 0 & 0 & 0 & 0\\
                0 & 0 & 0 & -1 & 0 & 0 & 0 & 0 & 0\\
                0 & 0 & 0 & -1 & 1 & 0 & 0 & 0 & 0\\
                0 & 0 & 0 & 0 & 0 & 0 & 1 & 0 & 0\\
                0 & 0 & 0 & 0 & 0 & 1 & 0 & 0 & 0\\
                0 & 0 & 0 & 0 & 0 & 0 & 0 & -1 & 0\\
                -1 & 0 & 0 & 0 & 0 & 0 & 0 & 0 & 0\\}.$$

Here we have $h^2=1$ and
\begin{equation}\label{e:h1}
  h X(u,v,x,y;b)h^{-1}=X(u,v,x,y;b^{-1}(uy-vx)).\tag{$*$}
\end{equation}

\begin{lemma*}
$\Ho^1\hm \Zm_0=\{[1],[h]\}$.
\end{lemma*}

\begin{proof}
Write $A=\Zm_0^\circ$, $B=\Zm_0$, $C=B/A = \{[1],[c]\}$.
We have a short exact sequence
\[1\to A\labelto{ } B\labelto{j} C\to 1,\]
where $j$ is the canonical epimorphism.
We have $\Ho^1\hm C=\{[1],[c]\}$ and  hence $\Ho^1\hm B=
\ker\,j_*\,\cup\,j_*^{-1}[c]$.

Since $A$ is isomorphic to $\GL(2,\C)\times \GL(1,\C)$, we have
$\Ho^1\hm A=1$, so that $\ker\, j_*=\{[1]\}$.

The fiber $j_*^{-1}[c]$ is in a canonical bijection with the image of
$\Ho^1\hm \hs_{h} A$ (Corollary \ref{c:39-cor2}). In order to compute
this set we first construct a $\Gamma$-equivariant surjective homomorphism
\[\phi\colon A\to (\C^\times)^2,\quad X(u,v,x,y;b)\mapsto ((uy-vx)^{-1},b).\]
We write $T=(\C^\times)^2$. Then we have a short exact sequence
\[1\to H \labelto{ } A\labelto{\phi} T\to 1,\]
where $H\cong \SL(2,\C)$. We twist this sequence to get the short exact
sequence
\[1\to {}_{h}\hm  H\to {}_{h}\hm  A\to {}_{h}\hm  T\to 1.\]
It follows from \eqref{e:h1} that $_{h}\hm  H=H$, and therefore
$\Ho^1\hm \hs_{h}\hm  H=1$.

The conjugation on $\hs_{h}\hm  T$ is given by $(a,b)\mapsto (\ov a, \ov a^{-1}
\ov b^{-1})$, hence it acts on $\X_*(T)=\Z^2$ by the matrix $\SmallMatrix{1&0\\-1&-1}$.
This matrix permutes the two vectors $\SmallMatrix{1\\0}$ and  $\SmallMatrix{1\\-1}$
constituting a basis of $\Z^2$.
It follows that $_h T$ is a quasi-trivial torus,
and by Proposition \ref{p:quasi} we have $\Ho^1{}_{h}\hm  T=1$.
We see that $\Ho^1\hm  \hs_{h}\hm  A=1$ and thus $j_*^{-1}[c]$ has only one
element, namely, $[h]$. This concludes the proof.
\end{proof}

\noindent{\bf 63} Representative: $e_{127}+e_{146}-e_{236}-e_{345}+e_{579}+e_{678}$.\\
Here $\Zm_0$ consists of
$$X(u,v,w,x,a)=\SmallMatrix{ a^3u & a^3v & 0 & 0 & 0 & 0 & 0 & 0 & 0\\
                a^3w & a^3x & 0 & 0 & 0 & 0 & 0 & 0 & 0\\
                0 & 0 & u & v & 0 & 0 & 0 & 0 & 0\\
                0 & 0 & w & x & 0 & 0 & 0 & 0 & 0\\
                0 & 0 & 0 & 0 & a^4 & 0 & 0 & 0 & 0\\
                0 & 0 & 0 & 0 & 0 & a & 0 & 0 & 0\\
                0 & 0 & 0 & 0 & 0 & 0 & a^{-2} & 0 & 0\\
                0 & 0 & 0 & 0 & 0 & 0 & 0 & a & 0\\
                0 & 0 & 0 & 0 & 0 & 0 & 0 & 0 & a^{-2}\\}, $$
with $ux-vw=a^{-4}$.

Consider the surjective homomorphism
\[j\colon  \Zm_0 \to \GL(2,\C),\quad \  X(u,v,w,x,a)\,\longmapsto\,\SmallMatrix{ u & v \\ w & x }\]
with kernel $\mu_4$\hs.
The short exact sequence
$$1\to \mu_4\labelto{i} \Zm_0\labelto{j} \GL(2,\C)\to 1$$
gives rise to a cohomology exact sequence
\[\Zm_0(\R)\labelto{j} \GL(2,\R)\labelto\delta \Ho^1\hm \mu_4 \labelto{i_*}
         \Ho^1\hm \Zm_0 \labelto{j_*} \Ho^1 \GL(2,\C)=1.\]
The image of the homomorphism $j\colon \Zm_0(\R)\to \GL(2,\R)$ is the group
$\GL(2,\R)^+$ consisting of matrices with {\em positive} determinant.
Since $\GL(2,\R)^+$ is a subgroup of index 2 in $\GL(2,\R)$, we see that the image of $\delta$ is a subgroup of order 2 in $\Ho^1\hm \mu_4$.
However, $\Ho^1\hm \mu_4$ is a group of order 2, and therefore the connecting
homomorphism $\delta$ is surjective. It follows that
\[\ker\big[\hs j_*\colon \Ho^1\hm \Zm_0 \to \Ho^1\hm \GL(2,\C))\hs\big]=\{[1]\}.\]
Taking into account that $\Ho^1 \GL(2,\C)=\{1\}$, we conclude that
$\Ho^1\hm \Zm_0=1$.

\noindent{\bf 64} Representative: $e_{137}-e_{246}-e_{345}+e_{568}+e_{579}$.\\
Here the identity component $\Zm_0^\circ$ consists of
\[D(s,t,u)=\diag(s^{-1}u^{-1},s^{-1}t^{-2}u,u,stu^{-1},s^{-1}t^{-1},t,s,s,t),\]
where $s,t,u\in \C^\times$.
The component group $C$ is of order 2 and is generated by the image of
$$Q=\SmallMatrix{ 0 & -1 & 0 & 0 & 0 & 0 & 0 & 0 & 0\\
                1 & 0 & 0 & 0 & 0 & 0 & 0 & 0 & 0\\
                0 & 0 & 0 & 1 & 0 & 0 & 0 & 0 & 0\\
                0 & 0 & -1 & 0 & 0 & 0 & 0 & 0 & 0\\
                0 & 0 & 0 & 0 & 1 & 0 & 0 & 0 & 0\\
                0 & 0 & 0 & 0 & 0 & 0 & 1 & 0 & 0\\
                0 & 0 & 0 & 0 & 0 & 1 & 0 & 0 & 0\\
                0 & 0 & 0 & 0 & 0 & 0 & 0 & 0 & 1\\
                0 & 0 & 0 & 0 & 0 & 0 & 0 & 1 & 0\\}. $$
We have $Q^2 = D(1,1,-1)$, $Q\cdot D(s,t,u)\cdot Q^{-1} = D(t,s,stu^{-1})$.

\begin{lemma*}
$\Ho^1\hm \Zm_0=\{[1], [P] \}$ where $P=D(1,1,i)\cdot Q$.
\end{lemma*}

\begin{proof}
Write $B=\Zm_0$ and $T=\Zm_0^\circ$.
Clearly $\Ho^1 C=\{[1],[c]\}$, where $c$ is the image of $Q$.
We have a short exact sequence
\[ 1\to T\labelto{ } B\labelto{j} C\to 1.\]
We have $\Ho^1\hm B=j_*^{-1}[1]\cup j_*^{-1}[c]$.
Since $\Ho^1\hm T=1$, we have $j_*^{-1}[1]=\{[1]\}$.
We compute $j_*^{-1}[c]$.

Set $P=D(1,1,i)\cdot Q$, where $i^2=-1$. Then
\[ P\cdot \ov P=D(1,1,i)\cdot Q\cdot D(1,1,-i)\cdot Q^{-1}\cdot Q^2=D(1,1,i)\cdot D(1,1,i)\cdot D(1,1,-1)=1.\]
Thus $P\in B$ is a cocycle. We twist the above exact sequence by $P$. The complex
conjugation on $\hs_PT$ is given by
$$D(s,t,u)\mapsto D(\ov t, \ov s, \ov s\ov t \ov u^{-1}).$$
We see that $\gamma$ acts on the cocharacter group $\X_*(\hs_P T)=\Z^3$ by the matrix
\[ \SmallMatrix{0&1&0\\1&0&0\\1&1&-1}.\]
The lattice $\X_*(T)$ has a basis
\[ v_1=\SmallMatrix{1\\0\\0},\ \ v_2=\SmallMatrix{0\\1\\1},\ \ v_3=\SmallMatrix{0\\0\\1}\]
such that $\upgam v_1=v_2$ and $\upgam v_3=-v_3$.
We have $v_3(-1)=D(1,1,-1)$.
It follows that $\Ho^1\hs\X_*(\hs_P T)=\{0,[v_3]\}$, and
by Proposition \ref{p:e*} we have $\Ho^1 \hs_PT = \{ [1], [D(1,1,-1)] \}$.

Clearly, $c$ is the image of $P$ in $C$.
We have
\begin{align*}
P^2=D(1,1,i)\cdot Q\cdot &D(1,1,i)\cdot Q^{-1}\cdot Q^2\\
=&D(1,1,i)\cdot D(1,1,-i)\cdot D(1,1,-1)=D(1,1,-1),
\end{align*}
and hence, $P^{-2}=D(1,1,-1)$.
By Lemma \ref{lem:1c}
\[[1]\cdot c=[P^{-2}]=[D(1,1,-1)]\in \Ho^1\hm \hs_P T.\]
Thus $C^\Ga$ acts on $\Ho^1\hm \hs_P T$ nontrivially, and
$\#j_*^{-1}[c]=1$.
We conclude that $\Ho^1\hm B=\{\hs[1],\hs[P]\hs\}$, as required.
\end{proof}

\noindent{\bf 65} Representative: $e_{137}-e_{246}-e_{247}-e_{345}+e_{569}+e_{678}$.\\
Here $\Zm_0^\circ$ is a 2-dimensional torus consisting of
$$D(s,t) = \diag(st^{-1}, s^{-1},s^{-1}t,s,t^{-1},1,1,1,t) \text{ for }
s,t\in \C^\times.$$
In the following $\zeta$ is a primitive third root of unity.
The component group is of order 18 and generated by $f_0,f_2,f_3$, where
$f_0 = \diag( -1, -1, -\zeta, -\zeta, \zeta, \zeta^2, \zeta^2, \zeta^2, 1 )$ and
$$f_2= \SmallMatrix{ 0 & -\zeta & 0 & 0 & 0 & 0 & 0 & 0 & 0\\
                \zeta^2 & 0 & 0 & 0 & 0 & 0 & 0 & 0 & 0\\
                0 & 0 & 0 & -\zeta & 0 & 0 & 0 & 0 & 0\\
                0 & 0 & 1 & 0 & 0 & 0 & 0 & 0 & 0\\
                0 & 0 & 0 & 0 & \zeta^2 & 0 & 0 & 0 & 0\\
                0 & 0 & 0 & 0 & 0 & \zeta & -\zeta & 0 & 0\\
                0 & 0 & 0 & 0 & 0 & 0 & -\zeta & 0 & 0\\
                0 & 0 & 0 & 0 & 0 & 0 & 0 & -\zeta & 0\\
                0 & 0 & 0 & 0 & 0 & 0 & 0 & 0 & 1\\},
\qquad
f_3 = \SmallMatrix{ 0 & 0 & 0 & 0 & 0 & 0 & 0 & 0 & 1\\
                1 & 0 & 0 & 0 & 0 & 0 & 0 & 0 & 0\\
                0 & 0 & 0 & 0 & 1 & 0 & 0 & 0 & 0\\
                0 & 0 & 1 & 0 & 0 & 0 & 0 & 0 & 0\\
                0 & 0 & 0 & 1 & 0 & 0 & 0 & 0 & 0\\
                0 & 0 & 0 & 0 & 0 & 0 & -1 & 0 & 0\\
                0 & 0 & 0 & 0 & 0 & 1 & -1 & 0 & 0\\
                0 & 0 & 0 & 0 & 0 & 0 & 0 & 1 & 0\\
                0 & 1 & 0 & 0 & 0 & 0 & 0 & 0 & 0\\},
$$
They satisfy the following relations:
$$
  f_3f_2 = f_2f_3^2,\quad
  f_2^2 = f_0\hs,\quad
  f_3^3 = 1,\quad
  f_0 \text{ is central},\quad
  f_0^3=D(-1,1).
$$
Using these relations we can write every element of the component group
uniquely as $f_2^if_3^jf_0^k$ with $i=0,1$, $j,k=0,1,2$. Furthermore, we
have
\begin{equation*}
f_2D(s,t)f_2^{-1} = D(s^{-1}t,t),\quad f_3D(s,t)f_3^{-1} = D(t^{-1},st^{-1}).
\end{equation*}
We have
\[\bar f_2=f_2\cdot f_0^2\cdot D(\zeta^2,1).\]
Set $f_2'=f_2\cdot D(i\zeta,1)$.

\begin{lemma*}
  $\Ho^1\hm  \Zm_0=\{[1],[f_2']\}$.
\end{lemma*}

\begin{proof}
  Consider the central subgroup $\langle f_0\rangle$ of order 3 and set
  $G=\Zm_0/\langle f_0\rangle$.
  By Lemma \ref{l:H1-bijective} the canonical epimorphism $\Zm_0\to G$ induces a bijection
  $\Ho^1\hm  \Zm_0\simeq \Ho^1\hs G$.

We compute $\Ho^1\hs G$. Write $T=\Zm_0^\circ$, which is a 2-dimensional torus.
Set $C=G/T$, which  is a finite group of order 6.
Let $h_2$ and $h_3$ denote the images in $G$ of $f_2$ and $f_3$, respectively.
Then the images $c_2$ of $h_2$ and $c_3$ of $h_3$ in $C$ generate $C$.
We have
\[ h_3^3=1,\quad h_2^2=1, \quad h_2h_3h_2^{-1}=h_3^2\hs,\]
whence
\[ c_3^3=1,\quad c_2^2=1, \quad c_2c_3c_2^{-1}=c_3^2\hs.\]
It follows that $C\simeq S_3$. From the formula for $\bar f_2$ it follows that
$\bar c_2=c_2$. Furthermore, obviously $\bar f_3=f_3$ so that $\bar c_3=c_3$.
Hence $\Gamma$ acts trivially on $C$.
It follows that $\#\Ho^1\hs  C=2$ and $\Ho^1\hs C=\{[1],[c_2]\}$.

We have a short exact sequence
\begin{equation}\label{e:TGC}
 1\to T\labelto{ } G\labelto{j} C\to 1.
\end{equation}
Since $\Ho^1\hs T=1$, we see that $j_*^{-1}[1]=\{[1]\}$.

We compute $j_*^{-1}[c_2]$. For this end we first show that $f_2'$ is a 1-cocycle.
Indeed,
\begin{align*}
f'_2\cdot\bar f'_2&=f_2\hs D(i\zeta, 1)\cdot f_2 f_0^2\hs D(\zeta^2,1)\hs D(-i\zeta^2,1)=f_2^2\hs f_0^2 \cdot D(-i\zeta^2, 1)D(\zeta^2,1)\hs D(-i\zeta^2,1)\\
&=f_0\hs f_0^2\cdot D(-\zeta^4,1)\hs D(\zeta^2,1)=f_0^3\cdot D(-1,1)=1.
\end{align*}
Let $h_2'$ denote the image of $f_2'$ in $G$.
The image of $h_2'$ in $C$ is $c_2$. Now we twist the exact sequence \eqref{e:TGC} by $h_2'$.
We have
$$f_2'\hs \ov{D(s,t)}\hs f_2^{\prime\hs-1} = D(\ov s^{\hs-1} \ov t, \ov t).$$
We see that $\gamma$ acts on $\X_*(\hs_{h_2'} T)=\Z^2$ by the matrix
$\SmallMatrix{-1&0\\1&1}$ and preserves the basis
$\SmallMatrix{1\\0},\ \SmallMatrix{-1\\1}$ of $\Z^2$.
It follows that $_{h_2'} T$ is a quasi-trivial torus,
and by Proposition \ref{p:quasi} we have  $\Ho^1 \hs_{h_2'} T=1$.
Now from Corollary \ref{c:39-cor2} it follows that $\#j_*^{-1}[c_2]=1$.
We conclude that $\#\Ho^1  G=2$ and hence $\#\Ho^1\hm \Zm_0=2$.

Since the image of $f_2'$ in $C$ is $c_2$\hs, it follows that
$\Ho^1\hm \Zm_0=\{\hs 1,\hs[f_2']\hs\}$, as required.
\end{proof}

\noindent{\bf 66} Representative: $e_{137}-e_{246}+e_{479}+e_{569}-e_{578}$.\\
Here $\Zm_0$ consists of
$$\SmallMatrix{ a_{11} & 0 & a_{13} & 0 & 0 & 0 & 0 & 0 & 0 \\
                0 & s^{-1}t^2 & 0 & 0 & 0 & 0 & 0 & 0 & 0 \\
                a_{31} & 0 & a_{33} & 0 & 0 & 0 & 0 & 0 & 0 \\
                0 & 0 & 0 & t^{-2} & 0 & 0 & 0 & 0 & 0 \\
                0 & 0 & 0 & 0 & s^{-1}t^{-1} & 0 & 0 & 0 & 0 \\
                0 & 0 & 0 & 0 & 0 & s & 0 & 0 & 0 \\
                0 & 0 & 0 & 0 & 0 & 0 & t & 0 & 0 \\
                0 & 0 & 0 & 0 & 0 & 0 & 0 & s & 0 \\
                0 & 0 & 0 & 0 & 0 & 0 & 0 & 0 & t } $$
with $s,t\in \C^\times$ and $a_{11}a_{33}-a_{13}a_{31} = t^{-1}$.
Since  $\Zm_0$ is isomorphic to $\GL(2,\C)\times \C^\times$,
we have $\Ho^1\hm \Zm_0 = 1$.

\noindent{\bf 67} Representative: $e_{137}+e_{146}-e_{236}-e_{245}+e_{579}$.
Rank 8, Gurevich number XXI.\\
Here $\Zm_0$ is of type $A_1$ plus a 2-dimensional torus.
The natural module splits as a direct sum of two 2-dimensional irreducible
representations and five 1-dimensional ones. The identity component of
the centralizer consists of
$$X(\SmallMatrix{ u & v\\ w & x },a,b)=
\SmallMatrix{ ab^{-1}u & 0 & -ab^{-1}v & 0 & 0 & 0 & 0 & 0 & 0\\
                0 & u & 0 & v & 0 & 0 & 0 & 0 & 0\\
                -ab^{-1}w & 0 & ab^{-1}x & 0 & 0 & 0 & 0 & 0 & 0\\
                0 & w & 0 & x & 0 & 0 & 0 & 0 & 0\\
                0 & 0 & 0 & 0 & a^2b^{-1} & 0 & 0 & 0 & 0\\
                0 & 0 & 0 & 0 & 0 & a & 0 & 0 & 0\\
                0 & 0 & 0 & 0 & 0 & 0 & b & 0 & 0\\
                0 & 0 & 0 & 0 & 0 & 0 & 0 & a & 0\\
                0 & 0 & 0 & 0 & 0 & 0 & 0 & 0 & a^{-2}\\}$$
with $a,b\in \C^\times$, and $\det(A)=a^{-2}b$, where
\[A=\SmallMatrix{u&v\\w&x}.\]
 The component group has order two and is generated by
$$Q=\SmallMatrix{ 0 & -1 & 0 & 0 & 0 & 0 & 0 & 0 & 0\\
                -1 & 0 & 0 & 0 & 0 & 0 & 0 & 0 & 0\\
                0 & 0 & 0 & 1 & 0 & 0 & 0 & 0 & 0\\
                0 & 0 & 1 & 0 & 0 & 0 & 0 & 0 & 0\\
                0 & 0 & 0 & 0 & 0 & 0 & 1 & 0 & 0\\
                0 & 0 & 0 & 0 & 0 & 1 & 0 & 0 & 0\\
                0 & 0 & 0 & 0 & 1 & 0 & 0 & 0 & 0\\
                0 & 0 & 0 & 0 & 0 & 0 & 0 & 1 & 0\\
                0 & 0 & 0 & 0 & 0 & 0 & 0 & 0 & -1\\}. $$
We have
$$X(A,a,b)X(A',a',b') = X(AA',aa',bb'),\quad Q^2=1,\quad  QX(A,a,b)Q^{-1} = X(ab^{-1}A,a,a^2b^{-1}).$$

\begin{lemma*}
Set
\[x=X\left(\hs\SmallMatrix{ 0 & i\\ -i & 0 },1,-1\right)\cdot Q.\]
Then $\Ho^1\hm \Zm_0 = \{ [1], [Q], [x] \}$.
\end{lemma*}

\begin{proof}
Set
\[ H=\{X(A,1,1)\mid A\in\SL(2,\C)\}\subset \Zm_0.\]
We see from the formulas that $H$ is a normal subgroup of $\Zm_0$.
Set $G=\Zm_0/H$.
Write $T=G^\circ$, the identity component of $G$.
Then $T$ is a torus.
Let $t(a,b)$ denote the image of $X(A,a,b)$ in $T$.
Then the group $\{1,Q\}$ acts on $T$ by
$Q *t(a,b)=t( a, a^2\hs b^{-1})$.  We consider the twisted $\R$-torus $_Q T$.
Then $\gamma$ acts in $_Q T$ by
\[^{\gamma*} t(a,b)=t(\bar a,\bar a ^2\bar b^{-1}).\]
We see that $\gamma$ acts on the cocharacter group $\X_*(\hs_Q T)=\Z^2$ by the matrix
$M=\SmallMatrix{1&0\\2&-1}$. The lattice $\X_*(\hs_Q T)$ has the basis
$v_1=\SmallMatrix{1\\1}$, $v_2=\SmallMatrix{0\\1}$ such that $M v_1=v_1$, $M v_2=-v_2$.
It follows that $\Ho^1 \hs_Q T=\{1,t(1,-1)\}$.

We compute the Galois cohomology of $G=\Zm_0/H$.
Set $C=\Zm_0/\Zm_0^\circ=G/T$. Then $C=\{1,c\}$, where  $c$ is the image of $Q$ in
$C$. Clearly, we have $\Ho^1\hs  C=\{[1],[c]\}$.

We have a short exact sequence
\[1\to T\to G\labelto{\pi} C\to 1.\]
We see that
\[\Ho^1  G=\pi_*^{-1}[1]\cup \pi_*^{-1}[c].\]
Since $\Ho^1\hs T=1$, we see that $\pi_*^{-1}[1]=\{[1]\}$.
We compute $\pi_*^{-1}[c]$. It comes from $\Ho^1 \hs_Q T$.

By Corollary \ref{c:39-cor2} $\pi_*^{-1}[c]\cong \Ho^1\hm \hs_Q T/C^\Ga.$
Since  $Q^2=1$, by Lemma \ref{lem:1c} the element $c\in C^\Ga=C$ acts trivially
on $[1]\in \Ho^1\hm \hs_Q T$.
Thus $C^\Ga$ acts trivially on $\Ho^1\hm \hs_Q T$.
We conclude that $\#\pi_*^{-1}[c]=2$ and $\#\Ho^1  G=3$.
We write $g=Q_G$, the image of $Q$ in $G$.
Then
\[\Ho^1 G=\{[1], [g], [t(1,-1)g]\}.\]

We compute the Galois cohomology of $\Zm_0$.
We have a short exact sequence
\[1\to H\to \Zm_0\labelto{\phi} G\to 1.\]
We show that the three cohomology classes $[1], [g], [y]\in \Ho^1\hm G$ can be lifted to $\Zm_0$.
Indeed, the cocycle $1$ lifts to $1$, and the cocycle $g$ lifts to $Q\in B$.
Moreover, it is clear that the cocycle $y\in \Zl^1\hm G$ is the image of the element
\[x=X\left(\hs\SmallMatrix{ 0 & i\\ -i & 0 },1,-1\right)\cdot Q\in \Zm_0,\]
and an easy calculation  shows that $x$ is a cocycle.

For any cocycle $b\in \Zl^1\hm  \Zm_0$, we have
\[\phi_*^{-1}(\phi_*[b])\cong \Ho^1\hm  \hs_b H/H^0(\Ga,\hs_b G);\]
see Corollary \ref{c:39-cor2}.
We have $\Ho^1\hm H=\{1\}$, and therefore, $\phi_*^{-1}[1]=\{1\}$.
Moreover, $_Q H=H$, hence $\Ho^1\hm \hs_Q H=\Ho^1\hm H=\{1\}$, and therefore,
$\phi_*^{-1}[g]$ contains only one class $[Q]$.
Furthermore, $_x H\simeq \SU(2)$ and $\Ho^1\hm  \hs_x H=\{[1],[-1]\}$.
We show below that $H^0(\Ga,\hs_x G)$ acts on $\Ho^1\hm  \hs_x H$
nontrivially, hence transitively,
and therefore, $\phi_*^{-1}(\phi_*[x])$ contains only one class $[x]$.

Write $y=t(1,-1)g$, where $g=Q_G\in G$.
We calculate:
\[^{\gamma*} y=y\cdot\bar y\cdot y^{-1}=yyy^{-1}=y,\]
because the element $y\in G$ is real.
Thus $y\in H^0(\Ga,\hs_y G)$.
We have
\begin{align*}
&x^2=X\left(\hs\SmallMatrix{ 0 & i\\ -i & 0 },1,-1\right)\cdot
QX\left(\hs\SmallMatrix{ 0 & i\\ -i & 0 },1,-1\right)Q^{-1}\cdot Q^2\\
&= X\left(\hs\SmallMatrix{ 0 & i\\ -i & 0 },1,-1\right)\cdot
X\left(\hs-\hs\SmallMatrix{ 0 & i\\ -i & 0 },1,-1\right)
=X\left(\hs\SmallMatrix{ -1 & 0\\ 0 & -1 },1,1\right)=-1\in H.
\end{align*}
By Lemma  \ref{lem:1c}, in $\Ho^1\hm \hs_x H$ we have
\[ [1]\cdot y=[x^{-2}]=[-1].\]
Thus $y\in H^0(\Ga,\hs_y G)$ acts on $\Ho^1\hm  \hs_x H$ nontrivially,
and therefore, $H^0(\Ga,\hs_x G)$ acts on $\Ho^1\hm  \hs_x H$  transitively.
We obtain that $\phi_*^{-1}[y]$ contains only one class $[x]$.
Thus $\Ho^1\hm \Zm_0=\{[1],[Q],[x]\}$, as required.
\end{proof}

\noindent{\bf 68} Representative: $e_{147}+e_{156}-e_{234}-e_{578}+e_{679}$.\\
Here $\Zm_0$ consists of
$$\SmallMatrix{ d^{-1} & 0 & 0 & 0 & 0 & 0 & 0 & 0 & 0\\
                0 & a_{22} & a_{23} & 0 & 0 & 0 & 0 & 0 & 0\\
                0 & a_{32} & a_{33} & 0 & 0 & 0 & 0 & 0 & 0\\
                0 & 0 & 0 & d^2 & 0 & 0 & 0 & 0 & 0\\
                0 & 0 & 0 & 0 & a_{99} & a_{98} & 0 & 0 & 0\\
                0 & 0 & 0 & 0 & a_{89} & a_{88} & 0 & 0 & 0\\
                0 & 0 & 0 & 0 & 0 & 0 & d^{-1} & 0 & 0\\
                0 & 0 & 0 & 0 & 0 & 0 & 0 & a_{88} & a_{89}\\
                0 & 0 & 0 & 0 & 0 & 0 & 0 & a_{98} & a_{99}\\} $$
where $d=a_{88}a_{99}-a_{89}a_{98}$, and $a_{22}a_{33}-a_{23}a_{32} = d^{-2}$.

We have a $\Gamma$-equivariant homomorphism $j\colon \Zm_0\to\C^\times$ sending
an element $(a_{ij};d)$ of $\Zm_0$ to $d$.
We obtain a short exact sequence
\[1\to G\labelto{ } \Zm_0\labelto{j} \C^\times\to 1,\]
where $G\cong\SL(2,\C)\times\SL(2,\C)$. So $\Ho^1 G=1$ and by
Corollary \ref{c:prop38} we conclude that $\Ho^1\hm \Zm_0=1$.

\noindent{\bf 69} Representative: $e_{156}-e_{237}-e_{246}-e_{345}+e_{479}+e_{678}$.\\
Here $\Zm_0$ consists of $\diag(s^{-3}t^{-2},s^3t,s^{-4}t^{-1},s^{-1}t^{-1},s^5t^2,
s^{-2},s,s,t)$ for $s,t\in \C^\times$. It follows that $\Ho^1\hm \Zm_0=1$.

\noindent{\bf 70} Representative: $e_{137}-e_{246}-e_{247}-e_{345}+e_{569}$.
Rank 8, Gurevich number XX.\\
Here $\Zm_0^\circ$ consists of
$$D(s,t,u)=\diag(st^{-2}u^{-1},s^{-1}t^{-1},s^{-1}tu,s,t^{-1}u^{-1},t,t,t,u)$$
for $s,t,u\in \C^\times$. The component group is of order 6 and is isomorphic to $S_3$.
It is generated by
$$g_2= \SmallMatrix{ 0 & -1 & 0 & 0 & 0 & 0 & 0 & 0 & 0\\
                -1 & 0 & 0 & 0 & 0 & 0 & 0 & 0 & 0\\
                0 & 0 & 0 & 1 & 0 & 0 & 0 & 0 & 0\\
                0 & 0 & 1 & 0 & 0 & 0 & 0 & 0 & 0\\
                0 & 0 & 0 & 0 & -1 & 0 & 0 & 0 & 0\\
                0 & 0 & 0 & 0 & 0 & -1 & 1 & 0 & 0\\
                0 & 0 & 0 & 0 & 0 & 0 & 1 & 0 & 0\\
                0 & 0 & 0 & 0 & 0 & 0 & 0 & 1 & 0\\
                0 & 0 & 0 & 0 & 0 & 0 & 0 & 0 & 1\\},\quad
g_3 = \SmallMatrix{ 0 & 0 & 0 & 0 & 0 & 0 & 0 & 0 & 1\\
                1 & 0 & 0 & 0 & 0 & 0 & 0 & 0 & 0\\
                0 & 0 & 0 & 0 & 1 & 0 & 0 & 0 & 0\\
                0 & 0 & 1 & 0 & 0 & 0 & 0 & 0 & 0\\
                0 & 0 & 0 & 1 & 0 & 0 & 0 & 0 & 0\\
                0 & 0 & 0 & 0 & 0 & 0 & -1 & 0 & 0\\
                0 & 0 & 0 & 0 & 0 & 1 & -1 & 0 & 0\\
                0 & 0 & 0 & 0 & 0 & 0 & 0 & 1 & 0\\
                0 & 1 & 0 & 0 & 0 & 0 & 0 & 0 & 0\\}
$$

We have
\[g_2\cdot D(s,t,u)\cdot g_2^{-1}=D(s^{-1}tu,t,u)\]
and
\[  g_3\cdot D(s,t,u)\cdot g_3^{-1}=D(s^{-1}tu,t,s^{-1}t^{-1}).\]
Furthermore
\[g_2^2=1,\quad g_3^3=1,\]
and
\[ g_2 g_3 g_2^{-1}=g_3^2\cdot D(1,1,-1).\]

\begin{lemma*}
$\Ho^1\hm  \Zm_0=\{[1],[g_2]\}$.
\end{lemma*}

\begin{proof}
Set
\[ T=\Zm_0^\circ=\{D(s,t,u)\mid s,t,u\in\C^\times\}.\]
We have a short exact sequence
\[1\to T\to \Zm_0\labelto{\phi} S_3\to 1.\]
The group $\Ga$ acts trivially on $S_3$\hs, whence
\[\Ho^1\hm S_3=\{[1],[\bar g_2]\},\]
where $\bar g_2=g_2 T\in S_3$\hs.
The cocycle $\bar g_2\in \Zl^1\hm S_3$ is the image of the cocycle
$g_2\in \Zl^1\hm  \Zm_0$. We have
\[ \Ho^1\hm \Zm_0=\phi_*^{-1}[1]\cup \phi_*^{-1}[\bar g_2].\]
Since $\Ho^1 T=1$, we see that $\phi_*^{-1}[1]$ contains only one class $[1]$.

We compute $\phi_*^{-1}[\bar g_2]$.
By Corollary \ref{c:39-cor2} there is a bijection between $\phi_*^{-1}[\bar g_2]$
and the set of orbits of $(\hs_{\bar g_2} S_3)^\Ga$ on $\Ho^1\hm \hs_{g_2} T$.

We compute $\Ho^1\hm \hs_{g_2} T$. The conjugation on $\hs_{g_2} T$ is given by
$$D(s,t,u) \mapsto D(\ov s^{-1} \ov t \ov u, \ov t, \ov u ).$$
Hence $\Ho^1\hm \hs_{g_2} T=1$.

Thus $\phi_*^{-1}[\bar g_2]$ contains only one class $[g_2]$.
We conclude that $\Ho^1\hm \Zm_0=\{[1],[g_2]\}$, as required.
\end{proof}

\noindent{\bf 71} Representative: $e_{137}-e_{246}+e_{568}+e_{579}$.\\
Here $\Zm_0^\circ$  consists of
$$X( \SmallMatrix{ a_{11} & a_{13} \\a_{31} & a_{33}},\SmallMatrix{ a_{22} & a_{24} \\a_{42} & a_{44}} )=
  \SmallMatrix{ a_{11} & 0 & a_{13} & 0 & 0 & 0 & 0 & 0 & 0 \\
                0 & a_{22} & 0 & a_{24} & 0 & 0 & 0 & 0 & 0 \\
                a_{31} & 0 & a_{33} & 0 & 0 & 0 & 0 & 0 & 0 \\
                0 & a_{42} & 0 & a_{44} & 0 & 0 & 0 & 0 & 0 \\
                0 & 0 & 0 & 0 & s^{-1}t^{-1} & 0 & 0 & 0 & 0 \\
                0 & 0 & 0 & 0 & 0 & t & 0 & 0 & 0 \\
                0 & 0 & 0 & 0 & 0 & 0 & s & 0 & 0 \\
                0 & 0 & 0 & 0 & 0 & 0 & 0 & s & 0 \\
                0 & 0 & 0 & 0 & 0 & 0 & 0 & 0 & t} $$
with $s=(a_{11}a_{33}-a_{13}a_{31})^{-1}$, $t=(a_{22}a_{44}-a_{24}a_{42})^{-1}$.
So $\Zm_0^\circ$ is isomorphic to $\GL(2,\C)\times \GL(2,\C)$. The component group $C$
is of order 2 and generated by the image $c$ of
$$Q=\SmallMatrix{ 0 & 1 & 0 & 0 & 0 & 0 & 0 & 0 & 0 \\
                1 & 0 & 0 & 0 & 0 & 0 & 0 & 0 & 0 \\
                0 & 0 & 0 & -1 & 0 & 0 & 0 & 0 & 0 \\
                0 & 0 & -1 & 0 & 0 & 0 & 0 & 0 & 0 \\
                0 & 0 & 0 & 0 & 1 & 0 & 0 & 0 & 0 \\
                0 & 0 & 0 & 0 & 0 & 0 & 1 & 0 & 0 \\
                0 & 0 & 0 & 0 & 0 & 1 & 0 & 0 & 0 \\
                0 & 0 & 0 & 0 & 0 & 0 & 0 & 0 & 1 \\
                0 & 0 & 0 & 0 & 0 & 0 & 0 & 1 & 0} $$

We have $Q^2=1$ and
$$Q X( \SmallMatrix{ a_{11} & a_{13} \\a_{31} & a_{33}},
\SmallMatrix{ a_{22} & a_{24} \\a_{42} & a_{44}} ) Q^{-1} =
X( \SmallMatrix{ a_{22} & -a_{24} \\-a_{42} & a_{44}},\SmallMatrix{ a_{11} & -a_{13} \\-a_{31} & a_{33}} ).$$

\begin{lemma*}
$\Ho^1\hm \Zm_0=\{[1],[Q]\}$.
\end{lemma*}

\begin{proof}
We have a short exact sequence
\[1\to \Zm_0^\circ\labelto{} \Zm_0\labelto{\pi} C\to 1.\]
Clearly $\Ho^1\hm C=\{[1],[c]\}$.
We have that $\pi_*^{-1}[1]\subset \Ho^1\hm \Zm_0$ is the image of
$\Ho^1\hm  \Zm_0^\circ$. As the latter is trivial we see that
$\pi_*^{-1}[1]=\{[1]\}$.
We compute $\pi_*^{-1}[c]\subset \Ho^1\hm \Zm_0$.
It comes from $\Ho^1\hm \hs_Q \Zm_0^\circ$.
Now $_Q \Zm_0^\circ\simeq R_{\C/\R}\GL(2,\C)$, and by Proposition \ref{p:Weil} we have
$\Ho^1\hm \hs_Q \Zm_0^\circ=1$.  Thus $\pi_*^{-1}[c]=\{[Q]\}\subset \Ho^1\hm \Zm_0$.
We see that $\Ho^1\hm \Zm_0=\{[1],[Q]\}$, as required.
\end{proof}

\noindent{\bf 72} Representative: $e_{147}+e_{156}-e_{237}-e_{246}-e_{345}-e_{578}+e_{679}$.\\
Here $\Zm_0$ consists of
$$\SmallMatrix{ \zeta d & 0 & 0 & 0 & 0 & 0 & 0 & 0 & 0 \\
                0 & \zeta d a_{99} & -\zeta d a_{98} & 0 & 0 & 0 & 0 & 0 & 0 \\
                0 & -\zeta d a_{89} & \zeta d a_{88} & 0 & 0 & 0 & 0 & 0 & 0 \\
                0 & 0 & 0 & d^2 & 0 & 0 & 0 & 0 & 0 \\
                0 & 0 & 0 & 0 & \zeta^2 d^2 a_{99} & \zeta^2 d^2a_{98} & 0 & 0 & 0 \\
                0 & 0 & 0 & 0 & \zeta^2 d^2 a_{89} & \zeta^2 d^2 a_{88} & 0 & 0 & 0 \\
                0 & 0 & 0 & 0 & 0 & 0 & \zeta & 0 & 0 \\
                0 & 0 & 0 & 0 & 0 & 0 & 0 & a_{88} & a_{89} \\
                0 & 0 & 0 & 0 & 0 & 0 & 0 & a_{98} & a_{99}} $$
with $\zeta^3=1$ and $d= a_{88}a_{99}-a_{98}a_{89}$ and $d^3=1$. So $\Zm_0^\circ$ is
isomorphic to $\SL(2,\C)$ and the component group $C$ is abelian of order 9.
By Corollary \ref{c:2m+1} we see that $\Ho^1\hm C=1$. Also we have $\Ho^1\hm
\Zm_0^\circ=1$. Hence by Corollary \ref{c:prop38} we also have $\Ho^1\hm \Zm_0=1$.

\noindent{\bf 73} Representative: $e_{137}-e_{256}-e_{346}+e_{479}-e_{578}$.\\
Here $\Zm_0$ consists of
$$
X(\SmallMatrix{ a_{33} & a_{34} \\ a_{43} & a_{44} }, a_{88} )=
\SmallMatrix{ a_{66}a_{88}^{-1}a_{44} & 0 & 0 & 0 & 0 & 0 & 0 & 0 &
  -a_{66}a_{88}^{-1}a_{43}\\
                0 & a_{66}^{-1}a_{88}^{2} & 0 & 0 & 0 & 0 & 0 & 0 & 0 \\
                0 & 0 & a_{33} & a_{34} & 0 & 0 & 0 & 0 & 0 \\
                0 & 0 & a_{43} & a_{44} & 0 & 0 & 0 & 0 & 0 \\
                0 & 0 & 0 & 0 & a_{88}^{-2} & 0 & 0 & 0 & 0 \\
                0 & 0 & 0 & 0 & 0 & a_{66} & 0 & 0 & 0 \\
                0 & 0 & 0 & 0 & 0 & 0 & a_{88} & 0 & 0 \\
                0 & 0 & 0 & 0 & 0 & 0 & 0 & a_{88} & 0 \\
                -a_{66}a_{88}^{-1}a_{34} & 0 & 0 & 0 & 0 & 0 & 0 & 0 &
                a_{66}a_{88}^{-1}a_{33}} $$
with $a_{66}=(a_{33}a_{44}-a_{34}a_{43})^{-1}$, $a_{88}\in \C^\times$. We have
$$X(A,a_{88}) X(A',a_{88}') = X(AA',a_{88}a_{88}').$$
We see that $\Zm_0\simeq \GL(2,\C)\times \C^\times$ and therefore $\Ho^1\hm \Zm_0=1$.

\noindent{\bf 74} Representative: $e_{156}-e_{237}-e_{246}-e_{345}+e_{479}$.  Rank 8. Gurevich number XVIII.\\
Here $\Zm_0$ consists of
$$X(\SmallMatrix{a_{55} & a_{56}\\ a_{65} & a_{66}},s,t)=
\SmallMatrix{ s^{-3}t^{-2} & 0 & 0 & 0 & 0 & 0 & 0 & 0 & 0\\
                0 & s^{-2}t^{-1}a_{55} & -s^{-2}t^{-1}a_{56} & 0 & 0 & 0 & 0 & 0 & 0\\
                0 & -s^{-2}t^{-1}a_{65} & s^{-2}t^{-1}a_{66} & 0 & 0 & 0 & 0 & 0 & 0\\
                0 & 0 & 0 & s^{-1}t^{-1} & 0 & 0 & 0 & 0 & 0\\
                0 & 0 & 0 & 0 & a_{55} & a_{56} & 0 & 0 & 0\\
                0 & 0 & 0 & 0 & a_{65} & a_{66} & 0 & 0 & 0\\
                0 & 0 & 0 & 0 & 0 & 0 & s & 0 & 0\\
                0 & 0 & 0 & 0 & 0 & 0 & 0 & s & 0\\
                0 & 0 & 0 & 0 & 0 & 0 & 0 & 0 & t\\} $$
with $a_{55}a_{66}-a_{56}a_{65} = s^3t^2$, $s,t\in \C^\times$.

Consider the homomorphism
\[\phi\colon \Zm_0\to (\C^\times)^2,\quad X(\SmallMatrix{a_{55} & a_{56}\\ a_{65} & a_{66}},s,t)\mapsto (s,t).\]
We obtain a short exact sequence
\[1\to\SL(2,\C)\to \Zm_0\to (\C^\times)^2\to 1.\]
Since $\Ho^1\SL(2,\C)=1$ and $\Ho^1 (\C^\times)^2=1$, we obtain that
$\Ho^1\hm \Zm_0=1$ (Corollary \ref{c:prop38}).

\noindent{\bf 75} Representative: $e_{137}-e_{246}-e_{247}+e_{569}+e_{678}$.\\
Here we have $\Zm_0= G\times \mu_3$.
The identity component $A$ of $G$ is semisimple of type $3A_1$ and consists of
$$X(\SmallMatrix{ a_{11} & a_{13} \\a_{31} & a_{33}},
\SmallMatrix{ a_{22} & a_{24} \\a_{42} & a_{44}},
\SmallMatrix{ a_{55} & a_{59} \\a_{95} & a_{99}})=
\SmallMatrix{ a_{11} & 0 & a_{13} & 0 & 0 & 0 & 0 & 0 & 0 \\
                0 & a_{22} & 0 & a_{24} & 0 & 0 & 0 & 0 & 0 \\
                a_{31} & 0 & a_{33} & 0 & 0 & 0 & 0 & 0 & 0 \\
                0 & a_{42} & 0 & a_{44} & 0 & 0 & 0 & 0 & 0 \\
                0 & 0 & 0 & 0 & a_{55} & 0 & 0 & 0 & a_{59} \\
                0 & 0 & 0 & 0 & 0 & 1 & 0 & 0 & 0 \\
                0 & 0 & 0 & 0 & 0 & 0 & 1 & 0 & 0 \\
                0 & 0 & 0 & 0 & 0 & 0 & 0 & 1 & 0 \\
                0 & 0 & 0 & 0 & a_{95} & 0 & 0 & 0 & a_{99}},$$
where the three obvious determinants are 1.
The component group $C$ of $G$ is isomorphic to $S_3$, and generated by the images of the matrices

$$P=\SmallMatrix{ 0 & 0 & 0 & 0 & -1 & 0 & 0 & 0 & 0 \\
                0 & 1 & 0 & 0 & 0 & 0 & 0 & 0 & 0 \\
                0 & 0 & 0 & 0 & 0 & 0 & 0 & 0 & 1 \\
                0 & 0 & 0 & 1 & 0 & 0 & 0 & 0 & 0 \\
                -1 & 0 & 0 & 0 & 0 & 0 & 0 & 0 & 0 \\
                0 & 0 & 0 & 0 & 0 & 0 & 1 & 0 & 0 \\
                0 & 0 & 0 & 0 & 0 & 1 & 0 & 0 & 0 \\
                0 & 0 & 0 & 0 & 0 & 0 & 0 & -1 & 0 \\
                0 & 0 & 1 & 0 & 0 & 0 & 0 & 0 & 0},
\qquad
Q=\SmallMatrix{ 0 & 0 & 0 & 0 & -1 & 0 & 0 & 0 & 0 \\
                0 & 0 & i & 0 & 0 & 0 & 0 & 0 & 0 \\
                0 & 0 & 0 & 0 & 0 & 0 & 0 & 0 & 1 \\
                i & 0 & 0 & 0 & 0 & 0 & 0 & 0 & 0 \\
                0 & 0 & 0 & -i & 0 & 0 & 0 & 0 & 0 \\
                0 & 0 & 0 & 0 & 0 & 0 & -1 & 0 & 0 \\
                0 & 0 & 0 & 0 & 0 & 1 & -1 & 0 & 0 \\
                0 & 0 & 0 & 0 & 0 & 0 & 0 & 1 & 0 \\
                0 & i & 0 & 0 & 0 & 0 & 0 & 0 & 0},$$
(where $i^2=-1$).
The computation used the method outlined in Example \ref{exa:cencomps}.

We have $P^2=1$, \[Q^3=\diag(-1,-1,-1,-1,-1,1,1,1,-1),\ \  P^{-1}QPQ^{-2} =\diag(1,1,1,1,-1,1,1,1,-1),\]
 and
\begin{align*}
PX(\SmallMatrix{ a_{11} & a_{13} \\a_{31} & a_{33}},
\SmallMatrix{ a_{22} & a_{24} \\a_{42} & a_{44}},
\SmallMatrix{ a_{55} & a_{59} \\a_{95} & a_{99}})P^{-1} &=
X(\SmallMatrix{ a_{55} & -a_{59} \\-a_{95} & a_{99}},
\SmallMatrix{ a_{22} & a_{24} \\a_{42} & a_{44}},
\SmallMatrix{ a_{11} & -a_{13} \\-a_{31} & a_{33}}),\\
QX(\SmallMatrix{ a_{11} & a_{13} \\a_{31} & a_{33}},
\SmallMatrix{ a_{22} & a_{24} \\a_{42} & a_{44}},
\SmallMatrix{ a_{55} & a_{59} \\a_{95} & a_{99}})Q^{-1} &=
X(\SmallMatrix{ a_{55} & -a_{59} \\-a_{95} & a_{99}},
\SmallMatrix{ a_{33} & a_{31} \\a_{13} & a_{11}},
\SmallMatrix{ a_{44} & -a_{42} \\-a_{24} & a_{22}}).
\end{align*}

\begin{lemma*}
$\Ho^1\hm  \Zm_0=\{[1],[P]\}$.
\end{lemma*}

\begin{proof}
Clearly, $\Ho^1\hm \Zm_0=\Ho^1\hm G$.
We have a short exact sequence
\[1\to A\labelto{} G\labelto{\pi} C\to 1.\]
Because $\Ga$ acts trivially on $C$ we have $\Ho^1\hm C=\{1,[c]\}$,
where $c$ is the image in $C$  of the real matrix $P$.
We compute $\pi_*^{-1}[1]\subset \Ho^1\hm G.$ Since $\Ho^1\hm  A=1$,
Corollary \ref{c:39-cor1} shows that  $\pi_*^{-1}[1]=\{[1]\}$.
We compute $\pi_*^{-1}[c]\subset \Ho^1\hm G$. For that we twist the exact
sequence by $P$. We have that $_P A\simeq \SL_{2,\R}\times R_{\C/\R}\SL_{2,\C}$,
and by Proposition \ref{p:Weil} we have $\Ho^1\hm \hs_P A=1$. So by Corollary
\ref{c:39-cor2} we conclude that $\pi_*^{-1}[c]=\{[P]\}\subset \Ho^1 G$.
We see that $\Ho^1\hm B=\Ho^1 G=\{[1],[P]\}$, as required.
\end{proof}

\noindent{\bf 76} Representative: $e_{147}+e_{156}-e_{237}-e_{246}-e_{345}+e_{679}$.  Rank 8, Gurevich number XV.\\
Here $\Zm_0$ consists of
$$\SmallMatrix{ \zeta s a_{66} & -\zeta s a_{67} &  0 &  0 &  0 &  0 &  0 &  0 &  0\\
-\zeta s a_{76} & \zeta s a_{77} &  0 &  0 &  0 &  0 &  0 &  0 &  0\\
0 &  0 &  (-\zeta - 1) s a_{66}^2 & (\zeta + 1) s a_{66} a_{67} & (-\zeta - 1) s a_{67}^2 &  0 &  0 &  0 &  0 \\
0 &  0 &  (2 \zeta + 2) s a_{66} a_{76} & (-2 \zeta - 2) s a_{67} a_{76} -\zeta - 1 &  (2 \zeta + 2) s a_{67} a_{77} &  0 &  0 &  0 &  0 \\
0 &  0 &  (-\zeta - 1) s a_{76}^2 & (\zeta + 1) s a_{76} a_{77} &  (-\zeta - 1) s a_{77}^2 &  0 &  0 &  0 &  0 &\\
0 &  0 &  0 &  0 &  0 &  a_{66} &  a_{67} &  0 &  0 &\\
0 &  0 &  0 &  0 &  0 &  a_{76} &  a_{77} &  0 &  0 &\\
0 &  0 &  0 &  0 &  0 &  0 &  0 &  \zeta s^{-1} &  0 &\\
0 &  0 &  0 &  0 &  0 &  0 &  0 &  0 &  s}$$
where $\zeta^3=1$ and $s = (a_{66}a_{77}-a_{67}a_{76})^{-1}$. We see that the
identity component is isomorphic to $\GL(2,\C)$ and the component
group $C$ is of order 3. So $\Ho^1\hm \Zm_0^\circ=1$ and $\Ho^1\hm C=1$
(the latter by Corollary \ref{c:2m+1}). Hence by Corollary \ref{c:prop38},
$\Ho^1\hm \Zm_0=1$.

\noindent{\bf 77} Representative: $e_{137}-e_{246}-e_{356}+e_{579}+e_{678}$.\\
Here $\Zm_0$ consists of
$$
X(\SmallMatrix{ a_{11} & a_{19} \\ a_{91} & a_{99} },
\SmallMatrix{ a_{22} & a_{24} \\ a_{42} & a_{44} }, s ) =
\SmallMatrix{ a_{11} & 0 & 0 & 0 & 0 & 0 & 0 & 0 & a_{19} \\
                0 & a_{22} & 0 & a_{24} & 0 & 0 & 0 & 0 & 0 \\
                0 & 0 & s^3a_{99} & 0 & -s^3a_{91} & 0 & 0 & 0 & 0 \\
                0 & a_{42} & 0 & a_{44} & 0 & 0 & 0 & 0 & 0 \\
                0 & 0 & -s^3a_{19} & 0 & s^3a_{11} & 0 & 0 & 0 & 0 \\
                0 & 0 & 0 & 0 & 0 & s^{-2} & 0 & 0 & 0 \\
                0 & 0 & 0 & 0 & 0 & 0 & s & 0 & 0 \\
                0 & 0 & 0 & 0 & 0 & 0 & 0 & s & 0 \\
                a_{91} & 0 & 0 & 0 & 0 & 0 & 0 & 0 & a_{99}} $$
with $a_{22}a_{44}-a_{24}a_{42}=s^2$, $a_{11}a_{99}-a_{19}a_{91} = s^{-4}$
and $s\in \C^\times$. We have
$$X(A,B,s)X(A',B',s') = X(AA',BB',ss').$$

We define a surjective homomorphism
\[\phi\colon \Zm_0\to \C^\times,\quad X(A,B,s)\mapsto s.\]
The kernel $K$ of $\phi$ is isomorphic to $\SL(2,\C)\times\SL(2,\C)$.
Hence $\Ho^1\hm K=1$. We have the exact sequence
\[ 1\to\, K\,\to \Zm_0\to \C^\times\to 1.\]
Since $\Ho^1 \C^\times=1$, we see by Corollary \ref{c:prop38} that
$\Ho^1\hm \Zm_0=1$ as well.

\noindent{\bf 78} Representative: $e_{137}-e_{246}-e_{247}+e_{569}$.  Rank 8, Gurevich number XIX.\\
Here $\Zm_0^\circ$ has a semisimple
part isomorphic to $\SL(2,\C)^3$, consisting of
$$X(\SmallMatrix{ a_{11} & a_{13} \\a_{31} & a_{33}},
\SmallMatrix{ a_{22} & a_{24} \\a_{42} & a_{44}},
\SmallMatrix{ a_{55} & a_{59} \\a_{95} & a_{99}})=
\SmallMatrix{ a_{11} & 0 & a_{13} & 0 & 0 & 0 & 0 & 0 & 0\\
                0 & a_{22} & 0 & a_{24} & 0 & 0 & 0 & 0 & 0\\
                a_{31} & 0 & a_{33} & 0 & 0 & 0 & 0 & 0 & 0\\
                0 & a_{42} & 0 & a_{44} & 0 & 0 & 0 & 0 & 0\\
                0 & 0 & 0 & 0 & a_{55} & 0 & 0 & 0 & a_{59}\\
                0 & 0 & 0 & 0 & 0 & 1 & 0 & 0 & 0\\
                0 & 0 & 0 & 0 & 0 & 0 & 1 & 0 & 0\\
                0 & 0 & 0 & 0 & 0 & 0 & 0 & 1 & 0\\
                0 & 0 & 0 & 0 & a_{95} & 0 & 0 & 0 & a_{99}\\} $$
(with the three obvious determinants equal to 1)
and a central torus consisting of $D(s)=\diag(s,s,s,s,s,s^{-2},s^{-2},s^{-2},s)$.
The component group is isomorphic to $S_3$ and
generated by the following elements
$$a=\SmallMatrix{ 0 & -1 & 0 & 0 & 0 & 0 & 0 & 0 & 0\\
                -1 & 0 & 0 & 0 & 0 & 0 & 0 & 0 & 0\\
                0 & 0 & 0 & -1 & 0 & 0 & 0 & 0 & 0\\
                0 & 0 & -1 & 0 & 0 & 0 & 0 & 0 & 0\\
                0 & 0 & 0 & 0 & 1 & 0 & 0 & 0 & 0\\
                0 & 0 & 0 & 0 & 0 & 1 & -1 & 0 & 0\\
                0 & 0 & 0 & 0 & 0 & 0 & -1 & 0 & 0\\
                0 & 0 & 0 & 0 & 0 & 0 & 0 & -1 & 0\\
                0 & 0 & 0 & 0 & 0 & 0 & 0 & 0 & 1\\},\quad
b=\SmallMatrix{ 0 & 0 & 0 & 0 & 0 & 0 & 0 & 0 & -1\\
                i & 0 & 0 & 0 & 0 & 0 & 0 & 0 & 0\\
                0 & 0 & 0 & 0 & 1 & 0 & 0 & 0 & 0\\
                0 & 0 & i & 0 & 0 & 0 & 0 & 0 & 0\\
                0 & 0 & 0 & -1 & 0 & 0 & 0 & 0 & 0\\
                0 & 0 & 0 & 0 & 0 & 0 & 1 & 0 & 0\\
                0 & 0 & 0 & 0 & 0 & -1 & 1 & 0 & 0\\
                0 & 0 & 0 & 0 & 0 & 0 & 0 & -1 & 0\\
                0 & 1 & 0 & 0 & 0 & 0 & 0 & 0 & 0\\}
$$
(with $i^2=-1$). The computation used the method outlined in Example
\ref{exa:cencomps}.

We have $a^2=1$, $b^3=D(-i)$ and
$a^{-1}bab^{-2} = \diag(-i,i,-i,i,i,-1,-1,-1,i)$ which is equal to
$D(-i)\diag(-1,1,-1,1,1,1,1,1,1)$. We have
\begin{align*}
aX(\SmallMatrix{ a_{11} & a_{13} \\a_{31} & a_{33}},
\SmallMatrix{ a_{22} & a_{24} \\a_{42} & a_{44}},
\SmallMatrix{ a_{55} & a_{59} \\a_{95} & a_{99}})a^{-1}&=
X(\SmallMatrix{ a_{22} & a_{24} \\a_{42} & a_{44}},
\SmallMatrix{ a_{11} & a_{13} \\a_{31} & a_{33}},
\SmallMatrix{ a_{55} & a_{59} \\a_{95} & a_{99}})\\
bX(\SmallMatrix{ a_{11} & a_{13} \\a_{31} & a_{33}},
\SmallMatrix{ a_{22} & a_{24} \\a_{42} & a_{44}},
\SmallMatrix{ a_{55} & a_{59} \\a_{95} & a_{99}})b^{-1} &=
X(\SmallMatrix{ a_{99} & -a_{95} \\-a_{59} & a_{55}},
\SmallMatrix{ a_{11} & a_{13} \\a_{31} & a_{33}},
\SmallMatrix{ a_{44} & -a_{42} \\-a_{24} & a_{22}}).
\end{align*}

\begin{lemma*}
$\Ho^1\hm  \Zm_0=\{[1],[a]\}$.
\end{lemma*}

\begin{proof}
We write
\[ T=\{D(s)\}, \quad S=\left\{
X(\SmallMatrix{ a_{11} & a_{13} \\a_{31} & a_{33}},
\SmallMatrix{ a_{22} & a_{24} \\a_{42} & a_{44}},
\SmallMatrix{ a_{55} & a_{59} \\a_{95} & a_{99}})\right\},
\quad G=T\cdot S\subset \Zm_0.
\]
Then $T\cap S=\{D(\pm1)\}$.
The subgroup $G$ is normal in $\Zm_0$; we set $C=\Zm_0/G$.
We have a short exact sequence
\[ 1\to G\labelto{ } \Zm_0\labelto{\pi} C \to 1.\]

We compute $\Ho^1 C$. Since $C$ is of order 6, we conclude that
$\#\Ho^1 C=2$ with cocycles $[1]$ and $[aG]$ (Lemma \ref{l:explicit};
note that $a^2=1$, $\ov a=a$). Thus
\[\Ho^1\hm \Zm_0=\pi_*^{-1}[1]\,\cup\,\pi_*^{-1}[aG].\]

We compute $\Ho^1 G$.
We have  $G=T\cdot S$.
We obtain a short exact sequence
\[ 1\to S\to G\to \ov T\to 1,\]
where $\ov T=G/S=T/(T\cap S)$.
We have $\Ho^1 S=1$.
Since $\ov T$ is a one-dimensional torus, we have $\Ho^1\hs \ov T=1$.
We see that $\Ho^1 G=1$.
Hence by Corollary \ref{c:39-cor1} we conclude that $\pi_*^{-1}[1]=\{[1]\}$.

We compute $\pi_*^{-1}[aG]\subset \Ho^1\hm \Zm_0$.
Note that $a^2=1$ and $a$ is real.
It follows that $a$ is a cocycle in $\Zm_0$.
Consider the twisted group
\[_a G=\hs_aT\cdot \hs_aS,\]
where $\hs_aT=T$ because $T$ is central in $\Zm_0$.
We obtain a short exact sequence
\[1\to \hs_a S\to\hs_aG\to\hs_a \ov T\to 1,\]
where $\hs_a \ov T=\ov T$.
 The group $(\hs_a S)(\R)$ is isomorphic to
$\SL(2,\C)\times \SL(2,\R)$. It follows that $\Ho^1\hm \hs_aS=1$.
Since also $\Ho^1\hs \ov T=1$, we see that $\Ho^1\hm \hs_aG=1$.
Using Corollary \ref{c:39-cor2}, we conclude that
$\#\pi_*^{-1}[aG]=[a]$ .

It follows that $\Ho^1\hm  \Zm_0=\{[1],[a]\}$, as required.
\end{proof}

\noindent{\bf 79} Representative: $e_{157}-e_{234}+e_{568}+e_{679}$.\\
Here $\Zm_0$ consists of
$$\SmallMatrix{ a_{11} & 0 & 0 & 0 & 0 & 0 & 0 & a_{18} & a_{19}\\
                0 & a_{22} & a_{23} & a_{24} & 0 & 0 & 0 & 0 & 0\\
                0 & a_{32} & a_{33} & a_{34} & 0 & 0 & 0 & 0 & 0\\
                0 & a_{42} & a_{43} & a_{44} & 0 & 0 & 0 & 0 & 0\\
                0 & 0 & 0 & 0 & a_{99} & -a_{91} & a_{98} & 0 & 0\\
                0 & 0 & 0 & 0 & -a_{19} & a_{11} & a_{18} & 0 & 0\\
                0 & 0 & 0 & 0 & a_{89} & -a_{81} & a_{88} & 0 & 0\\
                a_{81} & 0 & 0 & 0 & 0 & 0 & 0 & a_{88} & a_{89}\\
                a_{91} & 0 & 0 & 0 & 0 & 0 & 0 & a_{98} & a_{99}\\} $$
where
$$ \det \SmallMatrix{ a_{22} & a_{23} & a_{24}\\
                      a_{32} & a_{33} & a_{34}\\
                      a_{42} & a_{43} & a_{44} } = 1,
\det\SmallMatrix{a_{11} & a_{18} & a_{19}\\
                      a_{81} & a_{88} & a_{89}\\
                      a_{91} & a_{98} & a_{99} } = 1.$$
Hence $\Zm_0\simeq \SL(3,\C)\times \SL(3,\C)$, so that $\Ho^1\hm \Zm_0=1$.

\noindent{\bf 80} Representative: $e_{147}-e_{237}-e_{256}-e_{346}+e_{579}+e_{678}$.\\
Here $\Zm_0$ consists of
$$\SmallMatrix{ \zeta^2 s^{-2}a_{55} & 0 & 0 & 0 & 0 & 0 & 0 & 0 & -\zeta^2 s^{-2}
  a_{54}\\
                0 & \zeta^2 s a_{99} & \zeta^2 s a_{91} & 0 & 0 & 0 & 0 & 0 & 0 \\
                0 & -\zeta s^{-1} a_{54} & \zeta s^{-1} a_{55} & 0 & 0 & 0 & 0 & 0 & 0 \\
                0 & 0 & 0 & \zeta s^2 a_{99} & -\zeta s^2 a_{91} & 0 & 0 & 0 & 0 \\
                0 & 0 & 0 & a_{54} & a_{55} & 0 & 0 & 0 & 0 \\
                0 & 0 & 0 & 0 & 0 & \zeta^2 s^{-1} & 0 & 0 & 0 \\
                0 & 0 & 0 & 0 & 0 & 0 & \zeta & 0 & 0 \\
                0 & 0 & 0 & 0 & 0 & 0 & 0 & s & 0 \\
                a_{91} & 0 & 0 & 0 & 0 & 0 & 0 & 0 & a_{99}} $$
with $\zeta^3=1$, $a_{54}a_{91} + a_{55}a_{99} = \zeta^2$. So the identity
component is isomorphic to $\SL(2,\C)\times T_1$. The component group is
isomorphic to $\mu_3$. So using Corollary \ref{c:prop38} we see that
$\Ho^1\hm \Zm_0=1$.

\noindent{\bf 81} Representative: $e_{157}-e_{237}-e_{246}-e_{345}+e_{568}+e_{679}$.\\
Here $\Zm_0^\circ$ is isomorphic to
$\SL(3,\C)$. The natural 9-dimensional module decomposes as a direct sum
of three 3-dimensional submodules. So each matrix in the identity component
is block diagonal with $3\times 3$-blocks on the diagonal. The bases of the
three submodules are $\{v_1,v_8,v_9\}$, $\{v_2,v_3,v_4\}$ and $\{v_5,v_6,v_7\}$
(where $v_1,\ldots,v_9$ is the standard basis of the natural 9-dimensional
module). The component
group is of order 3 and is generated by
$$a=\diag(\zeta, \zeta^7,\zeta^7,\zeta^7,\zeta^4, \zeta^4, \zeta^4,\zeta,\zeta),
$$
where $\zeta$ is a primitive ninth root of unity. So $a^3\in \Zm_0^\circ$.
By Corollary \ref{c:prop38} we see that $\Ho^1\hm \Zm_0=1$.

\noindent{\bf 82} Representative: $e_{137}-e_{246}-e_{356}+e_{579}$.  Rank 8, Gurevich number XVII.\\
Here $\Zm_0$ consists of
$$
X(\SmallMatrix{ a_{22} & a_{24} \\ a_{42} & a_{ 44} },
\SmallMatrix{ a_{33} & a_{35} \\ a_{53} & a_{55} }, s ) =
\SmallMatrix{ ua_{55} & 0 & 0 & 0 & 0 & 0 & 0 & 0 & -ua_{53}\\
                0 & a_{22} & 0 & a_{24} & 0 & 0 & 0 & 0 & 0\\
                0 & 0 & a_{33} & 0 & a_{35} & 0 & 0 & 0 & 0\\
                0 & a_{42} & 0 & a_{44} & 0 & 0 & 0 & 0 & 0\\
                0 & 0 & a_{53} & 0 & a_{55} & 0 & 0 & 0 & 0\\
                0 & 0 & 0 & 0 & 0 & \delta^{-1} & 0 & 0 & 0\\
                0 & 0 & 0 & 0 & 0 & 0 & s & 0 & 0\\
                0 & 0 & 0 & 0 & 0 & 0 & 0 & s & 0\\
                -ua_{35} & 0 & 0 & 0 & 0 & 0 & 0 & 0 & ua_{33}\\} $$
where $u=s^{-1}\delta^{-1}$, $\delta = a_{33}a_{55}-a_{35}a_{53}$ and
$a_{22}a_{44}-a_{24}a_{42}=a_{33}a_{55}-a_{35}a_{53}$.
We have
$$X(A,B,s) X(A',B',s') = X(AA',BB',ss').$$

We define a surjective homomorphism
\[\phi\colon \Zm_0\to (\C^\times)^2\hs,\quad X(A,B,s)\mapsto (\det(A), s)\]
with kernel $K$ isomorphic to $\SL(2,\C)\times\SL(2,\C)$.
We have the short exact sequence
\[1\to\,K\,\to \Zm_0\to (\C^\times)^2\to 1.\]
Since $\Ho^1\hm K =1$ and $\Ho^1\hm  (\C^\times)^2=1$, we infer by
Corollary \ref{c:prop38} that $\Ho^1\hm \Zm_0=1$.

\noindent{\bf 83} Representative: $e_{157}-e_{247}-e_{256}-e_{346}+e_{458}+e_{679}$.\\
Here $\Zm_0^\circ$ is of type $B_2$. The matrices
in this component decompose into a $4\times 4$ - and a $5\times 5$-block,
corresponding to the 4- and 5-dimensional irreducible representations.
The Lie algebra of type $B_2$ has no outer automorphisms, so any other
component of $\Zm_0$ has to have elements that commute with all elements of the
Lie algebra. Adding the
corresponding equations to the original ones we find that the elements
that commute with the Lie algebra form a finite group of six elements,
consisting of $\diag(u,u,u,v,v,v,v,u,u)$ with $u^3=1$ and $v = \pm u$.
The group of diagonal matrices contained in $\Zm_0$ can be computed by adding the
equations expressing that a matrix is diagonal to the original equations.
It turns out that this group is
$$\diag(\delta^2s^{-2}t^{-1},\delta^2,\delta^2s^2t,\delta s^{-1},\delta s t ,
s^{-1}t^{-1},s,\delta t^{-1}, t ),$$
where $\delta^3=1$. This is not connected so $\Zm_0$ is not connected
either (because a maximal torus of a connected semisimple algebraic group
is connected as well). But by the above computation all components of $\Zm_0$
have a representative lying in the diagonal part. Hence
$\Zm_0 = \mu_3 \times \Zm_0^\circ$,
and $\Zm_0^\circ$ is isomorphic to $\Sp(4,\C)$.
Since  $\Ho^1\hs \Sp(4,\C)=1$ and $\Ho^1\hm  \mu_3=1$, we infer by
Corollary \ref{c:prop38} that $\Ho^1\hm \Zm_0=1$.

\noindent{\bf 84} Representative: $e_{147}-e_{237}-e_{256}-e_{346}+e_{579}$. Rank 8, Gurevich number XIV.\\
Here $\Zm_0$ consists of
$$
X(\SmallMatrix{ a_{44} & a_{45} \\ a_{54} & a_{55} }, u, a_{88} )=
\SmallMatrix{ a_{88}^{-2}u^{-2}a_{55} & 0 & 0 & 0 & 0 & 0 & 0 & 0 &
  -a_{88}^{-2}u^{-2}a_{54} \\
                0 & a_{88}^{-1}u^{-1}a_{44} & -a_{88}^{-1}u^{-1}a_{45} & 0 & 0 & 0 & 0 & 0 & 0 \\
                0 & -a_{88}^{-1}u^{-1}a_{54} & a_{88}^{-1}u^{-1}a_{55} & 0 & 0 & 0 & 0 & 0 & 0 \\
                0 & 0 & 0 & a_{44} & a_{45} & 0 & 0 & 0 & 0 \\
                0 & 0 & 0 & a_{54} & a_{55} & 0 & 0 & 0 & 0 \\
                0 & 0 & 0 & 0 & 0 & u^{-2}a_{88}^{-1} & 0 & 0 & 0 \\
                0 & 0 & 0 & 0 & 0 & 0 & u^{-1} & 0 & 0 \\
                0 & 0 & 0 & 0 & 0 & 0 & 0 & a_{88} & 0 \\
                -a_{88}^{-2}u^{-2}a_{45} & 0 & 0 & 0 & 0 & 0 & 0 & 0 &
                a_{88}^{-2}u^{-2}a_{44}} $$
where $u^3a_{88}^2= a_{55}a_{44}-a_{54}a_{45}$. Furthermore,
$$X(A,u,a_{88}) X(A',u',a_{88}') = X(AA',uu',a_{88}a_{88}').$$

We define a homomorphism
\[\phi\colon \Zm_0\to (\C^\times)^2,\quad X(A,u,a_{88})\mapsto (u,a_{88}).\]
Set $\ker\hs\phi=K$, then $K\simeq \SL(2,\C)$.
We obtain a short exact sequence
\[1\to K\to \Zm_0\to (\C^\times)^2\to 1.\]
Since $\Ho^1\hm K=1$ and $\Ho^1\hm (\C^\times)^2=1$, we see by
Corollary \ref{c:prop38} that $\Ho^1\hm \Zm_0=1$.

\noindent{\bf 85} Representative: $e_{147}-e_{256}-e_{346}+e_{579}+e_{678}$.\\
The Lie algebra of $\Zm_0$ is isomorphic to $\s+\mathfrak{t}_1$,
where $\s$ is isomorphic to $\ssl(2,\C)+\ssl(2,\C)$ and $\mathfrak{t}_1$
denotes a 1-dimensional center consisting of diagonal matrices.
In this case it is possible to compute the Gr\"obner basis and even to do
a primary decomposition of the corresponding ideal. The latter shows that the
group is connected (the ideal is prime). The natural 9-dimensional module
splits as direct sum of four irreducible submodules of dimensions 1, 2, 2, 4
with highest weights $(0,0)$, $(1,0)$, $(0,1)$, $(1,1)$. Let $S$ be the
connected algebraic subgroup of $\Zm_0$ with Lie algebra $\s$. Then the module decomposition
shows that $S \simeq \SL(2,\C)\times \SL(2,\C)$. Let $T_1$ be the connected
subgroup of $\Zm_0$ with Lie algebra $\mathfrak{t}_1$. Then the intersection of
$T_1$ and the diagonal maximal torus of $S$ consists just of 1. This implies that
$\Zm_0$ is isomorphic to $\SL(2,\C)\times \SL(2,\C)\times \C^\times$.
Hence $\Ho^1\hm \Zm_0=1$.

\noindent{\bf 86} Representative: $e_{147}-e_{256}-e_{346}+e_{579}$.  Rank 8, Gurevich number XIII.\\
Here the Lie algebra of $\Zm_0$ is isomorphic to $\s+\mathfrak{t}_2$
where $\s$ is isomorphic to $\ssl(2,\C)+\ssl(2,\C)$ and $\mathfrak{t}_2$
denotes a 2-dimensional center. Let $S$, $T_2$ denote the connected subgroups
of $\Zm_0$ with Lie algebras $\s$ and $\mathfrak{t}_2$ respectively. Then
$S$ consists of the elements
$$\SmallMatrix{ * & * & * & 0 & 0 & 0 & 0 & 0 & * \\
              * & * & * & 0 & 0 & 0 & 0 & 0 & * \\
              * & * & * & 0 & 0 & 0 & 0 & 0 & * \\
              0 & 0 & 0 & a_{44} & a_{45} & 0 & 0 & 0 & 0 \\
              0 & 0 & 0 & a_{54} & a_{55} & 0 & 0 & 0 & 0 \\
              0 & 0 & 0 & 0 & 0 & a_{66} & a_{67} & 0 & 0 \\
              0 & 0 & 0 & 0 & 0 & a_{76} & a_{77} & 0 & 0 \\
              0 & 0 & 0 & 0 & 0 & 0 & 0 & 1 & 0 \\
              * & * & * & 0 & 0 & 0 & 0 & 0 & *} $$
(where the obvious determinants are 1), whereas $T_2$ consists of
$$T_2(a,b)=\diag(a^{-1}b^{-1},a^{-1}b^{-1},a^{-1}b^{-1},a,a,b,b,a^2b^2,a^{-1}b^{-1})
\text{ for } a,b\in \C^\times.$$
We see that $\Zm_0^\circ \simeq \GL(2,\C)\times \GL(2,\C)$. In the sequel we
write the elements of $\Zm_0^\circ$ as $(g_1,g_2)$, where $g_i\in \GL(2,\C)$.
The component group is of order 2 and generated by the image of
$$Q=\SmallMatrix{ 0 & -1 & 0 & 0 & 0 & 0 & 0 & 0 & 0 \\
                -1 & 0 & 0 & 0 & 0 & 0 & 0 & 0 & 0 \\
                0 & 0 & 1 & 0 & 0 & 0 & 0 & 0 & 0 \\
                0 & 0 & 0 & 0 & 0 & -1 & 0 & 0 & 0 \\
                0 & 0 & 0 & 0 & 0 & 0 & -1 & 0 & 0 \\
                0 & 0 & 0 & 1 & 0 & 0 & 0 & 0 & 0 \\
                0 & 0 & 0 & 0 & 1 & 0 & 0 & 0 & 0 \\
                0 & 0 & 0 & 0 & 0 & 0 & 0 & -1 & 0 \\
                0 & 0 & 0 & 0 & 0 & 0 & 0 & 0 & 1}. $$
The computation used the methods of Example \ref{exa:cencomps}(2).

We have
\begin{gather*}
Q^2= \diag(1,1,1,-1,-1,-1,-1,1,1) =(-1,-1)\in \Zm_0^\circ,\\
Q\cdot(g_1,g_2)\cdot Q^{-1} = (g_2,g_1).
\end{gather*}

\begin{lemma*}
  $\Ho^1\hm \Zm_0 = \{ [1], [P] \}$ where $P=(w,w)\cdot Q$ with
  $w=\SmallMatrix{ \phantom{-}0 &1\\-1 &0}\in\GL(2,\C)$.
\end{lemma*}

\begin{proof}
We have a short exact sequence
\[1\to \Zm_0^\circ\to \Zm_0\labelto{\pi} \mu_2\to 1, \quad (g_1,g_2)\cdot Q^m\,\longmapsto\, (-1)^m,\]
whence a cohomology exact sequence
\[\Ho^1\hm \Zm_0^\circ\to \Ho^1\hm \Zm_0\to \Ho^1\hm \mu_2=\{[1],[-1]\}.\]
Since $\Ho^1\hm \Zm_0^\circ=1$, we see that the preimage in  $\Ho^1\hm \Zm_0$ of
$[1]\in  \Ho^1\hm \mu_2$ is only one class $[1]\in \Ho^1\hm \Zm_0$.

We compute the preimage in   $\Ho^1\hm \Zm_0$  of $[-1]\in  \Ho^1\hm \mu_2$.
We have $Q^2=(-1,-1)$ so that $P^2=(1,1)$; thus $P$ is a 1-cocycle lifting
$[-1]\in \Ho^1\hm \mu_2$.
We have a short exact sequence
\[1\to \hs_P\Zm_0^\circ\to \hs_P\Zm_0\to \mu_2\to 1, \quad (g_1,g_2)\cdot  P^m\,\longmapsto\, (-1)^m.\]
We wish to compute $\Ho^1\hm \hs_P\Zm_0^\circ$.
We have $_P\Zm_0^\circ\simeq R_{\C/\R}\GL(2,\C)$, and by Proposition \ref{p:Weil},
$\Ho^1\hm \hs_P\Zm_0^\circ=1$.
Thus $\pi_*^{-1}([-1])=[P]$,  which completes the proof.
\end{proof}

\noindent{\bf 87} Representative: $e_{127}+e_{379}-e_{456}+e_{678}$.\\
Here $\Zm_0$ is connected (here this can be established by computing a
primary decomposition of the defining ideal) and isomorphic to
$\Sp(4,\C)\times \SL(2,\C)\times \C^\times$ (the latter follows
by considering the Lie algebra). Hence $\Ho^1\hm \Zm_0=1$.

\noindent{\bf 88} Representative: $e_{157}-e_{247}-e_{256}-e_{346}+e_{679}$. Rank 8, Gurevich number XII.\\
Here $\Zm_0$ consists of
$$\SmallMatrix{ s^{-1}a_{66}^2 & s^{-1}a_{66}a_{67} & -s^{-1}a_{67}^2 & 0 & 0 & 0 & 0 & 0 & 0 \\
                2s^{-1}a_{66}a_{76} & 2s^{-1}a_{67}a_{76}+s^{-1}t^{-1} & -2s^{-1}a_{67}a_{77} & 0 & 0 & 0 & 0 & 0 & 0 \\
                -s^{-1}a_{76}^2 & -s^{-1}a_{76}a_{77} & s^{-1}a_{77}^2 & 0 & 0 & 0 & 0 & 0 & 0 \\
                0 & 0 & 0 & st^2a_{66} & st^2a_{67} & 0 & 0 & 0 & 0 \\
                0 & 0 & 0 & st^2a_{76} & st^2a_{77} & 0 & 0 & 0 & 0 \\
                0 & 0 & 0 & 0 & 0 & a_{66} & a_{67} & 0 & 0 \\
                0 & 0 & 0 & 0 & 0 & a_{76} & a_{77} & 0 & 0 \\
                0 & 0 & 0 & 0 & 0 & 0 & 0 & s & 0 \\
                0 & 0 & 0 & 0 & 0 & 0 & 0 & 0 & t} $$
with $a_{66}a_{77}-a_{67}a_{76} = t^{-1}$. So $\Zm_0$ is connected and
isomorphic to $\GL(2,\C)\times \C^\times$. Therefore $\Ho^1\hm \Zm_0=1$.

\noindent{\bf 89} Representative: $e_{157}-e_{237}-e_{456}+e_{478}+e_{679}$.\\
Here $\Zm_0^\circ$ is of type $A_1+A_2$. It consists of
matrices that decompose into two $3\times 3$-blocks, a $2\times 2$-block
and a $1\times 1$-block. The two $3\times 3$-blocks correspond to two dual
representations of $\SL(3,\C)$, the $2\times 2$-block is a copy of $\SL(2,\C)$,
in the $1\times 1$-block there simply is 1. The component group is isomorphic
to $\mu_3$. So $\Zm_0 = \mu_3 \times \SL(2,\C)\times \SL(3,\C)$.
Hence $\Ho^1\hm \Zm_0=1$.

\noindent{\bf 90} Representative: $e_{147}+e_{156}-e_{237}-e_{246}-e_{345}$.  Rank 7, Gurevich number X.\\
Here $\Zm_0$ is semisimple of type $G_2+A_1$.
The matrices in the centralizer have a $7\times 7$-block for $G_2$ and a
$2\times 2$-block for $A_1$. The centralizer of the Lie algebra $\z_0$ of
$\Zm_0$ consists of
$$\diag(t^4,t^4,t^4,t^4,t^4,t^4,t^4,t,t)$$
with $t^6=1$. The elements with $t=1$, $t=-1$ lie in the identity component,
the others do not. So $\Zm_0\simeq G_2(\C)\times \SL(2,\C) \times \mu_3$.

Since the $H^1$ is trivial for $\SL_2$ and $\mu_3$,
it suffices to compute $\Ho^1\hm  G_2$. It is known that $\#H^1=2$
(because $G_2$ has exactly two real forms: the split one and the compact one).
We explain how to find explicit cocycles.

Let $G=G_2$, $T\subset G$ be a split torus, $R=R(G,T)\subset {\sf X}^*(T)$
denote the root system,
$\Pi=\{\alpha_l,\alpha_s\}$ be a system of simple roots in $R$, where
$\alpha_l$ is a long root and $\alpha_s$ is a short root.
Then $\Pi$ is a basis of ${\sf X}^*(T)$.
Let $t\in T$ be the element such that
$$\alpha_l(t)=-1,\quad  \alpha_s(t)=1.$$
Then $t$ is real, $t^2=1$, so $t$ is a cocycle.
It is cohomologous to 1, because it is contained in the split torus $T$.

Now let $T'\subset G$ be an {\em anisotropic} maximal torus defined over $\R$,
that is, $T'\subset G$ is defined over $\R$ and $T'(\R)\simeq U(1)\times U(1)$.

We can find $g\in G$ such that $gTg^{-1}=T'$.
Set $t'=gtg^{-1}\in T'(\C)_2= T'(\R)_2$.
By Borovoi and Timashev  \cite[Theorem 13.3]{BT2021}, $1,t'$ are explicit cocycles for $G=G_2$.

Our computations show that we can take
$$t'=\SmallMatrix{ 0 & 0 & 0 & 0 & 0 & 0 & 1 & 0 & 0 \\
                0 & 0 & 0 & 0 & 0 & 1 & 0 & 0 & 0 \\
                0 & 0 & 0 & 0 & 1 & 0 & 0 & 0 & 0 \\
                0 & 0 & 0 & -1 & 0 & 0 & 0 & 0 & 0 \\
                0 & 0 & 1 & 0 & 0 & 0 & 0 & 0 & 0 \\
                0 & 1 & 0 & 0 & 0 & 0 & 0 & 0 & 0 \\
                1 & 0 & 0 & 0 & 0 & 0 & 0 & 0 & 0 \\
                0 & 0 & 0 & 0 & 0 & 0 & 0 & 1 & 0 \\
                0 & 0 & 0 & 0 & 0 & 0 & 0 & 0 & 1}. $$

\noindent{\bf 91} Representative: $e_{127}+e_{379}-e_{456}$.  Rank 8, Gurevich number XVI.\\
Here $\Zm_0$ is connected (this is established by computing a primary
decomposition of the defining ideal of $\Zm_0$) and isomorphic to
$\SL(3,\C)\times \Sp(4,\C) \times \C^\times$ (this follows by inspecting
the Lie algebra of $\Zm_0$).
In matrix form it splits as a $4\times 4$-block (for
$\Sp(4,\C)$) a $3\times 3$-block (for $\SL(3,\C)$) and two $1\times 1$-blocks.
Hence $\Ho^1\hm \Zm_0=1$.

\noindent{\bf 92} Representative: $e_{157}-e_{247}-e_{356}+e_{679}$.
Rank 8, Gurevich number XI.\\
The Lie algebra of $\Zm_0$ is isomorphic to $\ssl(2,\C)+\ssl(2,\C)+
\mathfrak{t}_2$, where the last summand denotes the 2-dimensional center.
Computing the primary decomposition of the defining ideal shows that
$\Zm_0$ is connected. The natural 9-dimensional module
splits as direct sum of six irreducible submodules of dimensions 1, 1, 1,
2, 2, 2. Looking at the torus with Lie algebra $\mathfrak{t}_2$ it is seen
that $\Zm_0$ is isomorphic to $\GL(2,\C)\times\GL(2,\C)$. So
$\Ho^1\hm \Zm_0=1$.

\noindent{\bf 93} Representative: $e_{157}-e_{247}-e_{256}-e_{346}$.
Rank 7, Gurevich number IX.\\
The Lie algebra of $\Zm_0$ is isomorphic to $\ssl(2,\C)+\ssl(2,\C)+\ssl(2,\C)+
\mathfrak{t}_1$, where the last summand denotes the 1-dimensional center.
The Gr\"obner basis is difficult to compute. However, by adding equations it
is easy to check that the elements of $\Zm_0$ map the summands isomorphic to
$\ssl(2,\C)$ in the Lie algebra to themselves. Secondly, the elements of $\Zm_0$
that act as the identity on the semisimple part form a 1-dimensional not
connected diagonal group, which we denote $U$. The diagonal matrices in
$\Zm_0$ form a connected 4-dimensional torus containing $U$. Therefore $\Zm_0$ is
connected.

The group $\Zm_0$ is the image of $\widehat{\Zm_0}
=G_1\times G_2\times G_3\times T_1$, where each $G_i$ is $\SL(2,\C)$ and
$T_1=\C^\times$, under the 9-dimensional representation
\[\rho \colon \widehat{\Zm_0}\to GL(V_9),\]
where  the 9-dimensional $\Zm_0$-module $V_9$ splits as a direct sum of 3 modules
\[ V_9=V_3\oplus V_4\oplus  V_2\]
of dimensions 3, 4, 2, with highest weights $(2,0,0,-2)$, $(1,1,0,1)$,
$(0,0,1,1)$, respectively.
It is easy to see that $\ker\rho$ is a group of order 4, namely,
\[\ker\rho\,=\,\{1,\,(-1,-1,1,1),\,(1,-1,-1,-1),\,(-1,1,-1,-1)\}.\]

\begin{lemma*}
We have $\# \Ho^1\hm \Zm_0=2$ with  cocycles
\[1\ \text{and } \rho(w,w,1,1),\quad\text{where }  w=\SmallMatrix{\phantom{-} 0 &1\\-1 &0}\in\SL(2,\C).   \]
\end{lemma*}

\begin{proof}
We write $\Zm_0^\mathrm{ss}=[\Zm_0,\Zm_0]=\rho(G_1\times G_2\times G_3\times 1)$.
First we compute $\Ho^1\hm  \Zm_0^\mathrm{ss}$.
Note that we have $\Zm_0^\mathrm{ss}\cong\rho(G_1\times G_2\times 1\times 1)\times G_3$, and $\Ho^1  G_3=1$.
Thus the embedding
\[\rho(G_1\times G_2\times 1 \times 1)\into \Zm_0^\mathrm{ss}\]
induces a bijection on $H^1$.
Therefore, it suffices to compute the $H^1$ of $\rho(G_1\times G_2)$,
where we write  $\rho(G_1\times G_2)$ for $\rho(G_1\times G_2\times 1 \times 1)$.
One can compute the $H^1$ by the method of Borovoi and Timashev \cite[Theorem 13.3]{BT2021}.
Here we explain the answer elementarily.

It suffices to compute the $H^1$ of $G=(G_1\times G_2)/\mu_2$, where $G_i=\SL(2,\C)$.
Consider $G_1$ and the compact maximal torus
\[T_1=\left\{\SmallMatrix{ \cos\alpha &\sin\alpha\\-\sin\alpha &\cos\alpha},\quad \alpha\in\R\right\}.\]
We set $T=(T_1\times T_2)/\mu_2\subset G$.
It is known (see \cite[Theorem 1]{Borovoi1988} or \cite[Theorem 9]{Borovoi2014}) that $\Ho^1 G$ comes from $T(\R)_2$.
Let $\xi\in T(\R)_2$; then $\xi$ comes from $(w^{n_1}, w^{n_2})\in T_1\times T_2$, where $n_1,n_2\in \{0,1,2,3\}$.

Note that the cocycles $(1,-1),\ (-1,1),\ (-1,-1)$ of $G$ come from cocycles of $G_1\times G_2$.
Since $\Ho^1  (G_1\times G_2)=1$, we see that these three cocycles are coboundaries in $G_1\times G_2$,
and therefore, they are coboundaries in $G$.
This we have 5 cocycles: $$1,\ (w,w),\ (w,-w),\ (-w,w),\ (-w,-w).$$
Since $\rho(-w,-w)=\rho(w,w)$ and $\rho(w,-w)=\rho(-w,w)$, we have three different cocycles: $1,\ \rho(w,w),\ \rho(w,-w)$ in $G$.
Thus
\[\Ho^1 G=\{1, [\rho(w,w)], [\rho(w,-w)]\}.\]
One can show that these three cohomology classes in $G$ are pairwise different, but we shall not need that.

We consider the embedding $G\into G\times \SL(2,\C)$.
Since $\Ho^1 \SL(2,\C)=1$, this embedding induces a bijection on $H^1$.

Write $Z_{123}=\rho(G_1\times G_2\times G_3\times 1)\cong G\times G_3$.
We have a short exact sequence
\[1\to Z_{123}\to \Zm_0\to\C^\times\to 1.\]
Since $\Ho^1 \C^\times=1$,
we see that the embedding
\[G\into Z_{123}\into \Zm_0\]
induces a surjection on $H^1$.
Thus
\[\Ho^1\hm \Zm_0=\{1, [\rho(w,w)], [\rho(w,-w)]\},\]
but maybe some of these cohomology classes coincide.

We show that $\rho(w,w)$ and $\rho(w,-w)$ are equivalent in $\Zm_0$.
Indeed, consider $b=\rho(1,1,i,i)\in \Zm_0$.
Then
\[b^{-1}\cdot\rho(w,w) \cdot \bar b:= b^{-1}\cdot\rho(w,w,1,1) \cdot \bar b=\rho(w,w,-1,-1)=\rho(w,-w,1,1)=\rho(w,-w),\]
as required.

Write $a=\rho(w,w)\in \Zm_0$.
We show that $a\not\sim 1$.
We consider the short exact sequence
\[1\to\ker\,\rho\labelto{ } \widehat{\Zm_0}\labelto{\rho}\Zm_0\to 1\]
and the coboundary
\[\Delta[a]=a\bar a =\rho(w^2,w^2,1,1)=\rho(-1,-1,1,1)\in \ker\,\rho=
\Ho^2 \ker\,\rho.\]
Since $\Delta[a]\neq 1$, we conclude that $[a]\neq 1\in \Ho^1\hm \Zm_0$.
Thus $\#\Ho^1\hm \Zm_0=2$ with cocycles  [1] and $a=\rho(w,w,1,1)$, as required.
\end{proof}

\noindent{\bf 94} Representative: $e_{167}-e_{247}-e_{356}$.  Rank 7, Gurevich number VIII.\\
In this case the primary decomposition of the defining ideal of $\Zm_0$ shows
that $\Zm_0$ is not connected.
We have that $\Zm_0^\circ$ consists of
$$ X(\SmallMatrix{ a_{22} & a_{24} \\a_{42} & a_{44}},
\SmallMatrix{ a_{33} & a_{35} \\a_{53} & a_{55}},
\SmallMatrix{ a_{88} & a_{89} \\a_{98} & a_{99}}) =
\SmallMatrix{ a_{11} & 0 & 0 & 0 & 0 & 0 & 0 & 0 & 0 \\
                0 & a_{22} & 0 & a_{24} & 0 & 0 & 0 & 0 & 0 \\
                0 & 0 & a_{33} & 0 & a_{35} & 0 & 0 & 0 & 0 \\
                0 & a_{42} & 0 & a_{44} & 0 & 0 & 0 & 0 & 0 \\
                0 & 0 & a_{53} & 0 & a_{55} & 0 & 0 & 0 & 0 \\
                0 & 0 & 0 & 0 & 0 & a_{66} & 0 & 0 & 0 \\
                0 & 0 & 0 & 0 & 0 & 0 & a_{77} & 0 & 0 \\
                0 & 0 & 0 & 0 & 0 & 0 & 0 & a_{88} & a_{89} \\
                0 & 0 & 0 & 0 & 0 & 0 & 0 & a_{98} & a_{99}} $$
with
\begin{align*}
  a_{11} &= (a_{88}a_{99}-a_{89}a_{98})^{-1}\\
  a_{66} &= (a_{33}a_{55}-a_{35}a_{53})^{-1}\\
  a_{77} &= (a_{22}a_{44}-a_{24}a_{42})^{-1},
\end{align*}
and
$$(a_{33}a_{55}-a_{35}a_{53})(a_{22}a_{44}-a_{24}a_{42})(a_{88}a_{99}-a_{89}a_{98})=1.$$
The component group is of order 2 and generated by
$$Q=\SmallMatrix{ -1 & 0 & 0 & 0 & 0 & 0 & 0 & 0 & 0 \\
                0 & 0 & 1 & 0 & 0 & 0 & 0 & 0 & 0 \\
                0 & 1 & 0 & 0 & 0 & 0 & 0 & 0 & 0 \\
                0 & 0 & 0 & 0 & 1 & 0 & 0 & 0 & 0 \\
                0 & 0 & 0 & 1 & 0 & 0 & 0 & 0 & 0 \\
                0 & 0 & 0 & 0 & 0 & 0 & 1 & 0 & 0 \\
                0 & 0 & 0 & 0 & 0 & 1 & 0 & 0 & 0 \\
                0 & 0 & 0 & 0 & 0 & 0 & 0 & 1 & 0 \\
                0 & 0 & 0 & 0 & 0 & 0 & 0 & 0 & 1}. $$
We have $Q^2=1$ and $QX(A,B,C)Q = X(B,A,C)$.

\begin{lemma*}
  $\Ho^1\hm \Zm_0 = \{[1],[Q]\}$.
\end{lemma*}

\begin{proof}
We have $\Zm_0= \Zm_0^\circ\rtimes\{1,Q\}$. Moreover,
\[\Zm_0^\circ=\ker[\GL(2,\C)\times\GL(2,\C)\times\GL(2,\C)\to \C^\times],\quad (A,B,C)\mapsto \det(A)\hs\det(B)\hs\det(C).\]
Consider the surjective homomorphism
\[\phi\colon \Zm_0^\circ\to \GL(2,\C)\times\GL(2,\C),\quad (A,B,C)\mapsto (A,B)\]
with kernel $\SL(2,\C)$.
From the short exact sequence
\[ 1\to \SL(2,\C)\to \Zm_0^\circ\labelto\phi \GL(2,\C)\times\GL(2,\C)\to 1\]
we see that $H^1 \Zm_0^\circ=1$.
From the short exact sequence
\[ 1\to \SL(2,\C)\to\hs_Q \Zm_0\labelto\phi R_{\C/\R}\GL(2,\C)\to 1\]
we see that $H^1 \hs_Q \Zm_0^\circ=1$.
Let $\pi\colon \Zm_0\to\Zm_0/\Zm_0^\circ\cong\mu_2$ denote the canonical homomorphism.
We see that
\[\Ho^1\hm \Zm_0=\ker \pi_*\,\cup\,\pi_*^{-1}\pi_*[Q]=\{1,[Q]\},\]
as required.
\end{proof}

\noindent{\bf 95} Representative: $e_{167}-e_{257}-e_{347}-e_{456}$. Rank 7, Gurevich number VII.\\
In this case we can compute the primary decomposition of the ideal
generated by the defining polynomials, and it shows that the centralizer is
connected. The Lie algebra of $\Zm_0$ is equal to $\s+\mathfrak{t}_1$,
where the second summand denotes the 1-dimensional center and $\s\simeq
\ssl(2,\C)+\ssl(3,\C)$. The natural 9-dimensional module has basis
$v_1,\ldots,v_9$. As $\s$-module it splits as direct sum of two 3-dimensional
modules (with bases $\{v_1,v_2,v_3\}$, $\{v_4,v_5,v_6\}$ respectively),
a 1-dimensional module with basis $\{v_7\}$ and a 2-dimensional module with
basis $\{v_8,v_9\}$.
Let $S$ denote the connected algebraic subgroup of $\Zm_0$ with Lie algebra
$\s$. Then the matrices in $S$ are block diagonal with first two
$3\times$-blocks, containing two dual representations of $\SL(3,\C)$,
then a $1\times 1$-block containing 1, and finally a $2\times$-block
containing the natural representation of $\SL(2,\C)$. Let $T$ be the connected
algebraic
subgroup with Lie algebra $\mathfrak{t}_1$. Then $T$ consists of
$\diag(t,t,t,1,1,1,t^{-1}, t^{-1}, t^{-1})$ for $t\in \C^\times$. Hence the
product is direct and $\Zm_0\simeq \SL(3,\C)\times \SL(2,\C)\times
T_1$, so that $\Ho^1\hm \Zm_0=1$.

\noindent{\bf 96} Representative: $-e_{247}-e_{356}$. Rank 6, Gurevich number V.\\
Here $\Zm_0^\circ$ is isomorphic to $\SL(3,\C)^3$. More precisely,
it consists of
$$ X(\SmallMatrix{a_{11} & a_{18} & a_{19}\\a_{81}&a_{88}&a_{89}\\a_{91}&a_{98}&a_{99}},
\SmallMatrix{a_{22} & a_{24} & a_{27}\\a_{42}&a_{44}&a_{47}\\a_{72}&a_{74}&a_{77}},
\SmallMatrix{a_{33} & a_{35} & a_{36}\\a_{53}&a_{55}&a_{56}\\a_{63}&a_{65}&a_{66}})=
\SmallMatrix{ a_{11} & 0 & 0 & 0 & 0 & 0 & 0 & a_{18} & a_{19} \\
                0 & a_{22} & 0 & a_{24} & 0 & 0 & a_{27} & 0 & 0 \\
                0 & 0 & a_{33} & 0 & a_{35} & a_{36} & 0 & 0 & 0 \\
                0 & a_{42} & 0 & a_{44} & 0 & 0 & a_{47} & 0 & 0 \\
                0 & 0 & a_{53} & 0 & a_{55} & a_{56} & 0 & 0 & 0 \\
                0 & 0 & a_{63} & 0 & a_{65} & a_{66} & 0 & 0 & 0 \\
                0 & a_{72} & 0 & a_{74} & 0 & 0 & a_{77} & 0 & 0 \\
                a_{81} & 0 & 0 & 0 & 0 & 0 & 0 & a_{88} & a_{89} \\
                a_{91} & 0 & 0 & 0 & 0 & 0 & 0 & a_{98} & a_{99}} $$
(with the determinants of the submatrices equal to 1). It is straightforward to
obtain this from the Lie algebra. The component group is determined with the
method of Example \ref{exa:cencomps}(2). Here
the diagram automorphism group of the Lie
algebra has order 48. For each element of that group we have added the equations
that express that an element of the centralizer permutes the Chevalley
generators in the way prescribed by the chosen outer automorphism. For most
diagram automorphisms this gave no elements in the centralizer, except for the
identity (which yielded elements already in the identity component) and
the diagram automorphism that permutes the second and third components. This
yields the element
$$ Q =\SmallMatrix{ -1 & 0 & 0 & 0 & 0 & 0 & 0 & 0 & 0 \\
                0 & 0 & 1 & 0 & 0 & 0 & 0 & 0 & 0 \\
                0 & 1 & 0 & 0 & 0 & 0 & 0 & 0 & 0 \\
                0 & 0 & 0 & 0 & 1 & 0 & 0 & 0 & 0 \\
                0 & 0 & 0 & 1 & 0 & 0 & 0 & 0 & 0 \\
                0 & 0 & 0 & 0 & 0 & 0 & 1 & 0 & 0 \\
                0 & 0 & 0 & 0 & 0 & 1 & 0 & 0 & 0 \\
                0 & 0 & 0 & 0 & 0 & 0 & 0 & -1 & 0 \\
                0 & 0 & 0 & 0 & 0 & 0 & 0 & 0 & -1} $$
that thus generates the component group.
We have $Q^2=1$ and $QX(A,B,C) Q = X(A,C,B)$.

\begin{lemma*}
$\Ho^1\hm \Zm_0=\{[1],[Q]\}$.
\end{lemma*}

\begin{proof}
Note that $Q$ is a 1-cocycle in $\Zm_0$. Let $C=\Zm_0/\Zm_0^\circ =
\{ \bar 1, \bar Q\}$. Then $\Ho^1\hm C = \{ [\bar 1], [\bar Q] \}$.
Consider the natural homomorphism $\pi\colon \Zm_0\to C$;
then
\[\Ho^1\hm  \Zm_0=\ker\pi_*\cup\pi_*^{-1}([\bar Q]).\]

Since $\Zm_0^\circ\simeq \SL(2,\C)^3$ we have $\Ho^1\hm \Zm_0^\circ=1$.
Thus $\ker\pi_*=1$.

Since $_Q \Zm_0^\circ\simeq R_{\C/\R}\SL(2,\C)\times\SL_2(\R)$ as a real algebraic
group, we have  $\Ho^1\hm \hs_Q \Zm_0^\circ=1$.
Thus by Corollary \ref{c:39-cor2}
$\# \pi_*^{-1}([\bar Q])=1$ and $\pi_*^{-1}([\bar Q])=[Q]$.
Thus $\Ho^1\hm \Zm_0=\{1,[Q]\}$, as required.
\end{proof}

\noindent{\bf 97} Representative: $-e_{267}-e_{357}-e_{456}$. Rank 6, Gurevich number IV.\\
By computing the primary decomposition of the defining ideal
we see that the centralizer is connected. It is of type $A_2+A_2+T_1$. The
centralizer consists of
block diagonal matrices, consisting of three $3\times 3$-blocks. One block
has a copy of $\SL(3,\C)$. The other copy of $\SL(3,\C)$  occupies the two
other blocks (where the corresponding $3$-dimensional representations are
isomorphic). The intersection of $T_1$ and the semisimple part is
$\mu_3$, so the product is not direct. Let $G= \SL(3,\C)\times \SL(3,\C)
\times T_1$; then $\Zm_0 \simeq G/N$, where $N$ is of order 3. Hence by
Lemma \ref{l:H1-bijective} there is a bijection between $\Ho^1\hm \Zm_0$ and
$\Ho^1\hm G$. As the latter is trivial, so is the former.

\noindent{\bf 98} Representative: $e_{167}-e_{257}-e_{347}-e_{789}$.\\
The centralizer is of type $C_4$. The matrices of the identity component
consist of a $8\times 8$-block (containing $\Sp(8,\C)$) and a $1\times 1$-block
(containing just 1). There are no diagram automorphisms. Adding the equations
expressing that an element commutes with the Lie algebra yields that the
centralizer is not connected and isomorphic to $\Sp(4,\C)\times \mu_3$.
Thus $\Ho^1\hm \Zm_0=1$.

\noindent{\bf 99} Representative: $e_{167}-e_{257}-e_{347}$. Rank 7, Gurevich number VI.\\
Here the Lie algebra of $\Zm_0$ is $\z_0=\s+\mathfrak{t}_1$, where the latter
denotes the 1-dimensional center. Furthermore, $\s$ is isomorphic to $\mathfrak{sp}(6,\C)
+\ssl(2,\C)$. Let $S$ be the connected subgroup of $\Zm_0$ with Lie algebra $\s$.
Then $S$ consists of matrices that are block
diagonal with first a $6\times 6$-block, then a $1\times 1$-block, then a
$2\times 2$-block. The $6\times 6$-block contains the 6-dimensional irreducible
representation of $\Sp(6,\C)$ with highest weight $(1,0,0)$. So in that block we
have a copy of $\Sp(6,\C)$. The $1\times 1$-block just has 1. The
$2\times 2$-block has a copy of $\SL(2,\C)$. Let $T_1$ be the connected subgroup
of $\Zm_0$ with Lie algebra $\mathfrak{t}_1$. Then $T_1$ consists of
$\diag(a,a,a,a,a,a,a^{-2},a^{-2},a^{-2})$. Its intersection with $S$
consists of 1 and $\diag(-1,-1,-1,-1,-1,-1,1,1,1)$. Here there are no
outer automorphisms, so each component has elements that commute with
the Lie algebra. Adding the corresponding linear equations we find that
the elements commuting with the Lie algebra form a diagonal group, contained in
the identity component. Hence the centralizer is connected.

It follows that
\[\Zm_0\simeq (\Sp(6,\C)\times T_1)/\mu_2\,\times \SL(2,\C).\]
We have that $\Ho^1\hm (\Sp(6,\C)\times T_1)/\mu_2=1$,
which follows from the short exact sequence
\[ 1\to \Sp(6,\C)\to\,  (\Sp(6,\C)\times T_1)/\mu_2\,\labelto{\phi}
\C^\times\to 1,\]
where $\phi((g_{ij})) = g_{77}$. Hence $\Ho^1 \Zm_0 = 1$.

\noindent{\bf 100} Representative: $-e_{367}-e_{457}$. Rank 5, Gurevich number III.\\
The Lie algebra of $\Zm_0$ is isomorphic to $\s+\mathfrak{t}_1$, where the
latter denotes the 1-dimensional center. Furthermore, $\s$ is isomorphic to
$\ssl(4,\C)+\mathfrak{sp}(6,\C)$. Let $S$ denote the connected subgroup of
$\Zm_0$ with Lie algebra $\s$. Then $S$ consists of matrices that are block
diagonal with two $4\times 4$-blocks and one $1\times 1$-block. The first
$4\times 4$-block (occupying the entries $(i,j)$ with $i,j\in \{1,2,8,9\}$)
contains a copy of $\SL(4,\C)$, the second one (occupying the entries
$(i,j)$ with $i,j\in \{3,4,5,6\}$) has a copy of $\Sp(4,\C)$.
The $1\times 1$-block (in position $(7,7)$) just has 1. Let $T_1$
denote the connected subgroup of $\Zm_0$ with Lie algebra $\mathfrak{t}_1$.
Then $T_1$ consists of $T_1(a)=\diag(a,a,a^{-2},a^{-2},a^{-2},a^{-2},a^{4},a,a)$.
Its intersection with the semisimple part consists of $T_1(a)$ with $a^4=1$.
Here there is one diagram automorphism. But adding the corresponding
equations gives that there is no element of the centralizer realizing that.
Adding the equations expressing that an elements commutes with the Lie algebra
yields a diagonal group, all of whose elements lie in $S$. It follows that
the centralizer is connected.

We have a short exact sequence
\[1\to\, \SL(4,\C)\times\Sp(4,\C)\,\labelto{ }\Zm_0\labelto{\phi}
\C^\times\to 1,\]
where $\phi((g_{ij})) = g_{77}$. It follows that $\Ho^1 \Zm_0=1$.

\noindent{\bf 101} Representative: $-e_{567}$. Rank 3, Gurevich number II.\\
The centralizer is isomorphic to $\SL(3,\C)\times \SL(6,\C)$.
(The identity component is directly obtained from the Lie algebra; computations
like in the previous cases show that the centralizer is connected.)
Hence $\Ho^1\hm \Zm_0=1$.


\section{The semisimple orbits}\label{sec:semsim}

In this section we identify the spaces $\g_1^\cC$ and $\bigwedge^3 \C^9$ on which
the group $\Gtil_0=\SL(9,\C)$ acts.
Similarly we identify $\g_1$ and $\bigwedge^3 \R^9$ on which the group
$\SL(9,\R)$ acts.

\subsection{Summary of some results of Vinberg-Elashvili}\label{sec:sumVE}

In \cite{VE1978} it is shown that

\begin{align*}
  p_1 &= e_{123}+e_{456}+e_{789}\\
  p_2 &= e_{147}+e_{258}+e_{369}\\
  p_3 &= e_{159}+e_{267}+e_{348}\\
  p_4 &= e_{168}+e_{249}+e_{357}
\end{align*}
span a Cartan subspace in $\g_1^\cC$. Throughout this section we will denote
this Cartan subspace by $\Cg$. As in Section \ref{subs:weyl} we denote the
corresponding Weyl group by $W$. Also for $p\in \Cg$ we define
\begin{align*}
  \Cg_p &= \{ h\in \Cg \mid wh=h \text{ for all } w\in W_p\}\\
  \Cg_p^\circ & = \{ q\in \Cg_p \mid W_q = W_p \},
\end{align*}
see  (\ref{eq:hp}), (\ref{eq:h0p}).

Let $\h$ be a Cartan subalgebra of $\g^\cC$ containing $\Cg$. It turns out
that there is a unique such Cartan subalgebra and $\h = (\h\cap \g_{-1}^\cC)
\oplus \Cg$. Let $\Phi$ be the root system of $\g^\cC$ with respect to
$\h$. Let $\theta$ denote the automorphism of $\g^\cC$ corresponding to
the given $\Z_3$-grading. Then $\h$ is $\theta$-stable and we define a
map $\theta_* : \h^*\to \h^*$, $\theta_*\gamma (h) = \gamma(\theta^{-1}h)$
for $\gamma$ in the dual space $\h^*$. We have $\theta_*^2+\theta_*+1=0$
so that for $\alpha\in \Phi$ the set $\Phi(\alpha) = \{ \pm\alpha,
\pm \theta_*(\alpha), \pm \theta_*^2(\alpha)\}$ lies in a 2-dimensional
space and forms a root subsystem of type $A_2$. The $\Phi(\alpha)$ form
a partition of $\Phi$, hence there are 40 such sets.
Note that for $h\in \Cg$ and $\alpha\in \Phi$ we have that $\theta_*(\alpha)(h)
= \zeta^2 \alpha(h)$ (where $\zeta$ is a third root of unity such that
$\theta(x) = \zeta x$ for all $x\in \g_1^\cC$). Hence the restrictions of
the elements of $\Phi(\alpha)$ to $\Cg$ are proportional.

To each $\Phi(\alpha)$ a complex reflection $w_\alpha$ (of order 3) is
associated, so we have 40 of those reflections.
For the details we refer to \cite{VE1978}, where it is also
shown that there exist $p_\alpha\in \Cg$ such that
\begin{equation}\label{eq:refl}
  w_\alpha(h) = h -\alpha(h) p_\alpha \text{ for } h\in \Cg.
\end{equation}

Furthermore, these are all complex reflections in $W$, and they form a
single conjugacy class.

Let $p\in \Cg$ and let $W_p$ denote the stabilizer of $p$ in $W$. It is known that
$W_p$ is generated by complex reflections (\cite[Proposition 14]{Vinberg1976})
so from \eqref{eq:refl} it follows that $W_p$ is generated by the $w_\alpha$
such that $\alpha(p)=0$.

As in Section \ref{subs:weyl} we define
$$\Cg_p^{\mathrm{reg}} = \{ q\in \Cg_p \mid \alpha(q)\neq 0 \text{ for all }
\alpha\in\Phi \text{ such that } \alpha |_{\Cg_p} \neq 0\}.$$
Let $\alpha\in \Phi$. Then $\alpha(p)=0$ if and only if $w_\alpha(p)=p$
if and only if $w_\alpha(q)=q$ for all $q\in \Cg_p$ if and only if
$\alpha(q)=0$ for all $q\in \Cg_p$. It follows that $\Cg_p^{\mathrm{reg}} =
\Cg_p^\circ$, which is a hypothesis for some of the results in Section
\ref{subs:weyl}. Hence all of these can be applied here.

\subsection{Using Galois cohomology for finding real orbits in semisimple trivectors}

Here we fix a $p\in \Cg$ and write $\Fm = \Cg_p^\circ$.
Let $\OOm=\Gtil_0\cdot p$ denote the $\Gtil_0$-orbit of $p$.
We wish to know whether $\OOm$ contains an $\R$-point
and whether $\OOm\cap\Fm$ contains an $\R$-point.
If $\OOm$ has a real point, we wish to classify real orbits in $\OOm$.

Define

\begin{align*}
  \Zm_{\Gtil_0}(\Fm) & = \{ g\in \Gtil_0 \mid gq=q \text{ for all }
  q\in \Fm\},\\
  \Nm_{\Gtil_0}(\Fm) & = \{ g\in \Gtil_0 \mid gq\in \Fm \text{ for all }
  q\in \Fm\}.
\end{align*}

As in Section \ref{subs:weyl} we set $\Gamma_p = \Nm_W(W_p)/W_p$.
Then from Lemma \ref{lem:Np} we get a surjective group homomorphism
$\varphi : \Nm_{\Gtil_0}(\Fm) \to \Gamma_p$ with kernel $\Zm_{\Gtil_0}(\Fm)$.
So $\varphi$ induces an isomorphism, which we also denote
$\varphi$, between $\Am =
\Nm_{\Gtil_0}(\Fm)/\Zm_{\Gtil_0}(\Fm)$ and $\Gamma_p$. Furthermore
for $g\in \Nm_{\Gtil_0}(\Fm)$ and $q\in \Fm$ we have that $gq = \varphi(g)q$.

\begin{proposition}\label{prop:ssorb}
  As before let $\OOm = \Gtil_0\cdot p$. Write $\Nm =  \Nm_{\Gtil_0}(\Fm)$,
  $\Zm=\Zm_{\Gtil_0}(\Fm)$.
\begin{enumerate}
  \item[\rm (i)]  $\OOm$ has an $\R$-point if and only if
  $\pbar=n^{-1}\cdot p$ for some $n\in \Zl^1\Nm$\hs.
  \item[\rm (ii)] Assume that $\OOm$ has an $\R$-point and let $n$ be as in (i).
  Write $a=n\Zm\in\Zl^1\Am$ and $\xi=[a]\in\Ho^1\Am$\hs.
  Then $\OOm$ has an $\R$-point in $\Fm$ if and only if $\xi=1$.
\end{enumerate}
\end{proposition}

\begin{proof}
(i) Assume that $\OOm$ has an $\R$-point $p_\R=g\cdot p$.
Then $\ov{g\cdot p}=g\cdot p$, whence
\[ \pbar=\gbar^{-1}\cdot g\cdot p=(g^{-1}\gbar)^{-1}\cdot p.\]
Write
\begin{equation}\label{e:n-g-gbar}
n= g^{-1}\gbar.
\end{equation}
Since $p,\pbar\in\Fm$\hs, we see that $n\in\Nm$ by Lemma \ref{lem:gp1p2}.
It follows from \eqref{e:n-g-gbar} that $n\in \Zl^1\Nm$\hs.

Conversely, assume that $\pbar=n^{-1}\cdot p$, where $n\in\Zl^1\Nm$.
Since $\Ho^1\Gtil_0=\{1\}$, there exists $g\in \Gtil_0$ such that
$n=g^{-1}\gbar$. Set
\[ p_\R=g\cdot p\in\OOm.\]
Then
\[\ov{p_\R}=\gbar\cdot\pbar=gn\cdot n^{-1}\cdot p = g\cdot p=p_\R\hs.\]
Thus $p_\R$ is an $\R$-point of $\OOm$, which proves (i).

(ii) Assume that $\OOm$ has an $\R$-point and let $n$ be as in (i).
We show that $\xi$ depends only on $\OOm$.
First suppose that $\pbar=\hat n^{-1}\cdot p$ for a $\hat n\in \Nm$.
Then $\hat n n^{-1} \in \Zm_{\Gtil_0}(p)$, which is equal to $\Zm$ by Lemma
\ref{lem:Zp}. Hence $p$ uniquely determines $a$ and $\xi$.

We show that $\xi$ does not depend on the choice of $p$.
Indeed, if $p'\in\OOm\cap\Fm$\hs,
then $p'=a'\cdot p$ for some $a'\in\Am$\hs,
and we have
\[\ov{p'}=\ov{a'}\cdot\pbar=\ov{a'}\cdot a^{-1}\cdot p
   =\ov{a'}\cdot a^{-1}(a')^{-1}\cdot a'\cdot p
   =(a'a\hs \ov{a'}^{\hs -1})^{-1}\cdot p'.\]
We obtain the 1-cocycle
\[a'a\hs\ov{a'}^{\hs -1}\sim a.\]
Thus $\xi=[a]$ does not depend on the choice of $p$.
We write $\xi(\OOm)$ for $\xi$.

Let $p\in \OOm\cap\Fm$\hs, and assume that $\xi(\OOm)=1$.
We have
\[\pbar=a^{-1}\cdot p\quad\text{and} \quad a=(a')^{-1}\hs \ov{a'}
   \text{ for some }a'\in\Am\hs.\]
Set $p_\R=a'\cdot p\in \OOm\cap\Fm$\hs.
Then
\[\ov{p_\R}=\ov{a'}\cdot\pbar=\ov{a'}\cdot a^{-1}\cdot p=\ov{a'}
     \cdot\ov{a'}^{\hs -1} \cdot a'\cdot p=a'\cdot p=p_\R\hs.\]
Thus $p_\R$ is real.

Conversely, if $\OOm\cap \Fm$ contains a real point $p_\R$\hs,
then clearly $\xi(\OOm)=1$, which proves (ii).
\end{proof}

\begin{subsec}\label{subsec:realpoints2}
Assume that $\OOm$ contains a real point $p_\R=g\cdot p$.
We wish to classify real orbits in $\OOm$.
We write $q$ for $p_\R$.
Write $C_q=\Zm_{\Gtil_0}(q)$ and set
\[\CC_q=(C_q,\sigma_q),\quad\text{where }\sigma_q(c)=\ov c.\]

Similarly we write $C_p=\Zm_{\Gtil_0}(p)$.
Since $q=g\cdot p$, we have an isomorphism
\[ \iota_g\colon C_p\labelto{\sim} C_q\hs, \quad c\mapsto gcg^{-1}.\]
We {\em transfer} the real structure $\sigma_q$ on $C_q$ to $C_p$ using $\iota_g$.
We obtain a real structure $\sigma_p$ on $C_p$:
\begin{align*}
\sigma_p\colon\, C_p\labelto{\iota_g} C_q\labelto{\sigma _q} C_q\labelto{\iota_g^{-1}} C_p\hs,\quad
c\mapsto gcg^{-1}\mapsto\ov{gcg^{-1}}\mapsto g^{-1}\hs\ov{gcg^{-1}} g.
\end{align*}
Let $n=g^{-1}\ov g$ (see also the proof of Proposition \ref{prop:ssorb}), then
$\sigma_p(c)=n\ov c n^{-1}$ for $c\in C_p$.
We obtain a real algebraic group
\[ \CC_p=(C_p,\sigma_p),\quad \text{where } \sigma_p(c)=n\ov c n^{-1}\ \text{ for }c\in C_p\hs,\]
and an isomorphism
\[\iota_g\colon \CC_p\labelto{\sim}\CC_q,\quad c\mapsto gcg^{-1}\]
inducing a bijection on cohomology
\[\Ho^1\CC_p\labelto{\sim}\Ho^1\CC_q\hs.\]

By Proposition \ref{p:coh-orbits}, the real orbits in $\OOm$
are classified by $\Ho^1\CC_q$\hs, and hence by $\Ho^1\CC_p$\hs.
The map is as follows.
To $c\in\Zl^1\CC_p$ we associate $gcg^{-1}\in \Zl^1\CC_q$.
We find  $g_1\in \Gtil_0$ such that $g_1^{-1}\hs\ov g_1=gcg^{-1}$ (see Section
\ref{sec:comprep}) and set $r=g_1\cdot q =g_1\hs g\cdot p$.
To $[c]$ we associate the real orbit $\Gtil_0(\R)\cdot r\subseteq \OOm$.

In order to check our calculations, we show that $r=g_1\hs g\cdot q$ is real.
We calculate:
\begin{align*}
\ov r =\ov{g_1\hs g\cdot p}=\ov g_1\cdot\gbar\cdot \pbar= g_1\hs gcg^{-1}\cdot gn\cdot n^{-1}\cdot p
=g_1\hs gc\cdot p=g_1\hs g\cdot p=r,
\end{align*}
because $c\in C_p=\Zm_{\Gtil_0}(p)$.
Thus $r$ is real.
\end{subsec}

\begin{subsec}\label{subsec:semsimgal}
{\em Algorithm of finding real orbits in $G\cdot \Fm$.}

We compute $\Ho^1\hm\Am$.
For any cohomology class $\xi=[a]\in\Ho^1\hm\Am$\hs,
we find all $p\in \Fm$ such that $\pbar=a^{-1}\cdot p$.
Further, we try to lift our $a\in\Zl^1\hm\Am$ to a {\em cocycle}
$n\in\Zl^1\Nm_{\Gtil_0}(\Fm)$ (this turns out to be always possible in the
cases that we dealt with).
Then we find $g\in \Gtil_0$ such that $g^{-1} \gbar=n$.
We set $p_\R=g\cdot p$; then $p_\R$ is real.
Thus we have a real point $q=p_\R=g\cdot p$ in the complex orbit
$\OOm=\Gtil_0\cdot p$.

In order to find {\em all} real orbits in $\OOm$, we consider the stabilizer
$C_p$ of $p$ in $\Gtil_0$, and the $n$-twisted complex conjugation $\sigma_p$ of $C_p$
given by $c\mapsto n\ov c n^{-1}$ and set $\CC_p = (C_p,\sigma_p)$
as in \ref{subsec:realpoints2}.
We compute the Galois cohomology $\Ho^1\CC_p$. For a cocycle
$c\in\Zl^1\CC_p$, we consider $gcg^{-1} \in C_q$ and we find $g_1\in \Gtil_0$
such that
\[g_1^{-1}\hs\ov g_1=gcg^{-1}.\]
We set
\begin{equation}\label{e:r=hgp}
r=g_1\cdot q=g_1\hs g\cdot p.
\end{equation}
To a cohomology class $[c]\in\Ho^1\CC_p$ we associate
the real orbit $\Gtil_0(\R)\cdot r\subseteq \OOm$, where $r$ is as in \eqref{e:r=hgp},
and in this way we obtain a bijection between $\Ho^1\CC_p$
and the set of real orbits in $\OOm$.
\end{subsec}

\subsection{Centralizers of semisimple elements}\label{subsec:semcent}

As written in Section \ref{sec:sumVE} there are seven canonical sets
$\Fm_k=\Cg_{p_k}^\circ$ for $1\leq k\leq 7$. They are explicitly described
in \cite{VE1978}. The elements of the same
canonical set have the same centralizer in $\Gtil_0$ by Lemma \ref{lem:Zp}.
Here we describe those centralizers. Throughout we write $\Zm_0(p)$ instead
of $\Zm_{\Gtil_0}(p)$.

The seventh canonical set just consists of 0, so we do not say anything about
it.

\begin{subsec}\label{subsec:cen6}
  $\Fm_6$ consists of the elements $\lambda p_1$
with $\lambda\neq 0$. We can easily compute the Lie algebra $\z_0(p_1)$ of
$\Zm_0(p_1)$, from which it follows that $\Zm_0(p_1)^\circ$ is isomorphic to
$\SL(3,\C)^3$. Its elements are of the form $\diag(A_1,A_2,A_3)$, where each
$A_i$ is a $3\times 3$-matrix with determinant 1. With the methods of
Example \ref{exa:cencomps}(2) we determined the component group. It is
generated by (the image of)
$$h_1=\SmallMatrix{ 0 & 0 & 0 & 0 & 0 & 0 & 1 & 0 & 0 \\
                0 & 0 & 0 & 0 & 0 & 0 & 0 & 1 & 0 \\
                0 & 0 & 0 & 0 & 0 & 0 & 0 & 0 & 1 \\
                1 & 0 & 0 & 0 & 0 & 0 & 0 & 0 & 0 \\
                0 & 1 & 0 & 0 & 0 & 0 & 0 & 0 & 0 \\
                0 & 0 & 1 & 0 & 0 & 0 & 0 & 0 & 0 \\
                0 & 0 & 0 & 1 & 0 & 0 & 0 & 0 & 0 \\
                0 & 0 & 0 & 0 & 1 & 0 & 0 & 0 & 0 \\
                0 & 0 & 0 & 0 & 0 & 1 & 0 & 0 & 0}. $$
We have $h_1^3=1$ and conjugation by $h_1$ permutes the blocks on the diagonal
of the elements of $\Zm_0(p_1)^\circ$ cyclically.
So $\Zm_0(p_1) = \langle h_1\rangle \ltimes \SL(3,\C)^3$.
\end{subsec}

\begin{subsec}\label{subsec:cen5}
  $\Fm_5$ consists of $\lambda(p_3-p_4)$, $\lambda\neq 0$.
From the computation of the Lie algebra it follows that $\Zm_0(p_3-p_4)^\circ$ is
equal to $H\cdot T_2$ where $H$ is the group consisting of $\diag(A_1,A_1,A_1)$
where $A_1$ is a $3\times 3$-matrix with determinant 1. Furthermore $T_2$
is a 2-dimensional central torus consisting of $\diag(a,a,a,b,b,b,(ab)^{-1},
(ab)^{-1},(ab)^{-1})$ for $a,b\in \C^*$. So the intersection of $T_2$ and $H$
is $\mu_3$. The component group of $\Zm_0(p_3-p_4)$ is computed in the same way
as for family six, by adding the equations corresponding to the diagram
automorphisms of the semisimple part. The component group is again
generated by the image of $h_1$. Hence
$\Zm_0(p_3-p_4) = \langle h_1\rangle \ltimes H\cdot T_2$.
\end{subsec}

\begin{subsec} $\Fm_4$ consists of $\lambda p_1 +\mu(p_3-p_4)$ with
$\lambda\mu(\lambda^3-\mu^3)(\lambda^3+8\mu^3)\neq 0$. Here the stabilizer
can be computed as in the previous two cases, or by observing that
from Lemma \ref{lem:Zp} it follows that the stabilizer in this case
is $\Zm_0(p_1) \cap \Zm_0(p_3-p_4)$. It follows that the identity component of the
stabilizer is $H$ (the same one as in Family five).
The component group is generated by cosets
of $h_1$ and $h_2=\diag(\zeta,\zeta,\zeta,\zeta^2,\zeta^2,\zeta^2,1,1,1)$,
where $\zeta\in \C$ is a primitive third root of unity. We have that
$h_1$, $h_2$ commute with $H$, and they commute with each other modulo
$\mu_3\subset H$. Hence the component group is isomorphic to $C_3\times C_3$.
\end{subsec}

\begin{subsec} $\Fm_3$ consists of $\lambda_1 p_1 + \lambda_2 p_2$ with
$\lambda_1\lambda_2 (\lambda_1^6-\lambda_2^6)\neq 0$.
By Lemma  \ref{lem:Zp} the stabilizer is equal to $\Zm_0(p_1)\cap \Zm_0(p_2)$.
The first of these already
has been determined. It is possible to determine $\Zm_0(p_2)$ in the same
way (in fact, it is a conjugate of $\Zm_0(p_1)$, the conjugating matrix can be
determined directly from the expressions for $p_1$, $p_2$). Then we can
determine the intersection; this poses no fundamental difficulty. Alternatively
we can proceed as in Example \ref{exa:cencomps}(3).

From the Lie algebra of $\Zm_0(\Fm_3)$ we conclude that the identity component
is a 4-dimensional torus consisting of the elements
$$T_4(t_1,t_2,t_3,t_4)=
\diag(t_1,t_2,(t_1t_2)^{-1},t_3,t_4,(t_3t_4)^{-1},(t_1t_3)^{-1},(t_2t_4)^{-1},
t_1t_2t_3t_4),$$
for $t_1,t_2,t_3,t_4\in \C^*$. It turns out that the
component group is generated by the cosets of  $h_1$ and
$h_3=\diag(N,N,N)$, where
$$N=\SmallMatrix{ 0 & 0 & 1\\ 1 & 0 & 0 \\ 0 & 1 & 0}.$$
Hence $\Zm_0(p) = \langle h_1, h_3 \rangle \ltimes T_4$.
Here $h_1$, $h_3$ commute and are
of order 3. They generate a group isomorphic to $C_3\times C_3$. An element
of $T_4$ naturally composes in three blocks of three. Conjugation by $h_1$
permutes these blocks cyclically. Conjugation by $h_3$ permutes the elements of
each block cyclically.
\end{subsec}

\begin{subsec} $\Fm_2$ consists of $\lambda_1 p_1+\lambda_2 p_2 -\lambda_3p_3$,
such that $F(\lambda_1,\lambda_2,\lambda_3)\neq 0$, where $F$ is a polynomial
that we do not repeat here. Write $\mathcal{B}_{12} = \Zm_0(p_1)\cap \Zm_0(p_2)$ and
similarly $\mathcal{B}_{123} = \Zm_0(p_1)\cap \Zm_0(p_2)\cap \Zm_0(p_3)$. We have determined
$\mathcal{B}_{12}$ in the previous case, and now are interested in $\mathcal{B}_{123}$.

From the Lie algebra of $\mathcal{B}_{123}$ we see that $\mathcal{B}_{123}^\circ$ consists of
$$\diag(a_1,a_2,(a_1a_2)^{-1},a_2,(a_1a_2)^{-1},a_1,(a_1a_2)^{-1},a_1,a_2)$$
for $a_i\in \C^*$.

For a subset $S\subset \SL(9,\C)$ write $C(S,p_3) = \{ g\in S\mid g\cdot
p_3 = p_3\}$, then
$$\mathcal{B}_{123} = C(\mathcal{B}_{12},p_3) =\bigcup_{i,j=0}^2 C( h_1^ih_3^j \mathcal{B}_{12}^\circ, p_3).
$$

It is straightforward to see that $C(\mathcal{B}_{12}^\circ,p_3)$ is normal in $\mathcal{B}_{123}$
and that the $C( h_1^ih_3^j \mathcal{B}_{12}^\circ, p_3)$ are its cosets. Moreover, if
$g_1\in C( h_1 \mathcal{B}_{12}^\circ, p_3)$, $g_2\in C( h_3 \mathcal{B}_{12}^\circ, p_3)$ then
$C( h_1^ih_3^j \mathcal{B}_{12}^\circ, p_3) = g_1^ig_2^j C(\mathcal{B}_{12}^\circ,p_3)$.

By Gr\"obner basis computations we can determine the sets
$C( h_1^ih_3^j \mathcal{B}_{12}^\circ, p_3)$. It turns out that
$C( \mathcal{B}_{12}^\circ, p_3) = \langle h_2 \rangle \mathcal{B}_{123}^\circ$.
Furthermore, $C( h_1\mathcal{B}_{12}^\circ, p_3)$ contains the element
$$g_1=\SmallMatrix{ 0 & 0 & 0 & 0 & 0 & 0 & 1 & 0 & 0 \\
                0 & 0 & 0 & 0 & 0 & 0 & 0 & 1 & 0 \\
                0 & 0 & 0 & 0 & 0 & 0 & 0 & 0 & 1 \\
                \zeta^2 & 0 & 0 & 0 & 0 & 0 & 0 & 0 & 0 \\
                0 & \zeta^2 & 0 & 0 & 0 & 0 & 0 & 0 & 0 \\
                0 & 0 & \zeta^2 & 0 & 0 & 0 & 0 & 0 & 0 \\
                0 & 0 & 0 & \zeta & 0 & 0 & 0 & 0 & 0 \\
                0 & 0 & 0 & 0 & \zeta & 0 & 0 & 0 & 0 \\
                0 & 0 & 0 & 0 & 0 & \zeta & 0 & 0 & 0}, $$
and $C( h_3\mathcal{B}_{12}^\circ, p_3)$ contains the element
$g_2 = \diag(A_1,A_2,A_3)$ where $A_1 = \zeta^2 N$, $A_2=\zeta N$, $A_3=N$,
and $N$ is as in Family three.
The elements $h_2,g_1,g_2$ have order 3 and commute modulo $\mu_3\subset
\mathcal{B}_{123}^\circ$. Hence the component group has order 27 and consists of the
cosets of $h_2^ig_1^jg_2^k$ for $i,j,k =0,1,2$.

We have $h_2g_1 = h_1 \cdot \diag(\zeta,\ldots,\zeta)$ and $h_2g_2 = h_3$.
Hence
$$\mathcal{B}_{123} = \bigcup_{i,j,k=0}^2 h_1^ih_2^jh_3^k \mathcal{B}_{123}^\circ.$$
The $h_i$, $1\leq i\leq 3$ commute modulo $\mu_3\subset \mathcal{B}_{123}^\circ$.
Hence the component group is isomorphic to $C_3^3$.
\end{subsec}

\begin{subsec}\label{subsec:F1}
  $\Fm_1$ consists of $\lambda_1p_1+\cdots +\lambda_4p_4$ where the
$\lambda_i\in \C$ satisfy $F_i(\lambda_1,\ldots,\lambda_4)\neq 0$, where the
$F_i$, $1\leq i\leq 5$  are polynomials in four indeterminates that we do not
repeat here. Set $\mathcal{B}_{1234} = \Zm_0(p_1)\cap \cdots \cap \Zm_0(p_4)$. Then
$\mathcal{B}_{1234}$ is a finite group.
It can be determined with the same strategy as in the previous case.
Alternatively, it is not difficult with Gr\"obner bases to determine the sets
$\Zm_0(p_3)\cap \Zm_0(p_4) \cap h_1^ih_3^j \mathcal{B}_{12}^\circ$; each has 27 elements.
(Note that the union of $h_1^ih_3^j \mathcal{B}_{12}^\circ$ is $\Zm_0(p_1)\cap \Zm_0(p_2)$,
see the computation of the centralizer for Family three.)
We have used both approaches and in both cases obtained the same group.
It is generated by $h_1,h_2,h_3,h_4,h_5$ where $h_4=
\diag(\zeta,\zeta^2,1,\zeta,\zeta^2,1,\zeta,\zeta^2,1)$ and $h_5 =
\diag(\zeta,\ldots,\zeta)$. The group is of order $3^5$; $\mu_3 = \langle
h_5\rangle$ is a normal subgroup of order 3. The other generators commute
modulo $\mu_3$ and have order 3. Hence the quotient group is isomorphic to
$C_3^4$. (In \cite{VE1978} it is stated that $\mathcal{B}_{1234}$ is cyclic
of order 81.)
\end{subsec}

\subsection{The Weyl group and the groups $\Gamma_p$}\label{subsec:cartsp}

In \cite{VE1978} a construction is given of a set of generators of the Weyl
group $W$. We have carried it out and found that $W$ is generated by
$$ w_1 = \SmallMatrix{ \zeta & 0 & 0 & 0 \\0 & 1 & 0 & 0 \\
  0 & 0 & 1 & 0 \\ 0 & 0 & 0 & 1},
 w_2 = \SmallMatrix{ 0 & -1 & 0 & 0 \\0 & 0 & 0 & 1 \\
  -1 & 0 & 0 & 0 \\ 0 & 0 & -1 & 0},
 w_3 = \SmallMatrix{
  \tfrac{\zeta+2}{3} & \tfrac{2\zeta+1}{3} & 0 & \tfrac{\zeta+2}{3}\\
  -\tfrac{\zeta+2}{3} & \tfrac{\zeta+2}{3} & 0 & \tfrac{2\zeta+1}{3}\\
  0 & 0 & 1 & 0\\
  -\tfrac{2\zeta+1}{3} & -\tfrac{\zeta+2}{3} & 0 & \tfrac{\zeta+2}{3}},$$
where $\zeta$ is a primitive third root of unity.
(These matrices are given with respect to the basis $p_1,\ldots,p_4$ of
$\Cg$. The third generator is a reflection of the form $w_\alpha$, see
Section \ref{sec:sumVE}.)
We have also found elements $\hat w_i\in \Nm_0$ inducing the
$w_i$. They are
$$\hat w_1 = \diag(\omega,\omega,\omega,\omega,\omega,\omega,\omega^7,\omega^7,
\omega^7),$$
$$\hat w_2=\SmallMatrix{ -1 & 0 & 0 & 0 & 0 & 0 & 0 & 0 & 0 \\
                0 & 0 & 0 & -1 & 0 & 0 & 0 & 0 & 0 \\
                0 & 0 & 0 & 0 & 0 & 0 & -1 & 0 & 0 \\
                0 & 0 & 0 & 0 & 0 & 0 & 0 & -1 & 0 \\
                0 &-1 & 0 & 0 & 0 & 0 & 0 & 0 & 0 \\
                0 & 0 & 0 & 0 & -1 & 0 & 0 & 0 & 0 \\
                0 & 0 & 0 & 0 & 0 & -1 & 0 & 0 & 0 \\
                0 & 0 & 0 & 0 & 0 & 0 & 0 & 0 & -1\\
                0 & 0 & -1& 0 & 0 & 0 & 0 & 0 & 0}, $$
$$\hat w_3=\frac{1}{3}\SmallMatrix{
  -\omega^4+\omega & 0 & 0 & 0 & -\omega^4-2\omega &
  0 & 0 & 0 & -\omega^4-2\omega \\
  0 & -\omega^4+\omega & 0 & 0 & 0 & -\omega^4+\omega &
  2\omega^4+\omega & 0 & 0 \\
  0 & 0 & -\omega^4+\omega &
  2\omega^4+\omega & 0 & 0 & 0 & -\omega^4+\omega & 0 \\
0 & 0 & -\omega^4+\omega & -\omega^4+\omega & 0 & 0 & 0 &
2\omega^4+\omega & 0 \\
-\omega^4-2\omega & 0 & 0 & 0 & -\omega^4+\omega & 0 & 0 & 0 &
-\omega^4-2\omega\\
0 & 2\omega^4+\omega & 0 & 0 & 0 & -\omega^4+\omega & -\omega^4+\omega &
0 & 0 \\
0 & -\omega^4+\omega & 0 & 0 & 0 & 2\omega^4+\omega & -\omega^4+\omega &
0 & 0 \\
0 & 0 & 2\omega^4+\omega & -\omega^4+\omega & 0 & 0 & 0 & -\omega^4+\omega &
0 \\
-\omega^4-2\omega & 0 & 0 & 0 & -\omega^4-2\omega & 0 & 0 & 0 &
-\omega^4+\omega}, $$
where $\omega$ is a primitive $9$-th root of unity such that $\omega^3=\zeta$.
Using these we can find elements of $\Nm_{\Gtil_0}(\Cg)$ mapping to any given
element of $W$.

Now for each canonical set $\Fm_k$ $1\leq k\leq 6$
we give generators of the group $\Gamma_p=N(W_p)/W_p$. By searching through
subgroups of $N(W_p)$ (with the help of {\sf GAP}4) in each case we found coset
representatives of $W_p$ in $N(W_p)$ that form a subgroup of $N(W_p)$.
Hence we give the generators of $\Gamma_p$ as elements of $W$. We also
determined the Galois cohomology sets $\Ho^1\Gamma_p$ using a brute force
computer program. We also give the output of that.

\begin{itemize}
\item[$\Fm_1$] In this case $\Gamma_p=W$. We have $\Ho^1 \Gamma_p = \{[1]\}$.
\item[$\Fm_2$] Here $\Gamma_p$ is generated by
  $$a_1 = \diag(-1,-1,-\zeta,\zeta+1)=w_2w_1w_2w_1^{-1}w_2^2,$$
  $$a_2 = \frac{1}{3}\SmallMatrix{-\zeta+1 & -\zeta+1 & \zeta+2 & 0 \\
  2\zeta+1 & -\zeta-2 & \zeta+2 & 0\\
  \zeta+2 & -2\zeta-1 & -\zeta-2 & 0\\
  0 & 0 & 0 & -3\zeta}=w_3w_1^{-1}w_3(w_2^{-1}w_1^{-1})^2w_3w_2^{-1}.$$
  We have $\Ho^1 \Gamma_p = \{ [1], [-1] \}$.
\item[$\Fm_3$]   Here $\Gamma_p$ is generated by
  $$b_1 = \SmallMatrix{0 & -1 & 0 & 0\\ -1 & 0 & 0 & 0\\
  0 & 0 & -1 & 0\\ 0 & 0 & 0 & 1} = w_1^{-1}w_2(w_3^{-1}w_1)^2w_2^{-1}w_1,$$
  $$b_2 = \SmallMatrix{ -1 & 0 & 0 & 0\\ 0 & \zeta & 0 & 0\\
  0 & 0 & 0 & 1\\ 0 & 0 & 1 & 0}=w_2w_1w_2w_1^{-1}w_2w_3w_1^{-1}w_3w_2.$$
  We have $\Ho^1 \Gamma_p = \{ [1], [-1], [u_1], [u_2] \}$ where
  $$u_1 = \SmallMatrix{ -1 & 0 & 0 & 0 \\  0 & 1 & 0 & 0 \\ 0 & 0 & 0 & 1 \\
    0 & 0 & 1 & 0 }, u_2 = \SmallMatrix{ 0 & 1 & 0 & 0 \\ 1 & 0 & 0 & 0 \\
    0 & 0 & 1 & 0 \\ 0 & 0 & 0 & -1}.$$
\item[$\Fm_4$]  Here $\Gamma_p$ is generated by
  $$c_1 = \diag(-1,-1,-\zeta,-\zeta)=(w_2w_1)^2w_2^2,$$
  $$c_2 = \frac{1}{3}\SmallMatrix{ -\zeta + 1 & 0 & -\zeta + 1  & \zeta - 1\\
       0  & -\zeta + 1  & -\zeta + 1 &  -\zeta + 1\\
 2\zeta + 1 &  2\zeta + 1 & -2\zeta - 1  &  0\\
 -2\zeta - 1 &  2\zeta + 1 & 0 & -2\zeta - 1}=w_3^{-1}w_1^{-1}w_2^2w_3^{-1}w_1w_3w_2^{-1}w_1^{-1}.$$
  We have $\Ho^1\Gamma_p = 1$.
\item[$\Fm_5$] Here $\Gamma_p$ is generated by the element $c_1$ of the
  previous case. We have $\Ho^1\Gamma_p = \{ [1],[-1] \}$.
\item[$\Fm_6$]   Here $\Gamma_p$ is generated by
  $$d= \diag(-\zeta,-1,-1,-1)=w_2^4w_1.$$
   We have $\Ho^1\Gamma_p = \{ [1],[-1] \}$.
\end{itemize}

\begin{theorem}\label{thm:cartanr}
All Cartan subspaces in $\g_1$ are conjugate under $\SL(9,\R)$.
\end{theorem}

\begin{proof}
The short exact sequence
\[1\to\Zm_{\Gtil_0}(\Cg)\to \Nm_{\Gtil_0}(\Cg)\to W\to 1,\]
gives rise to a cohomology exact sequence
\[ \Ho^1\Zm_{\Gtil_0}(\Cg) \to \Ho^1\Nm_{\Gtil_0}(\Cg)\to \Ho^1 W.\]
By brute force (computer) computations it is easily established that
$\Ho^1 W = 1$. As seen in \ref{subsec:F1} $\Zm_{\Gtil_0}(\Cg)$ is a
group of order $3^5$. It has a normal subgroup of order $3$ and the quotient
is abelian. Hence by Corollary \ref{c:2m+1} and Lemma \ref{l:H1-bijective}
we see that $\Ho^1 \Zm_{\Gtil_0}(\Cg) = 1$.
It follows that $\Ho^1\Nm_{\Gtil_0}(\Cg)=1$, and the theorem follows from Theorem
\ref{thm:cartan}.
\end{proof}

\subsection{Classification of semisimple elements}\label{sec:semsimclas}

We have seven canonical sets $\Fm_k$, $1\leq k\leq 7$,
in $\Cg$ such that each semisimple element
of $\g_1^\cC$ is conjugate to an element of precisely one canonical set.
In this subsection we obtain the $\Gtil_0(\R)$-orbits of the real points
in each complex semisimple orbit using the method of \ref{subsec:semsimgal}.

\begin{subsec}
Let $p\in \Fm_1$ and set $\OOm = \Gtil_0\cdot p$.

We have that $\Gamma_p=W$ and $\Ho^1 \Gamma_p=1$. So by Proposition
\ref{prop:ssorb} $\OOm$ has an $\R$-point if and only if it has one in
$\Fm_1$. Alternatively, this follows from Proposition \ref{prop:sconj}(2) and
Theorem \ref{thm:cartanr}.
So assume that $\ov p =p$. Then the $\Gtil_0(\R)$-orbits
in $\OOm$ are classified by $\Ho^1 \Zm_{\Gtil_0}(p)$.
We have that $Z_{\Gtil_0}(p)$ is a finite group of order $3^5$.
It has a normal subgroup of order $3$ and the quotient
is abelian. Hence by Corollary \ref{c:2m+1} and Lemma \ref{l:H1-bijective}
we see that $\Ho^1 \Zm_{\Gtil_0}(p) = 1$. We conclude
that $\OOm$ corresponds to one real orbit.
\end{subsec}

\begin{subsec}\label{subsec:F2}
If $p\in \Fm_2$ then $\Ho^1 \Gamma_p=\{[1],[-1]\}$. So by Proposition
\ref{prop:ssorb} the $\Gtil_0$-orbits of elements in $\Fm_2$ having real points
fall into two classes: the orbits of $p\in \Fm_2$ with $\ov p= p$ and the
orbits of $p\in \Fm_2$ with $\ov p = -p$.

Let $\OOm = \Gtil_0\cdot p$ with $\ov p = p$. The $\Gtil_0(\R)$-orbits
in $\OOm$ are classified by $\Ho^1 \Zm_{\Gtil_0}(p)$. The identity component
of $\Zm_{\Gtil_0}(p)$ is a 2-dimensional torus and the component group is
abelian of order 27. Hence by Corollary \ref{c:2m+1} and Corollary
\ref{c:prop38} $\Ho^1 \Zm_{\Gtil_0}(p)=1$. Hence $\OOm$ corresponds to
one real orbit.

Let $\OOm = \Gtil_0\cdot p$ with $\ov p = -p$. Then $p=i\lambda_1p_1
+i\lambda_2p_2 -i\lambda_3p_3$ with $\lambda_i\in \R$.
Set
$$n_1=\SmallMatrix{ 1 & 0 & 0 & 0 & 0 & 0 & 0 & 0 & 0 \\
                0 & 0 & 1 & 0 & 0 & 0 & 0 & 0 & 0 \\
                0 & 1 & 0 & 0 & 0 & 0 & 0 & 0 & 0 \\
                0 & 0 & 0 & 0 & 0 & 0 & 1 & 0 & 0 \\
                0 & 0 & 0 & 0 & 0 & 0 & 0 & 0 & 1 \\
                0 & 0 & 0 & 0 & 0 & 0 & 0 & 1 & 0 \\
                0 & 0 & 0 & 1 & 0 & 0 & 0 & 0 & 0 \\
                0 & 0 & 0 & 0 & 0 & 1 & 0 & 0 & 0 \\
                0 & 0 & 0 & 0 & 1 & 0 & 0 & 0 & 0},
g_1 = \SmallMatrix{ 1 & 0 & 0 & 0 & 0 & 0 & 0 & 0 & 0 \\
  0 & \tfrac{1}{2} & \tfrac{1}{2} & 0 & 0 & 0 & 0 & 0 & 0 \\
  0 & 0 & 0 & 1 & 0 & 0 & 1 & 0 & 0 \\
  0 & 0 & 0 & 0 & 0 & \tfrac{1}{2} & 0 & \tfrac{1}{2} & 0 \\
  0 & 0 & 0 & 0 & 1 & 0 & 0 & 0 & 1 \\
  0 & -\tfrac{1}{2}i & \tfrac{1}{2}i & 0 & 0 & 0 & 0 & 0 & 0 \\
  0 & 0 & 0 & -i & 0 & 0 & i & 0 & 0 \\
  0 & 0 & 0 & 0 & 0 & -i & 0 & i & 0 \\
  0 & 0 & 0 & 0 & -\tfrac{1}{2}i & 0 & 0 & 0 & \tfrac{1}{2}i}.
$$
Then $n_1$ is a cocycle, it normalizes $\Cg$ and acts as $-1$ on $\Cg$.
(It is found by writing $-1 = w_2^2$, and $n_1= \hat w_2^4$.)
Furthermore $g_1^{-1}\ov g_1 = n_1$. (It is found by the methods of Section
\ref{sec:comprep}.) Now
\begin{align*}
g_1\cdot p =& \lambda_1(-\tfrac{1}{2}e_{126}-\tfrac{1}{2}e_{349}+2e_{358}-e_{457}+e_{789})\\
+&\lambda_2(-2e_{137}-\tfrac{1}{4}e_{249}-e_{258}-\tfrac{1}{2}e_{456}-\tfrac{1}{2}e_{689})\\
-&\lambda_3(-e_{159}-e_{238}-\tfrac{1}{2}e_{247}-\tfrac{1}{2}e_{346}-e_{678}).
\end{align*}
This is the real representative of $\OOm$.

Now we consider the group $C_p = \Zm_{\Gtil_0}(p)$. Its identity component
is a 2-dimensional torus $T_2$ with elements $T_2(a_1,a_2)$. We have
$$n\overline{T_2(a_1,a_2)}n^{-1} = T_2(\ov a_1,\ov a_1^{-1}\ov a_2^{-1}).$$
So we set $\CC_p = (C_p, \sigma )$ where $\sigma(T_2(a_1,a_2)) =
T_2(\ov a_1,\ov a_1^{-1}\ov a_2^{-1})$. The matrix of $\sigma_*$ on
${\sf X}_*(T_2)$ is $\SmallMatrix{1 & 0\\ -1 & -1}$. We have that
$\ker(1+\sigma_*) = \im(1-\sigma_*) = \langle e_2 \rangle$ where $e_2$
denotes the second basis vector of $\Z^2$. Hence $\Ho^1 T_2=1$ and in the
same way as for $\Zm_{\Gtil_0}(p)$ we see that $\Ho^1 \CC_p = 1$. We conclude that
also in this case $\OOm$ contains exactly one real orbit.
\end{subsec}

\begin{subsec}
If $p\in \Fm_3$ then $\Ho^1 \Gamma_p=\{[1],[-1], [u_1],[u_2]\}$.
So by Proposition
\ref{prop:ssorb} the $\Gtil_0$-orbits of elements in $\Fm_3$ having real points
fall into four classes: the orbits of $p\in \Fm_3$ with $\ov p= p$, the
orbits of $p\in \Fm_3$ with $\ov p = -p$, the orbits of $p\in \Fm_3$ with
$\ov p = u_1\cdot p$, and the orbits of $p\in \Fm_3$ with $\ov p = u_2\cdot p$.

Let $\OOm = \Gtil_0\cdot p$ with $\ov p = p$. The identity component of
$\Zm_{\Gtil_0}(p)$ is a 4-dimensional torus. The component group is abelian
of order 9. In the same way as in \ref{subsec:F2} we conclude that
$\Ho^1 \Zm_{\Gtil_0}(p)=1$. Hence $\OOm$ corresponds to one real orbit.

Let $\OOm = \Gtil_0\cdot p$ with $\ov p = -p$. Then $p=i\lambda_1p_1
+i\lambda_2p_2$ with $\lambda_i\in \R$. We use the $n_1$, $g_1$ from
\ref{subsec:F2}. The real representative of $\OOm$ is
\begin{align*}
g_1\cdot p =& \lambda_1(-\tfrac{1}{2}e_{126}-\tfrac{1}{2}e_{349}+2e_{358}-e_{457}+e_{789})\\
+&\lambda_2(-2e_{137}-\tfrac{1}{4}e_{249}-e_{258}-\tfrac{1}{2}e_{456}-\tfrac{1}{2}e_{689})
\end{align*}
Let $C_p = \Zm_{\Gtil_0}(p)$. Its identity component is a 4-dimensional torus
$T_4$ with elements $T_4(t_1,t_2,t_3,t_4)$. We have
$$n \overline{T_4(t_1,t_2,t_3,t_4)} n^{-1} =
T_2(\bar t_1,\bar t_1^{-1}\bar t_2^{-1},\bar t_1^{-1}\bar t_3^{-1}, \bar t_1
\bar t_2 \bar t_3 \bar t_4).$$
So we set $\CC_p = (C_p, \sigma )$ where $\sigma$ on $T_4$ is given by the
above formula. The matrix of $\sigma_*$ on ${\sf X}_*(T_4)$ is
$\SmallMatrix{1&0&0&0\\-1&-1&0&0\\-1&0&-1&0\\1&1&1&1}$.
We have $\ker(1+\sigma) = \im(1-\sigma)$ and both are spanned by
$(0,1,1,-1)$, $(0,0,2,-1)$. Hence $\Ho^1 \CC_p=1$. Hence $\OOm$ contains
exactly one real orbit.

Now we consider $u_1\in W$. We have $u_1\cdot p_1=-p_1$, $u_2\cdot p_2 = p_2$.
So the set of $p\in \Fm_3$ with $\ov p = u_1\cdot p$ consists of
$p=i\lambda_1 p_1 +\lambda_2p_2$ with $\lambda_1,\lambda_2\in \R$. Let
$\OOm$ be the orbit with this representative.

We have $u_1 = w_3w_2w_3(w_2w_3^{-1}w_2)^2$. Set $n_2 = \hat w_3\hat w_2
\hat w_3(\hat w_2\hat w_3^{-1}\hat w_2)^2$. We also compute $g_2\in \Gtil_0$
with $g_2^{-1}\ov g_2= n_2$. We obtain
$$n_2= \SmallMatrix{ -1 & 0 & 0 & 0 & 0 & 0 & 0 & 0 & 0 \\
                0 & -1 & 0 & 0 & 0 & 0 & 0 & 0 & 0 \\
                0 & 0 & -1 & 0 & 0 & 0 & 0 & 0 & 0 \\
                0 & 0 & 0 & 0 & 0 & 0 & -1 & 0 & 0 \\
                0 & 0 & 0 & 0 & 0 & 0 & 0 & -1 & 0 \\
                0 & 0 & 0 & 0 & 0 & 0 & 0 & 0 & -1 \\
                0 & 0 & 0 & -1 & 0 & 0 & 0 & 0 & 0 \\
                0 & 0 & 0 & 0 & -1 & 0 & 0 & 0 & 0 \\
                0 & 0 & 0 & 0 & 0 & -1 & 0 & 0 & 0},
g_2 = \SmallMatrix{ 0 & 0 & 0 & -\tfrac{1}{2} & 0 & 0 & \tfrac{1}{2} & 0 & 0 \\
                0 & 0 & 0 & 0 & -1 & 0 & 0 & 1 & 0 \\
                0 & 0 & 0 & 0 & 0 & 1 & 0 & 0 & -1 \\
                \tfrac{1}{2}i & 0 & 0 & 0 & 0 & 0 & 0 & 0 & 0 \\
                0 & \tfrac{1}{2}i & 0 & 0 & 0 & 0 & 0 & 0 & 0 \\
                0 & 0 & -i & 0 & 0 & 0 & 0 & 0 & 0 \\
                0 & 0 & 0 & i & 0 & 0 & i & 0 & 0 \\
                0 & 0 & 0 & 0 & i & 0 & 0 & i & 0 \\
                0 & 0 & 0 & 0 & 0 & -i & 0 & 0 & -i}. $$
Also $n_2^2=1$ so that $n_2$ is a cocycle in $\Nm_{\Gtil_0}(\Fm_3)$. So the real
representative of $\OOm$ is
$$g_2\cdot p = \lambda_1(e_{129}-e_{138}+2e_{237}-\tfrac{1}{4}e_{456}-2e_{789}) +
\lambda_2(-\tfrac{1}{2}e_{147}-e_{258}+2e_{369}).$$
Now
$$n_2 \overline{T_4(t_1,t_2,t_3,t_4)} n_2^{-1} =
T_2(\bar t_1,\bar t_2,\bar t_1^{-1}\bar t_3^{-1}, \bar t_2^{-1}
\bar t_4^{-1} ).$$
We set $\CC_p = (C_p, \sigma )$ where $\sigma$ on $T_4$ is given by the
above formula. The matrix of $\sigma_*$ on ${\sf X}_*(T_4)$ is
$\SmallMatrix{1&0&0&0\\0&1&0&0\\-1&0&-1&0\\0&-1&0&-1}$.
Then $\ker(1+\sigma_*) = \im(1-\sigma_*)$ and both are spanned by
$(0,0,1,0)$, $(0,0,0,1)$. Hence the $\Ho^1\CC_p=1$ and $\OOm$ contains one
real orbit.

Now we consider $u_2\in W$. We have $u_2\cdot p_1=p_2$, $u_2\cdot p_2 = p_1$.
So the set of $p\in \Fm_3$ with $\ov p = u_2\cdot p$ consists of
$p=\lambda_1 p_1 +\lambda_2p_2$ with $\lambda_1=x+iy$, $\lambda_2=x-iy$,
$x,y\in \R$. Let $\OOm$ be the orbit with this representative.

We have $u_2 = w_3^{-1}w_2^2w_3^{-1}(w_2w_3)^2w_2$. Set $n_3 = \hat w_3^{-1}
\hat w_2^2\hat w_3^{-1}(\hat w_2\hat w_3)^2\hat w_2$. We also compute
$g_3\in \Gtil_0$ with $g_3^{-1}\ov g_3= n_3$. We obtain
$$n_3=\SmallMatrix{ -1 & 0 & 0 & 0 & 0 & 0 & 0 & 0 & 0 \\
                0 & 0 & 0 & 0 & 0 & 0 & -1 & 0 & 0 \\
                0 & 0 & 0 & -1 & 0 & 0 & 0 & 0 & 0 \\
                0 & 0 & -1 & 0 & 0 & 0 & 0 & 0 & 0 \\
                0 & 0 & 0 & 0 & 0 & 0 & 0 & 0 & -1 \\
                0 & 0 & 0 & 0 & 0 & -1 & 0 & 0 & 0 \\
                0 & -1 & 0 & 0 & 0 & 0 & 0 & 0 & 0 \\
                0 & 0 & 0 & 0 & 0 & 0 & 0 & -1 & 0 \\
                0 & 0 & 0 & 0 & -1 & 0 & 0 & 0 & 0},
g_3 = \SmallMatrix{ 0 & 0 & -1 & 1 & 0 & 0 & 0 & 0 & 0 \\
                0 & -1 & 0 & 0 & 0 & 0 & 0 & 0 & 0 \\
                0 & 0 & 0 & 0 & -\tfrac{1}{2} & 0 & 0 & 0 & \tfrac{1}{2} \\
                \tfrac{1}{2}i & 0 & 0 & 0 & 0 & 0 & 0 & 0 & 0 \\
                0 & 0 & i & i & 0 & 0 & 0 & 0 & 0 \\
                0 & 0 & 0 & 0 & 0 & i & 0 & 0 & 0 \\
                0 & i & 0 & 0 & 0 & 0 & i & 0 & 0 \\
                0 & 0 & 0 & 0 & 0 & 0 & 0 & \tfrac{1}{2}i & 0 \\
                0 & 0 & 0 & 0 & -i & 0 & 0 & 0 & -i}. $$
Also $n_3^2=1$ so that $n_3$ is a cocycle in $\Nm_{\Gtil_0}(\Fm_3)$. So the real
representative of $\OOm$ is
$$g_3\cdot p = x(e_{147}-2e_{169}-e_{245}+e_{289}-e_{356}-\tfrac{1}{2}e_{378}) + y(
e_{124}+e_{136}+\tfrac{1}{2}e_{238}-e_{457}+2e_{569}-e_{789}),$$
where $x,y\in \R$ are such that $xy(x^2-3y^2)(x^2-\tfrac{1}{3}y^2)\neq 0$.
Now
$$n_3 \overline{T_4(t_1,t_2,t_3,t_4)} n_3^{-1} =
T_2(\bar t_1,\bar t_1^{-1}\bar t_3^{-1}, \bar t_1^{-1}\bar t_2^{-1},\bar t_1
\bar t_2 \bar t_3 \bar t_4).$$
We set $\CC_p = (C_p, \sigma )$ where $\sigma$ on $T_4$ is given by the
above formula. The matrix of $\sigma$ on ${\sf X}_*(T_4)$ is
$\SmallMatrix{1&0&0&0\\-1&0&-1&0\\-1&-1&0&0\\1&1&1&1}$.
We have $\ker(1+\sigma_*) = \im(1-\sigma_*)$ and both are spanned by
$(0,1,1,-1)$. Hence $\Ho^2\CC_p=1$ and also in this case $\OOm$
contains one real orbit.
\end{subsec}

\begin{subsec}
If $p\in \Fm_4$ then $\Ho^1 \Gamma_p=\{[1]\}$.
So by Proposition \ref{prop:ssorb} the $\Gtil_0$-orbits of elements in
$\Fm_3$ having real points all have a representative in $\Fm_4$.

Let $p\in \Fm_4$ be such that $\ov p = p$ and let $\OOm$ be its $\Gtil_0$-orbit.
The identity component of $Z_{\Gtil_0}(p)$ is isomorphic to $\SL(3,\C)$ and
the component group is abelian of order 9. Hence $\Ho^1 Z_{\Gtil_0}(p)=1$ and
we conclude that $\OOm$ contains exactly one real orbit.
\end{subsec}

\begin{subsec}\label{subsec:F5}
If $p\in \Fm_5$ then $\Ho^1 \Gamma_p=\{[1],[-1]\}$. So by Proposition
\ref{prop:ssorb} the $\Gtil_0$-orbits of elements in $\Fm_5$ having real points
fall into two classes: the orbits of $p\in \Fm_5$ with $\ov p= p$ and the
orbits of $p\in \Fm_5$ with $\ov p = -p$.

Let $\OOm = \Gtil_0\cdot p$ with $\ov p = p$. The identity component
of $\Zm_{\Gtil_0}(p)$ is $T_2\cdot H$ where $H$ consists of $M(A)=\diag(A,A,A)$
for $A\in \SL(3,\C)$ and $T_2$ consists of the elements
$$T_2(a,b) = \diag(a,a,a,b,b,b,(ab)^{-1},(ab)^{-1},(ab)^{-1}).$$
The surjective map $T_2\times H\to T_2H$, $(t,h)\mapsto th$ has a kernel of
order 3. Hence by Lemma \ref{l:H1-bijective} we see that $\Ho^1 T_2H=1$.
The component group is abelian of order 9. Hence by Corollary \ref{c:2m+1}
and Corollary \ref{c:prop38} $\Ho^1 \Zm_{\Gtil_0}(p)=1$. So $\OOm$ contains
one real orbit.

Now let $p\in \Fm_5$ be such that $\ov p = -p$. Hence $p = \lambda i(p_3-p_4)$
with $\lambda\in \R$. We use $n_1,g_1$ from \ref{subsec:F2}. The real
representative of $\OOm$ is
$$g_1\cdot p = \lambda(e_{148}-e_{159}-e_{238}+\tfrac{1}{2}e_{239}-
\tfrac{1}{2}e_{247}+e_{257}-\tfrac{1}{2}e_{346}-e_{356}-e_{678}-
\tfrac{1}{2}e_{679}).$$

Let $C_p = \Zm_{\Gtil_0}(p)$ and $\CC_p=(C_p,\sigma)$, where $\sigma(g) =
n_1\ov g n_1^{-1}$. We have
\begin{align*}
\sigma(M(A)) &= M(V\overline{A}V) \text{ where }
V = \SmallMatrix{1&0&0\\0&0&1\\0&1&0},\\
\sigma(T_2(a,b)) &= T_2(\bar a, \bar a^{-1} \bar b^{-1}).
\end{align*}

The matrix of $\sigma$ on ${\sf X}_*(T_2)$ is $\SmallMatrix{1&0\\-1&-1}$.
Hence $\ker(1+\sigma) = \im(1-\sigma)$ and $\Ho^1 (T_2,\sigma)=1$.

For any natural number $n$ and for any cocycle $V\in {\rm PGL}(n,\C)$,
(that is, satisfying $V\ov V=1$), the involution  $\sigma(A) = V\ov{A}V^{-1}$
is called {\em inner}.
The group of real points of the corresponding twisted form  is isomorphic to
either $\SL(n,\R)$ or $\SL(n/2,{\mathbb H})$,
where ${\mathbb H}$ denotes the division algebra of Hamilton's quaternions.
In the first case we have $H^1(\SL(n,\C),\sigma)=\{1\}$.
In the second case, which can happen only for {\em even} $n=2m$, we have
$\#H^1(\SL(2m,\C),\sigma)=2$. Since here we have $n=3$ it follows that
$\Ho^1 \CC_p = 1$.
\end{subsec}

\begin{subsec}
If $p\in \Fm_6$ then $\Ho^1 \Gamma_p=\{[1],[-1]\}$. So by Proposition
\ref{prop:ssorb} the $\Gtil_0$-orbits of elements in $\Fm_6$ having real points
fall into two classes: the orbits of $p\in \Fm_6$ with $\ov p= p$ and the
orbits of $p\in \Fm_6$ with $\ov p = -p$.

Let $\OOm = \Gtil_0\cdot p$ with $\ov p = p$. The identity component
of $\Zm_{\Gtil_0}(p)$ is isomorphic to $\SL(3,\C)^3$, consisting of
$M(A_1,A_2,A_3) = \diag(A_1,A_2,A_3)$, where each $A_i\in \SL(3,\C)$.
The component group is of order 3. Hence $\Ho^1 \Zm_{\Gtil_0}(p)=1$.
Therefore $\OOm$ contains one real orbit.

Let $p\in \Fm_6$ be such that $\ov p = -p$. Hence $p = \lambda ip_1$
with $\lambda\in \R$. We use $n_1,g_1$ from \ref{subsec:F2}. The real
representative of $\OOm$ is
$$g_1\cdot p = \lambda(-\tfrac{1}{2}e_{126}-\tfrac{1}{2}e_{349}+2e_{358}-e_{457}
+e_{789}).$$
Let $C_p = \Zm_{\Gtil_0}(p)$ and $\CC_p=(C_p,\sigma)$, where $\sigma(g) =
n_1\ov g n_1^{-1}$. We have
$$\sigma(M(A_1,A_2,A_3)) = M(V\overline{A_1}V,V\overline{A_3}V,
V\overline{A_2}V),$$
with $V$ as in \ref{subsec:F5}.

In the same way as in the previous case it follows that $\Ho^1 \CC_p=1$.
\end{subsec}


\section{The orbits of mixed type}\label{sec:mixed}

\subsection{Methods}\label{sec:mixmeth}

As in Section \ref{sec:semsim} we let $\Gtil_0=\SL(9,\C)$ and
$\Gtil_0(\R)=\SL(9,\R)$.

As seen in Section \ref{subs:jordan} every $x\in \g_1^\cC$ can be written
as $x=p+e$ where $p\in \g_1^\cC$ is semisimple and $e\in \g_1^\cC$ is
nilpotent and $[p,e]=0$ (and the same statement holds for $\g_1$).
It is said that $x$ is
of {\em mixed type} if both $p,e$ are nonzero. Similarly, we say that
an orbit consisting of such elements is of mixed type.

First we briefly comment on the classification of the
$\Gtil_0$-orbits of elements of mixed type, see  also  Subsection
\ref{subs:mix}. It is clear that every such
orbit has a representative $p+e$ such that $p$ lies in one of the sets
$\Fm_k$ for $1\leq k\leq 6$ (see Section \ref{sec:semsim}). Now fix a
$p\in \Fm_k$ and write $\a = \z_{\g^\cC}(p)$. Then $\a$ inherits the grading from
$\g^\cC$. The nilpotent $e$ such that $p+e$ is of mixed type lie in $\a_1$.
Furthermore, $p+e_1$, $p+e_2$ are $\Gtil_0$-conjugate if and only if
$e_1,e_2$ are $\Zm_{\Gtil_0}(p)$-conjugate. The groups $\Zm_{\Gtil_0}(p)$ have been
determined in Section \ref{subsec:semcent}. The Lie algebra of
$\Zm_{\Gtil_0}(p)$ is $\a_0$. Hence by using the method of \cite{Vinberg1979}
(see also \cite{Graaf2017}, Chapter 8) we can determine the
$\Zm_{\Gtil_0}(p)^\circ$-orbits of nilpotent elements in $\a_1$. Finally we
reduce the list by considering conjugacy under the elements of the component
group of $\Zm_{\Gtil_0}(p)$. \cite{VE1978} contains lists of representatives
of nilpotent parts of elements of mixed type, for $p \in \Fm_k$,
$1\leq k\leq 6$. Let $\Tt^\cC_p$ be the set of homogeneous $\ssl_2$-triples in
$\a$. Note that by Theorem \ref{thm:3} the nilpotent
$\Zm_{\Gtil_0}(p)$-orbits in $\a_1$ correspond bijectively to the
$\Zm_{\Gtil_0}(p)$-orbits in $\Tt^\cC_p$.
Note also that $\Zm_{\Gtil_0}(p)$, $\z_{\g^\cC}(p)$ only depend
on $\Fm_k$, not on the particular element $p$. Hence the orbits of the
nilpotent parts also only depend only on $\Fm_k$.

Now we turn to the problem of classifying $\Gtil_0(\R)$-orbits of mixed type.
Section \ref{sec:semsim} has descriptions of the $\Gtil_0(\R)$-orbits of
semisimple elements in $\g_1$. Let $p$ be a representative of such an orbit and
again let $\a=\z_{\g^\cC}(p)$. Analogously to the complex case we need to
determine the $\Zm_{\Gtil_0(\R)}(p)$-orbits of real nilpotent elements in
$\a_1$. By Theorem \ref{thm:3} these correspond bijectively to the
$\Zm_{\Gtil_0(\R)}(p)$-orbits in the set $\Tt_p$ of homogeneous and real
$\ssl_2$-triples in $\a$. Let $t=(h,e,f)$ be a homogeneous and real
$\ssl_2$-triple in $\a$ containing $e$. Let
$$\Zm_0(p,t) = \{ g\in \Zm_{\Gtil_0}(p) \mid g\cdot h =h, g\cdot e=e,
g\cdot f=f\}.$$
Then by Proposition \ref{p:coh-orbits} there is a bijection between
the real $\Zm_{\Gtil_0(\R)}(p)$-orbits contained in the $\Zm_{\Gtil_0}(p)$-orbit
of $t$ and $\ker [ \Ho^1 \Zm_0(p,t) \to \Ho^1 \Zm_{\Gtil_0}(p)]$. Note that
in Section \ref{sec:semsimclas} it has been established hat
$\Ho^1 \Zm_{\Gtil_0}(p)=1$ in all cases. Hence the orbits we are interested
in here correspond bijectively to $\Ho^1 \Zm_0(p,t)$.

\begin{remark}
  More in general (when considering a different $\theta$-group) it may happen
  that $\Ho^1 \Zm_{\Gtil_0}(p)\neq 1$. Then $\Ho^1 \Zm_0(p,t)$ corresponds
  bijectively to the $\Gtil_0(\R)$-orbits contained in the $\Gtil_0$-orbit
  of the $4$-tuple $(p,h,e,f)$. So for each $[c]\in \Ho^1 \Zm_0(p,t)$ we
  compute a $g\in \Gtil_0$ with $g^{-1}\ov g = c$ and obtain the mixed element
  $g\cdot (p+e)$ with semisimple part $g\cdot p$ and nilpotent part $g\cdot e$.
  This way we get all real orbits of mixed elements contained in the
  complex orbit of $p+e$.
\end{remark}

We say that a real semisimple element is {\em canonical} if it lies in one of
the canonical sets $\Fm_k$.
The procedure that we use to classify the mixed elements whose semisimple
part is noncanonical is markedly more complex than the procedure for
classifying those with canonical semisimple part.

If the semisimple element $p$ is canonical then we consider a real nilpotent
$e\in \a=\z_{\g^\cC}(p)$ lying in the homogeneous real $\ssl_2$-triple $t=
(h,e,f)$. In order to compute representatives of the real $\Gtil_0(\R)$-orbits
contained in the $\Gtil_0$-orbit of $p+e$ we compute
the centralizer $\Zm_0(p,t)$ and its Galois cohomology $\Ho^1 \Zm_0(p,t)$.
The elements of the latter set correspond to the real orbits that we are looking
for.
It is possible to take the same approach when $p$ is not canonical.
However, in that case the groups $\Zm_0(p,t)$ tend to be difficult to describe
(they can be nonsplit, for example) and therefore difficult to work with.
In the next paragraphs we describe a different method for this case.
The main idea is the following. We have that $p$ is conjugate over $\C$
to a canonical
semisimple element $q$. So if $t_1$ is a real homogeneous $\ssl_2$-triple
in $\z_{\g^\cC}(p)$ then $t_1$ is conjugate to a homogeneous $\ssl_2$-triple
$t$ in $\z_{\g^\cC}(q)$. Also the stabilizers $\Zm_0(p,t_1)$ and $\Zm_0(q,t)$ are
conjugate. We define a conjugation on the latter so that these two groups
are $\Gamma$-equivariantly isomorphic. Finally we compute $\Ho^1 \Zm_0(q,t)$
where we use the modified conjugation.

Let $p\in \g_1$ be a real noncanonical semisimple element and of the form
$p = gq$, where $g\in \Gtil_0$ and $q\in \Fm_k$ where $\Fm_k$ is one of the
canonical sets found in Section \ref{sec:sumVE}. From our construction
(Proposition \ref{prop:ssorb}) it follows that setting $n=g^{-1}\ov g$
entails $n\in \Zl^1\Nm_{\Gtil_0}(\Fm_k)$ and $n\cdot q = \ov q$.
In all cases we have that $n$ is real, so that $n^2=1$. In the sequel these
$g$ and $n$ are fixed.

Also define $\varphi : \z_{\g^\cC}(q)\cap \g_1^\cC \to \z_{\g^\cC}(p)\cap \g_1^\cC$
by $\varphi(x) = g\cdot x$. Because $\Zm_{\Gtil_0}(p) = g\Zm_{\Gtil_0}(q)g^{-1}$
this is a bijection between the sets of nilpotent
elements in the respective spaces mapping $\Zm_{\Gtil_0}(q)$-orbits to
$\Zm_{\Gtil_0}(p)$-orbits. For $x\in \z_{\g^\cC}(q)\cap \g_1^\cC$
we have that $g\cdot x$ is real (that is, $\ov{g\cdot x} = g\cdot x$) if and
only if $n\ov x = x$. Since $n\in \Nm_{\Gtil_0}(\Fm_k)$ we have that
$q$ and $n\cdot q$ both lie in $\Fm_k$ and therefore
have the same centralizer in $\g^\C$ (Section
\ref{sec:sumVE}), hence $n\cdot \z_{\g^\cC}(q) = \z_{\g^\cC}(q)$.
From $n\cdot q=\ov q$ it follows  that $\z_{\g^\cC}(q)$ is stable under
complex conjugation $x\mapsto \ov x$. We set $\u = \z_{\g^\cC}(q)\cap \g_1^\cC$
and define $\mu : \u \to \u$ by $\mu(x) = n\ov x$. So for $x\in \u$ we have
that $\varphi(x)$ is real if and only if $\mu(x)=x$.
Because $n$ is a cocycle we have $\mu^2(x) = x$ for all $x\in \u$.

Now fix a nilpotent $e\in \z_{\g^\cC}(q)\cap \g_1^\cC=\u$ lying in a homogenous
$\ssl_2$-triple $t=(h,e,f)$. We let $Y = \Zm_{\Gtil_0}(q)\cdot e\subset \u$
be its orbit.
Then $\varphi(Y)$ is a $\Zm_{\Gtil_0}(p)$-orbit in $\z_{\g^\cC}(p)\cap \g_1^\cC$
(and all nilpotent $\Zm_{\Gtil_0}(p)$-orbits in $\z_{\g^\cC}(p)\cap \g_1^\cC$ are
obtained in this way).
We want to determine the real $\Zm_{\Gtil_0(\R)}(p)$-orbits contained in
$\varphi(Y)$.

\begin{lemma}\label{lem:Ymu}
Let $y_0$ be any element of $Y$.
We have $\mu(Y)=Y$ if and only if $\mu(y_0)\in Y$.
\end{lemma}

\begin{proof}
Only one direction needs proof, so suppose that $n\ov y_0 = g_1y_0$ for some
$g_1\in  \Zm_{\Gtil_0}(q)$.

Note that $\psi : \Zm_{\Gtil_0}(q) \to \Zm_{\Gtil_0}(p)$,
$\psi(h) = ghg^{-1}$, is an isomorphism. As $p$ is real, i.e., $\ov p = p$,
we have that $\Zm_{\Gtil_0}(p)$ is closed under conjugation. So for $h\in
\Zm_{\Gtil_0}(q)$ we see that $\psi^{-1} (\ov{\psi(h)})$ lies in $\Zm_{\Gtil_0}(q)$.
But the latter element is equal to $n \ov hn^{-1}$. We conclude that
$\Zm_{\Gtil_0}(q)$ is closed under $h\mapsto n \ov hn^{-1}$.

Let $g_2\in \Zm_{\Gtil_0}(q)$, then  $\mu(g_2y_0) =
n\ov{g_2y_0} = n \ov g_2 n^{-1} n\ov y_0 = n \ov g_2 n^{-1} g_1 y_0 =g_3y_0$ with
$g_3\in \Zm_{\Gtil_0}(q)$. We conclude that $\mu(Y)=Y$.
\end{proof}

Now there are two possibilities. If $\mu(Y)\neq Y$ then $\varphi(Y)$ has no
real points by the previous lemma. In this case there are no real mixed elements
of the form $p+e_1$ conjugate to $q+e$. So we do not consider this $e$.
Note that we can check whether $\varphi(Y) = Y$ using Lemma \ref{lem:Ymu} and
the methods for determining conjugacy of nilpotent elements outlined in
Section \ref{sec:decconj}.

On the other hand, if $\mu(Y)=Y$ then we consider the restriction of $\mu$ to
$Y$. We set $\YY = (Y,\mu)$. With the methods of Section \ref{sec:findreal}
we establish whether $\YY$ has a real point (that is a $y\in Y$ with $\mu(y)=y$)
and, if so, we find one.
If, on the other hand, $\YY$ does not have a real point, then $\varphi(Y)$
has no real points either and also in this case we do not consider this $e$.

We briefly summarize the main steps of the method of Section \ref{sec:findreal}
to find a real point in $\YY$. We set $H= \Zm_{\Gtil_0}(q)$ and define the
conjugation $\tau : H\to H$, $\tau(h) =n\ov hn^{-1}$. Set $\HH=(H,\tau)$ and we
assume that $\Ho^1\HH=1$. Then we do the following
\begin{enumerate}
\item Compute $h_0\in H$ such that $\mu(e) = h_0^{-1}e$.
\item Set $C=\Zm_H (e)$ and let $\nu : C\to C$ be the conjugation defined by
  $\nu(c) = h_0\tau(c)h_0^{-1}$. Set $\CC = (C,\nu)$.
\item Set $d=h_0\tau(h_0)$ and consider the class $[d]\in \Ho^2 \CC$.
  If $[d]\neq 1$ then $\YY$ has no real point and we stop.
\item Otherwise we can find a $c\in C$ with $c\nu(c)d=1$. Set $h_1 = ch_0$.
  (Then $\mu(e) = h_1^{-1}e$ and $h_1\tau(h_1) = 1$.)
\item Because $\Ho^1\HH=1$ we can find $u\in H$ such that $uh_1\tau(u)^{-1} =1$.
  Then $y=u\cdot e$ is a real point of $\YY$.
\end{enumerate}

Now assume that $\YY$ has a real point $e'$.
Then we set $e_1 = g\cdot e'$ and find a real homogeneous $\ssl_2$-triple
$t_1=(h_1,e_1,f_1)$ in $\z_{\g^\cC}(p)$. Set $t'=(h',e',f')$ where
$h'=g^{-1}\cdot h_1$, $f'=g^{-1}\cdot f_1$. Furthermore, we determine a $g' \in
\Zm_{\Gtil_0}(q)$ such that $g' \cdot t = t'$, then $gg'\cdot t = t_1$.
We set $g_0 = gg'$ and $n_0 = g_0^{-1}\ov g_0$. Then $n_0 \in \Zl^1
\Nm_{\Gtil_0}(\Fm_k)$.
We compute $Z_q = \Zm_{\Gtil_0}(q,t)$ (the stabilizer of $t$ in $\Zm_{\Gtil_0}(q)$).
We define $\sigma : Z_q \to Z_q$ by $\sigma(z) = n_0 \ov z n_0^{-1}$ (in the same
way as in the proof of Lemma \ref{lem:Ymu} it is seen that $Z_q$ is closed under
$\sigma$). We set $\mathbf{Z_q} =(Z_q,\sigma)$. Also, we set
$Z_p = \Zm_{\Gtil_0}(p,t_1)$
and $\mathbf{Z}_p =(Z_p,\ov{\phantom u})$. Then $\psi : \mathbf{Z}_q\to
\mathbf{Z}_p$, $\psi(z) = g_0zg_0^{-1}$ is a $\Gamma$-equivariant isomorphism.
Therefore the real orbits in $\varphi(Y)$ correspond bijectively to
$\Ho^1\mathbf{Z}_q$.

In order to find representatives of the orbits we use the methods of Section
\ref{sec:comprep}. The most straightforward way to do this is to use the group
$\Zm_{\Gtil_0}(q)$, the conjugation $\sigma$ and the elements of
$\Ho^1\mathbf{Z}_q$. For a class $[c]\in \Ho^1\mathbf{Z}_q$ we find an
$a\in \Zm_{\Gtil_0}(q)$ with $ac = \sigma(a) = n_0 \ov a n_0^{-1}$.
Then $g_0 a\cdot e$ is a nilpotent element in $\z_{\g^\cC}(p)\cap \g_1^\cC$
and $p+g_0 a\cdot e$ is a representative of the orbit of mixed elements
corresponding to the class $[c]$.

\begin{remark}\label{rem:adhc}
We have used an ad-hoc method by which it is possible to show in many
cases that $Y$ has a real point. The map $\mu : \u\to \u$ is an $\R$-linear
involution. Set $\u_\R = \{ x \in \u \mid \ov x =x\}$; as $\ov n=n$ in all cases
we have that $\mu(\u_\R)=\u_\R$. Hence $\u_\R = \u_+\oplus \u_-$, where the
latter are the eigenspaces of $\mu$ corresponding to the eigenvalues $1$ and
$-1$ respectively. We compute bases of these spaces and consider elements of
three forms: $u\in \u_+$, $iu$ for $u\in \u_-$ and $u_1+iu_2$ for $u_1\in \u_+$,
$u_2\in \u_-$. All these elements are fixed by $\mu$. We list many of these
elements that are nilpotent and check to which $\Zm_{\Gtil_0}(q)$-orbit they
belong (the latter by the methods described in Section \ref{sec:decconj}).
This way we have found real points in most cases.
\end{remark}

\subsection{The orbits}

Here we list the data that we have computed in order to find the classification
of the real orbits of mixed type. The semisimple part $p$ of an element of mixed
type is conjugate to an element of a canonical set $\Fm_k$ for $2\leq k\leq 6$.
(The semisimple elements in $\Fm_1$ have trivial centralizer and
$\Fm_7=\{0\}$.) For each $k$ between 2 and 6 we have a subsection containing
the data for the semisimple elements in $\Fm_k$.

If $p$ is canonical (that is, $p\in \Fm_k$) then we give the representatives
of the nilpotent $\Zm_{\Gtil_0}(p)$-orbits in $\z_{\g^\cC}(p)\cap \g_1^\cC$
(they have been taken from \cite{VE1978}). These representatives are
denoted $e$. We also give a description of $\Zm_{\Gtil_0}(p,t)$ where
$t$ is a homogeneous $\ssl_2$-triple containing $e$. For simplicity we
denote this group by $\Zm_{p,e}$. Thirdly we give $\Ho^1 \Zm_{p,e}$.

If $p$ is not canonical then $p=g\cdot q$ where $q\in \Fm_k$, $g\in \Gtil_0$.
In these cases we go through the nilpotent $\Zm_{\Gtil_0}(q)$-orbits in
$\z_{\g^\cC}(q)\cap \g_1^\cC$ {\em in the same order in which this was done
in the list for canonical $p$}. So we do not repeat the representatives $e$.
Instead we consider $\YY=(Y,\mu)$ where $Y=\Zm_{\Gtil_0}(q)\cdot e$ and
$\mu : Y\to Y$ is defined by $\mu(y) = n\ov y$ where $n=g^{-1}\ov g$.
We start by giving a real point $e'$ in $\YY$. If $\YY$ has no real points then
the corresponding orbit is simply not considered. This only happens once for
$\Fm_3$: for one of the $p$ we only give data corresponding to two nilpotent
orbits, instead of eight.
We give the element $g_0 = gg'$ which has the property that
$g_0\cdot e =g\cdot e'$ is a real nilpotent element in $\z_{\g^\cC}(p)\cap
\g_1^\cC$ (conjugate to $e$). Furthermore, we give the element $n_0 =
g_0^{-1} \ov g_0$. The centralizer $\Zm_{q,e} = \Zm_{\Gtil_0}(q,t)$ has already been
described in the list relative to the canonical semisimple element, so we
do not repeat it. However, we now use the conjugation $\sigma$ given by
$\sigma(z) = n_0 \ov z n_0^{-1}$. We also describe this conjugation more in
detail, allowing us to compute $\Ho^1 (\Zm_{q,e},\sigma)$ of which we also
give the elements.

Throughout we use the elements $h_1,h_2,h_3,h_4$ from Section
\ref{subsec:semcent}, and $n_1,n_2,n_3$ from Section \ref{sec:semsimclas}.
Also we use the Galois cohomology sets $\Ho^1 \Gamma_p$ determined in Section
\ref{subsec:cartsp}.
In particular, $\Ho^1 \Gamma_p=1$ if $p$ is in $\Fm_4$, it is $\{[1],[-1]\}$
if $p$ is in $\Fm_2$, $\Fm_5$, $\Fm_6$ and it is $\{[1],[-1],[u_1],[u_2]\}$
if $p$ is in $\Fm_3$. This means that for $\Fm_4$ we only have a list
of data corresponding to canonical $p$. For $\Fm_2$, $\Fm_5$, $\Fm_6$ we
have two lists, one for canonical $p$ and one for noncanonical $p$.
For $\Fm_3$ we have four lists, one for canonical $p$ and three for
noncanonical $p$.

The Galois cohomology of tori $\TT = (T,\sigma)$ is computed
using the algorithm of Section \ref{sec:H1T}. In all cases the computation is
straightforward and we do not go into the details.

\begin{subsec} Let $p\in \Fm_2$, $p=\ov p$. Then $\Zm_{\Gtil_0}(p)$ has two
nonzero nilpotent orbits in $\z_{\g^\cC}(p)\cap \g_1^\cC$.
\begin{enumerate}
\item  $e=e_{168}+e_{249}$. Let $\zeta$ be a primitive third root of unity
  and set $g_1=h_2$, $g_2=h_1h_3^2$, $g_3=\diag(\zeta^2,1,\zeta,1,\zeta,\zeta^2,
  \zeta,\zeta^2,1)$ and $g_4=\diag(\zeta^2,\zeta,1,\zeta,1,\zeta^2,1,\zeta^2,
  \zeta)$.
  Then $\Zm_{p,e}$ is of order 81 and generated by $g_1,\ldots,g_4$.
  So by Proposition \ref{p:C-3} we see that $\Ho^1 \Zm_{p,e}=1$.
\item $e=e_{168}$. Let $T_1$ denote the 1-dimensional torus consisting of
  the elements
  $T_1(t)=\diag(1,t,t^{-1},t,t^{-1},1,t^{-1},1,t)$, $t\in \C^\times$. Let
  $g_1,g_2,g_3$ be as in the previous case.
  Then $\Zm_{p,e} = g_1^{i_1}g_2^{i_2}g_3^{i_3} T_1$, $0\leq i_j \leq 2$.
  So using Proposition \ref{p:C-3} we see that $\Ho^1 \Zm_{p,e}=1$.
\end{enumerate}

Now let $p$ be a real semisimple element conjugate to $q\in\Fm_2$ with
$\ov q = -q$. Then the map $\mu$ is given by $\mu(y) = n_1\ov y$.
\begin{enumerate}
\item  $e_{249}-e_{357}$ is a real element in $\YY$. In this case
  $\Zm_{q,e}$ is a finite group of order 81. So whatever the conjugation is,
  by Proposition \ref{p:C-3} we have that $\Ho^1 \Zm_{q,e}=1$.
\item $ie_{168}$ is a real element in $\YY$. We have
  $$g_0 = \SmallMatrix{ -i & 0 & 0 & 0 & 0 & 0 & 0 & 0 & 0 \\
                0 & \tfrac{1}{2} & \tfrac{1}{2}i & 0 & 0 & 0 & 0 & 0 & 0 \\
                0 & 0 & 0 & 1 & 0 & 0 & i & 0 & 0 \\
                0 & 0 & 0 & 0 & 0 & -\tfrac{1}{2}i & 0 & -\tfrac{1}{2}i & 0 \\
                0 & 0 & 0 & 0 & i & 0 & 0 & 0 & 1 \\
                0 & -\tfrac{1}{2}i & -\tfrac{1}{2} & 0 & 0 & 0 & 0 & 0 & 0 \\
                0 & 0 & 0 & -i & 0 & 0 & -1 & 0 & 0 \\
                0 & 0 & 0 & 0 & 0 & -1 & 0 & 1 & 0 \\
                0 & 0 & 0 & 0 & \tfrac{1}{2} & 0 & 0 & 0 & \tfrac{1}{2}i},
  n_0 = \SmallMatrix{ -1 & 0 & 0 & 0 & 0 & 0 & 0 & 0 & 0 \\
                0 & 0 & -i & 0 & 0 & 0 & 0 & 0 & 0 \\
                0 & -i & 0 & 0 & 0 & 0 & 0 & 0 & 0 \\
                0 & 0 & 0 & 0 & 0 & 0 & -i & 0 & 0 \\
                0 & 0 & 0 & 0 & 0 & 0 & 0 & 0 & -i \\
                0 & 0 & 0 & 0 & 0 & 0 & 0 & -1 & 0 \\
                0 & 0 & 0 & -i & 0 & 0 & 0 & 0 & 0 \\
                0 & 0 & 0 & 0 & 0 & -1 & 0 & 0 & 0 \\
                0 & 0 & 0 & 0 & -i & 0 & 0 & 0 & 0}. $$
  For the conjugation $\sigma$ of $\Zm_{q,e}$ we have $\sigma(T_1(t)) =
  T_1(\ov t^{-1})$. So $\Ho^1 T_1 = \{[1],[ T_1(-1)]\}$.
  Hence by Proposition \ref{p:C-3} $\Ho^1 \Zm_{q,e}=\{[1],[ T_1(-1)]\}$.
\end{enumerate}
\end{subsec}

\begin{subsec} Let $p\in \Fm_3$, $p=\ov p$. Then $\Zm_{\Gtil_0}(p)$ has eight
nonzero
nilpotent orbits in $\z_{\g^\cC}(p)\cap \g_1^\cC$. In all cases the identity
component of $\Zm_{p,e}$ is a torus and the component group has order
$3^k$. So by Proposition \ref{p:C-3} we have that $\Ho^1 \Zm_{p,e}=1$ in
all cases.
\begin{enumerate}
\item  $e=e_{159}+e_{168}+e_{249}+e_{267}$. Here $\Zm_{p,e}$ consists of
  $$\diag(\zeta^2\delta^2, \zeta^2\delta\epsilon, \zeta^2\epsilon^2, \zeta
  \delta^2\epsilon, \zeta\delta\epsilon^2, \zeta,
  \delta^2\epsilon^2,\delta,\epsilon)
  $$
  where $\zeta^3=\delta^3=\epsilon^3=1$. Hence it is abelian of order 27.
\item  $e=e_{159}+e_{168}+e_{249}$. Let $T_1$ be the 1-dimensional torus
  consisting of
  $$T_1(s)=\diag(1,s^{-1},s,s,1,s^{-1},s^{-1},s,1)\text{ for }s\in \C^\times.$$
  Let $g=\diag(\delta,\delta^2,1,1,\delta,\delta^2,\delta^2,1,\delta)$
  where $\delta$ is a primitive third root of unity.
  Then $\Zm_{p,e}=\cup_{i,j=0}^2 h_2^ig^j T_1$.
\item  $e=e_{159}+e_{168}+e_{267}$. Let $T_1$ be the 1-dimensional torus
  consisting of
  $$T_1(s)= \diag(1,s,s^{-1},s,s^{-1},1,s^{-1},1,s)\text{ for }s\in \C^\times.$$
  Then $\Zm_{p,e}=\cup_{i,j=0}^2 h_2^ih_4^j T_1$.
\item  $e=e_{159}+e_{168}$. Let $T_2$ be the 2-dimensional torus consisting of
  $$T_2(s,t)= \diag(1,s^{-1}t,st^{-1},st,t^{-1},s^{-1},s^{-1}t^{-1},s, t)
  \text{ for }s,t\in \C^\times.$$
  Then $\Zm_{p,e}=\cup_{i=0}^2 h_2^i T_2$.
\item  $e=e_{159}+e_{267}$. Let $T_2$ be the 2-dimensional torus consisting of
  $$T_2(s,t)=\diag(s,t,s^{-1}t^{-1},t,s^{-1}t^{-1},s,s^{-1}t^{-1},s, t)
  \text{ for }s,t\in \C^\times.$$
  Then $\Zm_{p,e}=\cup_{i,j=0}^2 h_2^ig^j T_2$ where $g=h_1h_3$.
\item $e=e_{168}+e_{249}$. Let $T_2$ be the 2-dimensional torus consisting of
  $$T_2(s,t)= \diag(t,s^{-1}t^{-1},s,s,t,s^{-1} t^{-1}, s^{-1}t^{-1},s, t)
  \text{ for }s,t\in \C^\times.$$
\item $e=e_{159}$. Let $T_3$ be the 3-dimensional torus consisting of
  $$T_3(s,t,u)=\diag(s^{-1}u^{-1},s^{-1}t^{-1},s^{2}tu,stu^2,s,s^{-2}t^{-1}u^{-2},
  t^{-1}u^{-1},t,u) \text{ for }s,t,u\in \C^\times.$$
  Then $\Zm_{p,e} = \cup_{i=0}^2 g^i T_3$ where $g=h_1h_3$.
\item $e=e_{168}$. Let $T_3$ be the 3-dimensional torus consisting of
  $$T_3(s,t,u)=\diag(s^{-1}t^{-1},s^2tu,s^{-1}u^{-1},st^2u,s^{-2}t^{-2}u^{-1},s,
  t^{-1}u^{-1},t,u)\text{ for }s,t,u\in \C^\times.$$
  Then $\Zm_{p,e}=\cup_{i=0}^2 g^i T_3$, where $g=h_1h_3^2$.
\end{enumerate}
Let $p$ be a real semisimple element conjugate to $q\in\Fm_3$ with
$\ov q = -q$. Then the map $\mu$ is given by $\mu(y) = n_1 \ov y$.
\begin{enumerate}
\item $e_{249}-ie_{267}-ie_{348}-e_{357}$ is a real element in $\YY$. In this case
  the centralizer is abelian of order 27. Hence by Proposition \ref{p:C-3}
  we see that $\Ho^1 \Zm_{p,e}=1$.
\item $-ie_{159}+e_{249}-e_{357}$ is a real element in $\YY$. We have
  $$g_0 = \SmallMatrix{ 0 & 0 & 0 & 0 & -i & 0 & 0 & 0 & 0 \\
                0 & 0 & 0 & -\tfrac{1}{2} & 0 & -\tfrac{1}{2}i & 0 & 0 & 0 \\
                0 & i & 0 & 0 & 0 & 0 & 0 & 1 & 0 \\
                0 & 0 & -\tfrac{1}{2} & 0 & 0 & 0 & \tfrac{1}{2}i & 0 & 0 \\
                i & 0 & 0 & 0 & 0 & 0 & 0 & 0 & -i \\
                0 & 0 & 0 & -\tfrac{1}{2}i & 0 & -\tfrac{1}{2} & 0 & 0 & 0 \\
                0 & -1 & 0 & 0 & 0 & 0 & 0 & -i & 0 \\
                0 & 0 & -i & 0 & 0 & 0 & 1 & 0 & 0 \\
                -\tfrac{1}{2} & 0 & 0 & 0 & 0 & 0 & 0 & 0 & -\tfrac{1}{2}},
  n_0=\SmallMatrix{ 0 & 0 & 0 & 0 & 0 & 0 & 0 & 0 & 1 \\
                0 & 0 & 0 & 0 & 0 & 0 & 0 & -i & 0 \\
                0 & 0 & 0 & 0 & 0 & 0 & i & 0 & 0 \\
                0 & 0 & 0 & 0 & 0 & -i & 0 & 0 & 0 \\
                0 & 0 & 0 & 0 & -1 & 0 & 0 & 0 & 0 \\
                0 & 0 & 0 & -i & 0 & 0 & 0 & 0 & 0 \\
                0 & 0 & i & 0 & 0 & 0 & 0 & 0 & 0 \\
                0 & -i & 0 & 0 & 0 & 0 & 0 & 0 & 0 \\
                1 & 0 & 0 & 0 & 0 & 0 & 0 & 0 & 0}. $$
    For the conjugation $\sigma$ of $\Zm_{q,e}$ we have $\sigma(T_1(s)) =
    T_1(\ov s^{-1})$. So $\Ho^1 T_1 = \{[1],[ T_1(-1)]\}$.
    By  Proposition \ref{p:C-3} we have $\Ho^1 \Zm_{q,e}= \{[1],[ T_1(-1)]\}$.
\item $ie_{168}-e_{267}+e_{348}$ is a real element in $\YY$. We have
$$g_0=\SmallMatrix{ 0 & 0 & 0 & 0 & 0 & 0 & 0 & i & 0 \\
                0 & 0 & 0 & 0 & 0 & 0 & -\tfrac{1}{2}i & 0 & \tfrac{1}{2} \\
                0 & -1 & 0 & 0 & i & 0 & 0 & 0 & 0 \\
                \tfrac{1}{2}i & 0 & 0 & 0 & 0 & -\tfrac{1}{2}i & 0 & 0 & 0 \\
                0 & 0 & i & 1 & 0 & 0 & 0 & 0 & 0 \\
                0 & 0 & 0 & 0 & 0 & 0 & \tfrac{1}{2} & 0 & -\tfrac{1}{2}i \\
                0 & i & 0 & 0 & -1 & 0 & 0 & 0 & 0 \\
                1 & 0 & 0 & 0 & 0 & 1 & 0 & 0 & 0 \\
                0 & 0 & \tfrac{1}{2} & \tfrac{1}{2}i & 0 & 0 & 0 & 0 & 0},
n_0=\SmallMatrix{ 0 & 0 & 0 & 0 & 0 & 1 & 0 & 0 & 0 \\
                0 & 0 & 0 & 0 & i & 0 & 0 & 0 & 0 \\
                0 & 0 & 0 & -i & 0 & 0 & 0 & 0 & 0 \\
                0 & 0 & -i & 0 & 0 & 0 & 0 & 0 & 0 \\
                0 & i & 0 & 0 & 0 & 0 & 0 & 0 & 0 \\
                1 & 0 & 0 & 0 & 0 & 0 & 0 & 0 & 0 \\
                0 & 0 & 0 & 0 & 0 & 0 & 0 & 0 & i \\
                0 & 0 & 0 & 0 & 0 & 0 & 0 & -1 & 0 \\
                0 & 0 & 0 & 0 & 0 & 0 & i & 0 & 0}.  $$
For the conjugation $\sigma$ of $\Zm_{q,e}$ we have $\sigma(T_1(s)) =
T_1(\ov s^{-1})$. As in the previous case we see that $\Ho^1 \Zm_{q,e} =
\{[1],[ T_1(-1)]\}$.
\item $-ie_{159}+ie_{168}$ is a real element in $\YY$. We have
$$g_0=\SmallMatrix{ 1 & 0 & 0 & 0 & 0 & 0 & 0 & 0 & 0 \\
                0 & \tfrac{1}{2}i & -\tfrac{1}{2}i & 0 & 0 & 0 & 0 & 0 & 0 \\
                0 & 0 & 0 & 1 & 0 & 0 & 1 & 0 & 0 \\
                0 & 0 & 0 & 0 & 0 & \tfrac{1}{2}i & 0 & \tfrac{1}{2} & 0 \\
                0 & 0 & 0 & 0 & -i & 0 & 0 & 0 & 0 \\
                0 & \tfrac{1}{2} & \tfrac{1}{2} & 0 & 0 & 0 & 0 & 0 & 0 \\
                0 & 0 & 0 & -i & 0 & 0 & i & 0 & 0 \\
                0 & 0 & 0 & 0 & 0 & 1 & 0 & i & 0 \\
                0 & 0 & 0 & 0 & -\tfrac{1}{2} & 0 & 0 & 0 & \tfrac{1}{2}i},
  n_0=\SmallMatrix{ 1 & 0 & 0 & 0 & 0 & 0 & 0 & 0 & 0 \\
                0 & 0 & 1 & 0 & 0 & 0 & 0 & 0 & 0 \\
                0 & 1 & 0 & 0 & 0 & 0 & 0 & 0 & 0 \\
                0 & 0 & 0 & 0 & 0 & 0 & 1 & 0 & 0 \\
                0 & 0 & 0 & 0 & 0 & 0 & 0 & 0 & i \\
                0 & 0 & 0 & 0 & 0 & 0 & 0 & -i & 0 \\
                0 & 0 & 0 & 1 & 0 & 0 & 0 & 0 & 0 \\
                0 & 0 & 0 & 0 & 0 & -i & 0 & 0 & 0 \\
                0 & 0 & 0 & 0 & i & 0 & 0 & 0 & 0}.
  $$
  For the conjugation $\sigma$ of $\Zm_{q,e}$ we have $\sigma(T_2(s,t)) =
  T_2(\ov s^{-1},\ov t^{-1})$. So $\Ho^1 T_2 = \{[1], [T_2(-1,1)],
  [T_2(1,-1)], [T_2(-1,-1)] \}$. Because $\# \Ho^1 T_2 =4$ we cannot use
  Proposition \ref{p:C-3} in this case.  The component group of $\Zm_{q,e}$ is
  generated by $h_2=\diag(\zeta,\zeta,\zeta,\zeta^2,\zeta^2,\zeta^2,1,1,1)$
  and $\sigma(h_2) = h_2^2 T_2(\zeta,\zeta)$. Let $A=T_2$, $B=\Zm_{q,e}$ and
  $C=B/A$. Then we have the exact sequence $1\to A\to B\to C\to 1$,
  which by Proposition \ref{p:serre-prop38} gives the exact sequence
  $$\Ho^1 A \to \Ho^1 B \to \Ho^1 C.$$
  As $C$ is of order 3 we have that $\Ho^1 C=1$. The conjugation acts
  nontrivially on $C$ so that $C^\Ga$ is trivial. Hence by Proposition
  \ref{p:action-C-Gamma} we conclude that $\Ho^1 B = \{[1], [T_2(-1,1)],
  [T_2(1,-1)], [T_2(-1,-1)] \}$.
\item $-e_{267}+e_{348}$ is a real element in $\YY$. We have
$$g_0=\SmallMatrix{ 0 & 0 & -1 & 0 & 0 & 0 & 0 & 0 & 0 \\
                \tfrac{1}{2} & -\tfrac{1}{2} & 0 & 0 & 0 & 0 & 0 & 0 & 0 \\
                0 & 0 & 0 & 0 & 0 & -1 & 0 & 0 & 1 \\
                0 & 0 & 0 & 0 & -\tfrac{1}{2} & 0 & \tfrac{1}{2} & 0 & 0 \\
                0 & 0 & 0 & 0 & 1 & 0 & 0 & 1 & 0 \\
                -\tfrac{1}{2}i & -\tfrac{1}{2}i & 0 & 0 & 0 & 0 & 0 & 0 & 0 \\
                0 & 0 & 0 & 0 & 0 & i & 0 & 0 & i \\
                0 & 0 & 0 & 0 & i & 0 & i & 0 & 0 \\
                0 & 0 & 0 & -\tfrac{1}{2}i & 0 & 0 & 0 & \tfrac{1}{2}i & 0},
  n_0 = \SmallMatrix{ 0 & -1 & 0 & 0 & 0 & 0 & 0 & 0 & 0 \\
                -1 & 0 & 0 & 0 & 0 & 0 & 0 & 0 & 0 \\
                0 & 0 & 1 & 0 & 0 & 0 & 0 & 0 & 0 \\
                0 & 0 & 0 & 0 & 0 & 0 & 0 & 1 & 0 \\
                0 & 0 & 0 & 0 & 0 & 0 & -1 & 0 & 0 \\
                0 & 0 & 0 & 0 & 0 & 0 & 0 & 0 & -1 \\
                0 & 0 & 0 & 0 & -1 & 0 & 0 & 0 & 0 \\
                0 & 0 & 0 & 1 & 0 & 0 & 0 & 0 & 0 \\
                0 & 0 & 0 & 0 & 0 & -1 & 0 & 0 & 0}. $$
For the conjugation $\sigma$ of $\Zm_{q,e}$ we have $\sigma(T_2(s,t)) =
T_2(\ov t, \ov s)$.  It follows that $\Ho^1 \Zm_{q,e} = 1$.
\item $e_{249}-e_{357}$ is a real element in $\YY$. We have
$$g_0=\SmallMatrix{ 0 & 0 & -1 & 0 & 0 & 0 & 0 & 0 & 0 \\
                -\tfrac{1}{2} & \tfrac{1}{2} & 0 & 0 & 0 & 0 & 0 & 0 & 0 \\
                0 & 0 & 0 & 0 & 0 & -1 & 0 & 0 & 1 \\
                0 & 0 & 0 & 0 & \tfrac{1}{2} & 0 & \tfrac{1}{2} & 0 & 0 \\
                0 & 0 & 0 & -1 & 0 & 0 & 0 & 1 & 0 \\
                \tfrac{1}{2}i & \tfrac{1}{2}i & 0 & 0 & 0 & 0 & 0 & 0 & 0 \\
                0 & 0 & 0 & 0 & 0 & i & 0 & 0 & i \\
                0 & 0 & 0 & 0 & -i & 0 & i & 0 & 0 \\
                0 & 0 & 0 & \tfrac{1}{2}i & 0 & 0 & 0 & \tfrac{1}{2}i & 0},
  n_0= \SmallMatrix{ 0 & -1 & 0 & 0 & 0 & 0 & 0 & 0 & 0 \\
                -1 & 0 & 0 & 0 & 0 & 0 & 0 & 0 & 0 \\
                0 & 0 & 1 & 0 & 0 & 0 & 0 & 0 & 0 \\
                0 & 0 & 0 & 0 & 0 & 0 & 0 & -1 & 0 \\
                0 & 0 & 0 & 0 & 0 & 0 & 1 & 0 & 0 \\
                0 & 0 & 0 & 0 & 0 & 0 & 0 & 0 & -1 \\
                0 & 0 & 0 & 0 & 1 & 0 & 0 & 0 & 0 \\
                0 & 0 & 0 & -1 & 0 & 0 & 0 & 0 & 0 \\
                0 & 0 & 0 & 0 & 0 & -1 & 0 & 0 & 0}. $$
For the conjugation $\sigma$ of $\Zm_{q,e}$ we have $\sigma(T_2(s,t)) =
T_2(\ov s, \ov s^{-1} \ov t^{-1})$. It follows that $\Ho^1 \Zm_{q,e} = 1$.
\item $ie_{159}$ is a real element in $\YY$. We have
$$g_0=\SmallMatrix{ i & 0 & 0 & 0 & 0 & 0 & 0 & 0 & 0 \\
                0 & \tfrac{1}{2} & -\tfrac{1}{2}i & 0 & 0 & 0 & 0 & 0 & 0 \\
                0 & 0 & 0 & -i & 0 & 0 & 1 & 0 & 0 \\
                0 & 0 & 0 & 0 & 0 & \tfrac{1}{2}i & 0 & \tfrac{1}{2} & 0 \\
                0 & 0 & 0 & 0 & 1 & 0 & 0 & 0 & 0 \\
                0 & -\tfrac{1}{2}i & \tfrac{1}{2} & 0 & 0 & 0 & 0 & 0 & 0 \\
                0 & 0 & 0 & -1 & 0 & 0 & i & 0 & 0 \\
                0 & 0 & 0 & 0 & 0 & 1 & 0 & i & 0 \\
                0 & 0 & 0 & 0 & -\tfrac{1}{2}i & 0 & 0 & 0 & \tfrac{1}{2}i},
n_0 = \SmallMatrix{ -1 & 0 & 0 & 0 & 0 & 0 & 0 & 0 & 0 \\
  0 & 0 & i & 0 & 0 & 0 & 0 & 0 & 0 \\
                0 & i & 0 & 0 & 0 & 0 & 0 & 0 & 0 \\
                0 & 0 & 0 & 0 & 0 & 0 & i & 0 & 0 \\
                0 & 0 & 0 & 0 & 0 & 0 & 0 & 0 & 1 \\
                0 & 0 & 0 & 0 & 0 & 0 & 0 & -i & 0 \\
                0 & 0 & 0 & i & 0 & 0 & 0 & 0 & 0 \\
                0 & 0 & 0 & 0 & 0 & -i & 0 & 0 & 0 \\
                0 & 0 & 0 & 0 & 1 & 0 & 0 & 0 & 0}.$$
For the conjugation $\sigma$ of $\Zm_{q,e}$ we have $\sigma(T_3(s,t,u))=
T_3(\ov u, \ov s^{-2}\ov t^{-1}\ov u^{-2},\ov s )$. Hence
$\Ho^1 T_3 = \{ [1], [ T(1,-1,1)] \}$.  By  Proposition \ref{p:C-3} it
follows that $\Ho^1 \Zm_{q,e}= \{ [1], [ T(1,-1,1)] \}$.
\item $ie_{168}$ is a real element in $\YY$. We have
  $$g_0=\SmallMatrix{ i & 0 & 0 & 0 & 0 & 0 & 0 & 0 & 0 \\
                0 & -\tfrac{1}{2}i & \tfrac{1}{2} & 0 & 0 & 0 & 0 & 0 & 0 \\
                0 & 0 & 0 & -i & 0 & 0 & 1 & 0 & 0 \\
                0 & 0 & 0 & 0 & 0 & \tfrac{1}{2} & 0 & \tfrac{1}{2} & 0 \\
                0 & 0 & 0 & 0 & i & 0 & 0 & 0 & 1 \\
                0 & -\tfrac{1}{2} & \tfrac{1}{2}i & 0 & 0 & 0 & 0 & 0 & 0 \\
                0 & 0 & 0 & -1 & 0 & 0 & i & 0 & 0 \\
                0 & 0 & 0 & 0 & 0 & -i & 0 & i & 0 \\
                0 & 0 & 0 & 0 & \tfrac{1}{2} & 0 & 0 & 0 & \tfrac{1}{2}i},
  n_0=\SmallMatrix{ -1 & 0 & 0 & 0 & 0 & 0 & 0 & 0 & 0 \\
                0 & 0 & i & 0 & 0 & 0 & 0 & 0 & 0 \\
                0 & i & 0 & 0 & 0 & 0 & 0 & 0 & 0 \\
                0 & 0 & 0 & 0 & 0 & 0 & i & 0 & 0 \\
                0 & 0 & 0 & 0 & 0 & 0 & 0 & 0 & -i \\
                0 & 0 & 0 & 0 & 0 & 0 & 0 & 1 & 0 \\
                0 & 0 & 0 & i & 0 & 0 & 0 & 0 & 0 \\
                0 & 0 & 0 & 0 & 0 & 1 & 0 & 0 & 0 \\
                0 & 0 & 0 & 0 & -i & 0 & 0 & 0 & 0}. $$
For the conjugation $\sigma$ of $\Zm_{q,e}$ we have $\sigma(T_3(s,t,u))=
T_3(\ov t, \ov s \ov s^{-2}\ov t^{-2}\ov u^{-1} )$. Hence
$\Ho^1 T_3 = \{ [1], [ T(1,1,-1)] \}$.  By  Proposition \ref{p:C-3} we see that
$\Ho^1 \Zm_{q,e}=\{ [1], [ T(1,1,-1)] \}$.
\end{enumerate}
Let $p$ be a real semisimple element conjugate to $q\in\Fm_3$ with
$\ov q = u_1 q$. Then the map $\mu$ is given by $\mu(y) = n_2 \ov y$.
Computer calculations based on Lemma \ref{lem:Ymu} show that only two orbits
are stable under $\mu$, namely orbits (1) and (4).
\begin{enumerate}
\item[(1)] $-e_{159}-e_{168}+e_{249}+e_{267}$ is a real element in $\YY$.
  In this case
  the centralizer is abelian of order 27. Hence by Proposition \ref{p:C-3}
  we see that $\Ho^1 \Zm_{p,e}=1$.
\item[(4)] $e_{249}+e_{267}$ is a real element in $\YY$. We have
$$g_0=\SmallMatrix{ 0 & 0 & 0 & 0 & 0 & -\tfrac{1}{2} & 0 & 0 & \tfrac{1}{2} \\
                0 & 0 & 0 & -1 & 0 & 0 & 1 & 0 & 0 \\
                0 & 0 & 0 & 0 & 1 & 0 & 0 & -1 & 0 \\
                0 & 0 & \tfrac{1}{2}i & 0 & 0 & 0 & 0 & 0 & 0 \\
                \tfrac{1}{2}i & 0 & 0 & 0 & 0 & 0 & 0 & 0 & 0 \\
                0 & -i & 0 & 0 & 0 & 0 & 0 & 0 & 0 \\
                0 & 0 & 0 & 0 & 0 & i & 0 & 0 & i \\
                0 & 0 & 0 & i & 0 & 0 & i & 0 & 0\\
                0 & 0 & 0 & 0 & -i & 0 & 0 & -i & 0 },
  n_0=\SmallMatrix{ -1 & 0 & 0 & 0 & 0 & 0 & 0 & 0 & 0 \\
                0 & -1 & 0 & 0 & 0 & 0 & 0 & 0 & 0 \\
                0 & 0 & -1 & 0 & 0 & 0 & 0 & 0 & 0 \\
                0 & 0 & 0 & 0 & 0 & 0 & -1 & 0 & 0 \\
                0 & 0 & 0 & 0 & 0 & 0 & 0 & -1 & 0 \\
                0 & 0 & 0 & 0 & 0 & 0 & 0 & 0 & -1 \\
                0 & 0 & 0 & -1 & 0 & 0 & 0 & 0 & 0 \\
                0 & 0 & 0 & 0 & -1 & 0 & 0 & 0 & 0 \\
                0 & 0 & 0 & 0 & 0 & -1 & 0 & 0 & 0}.$$
  For the conjugation $\sigma$ we have $\sigma(T_2(s,t)) = T_2(\ov t^{-1},
  \ov s^{-1}  )$. Hence $\Ho^1 T_2 = 1$.  By  Proposition \ref{p:C-3} we also
  have $\Ho^1 \Zm_{q,e}=1$.
\end{enumerate}
Let $p$ be a real semisimple element conjugate to $q\in\Fm_3$ with
$\ov q = u_2 q$. Then the map $\mu$ is given by $\mu(y) = n_3 \ov y$.
In this case all eight orbits have real elements.
\begin{enumerate}
\item   $-e_{159}+e_{249}-e_{267}-e_{357}$ is a real element in $\YY$.
  In this case
  the centralizer is abelian of order 27. Hence by Proposition \ref{p:C-3}
  we see that $\Ho^1 \Zm_{p,e}=1$.
\item  $-e_{159}+e_{249}-e_{357}$ is a real element in $\YY$.
  Here
  $$g_0=\SmallMatrix{ 0 & 0 & 0 & 1 & 0 & 0 & 0 & 1 & 0 \\
                0 & 1 & 0 & 0 & 0 & 1 & 0 & 0 & 0 \\
                -\tfrac{1}{2} & 0 & 0 & 0 & 0 & 0 & 0 & 0 & -\tfrac{1}{2} \\
                0 & 0 & 0 & 0 & \tfrac{1}{2}i & 0 & 0 & 0 & 0 \\
                0 & 0 & 0 & -i & 0 & 0 & 0 & i & 0 \\
                0 & 0 & 0 & 0 & 0 & 0 & i & 0 & 0 \\
                0 & i & 0 & 0 & 0 & -i & 0 & 0 & 0 \\
                0 & 0 & -\tfrac{1}{2}i & 0 & 0 & 0 & 0 & 0 & 0 \\
                i & 0 & 0 & 0 & 0 & 0 & 0 & 0 & -i},
  n_0=\SmallMatrix{ 0 & 0 & 0 & 0 & 0 & 0 & 0 & 0 & 1 \\
                0 & 0 & 0 & 0 & 0 & 1 & 0 & 0 & 0 \\
                0 & 0 & -1 & 0 & 0 & 0 & 0 & 0 & 0 \\
                0 & 0 & 0 & 0 & 0 & 0 & 0 & 1 & 0 \\
                0 & 0 & 0 & 0 & -1 & 0 & 0 & 0 & 0 \\
                0 & 1 & 0 & 0 & 0 & 0 & 0 & 0 & 0 \\
                0 & 0 & 0 & 0 & 0 & 0 & -1 & 0 & 0 \\
                0 & 0 & 0 & 1 & 0 & 0 & 0 & 0 & 0 \\
                1 & 0 & 0 & 0 & 0 & 0 & 0 & 0 & 0}. $$
Furthermore $\sigma(T_1(s)) = T_1(\ov s)$. Hence $\Ho^1 \Zm_{q,e}=1$.
\item $-e_{159}+ie_{168}-e_{267}$ is a real element in $\YY$.
  Here
  $$g_0=\SmallMatrix{ 0 & 0 & i & -1 & 0 & 0 & 0 & 0 & 0 \\
                0 & -1 & 0 & 0 & 0 & 0 & i & 0 & 0 \\
                0 & 0 & 0 & 0 & -\tfrac{1}{2}i & 0 & 0 & 0 & \tfrac{1}{2} \\
                -\tfrac{1}{2} & 0 & 0 & 0 & 0 & 0 & 0 & 0 & 0 \\
                0 & 0 & 1 & -i & 0 & 0 & 0 & 0 & 0 \\
                0 & 0 & 0 & 0 & 0 & -1 & 0 & 0 & 0 \\
                0 & i & 0 & 0 & 0 & 0 & -1 & 0 & 0 \\
                0 & 0 & 0 & 0 & 0 & 0 & 0 & \tfrac{1}{2} & 0 \\
                0 & 0 & 0 & 0 & 1 & 0 & 0 & 0 & -i},
  n_0=\SmallMatrix{ 1 & 0 & 0 & 0 & 0 & 0 & 0 & 0 & 0 \\
                0 & 0 & 0 & 0 & 0 & 0 & i & 0 & 0 \\
                0 & 0 & 0 & i & 0 & 0 & 0 & 0 & 0 \\
                0 & 0 & i & 0 & 0 & 0 & 0 & 0 & 0 \\
                0 & 0 & 0 & 0 & 0 & 0 & 0 & 0 & i \\
                0 & 0 & 0 & 0 & 0 & 1 & 0 & 0 & 0 \\
                0 & i & 0 & 0 & 0 & 0 & 0 & 0 & 0 \\
                0 & 0 & 0 & 0 & 0 & 0 & 0 & 1 & 0 \\
                0 & 0 & 0 & 0 & i & 0 & 0 & 0 & 0}. $$
Furthermore $\sigma(T_1(s)) = T_1(\ov s^{-1})$. Hence
$\Ho^1 \Zm_{q,e}=\{[1],[T_1(-1)]\}$.
\item $-e_{159}+ie_{168}$ is a real element in $\YY$.
  Here
  $$ g_0=\SmallMatrix{ 0 & 0 & -1 & -i & 0 & 0 & 0 & 0 & 0 \\
                0 & i & 0 & 0 & 0 & 0 & 1 & 0 & 0 \\
                0 & 0 & 0 & 0 & -\tfrac{1}{2}i & 0 & 0 & 0 & \tfrac{1}{2} \\
                -\tfrac{1}{2} & 0 & 0 & 0 & 0 & 0 & 0 & 0 & 0 \\
                0 & 0 & i & 1 & 0 & 0 & 0 & 0 & 0 \\
                0 & 0 & 0 & 0 & 0 & i & 0 & 0 & 0 \\
                0 & 1 & 0 & 0 & 0 & 0 & i & 0 & 0 \\
                0 & 0 & 0 & 0 & 0 & 0 & 0 & \tfrac{1}{2}i & 0 \\
                0 & 0 & 0 & 0 & 0 & 1 & 0 & 0 & -i},
  n_0=\SmallMatrix{ 1 & 0 & 0 & 0 & 0 & 0 & 0 & 0 & 0 \\
                0 & 0 & 0 & 0 & 0 & 0 & -i & 0 & 0 \\
                0 & 0 & 0 & -i & 0 & 0 & 0 & 0 & 0 \\
                0 & 0 & -i & 0 & 0 & 0 & 0 & 0 & 0 \\
                0 & 0 & 0 & 0 & 0 & 0 & 0 & 0 & i \\
                0 & 0 & 0 & 0 & 0 & -1 & 0 & 0 & 0 \\
                0 & -i & 0 & 0 & 0 & 0 & 0 & 0 & 0 \\
                0 & 0 & 0 & 0 & 0 & 0 & 0 & -1 & 0 \\
                0 & 0 & 0 & 0 & i & 0 & 0 & 0 & 0}. $$
Furthermore $\sigma(T_2(s,t)) = T_2(\ov s, \ov t^{-1})$. Hence
$\Ho^1 \Zm_{q,e}=\{[1],[T_2(1,-1)]\}$.
\item $-e_{159}-e_{267}$ is a real element in $\YY$.
  Here
  $$ g_0=\SmallMatrix{ 0 & 0 & 1 & -1 & 0 & 0 & 0 & 0 & 0 \\
                0 & -1 & 0 & 0 & 0 & 0 & 0 & 0 & 0 \\
                0 & 0 & 0 & 0 & -\tfrac{1}{2} & 0 & 0 & 0 & \tfrac{1}{2} \\
                -\tfrac{1}{2}i & 0 & 0 & 0 & 0 & 0 & 0 & 0 & 0 \\
                0 & 0 & -i & -i & 0 & 0 & 0 & 0 & 0 \\
                0 & 0 & 0 & 0 & 0 & -i & 0 & 0 & 0 \\
                0 & i & 0 & 0 & 0 & 0 & i & 0 & 0 \\
                0 & 0 & 0 & 0 & 0 & 0 & 0 & \tfrac{1}{2}i & 0 \\
                0 & 0 & 0 & 0 & -i & 0 & 0 & 0 & -i},
  n_0=\SmallMatrix{ -1 & 0 & 0 & 0 & 0 & 0 & 0 & 0 & 0 \\
                0 & 0 & 0 & 0 & 0 & 0 & -1 & 0 & 0 \\
                0 & 0 & 0 & -1 & 0 & 0 & 0 & 0 & 0 \\
                0 & 0 & -1 & 0 & 0 & 0 & 0 & 0 & 0 \\
                0 & 0 & 0 & 0 & 0 & 0 & 0 & 0 & -1 \\
                0 & 0 & 0 & 0 & 0 & -1 & 0 & 0 & 0 \\
                0 & -1 & 0 & 0 & 0 & 0 & 0 & 0 & 0 \\
                0 & 0 & 0 & 0 & 0 & 0 & 0 & -1 & 0 \\
                0 & 0 & 0 & 0 & -1 & 0 & 0 & 0 & 0}.$$
Furthermore $\sigma(T_2(s,t)) = T_2(\ov s, \ov s^{-1}\ov t^{-1})$. Hence
$\Ho^1 \Zm_{q,e}=1$.
\item $e_{249}-e_{357}$ is a real element in $\YY$.
  Here
  $$g_0=\SmallMatrix{ 0 & -1 & 0 & 0 & 0 & -1 & 0 & 0 & 0 \\
                1 & 0 & 0 & 0 & 0 & 0 & 0 & 0 & 1 \\
                0 & 0 & 0 & \tfrac{1}{2} & 0 & 0 & 0 & \tfrac{1}{2} & 0 \\
                0 & 0 & -\tfrac{1}{2}i & 0 & 0 & 0 & 0 & 0 & 0 \\
                0 & i & 0 & 0 & 0 & -i & 0 & 0 & 0 \\
                0 & 0 & 0 & 0 & i & 0 & 0 & 0 & 0 \\
                -i & 0 & 0 & 0 & 0 & 0 & 0 & 0 & i \\
                0 & 0 & 0 & 0 & 0 & 0 & \tfrac{1}{2}i & 0 & 0 \\
                0 & 0 & 0 & i & 0 & 0 & 0 & -i & 0},
  n_0=\SmallMatrix{ 0 & 0 & 0 & 0 & 0 & 0 & 0 & 0 & 1 \\
                0 & 0 & 0 & 0 & 0 & 1 & 0 & 0 & 0 \\
                0 & 0 & -1 & 0 & 0 & 0 & 0 & 0 & 0 \\
                0 & 0 & 0 & 0 & 0 & 0 & 0 & 1 & 0 \\
                0 & 0 & 0 & 0 & -1 & 0 & 0 & 0 & 0 \\
                0 & 1 & 0 & 0 & 0 & 0 & 0 & 0 & 0 \\
                0 & 0 & 0 & 0 & 0 & 0 & -1 & 0 & 0 \\
                0 & 0 & 0 & 1 & 0 & 0 & 0 & 0 & 0 \\
                1 & 0 & 0 & 0 & 0 & 0 & 0 & 0 & 0}. $$
Furthermore $\sigma(T_2(s,t)) = T_2(\ov s, \ov t)$. Hence
$\Ho^1 \Zm_{q,e}=1$.
\item $e_{159}$ is a real element in $\YY$.
  Here
  $$g_0=\SmallMatrix{ 0 & 0 & -1 & 1 & 0 & 0 & 0 & 0 & 0 \\
                0 & -1 & 0 & 0 & 0 & 0 & 1 & 0 & 0 \\
                0 & 0 & 0 & 0 & -\tfrac{1}{2} & 0 & 0 & 0 & \tfrac{1}{2} \\
                \tfrac{1}{2}i & 0 & 0 & 0 & 0 & 0 & 0 & 0 & 0 \\
                0 & 0 & i & i & 0 & 0 & 0 & 0 & 0 \\
                0 & 0 & 0 & 0 & 0 & i & 0 & 0 & 0 \\
                0 & i & 0 & 0 & 0 & 0 & i & 0 & 0 \\
                0 & 0 & 0 & 0 & 0 & 0 & 0 & \tfrac{1}{2}i & 0 \\
                0 & 0 & 0 & 0 & -i & 0 & 0 & 0 & -i},
  n_0=\SmallMatrix{ -1 & 0 & 0 & 0 & 0 & 0 & 0 & 0 & 0 \\
                0 & 0 & 0 & 0 & 0 & 0 & -1 & 0 & 0 \\
                0 & 0 & 0 & -1 & 0 & 0 & 0 & 0 & 0 \\
                0 & 0 & -1 & 0 & 0 & 0 & 0 & 0 & 0 \\
                0 & 0 & 0 & 0 & 0 & 0 & 0 & 0 & -1 \\
                0 & 0 & 0 & 0 & 0 & -1 & 0 & 0 & 0 \\
                0 & -1 & 0 & 0 & 0 & 0 & 0 & 0 & 0 \\
                0 & 0 & 0 & 0 & 0 & 0 & 0 & -1 & 0 \\
                0 & 0 & 0 & 0 & -1 & 0 & 0 & 0 & 0}. $$
Furthermore $\sigma(T_3(s,t,u)) = T_3(\ov u, \ov t, \ov s)$. Hence
$\Ho^1 \Zm_{q,e}=1$.
\item $ie_{168}$ is a real element in $\YY$.
  Here
  $$g_0=\SmallMatrix{ 0 & 0 & -1 & -i & 0 & 0 & 0 & 0 & 0 \\
                0 & i & 0 & 0 & 0 & 0 & 1 & 0 & 0 \\
                0 & 0 & 0 & 0 & -\tfrac{1}{2}i & 0 & 0 & 0 & \tfrac{1}{2} \\
                -\tfrac{1}{2} & 0 & 0 & 0 & 0 & 0 & 0 & 0 & 0 \\
                0 & 0 & i & 1 & 0 & 0 & 0 & 0 & 0 \\
                0 & 0 & 0 & 0 & 0 & i & 0 & 0 & 0 \\
                0 & 1 & 0 & 0 & 0 & 0 & i & 0 & 0 \\
                0 & 0 & 0 & 0 & 0 & 0 & 0 & \tfrac{1}{2}i & 0 \\
                0 & 0 & 0 & 0 & 1 & 0 & 0 & 0 & -i},
  n_0= \SmallMatrix{ 1 & 0 & 0 & 0 & 0 & 0 & 0 & 0 & 0 \\
                0 & 0 & 0 & 0 & 0 & 0 & -i & 0 & 0 \\
                0 & 0 & 0 & -i & 0 & 0 & 0 & 0 & 0 \\
                0 & 0 & -i & 0 & 0 & 0 & 0 & 0 & 0 \\
                0 & 0 & 0 & 0 & 0 & 0 & 0 & 0 & i \\
                0 & 0 & 0 & 0 & 0 & -1 & 0 & 0 & 0 \\
                0 & -i & 0 & 0 & 0 & 0 & 0 & 0 & 0 \\
                0 & 0 & 0 & 0 & 0 & 0 & 0 & -1 & 0 \\
                0 & 0 & 0 & 0 & i & 0 & 0 & 0 & 0} $$
Furthermore $\sigma(T_3(s,t,u)) = T_3(\ov s, \ov t, \ov s^{-2} \ov t^{-2}
\ov u^{-1})$. Hence $\Ho^1 \Zm_{q,e}=\{[1],[T(1,1,-1)]\}$.
\end{enumerate}
\end{subsec}

\begin{subsec} Let $p\in \Fm_4$, $p=\ov p$. Then $\Zm_{\Gtil_0}(p)$ has five
nonzero nilpotent orbits in $\z_{\g^\cC}(p)\cap \g_1^\cC$.

We have $\Zm_{\Gtil_0}(p)= \cup_{i,j=0}^2 h_1^i h_2^j \Zm^\circ$. Moreover the
elements of $\Zm^\circ$ are block diagonal with equal $3\times 3$-blocks.
Write $\mathfrak{a} = \z_{\g^\cC}(p)$ which inherits the grading from $\g^\cC$.
Let $e\in \mathfrak{a}_1$ be a nilpotent element, and $(h,e,f)$ be a
homogeneous $\ssl_2$-triple in $\mathfrak{a}$. Let $C$ denote the stabilizer
of this triple in $\Zm^\circ$. Then also the elements of
$C$ are block diagonal with equal blocks. Below we
describe $C$ by just giving such a block. A computation shows that
$h_1,h_2$ act as the identity on $\mathfrak{a}_1$. This implies that the
stabilizer $\Zm_{p,e}$ of the triple in $\Zm_{\Gtil_0}(p)$ is
$\cup_{i,j=0}^2 h_1^i h_2^j C$.

In Case (3) we have $\#\Ho^1 C =2$ and in the other cases $\Ho^1 C =1$. Hence by
Proposition \ref{p:C-3} we have that $\Ho^1 \Zm_{p,e}$ consists of the classes of
the same elements as $\Ho^1 C$.
\begin{enumerate}
\item  $e=e_{149}+e_{167}+e_{258}+e_{347}$. Here $C=\mu_3$.
\item $e=e_{149}+e_{158}+e_{167}+e_{248}+e_{257}+e_{347}$. Here $C$ is $\mu_3$.
\item $e=e_{147}+e_{258}$. Here $C$ is the union of two sets with respective
  blocks
  $$X(\delta,\epsilon)=\SmallMatrix{ \delta^2\epsilon^2 & 0 & 0 \\ 0 & \delta & 0\\
    0 & 0 & \epsilon } \text{ and }
 Y(\delta,\epsilon)= \SmallMatrix{ 0 & \delta^2\epsilon^2 & 0 \\ \delta & 0 & 0\\
    0 & 0 & -\epsilon}$$
  where $\delta^3=\epsilon^3=1$.
  So $|C|=18$. Hence by Lemma \ref{l:explicit} $\Ho^1 C = \{ [1], [Y(1,1)] \}$.
\item $e=e_{148}+e_{157}+e_{247}$. Here the $3\times 3$-block is
  $\diag(t,t^{-2},t)$, where $t\in \C^*$.
\item $e=e_{147}$. The $3\times 3$-block is
  $$\SmallMatrix{ \delta & 0 & 0 \\ 0 & a_{22} & a_{23} \\ 0 & a_{32} & a_{33}}$$
  with $\delta^3=1$ and the determinant of the $2\times 2$-block is 1 as well.
\end{enumerate}
\end{subsec}

\begin{subsec} Let $p\in \Fm_5$, $p=\ov p$. Then $\Zm_{\Gtil_0}(p)$ has seventeen
nonzero nilpotent orbits in $\z_{\g^\cC}(p)\cap \g_1^\cC$.

Here $\Zm_{\Gtil_0}(p)^\circ$ is equal to $H\cdot T_2$, where $H$ is
block diagonal with three equal $3\times 3$-blocks of determinant 1,
and $T_2$ is a 2-dimensional
central torus. Let $e\in \z_{\g^\cC}(p)\cap \g_1^\cC$ be nilpotent, contained
in the homogeneous $\ssl_2$-triple $(h,e,f)$.
Let $C$ denote the stabilizer of $(h,e,f)$ in
$\Zm_{\Gtil_0}(p)^\circ$. It turns out that in all cases $C =
\widehat{H} \cdot \widehat{T}_2$ where $\widehat{H}$ is a subgroup of $H$ and
$\widehat{T}_2$ is a subgroup of $T_2$. We describe
$\widehat{H}$ by just giving one block denoted $B$.
The elements of $T_2$ are of the form
$$T_2(s,t) = \diag(s,s,s,t,t,t,(st)^{-1},(st)^{-1},(st)^{-1})$$
so we describe $\widehat{T}_2$ by giving the general form of its elements,
as $T_2(a,b)$.

We have $\Zm_{\Gtil_0}(p) = \cup_{i=0}^2 h_1^i \Zm_{\Gtil_0}(p)^\circ$.
It turns out that in all cases we have two possibilities: either
$\Zm_{p,e}=C$ or $\Zm_{p,e} =\cup_{i=0}^2 h_1^i C$. In the next list we
only describe $\Ho^1 \Zm_{p,e}$ when it is nontrivial.
\begin{enumerate}
\item $e=e_{123}+e_{149}+e_{167}+e_{258}+e_{347}+e_{456}$.  Here
  $B=\diag(\zeta,\zeta,\zeta)$, $\zeta^3=1$ and $\widehat{T}_2 =
  T_2(\epsilon,\delta)$ where $\epsilon^3=\delta^3=1$. Hence $\widehat{H}
  \subset \widehat{T}_2$ so that $C=\widehat{T}_2$, which is a group of order 9.
  We have $\Zm_{p,e}=C$.
\item $e=e_{123}+e_{149}+e_{158}+e_{167}+e_{248}+e_{257}+e_{347}+e_{456}$.
  Here $\Zm_{p,e}$ is exactly the same as in the previous case.
\item  $e=e_{123}+e_{149}+e_{167}+e_{258}+e_{347}$. Here
  $B=\diag(\zeta,\zeta,\zeta)$, $\zeta^3=1$ and $\widehat{T}_2 =
  T_2(\delta,\delta^2 t)$, where $\delta^3=1$, $t\in \C^\times$. Again we
  see that $C=\widehat{T}_2$. We have $\Zm_{p,e}=C$.
\item $e=e_{123}+e_{147}+e_{258}+e_{456}$. In this case $\widehat{H}$ is the
  union of two sets
  with respective blocks $\diag(\delta,\epsilon,\delta^2\epsilon^2)$ and
  $$Y(\delta,\epsilon)=\SmallMatrix{ 0 & \delta^2\epsilon^2 & 0 \\ \delta & 0 & 0 \\
    0 & 0 & -\epsilon},$$
  where $\delta^3=\epsilon^3=1$. Furthermore,
  $\widehat{T}_2= T_2(\zeta,\xi)$ with $\zeta^3=\xi^3=1$. We have $\Zm_{p,e}=C$.
  The group $C$ has order 54 ($|\widehat{H}|=18$, $|\widehat{T}_2|=9$ and
  $\widehat{H}\cap \widehat{T}_2 = \mu_3$). Hence by Lemma \ref{l:explicit}
  $\Ho^1 \Zm_{p,e} = \{ [1], [a] \}$, where $a\in \widehat{H}$ has
  $3\times 3$-block $Y(1,1)$.
\item $e=e_{123}+e_{149}+e_{158}+e_{167}+e_{248}+e_{257}+e_{347}$.
  Here $B=\diag(\zeta,\zeta,\zeta)$, $\zeta^3=1$ and $\widehat{T}_2 =
  T_2(\delta,\delta^2 t)$, where $\delta^3=1$, $t\in \C^\times$. Hence
  $\widehat{H}\subset \widehat{T}_2$ so that $C=\widehat{T}_2$.
  We have $\Zm_{p,e}=C$.
\item $e=e_{149}+e_{167}+e_{258}+e_{347}$. Here $B=\diag(\zeta,\zeta,\zeta)$,
  $\zeta^3=1$ and $\widehat{T}_2 = T_2(s,t)$, where $s,t\in \C^\times$. Also here
  we see that $C=\widehat{T}_2$.   We have
  $\Zm_{p,e} = \cup_{i=0}^2 h_1^i C$.
\item $e=e_{123}+e_{148}+e_{157}+e_{247}+e_{456}$. Here we have $B=\diag(u,u^{-2},u)$,
  $u\in \C^\times$, and $\widehat{T}_2=T_2(\delta,\epsilon)$ with $\delta^3=
  \epsilon^3=1$. Furthermore $\Zm_{p,e}=C$.
\item $e=e_{123}+e_{147}+e_{258}$. In this case $\widehat{H}$ is the same as in
  case (4). Furthermore  $\widehat{T}_2=T_2(\zeta,\zeta^2 t)$,
  where $\zeta^3=1$, $t\in \C^\times$. We have $\Zm_{p,e}=C$.

  Furthermore, $\widehat{H}\cap \widehat{T}_2 = \mu_3$. Hence
  $\Zm_{p,e} = \widehat{H}\times S$, where $S$ is a 1-dimensional torus
  consisting of $T_2(1,t)$, where $t\in \C^\times$.
  Since $\Ho^1 S=1$ we have that $\Ho^1 \widehat{H} = \{ [1], [a]\}$ and
  hence $\Ho^1 \Zm_{p,e} = \{[1],[a]\}$, where $a$ is as in case (4).
\item $e=e_{149}+e_{158}+e_{167}+e_{248}+e_{257}+e_{347}$. Here $B=\diag(\zeta,
  \zeta,\zeta)$, $\zeta^3=1$, $\widehat{T}_2 = T_2(s,t)$, where
  $s,t\in \C^\times$.
  So $\widehat{H}\subset \widehat{T}_2$.
  We have $\Zm_{p,e} = \cup_{i=0}^2 h_1^i C$.
\item $e=e_{123}+e_{148}+e_{157}+e_{247}$. Here we have $B=\diag(u,u^{-2},u)$,
  $u\in \C^\times$, and $\widehat{T}_2=T_2(\delta,\delta^2 t)$, where
  $\delta^3=1$, $t\in \C^\times$.  We have $\Zm_{p,e}=C$.
\item $e=e_{147}+e_{258}$. In this case $\widehat{H}$ is the same as in
  case (4). Furthermore $\widehat{T}_2 = T_2(s,t)$,
  where $s,t\in \C^\times$. We have $\Zm_{p,e} = \cup_{i=0}^2 h_1^i C$.

  In the same way as in case (8) we see that $\Ho^1 \Zm_{p,e} = \{ [1], [a]\}$
  where $a$ is as in case (4).
\item $e=e_{123}+e_{147}+e_{456}$. Here
  $$B = \SmallMatrix{ \zeta^2 & 0 & 0 \\ 0 & a_{22} & a_{23} \\
  0 & a_{32} & a_{33} } \text{ with } a_{22}a_{33}-a_{23}a_{32}=\zeta
  \text{ and } \zeta^3 = 1,$$
  and $\widehat{T}_2 =T_2(\delta,\epsilon)$ with $\delta^3=\epsilon^3=1$.
  We have  $\Zm_{p,e} = \cup_{i=0}^2 h_1^i C$.
\item $e=e_{148}+e_{157}+e_{247}$. We have $B=\diag(u,u^{-2},u)$, $u\in \C^\times$,
  and $\widehat{T}_2 = T_2(s,t)$, where $s,t\in \C^\times$. We have
  $\Zm_{p,e} = \cup_{i=0}^2 h_1^i C$.
\item $e=e_{123}+e_{147}$.  Here $B$ is as in case (12) and
  $\widehat{T}_2 = T_2(\delta,\delta^2 t)$, where $\delta^3=1$. We have
  $\Zm_{p,e}=C$.
\item $e=e_{147}$. Here $B$ is as in case (12) and $\widehat{T}_2 = T_2(s,t)$
  where $s,t\in \C^\times$. We have
  $\Zm_{p,e} = \cup_{i=0}^2 h_1^i C$.
\item $e=e_{123}+e_{456}$. Here
  $$B = \SmallMatrix{ a_{11} & a_{12} & a_{13}\\ a_{21} & a_{22} & a_{23}\\
  a_{31} & a_{32} & a_{33}} \text{ with } \det(B)=1$$
  $\widehat{T}_2 = T_2(\delta,\epsilon)$ with $\delta^3=\epsilon^3=1$.
  We have $\Zm_{p,e}=C$.
\item $e=e_{123}$. Here $B$ is as in the previous case
  and $\widehat{T}_2 = T_2(\delta,\delta^2 t)$, where $\delta^3=1$,
  $t\in \C^\times$.  We have $\Zm_{p,e}=C$.
\end{enumerate}
Let $p$ be a real semisimple element conjugate to $q\in\Fm_5$ with
$\ov q = -q$. Then the map $\mu$ is given by $\mu(y) = n_1 \ov y$.
\begin{enumerate}
\item $-ie_{148}-ie_{149}-ie_{157}-ie_{167}-ie_{247}-e_{258}+e_{259}+e_{268}-e_{269}-
  ie_{347}+e_{358}-e_{359}-e_{368}+e_{369}-e_{456}+e_{789}$ is a real element in $\YY$.
  In this case $\Zm_{q,e}$ is of order 9. Hence $\Ho^1 \Zm_{q,e}=1$.
\item $-\tfrac{1}{2}e_{148}+\tfrac{1}{2}e_{149}-\tfrac{1}{2}e_{157}+
  \tfrac{1}{2}e_{167}-\tfrac{1}{2}e_{247}-\tfrac{3}{4}ie_{258}+
  \tfrac{1}{4}ie_{259}+\tfrac{1}{4}ie_{268}+\tfrac{1}{4}ie_{269}+
  \tfrac{1}{2}2e_{347}+\tfrac{1}{4}ie_{358}+\tfrac{1}{4}ie_{359}+
  \tfrac{1}{4}ie_{368}-\tfrac{3}{4}ie_{369}-e_{456}+e_{789}$ is a real element in
  $\YY$.
  In this case $\Zm_{q,e}$ is of order 9. Hence $\Ho^1 \Zm_{q,e}=1$.
\item $ie_{123}+e_{148}-e_{149}+e_{157}-e_{167}+e_{247}+ie_{258}+ie_{259}+ie_{268}+
  ie_{269}-e_{347}+ie_{358}+ie_{359}+ie_{368}+ie_{369}$ is a real element in $\YY$.
  We have
  $$g_0=\SmallMatrix{ 2 & 0 & 0 & 0 & 0 & 0 & 0 & 0 & 0 \\
                0 & -1 & 0 & 0 & 0 & 0 & 0 & 0 & 0 \\
                0 & 0 & 0 & 2i & 0 & 0 & -2 & 0 & 0 \\
                0 & 0 & 0 & 0 & -\tfrac{1}{2}i & \tfrac{1}{8} & 0 & \tfrac{1}{2} & -\tfrac{1}{8}i \\
                0 & 0 & 0 & 0 & -i & -\tfrac{1}{4} & 0 & 1 & \tfrac{1}{4}i \\
                0 & 0 & \tfrac{1}{4} & 0 & 0 & 0 & 0 & 0 & 0 \\
                0 & 0 & 0 & 2 & 0 & 0 & -2i & 0 & 0 \\
                0 & 0 & 0 & 0 & -1 & -\tfrac{1}{4}i & 0 & i & \tfrac{1}{4} \\
                0 & 0 & 0 & 0 & -\tfrac{1}{2} & \tfrac{1}{8}i & 0 & \tfrac{1}{2}i & -\tfrac{1}{8}},
  n_0=\SmallMatrix{ 1 & 0 & 0 & 0 & 0 & 0 & 0 & 0 & 0 \\
                    0 & 1 & 0 & 0 & 0 & 0 & 0 & 0 & 0 \\
                    0 & 0 & 1 & 0 & 0 & 0 & 0 & 0 & 0 \\
                    0 & 0 & 0 & 0 & 0 & 0 & i & 0 & 0 \\
                    0 & 0 & 0 & 0 & 0 & 0 & 0 & i & 0 \\
                    0 & 0 & 0 & 0 & 0 & 0 & 0 & 0 & i \\
                    0 & 0 & 0 & i & 0 & 0 & 0 & 0 & 0 \\
                    0 & 0 & 0 & 0 & i & 0 & 0 & 0 & 0 \\
                    0 & 0 & 0 & 0 & 0 & i & 0 & 0 & 0}.
  $$
  The identity component $\Zm_{q,e}^\circ$ consists of $T_2(1,t)$ for
  $t\in \C^\times$. The conjugation $\sigma$ satisfies
  $\sigma(T_2(1,t)) = T_2(1,\ov t^{-1})$. Hence $\Ho^1 \Zm_{q,e}^\circ =
  \{[1],[T_2(1,-1)]\}$. Moreover, by Proposition \ref{p:C-3} also
  $\Ho^1 \Zm_{q,e}=\{[1],[T_2(1,-1)]\}$.
\item $-e_{258}+e_{369}-e_{456}+e_{789}$ is a real element in $\YY$. We have
$$g_0=\SmallMatrix{ 0 & 0 & 0 & 0 & 0 & 0 & 0 & 0 & -1 \\
                0 & 0 & 0 & 0 & 0 & 0 & -\tfrac{1}{2} & \tfrac{1}{2} & 0 \\
                0 & 0 & -1 & 0 & 0 & 1 & 0 & 0 & 0 \\
                -\tfrac{1}{2} & 0 & 0 & 0 & -\tfrac{1}{2} & 0 & 0 & 0 & 0 \\
                0 & 1 & 0 & 1 & 0 & 0 & 0 & 0 & 0 \\
                0 & 0 & 0 & 0 & 0 & 0 & -\tfrac{1}{2}i & -\tfrac{1}{2}i & 0 \\
                0 & 0 & i & 0 & 0 & i & 0 & 0 & 0 \\
                i & 0 & 0 & 0 & -i & 0 & 0 & 0 & 0 \\
                0 & -\tfrac{1}{2}i & 0 & \tfrac{1}{2}i & 0 & 0 & 0 & 0 & 0},
n_0=\SmallMatrix{ 0 & 0 & 0 & 0 & 1 & 0 & 0 & 0 & 0 \\
                0 & 0 & 0 & 1 & 0 & 0 & 0 & 0 & 0 \\
                0 & 0 & 0 & 0 & 0 & -1 & 0 & 0 & 0 \\
                0 & 1 & 0 & 0 & 0 & 0 & 0 & 0 & 0 \\
                1 & 0 & 0 & 0 & 0 & 0 & 0 & 0 & 0 \\
                0 & 0 & -1 & 0 & 0 & 0 & 0 & 0 & 0 \\
                0 & 0 & 0 & 0 & 0 & 0 & 0 & -1 & 0 \\
                0 & 0 & 0 & 0 & 0 & 0 & -1 & 0 & 0 \\
                0 & 0 & 0 & 0 & 0 & 0 & 0 & 0 & 1}.
  $$
In this case $\Zm_{q,e}$ is of order 54. Let $M(\delta,\epsilon)$ denote
the matrix with $3\times 3$-block equal to $Y(\delta,\epsilon)$. Then
$n_0 \overline{M(\delta,\epsilon)} n_0^{-1} = M(\ov \delta^2\ov\epsilon^2,
\ov \epsilon)$. Hence $M(1,1)$ is invariant under the conjugation $\sigma$
and by Lemma \ref{l:explicit} it follows that $\Ho^1 \Zm_{q,e} =
\{[1],[M(1,1)]\}$.
\item $ie_{123}-ie_{148}-ie_{149}-ie_{157}+\tfrac{1}{4}ie_{158}-\tfrac{1}{4}ie_{159}-
  ie_{167}-\tfrac{1}{4}ie_{168}+\tfrac{1}{4}ie_{169}-ie_{247}+\tfrac{1}{4}ie_{248}-
  \tfrac{1}{4}ie_{249}+\tfrac{1}{4}ie_{257}-\tfrac{1}{4}ie_{267}-ie_{347}-
  \tfrac{1}{4}ie_{348}+\tfrac{1}{4}ie_{349}-\tfrac{1}{4}ie_{357}+
  \tfrac{1}{4}ie_{367}$ is a real element in $\YY$. We have
  $$g_0=\SmallMatrix{ 1 & 0 & 0 & 0 & 0 & 0 & 0 & 0 & 0 \\
                0 & 0 & 1 & 0 & 0 & 0 & 0 & 0 & 0 \\
                0 & 0 & 0 & i & 0 & 0 & -1 & 0 & 0 \\
                0 & 0 & 0 & 0 & \tfrac{1}{4} & \tfrac{1}{2}i & 0 & -\tfrac{1}{4}i & -\tfrac{1}{2} \\
                0 & 0 & 0 & 0 & -\tfrac{1}{2} & i & 0 & \tfrac{1}{2}i & -1 \\
                0 & \tfrac{1}{2} & 0 & 0 & 0 & 0 & 0 & 0 & 0 \\
                0 & 0 & 0 & 1 & 0 & 0 & -i & 0 & 0 \\
                0 & 0 & 0 & 0 & -\tfrac{1}{2}i & 1 & 0 & \tfrac{1}{2} & -i \\
                0 & 0 & 0 & 0 & \tfrac{1}{4}i & \tfrac{1}{2} & 0 & -\tfrac{1}{4} & -\tfrac{1}{2}i},
  n_0=\SmallMatrix{ 1 & 0 & 0 & 0 & 0 & 0 & 0 & 0 & 0 \\
                0 & 1 & 0 & 0 & 0 & 0 & 0 & 0 & 0 \\
                0 & 0 & 1 & 0 & 0 & 0 & 0 & 0 & 0 \\
                0 & 0 & 0 & 0 & 0 & 0 & i & 0 & 0 \\
                0 & 0 & 0 & 0 & 0 & 0 & 0 & i & 0 \\
                0 & 0 & 0 & 0 & 0 & 0 & 0 & 0 & i \\
                0 & 0 & 0 & i & 0 & 0 & 0 & 0 & 0 \\
                0 & 0 & 0 & 0 & i & 0 & 0 & 0 & 0 \\
                0 & 0 & 0 & 0 & 0 & i & 0 & 0 & 0}. $$
  The identity component $\Zm_{q,e}^\circ$ consists of $T_2(1,t)$, $t\in \C^\times$.
  We have $\sigma(T_2(1,t)) = T_2(1,\ov t^{-1})$. Hence $\Ho^1 \Zm_{q,e}^\circ
  = \{[1],[T_2(1,-1)]\}$. The component group is of order 3, so by
  Proposition \ref{p:C-3} also $\Ho^1 \Zm_{q,e}= \{[1],[T_2(1,-1)]\}$.
\item $-ie_{148}-ie_{149}-ie_{157}-ie_{167}-ie_{247}-e_{258}+e_{259}+e_{268}-e_{269}-
  ie_{347}+e_{358}-e_{359}-e_{368}+e_{369}$ is a real element in $\YY$.
  We have
  $$g_0=\SmallMatrix{ 2i & 0 & 0 & 0 & 0 & 0 & 0 & 0 & 0 \\
                0 & 0 & \tfrac{1}{4}i & 0 & 0 & 0 & 0 & 0 & 0 \\
                0 & 0 & 0 & 2i & 0 & 0 & 2i & 0 & 0 \\
                0 & 0 & 0 & 0 & \tfrac{1}{2} & \tfrac{1}{8}i & 0 & -\tfrac{1}{2} & \tfrac{1}{8}i \\
                0 & 0 & 0 & 0 & -1 & \tfrac{1}{4}i & 0 & 1 & \tfrac{1}{4}i &\\
                0 & i & 0 & 0 & 0 & 0 & 0 & 0 & 0 \\
                0 & 0 & 0 & 2 & 0 & 0 & -2 & 0 & 0 \\
                0 & 0 & 0 & 0 & -i & \tfrac{1}{4} & 0 & -i & -\tfrac{1}{4} \\
                0 & 0 & 0 & 0 & \tfrac{1}{2}i & \tfrac{1}{8} & 0 & \tfrac{1}{2}i & -\tfrac{1}{8}},
  n_0=\SmallMatrix{ -1 & 0 & 0 & 0 & 0 & 0 & 0 & 0 & 0 \\
                0 & -1 & 0 & 0 & 0 & 0 & 0 & 0 & 0 \\
                0 & 0 & -1 & 0 & 0 & 0 & 0 & 0 & 0 \\
                0 & 0 & 0 & 0 & 0 & 0 &-1 & 0 & 0 \\
                0 & 0 & 0 & 0 & 0 & 0 & 0 & -1 & 0 \\
                0 & 0 & 0 & 0 & 0 & 0 & 0 & 0 & -1 \\
                0 & 0 & 0 & -1 & 0 & 0 & 0 & 0 & 0 \\
                0 & 0 & 0 & 0 & -1 & 0 & 0 & 0 & 0 \\
                0 & 0 & 0 & 0 & 0 & -1 & 0 & 0 & 0}. $$
  The identity component $\Zm_{q,e}^\circ$ consists of $T_2(s,t)$,
  $s,t\in \C^\times$. We have $\sigma(T_2(s,t)) =
  T_2(\ov s, \ov s^{-1}\ov t^{-1})$. Hence $\Ho^1 \Zm_{q,e}^\circ = 1$, and as
  the component group is of order 3, by
  Proposition \ref{p:C-3} also $\Ho^1 \Zm_{q,e}$ is trivial.
\item $e_{148}-e_{149}+e_{157}-e_{167}+e_{247}-e_{347}-e_{456}+e_{789}$ is a real
  element in $\YY$. We have
$$g_0=\SmallMatrix{ 0 & 0 & 0 & 0 & 0 & 0 & -2 & 0 & 0 \\
                0 & 0 & 0 & 0 & 0 & 0 & 0 & 0 & -1 \\
                -2 & 0 & 0 & 2 & 0 & 0 & 0 & 0 & 0 \\
                0 & \tfrac{1}{8} & -\tfrac{1}{2} & 0 & \tfrac{1}{8} & \tfrac{1}{2} & 0 & 0 & 0 \\
                0 & -\tfrac{1}{4} & -1 & 0 & -\tfrac{1}{4} & 1 & 0 & 0 & 0 \\
                0 & 0 & 0 & 0 & 0 & 0 & 0 & \tfrac{1}{4}i & 0 \\
                2i & 0 & 0 & 2i & 0 & 0 & 0 & 0 & 0 \\
                0 & -\tfrac{1}{4}i & i & 0 & \tfrac{1}{4}i & i & 0 & 0 & 0 \\
                0 & \tfrac{1}{8}i & \tfrac{1}{2}i & 0 & -\tfrac{1}{8}i & \tfrac{1}{2}i & 0 & 0 & 0},
  n_0=\SmallMatrix{ 0 & 0 & 0 & -1 & 0 & 0 & 0 & 0 & 0 \\
                0 & 0 & 0 & 0 & 1 & 0 & 0 & 0 & 0 \\
                0 & 0 & 0 & 0 & 0 & -1 & 0 & 0 & 0 \\
                -1 & 0 & 0 & 0 & 0 & 0 & 0 & 0 & 0 \\
                0 & 1 & 0 & 0 & 0 & 0 & 0 & 0 & 0 \\
                0 & 0 & -1 & 0 & 0 & 0 & 0 & 0 & 0 \\
                0 & 0 & 0 & 0 & 0 & 0 & 1 & 0 & 0 \\
                0 & 0 & 0 & 0 & 0 & 0 & 0 & -1 & 0 \\
                0 & 0 & 0 & 0 & 0 & 0 & 0 & 0 & 1}.
  $$
  The identity component $\Zm_{q,e}^\circ$ is a torus consisting of
  $S(u)$, $u\in \C^\times$. We have $\sigma(S(u)) =
  S(\ov u)$. Hence $\Ho^1 \Zm_{q,e}^\circ = 1$, and as
  the component group is of order 9, by
  Proposition \ref{p:C-3} also $\Ho^1 \Zm_{q,e}$ is trivial.
\item $ie_{123}+ie_{258}+ie_{369}$ is a real element in $\YY$. We have
  $$g_0= \SmallMatrix{ 0 & 0 & -i & 0 & 0 & 0 & 0 & 0 & 0 \\
                -\tfrac{1}{2} & -\tfrac{1}{2} & 0 & 0 & 0 & 0 & 0 & 0 & 0 \\
                0 & 0 & 0 & 0 & 0 & i & 0 & 0 & 1 \\
                0 & 0 & 0 & \tfrac{1}{2} & 0 & 0 & 0 & -\tfrac{1}{2}i & 0 \\
                0 & 0 & 0 & 0 & 1 & 0 & -i & 0 & 0 \\
                -\tfrac{1}{2}i & \tfrac{1}{2}i & 0 & 0 & 0 & 0 & 0 & 0 & 0 \\
                0 & 0 & 0 & 0 & 0 & 1 & 0 & 0 & i \\
                0 & 0 & 0 & -i & 0 & 0 & 0 & 1 & 0 \\
                0 & 0 & 0 & 0 & -\tfrac{1}{2}i & 0 & \tfrac{1}{2} & 0 & 0},
n_0 = \SmallMatrix{ 0 & 1 & 0 & 0 & 0 & 0 & 0 & 0 & 0 \\
                1 & 0 & 0 & 0 & 0 & 0 & 0 & 0 & 0 \\
                0 & 0 & -1 & 0 & 0 & 0 & 0 & 0 & 0 \\
                0 & 0 & 0 & 0 & 0 & 0 & 0 & i & 0 \\
                0 & 0 & 0 & 0 & 0 & 0 & i & 0 & 0 \\
                0 & 0 & 0 & 0 & 0 & 0 & 0 & 0 & -i \\
                0 & 0 & 0 & 0 & i & 0 & 0 & 0 & 0 \\
                0 & 0 & 0 & i & 0 & 0 & 0 & 0 & 0 \\
                0 & 0 & 0 & 0 & 0 & -i & 0 & 0 & 0}.
$$
We have $\Zm_{q,e} = \widehat{H}\times S$, where $\widehat{H}$ is of order 18 and
$S$ is a 1-dimensional torus consisting of $T_2(1,t)$, where $t\in \C^\times$.
By Lemma \ref{l:explicit} $\Ho^1 \widehat{H} =\{[1],[a]\}$, where
$a\in \widehat{H}$ is of order 2 and fixed under the conjugation $\sigma$.
A small computation shows that the matrix $a$ with three $3\times 3$-blocks
$Y(1,1)$ satisfies these requirements. Secondly, $\sigma(T_2(1,t)) =
T_2(1,\ov t^{-1})$. Hence $\Ho^1 S = \{[1],[T_2(1,-1)]\}$. It follows that
$\Ho^1 \Zm_{q,e} =\{[1],[a],[T_2(1,-1)],[aT_2(1,-1)]\}$.
\item $-ie_{148}-ie_{149}-ie_{157}+\tfrac{1}{4}ie_{158}-\tfrac{1}{4}ie_{159}-
  ie_{167}-\tfrac{1}{4}ie_{168}+\tfrac{1}{4}ie_{169}-ie_{247}+\tfrac{1}{4}ie_{248}-
  \tfrac{1}{4}ie_{249}+\tfrac{1}{4}ie_{257}-\tfrac{1}{4}ie_{267}-ie_{347}-
  \tfrac{1}{4}ie_{348}+\tfrac{1}{4}ie_{349}-\tfrac{1}{4}ie_{357}+
  \tfrac{1}{4}ie_{367}$ is a real element in $\YY$. We have
  $$g_0= \SmallMatrix{ i & 0 & 0 & 0 & 0 & 0 & 0 & 0 & 0 \\
                0 & 0 & i & 0 & 0 & 0 & 0 & 0 & 0 \\
                0 & 0 & 0 & i & 0 & 0 & i & 0 & 0 \\
                0 & 0 & 0 & 0 & \tfrac{1}{4} & \tfrac{1}{2}i & 0 & -\tfrac{1}{4} & \tfrac{1}{2}i \\
                0 & 0 & 0 & 0 & -\tfrac{1}{2} & i & 0 & \tfrac{1}{2} & i \\
                0 & \tfrac{1}{2}i & 0 & 0 & 0 & 0 & 0 & 0 & 0 \\
                0 & 0 & 0 & 1 & 0 & 0 & -1 & 0 & 0 \\
                0 & 0 & 0 & 0 & -\tfrac{1}{2}i & 1 & 0 & -\tfrac{1}{2}i & -1 \\
                0 & 0 & 0 & 0 & \tfrac{1}{4}i & \tfrac{1}{2} & 0 & \tfrac{1}{4}i & -\tfrac{1}{2}},
n_0=\SmallMatrix{ -1 & 0 & 0 & 0 & 0 & 0 & 0 & 0 & 0 \\
                0 & -1 & 0 & 0 & 0 & 0 & 0 & 0 & 0 \\
                0 & 0 & -1 & 0 & 0 & 0 & 0 & 0 & 0 \\
                0 & 0 & 0 & 0 & 0 & 0 &-1 & 0 & 0 \\
                0 & 0 & 0 & 0 & 0 & 0 & 0 & -1 & 0 \\
                0 & 0 & 0 & 0 & 0 & 0 & 0 & 0 & -1 \\
                0 & 0 & 0 & -1 & 0 & 0 & 0 & 0 & 0 \\
                0 & 0 & 0 & 0 & -1 & 0 & 0 & 0 & 0 \\
                0 & 0 & 0 & 0 & 0 & -1 & 0 & 0 & 0}.
$$
Here $\Zm_{q,e}^\circ=\widehat{T_2}$ which consists of the elements $T_2(s,t)$,
$s,t\in \C^\times$. The component group is of order 3. We have
$\sigma(T_2(s,t)) = T_2(\ov s, \ov s^{-1} \ov t^{-1})$. It follows that
$\Ho^1 \widehat{T}_2 = 1$ and by Proposition \ref{p:C-3} also
$\Ho^1 \Zm_{q,e}$ is trivial.
\item $ie_{123}-ie_{148}-ie_{149}-ie_{157}-ie_{167}-ie_{247}-ie_{347}$ is a real
  element of $\YY$. We have
  $$g_0=\SmallMatrix{ -2 & 0 & 0 & 0 & 0 & 0 & 0 & 0 & 0 \\
                0 & \tfrac{1}{4} & 0 & 0 & 0 & 0 & 0 & 0 & 0 \\
                0 & 0 & 0 & 2 & 0 & 0 & -2i & 0 & 0 \\
                0 & 0 & 0 & 0 & -\tfrac{1}{8} & \tfrac{1}{2}i & 0 & \tfrac{1}{8}i & -\tfrac{1}{2} \\
                0 & 0 & 0 & 0 & -\tfrac{1}{4} & -i & 0 & \tfrac{1}{4}i & 1 \\
                0 & 0 & 1 & 0 & 0 & 0 & 0 & 0 & 0 \\
                0 & 0 & 0 & -2i & 0 & 0 & 2 & 0 & 0 \\
                0 & 0 & 0 & 0 & \tfrac{1}{4}i & 1 & 0 & -\tfrac{1}{4} & -i \\
                0 & 0 & 0 & 0 & \tfrac{1}{8}i & -\tfrac{1}{2} & 0 & -\tfrac{1}{8} & \tfrac{1}{2}i},
n_0=\SmallMatrix{ 1 & 0 & 0 & 0 & 0 & 0 & 0 & 0 & 0 \\
                0 & 1 & 0 & 0 & 0 & 0 & 0 & 0 & 0 \\
                0 & 0 & 1 & 0 & 0 & 0 & 0 & 0 & 0 \\
                0 & 0 & 0 & 0 & 0 & 0 & i & 0 & 0 \\
                0 & 0 & 0 & 0 & 0 & 0 & 0 & i & 0 \\
                0 & 0 & 0 & 0 & 0 & 0 & 0 & 0 & i \\
                0 & 0 & 0 & i & 0 & 0 & 0 & 0 & 0 \\
                0 & 0 & 0 & 0 & i & 0 & 0 & 0 & 0 \\
                0 & 0 & 0 & 0 & 0 & i & 0 & 0 & 0}.
$$
Here $\Zm_{q,e}^\circ$ is a 2-dimensional torus consisting of diagonal elements
$S(u,t)$. The component group is of order 3. We have $\sigma(S(u,t)) =
S(\ov u, \ov t)$. Hence $\Ho^1 \Zm_{q,e}^\circ=1$ and by Proposition
\ref{p:C-3} also $\Ho^1 \Zm_{q,e}$ is trivial.
\item $-e_{258}+e_{369}$ is a real element in $\YY$. We have
  $$g_0=\SmallMatrix{ 0 & 0 & -1 & 0 & 0 & 0 & 0 & 0 & 0 \\
                -\tfrac{1}{2} & \tfrac{1}{2} & 0 & 0 & 0 & 0 & 0 & 0 & 0 \\
                0 & 0 & 0 & 0 & 0 & -1 & 0 & 0 & -1 \\
                0 & 0 & 0 & 0 & \tfrac{1}{2} & 0 & -\tfrac{1}{2} & 0 & 0 \\
                0 & 0 & 0 & -1 & 0 & 0 & 0 & 1 & 0 \\
                \tfrac{1}{2}i & \tfrac{1}{2}i & 0 & 0 & 0 & 0 & 0 & 0 & 0 \\
                0 & 0 & 0 & 0 & 0 & i & 0 & 0 & -i \\
                0 & 0 & 0 & 0 & -i & 0 & -i & 0 & 0 \\
                0 & 0 & 0 & \tfrac{1}{2}i & 0 & 0 & 0 & \tfrac{1}{2}i & 0},
n_0=\SmallMatrix{ 0 & -1 & 0 & 0 & 0 & 0 & 0 & 0 & 0 \\
                -1 & 0 & 0 & 0 & 0 & 0 & 0 & 0 & 0 \\
                0 & 0 & 1 & 0 & 0 & 0 & 0 & 0 & 0 \\
                0 & 0 & 0 & 0 & 0 & 0 & 0 & -1 & 0 \\
                0 & 0 & 0 & 0 & 0 & 0 & -1 & 0 & 0 \\
                0 & 0 & 0 & 0 & 0 & 0 & 0 & 0 & 1 \\
                0 & 0 & 0 & 0 & -1 & 0 & 0 & 0 & 0 \\
                0 & 0 & 0 & -1 & 0 & 0 & 0 & 0 & 0 \\
                0 & 0 & 0 & 0 & 0 & 1 & 0 & 0 & 0}.
  $$
Here $\Zm_{q,e}^\circ$ is a 2-dimensional torus consisting of elements
$T_2(s,t)$ for $s,t\in \C^\times$. We have $\sigma(T_2(s,t)) =
T_2(\ov s, \ov s^{-1} \ov t^{-1})$ so that $\Ho^1 \Zm_{q,e}^\circ = 1$.
The component group is of order $54 = 2\cdot 3^3$. Let $A(\epsilon,\delta)$
be the element with $3\times 3$-blocks $Y(\delta,\epsilon)$. Then
$\sigma(A(\delta,\epsilon)) = A(\ov\delta^2\ov\epsilon^2,\ov\epsilon^2)$.
So $\sigma(A(1,,1)) = A(1,1)$. Moreover $A(1,1)$ is of order 2. Hence
by Proposition \ref{p:C-3-2} we see that $\Ho^1 \Zm_{q,e} = \{ [1], [A(1,1)]\}$.
\item $-e_{258}+e_{259}+e_{268}-e_{269}+e_{358}-e_{359}-e_{368}+e_{369}-e_{456}+e_{789}$
  is a real element in $\YY$. We have
$$g_0=\SmallMatrix{ 0 & 0 & 0 & 0 & 0 & 0 & 0 & -\tfrac{1}{2} & 0\\
                0 & 0 & 0 & 0 & 0 & 0 & 0 & 0 & -1 \\
                0 & -\tfrac{1}{2} & 0 & 0 & \tfrac{1}{2} & 0 & 0 & 0 & 0 \\
                -\tfrac{1}{2} & 0 & -\tfrac{1}{2} & -\tfrac{1}{2} & 0 & \tfrac{1}{2} & 0 & 0 & 0 \\
                1 & 0 & -1 & 1 & 0 & 1 & 0 & 0 & 0 \\
                0 & 0 & 0 & 0 & 0 & 0 & -i & 0 & 0 \\
                0 & \tfrac{1}{2}i & 0 & 0 & \tfrac{1}{2}i & 0 & 0 & 0 & 0 \\
                i & 0 & i & -i & 0 & i & 0 & 0 & 0 \\
                -\tfrac{1}{2}i & 0 & \tfrac{1}{2}i & \tfrac{1}{2}i & 0 & \tfrac{1}{2}i & 0 & 0 & 0},
n_0=\SmallMatrix{ 0 & 0 & 0 & 1 & 0 & 0 & 0 & 0 & 0 \\
                0 & 0 & 0 & 0 & -1 & 0 & 0 & 0 & 0 \\
                0 & 0 & 0 & 0 & 0 & -1 & 0 & 0 & 0 \\
                1 & 0 & 0 & 0 & 0 & 0 & 0 & 0 & 0 \\
                0 & -1 & 0 & 0 & 0 & 0 & 0 & 0 & 0 \\
                0 & 0 & -1 & 0 & 0 & 0 & 0 & 0 & 0 \\
                0 & 0 & 0 & 0 & 0 & 0 & -1 & 0 & 0 \\
                0 & 0 & 0 & 0 & 0 & 0 & 0 & 1 & 0 \\
                0 & 0 & 0 & 0 & 0 & 0 & 0 & 0 & 1}.
  $$
The identity component $\Zm_{q,e}^\circ$ consists of matrices $M(B)$ that are block
diagonal with equal $3\times 3$-blocks
$$B=\SmallMatrix{1&0&0\\0&a_{22}&a_{23}\\0&a_{32}&a_{33}} \text{ with }
a_{22}a_{33}-a_{23}a_{32}=1.$$
We have $\sigma( M(B) ) = M(\overline{B})$. Hence $\Ho^1 \Zm_{q,e}^\circ=1$.
The component group is of order 27, so by Proposition \ref{p:C-3}
we see that $\Ho^1 \Zm_{q,e}=1$.
\item $e_{148}-e_{149}+e_{157}-e_{167}+e_{247}-e_{347}$ is a real element in $\YY$.
  We have
$$g_0=\SmallMatrix{ 2 & 0 & 0 & 0 & 0 & 0 & 0 & 0 & 0 \\
                0 & 0 & 1 & 0 & 0 & 0 & 0 & 0 & 0 \\
                0 & 0 & 0 & 2 & 0 & 0 & 2 & 0 & 0 \\
                0 & 0 & 0 & 0 & -\tfrac{1}{8} & \tfrac{1}{2} & 0 & \tfrac{1}{8} & \tfrac{1}{2} \\
                0 & 0 & 0 & 0 & \tfrac{1}{4} & 1 & 0 & -\tfrac{1}{4} & 1 \\
                0 & -\tfrac{1}{4}i & 0 & 0 & 0 & 0 & 0 & 0 & 0 \\
                0 & 0 & 0 & -2i & 0 & 0 & 2i & 0 & 0 \\
                0 & 0 & 0 & 0 & \tfrac{1}{4}i & -i & 0 & \tfrac{1}{4}i & i \\
                0 & 0 & 0 & 0 & -\tfrac{1}{8}i & -\tfrac{1}{2}i & 0 & -\tfrac{1}{8}i & \tfrac{1}{2}i},
n_0= \SmallMatrix{ 1 & 0 & 0 & 0 & 0 & 0 & 0 & 0 & 0 \\
                0 & -1 & 0 & 0 & 0 & 0 & 0 & 0 & 0 \\
                0 & 0 & 1 & 0 & 0 & 0 & 0 & 0 & 0 \\
                0 & 0 & 0 & 0 & 0 & 0 & 1 & 0 & 0 \\
                0 & 0 & 0 & 0 & 0 & 0 & 0 & -1 & 0 \\
                0 & 0 & 0 & 0 & 0 & 0 & 0 & 0 & 1 \\
                0 & 0 & 0 & 1 & 0 & 0 & 0 & 0 & 0 \\
                0 & 0 & 0 & 0 & -1 & 0 & 0 & 0 & 0 \\
                0 & 0 & 0 & 0 & 0 & 1 & 0 & 0 & 0}.
$$
Here $\Zm_{q,e}^\circ$ is a torus consisting of elements $Q(u,s,t)$. We have
$\sigma(Q(u,s,t)) = Q(\ov u,\ov s,\ov s^{-1}\ov t^{-1})$. It follows that
$\Ho^1  \Zm_{q,e}^\circ=1$. The component group is of order 3, so by
Proposition \ref{p:C-3} we see that $\Ho^1 \Zm_{q,e}=1$.
\item $ie_{123}+ie_{147}$ is a real element in $\YY$. We have
  $$g_0=\SmallMatrix{ -1 & 0 & 0 & 0 & 0 & 0 & 0 & 0 & 0 \\
                0 & \tfrac{1}{2} & -\tfrac{1}{2}i & 0 & 0 & 0 & 0 & 0 & 0 \\
                0 & 0 & 0 & 1 & 0 & 0 & -i & 0 & 0 \\
                0 & 0 & 0 & 0 & 0 & \tfrac{1}{2}i & 0 & \tfrac{1}{2}i & 0 \\
                0 & 0 & 0 & 0 & -1 & 0 & 0 & 0 & 1 \\
                0 & -\tfrac{1}{2}i & \tfrac{1}{2} & 0 & 0 & 0 & 0 & 0 & 0 \\
                0 & 0 & 0 & -i & 0 & 0 & 1 & 0 & 0 \\
                0 & 0 & 0 & 0 & 0 & 1 & 0 & -1 & 0 \\
                0 & 0 & 0 & 0 & \tfrac{1}{2}i & 0 & 0 & 0 & \tfrac{1}{2}i},
n_0=\SmallMatrix{ 1 & 0 & 0 & 0 & 0 & 0 & 0 & 0 & 0 \\
                0 & 0 & i & 0 & 0 & 0 & 0 & 0 & 0 \\
                0 & i & 0 & 0 & 0 & 0 & 0 & 0 & 0 \\
                0 & 0 & 0 & 0 & 0 & 0 & i & 0 & 0 \\
                0 & 0 & 0 & 0 & 0 & 0 & 0 & 0 & -1 \\
                0 & 0 & 0 & 0 & 0 & 0 & 0 & -1 & 0 \\
                0 & 0 & 0 & i & 0 & 0 & 0 & 0 & 0 \\
                0 & 0 & 0 & 0 & 0 & -1 & 0 & 0 & 0 \\
                0 & 0 & 0 & 0 & -1 & 0 & 0 & 0 & 0}.
  $$
Here $\Zm_{q,e}^\circ = H\times S$, where $H$ is isomorphic to $\SL(2,\C)$ and
consists of elements $M(A)$, where $A$ is a $2\times 2$-matrix of
determinant 1; furthermore $S$ is a 1-dimensional torus consisting of
elements $T_2(1,t)$, $t\in \C^\times$. Let $U=\SmallMatrix{0&1\\1&0}$. Then
$\sigma(M(A)) = U\ov A U^{-1}$. A straightforward calculation shows that
$M(A)$ is invariant under $\sigma$ if and only if $A=
\SmallMatrix{a&b\\\bar b&\bar a}$ with $a\bar a- b\bar b=1$. It follows that
the group of $\sigma$-fixed points of $H$ is a noncompact form of $\SL(2,\R)$,
and hence isomorphic to $\SL(2,\R)$. It follows that $H$ is
$\Gamma$-equivariantly isomorphic to $\SL(2,\C)$ and hence $\Ho^1 H = 1$.
Furthermore, $\sigma(T_2(1,t)) = T_2(1,\ov t^{-1})$. Therefore $\Ho^1
S = \{[1],[T_2(1,-1)]\}$. The component group is of order 9. So by
Proposition \ref{p:C-3} we see that $\Ho^1 \Zm_{q,e}=\{[1],[T_2(1,-1)]\}$.
\item $-e_{258}+e_{259}+e_{268}-e_{269}+e_{358}-e_{359}-e_{368}+e_{369}$ is a real
  element in $\YY$. We have
$$g_0=\SmallMatrix{ 0 & 1 & 0 & 0 & 0 & 0 & 0 & 0 & 0 \\
                0 & 0 & \tfrac{1}{2} & 0 & 0 & 0 & 0 & 0 & 0 \\
                0 & 0 & 0 & 0 & 1 & 0 & 0 & 1 & 0 \\
                0 & 0 & 0 & \tfrac{1}{2} & 0 & \tfrac{1}{4} & -\tfrac{1}{2} & 0 & \tfrac{1}{4} \\
                0 & 0 & 0 & -1 & 0 & \tfrac{1}{2} & 1 & 0 & \tfrac{1}{2} \\
                i & 0 & 0 & 0 & 0 & 0 & 0 & 0 & 0 \\
                0 & 0 & 0 & 0 & -i & 0 & 0 & i & 0 \\
                0 & 0 & 0 & -i & 0 & -\tfrac{1}{2}i & -i & 0 & \tfrac{1}{2}i \\
                0 & 0 & 0 & \tfrac{1}{2}i & 0 & -\tfrac{1}{4}i & \tfrac{1}{2}i & 0 & \tfrac{1}{4}i},
n_0=\SmallMatrix{ -1 & 0 & 0 & 0 & 0 & 0 & 0 & 0 & 0 \\
                0 & 1 & 0 & 0 & 0 & 0 & 0 & 0 & 0 \\
                0 & 0 & 1 & 0 & 0 & 0 & 0 & 0 & 0 \\
                0 & 0 & 0 & 0 & 0 & 0 & -1 & 0 & 0 \\
                0 & 0 & 0 & 0 & 0 & 0 & 0 & 1 & 0 \\
                0 & 0 & 0 & 0 & 0 & 0 & 0 & 0 & 1 \\
                0 & 0 & 0 & -1 & 0 & 0 & 0 & 0 & 0 \\
                0 & 0 & 0 & 0 & 1 & 0 & 0 & 0 & 0 \\
                0 & 0 & 0 & 0 & 0 & 1 & 0 & 0 & 0}.
$$
Here $\Zm_{q,e}^\circ = H\times S$, where $H$ is isomorphic to $\SL(2,\C)$ and
consists of elements $M(A)$, where $A$ is a $2\times 2$-matrix of
determinant 1; furthermore $S$ is a 2-dimensional torus consisting of
elements $T_2(s,t)$, $s,t\in \C^\times$. We have $\sigma(M(A)) = M(\ov A)$,
$\sigma(T_2(s,t)) = T_2(\ov t, \ov s )$. Hence $\Ho^1 H = \Ho^1 S =1$.
The component group is of order 3. So by
Proposition \ref{p:C-3} we see that $\Ho^1 \Zm_{q,e}=1$.
\item $-e_{456}+e_{789}$ is a real element in $\YY$. We have
 $$g_0=\SmallMatrix{ 0 & 0 & 0 & 0 & 0 & 0 & -1 & 0 & 0 \\
                0 & 0 & 0 & 0 & 0 & 0 & 0 & -\tfrac{1}{2} & \tfrac{1}{2} \\
                -1 & 0 & 0 & 1 & 0 & 0 & 0 & 0 & 0 \\
                0 & 0 & -\tfrac{1}{2} & 0 & \tfrac{1}{2} & 0 & 0 & 0 & 0 \\
                0 & -1 & 0 & 0 & 0 & 1 & 0 & 0 & 0 \\
                0 & 0 & 0 & 0 & 0 & 0 & 0 & \tfrac{1}{2}i & -\tfrac{1}{2}i \\
                i & 0 & 0 & i & 0 & 0 & 0 & 0 & 0 \\
                0 & 0 & i & 0 & i & 0 & 0 & 0 & 0 \\
                0 & \tfrac{1}{2}i & 0 & 0 & 0 & \tfrac{1}{2}i & 0 & 0 & 0},
n_0= \SmallMatrix{ 0 & 0 & 0 & -1 & 0 & 0 & 0 & 0 & 0 \\
                0 & 0 & 0 & 0 & 0 & -1 & 0 & 0 & 0 \\
                0 & 0 & 0 & 0 & -1 & 0 & 0 & 0 & 0 \\
                -1 & 0 & 0 & 0 & 0 & 0 & 0 & 0 & 0 \\
                0 & 0 & -1 & 0 & 0 & 0 & 0 & 0 & 0 \\
                0 & -1 & 0 & 0 & 0 & 0 & 0 & 0 & 0 \\
                0 & 0 & 0 & 0 & 0 & 0 & 1 & 0 & 0 \\
                0 & 0 & 0 & 0 & 0 & 0 & 0 & 0 & 1 \\
                0 & 0 & 0 & 0 & 0 & 0 & 0 & 1 & 0}.  $$
Here $\Zm_{q,e}^\circ$ consists of matrices $M(B)$, where $B$ is a
$3\times 3$-matrix of determinant 1. Let
$U=\SmallMatrix{ -1&0&0\\0&0&-1\\0&-1&0}$. Then $\sigma(M(B)) =
M(U\ov B U^{-1})$. Since $U$ lies in $\SL(3,\C)$ we see that
the group of $\sigma$-fixed points of $\Zm_{q,e}^\circ$ is an inner real form of
$\SL(3,\C)$. Now any inner real form of $\SL(3,\C)$ is isomorphic to
$\SL(3,\R)$. Hence $\Zm_{q,e}$ is $\Gamma$-equivariantly isomorphic to
$\SL(3,\C)$, so that $\Ho^1 \Zm_{q,e}^\circ = 1$. The component group is of
order 3. So by Proposition \ref{p:C-3} we see that $\Ho^1 \Zm_{q,e}=1$.
\item $ie_{123}$ is a real element in $\YY$. We have
$$g_0=\SmallMatrix{ -i & 0 & 0 & 0 & 0 & 0 & 0 & 0 & 0 \\
                0 & -\tfrac{1}{2}i & -\tfrac{1}{2}i & 0 & 0 & 0 & 0 & 0 & 0 \\
                0 & 0 & 0 & 1 & 0 & 0 & i & 0 & 0 \\
                0 & 0 & 0 & 0 & 0 & \tfrac{1}{2} & 0 & \tfrac{1}{2}i & 0 \\
                0 & 0 & 0 & 0 & 1 & 0 & 0 & 0 & i \\
                0 & -\tfrac{1}{2} & \tfrac{1}{2} & 0 & 0 & 0 & 0 & 0 & 0 \\
                0 & 0 & 0 & -i & 0 & 0 & -1 & 0 & 0 \\
                0 & 0 & 0 & 0 & 0 & -i & 0 & -1 & 0 \\
                0 & 0 & 0 & 0 & -\tfrac{1}{2}i & 0 & 0 & 0 & -\tfrac{1}{2}},
n_0=\SmallMatrix{ -1 & 0 & 0 & 0 & 0 & 0 & 0 & 0 & 0 \\
                0 & 0 & -1 & 0 & 0 & 0 & 0 & 0 & 0 \\
                0 & -1 & 0 & 0 & 0 & 0 & 0 & 0 & 0 \\
                0 & 0 & 0 & 0 & 0 & 0 & -i & 0 & 0 \\
                0 & 0 & 0 & 0 & 0 & 0 & 0 & 0 & -i \\
                0 & 0 & 0 & 0 & 0 & 0 & 0 & -i & 0 \\
                0 & 0 & 0 & -i & 0 & 0 & 0 & 0 & 0 \\
                0 & 0 & 0 & 0 & 0 & -i & 0 & 0 & 0 \\
                0 & 0 & 0 & 0 & -i & 0 & 0 & 0 & 0}.
$$
Here $\Zm_{q,e}=H\times S$ where $H$ consists of matrices $M(B)$, where $B$ is a
$3\times 3$-matrix of determinant 1, and $S$ is a 1-dimensional torus
consisting of elements $T_2(1,t)$. We have that $\sigma(M(B))$ is exactly
the same as in the previous case, so that $\Ho^1 H = 1$. Furthermore,
$\sigma(T_2(1,t))=T_2(1,\ov t^{-1})$. We conclude that $\Ho^1 S =
\{ [1], [T_2(1,-1)] \}$ so that also
$\Ho^1 \Zm_{q,e} = \{ [1], [T_2(1,-1)] \}$.
\end{enumerate}

For the cases 2, 5, 9, of the previous list
the ad-hoc procedure indicated in Remark \ref{rem:adhc} did not succeed in
finding a real point in $\YY$. So we resort to the methods
of Section \ref{sec:findreal}. These are summarized in Section
\ref{sec:mixmeth}, here we give the
details of the computation. The centralizer $H = \Zm_{\Gtil_0}(q)$ is described
in \ref{subsec:cen5}. We have $H=\langle h_1\rangle \ltimes M\cdot T_2$,
where $M$ consists of $M(A)=\diag(A,A,A)$, $\det(A)=1$ and $T_2$ consists of the
elements $T_2(a,b)=\diag(a,a,a,b,b,b,(ab)^{-1},
(ab)^{-1},(ab)^{-1})$ for $a,b\in \C^*$. For the conjugation $\tau$ we have
$\tau(h) = n_1\ov hn_1^{-1}$. Set $\HH = (H,\tau)$.
\begin{lemma}
  $\Ho^1 \HH = 1$.
\end{lemma}

\begin{proof}
We have $M\cap T_2=\mu_3$. Therefore,
\[ \Ho^1(M\cdot T_2,\tau)\simeq \Ho^1(M,\tau)\times H^1(T_2,\tau).\]

We have $n_1 \overline{M(A)} n_1^{-1} = M(U\overline{A}U^{-1})$ where
$$U=\SmallMatrix{ -1&0&0\\0&0&-1\\0&-1&0}$$
and $n_1\overline{T_2(a,b)} n_1^{-1} =  T(\bar a, \bar a^{-1}\bar b^{-1} )$.
Hence $(M,\tau)$ is an {\em inner }form of $\SL(3,\R)$. So it is isomorphic to
$\SL(3,\R)$ and therefore $\Ho^1(M,\tau)=1$. By the method of Section
\ref{sec:H1T} it is straightforward to see that $\Ho^1(T_2,\tau)=1$ as well.
Hence $\Ho^1 (M\cdot T_2,\tau)=1$.  By Proposition \ref{p:C-3} we conclude
that $\Ho^1 \HH = 1$.
\end{proof}

The first step now is to find $h_0\in H$ such that $e=h_0\mu(e)$.
Let $U$ be as in the proof of the previous lemma.
For orbit 2 we have $h_0 = h_1 \diag(-U,U,-U)$ (where $h_1$ is as in
\ref{subsec:cen6}).
For orbit 5 we have $h_0=\diag(-U,-U,U)$. For orbit 9 we have
$h_0=\diag(U,U,U)$.

Then we set $d=h_0\tau(h_0)$. For orbits 2 and 9 we have $d=1$, whereas for
orbit 5 we have $d=\diag(1,1,1,-1,-1,-1,-1,-1,-1)=T_2(1,-1)$. So only
for the latter orbit we need to change $h_0$. Since $d$ lies in $T_2$ we
try and just work with that group. With $\nu : T_2\to T_2$,
$\nu(c) = h_0\tau(c)h_0^{-1}$ we have $\nu(T_2(a,b)) = T_2(\ov a, \ov a^{-1}
\ov b^{-1})$. So $T_2(a,b)\nu(T_2(a,b))d = T_2(a\ov a,-b\ov a^{-1}\ov b^{-1})$,
and we want $a,b$ such that $a\ov a=1$ and $\ov a^{-1} b\ov b^{-1}=-1$.
We see that we can take $a=1$, $b=i$. Therefore for orbit 5 we replace $h_0$
by $T_2(1,i)h_0$. We get $h_0=\diag(-U,-iU,-iU)$.

Finally we compute $u\in H$ such that $uh_0\tau(u)^{-1} =1$. For this we
use the methods of Section \ref{sec:comprep}. For orbits 2, 5, 9 we get
respectively that $u$ is
$$\SmallMatrix{ 0 & 0 & 0 & 0 & 0 & 0 & 0 & i & 0 \\
                0 & 0 & 0 & 0 & 0 & 0 & -\tfrac{1}{2} & 0 & i \\
                0 & 0 & 0 & 0 & 0 & 0 & \tfrac{1}{2} & 0 & i \\
                0 & i & 0 & 0 & 0 & 0 & 0 & 0 & 0 \\
                -\tfrac{1}{2} & 0 & i & 0 & 0 & 0 & 0 & 0 & 0 \\
                \tfrac{1}{2} & 0 & i & 0 & 0 & 0 & 0 & 0 & 0 \\
                0 & 0 & 0 & 0 & -i & 0 & 0 & 0 & 0 \\
                0 & 0 & 0 & \tfrac{1}{2} & 0 & -i & 0 & 0 & 0 \\
                0 & 0 & 0 & -\tfrac{1}{2} & 0 & -i & 0 & 0 & 0},
$$
$\diag(A,iA,-A)$ with
$A=\SmallMatrix{ 1&0&0\\0&\tfrac{1}{2}i &1\\0&-\tfrac{1}{2}i &1}$,
and $\diag(B,B,B)$ with
$B=\SmallMatrix{ i & 0 & 0 \\
                0 & -\tfrac{1}{2} & i \\
                0 & \tfrac{1}{2} & i}$.
\end{subsec}

\begin{subsec} Let $p\in \Fm_6$, $p=\ov p$. Then $\Zm_{\Gtil_0}(p)$ has twenty
four nonzero nilpotent orbits in $\z_{\g^\cC}(p)\cap \g_1^\cC$.

Here $\Zm_{\Gtil_0}(p)^\circ$ consists of the elements $\diag(A_1,A_2,A_3)$
where each $A_i$ is a $3\times 3$-matrix of determinant 1. The component
group is generated by the coset of $h_1$.

In the next list, where we omit the computation of $\Ho^1 \Zm_{p,e}$, it
follows directly from the results in Section \ref{sec:galcohom}, especially
Proposition \ref{p:C-3} that it is trivial.

\begin{enumerate}
\item $e=e_{159}+e_{168}+e_{249}+e_{258}+e_{267}+e_{347}$. Here $\Zm_{p,e}$ consists of
  $$\diag(\delta^2\zeta^2,\delta^2\zeta^2,\delta^2\zeta^2,\delta,\delta,
  \delta,\zeta,\zeta,\zeta),$$
  where $\delta^3=\zeta^3=1$.
\item $e=e_{159}+e_{168}+e_{249}+e_{257}+e_{258}+e_{347}$. Here $\Zm_{p,e}$ is
the same as in the previous case.
\item  $e=e_{149}+e_{158}+e_{167}+e_{248}+e_{259}+e_{347}$. Here $\Zm_{p,e}$
  consists of
  $$X(\delta,\zeta,\epsilon)=\diag(\delta^2\zeta,\epsilon\delta^2\zeta,\epsilon\delta^2\zeta,
  \delta\zeta,\epsilon\delta\zeta,\epsilon \delta\zeta,\epsilon\zeta,
  \epsilon\zeta,\zeta),$$
  where $\delta^3=\zeta^3=1$ and $\epsilon^2=1$.

  By Lemma \ref{l:explicit} it follows that $\Ho^1 \Zm_{p,e} =
  \{[1],[X(1,1,-1)]\}$.
\item $e=e_{149}+e_{158}+e_{248}+e_{257}+e_{367}$. Here $\Zm_{p,e}$ consists of
  $$\diag(\zeta^2t^{-1},\zeta^2t^{-1},\zeta^2t^{2},t,t,t^{-2},
  \zeta,\zeta,\zeta)$$
  where $\zeta^3=1$ and $t\in \C^\times$.
\item $e=e_{149}+e_{167}+e_{168}+e_{257}+e_{348}$. Here $\Zm_{p,e}$ consists of
  $$\diag(\zeta^2t^{-1},\zeta^2t^2,\zeta^2t^{-1},t,t^{-2}, t,
  \zeta,\zeta,\zeta)$$
  where $\zeta^3=1$ and $t\in \C^\times$.
\item $e=e_{149}+e_{158}+e_{248}+e_{267}+e_{357}$. Here $\Zm_{p,e}$ consists of
  $$\diag(\zeta^2t,\zeta^2t^{-2},\zeta^2t,\zeta t,\zeta t^{-2}, \zeta t,
  t,t,t^{-2})$$
  where $\zeta^3=1$ and $t\in \C^\times$.
\item $e=e_{149}+e_{158}+e_{167}+e_{248}+e_{357}$. Here $\Zm_{p,e}$ consists of
  $$\diag(\zeta^2,\zeta^2t^3,\zeta^2t^{-3},\zeta t^{-1},\zeta t^2, \zeta t^{-1},
  t,t^{-2},t)$$
  where $\zeta^3=1$ and $t\in \C^\times$.
\item $e=e_{149}+e_{167}+e_{258}+e_{347}$. Here $\Zm_{p,e}^\circ$ consists of
  $$\diag(s^{-1}t^{-1},s^2t^2,s^{-1}t^{-1},s,s^{-2},s,t,t^{-2},t),$$
  where $s,t\in \C^\times$ and $\Zm_{p,e}=\cup_{i=0}^2 h_1^i \Zm_{p,e}^\circ$.
\item $e=e_{147}+e_{158}+e_{258}+e_{269}$. Then $\Zm_{p,e}^\circ$ consists of
  $$T(s,t)=\diag(1,1,1,st,s^{-1},t^{-1},(st)^{-1},s,t)$$
  where $s,t\in \C^\times$. The component group of $\Zm_{p,e}$ is of order 18
  and generated
  by the classes of $Q_1=\diag(\zeta^2,\zeta^2,\zeta^2,\zeta,\zeta,\zeta,1,1,1)$
  where $\zeta$ is a primitive third root of unity, and
  $$Q_2 = \diag\left( \SmallMatrix{ 1 & -1 & 0\\ 0 & -1 & 0\\ 0 & 0 & -1},
  \SmallMatrix{ -1 & 0 & 0\\0 & 0 & -1\\0 & -1 & 0},
  \SmallMatrix{ -1 & 0 & 0\\0 & 0 & 1\\0 & 1 & 0} \right),$$
  $$Q_3 = \diag\left( \SmallMatrix{ -1 & 1 & 0\\ -1 & 0 & 0\\ 0 & 0 & 1},
  \SmallMatrix{ 0 & 0 & 1\\-1 & 0 & 0\\0 & -1 & 0},
  \SmallMatrix{ 0 & 0 & 1\\1 & 0 & 0\\0 & 1 & 0} \right).$$
  Here $[Q_1]$ commutes with $[Q_2]$, $[Q_3]$ and the latter two generate a
  group isomorphic to $S_3$. We have that $[Q_2]$ is of order 2 and $[Q_3]$ is
  of order 3. Furthermore, $Q_2T(s,t)Q_2 = T(t,s)$ and $Q_3T(s,t)Q_3^{-1} =
  T((st)^{-1},s)$.

  We use Proposition \ref{p:C-3-2} to show that $\Ho^1 \Zm_{p,e} =
  \{ [1], [Q_2] \}$. Conditions (1) and (2) of that proposition are obviously
  satisfied. For condition (3) we note that $Q_2^2=1$, so that $[Q_2]$ in the
  component group lifts to a cocycle $Q_2\in \Zl^1 \Zm_{p,e}$. For the
  corresponding twisted conjugation $\sigma$ of $\Zm_{p,e}^\circ$ we have
  $\sigma(T(s,t)) = Q_2T(\ov s, \ov t)Q_2 = T(\ov t, \ov s)$. Hence
  we have $\Ho^1 _{Q_2} \Zm_{p,e}^\circ = 1$. So the desired conclusion
  indeed follows by Proposition \ref{p:C-3-2}.
\item $e=e_{149}+e_{158}+e_{167}+e_{248}+e_{257}+e_{347}$. Here $\Zm_{p,e}$ is of
  order $81$ and generated by
  $\diag(\zeta^2,\zeta^2,\zeta^2,\zeta,\zeta,\zeta,1,1,1)$,
  $\diag(\zeta,1,\zeta^2,\zeta^2,\zeta,1,\zeta^2,\zeta,1)$,
  $\diag(\zeta,\zeta^2,1,\zeta,\zeta^2,1,\zeta^2,1,\zeta)$ and $h_1$.
  From Lemma \ref{l:H1-bijective} it follows that $\Ho^1 \Zm_{p,e}=1$.
\item $e=e_{149}+e_{167}+e_{248}+e_{357}$. Here $\Zm_{p,e}$ consists of
  $$\diag(s,st^3,s^{-2}t^{-3},s^{-1}t^{-1},s^2t^2,s^{-1}t^{-1},
  t,t^{-2},t),$$
  where $s,t\in \C^\times$.
\item $e=e_{149}+e_{167}+e_{247}+e_{258}$. Here $\Zm_{p,e}$ consists of
  $$\diag(s,s^{-2},s,t,s^3t^{-2},s^{-3}t,s^2t^{-1},s^{-1}t^2,
  s^{-1}t^{-1}),$$
  where $s,t\in \C^\times$.
\item $e= e_{149}+e_{158}+e_{167}+e_{248}+e_{257}$. Here $\Zm_{p,e}$ consists of
  $$\diag(\zeta^2\delta,\zeta^2 s, \zeta^2\delta^2s^{-1},\zeta\delta^2s^{-1},
  \zeta\delta,\zeta s, \delta^2s^{-1},\delta,s),$$
  where $\zeta^3=\delta^3=1$ and $s\in \C^\times$.
\item $e=e_{149}+e_{157}+e_{168}+e_{247}+e_{348}$. Here
  $\Zm_{p,e}=H\times \Zm_{p,e}^\circ$. The identity component $\Zm_{p,e}^\circ$
  consists of $\diag(1,A,1,A,A^{-T},1)$ where $A$ is a
  $2\times 2$-matrix of determinant 1.
  The group $H$ is of order 9 and consists of the elements
  $$\diag(\zeta^2\delta^2,\zeta^2\delta^2,\zeta^2\delta^2,\zeta,\zeta,\zeta,
  \delta,\delta,\delta)$$
  where $\zeta^3=\delta^3=1$.
  Hence $\Ho^1 \Zm_{p,e}=\Ho^1 H \times \Ho^1 \Zm_{p,e}^\circ$. Now
  $\Ho^1 H = 1$ by Corollary \ref{c:2m+1} and $\Zm_{p,e}^\circ$ is isomorphic to
  $\SL(2,\C)$ and hence has trivial cohomology as well. We conclude that
  $\Ho^1 \Zm_{p,e}=1$.
\item $e=e_{158}+e_{169}+e_{247}$. Here $\Zm_{p,e}= D\times T_1$. The elements of
  $D$ are
  of the form $\diag(1,1,1,d^{-1},A,d,A^{-T})$ where $A$ is a $2\times 2$-matrix
  and $d=\det(A)\in \C^\times$. Furthermore $T_1$ consists of
  $\diag(s,s^{-2},s,1,1,1,s^2,s^{-1},  s^{-1})$  where $s\in \C^\times$.
  We see that $\Ho^1 D=\Ho^1 T_1=1$ and therefore $\Ho^1 \Zm_{p,e}=1$.
\item $e=e_{149}+e_{158}+e_{167}+e_{247}$. Here $\Zm_{p,e}$ consists of
    $$\diag(\zeta^2,\zeta^2st^2,\zeta^2s^{-1}t^{-2},\zeta t^{-1}, \zeta s^{-1},
    \zeta st, s^{-1}t^{-1},s,t),$$
    where $s\in \C^\times$ and $\zeta^3=1$.
  \item $e=e_{148}+e_{157}+e_{249}+e_{267}$. Here $\Zm_{p,e}^\circ$ consists of
    elements of the form
    $$\diag(A,d^{-1},d^{-1},dA^{-T},d^{-1},dA^{-T}),$$
    where $A$ is a $2\times 2$-matrix and $d=\det(A)\neq 0$.
    Let $g=\diag(\zeta^2,\zeta^2,\zeta^2,\zeta,\zeta,\zeta,
    1,1,1)$ where $\zeta$ is a primitive third root of unity. Then
    $\Zm_{p,e} = \cup_{i=0}^2 g^i \Zm_{p,e}^\circ$.
\item $e=e_{147}+e_{158}+e_{248}+e_{259}$. Here $\Zm_{p,e}^\circ$ consists of
  elements of the form $\diag(A_1,A_2,B)$ where $A_1,A_2,B$ are
  $3\times 3$-matrices. We have $A_1=\diag(A,d^{-1})$ with $A$ a
  $2\times 2$-matrix and $d=\det(A)$. Secondly, $A_2 = \diag(d^{-1}A,d)$.
  Thirdly, there is a surjective homomorphism $\sigma: \GL(2,\C) \to \PSL(2,\C)$
  with $B= \sigma(A)$.
  Let $g= \diag(\zeta^2,\zeta^2,\zeta^2,1,1,1,\zeta,\zeta,\zeta)$,
  where $\zeta$ is a primitive third root of unity. Then
  $\Zm_{p,e} = \cup_{i=0}^2 g^i \Zm_{p,e}^\circ$.
\item $e=e_{149}+e_{157}+e_{248}$. Here $\Zm_{p,e}$ consists of
  $$\diag(s^{-1}tu,s^{-1}u^2,s^2t^{-1}u^{-3},st^{-1}u^{-2},s,s^{-2}tu^2,t^{-1}u^{-1},
  t,u)$$
  with $s,t,u\in \C^\times$.
\item $e=e_{147}+e_{258}$. We have $\Zm_{p,e}=H\ltimes \Zm_{p,e}^\circ$,
  where $\Zm_{p,e}^\circ$ consists of
  $$T_4(s,t,u,v)=
  \diag(s^{-1}u^{-1},t^{-1}v^{-1},stuv,s,t,s^{-1}t^{-1},u,v,u^{-1}v^{-1})$$
  with $s,t,u,v\in \C^\times$.
  Let $g=\diag(A,A,A)$ where $A$ is the $3\times 3$-matrix
  $$\SmallMatrix{ 0 & 1 & 0 \\ 1 & 0 & 0 \\ 0 & 0 & -1 }.$$
  Then $H$ is generated by $g$ and $h_1$. This group is abelian of order
  6. We have $\Ho^1 \Zm_{p,e}^\circ =1$. Consider the conjugation $\tau$ of
  $\Zm_{p,e}$ given by $\tau(a) = g\ov a g^{-1}$. Then
  $$\tau(T_4(s,t,u,v)) = T_4(\ov t, \ov s, \ov v, \ov u ).$$
  So $\Ho^1( \Zm_{p,e}^\circ,\tau) = 1$.
  By Proposition \ref{p:C-3-2} it now follows that
  $\Ho^1 \Zm_{p,e} =\{[1],[g]\}$.
\item $e=e_{148}+e_{157}+e_{247}$. Here $\Zm_{p,e}^\circ$ consists of
  $$\diag(s^{-1}tu,s^{-1}t^3u^2,s^2t^{-4}u^{-3},st^{-2}u^{-1},s,s^{-2}t^2u,
  t^{-1}u^{-1},t,u)$$
  where $s,t,u\in \C^\times$. Furthermore $\Zm_{p,e}=\cup_{i=0}^2 h_1^i
  \Zm_{p,e}^\circ$.
\item $e= e_{147}+e_{158}+e_{169}$. Here $\Zm_{p,e} = H\times \Zm_{p,e}^\circ$
  where $\Zm_{p,e}^{\circ}$ consists of elements of the form
  $\diag(1,A,B^{-T},B)$, where $A$ is a $2\times 2$-matrix of
  determinant 1 and $B$ is a $3\times 3$-matrix of determinant 1. (So that
  $\Zm_{p,e}^\circ \cong \SL(2,\C)\times \SL(3,\C)$.) The group $H$ consists of
  $\diag(\epsilon, \epsilon, \epsilon,
  \delta^2\epsilon^2, \delta^2\epsilon^2,\delta^2\epsilon^2,\delta,\delta,
  \delta)$ with $\delta^3=\epsilon^3=1$.
\item $e=e_{147}+e_{158}$. Here $\Zm_{p,e}$ consists of the elements
  $$\diag(s^{-1}, A, B, t^{-1}, sB^{-T}, s^{-2}t )$$
  where $A,B$ are $2\times 2$-matrices and $s=\det(A)$, $t=\det(B)$.
  We see that $\Zm_{p,e}$ is isomorphic to $\GL(2,\C)\times \GL(2,\C)$.
  Hence $\Ho^1 \Zm_{p,e}=1$.
\item $e=e_{147}$. Here $\Zm_{p,e}^\circ$ consists of
  $$\diag(st, A, s^{-1}, B, t^{-1}, C)$$
  where $A,B,C$ are $2\times 2$-matrices and $s=\det(B)$, $t=\det(C)$ and
  such that $\det(A) = s^{-1}t^{-1}$. We have $\Zm_{p,e}=\cup_{i=0}^2 h_1^i
  \Zm_{p,e}^\circ$.
\end{enumerate}

Let $p$ be a real semisimple element conjugate to $q\in\Fm_6$ with
$\ov q = -q$. Then the map $\mu$ is given by $\mu(y) = n_1 \ov y$.
\begin{enumerate}
\item $-ie_{149}-ie_{157}+ie_{168}-e_{247}-e_{258}+e_{269}+e_{347}-e_{358}+e_{369}$
  is a real element in $\YY$.
  In this case $\Zm_{q,e}$ is of order 9. Hence $\Ho^1 \Zm_{q,e}=1$.
\item  $-e_{149}+e_{157}+e_{259}-e_{267}+e_{348}-e_{359}$
  is a real element in $\YY$.
  In this case $\Zm_{q,e}$ is of order 9. Hence $\Ho^1 \Zm_{q,e}=1$.
\item $-e_{149}+e_{157}+e_{158}-e_{169}-e_{267}+e_{348}$ is a real element in $\YY$.
  We have
  $$g_0=\SmallMatrix{ 0 & 0 & 0 & 0 & 0 & 0 & -2\alpha & 0 & 0 \\
                0 & 0 & 0 & 0 & 0 & 0 & 0 & \alpha & 0 \\
                \alpha & \alpha & 0 & -\alpha & \alpha & 0 & 0 & 0 & 0 \\
                \tfrac{1}{2}\alpha & -\tfrac{1}{2}\alpha & 0 & -\tfrac{1}{2}\alpha & -\tfrac{1}{2}\alpha & 0 & 0 & 0 & 0 \\
                0 & 0 & -2\alpha & 0 & 0 & -2\alpha & 0 & 0 & 0 \\
                0 & 0 & 0 & 0 & 0 & 0 & 0 & 0 & i\alpha \\
                -i\alpha & -i\alpha & 0 & -i\alpha & i\alpha & 0 & 0 & 0 & 0 \\
                -i\alpha & i\alpha & 0 & -i\alpha & -i\alpha & 0 & 0 & 0 & 0 \\
                0 & 0 & i\alpha & 0 & 0 & -i\alpha & 0 & 0 & 0},
  n_0 = \SmallMatrix{ 0 & 0 & 0 & -1 & 0 & 0 & 0 & 0 & 0 \\
                0 & 0 & 0 & 0 & 1 & 0 & 0 & 0 & 0 \\
                0 & 0 & 0 & 0 & 0 & 1 & 0 & 0 & 0 \\
                -1 & 0 & 0 & 0 & 0 & 0 & 0 & 0 & 0 \\
                0 & 1 & 0 & 0 & 0 & 0 & 0 & 0 & 0 \\
                0 & 0 & 1 & 0 & 0 & 0 & 0 & 0 & 0 \\
                0 & 0 & 0 & 0 & 0 & 0 & 1 & 0 & 0 \\
                0 & 0 & 0 & 0 & 0 & 0 & 0 & 1 & 0 \\
                0 & 0 & 0 & 0 & 0 & 0 & 0 & 0 & -1}.$$
  In this case $\Zm_{q,e}\simeq C_2\times C_3\times C_3$. So the cohomology
  is determined by that of $C_2$. The order 2 element is
  $u=\diag(1,-1,-1,1,-1,-1,-1,-1,1)$. We have $\sigma(h)=n_0\ov h n_0^{-1}=
  h^{-1}$ for $h\in \Zm_{q,e}$. So $\sigma(u)=u$  and hence
  $\Ho^1 \Zm_{q,e} = \{ [1], [u]\}$.
\item $-ie_{159}-ie_{247}+ie_{258}+ie_{268}-ie_{347}+ie_{368}+ie_{369}$ is a real
  element in $\YY$. We have
  $$g_0=\SmallMatrix{ 0 & 0 & 0 & 0 & 0 & 0 & 0 & 0 & -i \\
                0 & 0 & 0 & 0 & 0 & 0 & i & 0 & -\tfrac{1}{4}i \\
                0 & 0 & 4i & 0 & 0 & \tfrac{1}{4}i & 0 & 0 & 0 \\
                0 & \tfrac{1}{4} & 0 & 0 & 1 & 0 & 0 & 0 & 0 \\
                -\tfrac{1}{2}i & \tfrac{1}{4} & 0 & 2i & 1 & 0 & 0 & 0 & 0 \\
                0 & 0 & 0 & 0 & 0 & 0 & 0 & \tfrac{1}{2}i & 0 \\
                0 & 0 & 4 & 0 & 0 & -\tfrac{1}{4} & 0 & 0 & 0 \\
                0 & -\tfrac{1}{2}i & 0 & 0 & 2i & 0 & 0 & 0 & 0 \\
                -\tfrac{1}{4} & -\tfrac{1}{8}i & 0 & -1 & \tfrac{1}{2}i & 0 & 0 & 0 & 0},
  n_0=\SmallMatrix{ 0 & 0 & 0 & 4 & 0 & 0 & 0 & 0 & 0 \\
                0 & 0 & 0 & 0 & 4 & 0 & 0 & 0 & 0 \\
                0 & 0 & 0 & 0 & 0 & -\tfrac{1}{16} & 0 & 0 & 0 \\
                \tfrac{1}{4} & 0 & 0 & 0 & 0 & 0 & 0 & 0 & 0 \\
                0 & \tfrac{1}{4} & 0 & 0 & 0 & 0 & 0 & 0 & 0 \\
                0 & 0 & -16 & 0 & 0 & 0 & 0 & 0 & 0 \\
                0 & 0 & 0 & 0 & 0 & 0 & -1 & 0 & 0 \\
                0 & 0 & 0 & 0 & 0 & 0 & 0 & -1 & 0 \\
                0 & 0 & 0 & 0 & 0 & 0 & 0 & 0 & -1}.$$
  The identity component of $\Zm_{q,e}$ is a 1-dimensional torus consisting of
  elements $T_1(t)$. We have $n_0\overline{T_1(t)}n_0 ^{-1} = T_1(\ov t^{-1})$.
  Hence $\Ho^1 \Zm_{q,e}^\circ = \{[1],[T_1(-1)]\}$. The component group is of
  order 3 hence by Proposition \ref{p:C-3} also $\Ho^1 \Zm_{q,e} =
  \{[1],[T_1(-1)]\}$.
\item $-ie_{159}+ie_{249}+ie_{268}+ie_{357}+ie_{368}$ is a real element in $\YY$.
  We have
  $$\SmallMatrix{ 0 & 0 & 0 & 0 & 0 & 0 & 0 & 0 & -1 \\
                0 & 0 & 0 & 0 & 0 & 0 & i & -\tfrac{1}{2}i & 0 \\
                0 & 0 & -1 & 0 & 0 & -i & 0 & 0 & 0 \\
                0 & \tfrac{1}{2}i & 0 & 0 & -\tfrac{1}{2}i & 0 & 0 & 0 & 0 \\
                i & 0 & 0 & 1 & 0 & 0 & 0 & 0 & 0 \\
                0 & 0 & 0 & 0 & 0 & 0 & 0 & -\tfrac{1}{2} & 0 \\
                0 & 0 & i & 0 & 0 & 1 & 0 & 0 & 0 \\
                0 & 1 & 0 & 0 & 1 & 0 & 0 & 0 & 0 \\
                \tfrac{1}{2} & 0 & 0 & \tfrac{1}{2}i & 0 & 0 & 0 & 0 & 0},
  n_0=\SmallMatrix{ 0 & 0 & 0 & -i & 0 & 0 & 0 & 0 & 0 \\
                0 & 0 & 0 & 0 & 1 & 0 & 0 & 0 & 0 \\
                0 & 0 & 0 & 0 & 0 & -i & 0 & 0 & 0 \\
                -i & 0 & 0 & 0 & 0 & 0 & 0 & 0 & 0 \\
                0 & 1 & 0 & 0 & 0 & 0 & 0 & 0 & 0 \\
                0 & 0 & -i & 0 & 0 & 0 & 0 & 0 & 0 \\
                0 & 0 & 0 & 0 & 0 & 0 & -1 & 1 & 0 \\
                0 & 0 & 0 & 0 & 0 & 0 & 0 & 1 & 0 \\
                0 & 0 & 0 & 0 & 0 & 0 & 0 & 0 & 1}.$$
  The identity component of $\Zm_{q,e}$ is a 1-dimensional torus consisting of
  elements $T_1(t)$. We have $n_0\overline{T_1(t)}n_0 ^{-1} = T_1(\ov t^{-1})$.
  Hence $\Ho^1 \Zm_{q,e}^\circ = \{[1],[T_1(-1)]\}$. The component group is of
  order 3 hence by Proposition \ref{p:C-3} also $\Ho^1 \Zm_{q,e} =
  \{[1],[T_1(-1)]\}$.
\item $-e_{148}+e_{167}+e_{249}-e_{257}+e_{259}+e_{349}-e_{357}-e_{359}$ is a real
  element in $\YY$. We have
  $$g_0=\SmallMatrix{ 0 & 0 & 0 & 0 & 0 & 0 & -1 & 0 & 0 \\
                0 & 0 & 0 & 0 & 0 & 0 & 0 & \tfrac{1}{2} & 0 \\
                0 & -2 & 0 & 0 & 2 & 0 & 0 & 0 & 0 \\
                0 & 0 & -\tfrac{1}{4} & 0 & 0 & -\tfrac{1}{4} & 0 & 0 & 0 \\
                -1 & 0 & 0 & -1 & 0 & 0 & 0 & 0 & 0 \\
                0 & 0 & 0 & 0 & 0 & 0 & 0 & 0 & -i \\
                0 & 2i & 0 & 0 & 2i & 0 & 0 & 0 & 0 \\
                0 & 0 & \tfrac{1}{2}i & 0 & 0 & -\tfrac{1}{2}i & 0 & 0 & 0 \\
                \tfrac{1}{2}i & 0 & 0 & -\tfrac{1}{2}i & 0 & 0 & 0 & 0 & 0},
  n_0=\SmallMatrix{ 0 & 0 & 0 & 1 & 0 & 0 & 0 & 0 & 0 \\
                0 & 0 & 0 & 0 & -1 & 0 & 0 & 0 & 0 \\
                0 & 0 & 0 & 0 & 0 & 1 & 0 & 0 & 0 \\
                1 & 0 & 0 & 0 & 0 & 0 & 0 & 0 & 0 \\
                0 & -1 & 0 & 0 & 0 & 0 & 0 & 0 & 0 \\
                0 & 0 & 1 & 0 & 0 & 0 & 0 & 0 & 0 \\
                0 & 0 & 0 & 0 & 0 & 0 & 1 & 0 & 0 \\
                0 & 0 & 0 & 0 & 0 & 0 & 0 & 1 & 0 \\
                0 & 0 & 0 & 0 & 0 & 0 & 0 & 0 & -1}.   $$
  The identity component of $\Zm_{q,e}$ is a 1-dimensional torus consisting of
  elements $T_1(t)$. We have $n_0\overline{T_1(t)}n_0 ^{-1} = T_1(\ov t)$.
  Hence $\Ho^1 \Zm_{q,e}^\circ = 1$. The component group is of
  order 3 hence by Proposition \ref{p:C-3} also $\Ho^1 \Zm_{q,e} =1$.
\item $-ie_{149}-ie_{157}+ie_{168}-ie_{267}-ie_{348}$  is a real element in $\YY$.
  We have
  $$g_0=\SmallMatrix{ -i & 0 & 0 & 0 & 0 & 0 & 0 & 0 & 0 \\
                0 & \tfrac{1}{2} & -\tfrac{1}{2}i & 0 & 0 & 0 & 0 & 0 & 0 \\
                0 & 0 & 0 & 1 & 0 & 0 & i & 0 & 0 \\
                0 & 0 & 0 & 0 & -\tfrac{1}{2}i & 0 & 0 & -\tfrac{1}{2}i & 0 \\
                0 & 0 & 0 & 0 & 0 & -i & 0 & 0 & 1 \\
                0 & \tfrac{1}{2}i & -\tfrac{1}{2} & 0 & 0 & 0 & 0 & 0 & 0 \\
                0 & 0 & 0 & -i & 0 & 0 & -1 & 0 & 0 \\
                0 & 0 & 0 & 0 & -1 & 0 & 0 & 1 & 0 \\
                0 & 0 & 0 & 0 & 0 & -\tfrac{1}{2} & 0 & 0 & \tfrac{1}{2}i},
n_0=\SmallMatrix{ -1 & 0 & 0 & 0 & 0 & 0 & 0 & 0 & 0 \\
                0 & 0 & i & 0 & 0 & 0 & 0 & 0 & 0 \\
                0 & i & 0 & 0 & 0 & 0 & 0 & 0 & 0 \\
                0 & 0 & 0 & 0 & 0 & 0 & -i & 0 & 0 \\
                0 & 0 & 0 & 0 & 0 & 0 & 0 & -1 & 0 \\
                0 & 0 & 0 & 0 & 0 & 0 & 0 & 0 & i \\
                0 & 0 & 0 & -i & 0 & 0 & 0 & 0 & 0 \\
                0 & 0 & 0 & 0 & -1 & 0 & 0 & 0 & 0 \\
                0 & 0 & 0 & 0 & 0 & i & 0 & 0 & 0}.  $$
  The identity component of $\Zm_{q,e}$ is a 1-dimensional torus consisting of
  elements $T_1(t)$. We have $n_0\overline{T_1(t)}n_0 ^{-1} = T_1(\ov t^{-1})$.
  Hence $\Ho^1 \Zm_{q,e}^\circ = \{[1],[T_1(-1)]\}$. The component group is of
  order 3 hence by Proposition \ref{p:C-3} also $\Ho^1 \Zm_{q,e} =
  \{[1],[T_1(-1)]\}$.
\item $e_{158}-e_{169}-ie_{247}-e_{268}-ie_{347}+e_{368}$ is a real element in $\YY$.
  We have
  $$g_0=\SmallMatrix{ -1 & 0 & 0 & 0 & 0 & 0 & 0 & 0 & 0 \\
                0 & -i & 0 & 0 & 0 & 0 & 0 & 0 & 0 \\
                0 & 0 & 0 & 0 & -2i & 0 & 0 & \tfrac{1}{2}i & 0 \\
                0 & 0 & 0 & \tfrac{1}{2} & 0 & 0 & i & 0 & 0 \\
                0 & 0 & 0 & 0 & 0 & \tfrac{1}{2}i & 0 & 0 & 1 \\
                0 & 0 & \tfrac{1}{2} & 0 & 0 & 0 & 0 & 0 & 0 \\
                0 & 0 & 0 & 0 & -2 & 0 & 0 & -\tfrac{1}{2} & 0 \\
                0 & 0 & 0 & -i & 0 & 0 & -2 & 0 & 0 \\
                0 & 0 & 0 & 0 & 0 & \tfrac{1}{4} & 0 & 0 & \tfrac{1}{2}i},
n_0= \SmallMatrix{ 1 & 0 & 0 & 0 & 0 & 0 & 0 & 0 & 0 \\
                0 & -1 & 0 & 0 & 0 & 0 & 0 & 0 & 0 \\
                0 & 0 & 1 & 0 & 0 & 0 & 0 & 0 & 0 \\
                0 & 0 & 0 & 0 & 0 & 0 & -2i & 0 & 0 \\
                0 & 0 & 0 & 0 & 0 & 0 & 0 & \tfrac{1}{4} & 0 \\
                0 & 0 & 0 & 0 & 0 & 0 & 0 & 0 & -2i \\
                0 & 0 & 0 & -\tfrac{1}{2}i & 0 & 0 & 0 & 0 & 0 \\
                0 & 0 & 0 & 0 & 4 & 0 & 0 & 0 & 0 \\
                0 & 0 & 0 & 0 & 0 & -\tfrac{1}{2}i & 0 & 0 & 0}.$$
  The identity component of $\Zm_{q,e}$ is a 2-dimensional torus consisting of
  elements $T_2(s,t)$. We have $n_0\overline{T_2(s,t)}n_0 ^{-1} =
  T_2(\ov t,\ov s)$.
  Hence $\Ho^1 \Zm_{q,e}^\circ = 1$. The component group is of
  order 3 hence by Proposition \ref{p:C-3} also $\Ho^1 \Zm_{q,e} =1$.
\item $e_{249}-e_{268}-e_{357}+e_{368}$ is a real element in $\YY$. We have
$$g_0=\SmallMatrix{ 1 & 0 & 0 & 0 & 0 & 0 & 0 & 0 & 0 \\
                0 & \tfrac{1}{2} & \tfrac{1}{2} & 0 & 0 & 0 & 0 & 0 & 0 \\
                0 & 0 & 0 & 1 & 0 & 0 & 1 & 0 & 0 \\
                0 & 0 & 0 & 0 & 0 & \tfrac{1}{2} & 0 & \tfrac{1}{2} & 0 \\
                0 & 0 & 0 & 0 & 1 & 0 & 0 & 0 & 1 \\
                0 & -\tfrac{1}{2}i & \tfrac{1}{2}i & 0 & 0 & 0 & 0 & 0 & 0 \\
                0 & 0 & 0 & -i & 0 & 0 & i & 0 & 0 \\
                0 & 0 & 0 & 0 & 0 & -i & 0 & i & 0 \\
                0 & 0 & 0 & 0 & -\tfrac{1}{2}i & 0 & 0 & 0 & \tfrac{1}{2}i},
n_0=\SmallMatrix{ 0 & -1 & 0 & 0 & 0 & 0 & 0 & 0 & 0 \\
                -1 & 0 & 0 & 0 & 0 & 0 & 0 & 0 & 0 \\
                0 & 0 & 1 & 0 & 0 & 0 & 0 & 0 & 0 \\
                0 & 0 & 0 & 0 & 0 & 0 & 0 & 0 & -1 \\
                0 & 0 & 0 & 0 & 0 & 0 & 0 & 1 & 0 \\
                0 & 0 & 0 & 0 & 0 & 0 & -1 & 0 & 0 \\
                0 & 0 & 0 & 0 & 0 & -1 & 0 & 0 & 0 \\
                0 & 0 & 0 & 0 & 1 & 0 & 0 & 0 & 0 \\
                0 & 0 & 0 & -1 & 0 & 0 & 0 & 0 & 0}.
$$
We write $C=\Zm_{q,e}$. We consider the conjugation $\sigma$ of $C$
defined by $\sigma(c) = n_0 \ov c n_0^{-1}$. The group $C^\circ$ is a
2-dimensional torus consisting of the elements $T_2(s,t)$. We have
$\sigma(T_2(s,t)) = T_2(\ov s^{-1},\ov s \ov t)$. The component group is
generated by the classes of $Q_1,Q_2,Q_3$. We have $\sigma(Q_1) =
Q_1 T_2(\zeta,\zeta)$, $\sigma(Q_2) = Q_3^{-1}Q_2Q_3$, $\sigma(Q_3) = Q_3^{-1}$.

Write $K=C/C^\circ$ and $\CC^\circ = (C^\circ,\sigma)$, $\CC = (C,\sigma)$ and
$\KK = (K,\sigma)$. As $K$ is of order $2\cdot 9$ it follows by
Lemma \ref{l:explicit} that $\Ho^1  \KK = \{ [1], [h] \}$ where
$h$ is a $\sigma$-fixed element of order 2. A computation shows that we
can take
$$ h = \diag( \SmallMatrix{ 0 & 1 & 0\\ 1 & 0 & 0\\ 0 & 0 & -1},
\SmallMatrix{0 & 0 & 1\\ 0 & -1 & 0 \\ 1 & 0 & 0},
\SmallMatrix{0 & 0 & 1\\ 0 & -1 & 0 \\ 1 & 0 & 0} )=Q_3Q_2Q_3^{-1}.$$
Then $\sigma(h)=h$ and $h^2=1$.

We have a short exact sequence
\[1\to C^\circ\labelto{i} C\labelto{j} K\to 1,\]
which gives rise to a cohomology exact sequence
\[ \KK(\R)\to \Ho^1\CC^\circ\labelto{i_*} \Ho^1\CC\labelto{j_*} \Ho^1
\KK.\]
Since $\Ho^1\CC^\circ=1$ we see from the cohomology exact sequence that
$j_*^{-1}[1]=\{[1]\}$.

In order to compute $j_*^{-1}[h]$ we use Corollary \ref{c:39-cor2}.
For this we twist the exact sequence by $h$; note
that $h$ is a cocycle in $\CC$. So we define a new conjugation $\tau$ of
$C$ by $\tau(c) = h\sigma(c)h^{-1}$. We write $\widehat{\KK} = (K,\tau)$
and similarly we define $\widehat{\CC}^\circ$, $\widehat{\CC}$.
Following Construction \ref{con:rightact} we define a right action of
$\widehat{\KK}(\R)$ on $\Ho^1 \widehat{\CC}^\circ$ by
$$[t]*[a] = [a^{-1}\cdot t \cdot \tau(a)] \text{ for } t\in \Zl^1
\widehat{\CC}^\circ, aC^\circ \in \widehat{\KK}(\R).$$

Now we compute the various objects that play a role here. We have
$\tau(T_2(s,t)) = T_2(\ov s^{-1}, \ov t^{-1})$, implying that
$$\Ho^1 \widehat{\CC}^\circ = \{ [1], [T_2(1,-1)], [T_2(-1,1)], [T_2(-1,-1)]\}.$$
Furthermore, for the action of $\tau$ on $K$ we have
$\tau([Q_1]) = [Q_1^2]$, $\tau([Q_2]) = [Q_2]$ and $\tau([Q_3]) = [Q_3]$.
Hence $\widehat{\KK}(\R)$ is the group of order 6 generated by
$[Q_2]$, $[Q_3]$. In the action defined above we have that $[Q_2]$ interchanges
$[T_2(1,1)]$, $[T_2(-1,-1)]$ and it fixes the other two. Furthermore $[Q_3]$
permutes $[T_2(1,1)]\to [T_2(-1,1)] \to [T_2(-1,-1)]\to [T_2(1,1)]$ and fixes
$[T_2(1,-1)]$. Hence there are two orbits with respective representatives
$[1]$ and $[T_2(1,-1)]$.

Now by Corollary \ref{c:39-cor2} it now follows that $j_*^{-1}[h] =
\{ [h], [T_2(1,-1)h] \}$. Hence $\Ho^1 \Zm_{q,e} = \{ [1], [h], [T_2(1,-1)h] \}$.
\item $-ie_{159}-ie_{247}+ie_{257}+ie_{258}+ie_{269}-ie_{347}+ie_{349}+ie_{358}+ie_{369}$
  is a real element in $\YY$. Here $\Zm_{q,e}$ is of order 81. Hence
  $\Ho^1 \Zm_{q,e}=1$, we do not need to work out the conjugation.
\item $-e_{149}+e_{157}+e_{269}-e_{358}$ is a real element in $\YY$. We have
$$g_0=\SmallMatrix{ 1 & 0 & 0 & 0 & 0 & 0 & 0 & 0 & 0 \\
                0 & \tfrac{1}{2} & -\tfrac{1}{2} & 0 & 0 & 0 & 0 & 0 & 0 \\
                0 & 0 & 0 & 0 & 0 & -1 & 0 & 0 & 1 \\
                0 & 0 & 0 & 0 & -\tfrac{1}{2} & 0 & 0 & -\tfrac{1}{2} & 0 \\
                0 & 0 & 0 & 1 & 0 & 0 & 1 & 0 & 0 \\
                0 & \tfrac{1}{2}i & \tfrac{1}{2}i & 0 & 0 & 0 & 0 & 0 & 0 \\
                0 & 0 & 0 & 0 & 0 & i & 0 & 0 & i \\
                0 & 0 & 0 & 0 & i & 0 & 0 & -i & 0 \\
                0 & 0 & 0 & -\tfrac{1}{2}i & 0 & 0 & \tfrac{1}{2}i & 0 & 0},
  n_0=\SmallMatrix{ 1 & 0 & 0 & 0 & 0 & 0 & 0 & 0 & 0 \\
                0 & 0 & -1 & 0 & 0 & 0 & 0 & 0 & 0 \\
                0 & -1 & 0 & 0 & 0 & 0 & 0 & 0 & 0 \\
                0 & 0 & 0 & 0 & 0 & 0 & 1 & 0 & 0 \\
                0 & 0 & 0 & 0 & 0 & 0 & 0 & 1 & 0 \\
                0 & 0 & 0 & 0 & 0 & 0 & 0 & 0 & -1 \\
                0 & 0 & 0 & 1 & 0 & 0 & 0 & 0 & 0 \\
                0 & 0 & 0 & 0 & 1 & 0 & 0 & 0 & 0 \\
                0 & 0 & 0 & 0 & 0 & -1 & 0 & 0 & 0}.$$
Here $\Zm_{q,e}$ is a 2-dimensional torus consisting of elements $T_2(s,t)$.
We have $n_0  \overline{T_2(s,t)} n_0^{-1} = T_2(\ov s,\ov s^{-1}\ov t^{-1})$.
Hence $\Ho^1 \Zm_{q,e}=1$.
\item $-e_{149}+e_{157}+e_{259}-e_{268}-e_{359}+e_{368}$ is a real element in
  $\YY$. We have
$$g_0=\SmallMatrix{ 1 & 0 & 0 & 0 & 0 & 0 & 0 & 0 & 0 \\
                0 & 0 & -\tfrac{1}{2} & 0 & 0 & 0 & 0 & 0 & 0 \\
                0 & 0 & 0 & 0 & 0 & -1 & 0 & 0 & -1 \\
                0 & 0 & 0 & 0 & \tfrac{1}{2} & 0 & 0 & \tfrac{1}{2} & 0 \\
                0 & 0 & 0 & -1 & 0 & 0 & 1 & 0 & 0 \\
                0 & i & 0 & 0 & 0 & 0 & 0 & 0 & 0 \\
                0 & 0 & 0 & 0 & 0 & i & 0 & 0 & -i \\
                0 & 0 & 0 & 0 & -i & 0 & 0 & i & 0 \\
                0 & 0 & 0 & \tfrac{1}{2}i & 0 & 0 & \tfrac{1}{2}i & 0 & 0},
n_0= \SmallMatrix{ 1 & 0 & 0 & 0 & 0 & 0 & 0 & 0 & 0 \\
                0 & -1 & 0 & 0 & 0 & 0 & 0 & 0 & 0 \\
                0 & 0 & 1 & 0 & 0 & 0 & 0 & 0 & 0 \\
                0 & 0 & 0 & 0 & 0 & 0 & -1 & 0 & 0 \\
                0 & 0 & 0 & 0 & 0 & 0 & 0 & 1 & 0 \\
                0 & 0 & 0 & 0 & 0 & 0 & 0 & 0 & 1 \\
                0 & 0 & 0 & -1 & 0 & 0 & 0 & 0 & 0 \\
                0 & 0 & 0 & 0 & 1 & 0 & 0 & 0 & 0 \\
                0 & 0 & 0 & 0 & 0 & 1 & 0 & 0 & 0}.$$
Here $\Zm_{q,e}$ is a 2-dimensional torus consisting of elements $T_2(s,t)$.
We have $n_0\overline{T_2(s,t)} n_0^{-1} = T_2(\ov s, \ov s^2\ov t^{-1})$.
Hence $\Ho^1 \Zm_{q,e} = \{ [1], [T_2(1,-1)] \}$.
\item $ie_{148}-ie_{159}+ie_{167}+ie_{249}+ie_{257}+ie_{349}+ie_{357}$ is a real
  element in $\YY$. We have
$$g_0=\SmallMatrix{ i & 0 & 0 & 0 & 0 & 0 & 0 & 0 & 0 \\
                0 & -i & 0 & 0 & 0 & 0 & 0 & 0 & 0 \\
                0 & 0 & 0 & -1 & 0 & 0 & 1 & 0 & 0 \\
                0 & 0 & 0 & 0 & 0 & \tfrac{1}{2} & 0 & 0 & -\tfrac{1}{2} \\
                0 & 0 & 0 & 0 & -1 & 0 & 0 & 1 & 0 \\
                0 & 0 & \tfrac{1}{2}i & 0 & 0 & 0 & 0 & 0 & 0 \\
                0 & 0 & 0 & i & 0 & 0 & i & 0 & 0 \\
                0 & 0 & 0 & 0 & 0 & -i & 0 & 0 & -i \\
                0 & 0 & 0 & 0 & \tfrac{1}{2}i & 0 & 0 & \tfrac{1}{2}i & 0},
n_0=\SmallMatrix{ -1 & 0 & 0 & 0 & 0 & 0 & 0 & 0 & 0 \\
                0 & -1 & 0 & 0 & 0 & 0 & 0 & 0 & 0 \\
                0 & 0 & -1 & 0 & 0 & 0 & 0 & 0 & 0 \\
                0 & 0 & 0 & 0 & 0 & 0 & -1 & 0 & 0 \\
                0 & 0 & 0 & 0 & 0 & 0 & 0 & -1 & 0 \\
                0 & 0 & 0 & 0 & 0 & 0 & 0 & 0 & -1 \\
                0 & 0 & 0 & -1 & 0 & 0 & 0 & 0 & 0 \\
                0 & 0 & 0 & 0 & -1 & 0 & 0 & 0 & 0 \\
                0 & 0 & 0 & 0 & 0 & -1 & 0 & 0 & 0}.  $$
Here $\Zm_{q,e}^\circ$ is a 1-dimensional torus consisting of elements $T_1(s)$.
We have $n_0 \overline{T_1(s)} n_0^{-1} = T_1(\ov s)$. Hence $\Ho^1 \Zm_{q,e}^\circ
=1$. The component group is of order 9. Hence by Proposition \ref{p:C-3}
also $\Ho^1 \Zm_{q,e} =1$.
\item $-ie_{159}+ie_{249}+ie_{258}+ie_{357}+ie_{369}$ is a real element in $\YY$.
  We have
$$g_0=\SmallMatrix{ 0 & 0 & 0 & 0 & 0 & 0 & 0 & 0 & i \\
                0 & 0 & 0 & 0 & 0 & 0 & \tfrac{1}{2} & \tfrac{1}{2}i & 0 \\
                0 & 0 & -i & 0 & 1 & 0 & 0 & 0 & 0 \\
                0 & \tfrac{1}{2} & 0 & 0 & 0 & -\tfrac{1}{2}i & 0 & 0 & 0 \\
                i & 0 & 0 & i & 0 & 0 & 0 & 0 & 0 \\
                0 & 0 & 0 & 0 & 0 & 0 & \tfrac{1}{2}i & \tfrac{1}{2} & 0 \\
                0 & 0 & -1 & 0 & i & 0 & 0 & 0 & 0 \\
                0 & -i & 0 & 0 & 0 & 1 & 0 & 0 & 0 \\
                \tfrac{1}{2} & 0 & 0 & -\tfrac{1}{2} & 0 & 0 & 0 & 0 & 0},
n_0= \SmallMatrix{ 0 & 0 & 0 & -1 & 0 & 0 & 0 & 0 & 0 \\
                0 & 0 & 0 & 0 & 0 & i & 0 & 0 & 0 \\
                0 & 0 & 0 & 0 & i & 0 & 0 & 0 & 0 \\
                -1 & 0 & 0 & 0 & 0 & 0 & 0 & 0 & 0 \\
                0 & 0 & i & 0 & 0 & 0 & 0 & 0 & 0 \\
                0 & i & 0 & 0 & 0 & 0 & 0 & 0 & 0 \\
                0 & 0 & 0 & 0 & 0 & 0 & 0 & -i & 0 \\
                0 & 0 & 0 & 0 & 0 & 0 & -i & 0 & 0 \\
                0 & 0 & 0 & 0 & 0 & 0 & 0 & 0 & -1}.  $$
We have $\Zm_{q,e} = H\times \Zm_{q,e}^\circ$. Here $H$ is a group of order 9
and $\Zm_{q,e}^\circ$ consists of the elements $M(A)$, where $A$ is a
$2\times 2$-matrix of determinant 1. For the conjugation we have
$n_0  \overline{M(A)} n_0^{-1} =M(U\overline{A}U^{-1})$ with
$U=\SmallMatrix{ 0 & 1 \\ 1 & 0}$. An easy calculation shows that the group
of real points of $\Zm_{q,e}^\circ$ consists of
\[M\SmallMatrix{a&b\\\bar b&\bar a} \text{ where $a,b\in \C$ and
  $a\bar a - b\bar b=1.$}\]
So this group is a noncompact real form of $\SL(2,\C)$ and hence isomorphic
to $\SL(2,\R)$. It follows that $\Ho^1 \Zm_{q,e}^\circ=1$. As $H$ is of order
9 also $\Ho^1 H = 1$ (whatever the conjugation is). We conclude that
$\Ho^1 \Zm_{q,e}=1$.
\item $-e_{149}+e_{157}-e_{268}+e_{368}$ is a real element in $\YY$. We have
  $$g_0=\SmallMatrix{ 1 & 0 & 0 & 0 & 0 & 0 & 0 & 0 & 0 \\
                0 & 0 & \tfrac{1}{2} & 0 & 0 & 0 & 0 & 0 & 0 \\
                0 & 0 & 0 & 0 & 0 & 1 & 0 & 1 & 0 \\
                0 & 0 & 0 & -\tfrac{1}{2} & 0 & 0 & \tfrac{1}{2} & 0 & 0 \\
                0 & 0 & 0 & 0 & 1 & 0 & 0 & 0 & -1 \\
                0 & -i & 0 & 0 & 0 & 0 & 0 & 0 & 0 \\
                0 & 0 & 0 & 0 & 0 & -i & 0 & i & 0 \\
                0 & 0 & 0 & i & 0 & 0 & i & 0 & 0 \\
                0 & 0 & 0 & 0 & -\tfrac{1}{2}i & 0 & 0 & 0 & -\tfrac{1}{2}i},
n_0=\SmallMatrix{ 1 & 0 & 0 & 0 & 0 & 0 & 0 & 0 & 0 \\
                0 & -1 & 0 & 0 & 0 & 0 & 0 & 0 & 0 \\
                0 & 0 & 1 & 0 & 0 & 0 & 0 & 0 & 0 \\
                0 & 0 & 0 & 0 & 0 & 0 & -1 & 0 & 0 \\
                0 & 0 & 0 & 0 & 0 & 0 & 0 & 0 & -1 \\
                0 & 0 & 0 & 0 & 0 & 0 & 0 & 1 & 0 \\
                0 & 0 & 0 & -1 & 0 & 0 & 0 & 0 & 0 \\
                0 & 0 & 0 & 0 & 0 & 1 & 0 & 0 & 0 \\
                0 & 0 & 0 & 0 & -1 & 0 & 0 & 0 & 0}.$$
Here $\Zm_{q,e}$ consists of the elements
$$M(s,A) = \diag(s,s^{-2},s,\det(A)^{-1},A,s^2\det(A),s^{-1}A^{-T})$$
where $s\in \C^\times$, $A\in \GL(2,\C)$. For the conjugation
$\sigma$ we have
$$\sigma(M(s,A)) = n_0 M(\ov s, \ov A)n_0^{-1} = M(\ov s, \ov s^{-1}
\det(\ov A)^{-1}\ov A).$$
Set $T_2= \C^\times \times \C^\times$. Define $j : \Zm_{q,e} \to T_2$ by
$j(M(s,A)) = (s,\det(A))$.
Let $N$ be the subgroup consisting of $M(1,A)$ for $A\in \SL(2,\C)$. Then we
have the exact sequence
$$1\to N\to \Zm_{q,e} \labelto{j} T_2\to 1.$$
On $T_2$ we use the conjugation $\sigma$ given by $\sigma(s,t) = (\ov s,
\ov s^{-2}\ov t^{-1})$. Then the map $j$ is $\Gamma$-equivariant. We have that
$\Ho^1 (T_2,\sigma) = \{ [(1,1)], [(1,-1)] \}$. By Proposition
\ref{p:serre-prop38} we have the exact sequence
$$ \Ho^1 (N,\sigma) \to \Ho^1 (\Zm_{q,e},\sigma) \labelto{j_*} \Ho^1 (T_2,\sigma).$$
We have $\sigma(M(1,A)) = M(1,\ov A)$ for $A\in \SL(2,\C)$. Hence
$\Ho^1 (N,\sigma)=1$, so that $j_*^{-1}([(1,1)]) = \{ [M(1,1)] \}$.

Consider $b=M(1,\SmallMatrix{0&i\\-i&0})$. Then $b\sigma(b) = 1$ so that
$b\in \Zl^1(\Zm_{q,e},\sigma)$. Furthermore $j(b) = (1,-1)$. We twist the
above sequence by $b$. By $\tau$ we denote the twisted conjugation on
$\Zm_{q,e}$. A short computation shows that $\tau(M(s,A)) =
M(\ov s, \ov s^{-1} (\ov A)^{-T})$. Hence $N^\tau$ is isomorphic to
$\mathrm{SU}(2)$ so that $\Ho^1 (N,\tau) = \{ [M(1,1)], [M(1,-1)]\}$.
Let $b' = M(-1,\SmallMatrix{1&0\\0&-1})$, so $j(b') = (-1,-1)\in T_2^\tau$.
We have
$$(b')^{-1} M(1,1) \tau(b') = M(-1,\SmallMatrix{1&0\\0&-1})M(-1,
-\SmallMatrix{1&0\\0&-1})=M(1,-1).$$
So for the action defined in Construction \ref{con:rightact} we see that
there is one orbit. Hence by Corollary \ref{c:39-cor2} we see that
$j_*^{-1}([(1,-1)]) = \{ [b] \}$. The conclusion is that $\Ho^1(\Zm_{q,e},\sigma)
=\{ [M(1,1)], [b] \}$.
\item $ie_{147}-ie_{158}-ie_{169}+ie_{259}+ie_{359}$ is a real element in $\YY$.
  We have
$$g_0=\SmallMatrix{ -i & 0 & 0 & 0 & 0 & 0 & 0 & 0 & 0 \\
                0 & i & 0 & 0 & 0 & 0 & 0 & 0 & 0 \\
                0 & 0 & 0 & 0 & -1 & 0 & 0 & 1 & 0 \\
                0 & 0 & 0 & 0 & 0 & \tfrac{1}{2} & 0 & 0 & \tfrac{1}{2} \\
                0 & 0 & 0 & 1 & 0 & 0 & 1 & 0 & 0 \\
                0 & 0 & \tfrac{1}{2}i & 0 & 0 & 0 & 0 & 0 & 0 \\
                0 & 0 & 0 & 0 & i & 0 & 0 & i & 0 \\
                0 & 0 & 0 & 0 & 0 & -i & 0 & 0 & i \\
                0 & 0 & 0 & -\tfrac{1}{2}i & 0 & 0 & \tfrac{1}{2}i & 0 & 0},
n_0=\SmallMatrix{ -1 & 0 & 0 & 0 & 0 & 0 & 0 & 0 & 0 \\
                0 & -1 & 0 & 0 & 0 & 0 & 0 & 0 & 0 \\
                0 & 0 & -1 & 0 & 0 & 0 & 0 & 0 & 0 \\
                0 & 0 & 0 & 0 & 0 & 0 & 1 & 0 & 0 \\
                0 & 0 & 0 & 0 & 0 & 0 & 0 & -1 & 0 \\
                0 & 0 & 0 & 0 & 0 & 0 & 0 & 0 & 1 \\
                0 & 0 & 0 & 1 & 0 & 0 & 0 & 0 & 0 \\
                0 & 0 & 0 & 0 & -1 & 0 & 0 & 0 & 0 \\
                0 & 0 & 0 & 0 & 0 & 1 & 0 & 0 & 0}.$$
Here $\Zm_{q,e}$ is a 2-dimensional torus consisting of elements $T_2(s,t)$
for $s,t\in \C^\times$. Now $n_0\overline{T_2(s,t)} n_0^{-1} =
T_2(\ov s^{-1}, \ov s\ov t)$. Hence $\Ho^1 \Zm_{q,e}^\circ = 1$. The component
group is of order 3, so that also $\Ho^1 \Zm_{q,e}=1$ by Proposition \ref{p:C-3}.
\item $e_{249}-e_{258}-e_{357}+e_{369}$ is a real element in $\YY$. We have
$$g_0=\SmallMatrix{ 0 & 0 & -1 & 0 & 0 & 0 & 0 & 0 & 0 \\
                -\tfrac{1}{2} & -\tfrac{1}{2} & 0 & 0 & 0 & 0 & 0 & 0 & 0 \\
                0 & 0 & 0 & 0 & 0 & -1 & 0 & 1 & 0 \\
                0 & 0 & 0 & 0 & -\tfrac{1}{2} & 0 & 0 & 0 & \tfrac{1}{2} \\
                0 & 0 & 0 & 1 & 0 & 0 & 1 & 0 & 0 \\
                -\tfrac{1}{2}i & \tfrac{1}{2}i & 0 & 0 & 0 & 0 & 0 & 0 & 0 \\
                0 & 0 & 0 & 0 & 0 & i & 0 & i & 0 \\
                0 & 0 & 0 & 0 & i & 0 & 0 & 0 & i \\
                0 & 0 & 0 & -\tfrac{1}{2}i & 0 & 0 & \tfrac{1}{2}i & 0 & 0},
n_0=\SmallMatrix{ 0 & 1 & 0 & 0 & 0 & 0 & 0 & 0 & 0 \\
                1 & 0 & 0 & 0 & 0 & 0 & 0 & 0 & 0 \\
                0 & 0 & 1 & 0 & 0 & 0 & 0 & 0 & 0 \\
                0 & 0 & 0 & 0 & 0 & 0 & 1 & 0 & 0 \\
                0 & 0 & 0 & 0 & 0 & 0 & 0 & 0 & -1 \\
                0 & 0 & 0 & 0 & 0 & 0 & 0 & -1 & 0 \\
                0 & 0 & 0 & 1 & 0 & 0 & 0 & 0 & 0 \\
                0 & 0 & 0 & 0 & 0 & -1 & 0 & 0 & 0 \\
                0 & 0 & 0 & 0 & -1 & 0 & 0 & 0 & 0}.  $$
Here $\Zm_{q,e}^\circ$ consists of elements $M(A)$ where $A\in \GL(2,\C)$.
We have $n_0 \overline{M(A)} n_0^{-1} = M(U\ov A U^{-1})$, where $U=\SmallMatrix{
  0&1\\1&0}$. The group of real points under this conjugation consists of
$$\SmallMatrix{a&b\\ \ov b & \ov a} \text{ with } a\ov a-b\ov b\neq 0.$$
Hence it is a noncompact real form of $\GL(2,\C)$, and therefore isomorphic
to $\GL(2,\R)$. It follows that $\Ho^1 \Zm_{q,e}^\circ=1$. The component group
is of order 3, hence also $\Ho^1 \Zm_{q,e}=1$ by Proposition \ref{p:C-3}.
\item $-ie_{159}-ie_{247}+ie_{249}-ie_{347}+ie_{357}$ is a real element in $\YY$.
  We have
$$g_0=\SmallMatrix{ 0 & 0 & 0 & 0 & 0 & 0 & -1 & -2 & -4 \\
                0 & 0 & 0 & 0 & 0 & 0 & \tfrac{1}{2} & \tfrac{1}{2} & 1 \\
                2i & -i & 0 & -2 & 1 & 0 & 0 & 0 & 0 \\
                0 & 0 & -\tfrac{1}{2} & 0 & 0 & \tfrac{1}{2} & 0 & 0 & 0 \\
                -i & i & 0 & 1 & 0 & 0 & 0 & 0 & 0 \\
                0 & 0 & 0 & 0 & 0 & 0 & -\tfrac{1}{2}i & -\tfrac{1}{2}i & 0 \\
                2 & -1 & 0 & -2i & i & 0 & 0 & 0 & 0 \\
                0 & 0 & i & 0 & 0 & i & 0 & 0 & 0 \\
                -\tfrac{1}{2} & \tfrac{1}{2} & 0 & \tfrac{1}{2}i & 0 & 0 & 0 & 0 & 0},
n_0=\SmallMatrix{ 0 & 0 & 0 & i & -i & 0 & 0 & 0 & 0 \\
                0 & 0 & 0 & 0 & -i & 0 & 0 & 0 & 0 \\
                0 & 0 & 0 & 0 & 0 & -1 & 0 & 0 & 0 \\
                i & -i & 0 & 0 & 0 & 0 & 0 & 0 & 0 \\
                0 & -i & 0 & 0 & 0 & 0 & 0 & 0 & 0 \\
                0 & 0 & -1 & 0 & 0 & 0 & 0 & 0 & 0 \\
                0 & 0 & 0 & 0 & 0 & 0 & 1 & 0 & 0 \\
                0 & 0 & 0 & 0 & 0 & 0 & -2 & -1 & 0 \\
                0 & 0 & 0 & 0 & 0 & 0 & 1 & 1 & 1}.  $$
Here $\Zm_{q,e}^\circ$ consists of the elements $M(A)$ where $A$ is a
$2\times 2$-matrix of nonzero determinant. For the conjugation $\sigma$ we have
$\sigma(M(A)) = M(\det(\ov A)^{-1} U\ov AU^{-1})$, where
$$U = \SmallMatrix{1&-1\\0&-1}\text{ so that } U^{-1}=U.$$
We let $N$ be the subgroup consisting of $M(A)$ with $\det(A)=1$. We have an
exact sequence
$$1\to N\to \Zm_{q,e}^\circ \labelto{j} \C^\times \to 1$$
where $j(M(A)) = \det(A)$.
On $\C^\times$ we use the conjugation $\sigma$ given by
$\sigma(s) = \ov s^{-1}$; then $j$ is $\Gamma$-equivariant. This also means that
$\Ho^1 \C^\times = \{ [1], [-1] \}$. By Proposition \ref{p:serre-prop38} we have
the exact sequence
$$ \Ho^1 N \to \Ho^1 \Zm_{q,e}^\circ \labelto{j_*} \Ho^1 \C^\times.$$
By direct computation it is straightforward to see that $N^\sigma = \{ h\in N
\mid \sigma(h)=h\}$ is isomorphic to $\SL(2,\R)$. Hence $\Ho^1 N = 1$.
From the exactness of the sequence it now follows that $j_*^{-1}([1]) = \{[1]\}$.

Now we want to determine $j_*([-1])$ using Corollary \ref{c:39-cor2}. For this
we first lift $-1\in \C^\times$ to a cocycle $M(V)$ in $\Zm_{q,e}^\circ$. This
$V$ has to satisfy $V\sigma(V)=1$ and $\det(V) = -1$. But $\sigma(V) =
\det(\ov V)^{-1}U\ov V U^{-1}=-U\ov VU^{-1}$. So we need $V$ with $\det(V)=-1$ and
$U\ov V =-V^{-1}U$. A calculation shows that we can take
$$V=\SmallMatrix{1&1\\1&0}.$$
Set $W=VU = \SmallMatrix{1&-2\\1&-1}$.
Then the twisted conjugation is given by $\tau(M(A)) =
M(\det(\ov A)^{-1} W\ov AW^{-1})$. A small computation shows that $N^\tau$ is
isomorphic to $\SU(2)$, so that $\Ho^1 (N,\tau) = \{ [1], [-1] \}$
(the latter follows from \cite[III.4.5, Theorem 6 and Example (a)]{Serre1997}).
Corollary \ref{c:39-cor2} shows that $$j_*^{-1}([-1]) = \{ [M(V)],[M(-V)]\}.$$

We show that $[M(V)]\neq [M(-V)]$. By Proposition \ref{p:serre-prop38} we have
the exact sequence
$$(\C^\times)^{\tau} \labelto{\delta} \Ho^1(N,\tau) \labelto{i_*}
\Ho^1(\Zm_{q,e}^\circ,\tau).$$
The twisted conjugation on $\C^\times$ is equal to $\sigma$. The fixed point
group  is $(\C^\times)^\sigma = \{ s\in \C \mid s\ov s = 1\}$ is connected.
This implies that $\delta((\C^\times)^\tau)$ is trivial.
So the kernel of $i_*$ is trivial, and hence $[M(1)]$, $[M(-1)]$ are
distinct in $\Ho^1(\Zm_{q,e}^\circ,\tau)$. By \ref{subsec:twistedbij} we see
that the map $\Ho^1(\Zm_{q,e}^\circ,\tau) \to \Ho^1(\Zm_{q,e}^\circ,\sigma)$,
$[x]\mapsto [xM(V)]$ is bijective. It follows that $[M(V)]$, $[M(-V)]$ are
distinct elements in $\Ho^1(\Zm_{q,e}^\circ,\sigma)$. So we conclude that
$\Ho^1(\Zm_{q,e}^\circ,\sigma)=\{ [1], [M(V)],[M(-V)]\}$. We have
$$M(V) =\SmallMatrix{ 1 & 1 & 0 & 0 & 0 & 0 & 0 & 0 & 0 \\
                1 & 0 & 0 & 0 & 0 & 0 & 0 & 0 & 0 \\
                0 & 0 & -1 & 0 & 0 & 0 & 0 & 0 & 0 \\
                0 & 0 & 0 & -1 & -1 & 0 & 0 & 0 & 0 \\
                0 & 0 & 0 & -1 & 0 & 0 & 0 & 0 & 0 \\
                0 & 0 & 0 & 0 & 0 & -1 & 0 & 0 & 0 \\
                0 & 0 & 0 & 0 & 0 & 0 & 0 & 0 & -1 \\
                0 & 0 & 0 & 0 & 0 & 0 & 0 & -1 & 2 \\
                0 & 0 & 0 & 0 & 0 & 0 & -1 & 1 & -1},
M(-V) = \SmallMatrix{ -1 & -1 & 0 & 0 & 0 & 0 & 0 & 0 & 0 \\
                -1 & 0 & 0 & 0 & 0 & 0 & 0 & 0 & 0 \\
                0 & 0 & -1 & 0 & 0 & 0 & 0 & 0 & 0 \\
                0 & 0 & 0 & 1 & 1 & 0 & 0 & 0 & 0 \\
                0 & 0 & 0 & 1 & 0 & 0 & 0 & 0 & 0 \\
                0 & 0 & 0 & 0 & 0 & -1 & 0 & 0 & 0 \\
                0 & 0 & 0 & 0 & 0 & 0 & 0 & 0 & -1 \\
                0 & 0 & 0 & 0 & 0 & 0 & 0 & -1 & 2 \\
                0 & 0 & 0 & 0 & 0 & 0 & -1 & 1 & -1}. $$
We also describe a second approach. Write $C=\Zm_{q,e}$. We know that
$C^\circ\simeq \GL(2,\C)$ over $\C$. Let $D$ be the commutator subgroup
of $C^\circ$ and let $D^\sigma$ be its group of real points. Then
$D^\sigma$ is isomorphic to $\SL(2,\R)$ and hence to $\SU(1,1)$. The
quotient $(C^\circ)^\sigma/D^\sigma$ is a compact 1-dimensional torus.
It follows that $(C^\circ)^\sigma \simeq \mathrm{U}(1,1)$. Here
we regard $\mathrm{U}(1,1)$ as the unitary group of the Hermitian form
(with matrix) $Q=\diag(1,-1)$. In other words we define the conjugation
$\eta : \GL(2,\C)\to \GL(2,\C)$ by $\eta(g) = Q(\ov g)^{-T} Q^{-1}$. Then
$\mathrm{U}(1,1) = \GL(2,\C)^\eta$.

\begin{lemma}[well known] $\# \Ho^1(\GL(2,\C),\eta)=3$ with cocycles
$c_1=\diag(1,-1)$, $c_0=\diag(1,1)=1$, $c_2=\diag(-1,1)$.
\end{lemma}

\begin{proof}[Idea of proof]
We have that $\Ho^1(\GL(2,\C),\eta)$ classifies equivalence classes of
nondegenerate Hermitian forms on $\C^2$. By the Hermitian version of the
inertia law, there are three classes of such forms with representatives
$F_{+,+}=\diag(1,1),\ F_{+,-}=\diag(1,-1),\ F_{-,-}=\diag(-1,-1)$.
Thus $\# \Ho^1(\GL(2,\C),\eta)=3$.
A calculation  shows that
$_{c_1}F_{+,-}=F_{+,+},\ _{c_2}F_{+,-}=F_{-,-}$,
and of course $_{c_0}F_{+,-}=F_{+,-}$,
which completes the proof of the lemma.
\end{proof}

Now we determine a $\Gamma$-equivariant isomorphism $\varphi : \GL(2,\C)\to
C^\circ$. For this we first determine an isomorphism between the Lie algebras
of $\mathrm{U}(1,1)=\GL(2,\C)^\eta$ and $(C^\circ)^\sigma$. We have that $g\in
\GL(2,\C)$ lies in $\mathrm{U}(1,1)$ if and only if $gQ\ov g^{T} =Q$. So
a $u\in \gl(2,\C)$ lies in $\Lie (\mathrm{U}(1,1))$ if and only if
$uQ+Q\ov u^T =0$. A short calculation shows that
$\Lie (\mathrm{U}(1,1))$ is spanned by
$$I=\SmallMatrix{i&0\\0&i}, \hat h = \SmallMatrix{0&i\\-i&0},
\hat e = \tfrac{1}{2}\SmallMatrix{1&i\\1&-1},
\hat f = \tfrac{1}{2}\SmallMatrix{-i&1\\1&i}.$$
The Lie algebra of $C^\circ$ consists of elements $L(u)$ for $u\in \gl(2,\C)$
where
$$L(u) = \diag(u,-\mathrm{Tr}(u),u+\mathrm{Tr}(u)\cdot 1, \mathrm{Tr}(u),v)$$
where $v$ is a $3\times 3$-matrix depending linearly on $u$ (it is the matrix
of the adjoint representation of $\gl(2,\C)$ on $\ssl(2,\C)$). We have that
$L(u)$ lies in $\Lie( (C^\circ)^\sigma )$ if and only if $U\ov u U = u+
\mathrm{Tr}(\ov u)\cdot 1$.
A computation shows that the Lie algebra of $(C^\circ)^\sigma$ is spanned by
$L(\diag(i,i))$, $L(h)$, $L(e)$, $L(f)$, where
$$h=\SmallMatrix{1&-1\\0&-1}, e=\SmallMatrix{0&i\\0&0},
f=\SmallMatrix{-\tfrac{i}{2} & \tfrac{i}{4}\\-i&\tfrac{i}{2}}.$$
Mapping $\diag(i,i)\mapsto L(\diag(i,i))$, $\hat h\mapsto h$,
$\hat e\mapsto e$, $\hat f\mapsto f$ yields an isomorphism
$d\varphi : \mathrm{U}(1,1) \to (C^\circ)^\sigma$.
Now we define $\varphi$ by $\varphi( \exp(u) ) = \exp (d\varphi(u))$ for
$u\in \Lie(\mathrm{U}(1,1))$.

We have
\begin{align*}
  c_1&=\SmallMatrix{i&0\\0&i} \exp(-\hat e)\exp(\hat f)\exp(-\hat e)\\
  c_2&= \SmallMatrix{i&0\\0&i} \exp(\hat e)\exp(-\hat f)\exp(\hat e).
\end{align*}
So we set
\begin{align*}
q_1 &= \SmallMatrix{i&0\\0&i} \exp(-e)\exp(f))\exp(-e)=
\SmallMatrix{\tfrac{1}{2} & \tfrac{3}{4}\\ 1 & -\tfrac{1}{2}}\\
q_2 &= \SmallMatrix{i&0\\0&i} \exp(e)\exp(-f))\exp(e)=
\SmallMatrix{-\tfrac{1}{2} & -\tfrac{3}{4}\\ -1 & \tfrac{1}{2}},
\end{align*}
with which we have $\varphi(c_1) = M(q_1)$ and $\varphi(c_2) = M(q_2)$.
It follows that $\Ho^1 (C^\circ,\sigma) = \{ [1], [M(q_1)], [M(q_2)]\}$.

Now we want to compute $\Ho^1 (C,\sigma)$. Write $K=C/C^\circ$. We have a short
exact sequence
$$1\to C^\circ\labelto{i} C\labelto{j} K\to 1$$
which by Proposition \ref{p:serre-prop38} gives rise to the exact sequence
$$\Ho^1 (C^\circ,\sigma) \labelto{i_*} \Ho^1(C,\sigma) \labelto{j_*}
\Ho^1(K,\sigma).$$
Since $K$ is of order 3 we have $\Ho^1(K,\sigma)=1$ (Corollary \ref{c:2m+1}).
Hence $i_*$ is surjective. Write
$g= \diag(\zeta^2,\zeta^2,\zeta^2,1,1,1,\zeta,\zeta,\zeta)$, where $\zeta$ is a
primitive third root of unity. Then the class of
$g$ generates $K$. We have
$$n_0 \ov g n_0^{-1} = g^2 M(\diag(\zeta^2,\zeta^2)).$$
It follows that $K^\sigma=1$, so by Proposition \ref{p:action-C-Gamma} we see
that $i_*$ is injective. Hence $\Ho^1 (C,\sigma) =\{ [1], [M(q_1)], [M(q_2)]\} $.
\item $-ie_{159}+ie_{249}+ie_{259}+ie_{357}+ie_{359}$ is a real element in $\YY$.
  We have
$$g_0=\SmallMatrix{ 0 & 0 & 0 & 0 & 0 & 0 & 0 & 0 & 1 \\
                0 & 0 & 0 & 0 & 0 & 0 & -\tfrac{1}{2} & \tfrac{1}{2} & -1 \\
                0 & i & 0 & 0 & 1 & 0 & 0 & 0 & 0 \\
                0 & 0 & -\tfrac{1}{2} & 0 & 0 & \tfrac{1}{2} & 0 & 0 & 0 \\
                -i & 0 & 0 & 1 & 0 & 0 & 0 & 0 & 0 \\
                0 & 0 & 0 & 0 & 0 & 0 & -\tfrac{1}{2}i & -\tfrac{1}{2}i & 0 \\
                0 & 1 & 0 & 0 & i & 0 & 0 & 0 & 0 \\
                0 & 0 & i & 0 & 0 & i & 0 & 0 & 0 \\
                -\tfrac{1}{2} & 0 & 0 & \tfrac{1}{2}i & 0 & 0 & 0 & 0 & 0},
n_0=\SmallMatrix{ 0 & 0 & 0 & i & 0 & 0 & 0 & 0 & 0 \\
                0 & 0 & 0 & 0 & -i & 0 & 0 & 0 & 0 \\
                0 & 0 & 0 & 0 & 0 & -1 & 0 & 0 & 0 \\
                i & 0 & 0 & 0 & 0 & 0 & 0 & 0 & 0 \\
                0 & -i & 0 & 0 & 0 & 0 & 0 & 0 & 0 \\
                0 & 0 & -1 & 0 & 0 & 0 & 0 & 0 & 0 \\
                0 & 0 & 0 & 0 & 0 & 0 & 0 & -1 & 0 \\
                0 & 0 & 0 & 0 & 0 & 0 & -1 & 0 & 0 \\
                0 & 0 & 0 & 0 & 0 & 0 & 0 & 0 & 1}.
$$
Here $\Zm_{q,e}$ consists of the elements $T_3(s,t,u)$. We have
$$n_0  \overline{T_3(s,t,u)} n_0^{-1} =
T_3(\bar s^{-1}\bar u^2,\bar t^{-1}\bar u^{-1}, \bar u).$$
Using the algorithm of Section \ref{sec:H1T} it is straightforward to see
that $\Ho^1 \Zm_{q,e} = \{ [1], [T_3(-1,1,1)]\}$.
\item $e_{249}-e_{357}$ is a real element in $\YY$. We have
  $$g_0=\SmallMatrix{ 0 & 0 & -1 & 0 & 0 & 0 & 0 & 0 & 0 \\
                -\tfrac{1}{2} & -\tfrac{1}{2} & 0 & 0 & 0 & 0 & 0 & 0 & 0 \\
                0 & 0 & 0 & 0 & -1 & 0 & 1 & 0 & 0 \\
                0 & 0 & 0 & 0 & 0 & \tfrac{1}{2} & 0 & 0 & -\tfrac{1}{2} \\
                0 & 0 & 0 & 1 & 0 & 0 & 0 & 1 & 0 \\
                -\tfrac{1}{2}i & \tfrac{1}{2}i & 0 & 0 & 0 & 0 & 0 & 0 & 0 \\
                0 & 0 & 0 & 0 & i & 0 & i & 0 & 0 \\
                0 & 0 & 0 & 0 & 0 & -i & 0 & 0 & -i \\
                0 & 0 & 0 & -\tfrac{1}{2}i & 0 & 0 & 0 & \tfrac{1}{2}i & 0},
n_0=\SmallMatrix{ 0 & 1 & 0 & 0 & 0 & 0 & 0 & 0 & 0 \\
                1 & 0 & 0 & 0 & 0 & 0 & 0 & 0 & 0 \\
                0 & 0 & 1 & 0 & 0 & 0 & 0 & 0 & 0 \\
                0 & 0 & 0 & 0 & 0 & 0 & 0 & 1 & 0 \\
                0 & 0 & 0 & 0 & 0 & 0 & -1 & 0 & 0 \\
                0 & 0 & 0 & 0 & 0 & 0 & 0 & 0 & -1 \\
                0 & 0 & 0 & 0 & -1 & 0 & 0 & 0 & 0 \\
                0 & 0 & 0 & 1 & 0 & 0 & 0 & 0 & 0 \\
                0 & 0 & 0 & 0 & 0 & -1 & 0 & 0 & 0}.  $$
Here $\Zm_{q,e} = H\ltimes \Zm_{q,e}^\circ$. The identity component $\Zm_{q,e}^\circ$
consists of elements $T_4(s,t,u,v)$. We consider the conjugation $\sigma$ given
by $\sigma(a)= n_0 \ov a n_0^{-1}$ for $a\in \Zm_{q,e}$.
We have
$$\sigma(T_4(s,t,u,v)) n_0^{-1} = T_4(\ov v, \ov u, \ov t, \ov s).$$
Hence $\Ho^1 (\Zm_{q,e}^\circ,\sigma) = 1$.
The group $H$ is abelian and of order 6.
It is generated by $g$ (of order 2) and $h_1$ (of order 3). Set
$$d=T_4(1,1,-1,-1) =\diag(-1,-1,1,1,1,1,-1,-1,1).$$
Then with $g'=gd$ we have $g' n_0 \ov g' n_0^{-1} =1$ so that $g$ lifts to
the cocycle $g'\in \Zm_{q,e}$. Now consider the conjugation $\tau$ of $\Zm_{q,e}$
given by $\tau(a) = g' \sigma(a) (g')^{-1}$. Then
$\tau(T_4(s,t,u,v)) = T_4(\ov u, \ov v, \ov s, \ov t)$. Hence
$\Ho^1 (\Zm_{q,e}^\circ,\tau)=1$. By Proposition \ref{p:C-3-2} it follows that
$\Ho^1 (\Zm_{q,e},\sigma)=\{[1],[g']\}$.
\item $-e_{149}+e_{157}+e_{259}-e_{359}$ is a real element in $\YY$. We have
$$g_0=\SmallMatrix{ 1 & 0 & 0 & 0 & 0 & 0 & 0 & 0 & 0 \\
                0 & \tfrac{1}{2} & \tfrac{1}{2} & 0 & 0 & 0 & 0 & 0 & 0 \\
                0 & 0 & 0 & 1 & 0 & 0 & 1 & 0 & 0 \\
                0 & 0 & 0 & 0 & 0 & \tfrac{1}{2} & 0 & \tfrac{1}{2} & 0 \\
                0 & 0 & 0 & 0 & 1 & 0 & 0 & 0 & 1 \\
                0 & -\tfrac{1}{2}i & \tfrac{1}{2}i & 0 & 0 & 0 & 0 & 0 & 0 \\
                0 & 0 & 0 & -i & 0 & 0 & i & 0 & 0 \\
                0 & 0 & 0 & 0 & 0 & -i & 0 & i & 0 \\
                0 & 0 & 0 & 0 & -\tfrac{1}{2}i & 0 & 0 & 0 & \tfrac{1}{2}i},
  n_0=\SmallMatrix{ 1 & 0 & 0 & 0 & 0 & 0 & 0 & 0 & 0 \\
                0 & -1 & 0 & 0 & 0 & 0 & 0 & 0 & 0 \\
                0 & 0 & 1 & 0 & 0 & 0 & 0 & 0 & 0 \\
                0 & 0 & 0 & 0 & 0 & 0 & 1 & 0 & 0 \\
                0 & 0 & 0 & 0 & 0 & 0 & 0 & -1 & 0 \\
                0 & 0 & 0 & 0 & 0 & 0 & 0 & 0 & 1 \\
                0 & 0 & 0 & 1 & 0 & 0 & 0 & 0 & 0 \\
                0 & 0 & 0 & 0 & -1 & 0 & 0 & 0 & 0 \\
                0 & 0 & 0 & 0 & 0 & -1 & 0 & 0 & 0}.$$
Here $\Zm_{q,e}^\circ$ is a 3-dimensional torus consisting of the elements
$T_3(s,t,u)$. We have
$$n_0 \overline{T_3(s,t,u)} n_0^{-1} = T_3(\ov t, \ov s, \ov s^{-2} \ov t^2\ov u).
$$
Hence $\Ho^1 \Zm_{q,e}^\circ = 1$. The component group is of order 3 hence
by Proposition \ref{p:C-3} also $\Ho^1 \Zm_{q,e} = 1$.
\item $-ie_{149}-ie_{157}+ie_{168}$ is a real element in $\YY$. We have
$$g_0=\SmallMatrix{ i & 0 & 0 & 0 & 0 & 0 & 0 & 0 & 0 \\
                0 & -\tfrac{1}{2}i & \tfrac{1}{2} & 0 & 0 & 0 & 0 & 0 & 0 \\
                0 & 0 & 0 & -1 & 0 & 0 & 0 & 1 & -1 \\
                0 & 0 & 0 & 0 & \tfrac{1}{2} & \tfrac{1}{2} & 0 & \tfrac{1}{2} & 0 \\
                0 & 0 & 0 & 0 & 0 & 1 & 1 & 0 & 0 \\
                0 & -\tfrac{1}{2} & \tfrac{1}{2}i & 0 & 0 & 0 & 0 & 0 & 0 \\
                0 & 0 & 0 & i & 0 & 0 & 0 & i & -i \\
                0 & 0 & 0 & 0 & -i & -i & 0 & i & 0 \\
                0 & 0 & 0 & 0 & 0 & -\tfrac{1}{2}i & \tfrac{1}{2}i & 0 & 0},
n_0=\SmallMatrix{ -1 & 0 & 0 & 0 & 0 & 0 & 0 & 0 & 0 \\
                0 & 0 & i & 0 & 0 & 0 & 0 & 0 & 0 \\
                0 & i & 0 & 0 & 0 & 0 & 0 & 0 & 0 \\
                0 & 0 & 0 & 0 & 0 & 0 & 0 & -1 & 1 \\
                0 & 0 & 0 & 0 & 0 & 0 & -1 & 1 & 0 \\
                0 & 0 & 0 & 0 & 0 & 0 & 1 & 0 & 0 \\
                0 & 0 & 0 & 0 & 0 & 1 & 0 & 0 & 0 \\
                0 & 0 & 0 & 0 & 1 & 1 & 0 & 0 & 0 \\
                0 & 0 & 0 & 1 & 1 & 1 & 0 & 0 & 0}.  $$
Here $\Zm_{q,e}$ consists of the elements $M(A,B)$ where $A$, $B$ are
respectively $2\times 2$ and $3\times 3$-matrices of determinant 1. For the
conjugation $\sigma$ we have
$$\sigma(M(A,B))= M(\overline{A'}, U\overline{B}^{-T}U^{-1}).$$
Here if $A= \SmallMatrix{ a & b \\ c & d }$ then $A'=\SmallMatrix{ d & c \\
  b & a }$ and
$$U=\SmallMatrix{ 0 & 0 & 1\\ 0 & 1 & 1 \\ 1 & 1 & 1 }.$$
We write $C^\circ=C_2\times C_3$, where $C^1=\{M(A,1)\}$ and $C_3=\{M(1,B)\}$.

We compute $\Ho^1(C_2,\sigma)$. We have
\[C_2^\sigma=\left\{M(A,1) \mid A= \SmallMatrix{a&b\\ \bar b&\bar a}
\text{ with } a\bar a-b\bar b=1\right\},\]
which is noncompact, hence isomorphic to $\SL(2,\R)$.
It follows that $\Ho^1(C_2,\sigma)=1$.

We have that $M(1,B)\in C_3^\sigma$ if and only if $B^{-1} U \ov B^{-T} =U$.
Passing to the inverse we see that this is equivalent to $BU\ov B^T=U$.
Now $U$ is a symmetric matrix with one negative and two positive eigenvalues.
Hence $C_3^\sigma$ is isomorphic to $\SU(2,1)$. Set
$$A = \SmallMatrix{0&0&1\\ -1&1&0\\ -1&0&1}.$$
Then $AUA^T = \diag(1,1,-1)$. Hence $g\mapsto A^{-1} gA$ is an isomorphism
$\SU(2,1)\to C_3^\sigma$. By Proposition \ref{p:U(m,n)} we have that
$\Ho^1 \SU(2,1)$ consists of the classes of 1 and $D=\diag(-1,-1,1)$.
So $\Ho^1(C_3,\sigma)$ consists of the classes of 1 and $M(1,B)$ where
$$B=A^{-1} DA = \SmallMatrix{1&0&-2\\2&-1&-2\\0&0&-1}.$$

Alternatively we can consider the matrix $B=-U$. A direct computation shows that
it lies in $\Zl^1(C_3,\sigma)$. Consider the Hermitian form $\mathcal{H}_U$ given by
\[(x,y)\mapsto x^t\cdot U\cdot y, \quad x,y\in\C^3.\]
We can diagonalize the matrix $U$ using computer, that is, find a basis
$e_1,e_2,e_3$ in $\C^3$
in which our matrix $U$ will be $\diag(\lambda_1,\lambda_2,\lambda_3)$,
where
\[\lambda_1=-0.8019\dots<0,\ \lambda_2=0.5549\dots>0,\ \lambda_3=2.2469\dots>0.\]
Now our $B=-U$ multiplies $e_1$ by the {\em  positive} number $-\lambda_1$,
and it multiplies $e_2$ and $e_3$ by the {\em negative} numbers $-\lambda_2$
and $-\lambda_3$, respectively.
It follows the  $B$-twisted  Hermitian form $_B (\mathcal{H}_U)$ is negative definite,
and hence nonisomorphic to $\mathcal{H}_U$. Thus $B=-U$ is a nontrivial cocycle.
So $\Ho^1 (C_3,\sigma) = \{ [1], [M(1,B)] \}$.

Using Proposition \ref{p:C-3} we now see that $\Ho^1 \Zm_{q,e}^\circ=
\{[1],[M(1,B)]\}$ so that also $\Ho^1 \Zm_{q,e} =\Ho^1 \{[1],[M(1,B)]\}$.
\item  $-e_{149}+e_{157}$ is a real element in $\YY$. We have
$$g_0=\SmallMatrix{ 1 & 0 & 0 & 0 & 0 & 0 & 0 & 0 & 0 \\
                0 & \tfrac{1}{2} & \tfrac{1}{2} & 0 & 0 & 0 & 0 & 0 & 0 \\
                0 & 0 & 0 & -1 & 1 & 0 & 1 & 1 & 0 \\
                0 & 0 & 0 & 0 & 0 & -\tfrac{1}{2} & 0 & 0 & \tfrac{1}{2} \\
                0 & 0 & 0 & 0 & 1 & 0 & 1 & 0 & 0 \\
                0 & -\tfrac{1}{2}i & \tfrac{1}{2}i & 0 & 0 & 0 & 0 & 0 & 0 \\
                0 & 0 & 0 & i & -i & 0 & i & i & 0 \\
                0 & 0 & 0 & 0 & 0 & i & 0 & 0 & i \\
                0 & 0 & 0 & 0 & -\tfrac{1}{2}i & 0 & \tfrac{1}{2}i & 0 & 0},
n_0=\SmallMatrix{ 1 & 0 & 0 & 0 & 0 & 0 & 0 & 0 & 0 \\
                0 & 0 & 1 & 0 & 0 & 0 & 0 & 0 & 0 \\
                0 & 1 & 0 & 0 & 0 & 0 & 0 & 0 & 0 \\
                0 & 0 & 0 & 0 & 0 & 0 & 0 & -1 & 0 \\
                0 & 0 & 0 & 0 & 0 & 0 & 1 & 0 & 0 \\
                0 & 0 & 0 & 0 & 0 & 0 & 0 & 0 & -1 \\
                0 & 0 & 0 & 0 & 1 & 0 & 0 & 0 & 0 \\
                0 & 0 & 0 & -1 & 0 & 0 & 0 & 0 & 0 \\
                0 & 0 & 0 & 0 & 0 & -1 & 0 & 0 & 0}.  $$
Here $\Zm_{q,e}$ consists of $M(A,B)$ where $A,B\in \GL(2,\C)$. For the
conjugation $\sigma$ we have
$$\sigma(M(A,B)) = M(U\ov AU^{-1}, \det(\ov A) V(\ov B)^{-T} V^{-1} )$$
where
$$U=\SmallMatrix{0&1\\1&0}, V=\SmallMatrix{0&1\\-1&0}.$$
Set $T_2 = (\C^\times)^2$ and define $j : \Zm_{q,e} \to T_2$, $j(M(A,B)) =
(\det(A),\det(B))$. Let $N\subset \Zm_{q,e}$ consist of all $M(A,B)$ with
$\det(A)=\det(B)=1$ (so $N$ is the kernel of $j$). On $T_2$ we define
$\sigma(s,t) = (\ov s, \ov s^2\ov t^{-1})$, then $j$ commutes with $\sigma$.
We have the exact sequence
$$1\to N\to \Zm_{q,e}\labelto{j} T_2\to 1$$
yielding the exact sequence
$$\Ho^1(N,\sigma) \to \Ho^1 (\Zm_{q,e},\sigma) \labelto{j_*} \Ho^1(T_2,\sigma).$$
Let $G_1,G_2$ be the subgroups of $N$ consisting respectively of all
$M(A,1)$ and $M(1,B)$ for $A,B\in \SL(2,\C)$. Then $N=G_1\times G_2$.

Let
\[ A=\SmallMatrix{a&b\\c&d},\text{ then } U\ov A U^{-1}=
\SmallMatrix{\bar d&\bar c\\ \bar b&\bar a}.\]
Thus
\[G_1^\sigma=\left\{M(A,1) \mid A= \SmallMatrix{\bar a&\bar b\\ \bar b&\bar a}\,
\text{ with } a\bar a-b\bar b=1\right\}.\]
We see that the group $G_1^\sigma$ is noncompact, and hence isomorphic to
$\SL(2,\R)$. We conclude that $\Ho^1(G_1,\sigma)=1$.

Similarly write
\[ B=\SmallMatrix{a&b\\c&d},\text{ then } V\ov A V^{-1}=
\SmallMatrix{\bar a&\bar b\\ \bar c&\bar d}.\]
Hence $G_2^\sigma$ is equal to $\SL(2,\R)$ and therefore $\Ho^1(G_2,\sigma)=1$.
We conclude that $\Ho^1(N,\sigma)=1$.

By the method of Section \ref{sec:H1T} it is straightforward to show
that $\Ho^1 (T_2,\sigma) = \{ [1], [T_2(1,-1)] \}$.

Because the above sequence is exact it follows that $j_*^{-1}([1]) = \{[1]\}$.
We want to determine $j_*^{-1}([T_2(1,-1)])$. Let $B_0=\SmallMatrix{0&i\\-i&0}$
and $b=M(1,B_0)$, then $j(b) = T(1,-1)$ and $b\in \Zl^1 (\Zm_{q,e},\sigma)$.
We twist the conjugation $\sigma$ by $b$, that is, we define $\tau(g) =
b\sigma(g)b^{-1}$ for $g\in \Zm_{q,e}$. Then
$$\tau(M(A,B)) = M(U\ov AU^{-1}, \det(\ov A) (\ov B)^{-T} )$$
Then for $g\in G_1$ we have
$\tau(g)=\sigma(g)$ so that $\Ho^1(G_1,\tau)=1$. For $g=M(1,B)\in G_2$ we have
$\tau(g)=M(1,(\ov B)^{-T})$. It follows that $G_2^\tau = \SU(2)$ and therefore
$\Ho^1(G_2,\tau) = \{ [1], [M(1,-1)]\}$. On $T_2$ the twisted conjugation is
equal to $\sigma$ because $T_2$ is abelian. A short computation shows that
$$T_2\tau = T_2^\sigma = \{ (s,t) \mid s\in \R, t\in \C, s^2 = t\ov t\}.$$
Consider $c=(-1,1)\in T_2^\tau$; we compute $[1]\cdot c$ where the action is
defined in Construction \ref{con:rightact}. Let $g=M(\SmallMatrix{1&0\\0&-1},1)$
so that $j(g)=c$. Then $\tau(g) = M(\SmallMatrix{-1&0\\0&1},-1)$,
and
$$[1]\cdot c = [ g^{-1}\tau(g) ] = [ M(-1,-1) ] = [M(1,-1)].$$
Hence $T_2^\tau$ has only one orbit on $\Ho^1(N,\tau)$. It follows that
$j_*^{-1}([T_2(1,-1)]) = \{[b]\}$. We conclude that $\Ho^1(\Zm_{q,e},\sigma) =
\{[1],[b]\}$.
\item $ie_{147}$ is a real element in $\YY$. We have
$$g_0=\SmallMatrix{ 1 & 0 & 0 & 0 & 0 & 0 & 0 & 0 & 0 \\
                0 & \tfrac{1}{2} & \tfrac{1}{2} & 0 & 0 & 0 & 0 & 0 & 0 \\
                0 & 0 & 0 & 1 & 0 & 0 & 1 & 0 & 0 \\
                0 & 0 & 0 & 0 & 0 & \tfrac{1}{2} & 0 & \tfrac{1}{2} & 0 \\
                0 & 0 & 0 & 0 & 1 & 0 & 0 & 0 & 1 \\
                0 & -\tfrac{1}{2}i & \tfrac{1}{2}i & 0 & 0 & 0 & 0 & 0 & 0 \\
                0 & 0 & 0 & -i & 0 & 0 & i & 0 & 0 \\
                0 & 0 & 0 & 0 & 0 & -i & 0 & i & 0 \\
                0 & 0 & 0 & 0 & -\tfrac{1}{2}i & 0 & 0 & 0 & \tfrac{1}{2}i},
n_0=\SmallMatrix{ 1 & 0 & 0 & 0 & 0 & 0 & 0 & 0 & 0 \\
                0 & 0 & 1 & 0 & 0 & 0 & 0 & 0 & 0 \\
                0 & 1 & 0 & 0 & 0 & 0 & 0 & 0 & 0 \\
                0 & 0 & 0 & 0 & 0 & 0 & 1 & 0 & 0 \\
                0 & 0 & 0 & 0 & 0 & 0 & 0 & 0 & 1 \\
                0 & 0 & 0 & 0 & 0 & 0 & 0 & 1 & 0 \\
                0 & 0 & 0 & 1 & 0 & 0 & 0 & 0 & 0 \\
                0 & 0 & 0 & 0 & 0 & 1 & 0 & 0 & 0 \\
                0 & 0 & 0 & 0 & 1 & 0 & 0 & 0 & 0}.  $$
The group $\Zm_{q,e}$ consists of the elements $M(A,B,C)$ where $A,B,C\in
\GL(2,\C)$, $\det(A)\hs\det(B)\hs\det(C)=1$. For the conjugation $\sigma$ we
have the following. Write
$$M(A,B,C)=M(\SmallMatrix{ a_{22} & a_{23}\\ a_{32} & a_{33} },
\SmallMatrix{ a_{55} & a_{56}\\ a_{65} & a_{66} }
\SmallMatrix{ a_{88} & a_{89}\\ a_{98} & a_{99} }).$$
then the $\sigma(M(A,B,C))$ is
$$M(\SmallMatrix{ \bar a_{33} & \bar a_{32}\\ \bar a_{23} & \bar a_{22} },
\SmallMatrix{ \bar a_{99} & \bar a_{98}\\ \bar a_{89} & \bar a_{88} },
\SmallMatrix{ \bar a_{66} & \bar a_{65}\\ \bar a_{56} & \bar a_{55} }).$$

Let $H'$ be the subgroup of $\Zm_{q,e}$ consisting of the elements
$M(A,1,1)$ with $\det(A)=1$. Let $H''=\GL(2,\C)\times \GL(2,\C)$.
We define a homomorphism
\[\phi\colon \Zm_{q,e}\to H'',\quad M(A,B,C)\mapsto (B,C),\]
and we define the conjugation $\sigma$ on $H''$ in such a way that $\phi$
commutes with $\sigma$. We obtain a short exact sequence
\[1\to H'\to H\to H''\to 1.\]
We have
\[(H'')^\sigma = \left\{(\SmallMatrix{ a_{55} & a_{56}\\ a_{65} & a_{66} },
\SmallMatrix{ \bar a_{66} & \bar a_{65}\\ \bar a_{56} & \bar a_{55} })
\mid \SmallMatrix{ a_{55} & a_{56}\\ a_{65} & a_{66} }\in\GL(2,\C) \right\}\simeq \GL(2,\C),\]
whence $\Ho^1(H'',\sigma)=1$ by Proposition \ref{p:Weil}(i).
We have
\[(H')^\sigma\simeq \left\{\SmallMatrix{ a_{22} & a_{23}\\ \bar a_{23} & \bar a_{22}}\mid a_{22}\bar a_{22}- a_{23}\bar a_{23}=1\right\}.\]
We see that $(H')^\sigma$ is noncompact. Thus  $(H')^\sigma$ is a noncompact
real form of $\SL(2,\C)$, and therefore $(H')^\sigma\simeq\SL(2,\R)$ and
$\Ho^1(H',\sigma)=1$.
We conclude that $\Ho^1(\Zm_{q,e},\sigma)=1$.
\end{enumerate}

\end{subsec}


\begin{thebibliography}{Voskresenskii1998}
\bibitem[AT2018]{AT2018}
{\sc J. Adams and O. Ta\"\i bi},
{ Galois and Cartan cohomology of real groups,}
Duke Math. J. {167} (2018), 1057--1097.

\bibitem[AE1982]{AE1982}{\sc L.\,V. Antonyan  and A.\,G. Elashvili},
Classification of spinors in dimension 16,
Proceedings of  the  Tbilici Institute of Mathematics (in Russian), t. LXX, pp. 5--23.

\bibitem[AW1967]{AW}
{\sc M.\,F. Atiyah and C.\,T.\,C. Wall}, {Cohomology of groups}, in:
{A}lgebraic {N}umber {T}heory (J.\,W.\,S.\@ Cassels and A.\@ Fr{\"o}hlich, eds.),
Academic Press/Thompson Book Co., Inc., London/Washington, D.C., 1967,
pp.~94--115.

\bibitem[Borel1966]{Borel1966}
{\sc A. Borel}, Linear algebraic groups, in:
Algebraic Groups and Discontinuous Subgroups
(Proc. Sympos. Pure Math., Boulder, Colo., 1965),
Amer. Math. Soc., Providence, R.I, 1966, pp.~3--19.

\bibitem[Borel-Tits1972]{Borel-Tits}
{\sc A. Borel et J. Tits},
{ Compl\'ements \`a l'article: ``Groupes r\'eductifs'',}
Inst. Hautes \'Etudes Sci. Publ. Math. {\bf 41} (1972), 253--276.

\bibitem[Borovoi1988]{Borovoi1988}
{\sc M. Borovoi},
Galois cohomology of real reductive groups, and real forms of simple Lie algebras,
Functional. Anal. Appl. 22:2 (1988), 135--136.

\bibitem[Borovoi1993]{Borovoi1993}
{\sc M. Borovoi}, M. ``Abelianization of the second nonabelian {G}alois
  cohomology.'' {Duke Math. J.} 72, no.~1 (1993), 217--239.

\bibitem[Borovoi2014]{Borovoi2014}
{\sc M. Borovoi},
Galois cohomology of reductive algebraic groups over the field of real numbers,
\url{arXiv:1401.5913 [math.GR]}

\bibitem[Borovoi2020]{Borovoi2020}
{\sc M. Borovoi}, with an appendix by G. Gagliardi, Equivariant models of
spherical varieties, Transform. Groups 25 (2020), 391--439.

\bibitem[BDR2018]{BDR}
{\sc M.~Borovoi, C.~Daw, and J.~Ren},
  {Conjugation of semisimple subgroups over real number fields of bounded degree},
 \url{arXiv:1802.05894[math.GR].}

\bibitem[BT2021]{BT2021}
{\sc M. Borovoi and D.\,A. Timashev},
Galois cohomology of real semisimple groups via Kac labelings,
to appear in Transform. Groups,
\url{arXiv:2008.11763 [math.GR]},  DOI: 10.1007/S00031-021-09646-z.

\bibitem[CH1988]{CH1988}
{\sc A. M. Cohen and A. G. Helminck},
Trilinear alternating forms on a vector space of dimension 7,
Comm. Algebra 16 (1988), no. 1, 1--25.

\bibitem[CMT2004]{CMT04}
{\sc A.\,M. Cohen, S.\,H. Murray and D.\,E. Taylor},
Computing in groups of {L}ie type. Mathematics of  Computation,
73(247):1477--1498, 2004.

\bibitem[CM1993]{CM1993}{\sc  S.\,H. Collingwood and W.\,M. McGovern},
Nilpotent Orbits in Semisimple Lie Algebras, Van Nostrand Reinhold, New  York, 1993.

\bibitem[CLO2015]{CLO15}
{\sc D. A. Cox, J. Little, and D. O'Shea},
Ideals, varieties, and algorithms.
Undergraduate Texts in Mathematics. Springer, Cham, fourth edition, 2015.

\bibitem[CS2009]{CS2009} {\sc A. \v Cap and  J. Slov\'ak}, Parabolic Geometries. I.
Background and General Theory. Mathematical Surveys and Monographs, 154,
American Mathematical Society, Providence, RI, 2009.,

\bibitem[DGPS]{DGPS}
{\sc W. Decker, G.-M. Greuel, G. Pfister, and H. Sch{\"o}nemann},
\newblock { Singular} {4-2-0} --- {A} computer algebra system for polynomial computations.
\newblock {http://www.singular.uni-kl.de} (2019).

\bibitem[DLA2019]{DLA2019}
{\sc C. Demarche et G. Lucchini Arteche},
Le principe de Hasse pour les espaces homog\`enes:
r\'eduction au cas des stabilisateurs finis,
Compos. Math. 155 (2019), no. 8, 1568--1593.

\bibitem[Djokovi\'c1982]{Djokovic1982}{\sc D.\,J. Djokovi\'c},
Classification of $\Z$-graded real semisimple Lie algebras,  J
ournal of Algebra 76 (1982), 367--382.

\bibitem[Djokovi\'c1983]{Djokovic1983}{\sc D.\,J. Djokovi\'c},
Classification of trivectors of an eight-dimensional real vector  space,
Linear and Multilinear Algebra 13 (1983),  3--39.

\bibitem[GAP]{gap4}
{\sc The GAP~Group}, {GAP -- Groups, Algorithms, and Programming,
Version 4.11.1}; 2021,
\url{https://www.gap-system.org}.

\bibitem[GOV1994]{GOV}
{\sc V.\,V. Gorbatsevich, A.\,L. Onishchik, and  E.\,B. Vinberg},
{Structure of Lie groups and Lie algebras,}
{Lie Groups and Lie Algebras III,}
Encyclopaedia of Mathematical Sciences, Vol. 41,
Springer-Verlag,  Berlin, 1994.

\bibitem[Graaf2017]{Graaf2017}{\sc W.\,A. de Graaf}, Computation with
  Linear Algebraic Groups, CRC Press 2017.

\bibitem[GT2019]{sla} {\sc W. A. de Graaf and the GAP Team}, SLA, Computing with
  simple Lie algebras, Version 1.5.3 (2019)
(Refereed GAP package), \url{https://gap-packages.github.io/sla/}.

\bibitem[Gurevich1935]{Gurevich1935a}{\sc G. B. Gurevich},
Classification des  trivecteurs ayant le rang huit,
Dokl. Akad. Nauk SSSR,  II, Nr. 5-6(1935), 51-113.

\bibitem[Gurevich1964]{Gurevich1964}{\sc G. B. Gurevich},
Foundations of the theory of algebraic invariants,
Noordhoff Ltd. Groningen, The Netherlands, 1964.

\bibitem[Helgason1978]{Helgason1978}{\sc S. Helgason},
Differential Geometry, Lie Groups and Symmetric Spaces, Academic Press 1978.

\bibitem[Hitchin2000]{Hitchin2000} {\sc N. Hitchin},
 The geometry of three-forms in 6 and 7 dimensions, J.D.G. 55 (2000), 547-576.

\bibitem[HL1982] {HL1982}{\sc R. Harvey and  H. B. Lawson},
Calibrated geometries, Acta Math. (1982), 47-157.

\bibitem[HP2020]{HP2020}{\sc J. Hora and P. Pudlak}
Classification of 9-dimensional trilinear alternating forms over GF(2),
 Finite Fields and Their Applications
Volume 70, February 2021, 101788, https://doi.org/10.1016/j.ffa.2020.101788.

\bibitem[Humphreys1980]{Humphreys1980}{\sc J.\,E. Humphreys},
Introduction to Lie algebras and Representation Theory,
 3rd  revised Edition, Springer 1980.

\bibitem[Jacobson1962]{jac}
{\sc N. Jacobson},
\newblock {Lie algebras}.
\newblock Interscience Tracts in Pure and Applied Mathematics, No. 10.
Interscience Publishers (a division of John Wiley \& Sons),
New York-London, 1962, corrected version: 1969.

\bibitem[Jahnel2000]{Jahnel}
{\sc J. Jahnel}, {The {B}rauer-{S}everi variety associated with a central
simple algebra: a survey},
Preprint server: Linear Algebraic Groups and Related Structures, No. 52, 2000,
\url{https://www.math.uni-bielefeld.de/LAG/man/052.pdf}.

\bibitem[Joyce2007]{Joyce2007}{\sc D. Joyce},
Riemannian Holonomy Groups and Calibrated Geometry, Oxford, 2007.

\bibitem[Kac1995]{Kac1995}
V. G. Kac,  Infinite  dimensional  Lie algebras, Cambridge University  Press,  1995.

\bibitem[Kostant1959]{Kostant1959}{\sc B. Kostant},
The principal three-dimensional subgroup
and the Betti numbers of a complex simple Lie group,
Amer. J. Math. 81(1959), 973-1032.

\bibitem[KR1971]{KR71}
{\sc B. Kostant and S. Rallis},  Orbits and representations associated with
  symmetric spaces. Amer. J. Math. 93 (1971), 753--809.

\bibitem[Le1998]{Le1998}{\sc H.\,V. L\^e},
Compact symplectic manifolds of low cohomogeneity,
Journal of Geometry and Physics, 25 (1998), 205--226,
Corrigendum in  Journal of Geometry and Physics, 96 (2015),  212--215.

\bibitem[Le2011]{Le2011}{\sc H.\,V. L\^e},
Orbits in real  $\Z_m$-graded semisimple Lie algebras,
J. Lie Theory, 21(2) (2011), 285--305.

\bibitem[LV2020]{LV2020}{\sc H.\,V. L\^e and J. Van\v zura},
Classification of $k$-forms on $\R^n$
 and the existence of associated geometry on manifolds,
Chebyshevskii Sbornik, v. 21 (2020), pp.  352--372.

\bibitem[MN2013]{MN2013}
{\sc N. Midoune and L. Noui},
Trilinear alternating forms on a vector space
of dimension 8 over a finite field,
Linear Multilinear Algebra 61 (2013), no. 1, 15--21.

\bibitem[Milne2017]{Milne2017}
{\sc J.\,S. Milne}, {Algebraic groups. The theory of group schemes of finite
  type over a field}, Cambridge Studies in Advanced Mathematics, vol. 170,
  Cambridge University Press, Cambridge, 2017.

\bibitem[Nahm1978]{Nahm1978}
{\sc W. Nahm},   Supersymmetries and their representations,
Nucl. Phys. B 135 (1978),  149--66.

\bibitem[Noui1997]{Noui1997}
{\sc L. Noui},
Transvecteur de rang 8 sur un corps alg\'ebriquement clos,
C. R. Acad. Sci. Paris Sér. I Math. 324 (1997), no. 6, 611--614.

\bibitem[NR1994]{NR1994}
{\sc L. Noui et P. Revoy},
Formes multilin\'eaires alternées,
Ann. Math. Blaise Pascal 1 (1994), no. 2, 43--69.

\bibitem[OV1990]{OV}
{\sc  A.\,L. Onishchik and E.\,B. Vinberg},  Lie Groups and Algebraic Groups,
Springer Series in Soviet Mathematics, Springer-Verlag, Berlin, 1990.

\bibitem[PR1994]{PR}
{\sc V.\,P. Platonov and A.\,S. Rapinchuk},
Algebraic groups and number theory, Pure and Applied Mathematics, 139,
Academic Press,  Boston, MA, 1994.

\bibitem[Rashevski1974]{Rashevski1974}{\sc P.\,K. Rashevski},
Theorem on the  connectedness of   a  subgroup of a  simply connected Lie group,
commuting  with    an automorphism of the group,
Trudy Moscow Math.  Society 30, (1974),  3--22.

\bibitem[Reichel1907]{Reichel}
  {\sc W. Reichel}, \"Uber die trilinearen alternierenden Formen in 6 and 7
  Ver\"anderlichen (Dissertation, Greifswald), 1907.

\bibitem[Revoy1979]{Revoy}
{\sc Ph. Revoy}, Trivecteurs de rang 6. Colloque sur les Formes Quadratiques,
2 (Montpellier, 1977). Bull. Soc. Math. France M\'em. 1979, no. 59, 141–155.

\bibitem[Rotman1995]{Rotman1995}
{\sc J.\,J. Rotman},
An introduction to the theory of groups,
Fourth edition, Graduate Texts in Mathematics, 148, Springer-Verlag, New York, 1995.

\bibitem[Sansuc1981]{Sansuc1981}{\sc J.-J. Sansuc},
Groupe de Brauer et arithm\'etique des groupes alg\'ebriques lin\'eaires
sur un corps de nombres, J. Reine Angew. Math. 327 (1981), 12-80.

\bibitem[Schouten1931]{Schouten31}
  {\sc J. A. Schouten}, Klassifizierung der alternierenden
  Gr\"ossen dritten Grades in 7 Dimensionen. Rend. circ. mat. Palermo, 55, 1931.

\bibitem[Serre1979]{Serre1979}{\sc J.-P. Serre}, Local fields,
Graduate Texts in Mathematics, 67, Springer-Verlag, New York-Berlin, 1979.

\bibitem[Serre1997]{Serre1997}{\sc J.-P. Serre},
Galois cohomology, Springer-Verlag, Berlin, 1997.

\bibitem[Sim1994]{Sims}
{\sc C.~C. Sims}, Computation with finitely presented groups,
Cambridge University Press, Cambridge, 1994.

\bibitem[Springer1966]{Springer1966}
{\sc T. A. Springer}, Nonabelian {$H^{2}$} in {G}alois cohomology. In:
  {Algebraic {G}roups and {D}iscontinuous {S}ubgroups ({P}roc. {S}ympos. {P}ure
  {M}ath., {B}oulder, {C}olo., 1965), Proc. Sympos. Pure Math. IX,}  pp.~164--182.
  Amer. Math. Soc., Providence, R.I., 1966.

\bibitem[Steinberg1968]{Steinberg1968} {\sc R. Steinberg},
Endomorphisms of linear algebraic groups, Mem. Amer. Math. Soc. 80 (1968).

\bibitem[Steinberg1975]{Steinberg1975}{\sc R. Steinberg},
Torsion in reductive groups, Advances in Math., 15 (1975),  63--92.

\bibitem[Swann1991]{Swann1991}{\sc  A. Swann},
Hyper-K\"ahler and quaternionic K\"ahler geometry,
Math. Ann. 289 (1991), 421-450.

\bibitem[TY2005]{tauvelyu}
{\sc P. Tauvel and R.\,W.\,T. Yu},
{Lie Algebras and Algebraic Groups}, Springer-Verlag, Berlin, 2005.

\bibitem[Vinberg1975]{Vinberg1975}{\sc E.\,B. Vinberg},
On linear groups connected  with periodic  automorphisms
of semisimple algebraic groups,
Soviet Math.  Dokl.  16:2(1975), 406--409
(originally published in Dokl. Akad. Nauk. SSSR 221:4, (1975), 767--770.)

\bibitem[Vinberg1976]{Vinberg1976}{\sc E.\,B.   Vinberg},
The Weyl group of a graded Lie algebra (Russian),
Izv. Akad. Nauk SSSR Ser. Mat. 40 (1976), no. 3, 488--526, 709.

\bibitem[Vinberg1979]{Vinberg1979}{\sc E.\,B. Vinberg},
Classification of  homogeneous nilpotent elements
of a semisimple graded Lie algebra (Russian),
Trudy Semin. Vektor. Tensor. Anal. 19 (1979), 155--177.
{ English translation: } Selecta Math. Sov. 6 (1987), 15--35.

\bibitem[VE1978]{VE1978}{\sc  E.\,B. Vinberg  and A.\,G. Elashvili},
A classification of the trivectors of nine-dimensional space  (Russian),
Trudy Sem. Vektor. Tenzor. Anal. 18 (1978), 197--233,
English translation: Selecta Math. Sov., 7 (1988), 63--98.

\bibitem[Voskresenskii1998]{Vos}{\sc V. E. Voskresenskii,}
Algebraic Groups and Their Birational Invariants,
Translations of Mathematical Monographs, 179, Amer. Math. Soc., Providence, RI, 1998.

\bibitem[Westwick1981]{Westwick} {\sc R. Westwick}, Real trivectors of rank seven.
Linear and Multilinear Algebra 10 (1981), no. 3, 183--204.

\bibitem[WG1968]{WG1968}{\sc J.\,A. Wolf and A. Gray},
Homogeneous spaces  defined by Lie group automorphisms,
 I, II,  J. Differential Geometry 2 (1968),  77--114, 115--159.

\end{thebibliography}
\end{document}